\documentclass[a4paper,12pt]{report}
\usepackage{amssymb,amsmath,amsfonts,amsthm,stmaryrd}
\usepackage{mathrsfs}

\numberwithin{equation}{section}
\usepackage[left=1 in, right=1 in,top=1 in, bottom=1 in]{geometry}

\providecommand{\abs}[1]{\left\vert#1\right\vert}
\providecommand{\norm}[1]{\left\Vert#1\right\Vert}
\providecommand{\pnorm}[2]{\left\Vert#1\right\Vert_{L^{#2}}}
\providecommand{\pnormspace}[3]{\left\Vert#1\right\Vert_{L^{#2}(#3)}}

\providecommand{\Rn}[1]{\mathbb{R}^{#1}}

\providecommand{\br}[1]{\langle #1 \rangle}
\providecommand{\snorm}[2]{\left\Vert#1\right\Vert_{H^{#2}}}
\providecommand{\snormspace}[3]{\left\Vert#1\right\Vert_{H^{#2}({#3})}}
\providecommand{\sd}[1]{\mathcal{D}_{#1}}
\providecommand{\se}[1]{\mathcal{E}_{#1}}
\providecommand{\sdb}[1]{\bar{\mathcal{D}}_{#1}}
\providecommand{\seb}[1]{\bar{\mathcal{E}}_{#1}}
\providecommand{\snorm}[2]{\left\Vert#1\right\Vert_{H^{#2}}}
\providecommand{\snormspace}[3]{\left\Vert#1\right\Vert_{H^{#2}({#3})}}
\providecommand{\ns}[1]{\norm{#1}^2}
\providecommand{\pns}[2]{\norm{#1}^2_{L^{#2}}}
\providecommand{\dbm}[1]{\bar{D}_m^{#1}}
\providecommand{\ip}[2]{\left(#1,#2\right)}
\providecommand{\iph}[3]{\left(#1,#2\right)_{\mathcal{H}^{#3}}}
\providecommand{\hn}[2]{\norm{#1}_{\mathcal{H}^{#2}}}

\def\wstar{\overset{*}{\rightharpoonup}}
\def\nab{\nabla}
\def\dt{\partial_t}
\def\hal{\frac{1}{2}}
\def\ep{\varepsilon}

\def\ls{\lesssim}
\def\p{\partial}
\def\pa{\partial^\alpha}
\def\sg{\mathbb{D}}
\def\db{\bar{D}}
\def\da{\Delta_{\mathcal{A}}}
\def\naba{\nab_{\mathcal{A}}}
\def\diva{\diverge_{\mathcal{A}}}
\def\Sa{S_{\mathcal{A}}}
\def\Hsig{{_0}H^1_\sigma(\Omega)}
\def\H1{{_0}H^1(\Omega)}

\def\a{\mathcal{A}}
\def\f{\mathcal{F}_{2N}}
\def\g{\mathcal{G}_{2N}}
\def\h{\mathcal{H}}
\def\i{\mathcal{I}}
\def\il{\i_{\lambda}}
\def\k{\mathcal{K}}
\def\n{\mathcal{N}}
\def\w{\mathcal{W}}
\def\x{\mathcal{X}}
\def\y{\mathcal{Y}}
\def\z{\mathcal{Z}}

\def\rest{\hskip 1pt{\hbox to 10.8pt{\hfill\vrule height 7pt width 0.4pt depth 0pt\hbox{\vrule height 0.4pt
width 7.6pt depth 0pt}\hfill}}}

\def\evalu{\hskip 1pt{\hbox to 2pt{\hfill \vrule height -6pt width 0.4pt depth0pt}}}

\DeclareMathOperator{\curl}{curl}
\DeclareMathOperator{\diverge}{div}
\DeclareMathOperator{\supp}{supp}

\newtheorem{lem}{Lemma}[section]
\newtheorem{cor}[lem]{Corollary}
\newtheorem{prop}[lem]{Proposition}
\newtheorem{thm}[lem]{Theorem}
\newtheorem{remark}[lem]{Remark}

\title{Decay of viscous surface waves without surface tension}
\author{Yan Guo\footnote{Supported in part by NSF grant DMS-0905255 and Chinese NSF grant 10828103.  Email: \tt guoy@dam.brown.edu. }\, and Ian Tice\footnote{Supported by an NSF
Postdoctoral Research Fellowship.  Email: \tt  tice@dam.brown.edu.}\\{Brown University, Division of Applied Mathematics}\\
}

\begin{document}

\maketitle

\begin{abstract}
Consider a viscous fluid of finite depth below the air. In the absence of the surface tension effect at the air-fluid interface, the long time behavior of a free surface with small amplitude has been an intriguing question since the work of Beale \cite{beale_1}. In this monograph, we develop a new mathematical framework to resolve this question.  If the free interface is horizontally infinite, we establish that it decays to a flat surface at an algebraic rate.  On the other hand, if the free interface is periodic,  we establish that it decays at an almost exponential rate, i.e. at an arbitrarily fast algebraic rate determined by the smallness of the data.  Our framework contains several novel techniques, which include: (1) a local well-posedness theory of the Navier-Stokes equations in the presence of a moving boundary; (2) a two-tier energy method that couples the boundedness of high-order energy to the decay of low-order energy, the latter of which is necessary to  balance out the growth of the highest derivatives of the free interface; (3) control of both negative and positive Sobolev norms, which enhances interpolation estimates and allows for the decay of infinite surface waves; (4) a localization procedure that is compatible with the energy method and allows for curved lower surface geometry in the periodic case.   Our decay results lead to the construction of global-in-time solutions to the surface wave problem.
\end{abstract}

\tableofcontents

\chapter{Introduction}

\section{Presentation of the problem}

\subsection{Formulation of the equations in Eulerian coordinates}

We consider a viscous, incompressible fluid evolving in a moving domain 
\begin{equation}
\Omega(t) = \{ y \in \Sigma \times \Rn{} \;\vert\; -b(y_1,y_2) < y_3 < \eta(y_1,y_2,t)\}.
\end{equation}
Here we assume that either  $\Sigma = \Rn{2}$, or else $\Sigma = (L_1 \mathbb{T}) \times (L_2 \mathbb{T})$ for $\mathbb{T} = \Rn{} / \mathbb{Z}$ the usual $1-$torus and $L_1, L_2 >0$ the periodicity lengths.  The lower boundary $b$ is assumed to be fixed and given, but the upper boundary is a free surface that is the graph of the unknown function $\eta: \Sigma \times \Rn{+} \to \Rn{}$.  We assume that
\begin{equation}
\begin{cases}
0 < b \in C^\infty(\Sigma)  & \text{if } \Sigma = (L_1 \mathbb{T}) \times (L_2 \mathbb{T}) \\ 
b \in(0,\infty) \text{ is constant}  & \text{if } \Sigma = \Rn{2}.
\end{cases}
\end{equation}
For each $t$, the fluid is described by its velocity and pressure functions $(u,p) :\Omega(t) \to \Rn{3} \times \Rn{}$.  We require that $(u, p, \eta)$ satisfy the gravity-driven incompressible
Navier-Stokes equations in $\Omega(t)$ for $t>0$: 
\begin{equation}\label{ns_euler}
\begin{cases}
\partial_t u + u \cdot \nabla u + \nabla p = \mu \Delta u & \text{in }
\Omega(t) \\ 
\diverge{u}=0 & \text{in }\Omega(t) \\ 
\partial_t \eta = u_3 - u_1 \partial_{y_1}\eta - u_2 \partial_{y_2}\eta & 
\text{on }\{ y_3 = \eta(y_1,y_2,t)\} \\ 
(p I - \mu \mathbb{D}(u) ) \nu = g \eta \nu & \text{on } \{ y_3 =
\eta(y_1,y_2,t)\} \\ 
u = 0 & \text{on } \{y_3 = -b(y_1,y_2)\}%
\end{cases}%
\end{equation}
for $\nu$ the outward-pointing unit normal on $\{y_3 = \eta\}$, $I$ the $3 \times 3$ identity matrix,  $(\mathbb{D} u)_{ij} = \partial_i u_j + \partial_j u_i$ the symmetric gradient of $u$, $g>0$ the strength of gravity, and $\mu>0$ the viscosity.  The tensor $(p I - \mu \mathbb{D}(u))$ is known as the viscous stress tensor.  The third equation in \eqref{ns_euler} implies that the free surface is advected with the fluid.  Note that in \eqref{ns_euler} we have shifted the gravitational forcing to the boundary and eliminated the constant atmospheric pressure, $p_{atm}$, in the usual way by adjusting the actual pressure $\bar{p}$ according to $p = \bar{p} + g y_3 - p_{atm}$.

The problem is augmented with initial data $(u_0, \eta_0)$ satisfying certain compatibility conditions, which for brevity we will not write now.  We will assume that $\eta_0 > -b$ on $\Sigma$.   When $\Sigma = (L_1 \mathbb{T}) \times (L_2 \mathbb{T})$ we shall refer to the problem as either the ``periodic problem''  or the ``periodic case,''  and when $\Sigma = \Rn{2}$  we shall refer to it as either the ``non-periodic problem'' or the ``infinite case.''

Without loss of generality, we may assume that $\mu = g = 1$.  Indeed, a standard scaling argument allows us to scale so that $\mu = g =1$, at the price of multiplying $b$ and the periodicity lengths $L_1,L_2$ by positive constants and rescaling $b$.  This means that, up to renaming  $b$, $L_1$, and $L_2$, we arrive at the above problem with $\mu=g=1$.  

In the periodic case, we assume that the initial surface function satisfies the ``zero average'' condition 
\begin{equation}\label{z_avg}
\frac{1}{L_1 L_2} \int_\Sigma \eta_0 =0.
\end{equation}
If it happens that $\eta_0$ does not satisfy \eqref{z_avg} but does satisfy the extra condition that $\inf_\Sigma b + (\eta_0)>0$, where we have written $(\eta_0)$ for the left side of \eqref{z_avg}, then it is possible to shift the problem to obtain a solution to \eqref{ns_euler} with $\eta_0$ satisfying \eqref{z_avg}.  Indeed,  we may change 
\begin{equation}
 y_3 \mapsto y_3 - (\eta_0), \eta \mapsto \eta - (\eta_0), b \mapsto b + (\eta_0), \text{ and } 
p \mapsto p - (\eta_0)
\end{equation}
to find a new solution with the initial surface function satisfying \eqref{z_avg}.  The data $u_0$ and $\eta_0 - (\eta_0)$ will still satisfy the compatibility conditions, and  $b + (\eta_0) \ge \inf_\Sigma b + (\eta_0)>0$, so after renaming we arrive at the above problem with $\eta_0$ satisfying \eqref{z_avg}.  Note that for sufficiently regular solutions to the periodic problem, the condition \eqref{z_avg} persists in time since $\dt \eta = u \cdot \nu \sqrt{1 + (\p_{y_1} \eta)^2 + (\p_{y_2} \eta)^2}$:
\begin{equation}\label{avg_prop}
 \frac{d}{dt}  \int_{\Sigma} \eta =  \int_{\Sigma} \dt \eta  = \int_{ \{y_3 = \eta(y_1,y_2,t)\} } u \cdot \nu = \int_{\Omega(t)} \diverge{u} = 0.
\end{equation}
The zero average of $\eta(t)$ for $t\ge 0$ is analytically useful in that it allows us to apply the Poincar\'e inequality on $\Sigma$ for all $t\ge 0$.  Moreover, we are interested in the decay $\eta(t) \to 0$ as $t \to \infty$, in say $L^2(\Sigma)$ or $L^\infty(\Sigma)$;  due to the conservation of $(\eta_0)$, we cannot expect this decay unless $(\eta_0) =0$.

The problem \eqref{ns_euler} possesses a natural physical energy.  For sufficiently regular solutions to both the periodic and non-periodic problems, we have an energy evolution equation that expresses how the change in physical energy is related to the dissipation:
\begin{equation}\label{natural_energy}
 \hal \int_{\Omega(t)} \abs{u(t)}^2 + \hal \int_\Sigma \abs{\eta(t)}^2 + \hal \int_0^t \int_{\Omega(s)} \abs{\sg u(s)}^2 ds = \hal \int_{\Omega(0)} \abs{u_0}^2 + \hal \int_\Sigma \abs{\eta_0}^2.
\end{equation}
The first two integrals constitute the kinetic and potential energies, while the third constitutes the dissipation. The structure of this energy evolution equation is the basis of the energy method we will use to analyze \eqref{ns_euler}.

\subsection{Beale's non-decay theorem}\label{beale_non_decay}

In \cite{beale_1}, Beale developed a local existence theory for the non-periodic problem in Lagrangian coordinates.  In Lagrangian coordinates, the problem is transformed to one on a fixed domain $\Omega_0 = \Omega(0)$.  The geometry of the domain (and in particular the free surface)  is encoded in the flow map, $\zeta : \Omega_0 \times (0,T) \rightarrow \Omega(t)$, which gives the trajectory, $t \mapsto \zeta(x,t)$, of a particle located at $x \in \Omega_0$ at $t=0$.  The flow map satisfies $\dt \zeta(x,t) = v(x,t)$, for $v = u \circ \zeta$ the Lagrangian velocity field in the fixed domain $\Omega_0$,  and $u$ the Eulerian velocity field in the moving domain $\Omega(t)$.  The Lagrangian pressure is $q = p\circ \zeta$, for $p$ the pressure in Eulerian coordinates.

The function spaces employed in \cite{beale_1} were 
\begin{equation}\label{beale_kr}
 K^r(\Omega_0 \times(0,T)) := H^0((0,T); H^r(\Omega_0)) \cap H^{r/2}((0,T); H^0(\Omega_0)),
\end{equation}
where $H^k$ denotes the usual Sobolev space.   The local well-posedness showed that, given $v_0= u_0 \in H^{r-1}(\Omega_0)$ for $r \in (3,7/2)$, there exists a unique solution  on a time interval $(0,T)$, with $T$ depending on $u_0$ and $\Omega_0$, so that  $v \in K^r(\Omega_0 \times (0,T))$.  A second local existence theorem was then proved for small data.  It showed that for any fixed $0 < T < \infty$, there exists a collection of sufficiently small data so that a unique solution exists on $(0,T)$ and so that the solutions depend analytically on the data.  

The second result suggests that solutions should exist globally in time for small data.  If global solutions do exist, it is natural to expect the free surface to decay to $0$ as $t \to \infty$.  However, Beale's third result in \cite{beale_1} was a non-decay theorem that showed that a ``reasonable'' extension to small-data global well-posedness with decay of the free surface fails.  More precisely, Theorem 6.4 of \cite{beale_1} establishes that it is possible to choose $\Theta \in H^1(\Omega_0)$ with $\Theta=0$ on $\{x_3=-b\}$ so that there cannot exist a curve of solutions in Lagrangian coordinates, written  $(v(\ep),q(\ep))$ for $\ep$ near $0$, so that
\begin{equation}\label{no_go_3}
  v(\ep) \in K^r(\Omega_0 \times (0,\infty)), \; v(\ep) \in L^1((0,\infty); H^r(\Omega_0)), 
\end{equation}
\begin{equation}
  q(\ep) \in K^{r-3/2}(\Sigma \times (0,\infty)), \; \nab q(\ep) \in K^{r-2}(\Omega_0 \times (0,\infty)),
\end{equation}
for $r \in (3 , 7/2)$, 
\begin{equation}\label{no_go_1}
 \zeta_0(\ep) = Id + \ep \Theta,\;  v_0(\ep) = 0, 
\end{equation}
\begin{equation}\label{no_go_2}
 \lim_{t\to \infty} \zeta_3(\ep) \vert_{\Sigma} =0, 
\text{ and }
v(\ep) = \ep v^1 + \ep^2 v^2 + O(\ep^3).
\end{equation}
The proof, which is a reductio ad absurdum, hinges on $\Theta$ satisfying the properties
\begin{equation}\label{no_go_4}
\diverge{\Theta}=0 \text{ and } \int_{\Sigma} \p_3 \Theta_3 \cdot \Theta_3 \neq 0.
\end{equation}
The condition \eqref{no_go_1} says that the domain is initially close to equilibrium, and the first condition in \eqref{no_go_2} says that the free surface returns to equilibrium as $t \to \infty$.  In the discussion of this result, Beale pointed out that it does not imply the non-existence of global-in-time solutions, but rather that establishing global-in-time results requires weaker or different hypotheses than those imposed in the non-decay theorem.

The non-decay theorem raises two intriguing questions.  First, is viscosity alone capable of producing global well-posedness?  Second, if global solutions exist, do they decay as $t \to \infty$?  Our main results answer both questions in the affirmative.  In order to avoid the applicability of the non-decay theorem, we must impose conditions that prevent its hypotheses from being satisfied.  We would like to highlight three crucial ways in which we do this.  The first and most obvious is that we require higher regularity and more compatibility conditions for the initial data.

Second, in the non-periodic case we will find (see Remark \ref{intro_inf_rem_2}) that while $u$ does decay, it does not do so sufficiently rapidly to guarantee that  $u$ belongs to the space  $L^1((0,\infty); H^2(\Omega_0))$, which is in violation of a key assumption \eqref{no_go_3} in the non-extension result. Technically, our $u$ is in Eulerian coordinates, but if we formally identify $u$ with $v$, then we see the difficulty clearly: we cannot integrate the equation $\partial_t \zeta = v$ to obtain $\zeta$ as $t \to \infty$, so that we cannot make sense of \eqref{no_go_2}. One of the advantages of the Eulerian formulation is that the free surface function $\eta$ may be analyzed without regard to what is happening to the entire flow map $\zeta$ in $\Omega$. It is conceivable that the graph $\zeta_3(\Sigma,t)$ does, in fact, decay to the flat graph of $\Sigma$, but that $\zeta$ as a whole does not decay to the identity in $\Omega$. In Eulerian coordinates this would be observed in the decay of $\eta$, which we can deduce in both $L^\infty$ and $L^2$ (see \eqref{intro_inf_gwp_03} of Theorem \ref{intro_inf_gwp}).

The third difference is found in the periodic setting, where we assume that $\eta_0$ has zero average in \eqref{z_avg}. We claim that this condition makes \eqref{no_go_4} impossible, i.e. the zero average condition prevents the choice of $\Theta$ satisfying \eqref{no_go_4}, which then breaks the reductio ad absurdum used to prove the non-decay theorem. The argument in the theorem, which is proved in the non-periodic case but would work in the periodic case as well, goes as follows. The expansion of $v(\ep)$ in \eqref{no_go_2}, and the $L^1$ condition \eqref{no_go_3} imply an expansion $\zeta(\varepsilon) = \varepsilon \zeta^1 + \varepsilon^2 \zeta^2 + O(\varepsilon^3)$. The term $v^1$ is assumed to be known, and a contradiction is derived in solving for $v^2$ using the $\zeta$ expansion, if $\Theta$ is chosen to satisfy \eqref{no_go_4}.

To show that the zero average condition prevents the choice of $\Theta$ satisfying \eqref{no_go_4}, we must first compare the flow map, $\zeta$, to the free surface function, $\eta$.  Since $\zeta$ and $\eta$ yield the same surface, we must have that as graphs, 
\begin{equation}
 \{ (\zeta_1(x_1,x_2,0,t),\zeta_2(x_1,x_2,0,t), \zeta_3(x_1,x_2,0,t))\} = \{ (x_1,x_2,\eta(x_1,x_2,t))\}.
\end{equation}
Let $\psi_i(x_1,x_2,t) = \zeta_i(x_1,x_2,0,t)$ for $i=1,2$.  If $\zeta$ is a diffeomorphism, then it is possible to solve $\psi(y_1,y_2,t) = (x_1,x_2) = x'$ for $y'=(y_1,y_2)$, i.e. $y' = \psi^{-1}(x',t)$.  Hence
\begin{equation}
 \eta(x_1,x_2,t) = \zeta_3(\psi^{-1}(x_1,x_2,t),0,t) \text{ for all } x' \in \Sigma, t\ge 0.
\end{equation}
At time $t=0$ we have
\begin{equation}
\psi_0(\ep) = (x_1 + \ep \Theta_1) e_1 + ( x_2 + \ep \Theta_2)e_2, \text{ and } e_3 \cdot \zeta_0(\ep)(y',0) = \ep \Theta_3(y',0), 
\end{equation}
so that $\eta_0(x') = \ep \Theta_3(\psi_0(\ep)^{-1}(x'),0).$  Using the zero average condition and a change of variables shows that
\begin{equation}
 0 = \int_\Sigma \eta_0(x') dx' = \ep \int_\Sigma \Theta_3(\psi_0(\ep)^{-1}(x'),0) dx' = \ep \int_\Sigma \Theta_3(y',0) \abs{\det{D_y \psi_0(\ep)}}dy',
\end{equation}
but it is easily verified that for $\ep$ near $0$, 
\begin{equation}
\abs{\det{D_y \psi_0(\ep)}} =  \det{D_y \psi_0(\ep)} = 1 + \ep (\p_1 \Theta_1 + \p_2 \Theta_2) + O(\ep^2),
\end{equation}
so that
\begin{equation}
 0 = \ep \int_\Sigma \Theta_3 \left(1 + \ep (\p_1 \Theta_1 + \p_2 \Theta_2) + O(\ep^2) \right) dy'.
\end{equation}
Sending $\ep \to 0$, we find that $\Theta$ must satisfy
\begin{equation}
 0 = \int_\Sigma \Theta_3 dy' \text{ and } 0=\int_\Sigma  \Theta_3  (\p_1 \Theta_1 + \p_2 \Theta_2)  dy'.
\end{equation}
However, $\diverge{\Theta}=0$ implies that $\p_3 \Theta_3 = -(\p_1 \Theta_1 + \p_2 \Theta_2)$, so that the latter condition becomes
\begin{equation}
 0=\int_\Sigma  \Theta_3 \p_3 \Theta_3,
\end{equation}
in violation of assumption \eqref{no_go_4}.

This analysis shows that imposing condition  \eqref{no_go_4} on the initial data for the flow map is essentially equivalent to choosing an initial coordinate system in which the average disturbance of the free surface does not vanish.  If the system returns to equilibrium, then the map describing the equilibrium surface should be a non-zero constant (whatever the initial average was), and hence we should not expect $L^2$ or $L^\infty$ decay of this map.  Choosing the initial data with zero average circumvents this problem and allows for $L^2$ and $L^\infty$ decay.

\subsection{Geometric form of the equations}

In order to work in a fixed domain, we want to flatten the free surface via a coordinate transformation.  We will not use a Lagrangian coordinate transformation, but rather a flattening transformation introduced by Beale in \cite{beale_2}.  To this end, we consider the fixed domain 
\begin{equation}
\Omega:= \{x \in \Sigma \times \Rn{} \; \vert\;  -b(x_1,x_2) < x_3 < 0  \}
\end{equation}
for which we will write the coordinates as $x\in \Omega$.  We will think of $\Sigma$ as the upper boundary of $\Omega$, and we will write $\Sigma_b := \{x_3 = -b(x_1,x_2)\}$ for the  lower boundary.  We continue to view $\eta$ as a function on $\Sigma \times \Rn{+}$.  We then define \begin{equation}
 \bar{\eta}:= \mathcal{P} \eta = \text{harmonic extension of }\eta \text{ into the lower half space},
\end{equation}
where $\mathcal{P} \eta$ is defined by \eqref{poisson_def_inf} when $\Sigma = \Rn{2}$ and by \eqref{poisson_def_per} when $\Sigma = (L_1 \mathbb{T}) \times (L_2 \mathbb{T})$.  The harmonic extension $\bar{\eta}$ allows us to flatten the coordinate domain via the mapping
\begin{equation}\label{mapping_def}
 \Omega \ni x \mapsto   (x_1,x_2, x_3 +  \bar{\eta}(x,t)(1+ x_3/b(x_1,x_2) )) = \Phi(x,t) = (y_1,y_2,y_3) \in \Omega(t).
\end{equation}
Note that $\Phi(\Sigma,t) = \{ y_3 = \eta(y_1,y_2,t) \}$ and $\Phi(\cdot,t)\vert_{\Sigma_b} = Id_{\Sigma_b}$, i.e. $\Phi$ maps $\Sigma$ to the free surface and keeps the lower surface fixed.   We have
\begin{equation}\label{A_def}
 \nab \Phi = 
\begin{pmatrix}
 1 & 0 & 0 \\
 0 & 1 & 0 \\
 A & B & J
\end{pmatrix}
\text{ and }
 \mathcal{A} := (\nab \Phi^{-1})^T = 
\begin{pmatrix}
 1 & 0 & -A K \\
 0 & 1 & -B K \\
 0 & 0 & K
\end{pmatrix}
\end{equation}
for 
\begin{equation}\label{ABJ_def}
\begin{split}
A &= \p_1 \bar{\eta} \tilde{b} -( x_3 \bar{\eta} \p_1 b )/b^2,\;\;\;  B = \p_2 \bar{\eta} \tilde{b} -( x_3 \bar{\eta} \p_2 b )/b^2,  \\ 
J &=  1+ \bar{\eta}/b + \p_3 \bar{\eta} \tilde{b},  \;\;\; K = J^{-1}, \\ 
\tilde{b}  &= (1+x_3/b).  
\end{split}
\end{equation} 
Here $J = \det{\nab \Phi}$ is the Jacobian of the coordinate transformation.

If $\eta$ is sufficiently small (in an appropriate Sobolev space), then the mapping $\Phi$ is a $C^1$ diffeomorphism.  This allows us to transform the problem to one on the fixed spatial domain $\Omega$ for $t \ge 0$.  In the new coordinates, the PDE \eqref{ns_euler} becomes
\begin{equation}\label{geometric}
 \begin{cases}
  \dt u - \dt \bar{\eta} \tilde{b} K \p_3 u + u \cdot \naba u -\da u + \naba p     =0 & \text{in } \Omega \\
 \diva u = 0 & \text{in }\Omega \\
 S_\a(p,u) \n = \eta \n & \text{on } \Sigma \\
 \dt \eta = u \cdot \n & \text{on } \Sigma \\
 u = 0 & \text{on } \Sigma_b \\
 u(x,0) = u_0(x), \eta(x',0) = \eta_0(x').
 \end{cases}
\end{equation}
Here we have written the differential operators $\naba$, $\diva$, and $\da$ with their actions given by $(\naba f)_i := \a_{ij} \p_j f$, $\diva X := \a_{ij}\p_j X_i$, and $\da f = \diva \naba f$ for appropriate $f$ and $X$; for $u\cdot \naba u$ we mean $(u \cdot \naba u)_i := u_j \a_{jk} \p_k u_i$.  We have also written  $\n := -\p_1 \eta e_1 - \p_2 \eta_2 e_2 + e_3$ for the non-unit normal to $\Sigma$,  and we write $\Sa(p,u)  = (p I  - \sg_{\a} u)$ for the stress tensor, where
$I$ the $3 \times 3$ identity matrix and $(\sg_{\a} u)_{ij} = \a_{ik} \p_k u_j + \a_{jk} \p_k u_i$ is the symmetric $\a-$gradient.  Note that if we extend $\diva$ to act on symmetric tensors in the natural way, then $\diva \Sa(p,u) = \naba p - \da u$ for vector fields satisfying $\diva u=0$.

Recall that $\a$ is determined by $\eta$ through the relation \eqref{A_def}.  This means that all of the differential operators in \eqref{geometric} are connected to $\eta$, and hence to the geometry of the free surface.  This geometric structure is essential to our analysis, as it allows us to control high-order derivatives that would otherwise be out of reach.

\section{Main results}

\subsection{Local well-posedness}

The standard method for constructing solutions in the existing literature is based on the parabolic regularity theory pioneered by Beale \cite{beale_1} for domains like ours and by  Solonnikov \cite{solon} for bounded, non-periodic domains. The advantage of full parabolic regularity is that it enables one to treat viscous surface waves as a perturbation of the ``flat surface'' problem, which is obtained by setting $\eta=0$, $\a = I$, $\n = e_3$, etc in \eqref{geometric}.  The actual problem \eqref{geometric} is then rewritten as the flat surface problem with nonlinear forcing terms that correspond to the difference between the two forms of the equations.  The key to the existence theory of, say \cite{beale_1}, is regularity in $H^r$ with the choice of $r =3 + \delta$ for $\delta \in (0,1/2)$ (see the discussion in Section \ref{beale_non_decay}).  According to the natural energy structure of the problem, \eqref{natural_energy}, one might expect $r$ to naturally be an integer.  The extra gain of $\delta>0$ regularity allows for enough control of the nonlinear forcing terms to produce a local  solution to  \eqref{geometric}  from solutions to the flat surface problem and an iteration argument.   As recognized early on by Beale himself, a disadvantage of Beale-Solonnikov theory is that the function spaces $K^r$, defined by \eqref{beale_kr}, involve time integration, which makes it difficult to extract time decay information.

Our a priori estimates, which produce our decay results, are developed through the energy method for high derivatives.  This necessitates using the natural energy structure of the problem, \eqref{natural_energy}, which in turn requires us to use positive integer Sobolev indices for $u$.  The advantage of the natural energy structure is that it produces two distinct types of estimates: roughly speaking, $L^\infty([0,T]; L^2)$ ``energy estimates''  and $L^2([0,T]; H^1)$ ``dissipation estimates.''  As we will discuss later, the interplay between the energy and the dissipation naturally leads to time decay information.  The disadvantage of the energy structure is that our regularity index $r$ must be an integer, so we cannot use the $\delta >0$ gain that would allow us to treat the problem \eqref{geometric} as a perturbation of the flat surface problem.

The difficulty in proving local well-posedness in the natural energy structure is thus clear.  We cannot use solutions to the standard flat surface problem to produce solutions to \eqref{geometric} via an iteration argument since the forcing terms cannot be controlled in the iteration.  For example, we would have trouble controlling the interaction between the  highest order temporal derivatives of $p$ and $\diverge{u}$.  Our solution, then, is to abandon the flat surface problem and prove local existence directly, using the geometric structure of \eqref{geometric}.  The geometric structure is crucial since it decreases the derivative count of the forcing terms, which then allows us to close an iteration argument using only the natural energy structure.  The essential difficulty is that the geometric structure requires us to solve the Navier-Stokes equations in moving domains.
In the presence of such a time-dependent geometric effect, even the construction of local-in-time solutions to the linear Navier-Stokes equations is highly delicate and has to be carried out from the beginning.  

Before we state our local existence result, let us mention the issue of compatibility conditions for the initial data $(u_0,\eta_0)$.  We will work in a high-regularity context, essentially with regularity up to $2N$ temporal derivatives for $N \ge 3$ an integer.  This requires us to use $u_0$ and $\eta_0$ to construct the initial data $\dt^j u(0)$ and $\dt^j \eta(0)$ for $j=1,\dotsc,2N$ and $\dt^j p(0)$ for $j = 0,\dotsc, 2N-1$.  These other data must then satisfy various conditions (essentially what one gets by applying $\dt^j$ to \eqref{geometric} and then setting $t=0$), which in turn require $u_0$ and $\eta_0$ to satisfy $2N$ compatibility conditions.  We describe these conditions in detail in Section \ref{l_data_section} and state them explicitly in \eqref{l_comp_cond_2N}, so for brevity we will not state them here.

In order to state our result, we must explain our notation for Sobolev spaces and norms.  We take $H^k(\Omega)$ and $H^k(\Sigma)$ for $k\ge 0$  to be the usual Sobolev spaces.  When we write norms we will suppress the $H$ and $\Omega$ or $\Sigma$.  When we write  $\norm{\dt^j u}_{k}$ and $\norm{\dt^j p}_{k}$ we always mean that the space is $H^k(\Omega)$, and when we write $\norm{\dt^j \eta}_{k}$ we always mean that the space is $H^k(\Sigma)$.  In the following result we write $H^{-1}(\Omega) = (\H1)^*$, where $\H1$ is defined later in \eqref{function_spaces}.

\begin{thm}\label{intro_lwp}
Let $N \ge 3$ be an integer.  Assume that $u_0$ and $\eta_0$ satisfy the bounds $\ns{u_0}_{4N} + \ns{\eta_0}_{4N+1/2} < \infty$ as well as the $(2N)^{th}$ compatibility conditions  \eqref{l_comp_cond_2N}.    There exist  $0<\delta_0, T_0 <1$ so that if  
\begin{equation} 
0 < T \le T_0  \min\left\{1, \frac{1}{\ns{\eta_0}_{4N+1/2}}  \right\},
\end{equation}
and $\ns{u_0}_{4N} +  \ns{\eta_0}_{4N} \le \delta_0,$ then there exists a unique solution $(u,p,\eta)$ to \eqref{geometric} on the interval $[0,T]$ that achieves the initial data.  The solution obeys the estimates 
\begin{multline}\label{intro_lwp_1}
\sum_{j=0}^{2N} \sup_{0\le t\le T} \ns{\dt^j u}_{4N-2j} + 
\sum_{j=0}^{2N} \sup_{0\le t \le T} \ns{\dt^j \eta}_{4N-2j} + 
\sum_{j=0}^{2N-1} \sup_{0\le t\le T} \ns{\dt^j p}_{4N-2j-1} \\ +
\int_0^T \left( \sum_{j=0}^{2N+1} \ns{\dt^j u}_{4N-2j +1} + \sum_{j=0}^{2N} \ns{\dt^j p}_{4N-2j} \right) \\
+ \int_0^T \left( \ns{\eta}_{4N+1/2} + \ns{\dt \eta}_{4N-1/2} +  \sum_{j=2}^{2N+1} \ns{\dt^j \eta}_{ 4N-2j+5/2} \right) \\
\le C\left( \ns{u_0}_{4N} + \ns{\eta_0}_{4N} + T \ns{\eta_0}_{4N+1/2} \right)
\end{multline}
and
\begin{equation}\label{intro_lwp_2}
 \sup_{0\le t \le T} \ns{\eta}_{4N+1/2} \le C \left(   \ns{u_0}_{4N}  + (1+T) \ns{\eta_0}_{4N+1/2} \right) 
\end{equation}
for a universal constant $C>0$.  The solution is unique among functions that achieve the initial data and for which the sum of the first three sums in \eqref{intro_lwp_1} is finite.  Moroever, $\eta$ is such that the mapping $\Phi(\cdot,t)$, defined by \eqref{mapping_def}, is a $C^1$ diffeomorphism for each $t \in [0,T]$. 
\end{thm}

The proof of Theorem \ref{intro_lwp} is carried out in Chapter \ref{section_lwp}.  We will sketch here the main ideas of the proof.

{\bf Linear $\a-$Navier-Stokes  }  

Our iteration procedure is based on a geometric variant of the linear Navier-Stokes problem.  We consider $\eta$ (and hence $\a, \n$, etc) as given and then solve the linear $\a-$Navier-Stokes equations for $(u,p)$:
\begin{equation}\label{intro_A_NS}
 \begin{cases}
\dt u - \da u + \naba p = F^1 & \text{in }\Omega \\
\diva{u}=0 & \text{in }\Omega\\
\Sa(p,u) \n = F^3 & \text{on }\Sigma \\
u =0 & \text{on }\Sigma_b,
 \end{cases}
\end{equation}
with initial data $u_0$.  Transforming this problem back to a moving domain $\Omega(t)$ using the mapping $\Phi$ defined in \eqref{mapping_def} shows that this problem is essentially equivalent (we have absorbed the correction to the time derivative into $F^1$, so it does not transform exactly) to solving the linear Navier-Stokes equations in a domain whose upper boundary is given by $\eta(t)$.  In other words, we are really solving the usual linear problem in a moving domain.

{\bf Pressure as a Lagrange multiplier in time-dependent function spaces  }

It is well-known \cite{solonnikov_skadilov,beale_1, coutand_shkoller_2, coutand_shkoller_1} that for the usual linear Navier-Stokes equations, the pressure can be viewed as a Lagrange multiplier that arises by restricting the dynamics to the class of vectors satisfying $\diverge u=0$.  To adapt this idea to the problem  \eqref{intro_A_NS}, we must restrict to the class of vectors satisfying $\diva u=0$, which is a time-dependent condition since $\eta$ (and hence $\a$) depends on $t$.  This leads us to build time-dependent variants of the usual Sobolev spaces $H^0=L^2$ and $H^1$ so that we can make sense of this time-dependent collection of $\diva-$free vectors.  For the purposes of estimates, we want the  time-dependent norms on these spaces to all be comparable to the usual Sobolev norms; this can be achieved through a smallness assumption on $\eta$, which we quantify.  With the spaces in hand, we then adapt a technique from \cite{solonnikov_skadilov} to introduce the pressure as a Lagrange multiplier for $\diva-$free dynamics.

{\bf Elliptic estimates for $\a-$problems }

In order to get the regularity we need for solutions to the parabolic problem \eqref{intro_A_NS}, we first need the corresponding elliptic regularity theory.  We accomplish this by using \eqref{mapping_def} to transform these elliptic problems back into Eulerian coordinates so that the PDEs transform to ones with constant coefficients.  We then apply standard estimates for elliptic equations and systems, proved in \cite{adn_1,adn_2}, and then transform these estimates on the Eulerian domain back to estimates on $\Omega$.  The only problem with this process is that the Eulerian domain has a boundary whose regularity is dictated by $\eta$ and is phrased in $H^k$ norms rather than $C^k$ norms, which are what appear in \cite{adn_1,adn_2}.  We get around this problem by using a smoothing operator, a limiting argument, and the smallness of $\eta$.

{\bf Galerkin method with a time-dependent basis}

We construct solutions to \eqref{intro_A_NS} by using a time-dependent Galerkin method.  This requires a countable basis of our space of $\diva-$free vector fields.   Since the requirement $\diva u=0$ is time-dependent, any basis of this space must also be time-dependent.  For each $t \in [0,T]$, the space we work in (basically $H^2$ with $\diva u=0$) is separable, so the existence of a countable basis is not an issue.  The technical difficulty is that, in order for the basis to be useful in the Galerkin method, we must be able to differentiate the basis elements in time, and we must be able to express these time derivatives in terms of finitely many basis elements.  Fortunately, due to a clever observation in \cite{beale_2}, we are able construct an explicit time-dependent isomorphism that maps the $\diverge-$free vector fields to the $\diva-$free fields.  This allows us to construct the desired basis and push through the Galerkin method to produce ``pressureless'' weak solutions that are restricted to the collection of $\diva-$free fields.  We then use our previous analysis to introduce the pressure as a Lagrange multiplier, which gives a weak solution to \eqref{intro_A_NS}.  We also use the Galerkin scheme to get higher regularity, showing that the solution is actually strong.  The compatibility conditions serve as necessary condition for controlling the temporal derivatives of the approximate solutions in the Galerkin scheme.  The result of our strong existence theorem then allows us to iteratively deduce higher regularity, given that the forcing terms are more regular and higher-order compatibility conditions are satisfied.

{\bf Transport estimates}

The problem \eqref{intro_A_NS} considers $\eta$ as given and then produces $(u,p)$.  The second step in our iteration procedure is to take $u$ as given and then solve $\dt \eta + u_1 \p_1 \eta + u_2 \p_2 \eta = u_3$ on $\Sigma$.  This is a standard transport equation, so solving it presents no real obstacle.  The difficulty is that in our analysis of \eqref{intro_A_NS}, we need control of $\sup_{0\le t\le T} \ns{\eta(t)}_{4N+1/2}$, but owing to the transport structure, the only available estimate is, roughly speaking, 
\begin{equation}\label{intro_F_est}
 \sup_{0\le t \le T} \ns{\eta}_{4N+1/2} \le C \exp\left(  C \int_0^T \snormspace{D u(t)}{2}{\Sigma} dt \right)
\left[ \ns{\eta_0}_{4N+1/2}  + T \int_0^T \ns{u(t)}_{4N+1} dt \right].
\end{equation}
Without knowing a priori that $u$ decays, the right side of this estimate has the potential to grow at the rate of $(1+T) e^{\sqrt{T}}$.  Even if $u$ decays rapidly, the right side can still grow like $(1+T)$.  Of course, such a growth in time is disastrous for stability analysis, but even in our local-existence iteration scheme, a delicate technique is required to accommodate such a growth without breaking the estimates of Theorem \ref{intro_lwp}.

{\bf Closing the iteration with a two-tier energy scheme}

Our iteration scheme then proceeds as described, using $\eta^m$ to produce $(u^{m+1},p^{m+1})$, and then using $u^{m+1}$ to produce $\eta^{m+1}$.  Iterating in this manner without losing control of our high-order energy estimates is rather delicate, and can only be completed by using sufficiently small initial data.  The boundedness of the infinite sequence $(u^m,p^m,\eta^m)$ in our high-order norms gives weak limits in the usual way, but because of the nature of our iteration scheme, we cannot guarantee a priori that the weak limits constitute a solution to \eqref{geometric}.  Instead of using high-order weak limits, we instead show that the sequence contracts in low-order norms, yielding strong convergence in low norms.  The contraction argument gives a first glimpse of the utility of our two-tier energy method: the boundedness of the high norms allows us to close the contraction estimate for the low norms.   We then combine the low-order strong convergence with the high-order weak convergence and an interpolation argument to deduce strong convergence in higher (but not all the way to the highest order) norms, which then suffices for passing to the limit $m\to \infty$ to produce a solution to \eqref{geometric}.

\subsection{Global well-posedness and decay in the infinite, flat bottom case}

Sylvester \cite{sylvester} and Tani and Tanaka \cite{tani_tanaka} studied the existence of small-data global-in-time solutions via the Beale-Solonnikov parabolic regularity method.  The results say nothing about decay of the free surface, nor do they contradict  Beale's non-decay theorem since they require higher regularity and more compatibility conditions.

To state our global well-posedness result, we must first define various energies and dissipations.  These will be somewhat different than those used in the periodic case.  Also, the exact form of some of the energies is too complicated to write out here, so we will neglect to do so, referring to the proper definitions later in the paper, in Chapter \ref{section_inf}.  We assume that $\lambda \in (0,1)$ is a fixed constant and we define  $\il u$ according to \eqref{il_def_1}  and $\il \eta$ according to \eqref{il_def_2}.  The high-order energy is
\begin{equation}
 \se{10} := 
\ns{\i_{\lambda} u}_{0}+  \sum_{j=0}^{10}  \ns{\dt^j  u}_{20-2j}       
+  \sum_{j=0}^{9} \ns{\dt^j  p}_{19-2j}   
+ \ns{\i_{\lambda} \eta}_{0} +  \sum_{j=0}^{10} \ns{\dt^j \eta}_{20-2j},
\end{equation}
and  the high-order dissipation rate  is 
\begin{multline}
 \sd{10} := 
\ns{\i_{\lambda} u}_{1} + \sum_{j=0}^{10}  \ns{\dt^j  u}_{21-2j}   
+ \ns{ \nab  p  }_{19}    
+  \sum_{j=1}^{9} \ns{\dt^j  p}_{20-2j}   \\
+ \ns{D \eta}_{20-3/2} + \ns{\dt \eta}_{20-1/2} +\sum_{j=2}^{11} \ns{\dt^j \eta}_{20-2j+ 5/2}
\end{multline}
We define the low-order energies $\se{7,1}$ and $\se{7,2}$ according to \eqref{i_energy_min_1} and \eqref{i_energy_min_2} with $n=7$.  Here the index $m$ in $\se{7,m}$ is a ``minimal derivative'' count that is included in order to improve decay rates in our estimates.  We write $\mathcal{F}_{10} :=   \ns{\eta}_{20+1/2}.$  Finally, we define the total energy
\begin{equation}
 \mathcal{G}_{10}(t) = \sup_{0 \le r \le t} \se{10}(r) + \int_0^t \sd{10}(r) dr + \sum_{m=1}^2 \sup_{0 \le r \le t} (1+r)^{m+\lambda} \se{7,m}(r) 
+  \sup_{0\le r \le t} \frac{ \mathcal{F}_{10}(r)}{(1+r)}.
\end{equation}
Notice that the low-order terms $\se{7,m}$ are weighted, so bounds on $\mathcal{G}_{10}$ yield decay estimates for $\se{7,m}$.

\begin{thm}\label{intro_inf_gwp}
Suppose the initial data $(u_0,\eta_0)$ satisfy the compatibility conditions of Theorem \ref{intro_lwp}. There exists a $\kappa >0$ so that if $\se{10}(0) + \mathcal{F}_{10}(0) < \kappa$, then there exists a unique solution $(u,p,\eta)$ on the interval $[0,\infty)$ that achieves the initial data.  The solution obeys the estimate
\begin{equation}\label{intro_inf_gwp_01}
 \mathcal{G}_{10}(\infty) \le C_1 \left( \se{10}(0) + \mathcal{F}_{10}(0) \right) < C_1 \kappa,
\end{equation}
where $C_1>0$ is a universal constant.  For any $0 \le \rho < \lambda$, we have that
\begin{equation}\label{intro_inf_gwp_02}
 \sup_{t \ge 0} \left[  (1+t)^{2+\rho} \ns{u(t)}_{C^2(\Omega)}    \right] \le C(\rho) \kappa,
\end{equation}
for $C(\rho)>0$ a constant depending on $\rho$.  Also,
\begin{equation}\label{intro_inf_gwp_03}
 \sup_{t \ge 0} \left[(1+t)^{1+\lambda} \ns{u(t)}_{2} +   (1+t)^{1+\lambda} \pns{\eta(t)}{\infty} + \sum_{j=0}^1 (1+t)^{j+\lambda} \ns{D^j \eta(t)}_{0}  \right] \le C \kappa
\end{equation}
for a universal constant $C>0$.
\end{thm}

\begin{remark}\label{intro_inf_rem_0}
The bound $\se{10}(0) < \kappa$ requires, in particular, that the initial data satisfy $\ns{\il u_0}_{0} < \infty$  and  $\ns{\il \eta_0}_{0} < \infty$.  The latter condition can be viewed as a sort of weak zero average condition in the infinite case, which serves as the analog to the zero average condition in the periodic case, \eqref{z_avg}.  To see this, note that if $\eta_0$ is sufficiently nice, say $L^1(\Sigma)$, then
\begin{equation}
0= \int_\Sigma \eta_0  \Leftrightarrow \hat{\eta}_0(0) = 0,
\end{equation}
for $\hat{\cdot}$ the Fourier transform.  This means that the zero average condition is equivalent to requiring that $\hat{\eta}_0$ vanishes at the origin.   We enforce a weak version of this by requiring that $\il \eta_0 \in L^2(\Sigma)=H^0(\Sigma)$, which requires that $\abs{\xi}^{-2\lambda} \abs{\hat{\eta}_0(\xi)}^2$ is integrable near $\xi=0$.  Since $\lambda <1$, this does not require $\hat{\eta}_0(0)=0$, but it does prevent $\abs{\hat{\eta}_0}$ from being ``too big'' at the origin.

\end{remark}

\begin{remark}\label{intro_inf_rem_1}
The decay estimates \eqref{intro_inf_gwp_02} and \eqref{intro_inf_gwp_03} do not follow directly from the decay of $\se{7,2}(t)$ implied by  \eqref{intro_inf_gwp_01}.  Rather, they are deduced via auxiliary arguments, employing \eqref{intro_inf_gwp_01}.
\end{remark}

\begin{remark}\label{intro_inf_rem_2}
The decay of $\ns{u(t)}_{2}$ given in \eqref{intro_inf_gwp_03} is not fast enough to guarantee that $u \in L^1([0,\infty);H^2(\Omega))$.  Even if we could take $\lambda =1$, we would still get logarithmic blow-up of the $L^1 H^2$ norm.
\end{remark}

\begin{remark}\label{intro_inf_rem_3}
The surface $\eta$ is sufficiently small to guarantee that the mapping $\Phi(\cdot,t)$, defined in \eqref{mapping_def}, is a diffeomorphism for each $t\ge 0$.  As such, we may change coordinates to $y \in \Omega(t)$ to produce a global-in-time, decaying solution to \eqref{ns_euler} in the non-periodic case.
\end{remark}

\begin{remark}\label{intro_inf_rem_4}
Later in the paper, we perform our analysis in terms of estimates at the $2N$ and $N+2$ levels; we take $N=5$ in the present case to get the $10$ and $7$ appearing above.  This is not optimal.  With somewhat more work, we can improve our results to $N=4$ with the restriction that $\lambda \in (3/5,1)$.  It is likely that this can be further improved by adjusting the scheme from $2N$ and $N+2$ to something slightly different.  We have sacrificed optimality in order to simplify the presentation and make our ``two-tier energy method'' clearer.  The first tier is at the level $2N$ and the second at the level $N+2$, which is meant to be roughly half of the first tier.  The extra $+2$ is added to aid in applying some Sobolev embeddings.  
\end{remark}

Theorem \ref{intro_inf_gwp} is proved in Chapter \ref{section_inf}.  We now present a sketch of the key ideas.

{ \bf Horizontal energy evolution estimates }

In order to use the natural energy structure of the problem (given in Eulerian coordinates by \eqref{natural_energy}) to study high-order derivatives,  we can only apply derivatives that do not break the structure of the boundary condition $u=0$ on $\Sigma_b$.  Since $\Sigma_b$ is flat, any differential operator $\pa = \dt^{\alpha_0} \p_1^{\alpha_1} \p_2^{\alpha_2}$ is allowed.  We apply these operators for various choices of $\alpha$ and sum the resulting energy evolution equations.  After estimating the nonlinear terms that appear from differentiating \eqref{geometric}, we are eventually led to evolution equations for these ``horizontal'' energies and dissipations, $\seb{10}$, $\sdb{10}$, $\seb{7,m}$, and $\sdb{7,m}$ for $m=1,2$ 
(see Chapter \ref{section_inf} for precise definitions).  Here we write bars to indicate ``horizontal'' derivatives.  Roughly speaking, these read
\begin{equation}\label{intro_inf_1}
 \seb{10}(t) + \int_0^t \sdb{10}(r) dr \ls \se{10}(0) + \int_0^t (\se{10}(r))^\theta \sd{10}(r) dr + \int_0^t \sqrt{\sd{10}(r) \k(r) \mathcal{F}_{10}(r)} dr
\end{equation}
and
\begin{equation}\label{intro_inf_2}
 \dt \seb{7,m} + \sdb{7,m} \ls \se{10}^\theta \sd{7,m},
\end{equation}
where $\theta>0$, $\sd{7,m}$ is the low-order dissipation, and $\k$ is of the form  $\k = \ns{\nab u}_{C^1} + \snormspace{D u}{2}{\Sigma}^2.$  Notice that the product $\k \mathcal{F}_{10}$  in \eqref{intro_inf_1} multiplies low-order norms of $u$ against the highest-order norm of $\eta$.  Technically, the estimate \eqref{intro_inf_1} also involves $\il u$ and $\il \eta$ in addition to horizontal derivatives.  For the moment let us ignore these terms and continue with the discussion of our energy method.  We will discuss $\il$ in detail below.

The actual derivation of bounds like \eqref{intro_inf_1}--\eqref{intro_inf_2} is rather delicate and depends crucially on the geometric structure of the equations given in \eqref{geometric}.  Indeed, if we attempted to use the ``flat surface'' perturbation form of the equations, we would fail to achieve the estimate \eqref{intro_inf_1} due to problems with the highest-order temporal derivatives.  We note that all of the computations that lead to these estimates are justified by the boundedness of the terms in \eqref{intro_lwp_1}--\eqref{intro_lwp_2} of Theorem \ref{intro_lwp}.

{\bf  Comparison estimates  }

The next step in the analysis is to replace the horizontal energies and dissipations with the full energies and dissipations.  We prove that there is a universal $0 < \delta <1$ so that if $\se{10} \le \delta,$ then
\begin{equation}\label{intro_inf_3}
 \begin{split}
\se{10} &\ls \seb{10},\;\;\;  \sd{10} \ls \sdb{10}  +   \k \mathcal{F}_{10}, \\
\se{7,m} &\ls \seb{7,m},\;\;\;  \sd{7,m} \ls \sdb{7,m}  
 \end{split}
\end{equation}
This estimate is extremely delicate and can only be obtained by carefully using the structure of the equations.  We make use every bit of information from the boundary conditions and the vorticity equations to establish it.  There are two structural components of the estimates that are of such importance that we mention them now.  First, the equation $\diva u=0$ allows us to write $\p_3 u_3 = -(\p_1 u_1 + \p_2 u_2) + G^2$ for some quadratic nonlinearity $G^2$.  This allows us to ``trade'' a vertical derivative of $u_3$ for horizontal derivatives of $u_1$ and $u_2$,  an indispensable trick in our analysis.  Second, the interaction between the parabolic scaling of $u$ ($\dt u \sim \Delta u$) and the transport scaling of $\eta$ ($\dt \eta \sim u_3 \vert_{\Sigma}$) allows us to gain  regularity for the temporal derivatives of $\eta$ in the dissipation, and it also gives us control of $\dt^{11} \eta$, which is one more time derivative than appears in the energy.

{ \bf Two-tier energy method}

Suppose we know that 
\begin{equation}\label{intro_inf_4}
 \k(r) \le \frac{\delta}{(1+r)^{2+\gamma}}
\end{equation}
for some  $0<\delta<1$ and $\gamma>0$.  We know from the transport estimate \eqref{intro_F_est} that  we can then expect an estimate of the form 
\begin{equation}\label{intro_inf_5}
 \sup_{0\le r \le t} \mathcal{F}_{10}(r) \ls \mathcal{F}_{10}(0) + t \int_0^t \sd{10}(r) dr. 
\end{equation}
Note that $\gamma >0$ in \eqref{intro_inf_4} is essential; we would not be able to tame the exponential term in \eqref{intro_F_est} without it, and then \eqref{intro_inf_5} would not hold. This estimate allows for $\mathcal{F}_{10}(t)$ to grow linearly in time, but in the product $\k(r) \mathcal{F}_{10}(r)$ that appears in \eqref{intro_inf_1}, we can use the decay of $\k$ to balance this growth.  Then if $\sup_{0\le r \le t} \se{10}(r) \le \delta$ with $\delta$  small enough, we can combine \eqref{intro_inf_1}, \eqref{intro_inf_3}, \eqref{intro_inf_4}, and \eqref{intro_inf_5} to get an estimate 
\begin{equation}\label{intro_inf_6}
  \se{10}(t) + \int_0^t \sd{10}(r) dr \ls \se{10}(0) +\mathcal{F}_{10}(0).
\end{equation}
This highlights the first step of our two-tier energy method: the decay of low-order terms (i.e. $\k$) can balance the growth of $\mathcal{F}_{10}$, yielding boundedness of the high-order terms.  In order to close this argument, we must use a second step: the boundedness of the high-order terms implies the decay of low-order terms, and in particular the decay of $\k$.

To get at this decay, we combine \eqref{intro_inf_2} and \eqref{intro_inf_3} to see that
\begin{equation}\label{intro_inf_7}
  \dt \seb{7,m} + \hal \sd{7,m} \le 0
\end{equation}
if $\se{10} \le \delta$ for $\delta$ small enough.  If we could show that $\seb{7,m} \ls \sd{7,m}$, then this estimate would yield exponential decay of $\seb{7,m}$ and $\se{7,m}$.  An inspection of $\seb{7,m}$ and $\sd{7,m}$ (see the beginning of Chapter \ref{section_inf}) shows that $\sd{7,m}$ can control every term in $\seb{7,m}$ except $\ns{\eta}_{0}$ (and $\ns{\dt \eta}_{0}$ when $m=2$).  In a sense, this means that exponential decay fails precisely because the dissipation fails to control $\eta$ at the lowest order.  In lieu of $\seb{7,m} \ls \sd{7,m}$, we instead interpolate between $\se{10}$ (which can control all the lowest-order terms of $\eta$) and $\sd{7,m}$: 
\begin{equation}\label{intro_inf_8}
 \seb{7,m} \ls \se{10}^{1/(m+\lambda+1)}  \sd{7,m}^{(m+\lambda)/(m+\lambda+1)}.
\end{equation}
Combining \eqref{intro_inf_7} with \eqref{intro_inf_8} and  the boundedness of $\se{10}$ in terms of the data, \eqref{intro_inf_6},  then allows us to deduce that 
\begin{equation}
\dt \seb{7,m} + \frac{C}{(\se{10}(0) + \mathcal{F}_{10}(0))^{1/(m+\lambda)}} (\seb{7,m})^{1 + 1/(m+\lambda)} \le 0.
\end{equation}
Gronwall's inequality (along with some auxiliary estimates) then leads us to  the bound
\begin{equation}\label{intro_inf_9}
 \se{7,m}(t) \ls \seb{7,m}(t) \ls \frac{\se{10}(0) + \mathcal{F}_{10}(0)}{(1+t)^{m+\lambda}}.
\end{equation}
We thus use the boundedness of high-order terms to deduce the decay of low-order terms, completing the second step of the two-tier energy estimates.

{\bf  Negative Sobolev estimates via $\il$ }

Notice that the decay rate in \eqref{intro_inf_9} is enhanced by $\lambda \in (0,1)$.  As we will see below, the parameter $\gamma>0$ in the decay of $\k$, given in \eqref{intro_inf_4}, is  determined by the rate $m + \lambda$.  If  we took $\lambda =0$, then we would not get $\gamma>0$, and we would be unable to balance the growth of $\mathcal{F}_{10}$.  Then estimates \eqref{intro_inf_5} and \eqref{intro_inf_6} would fail, and we would be unable to close our estimates.  We thus see the necessity of introducing the ``negative Sobolev'' estimates via the horizontal Riesz potential $\il$.

The difficulty, then, is that we must apply the non-local operator $\il$ to a nonlinear PDE and then study the evolution of $\il u$ and $\il \eta$.  The flatness of the lower boundary $\Sigma_b$ is essential here since it allows us to have $\il u =0$ on $\Sigma_b$.  This means that the operator $\il$ does not break the boundary conditions, and we can use the natural energy structure to include $\ns{\il u}_{0}$ and $\ns{\il \eta}_0$ in the energy  and $\ns{\il u}_{1}$ in the dissipation.  To close the estimates for these terms, we must be able to estimate $\il$ acting on various nonlinearities in terms of $\se{10}^\theta \sd{10}$ for some $\theta>0$.  These estimates turn out to be rather delicate, and we must again employ almost all of the structure of the equations and boundary conditions in order to derive them.  They are also responsible for the constraint $\lambda <1$.  For $\lambda \ge 1$, the nonlinear estimates would not work as we need them to.

We should point out that a priori, we do not know that $\il u(t)$ or $\il \eta(t)$ even make sense for $t>0$, since this is not provided by Theorem \ref{intro_lwp}.  To show that these terms are well-defined, which then justifies applying $\il$ to the equations, we must actually prove a specialization of the local well-posedness theorem that includes the boundedness of $\il u$, $\il p$, and $\il \eta$.

{\bf Interpolation estimates and minimal derivative counts}

The negative Sobolev estimates alone do not close the overall estimates in our two-tier energy method.  To do that, we must verify that $\k$ decays as in \eqref{intro_inf_4} for some $\gamma>0$.  An inspection of $\se{7,m}$ shows that we cannot directly control $\k \ls \se{7,m}$ for either $m=1,2$, so we must resort to an interpolation argument.  We show that through interpolation it is actually possible to control $\k \ls \se{7,1}$, but the $\se{7,1}$ only decays like  $(1+t)^{-1-\lambda}$, which is not fast enough for \eqref{intro_inf_4}.  The energy $\se{7,2}$ decays at a faster rate, but we cannot show that $\k \ls \se{7,2}$.  Instead, we show that if $\se{7,2}(t) \le \ep (1+t)^{-2-\lambda}$, then
\begin{equation}\label{intro_inf_10}
 \k \ls \se{7,2}^{(8+2\lambda)/(8+4\lambda)} \ls \ep^{(8+2\lambda)/(8+4\lambda)} \frac{1}{(1+t)^{2+\lambda/2}},
\end{equation}
so that after renaming $\delta=C \ep^{(8+2\lambda)/(8+4\lambda)}$ and $\gamma = \lambda/2 >0$ we find that \eqref{intro_inf_4} does hold.  

The parameters $m$ and $\lambda$ interact in an important way.    The decay rate increases with $m$ and with $\lambda$.  As mentioned above, we are technically constrained to $\lambda < 1$, so we must increase $m$ to $2$ in order to hit the target decay rate in \eqref{intro_inf_4}.  It is tempting, then, to consider abandoning the $\il$ operators and simply use a third energy with $m\ge 3$, which should decay like $(1+t)^{-m}$.  However, if one were to do this for any $m \ge 3$, one would find that there is a corresponding decrease in the interpolation power: $\k \ls \se{7,m}^{\theta(m)}$, where $\theta(m)$ decreases with $m$ in such a way that $m \theta(m) \le 2$ so that \eqref{intro_inf_4} would fail.  We thus see that the negative estimates are not just a convenience, but rather a necessity.

The derivation of \eqref{intro_inf_10} is delicate, requiring a two-step bootstrap process to iteratively improve the interpolation powers.  We again crucially make use of the structure of the equations and boundary conditions.  We extensively interpolate between our negative Sobolev estimates and our positive Sobolev estimates.  The utility of the negative estimates is quite clear here: the interpolation powers improve when we interpolate with negative derivatives (as opposed to say, no derivatives).  

To complete the proof of \eqref{intro_inf_10}, we crucially use an estimate for $\i_{1} \dt \eta$. This corresponds to $\lambda=1$, so we are not able to apply $\i_{1} \dt$ to the equations to get at the estimate.  Rather, the estimate comes for free from the transport equation for $\eta$, which allows us to write $\dt \eta = -\p_1 U_1 -\p_2 U_2$ for $U_i \in H^1$.  This computation is valid also in the periodic case and gives a second proof of \eqref{avg_prop}, which in turn gives rise to a Poincar\'e inequality $\ns{\eta}_{0} \ls \ns{D \eta}_{0}$ on $\Sigma = (L_1 \mathbb{T})\times (L_2 \mathbb{T})$.  From this we see that the estimate for $\i_{1} \dt \eta$ can be viewed as a sort of substitute for a Poincar\'e inequality on $\Sigma = \Rn{2}$, which is unavailable in general.

The interpolation of negative and positive Sobolev estimates provides a completely new tool in the study of time decay of dissipative PDE problems in the whole  (or semi-infinite) space.  Our estimate is new even for the simple heat equation.   A particular advantage of the negative-positive method is that, unlike the usual $L^{p}-L^{q}$ machinery, our norms are preserved along time evolution.  We anticipate that this method will prove useful in the analysis of other dissipative equations.

\subsection{Global well-posedness and decay in the periodic, curved bottom case}

In \cite{hataya}, Hataya studied the periodic problem with a flat bottom, $b(x') = b \in (0,\infty)$.  Using the Beale-Solonnikov parabolic theory, it was shown that  if $\eta _{0}$ has zero average, \eqref{z_avg}, then 
\begin{equation}
\int_{0}^{\infty }(1+t)^{2}\ns{u(t)}_{r-1} dt + \sup_{t \ge 0 }(1+t)^{2} \ns{\eta(t)}_{r-2} < \infty
\end{equation}
for $r\in (5,11/2)$. Our result on the periodic problem is an improvement of this in two important ways. First, we allow for a more general non-flat bottom geometry. Second, we establish faster decay rates by working in a higher regularity context.

To state our result, we must first define our energies and dissipations.  These are slightly different from the ones used in Theorem \ref{intro_inf_gwp}.  For any integer $N \ge 3$ we write the high-order energy as
\begin{equation}
 \se{2N} = \sum_{j=0}^{2N} \left( \ns{\dt^j u}_{4N-2j} + \ns{\dt^j \eta}_{4N-2j} \right) + \sum_{j=0}^{2N-1} \ns{\dt^j p}_{4N-2j-1}
\end{equation}
and the corresponding dissipation as
\begin{multline}
 \sd{2N} = \sum_{j=0}^{2N} \ns{\dt^j u}_{4N-2j+1} + \sum_{j=0}^{2N-1} \ns{\dt^j p}_{4N-2j} \\
+  \ns{ \eta}_{4N-1/2} + \ns{\dt \eta}_{4N-1/2} + \sum_{j=2}^{2N+1} \ns{\dt^j \eta}_{4N-2j+5/2}.
\end{multline}
We define the low-order energy as
\begin{equation}
 \se{N+2} = \sum_{j=0}^{N+2} \left( \ns{\dt^j u}_{2(N+2)-2j} + \ns{\dt^j \eta}_{2(N+2)-2j} \right) + \sum_{j=0}^{N+1} \ns{\dt^j p}_{2(N+2)-2j-1}.
\end{equation}
Notice that, unlike in the non-periodic case, in the periodic case we do not need to bother with either minimal derivative counts or $\il$ estimates.   We write $\f :=   \ns{\eta}_{4N+1/2}.$  Finally, we define total energy
\begin{equation}
  \g(t) = \sup_{0 \le r \le t} \se{2N}(r) + \int_0^t \sd{2N}(r) dr + \sup_{0 \le r \le t} (1+r)^{4N-8} \se{N+2}(r) + \sup_{0 \le r \le t} \frac{\f(r)}{(1+r)}.
\end{equation}
Notice that the low-order terms $\se{N+2}$ are weighted, so bounds on $\g$ imply decay estimates  $\se{N+2}(t) \ls (1+t)^{-4N+8}$.

\begin{thm}\label{intro_per_gwp}
Suppose the initial data $(u_0,\eta_0)$ satisfy the compatibility conditions of Theorem \ref{intro_lwp} and that $\eta_0$ satisfies the zero average condition \eqref{z_avg}.  Let $N \ge 3$ be an integer.  There exists a $0 < \kappa = \kappa(N)$ so that if $\se{2N}(0) + \f(0) < \kappa$, then there exists a unique solution $(u,p,\eta)$ on the interval $[0,\infty)$ that achieves the initial data.  The solution obeys the estimate
\begin{equation}
 \g(\infty) \le C_1 \left( \se{2N}(0) + \f(0) \right) < C_1 \kappa,
\end{equation}
where $C_1>0$ is a universal constant.

\end{thm}

\begin{remark}\label{intro_per_rem_1}
 The decay of $\se{N+2}(t)$ implies that
\begin{equation}
\sup_{t \ge 0 } (1+t)^{4N-8} \left[\ns{u(t)}_{2N+4}  +  \ns{\eta(t)}_{2N+4} \right] \le C_1 \kappa.
\end{equation}
Since $N$ may be taken to be arbitrarily large, this decay result can be regarded as an ``almost exponential'' decay rate.
\end{remark}

\begin{remark}\label{intro_per_rem_2}
A key difference between the periodic result, Theorem \ref{intro_per_gwp}, and the non-periodic result, Theorem \ref{intro_inf_gwp}, is that in the periodic case, increasing $N$ also increases the decay rate of $\se{N+2}(t)$.  No such gain is possible in the non-periodic case, which is why we specialize to the case $N=5$ there. In the periodic case, we do not use the same type of interpolation arguments that we use in the infinite case.  This allows us to relax to $N \ge 3$.
\end{remark}

\begin{remark}\label{intro_per_rem_3}
The surface $\eta$ is sufficiently small to guarantee that the mapping $\Phi(\cdot,t)$, defined in \eqref{mapping_def}, is a diffeomorphism for each $t\ge 0$.  As such, we may change coordinates to $y \in \Omega(t)$ to produce a global-in-time, decaying solution to \eqref{ns_euler} in the periodic case.
\end{remark}

Theorem \ref{intro_per_gwp} is proved in Chapter \ref{section_per}.  The proof follows the same basic outline that we use in the non-periodic case.  We apply horizontal derivatives and estimate their evolution equations, we prove comparison estimates that bound the full energies in terms of the horizontal energies, and we fit everything together in a two-tier energy method that couples the boundedness of $\se{2N}$ to the decay of $\se{N+2}$.   Many of the proofs in the periodic case are easier than in the non-periodic case because of some auxiliary estimates available (stemming from the Poincar\'e inequality).  Rather than reiterate this method (it can be understood by replacing $10$ with $2N$ and $7$ with $N+2$ in the sketch of the proof of Theorem \ref{intro_inf_gwp}), we will highlight the novel features of  the periodic case.

{ \bf Poincar\'e from the zero average condition  }

Owing to \eqref{avg_prop}, we know that the average of $\eta(t)$ over $\Sigma$ vanishes for all $t \ge 0$.  This allows us to utilize the standard Poincar\'e inequality on $\Sigma$ to estimate $\ns{\eta}_0 \ls \ns{D \eta}_{0}$.  This is useful because we will be able to control $\ns{D \eta}_{0}$ with the dissipation (through careful use of the boundary conditions), which means we will gain full control of $\eta$ itself.  In turn, this eliminates the need to use either minimal derivative counts at the $N+2$ level or negative Sobolev estimates via the $\il$ operator, and it paves the way for decay rates that increase with $N$.  Indeed, we show that (roughly speaking)
\begin{equation}
 \dt \se{N+2} + C \sd{N+2} \le 0 \text{ and } \se{N+2} \ls (\sd{N+2})^{(4N-8)/(4N-7)} (\se{2N})^{1/(4N-7)},
\end{equation}
which yield the decay estimate
\begin{equation}
 \se{N+2}(t) \ls \frac{\se{2N}(0) + \f(0) }{(1+t)^{4N-8}}.
\end{equation}

{ \bf Localization for the curved bottom }

We allow the lower boundary $\Sigma_b$ to be curved.  This means that spatial derivatives in the $x_1$ and $x_2$ directions are not compatible with the boundary condition $u=0$ on $\Sigma_b$.  This prohibits us from applying, say  $\p_1^k$, to the equations and studying the evolution of $\p_1^k u$ and $\p_1^k \eta$.  The only operator that does not break the boundary condition is $\dt$, which we can apply as before.  To get around this problem we introduce a localization procedure.  We localize in a horizontal strip near $\Sigma$, and in an area around $\Sigma_b$.  Near $\Sigma$ the problem behaves like a free boundary problem with a flat bottom, and we are free to apply all horizontal derivatives.  In the lower domain, near $\Sigma_b$, the problem behaves like a fixed boundary problem with curved lower boundary.  The only derivatives we can apply are temporal, but they are sufficient for controlling all derivatives because of the fixed upper boundary.  We then build our a priori estimates out of these localized energies and patch both of them together for estimates in all of $\Omega$ at the end.

The main difficulty in this procedure is that it introduces ``localization forces'' that appear because the cutoff functions we multiply by to localize do not commute with all of the differential operators.  These localization forces can only be controlled in terms of the dissipation by employing the Poincar\'e inequality for $\eta$ on $\Sigma$.  Through a careful balance of how and where we localize, we are able to control the localization forces and close our estimates.  However, if we were to attempt the same procedure in the non-periodic case, the lack of a Poincar\'e inequality would prohibit us from controlling the localization forces in terms of the dissipation, and our estimates would fail to close.

\section{Comparison to the case with surface tension}

If the effect of surface tension is included at the air-fluid free interface, then the formulation of the PDE must be changed.  Surface tension is modeled by modifying the fourth equation in \eqref{ns_euler} to be
\begin{equation}
 (p I - \mu \sg(u)) \nu = g \eta \nu - \sigma H \nu,
\end{equation}
where $H=\p_i (\p_i \eta/\sqrt{1+ \abs{D\eta}^2})$ is the mean curvature of the surface $\{y_3 = \eta(t)\}$ and $\sigma >0$ is the surface tension.

In \cite{beale_2}, Beale proved global well-posedness for the non-periodic problem with surface tension.  The flattened coordinate system we employ was introduced in \cite{beale_2} and used in place of Lagrangian coordinates.  However, Beale employed a change of unknown velocities that is more complicated than just a coordinate change. Well-posedness was demonstrated with (essentially) $u \in K^r(\Omega\times(0,\infty))$ and  $\eta \in K^{r+1/2}(\Sigma\times(0,\infty))$, given that $u_0 \in H^{r-1/2}(\Omega)$, $\eta_0 \in H^{r}(\Sigma)$ are sufficiently small for $r\in(3,7/2)$.  In this context it is understood that surface tension leads to the decay of certain modes, thereby aiding global existence.

In \cite{beale_nishida}, Beale and Nishida studied the asymptotic properties of the solutions constructed in \cite{beale_2}.  They showed that if $\eta_0 \in L^1(\Sigma)$, then
\begin{equation}
\sup_{t \ge 0} (1+t)^2 \ns{u(t)}_{2}  +  \sup_{t\ge 0} \sum_{j=1}^2 (1+t)^{1+j} \ns{D^j \eta (t)}_{0}  < \infty,
\end{equation}
and that this decay rate is optimal.  Taking $\lambda \approx 1$ in our Theorem \ref{intro_inf_gwp}, the estimates \eqref{intro_inf_gwp_03} yield almost the same decay rates.

In \cite{nishida_1}, Nishida, Teramoto, and Yoshihara showed that in the periodic case with surface tension and a flat bottom,  if $\eta_0$ has zero average, then there exists a $\gamma>0$ so that
\begin{equation}
\sup_{t\ge 0}  e^{\gamma t} \left[ \ns{u (t)}_{2} +   \ns{\eta (t)}_{3} \right] < \infty.
\end{equation}
Thus, if surface tension is added in the periodic case, fully exponential decay is possible, whereas without surface tension we only recover algebraic decay of arbitrary order in Theorem \ref{intro_per_gwp}.

The comparison of these two results with ours establishes a nice contrast between the surface tension and non-surface tension cases.  Without surface tension we can recover ``almost'' the same decay rate as in the case with surface tension.  This suggests that viscosity is the basic decay mechanism and that surface tension acts to enhance the decay.

\section{Definitions and terminology}\label{def_and_term}

We now mention some of the definitions, bits of notation, and conventions that we will use throughout the paper.

{ \bf Einstein summation and constants }

We will employ the Einstein convention of summing over  repeated indices for vector and tensor operations.  Throughout the paper $C>0$ will denote a generic constant that can depend on the parameters of the problem, $N$, and $\Omega$, but does not depend on the data, etc.  We refer to such constants as ``universal.''  They are allowed to change from one inequality to the next.  When a constant depends on a quantity $z$ we will write $C = C(z)$ to indicate this.  We will employ the notation $a \ls b$ to mean that $a \le C b$ for a universal constant $C>0$.

{ \bf Norms }

We write $H^k(\Omega)$ with $k\ge 0$ and and $H^s(\Sigma)$ with $s \in \Rn{}$ for the usual Sobolev spaces.  We will not need negative index spaces on $\Omega$ except for $H^{-1}(\Omega):= (\H1)^*$, where $\H1$ is defined later in \eqref{function_spaces}.  We will typically write $H^0 = L^2$; the exception to this is mostly in Chapter \ref{section_lwp}, where we use $L^2([0,T]; H^k)$ notation to indicate the space of square-integrable functions with values in $H^k$.

To avoid notational clutter, we will avoid writing $H^k(\Omega)$ or $H^k(\Sigma)$ in our norms and typically write only $\norm{\cdot}_{k}$.  Since we will do this for functions defined on both $\Omega$ and $\Sigma$, this presents some ambiguity.  We avoid this by adopting two conventions.  First, we assume that functions have natural spaces on which they ``live.''  For example, the functions $u,$ $p$, and $\bar{\eta}$ live on $\Omega$, while $\eta$ itself lives on $\Sigma$.  As we proceed in our analysis, we will introduce various auxiliary functions; the spaces they live on will always be clear from the context.  Second, whenever the norm of a function is computed on a space different from the one in which it lives, we will explicitly write the space.  This typically arises when computing norms of traces onto $\Sigma$ of functions that live on $\Omega$.

{ \bf Derivatives } 

We write $\mathbb{N} = \{ 0,1,2,\dotsc\}$ for the collection of non-negative integers.  When using space-time differential multi-indices, we will write $\mathbb{N}^{1+m} = \{ \alpha = (\alpha_0,\alpha_1,\dotsc,\alpha_m) \}$ to emphasize that the $0-$index term is related to temporal derivatives.  For just spatial derivatives we write $\mathbb{N}^m$.  For $\alpha \in \mathbb{N}^{1+m}$ we write $\pa = \dt^{\alpha_0} \p_1^{\alpha_1}\cdots \p_m^{\alpha_m}.$ We define the parabolic counting of such multi-indices by writing $\abs{\alpha} = 2 \alpha_0 + \alpha_1 + \cdots + \alpha_m.$  We will write $Df$ for the horizontal gradient of $f$, i.e. $Df = \p_1 f e_1 + \p_2 f e_2$, while $\nab f$ will denote the usual full gradient.

For a given norm $\norm{\cdot}$ and  integers $k,m\ge 0$, we introduce the following notation for sums of spatial derivatives:
\begin{equation}
 \norm{D_m^k f}^2 := \sum_{\substack{\alpha \in \mathbb{N}^2 \\ m\le \abs{ \alpha}\le k} } \norm{\pa  f}^2 \text{ and } 
\norm{\nab_m^k f}^2 := \sum_{\substack{\alpha \in \mathbb{N}^{3} \\  m\le \abs{\alpha}\le k} } \norm{\pa  f}^2.
\end{equation}
The convention we adopt in this notation is that $D$ refers to only ``horizontal'' spatial derivatives, while $\nab$ refers to full spatial derivatives.   For space-time derivatives we add bars to our notation:
\begin{equation}
 \norm{\dbm{k} f}^2 := \sum_{\substack{\alpha \in \mathbb{N}^{1+2} \\  m\le \abs{\alpha}\le k} } \norm{\pa  f}^2 \text{ and } 
\norm{\bar{\nab}_m^k f}^2 := \sum_{\substack{\alpha \in \mathbb{N}^{1+3} \\ m\le  \abs{\alpha}\le k} } \norm{\pa  f}^2.
\end{equation}
When $k=m\ge 0$ we will write 
\begin{equation}
 \norm{D^k f}^2 = \norm{D_k^k f}^2, \norm{\nab^k f}^2 =\norm{\nab_k^k f}^2, \norm{\bar{D}^k f}^2 = \norm{\bar{D}_k^k f}^2, \norm{\bar{\nab}^k f}^2 =\norm{\bar{\nab}_k^k f}^2.
\end{equation}
We allow for composition of derivatives in this counting scheme in a natural way; for example, we write
\begin{equation}
 \norm{D D_m^{k} f}^2 = \norm{ D_m^k D f}^2 = \sum_{\substack{\alpha \in \mathbb{N}^{2} \\  m\le \abs{\alpha}\le k} } \norm{\pa  D f}^2  = \sum_{\substack{\alpha \in \mathbb{N}^{2} \\  m+1\le \abs{\alpha}\le k+1} } \norm{\pa   f}^2.
\end{equation}

\chapter{Local well-posedness}\label{section_lwp}

\section{Introduction}

In this chapter we will prove the local well-posedness result, Theorem \ref{intro_lwp}.  Our proof employs an iteration that is based on the following linear problem
\begin{equation}\label{l_linear_forced}
 \begin{cases}
\dt u - \da u + \naba p = F^1 & \text{in }\Omega \\
\diva{u}=0 & \text{in }\Omega\\
\Sa(p,u) \n = F^3 & \text{on }\Sigma \\
u =0 & \text{on }\Sigma_b,
 \end{cases}
\end{equation}
subject to initial conditions $u(0) = u_0$.  Note that the first equation in \eqref{l_linear_forced} may be rewritten as $\dt u + \diva \Sa(p,u) = F^1$.

In Section \ref{lwp_1} we develop the machinery of time-dependent function spaces so that we can consider the class of $\diva-$free vector fields.  We use an orthogonal splitting of a space to introduce the pressure as a Lagrange multiplier.  In Section \ref{lwp_2} we record some elliptic estimates for the $\a-$Stokes problem and the $\a-$Poisson problem.  In Section \ref{lwp_3} we develop the local existence theory for \eqref{l_linear_forced} by using a time-dependent Galerkin scheme.  We iterate this result to produce high-regularity solutions.  In Section \ref{lwp_4} we do some preliminary work for the nonlinear problem, constructing initial data, detailing the compatibility conditions, and constructing solutions to the transport equation with high-regularity estimates.  In Section \ref{lwp_5} we construct solutions to \eqref{geometric} through the use of iteration and contraction arguments, completing the proof of Theorem \ref{intro_lwp}.  

Throughout the chapter we assume that $N \ge 3$ is an integer.  We consider both the non-periodic and periodic cases simultaneously.  When different analysis is needed for each case, we will indicate so.  Otherwise, the argument we write works in both cases.

\section{Functional setting}\label{lwp_1}

\subsection{Time-dependent function spaces}

We begin our analysis of \eqref{l_linear_forced} by introducing some function spaces.  We write $H^k(\Omega)$ and $H^k(\Sigma)$ for the usual $L^2$-based Sobolev spaces of either scalar or vector-valued functions.  Define
\begin{equation}\label{function_spaces}
\begin{split}
 \H1 &:= \{ u \in H^1(\Omega) \;\vert\;  u\vert_{\Sigma_b}=0\}, \\
{^0}H^1(\Omega) &:= \{u\in H^1(\Omega) \;\vert\; u\vert_{\Sigma}=0\}, \text{ and } \\
\Hsig &:= \{ u \in \H1 \;\vert\; \diverge{u}=0 \},
\end{split}
\end{equation}
with the obvious restriction that the last space is for vector-valued functions only.

For our time-dependent function spaces we will consider $\eta$ (and hence $\a$, $J$, etc) as given; in our subsequent analysis $\eta$ will always be sufficiently regular for all terms derived from $\eta$ to make sense.  We define a time-dependent inner-product on  $L^2=H^0$ by introducing 
\begin{equation}
 \iph{u}{v}{0} := \int_\Omega  ( u \cdot v)  J(t)
\end{equation}
with corresponding norm $\hn{u}{0} := \sqrt{\iph{u}{u}{0}}$.  Then we write
$\h^0(t) := \{ \hn{u}{0} < \infty \}$.  Similarly, we  define a time-dependent inner-product on $\H1$ according to
\begin{equation}
 \iph{u}{v}{1} :=  \int_\Omega \left(\sg_{\a(t)} u : \sg_{\a(t)} v \right)  J(t),
\end{equation}
and we define the corresponding norm by $\hn{u}{1} = \sqrt{\iph{u}{u}{1}}$.  Then we define 
\begin{equation}
 \h^1(t) := \{ u \;\vert\;  \hn{u}{1} < \infty ,  u\vert_{\Sigma}=0\} \text{ and }\x(t) := \{ u \in \h^1(t) \;\vert\; \diverge_{\a(t)}{u}=0\}.
\end{equation}
We will also need the orthogonal decomposition $\h^0(t) = \y(t) \oplus \y(t)^\bot,$ where
\begin{equation}\label{l_Y_space_def}
 \y(t)^\bot := \{ \nab_{\a(t)} \varphi \;\vert\; \varphi \in {^0}H^1(\Omega)  \}.
\end{equation}
A further discussion of the space $\y(t)$ can be found later in Remark \ref{l_Y_characterization}.  In our use of these norms and spaces, we will often drop the $(t)$ when there is no potential for confusion.

Finally, for $T>0$ and $k=0,1$, we define inner-products on $L^2([0,T];H^k(\Omega))$ by
\begin{equation}
 \ip{u}{v}_{\h^k_T} = \int_0^T \iph{u(t)}{v(t)}{k} dt.
\end{equation}
Write $\norm{u}_{\h^k_T}$ for the corresponding norms and $\h^k_T$ for the corresponding spaces.  We define the subspace of $\diva$-free vector fields as
\begin{equation}
 \x_T := \{ u \in \h^1_T \;\vert\; \diverge_{\a(t)}{u(t)} =0 \text{ for a.e. } t\in[0,T]\}.
\end{equation}

A priori we do not know that the spaces $\h^k(t)$ and $\h^k_T$ have the same topology as $H^k$ and $L^2 H^k$, respectively.  This can be established under a smallness assumption on $\eta$.

\begin{lem}\label{l_norm_equivalence}
There exists a universal $\ep_0 > 0$ so that if 
\begin{equation}\label{l_norm_e_0}
\sup_{0\le t \le T} \norm{\eta(t)}_{3} < \ep_0, 
\end{equation}
then 
\begin{equation}\label{l_norm_e_01}
\frac{1}{\sqrt{2}} \norm{u}_{k} \le \hn{u}{k} \le \sqrt{2} \norm{u}_{k}  
\end{equation}
for $k=0,1$ and for all $t \in [0,T]$.  As a consequence, for $k=0,1$, 
\begin{equation}\label{l_norm_e_02}
\frac{1}{\sqrt{2}} \norm{u}_{L^2 H^k} \le \norm{u}_{\h^k_T} \le \sqrt{2} \norm{u}_{L^2 H^k}. 
\end{equation}
\end{lem}

\begin{proof}
Consider $\ep \in (0,1/2)$ with precise value to be chosen later.  It is straightforward to verify, using Lemma \ref{i_poisson_interp} in the non-periodic case and Lemma \ref{p_poisson_2} in the periodic case, that 
\begin{equation}\label{l_norm_e_1}
\sup\{ \pnorm{J-1}{\infty}, \pnorm{A}{\infty}, \pnorm{B}{\infty}  \}  \le C \norm{\eta}_{3}.
\end{equation}
Then we may choose $\ep_0= \ep/C$ so that the right side of \eqref{l_norm_e_1} is bounded by $\ep$. Since $K = 1/J$, this implies that
\begin{equation}
 \norm{K-1}_{L^\infty}\le \frac{\ep}{1-\ep}, \norm{K}_{L^\infty}\le \frac{1}{1-\ep}, 
\end{equation}
and
\begin{equation}
 \norm{I - \a}_{L^\infty}\le \frac{3\ep}{1-\ep}, \norm{\a+I}_{L^\infty} \le 2 \sqrt{3} + \frac{3\ep}{1-\ep}.
\end{equation}
In turn, this implies that
\begin{equation}\label{l_norm_e_2}
 \norm{J}_{L^\infty}\norm{I-\a}_{L^\infty} \norm{I+\a}_{L^\infty} \le \frac{3\ep(1+\ep)(2\sqrt{3}-(2\sqrt{3}-3)\ep)}{(1-\ep)^2} := g(\ep).
\end{equation}
Notice that $g$ is a continuous, increasing function on $(0,1/2)$ so that $g(0)=0$.   With the estimates \eqref{l_norm_e_1} and \eqref{l_norm_e_2} in hand, we can show that if $\ep$ is chosen sufficiently small, then \eqref{l_norm_e_01} and \eqref{l_norm_e_02} hold.

In the case $k=0$, the estimate \eqref{l_norm_e_01} follows directly from the estimate for $J$ in \eqref{l_norm_e_1}:
\begin{equation}
\hal \int_\Omega \abs{u}^2 \le (1-\ep) \int_\Omega \abs{u}^2 \le \int_\Omega J \abs{u}^2 \le (1+\ep) \int_\Omega  \abs{u}^2  \le 2 \int_\Omega \abs{u}^2.
\end{equation}
To derive \eqref{l_norm_e_01} when $k=1$, we first rewrite 
\begin{equation}\label{l_norm_e_3}
\int_\Omega J \abs{\sg_{\a} u}^2 = \int_\Omega J \abs{\sg u}^2 + \int_\Omega J (\sg_{\a} u + \sg u):(\sg_{\a} u - \sg u).
\end{equation}
To estimate the last term, we note that
$ \abs{(\sg_{\a} u \pm \sg u)} \le 2 \abs{\a \pm I} \abs{\nab u},$ which implies that
\begin{multline}\label{l_norm_e_4}
 \abs{\int_\Omega J (\sg_{\a} u + \sg u):(\sg_{\a} u - \sg u)} \le 4 \pnorm{J}{\infty} \pnorm{I-\a}{\infty}\pnorm{I + \a}{\infty} \int_\Omega \abs{\nab u}^2 \\
\le 4 C_\Omega g(\ep) \int_\Omega \abs{\sg u}^2,
\end{multline}
where $C_\Omega$ is the constant in Korn's inequality, Lemma \ref{i_korn}.  We may then employ the bounds \eqref{l_norm_e_1} and \eqref{l_norm_e_4} in \eqref{l_norm_e_3} to estimate
\begin{equation}\label{l_norm_e_5}
 \int_\Omega \abs{\sg_{\a} u}^2 J \ge \int_\Omega J \abs{\sg u}^2 - 4 C_\Omega g(\ep) \int_\Omega \abs{\sg u}^2 \ge (1 - \ep  -4 C_\Omega g(\ep) )  \int_\Omega \abs{\sg u}^2
\end{equation}
and 
\begin{equation}\label{l_norm_e_6}
 \int_\Omega \abs{\sg_{\a} u}^2 J \le \int_\Omega J \abs{\sg u}^2 + 4 C_\Omega g(\ep) \int_\Omega \abs{\sg u}^2  \le (1+ \ep + 4 C_\Omega g(\ep))  \int_\Omega \abs{\sg u}^2.
\end{equation}
Then \eqref{l_norm_e_01} with $k=1$ follows from \eqref{l_norm_e_5}--\eqref{l_norm_e_6} by choosing $\ep$ small enough so that $\ep + 4 C_\Omega g(\ep) \le 1/2$.   The estimates \eqref{l_norm_e_02} follow by applying \eqref{l_norm_e_01} for a.e.  $t \in [0,T]$, squaring, and integrating over $t \in [0,T]$.
\end{proof}

\begin{remark}  Throughout the rest of this chapter, we will assume that \eqref{l_norm_e_0} is satisfied so that \eqref{l_norm_e_01}--\eqref{l_norm_e_02} hold.
\end{remark}

\begin{remark}\label{l_A_korn_trace}
Because of the bound \eqref{l_norm_e_01} and the usual Korn inequality on $\Omega$, Lemma \ref{i_korn}, we have a corresponding Korn-type inequality in $\h^1(t)$: $\hn{u}{0} \ls \hn{u}{1}.$  The standard trace embedding $H^1(\Omega) \hookrightarrow H^{1/2}(\Sigma)$ and \eqref{l_norm_e_01} imply that $\snormspace{u}{1/2}{\Sigma} \ls \hn{u}{1}$ for all $t \in [0,T]$.  Similarly, given $f \in H^{1/2}(\Sigma)$, we may construct an extension $\tilde{f} \in \h^1(t)$ so that $\hn{f}{1} \ls \snormspace{f}{1/2}{\Sigma}$.
\end{remark}

We now prove a result about the differentiability of norms in our time-dependent spaces.

\begin{lem}\label{l_x_time_diff}
Suppose that $u \in \h^1_T$, $\dt u \in (\h^1_T)^*$.  Then the mapping $t \mapsto \norm{u(t)}_{\h^0(t)}^2$ is absolutely continuous, and 
\begin{equation}\label{l_x_t_d_01}
 \frac{d}{dt}  \norm{u(t)}_{\h^0}^2 = 2 \br{\dt u(t),u(t)}_{(\h^1)^*} + \int_\Omega \abs{u(t)}^2 \dt J(t) 
\end{equation}
for a.e. $t \in [0,T]$.  Moreover, $u \in C^0([0,T];H^0(\Omega))$.  If $v \in \h^1_T$, $\dt v \in (\h^1_T)^*$ as well, then
\begin{equation}\label{l_x_t_d_02}
 \frac{d}{dt} \iph{u(t)}{v(t)}{0} =  \br{\dt u(t),v(t)}_{(\h^1)^*} +  \br{\dt v(t),u(t)}_{(\h^1)^*} + \int_\Omega u(t) \cdot v(t) \dt J(t). 
\end{equation}
\end{lem}
\begin{proof}
In light of Lemma \ref{l_norm_equivalence}, the time-dependent spaces  $\h^0_T$, $\h^1_T$, $(\h^1_T)^*$ present no obstacle  to the usual method of approximation by temporally smooth functions via convolution.  This allows us to argue as in Theorem 3 in Section 5.9 of \cite{evans} to deduce \eqref{l_x_t_d_01} and the continuity $u \in C^0([0,T];H^0(\Omega))$.  The equality \eqref{l_x_t_d_02} follows by applying \eqref{l_x_t_d_01} to $u + v$ and canceling terms by using \eqref{l_x_t_d_01} with $u$ and with $v$.
\end{proof}

Now we want to show the spaces $\H1$ and $\Hsig$ are related to the spaces $\h^1(t)$ and $\x(t)$.  To this end, we define the matrix
\begin{equation}\label{l_M_def}
 M := M(t) = K \nab \Phi =
\begin{pmatrix}
K & 0 & 0 \\
0 & K & 0 \\
AK & BK & 1
\end{pmatrix}.
\end{equation}
Note that $M$ is invertible, and $M^{-1} = J \a^T$.  Since $J \neq 0$ and $\p_j(J \a_{ij})=0$ for each $i=1,2,3$, 
\begin{multline}\label{l_div_preserve}
 p = \diva v \Leftrightarrow \\
 Jp = J \diva v = J \a_{ij}\p_j v_i = \p_j ( J \a_{ij} v_i ) = \p_j (J \a^T v)_j = \p_j( M^{-1} v)_j = \diverge( M^{-1} v ).
\end{multline}
The matrix $M(t)$ induces a linear operator $\mathcal{M}_t: u \mapsto \mathcal{M}_t(u) = M(t) u$ that possesses several nice properties, the most important of which is that $\diverge$-free vector fields are mapped to $\diva$-free vector fields.  We record these now.

\begin{prop}\label{l_M_iso}
For each $t \in[0,T]$, $\mathcal{M}_t$ is a bounded, linear isomorphism:
from $H^k(\Omega)$ to $H^k(\Omega)$ for $k=0,1,2$;  from $L^2(\Omega)$ to $\h^0(t)$;  from $\H1$ to $\h^1(t)$; and  from $\Hsig$ to $\x(t)$.  In each case the norms of the operators $\mathcal{M}_t, \mathcal{M}_t^{-1}$ are bounded by a constant times $1 + \norm{\eta(t)}_{9/2}$.

Moreover,  the mapping $\mathcal{M}$ given by $\mathcal{M}u(t) := \mathcal{M}_t u(t)$ is a bounded, linear isomorphism: from $L^2([0,T];H^k(\Omega))$ to $L^2([0,T];H^k(\Omega))$ for $k=0,1,2$; from $L^2([0,T];H^0(\Omega))$ to $\h^0_T$;  from $L^2([0,T];\H1)$ to $\h^1_T$; and from $L^2([0,T];\Hsig)$ to $\x_T$.  In each case, the norms of the operators $\mathcal{M}$ and  $\mathcal{M}^{-1}$ are bounded by a constant times the sum $1+\sup_{0\le t \le T} \norm{\eta(t)}_{9/2}$. 
\end{prop}

\begin{proof}
For each $t \in [0,T]$, it is  easy to see that
\begin{equation}\label{l_Mi_1}
 \norm{ \mathcal{M}_t u}_{k} \ls \norm{M(t)}_{C^3} \norm{u}_{k} \ls (1+\norm{\eta(t)}_{9/2}) \norm{u}_{k}
\end{equation}
for $k=0,1,2$, which establishes that $\mathcal{M}_t$ is a bounded operator on $H^k$.  Since $M(t)$ is an invertible matrix,  $\mathcal{M}_t^{-1} v  = M(t)^{-1} v = J \nab \Phi(t) v$, which allows us to argue similarly to see that for $k=0,1,2$, $ \norm{ \mathcal{M}_t^{-1} v}_{k} \ls  (1+\norm{\eta(t)}_{9/2}) \norm{v}_{k}.$  Hence $\mathcal{M}_t$ is an isomorphism of $H^k$ to itself for $k=0,1,2$.  With this fact in hand, Lemma \ref{l_norm_equivalence} implies that $\mathcal{M}_t$ is an isomorphism of $H^0(\Omega)$ to $\h^0(t)$ and of $\H1$ to $\h^1(t)$.  

To prove that $\mathcal{M}_t$ is an isomorphism of $\Hsig$ to $\x(t)$, we must only establish that $\diverge{u} =0$ if and only if $\diva(M u) =0$.   To see this we appeal to \eqref{l_div_preserve} with $p=0$ to see that $0 = \diva v$ if and only if $0 = \diverge(M^{-1} v)$.  Hence, writing $v = M u$, we see that $\diverge{u} =0$ if and only if $\diva(M u) =0$.

The mapping properties of the operator $\mathcal{M}$ on space-time functions may be established in a similar manner.
\end{proof}

\subsection{Pressure as a Lagrange multiplier}

It is well-known \cite{solonnikov_skadilov, beale_1, coutand_shkoller_1} that the space $\H1$ can be orthogonally decomposed as $\H1 = \Hsig \oplus R(Q)$, where $R(Q)$ is the range of the operator $Q:H^0(\Omega) \to \H1$, defined by the Riesz representation theorem via the relation
\begin{equation}
 \int_\Omega p \diverge{u} = \int_\Omega \sg (Qp) : \sg u \text{ for all } u \in \H1.
\end{equation}
We now wish to establish a similar decomposition for our spaces $\x(t) \subset \h^1(t)$.  Unfortunately, the mappings $\mathcal{M}_t$, while isomorphisms, are not isometries, so we cannot use the known result to decompose $\h^1(t)$.  Instead, we must adapt the method of \cite{solonnikov_skadilov} to our time-dependent context.

For $p \in \h^0(t)$, we define the functional $\mathcal{Q}_t \in (\h^1(t))^*$ by $\mathcal{Q}_t(v) = \iph{p}{\diva v}{0}$.  By the Riesz representation theorem, there exists a unique $ Q_t p \in \h^1(t)$ so that $ \mathcal{Q}_t(v) = \iph{Q_t p}{v}{1}$  for all $v\in \h^1(t).$ This defines a linear operator $Q_t : \h^0(t) \to \h^1(t)$, which is bounded since we may take $v = Q_t p$ to bound
\begin{multline}
\hn{Q_t p}{1}^2 =  \iph{Q_t p}{Q_t p}{1} = \mathcal{Q}_t(v) =\iph{p}{\diva v}{0} \\
\le \hn{p}{0} \hn{\diva v}{0} \le \hn{p}{0} \hn{ v}{1} = \hn{p}{0} \hn{ Q_t p}{1},
\end{multline}
so that $\hn{Q_t p}{1} \le \hn{p}{0}$.  In the previous inequality we have used the simple bound $ \hn{\diva v}{0} \le  \hn{ v}{1}$, which follows from the fact that $ \diva v = \text{tr}(\sg_{\a} u)/2$.  In a straightforward manner, we may also define a bounded linear operator $Q: \h^0_T \to \h^1_T$ via the relation
\begin{equation}
\ip{p}{\diva v}_{\h^0_T} = \ip{Q p}{v}_{\h^1_T} \text{ for all } v\in \h^1_T.
\end{equation}
Arguing as above, we can show that $Q$ satisfies $\norm{Q p}_{\h^1_T} \le \norm{p}_{\h^0_T}$.

In order to study the range of $Q_t$ in $\h^1(t)$ and of $Q$ in $\h^1_T$, we will first need a lemma on the solvability of the equation $\diva v = p$.

\begin{lem}\label{l_diverge_solvable}
 Let $p \in \h^0(t)$.  Then there exists a $v \in \h^1(t)$ so that $\diva v = p$ and $\hn{v}{1} \ls (1+ \norm{\eta(t)}_{9/2}) \hn{p}{0}$.  If instead $p \in \h^0_T$, then there exists a $v \in \h^1_T$ so that $\diva v=p$ for a.e. $t \in [0,T]$,  and $\norm{v}_{\h^1_T} \ls (1 + \sup_{0\le t \le T} \norm{\eta(t)}_{9/2} ) \norm{p}_{\h^0_T}$.
\end{lem}
\begin{proof}
It is established in the proof of Lemma 3.3 of \cite{beale_1} that for any $q \in L^2(\Omega)$ the problem $\diverge{u} = q$ admits a solution $u \in \H1$ so that $\norm{u}_{1} \ls \norm{q}_{0}$.  The result in \cite{beale_1} concerns the non-periodic case, but its proof  may be easily adapted to the periodic case as well.  Choose $q = J p$ so that
\begin{equation}
\norm{q}_0^2 =  \int_\Omega \abs{q}^2 = \int_\Omega \abs{p}^2 J^2 \le \pnorm{J}{\infty} \hn{p}{0}^2 \le 2 \hn{p}{0}^2.
\end{equation}
Then by \eqref{l_div_preserve} we know that $v = M(t) u \in \h^1(t)$ satisfies $\diva{v} = p$, and  Proposition \ref{l_M_iso} implies that
\begin{equation}\label{l_div_solve_1}
 \hn{v}{1} \ls (1+\norm{\eta(t)}_{9/2} )\norm{u}_{1} \ls (1+\norm{\eta(t)}_{9/2}) \norm{q}_{0} \ls (1+\norm{\eta(t)}_{9/2}) \hn{p}{0}.
\end{equation}

If $p \in \h^0_T$, then for a.e. $t \in [0,T]$, $p(t) \in \h^0(t)$, so we may apply the above analysis to find $v(t) \in \h^1(t)$ so that $\diva v(t) = p(t)$ and the bound \eqref{l_div_solve_1} holds with $v = v(t)$ and $p=p(t)$.  We may then square both sides and integrate over $t \in [0,T]$ to deduce that 
\begin{multline}\label{l_div_solve_2}
\norm{v}_{\h^1_T}^2 =  \int_0^T \hn{v(t)}{1}^2 dt \ls \left(1+ \sup_{0\le t \le T} \norm{\eta(t)}_{9/2}^2\right) \int_0^T \hn{p(t)}{0}^2 dt 
\\ \ls \left(1+ \sup_{0\le t \le T} \norm{\eta(t)}_{9/2}^2\right) \norm{v}_{\h^0_T}^2.
\end{multline}

\end{proof}

With this lemma in hand, we can show that $R(Q_t)$ is a closed subspace of $\h^1(t)$ and that $R(Q)$ is a closed subspace of $\h^1_T$.

\begin{lem}\label{l_closed_range}
 $R(Q_t)$ is closed in $\h^1(t)$, and $R(Q)$ is closed in $\h^1_T$.
\end{lem}
\begin{proof}
 For $p \in \h^0(t)$ let $v \in \h^1(t)$ be the solution to $\diva{v}=p$ provided by Lemma \ref{l_diverge_solvable}.  Then
\begin{multline}
 \hn{p}{0}^2 = \iph{p}{\diva{v}}{0} = \mathcal{Q}_t(v) = \iph{Q_t p}{v}{1} \\
\le \hn{Q_t p}{1} \hn{v}{1} \ls \hn{Q_t p}{1}(1+ \norm{\eta(t)}_{9/2}) \hn{p}{0} 
\end{multline}
so that $\hn{Q_t p}{1} \le \hn{p}{0} \ls   (1+\norm{\eta(t)}_{9/2}) \hn{Q_t p}{1}.$  Hence $R(Q_t)$ is closed in $\h^1(t)$.  A similar analysis shows that $R(Q)$ is closed in $\h^1_T$.
\end{proof}

Now we can perform the orthogonal decomposition of $\h^1(t)$ and $\h^1_T$.

\begin{lem}\label{l_orthogonal_decomp}
 We have that $\h^1(t) = \x(t) \oplus R(Q_t)$, i.e. $\x(t)^\bot = R(Q_t)$.  Also, $\h^1_T = \x_T \oplus R(Q)$, i.e. $\x_T^\bot = R(Q)$.
\end{lem}
\begin{proof}
By Lemma \ref{l_closed_range}, $R(Q_t)$ is a closed subspace of $\h^1(t)$, and so it suffices to prove that $R(Q_t)^\bot = \x(t)$.

Let $v \in R(Q_t)^\bot$.  Then for all $p \in \h^0(t)$, we know that
\begin{equation}
 \int_\Omega p \diva{v} J = \mathcal{Q}_t(v) = \iph{Q_t p}{v}{1} = 0,
\end{equation}
and hence $\diva{v}=0$.  This implies that $R(Q_t)^\bot \subseteq \x(t)$.

Now suppose that $v \in \x(t)$.  Then $\diva{v}=0$ implies that
\begin{equation}
 0 =  \int_\Omega p \diva{v} J  = \mathcal{Q}_t(v) = \iph{Q_t p}{v}{1}
\end{equation}
for all $p \in \h^0(t)$.  Hence $v \in R(Q_t)^\bot$, and we see that $\x(t) \subseteq R(Q_t)^\bot$.

A similar argument shows that $\h^1_T = \x_T \oplus R(Q)$.
\end{proof}

This decomposition will eventually allow us to introduce the pressure function.  This will be accomplished by use of the following result.

\begin{prop}\label{l_pressure_decomp}
If $\Lambda_t \in (\h^1(t))^*$ is such that $\Lambda_t(v) = 0$ for all $v \in \x(t)$, then there exists a unique $p(t) \in \h^0(t)$ so that
\begin{equation}
 \iph{p(t)}{\diva v}{0} = \Lambda_t(v) \text{ for all } v\in \h^1(t)
\end{equation}
and $\hn{p(t)}{0} \ls (1+\norm{\eta(t)}_{9/2}) \norm{\Lambda_t}_{(\h^1(t))^*}$.

If $\Lambda \in (\h^1_T)^*$ is such that $\Lambda(v) = 0 $ for all $v \in \x_T$, then there exists a unique $p \in \h^0_T$ so that
\begin{equation}
 \ip{p}{\diva v}_{\h^0_T} = \Lambda(v) \text{ for all } v\in \h^1_T
\end{equation}
and $\norm{p}_{\h^0_T} \ls (1+\sup_{0\le t\le T} \norm{\eta(t)}_{9/2}) \norm{\Lambda}_{(\h^1_T)^*}$.
\end{prop}

\begin{proof}
If $\Lambda_t(v) = 0$ for all $v \in \x(t)$, then  the Riesz representation theorem yields the existence of a unique $w \in \x(t)^\bot$ so that $\Lambda_t(v) = \iph{w}{v}{1}$ for all $v \in \h^1(t)$.  By Lemma \ref{l_orthogonal_decomp}, $w = Q_t p(t)$ for some $p(t) \in \h^0(t)$.  Then 
 $\Lambda_t(v) = \iph{Q_t p(t)}{v}{1} = \iph{p(t)}{\diva v}{0}$ for all $v \in \h^1(t)$.  By Lemma \ref{l_diverge_solvable}, we may find $v(t) \in \h^1(t)$ so that $\diva v(t) = p(t)$ and 
$\hn{v(t)}{1} \ls  (1+\norm{\eta(t)}_{9/2}) \hn{p(t)}{0}$.  Hence
\begin{equation}
 \hn{p(t)}{0}^2 = \iph{p(t)}{\diva v(t)}{0} = \Lambda_t(v(t)) \le \norm{\Lambda_t}_{(\h^1(t))^*} (1+\norm{\eta(t)}_{9/2}) \hn{p(t)}{0},
\end{equation}
and the desired estimate holds.  A similar argument proves the result for $\Lambda \in (\h^1_T)^*$ such that $\Lambda(v) = 0 $ for all $v \in \x_T$.
\end{proof}


\section{Elliptic estimates}\label{lwp_2}

\subsection{Preliminary estimates}

In studying the elliptic problems in the rest of this section we will utilize the fact that the equations can be transformed into constant coefficient equations on the domain $\Omega' = \Phi(\Omega)$.  In order to properly utilize this transformation we must verify that composition with $\Phi$ generates an isomorphism of $H^k(\Omega')$ to $H^k(\Omega)$.  This type of result is standard (see the appendix of \cite{bourg_brezis} for a bounded domain, or Lemma 5.2 of \cite{beale_2} and Lemma 6.2  of \cite{sylvester} for domain $\Rn{n}$), but the precise form we need is not readily available in the literature, so we record it now.

\begin{lem}\label{l_sobolev_composition}
 Let $\Psi: \Omega \to \Omega'$ be a $C^1$ diffeomorphism satisfying $\pnorm{1 - \det{\nab \Psi}}{\infty} \le 1/2$ and $\nab \Psi - I \in H^{k}(\Omega)$ for an integer $k \ge 3$.  If  $v \in H^m(\Omega')$, then  $v \circ \Psi \in H^m(\Omega)$ for $m =0,1,\dotsc,k+1$, and 
\begin{equation}\label{l_sob_c_0}
 \norm{v \circ \Psi}_{H^m(\Omega)} \ls C(\norm{\nab \Psi -I}_{H^{k}(\Omega)}) \norm{v}_{H^m(\Omega')}
\end{equation}
for $C(\norm{\nab \Psi -I}_{H^{k}(\Omega)})$ a constant depending on $\norm{\nab \Psi -I}_{H^{k}(\Omega)}$.  Similarly, for $u \in H^m(\Omega)$, $u \circ \Psi^{-1}\in H^m(\Omega')$ for $m=0,1,\dotsc,k+1$, and 
\begin{equation}\label{l_sob_c_00}
 \norm{u \circ \Psi^{-1}}_{H^m(\Omega')} \ls C(\norm{\nab \Psi -I}_{H^{k}(\Omega)}) \norm{u}_{H^m(\Omega)}.
\end{equation}

Let $\Sigma' = \Psi(\Sigma)$ denote the upper boundary of $\Omega'$.  If $v \in H^{m-1/2}(\Sigma')$ for $m=1,\dotsc,k-1$, then $v \circ \Psi \in H^{m-1/2}(\Sigma)$
and 
\begin{equation}\label{l_sob_c_03}
 \norm{v \circ \Psi}_{H^{m-1/2}(\Sigma)} \ls C(\norm{\nab \Psi -I}_{H^{k}(\Omega)}) \norm{v}_{H^{m-1/2}(\Sigma')}.
\end{equation}
If $u \in H^{m-1/2}(\Sigma)$ for $m=1,\dotsc,k-1$, then $v \circ \Psi^{-1} \in H^{m-1/2}(\Sigma')$
and 
\begin{equation}\label{l_sob_c_04}
 \norm{u \circ \Psi^{-1}}_{H^{m-1/2}(\Sigma')} \ls C(\norm{\nab \Psi -I}_{H^{k}(\Omega)}) \norm{u}_{H^{m-1/2}(\Sigma)}.
\end{equation}

\end{lem}

\begin{proof}
The proof of \eqref{l_sob_c_0}--\eqref{l_sob_c_00} is similar to the proofs of the results in \cite{bourg_brezis, beale_2, sylvester} mentioned above, so we present only a sketch.  We first prove that for $m=0,1,2$, it holds that
\begin{equation}\label{l_sob_c_1}
 \norm{v \circ \Psi}_{H^m(\Omega)} \ls C(\norm{\nab \Psi -I}_{H^{k-1}(\Omega)}) \norm{v}_{H^m(\Omega')}.
\end{equation}
Such a bound follows easily from the size of $k$ and the bound on $\det \nab \Psi$.  We then proceed inductively for $m=3, \dotsc, k+1$.  Suppose the bound \eqref{l_sob_c_1} holds for $m=0,1,2,\dotsc,m_0$ for $2 \le m_0 \le k$.  To show that it holds for $m_0+1$ we write $x$ for coordinates in $\Omega$ and $y$ for coordinates in $\Omega'$ and note that 
\begin{equation}
 \frac{\p}{\p x_i} (v \circ \Psi )(x) = \frac{\p v}{\p y_j} \circ \Psi(x) \cdot \frac{\p \Psi_j}{\p x_i}(x) = \frac{\p v}{\p y_i} \circ \Psi(x) + \frac{\p v}{\p y_j} \circ \Psi(x) \cdot \left( \frac{\p \Psi_j}{\p x_i}(x) - I_{ij} \right).
\end{equation}
By the induction hypothesis, if $v \in H^{m_0+1}$, then
\begin{equation}
 \frac{\p v}{\p y_j} \circ \Psi \in H^{m_0} \text{ for all } j=1,2,3, 
\end{equation}
and since we have the multiplicative embedding $H^{m_0} \cdot H^{k} \hookrightarrow H^{m_0}$ for $m_0 \ge 2$ and $k\ge 3$, we deduce that 
\begin{equation}
 \frac{\p}{\p x_i} (v \circ \Psi ) \in H^{m_0} \text{ for all } i=1,2,3,
\end{equation}
and hence that $v \circ \Psi \in H^{m_0+1}$.  Moreover, an estimate of the form \eqref{l_sob_c_1} holds.  By induction, we deduce that \eqref{l_sob_c_0} holds.  The result \eqref{l_sob_c_00} follow similarly, utilizing the fact that $\nab \Psi^{-1} (y) = (\nab \Psi)^{-1} \circ \Psi^{-1}(y)$.

We now turn to the proof of  \eqref{l_sob_c_03}--\eqref{l_sob_c_04}.  First note that since $\Psi \in H^{k+1}_{loc}$, we have that $\Sigma'$ is locally the graph of a $C^{k-1,1/2}$ function.  As such (cf. \cite{adams}), there exists a bounded extension operator $E:H^{m-1/2}(\Sigma') \to H^{m}(\Omega')$ for $m=1,\dotsc,k-1$ with the norm of the operator depending on $C(\norm{\nab \Psi -I}_{H^{k}(\Omega)})$.  For $v \in H^{m-1/2}(\Sigma')$, let $V = Ev \in H^{m}(\Omega')$.  By \eqref{l_sob_c_0}, we have that  $V \circ \Psi \in H^{m}(\Omega)$, and by the usual trace theory, $v \circ \Psi = V \circ \Psi \vert_\Sigma \in H^{m-1/2}(\Sigma)$.  Moreover, 
\begin{multline}
 \norm{v \circ \Psi}_{H^{m-1/2}(\Sigma)} \ls \norm{ V \circ \Psi}_{H^m(\Omega)} \ls 
C(\norm{\nab \Psi -I}_{H^{k}(\Omega)}) \norm{Ev}_{H^m(\Omega')} 
\\ \ls 
C(\norm{\nab \Psi -I}_{H^{k}(\Omega)}) \norm{v}_{H^{m-1/2}(\Sigma')},
\end{multline}
which is \eqref{l_sob_c_03}.  The bound \eqref{l_sob_c_04} follows similarly.
\end{proof}

\begin{remark}\label{C1_diff}
It is easy to show, using Lemma \ref{p_poisson_2} in the periodic case and Lemma \ref{i_poisson_interp} in the non-periodic case,  that if $\ns{\eta}_{k+1/2}$ is sufficiently small for $k\ge 3$, then the mapping $\Phi$ defined by \eqref{mapping_def} is a $C^1$ diffeomorphism that satisfies the hypotheses of Lemma \ref{l_sobolev_composition}.
\end{remark}

We will also need the following $H^{-1/2}$ boundary estimates for functions satisfying $u\in \h^0(t)$ and $\diva u \in \h^0(t)$. 

\begin{lem}\label{l_boundary_dual_estimate}
If $v \in \h^0(t)$ and $\diva v \in \h^0(t)$, then $v \cdot \n \in H^{-1/2}(\Sigma)$, $v \cdot \nu \in H^{-1/2}(\Sigma_b)$ (with $\nu$ the unit normal on $\Sigma_b$), and 
\begin{equation}
 \snormspace{v \cdot \n}{-1/2}{\Sigma} +  \snormspace{v \cdot \nu}{-1/2}{\Sigma_b} \ls \hn{v}{0} + \hn{\diva{v}}{0}.
\end{equation}
\end{lem}
\begin{proof}
We will only prove the result on $\Sigma$; the result on $\Sigma_b$ may be derived in a similar manner, using the fact that $J \a \nu = \nu$ on $\Sigma_b$.

Let $\varphi \in H^{1/2}(\Sigma)$ be a scalar function, and let $\tilde{\varphi} \in \H1$ be a bounded extension.  If we define the vector field $w = \tilde{\varphi} e_1$, then a straightforward computation reveals that
\begin{equation}
 2 \int_\Omega \abs{\naba \tilde{\varphi}}^2 J \le \hn{w}{1}^2 \text{ and that } \norm{w}_{\H1}^2 \le 4 \int_\Omega \abs{\nab \tilde{\varphi}}^2,
\end{equation}
which, when combined with Lemma \ref{l_norm_equivalence}, implies that $\hn{\tilde{\varphi}}{0}+ \hn{\naba \tilde{\varphi}}{0} \ls  \snormspace{\varphi}{1/2}{\Sigma}$.  Then
\begin{multline}
 \int_\Sigma \varphi v \cdot \n = \int_\Sigma J \a_{ij} v_i \varphi (e_j\cdot e_3)  =\int_\Omega \diva(v \tilde{\varphi}) J = \int_\Omega \tilde{\varphi} \diva{v} J + v \cdot \naba \tilde{\varphi} J \\
 \le \hn{\tilde{\varphi}}{0} \hn{\diva{v}}{0} + \hn{v}{0} \hn{\naba \tilde{\varphi}}{0}
\ls \snormspace{\varphi}{1/2}{\Sigma} \left(\hn{v}{0} +  \hn{\diva{v}}{0} \right).
\end{multline}
The desired bound follows from this inequality by taking the supremum over all $\varphi$ so that $\snormspace{\varphi}{1/2}{\Sigma} \le 1$.
\end{proof}

\begin{remark}\label{l_Y_characterization}
Recall the space $\y(t) \subset \h^0(t)$, defined by \eqref{l_Y_space_def}.  It can be shown that if $v \in \y(t)$, then $\diva v =0$ in the weak sense, so that Lemma \ref{l_boundary_dual_estimate} implies that $v \cdot \n \in H^{-1/2}(\Sigma)$ and $v \cdot \nu \in H^{-1/2}(\Sigma_b)$.  Moreover, since the elements of $\y(t)$ are orthogonal to each $\naba \varphi$ for $\varphi \in {^0}H^1(\Omega)$, we find that $v \cdot \nu =0$ on $\Sigma_b$.
\end{remark}

\subsection{The $\a-$Stokes problem}

In order to derive the regularity for our solutions to \eqref{l_linear_forced}, we will first need to study the regularity of the corresponding stationary problem
\begin{equation}\label{l_linear_elliptic}
 \begin{cases}
 - \da u + \naba p = F^1 & \text{in }\Omega \\
\diva{u}=F^2 & \text{in }\Omega\\
\Sa(p,u) \n = F^3 & \text{on }\Sigma \\
u =0 & \text{on }\Sigma_b.
 \end{cases}
\end{equation}
Since this problem is stationary, we will temporarily ignore the time dependence of $\eta, \a,$ etc.

We are interested in the regularity theory for strong solutions to \eqref{l_linear_elliptic}, but before discussing that, we shall mention the weak formulation.  Our method of solution is similar to that of \cite{solonnikov_skadilov, beale_1, coutand_shkoller_1}; we utilize Proposition \ref{l_pressure_decomp} to introduce $p$ after first solving a pressureless problem.   Suppose $F^1 \in (\h^1)^*$, $F^2 \in \h^0$, $F^3 \in H^{-1/2}(\Sigma)$.  We say $(u,p)\in \h^1 \times \h^0$ is a weak solution to \eqref{l_linear_elliptic} if $\diva u = F^2$ a.e. in $\Omega$, and
\begin{equation}\label{l_elliptic_weak}
\hal \iph{u}{v}{1} - \iph{p}{\diva v}{0} = \br{F^1,v}_{(\h^1)*} - \br{F^3,v}_{-1/2} \text{ for all } v \in \h^1,
\end{equation}
where $\br{\cdot,\cdot}_{(\h^1)*}$ denotes the dual pairing in $\h^1$ and  $\br{\cdot,\cdot}_{-1/2}$ denotes the dual pairing between $H^{-1/2}(\Sigma)$ and $H^{1/2}(\Sigma)$. 

\begin{prop}\label{l_elliptic_weak_soln}
Suppose $F^1 \in (\h^1)^*$, $F^2 \in \h^0$, $F^3 \in H^{-1/2}(\Sigma)$.  Then there is exists a unique weak solution $(u,p) \in \h^1 \times \h^0$ to \eqref{l_elliptic_weak}.
\end{prop}
\begin{proof}
By Lemma \ref{l_diverge_solvable}, there exists a $\bar{u} \in \h^1$ so that $\diva \bar{u} = F^2$.  We may then switch unknowns to $w = u - \bar{u}$ so that the weak formulation for $w$ is $\diva w=0$ and
\begin{equation}\label{l_elliptic_weak_2}
\hal \iph{w}{v}{1} - \iph{p}{\diva v}{0} = -\hal \iph{\bar{u}}{v}{1} + \br{F^1,v}_{(\h^1)*} - \br{F^3,v}_{-1/2} \text{ for all } v \in \h^1.
\end{equation}
To solve for $w$ without $p$ we restrict the test functions to $v \in \x$ so that the second term on the left vanishes.  A straightforward application of the Riesz representation theorem then provides a unique $w \in \x$ satisfying 
\begin{equation}\label{l_elliptic_weak_3}
\hal \iph{w}{v}{1}  = -\hal \iph{\bar{u}}{v}{1} + \br{F^1,v}_{(\h^1)*} - \br{F^3,v}_{-1/2} \text{ for all } v \in \x.
\end{equation}
To introduce the pressure, $p$, we define $\Lambda \in (\h^1)^*$ as the difference between the left and right sides of \eqref{l_elliptic_weak_3}.  Then $\Lambda(v)=0$ for all $v \in \x$, so by Proposition \ref{l_pressure_decomp}, there exists a unique $p \in \h^0$ satisfying $\iph{p}{\diva v}{0} = \Lambda(v)$ for all $v \in \h^1$, which is equivalent to \eqref{l_elliptic_weak_2}.
\end{proof}

The regularity gain available for solutions to \eqref{l_linear_elliptic} is limited by the regularity of the coefficients of the operators $\da, \naba, \diva$, and hence by the regularity of $\eta$.  In the next result we establish the strong solvability of \eqref{l_linear_elliptic} and present some elliptic estimates, but we do not yet seek the optimal regularity.

\begin{lem}\label{l_stokes_solvable}
Suppose that $\eta \in H^{k+1/2}(\Sigma)$ for $k\ge 3$ is as small as in Remark \ref{C1_diff} so that the mapping $\Phi$ defined by \eqref{mapping_def} is a $C^1$ diffeomorphism of $\Omega$ to $\Omega'=\Phi(\Omega).$  If $F^1 \in H^0(\Omega)$, $F^2\in H^1(\Omega)$, and $F^3 \in H^{1/2}(\Sigma)$, then the problem \eqref{l_linear_elliptic} admits a unique strong solution $(u,p) \in H^2(\Omega) \times H^1(\Omega)$, i.e. $u,p$ satisfy \eqref{l_linear_elliptic} a.e. in $\Omega$, $\Sigma$, and $\Sigma_b$.  Moreover, for $r = 2,\dotsc,k-1$ we have the estimate
\begin{equation}\label{l_s_solve_0}
 \norm{u}_{r} + \norm{p}_{r-1} \ls C(\eta) \left( \norm{F^1}_{r-2} + \norm{F^2}_{r-1} + \norm{F^3}_{r-3/2}  \right),
\end{equation}
whenever the right hand side is finite, where $C(\eta)$ is a constant depending on $\norm{\eta}_{k+1/2}$.
\end{lem}

\begin{proof}
We transform the problem \eqref{l_linear_elliptic} to one on $\Omega' = \Phi(\Omega)$ by introducing the unknowns $v, q$ according to $u = v \circ \Phi$, $p = q \circ \Phi$.  Then $v,q$ should be solutions to the usual Stokes problem on $\Omega' = \{ -b(y_1,y_2) \le y_3 \le \eta(y_1,y_2)\}$ with upper boundary $\Sigma'= \{ y_3 = \eta\}$: 
\begin{equation}\label{l_linear_stokes}
 \begin{cases}
 - \Delta v + \nab q = G^1 = F^1 \circ \Phi^{-1} & \text{in }\Omega' \\
\diverge{v}=  G^2 = F^2 \circ \Phi^{-1} & \text{in }\Omega' \\
(qI - \sg v) \n = G^3 = F^3 \circ \Phi^{-1} & \text{on } \Sigma' \\
v =0 & \text{on }\Sigma_b.
 \end{cases}
\end{equation}
Note that, according to Lemma \ref{l_sobolev_composition}, $G^1 \in H^0(\Omega')$, $G^2 \in H^1(\Omega')$, and $G^3 \in H^{1/2}(\Sigma')$.  We claim that there exist unique $v \in H^2(\Omega')$, $q \in H^1(\Omega')$, solving problem \eqref{l_linear_stokes} with 
\begin{equation}\label{l_s_solve_1}
 \norm{v}_{H^2(\Omega')} + \norm{q}_{H^1(\Omega')}  \ls C(\eta) \left( \norm{G^1}_{H^0(\Omega')} + \norm{G^2}_{H^1(\Omega')} + \norm{G^3}_{H^{1/2}(\Sigma')}  \right), 
\end{equation}
for $C(\eta)$ a constant depending on $\norm{\eta}_{k+1/2}$.  Let us assume for the moment that the claim is true; we  first show how \eqref{l_s_solve_0} follows from the claim, and then turn to its proof.

To go from $H^2 \times H^1$ to  higher regularity, we appeal to the theory of elliptic systems with complementary boundary conditions, developed in \cite{adn_2}.  It is well-known that the Stokes system \eqref{l_linear_stokes} is such an elliptic system.  Theorem 10.5 of  \cite{adn_2} provides estimates in bounded domains, but we may argue as in Lemma 3.3 of \cite{beale_1} to transform the localized estimates into estimates in all of $\Omega'$, provided that the boundary $\Sigma'$ is sufficiently smooth.  In order for estimates of the form \eqref{l_s_solve_0} to hold for $r=2,\dotsc,k-1$,  \cite{adn_2} requires that $\Sigma'$ be $C^{k-1}$, which is satisfied since $\eta \in C^{k-1,1/2}(\Sigma)$.  Hence, for $r=2,\dotsc,k-1$, 
\begin{equation}\label{l_s_solve_2}
 \norm{v}_{H^r(\Omega)'} + \norm{q}_{H^{r-1}(\Omega')} 
 \ls C(\eta) \left( \norm{G^1}_{H^{r-2}(\Omega')} + \norm{G^2}_{H^{r-1}(\Omega)'} + \norm{G^3}_{H^{r-3/2}(\Sigma')}  \right), 
\end{equation}
for $C(\eta)$ a constant depending on $\norm{\eta}_{k+1/2}$, whenever the right side is finite.

We now transform back to $\Omega$ with $u = v \circ \Phi$, $p = q \circ \Phi$.  It is readily verified that $u,p$ are strong solutions of \eqref{l_linear_elliptic}.  Since $\Phi$ satisfies $\nab \Phi -I \in H^{k}$,  Lemma \ref{l_sobolev_composition} and \eqref{l_s_solve_2} imply that
\begin{equation}
 \norm{u}_{r} + \norm{p}_{r-1} \ls C(\eta) \left( \norm{F^1}_{r-2} + \norm{F^2}_{r-1} + \norm{F^3}_{r-3/2}  \right).  
\end{equation}
for $r = 2,\dotsc,k-1$ whenever the right side is finite.  This is \eqref{l_s_solve_0}.

We now turn to the proof of the above claim, which employs ideas from \cite{beale_1}.  To demonstrate the existence of $H^2 \times H^1$ solutions of \eqref{l_linear_stokes}, we first consider the special case in which $G^2=0$, $G^3=0$, and $G^1 \in H^0(\Omega')$ is arbitrary.   In this case, we may argue as in Lemma 3.3 of \cite{beale_1} (which in turn invokes \cite{solonnikov_skadilov}) to deduce the existence of a unique solution to \eqref{l_linear_stokes} satisfying \eqref{l_s_solve_1} with $G^2 =0$, $G^3=0$.

To handle the case of non-vanishing $G^2$ and $G^3$, we construct some special auxiliary functions that allow us to reduce to the special case.   First, there exists a  $v^1 \in H^2(\Omega') \cap {_0}H^1(\Omega')$ so that $\diverge{v^1} = G^2 \in H^1(\Omega')$ and 
\begin{equation}\label{l_s_solve_3}
 \norm{v^1}_{H^2(\Omega')} \ls \norm{G^2}_{H^1(\Omega')}.
\end{equation}
The existence of $v^1$ may be established as in Lemma 3.3 and Section 4 of \cite{beale_1}.  To deal with the boundary term $G^3$ we first need some projections.  For a vector field $X: \Sigma' \to \Rn{3}$ let us write   $\Pi X$ for the vector field so that  $\Pi X (y)$ is the orthogonal projection of $X(y)$ onto the space of vectors orthogonal to $\n(y)$, and let us write  $\Pi^\bot X (y)$ for the orthogonal projection onto the line generated by $\n(y)$.  Our second special function is  $v^2 \in H^2(\Omega') \cap {_0}H^1_{\sigma}(\Omega')$ that satisfies $\Pi(-\sg v^2 \n) = \Pi(G^3 + \sg v^1 \n )$ and
\begin{equation}\label{l_s_solve_4}
 \norm{v^2}_{H^2(\Omega')} \ls C(\eta) \left(\norm{G^3 + \sg v^1 \n }_{H^{1/2}(\Sigma')}\right) \ls   C(\eta) \left(\norm{G^2}_{H^1(\Omega')} + \norm{G^3}_{H^{1/2}(\Sigma')} \right).
\end{equation}
The construction of $v^2$ may be carried out through a simple modification of the proof of Lemma 4.2 in \cite{beale_1}, working in Sobolev spaces defined on $\Omega'$ rather than $\Omega' \times (0,T)$.  The third special function is $q^1 \in H^1(\Omega')$ that satisfies $q\vert_{\Sigma'} = \Pi^\bot (G^3 + \sg v^1 \n)$ and 
\begin{equation}\label{l_s_solve_5}
 \norm{q^1}_{H^1(\Omega')} \ls C(\eta) \left(\norm{G^3 + \sg v^1 \n }_{H^{1/2}(\Sigma')}\right) \ls   C(\eta) \left(\norm{G^2}_{H^1(\Omega')} + \norm{G^3}_{H^{1/2}(\Sigma')} \right).
\end{equation}
The existence of $q^1$ follows from the usual trace and extension theory since $G^3 + \sg v^1 \n \in H^{1/2}(\Sigma')$.  

Now, with $v^1, v^2$ and $q^1$ in hand, we reduce the solvability of \eqref{l_linear_stokes} with the estimate \eqref{l_s_solve_1} to the special case discussed above.  The construction of these special functions guarantees that $w = v -v^1 - v^2$, $Q = q - q^1$  should satisfy 
\begin{equation}
 \begin{cases}
 - \Delta w + \nab Q = G^1 + \Delta v^1 + \Delta v^2 - \nab p^2 \in H^0(\Omega') & \text{in } \Omega' \\
\diverge{w}=  0  & \text{in } \Omega'\\
(QI - \sg w) \n = 0 & \text{on } \Sigma' \\
w =0 & \text{on }\Sigma_b.
 \end{cases}
\end{equation}
As above, there exist unique $w, Q$ solving this so that 
\begin{equation}\label{l_s_solve_6}
 \norm{w}_{H^2(\Omega')} + \norm{Q}_{H^1(\Omega')}  \ls C(\eta)  \norm{G^1 + \Delta v^1 + \Delta v^2 - \nab p^2}_{H^0(\Omega')}.
\end{equation}
The existence of unique $v,q$ solving \eqref{l_linear_stokes} is immediate, and the estimate \eqref{l_s_solve_1} follows by combining \eqref{l_s_solve_6} with \eqref{l_s_solve_3}--\eqref{l_s_solve_5}, finishing the proof of the claim.

\end{proof}

It turns out that we can achieve somewhat more of a regularity gain than is mentioned in Lemma \ref{l_stokes_solvable} by making a smallness assumption on $\eta$.  The smallness allows us to view the problem \eqref{l_linear_elliptic} as a perturbation of the Stokes problem on $\Omega$.  For this problem there is no constraint to regularity gain since the coefficients are constant and the boundary is smooth.  This allows us to shift the constraint of regularity gain to the regularity of $\eta$ in $H^{k+1/2}$ rather than in $C^{k-1}$.  We note that although we require $\eta \in H^{k+1/2}$, the smallness assumption is written in terms of $\norm{\eta}_{k-1/2}$.

\begin{prop}\label{l_stokes_regularity}
Let $k\ge 4$ be an integer and suppose that $\eta \in H^{k+1/2}$.  There exists $\ep_0 >0$ so that if $\norm{\eta}_{k-1/2} \le \ep_0$, then solutions to \eqref{l_linear_elliptic} satisfy
\begin{equation}\label{l_s_reg_0}
  \norm{u}_{r} + \norm{p}_{r-1} \le C \left( \norm{F^1}_{r-2} + \norm{F^2}_{r-1} + \norm{F^3}_{r-3/2}  \right)
\end{equation}
for $r=2,\dotsc,k$, whenever the right side is finite.   Here $C$ is a constant that does not depend on $\eta$.

In the case $r=k+1$,  solutions to \eqref{l_linear_elliptic} satisfy
\begin{multline}\label{l_s_reg_01}
\norm{u}_{k+1} + \norm{p}_{k} \le C \left( \norm{F^1}_{k-1} + \norm{F^2 }_{k} + \norm{F^3}_{k-1/2}  \right) \\
+ C \norm{\eta}_{k+1/2} \left( \norm{F^1}_{2} + \norm{F^2 }_{3} + \norm{F^3}_{5/2}  \right).
\end{multline}

\end{prop}

\begin{proof}

In the case that $\Sigma = \Rn{2}$, we let $\rho \in C^\infty_c(\Rn{2})$ be such that $\supp(\rho) \subset B(0,2)$ and $\rho(x) = 1$ for $x \in B(0,1)$.  For $m\in \mathbb{N}$ define $\eta^m$ by $\mathcal{F} \eta^m(\xi) = \rho(\xi/m)\mathcal{F}\eta(\xi)$, where $\mathcal{F}$ denotes the Fourier transform.  Clearly, for each $m$, $\eta^m \in H^j(\Sigma)$ for all $j\ge 0$, and also $\eta^m \to \eta$ in $H^{k-1/2}(\Sigma)$ (and in $H^{k+1/2}(\Sigma)$ if $\eta \in H^{k+1/2}(\Sigma)$) as $m \to \infty$.  In the periodic case, we similarly define $\eta^m$ by throwing away high frequencies: $\mathcal{F} \eta^m(n) = 0$ for $\abs{n} \ge m$.  In this case $\eta^m$ has the same convergence properties as before.  Let $\a^m$ and $\n^m$ be defined in terms of $\eta^m$.  Initially let $\ep_0$ be small enough so that $\eta^m$ is as small as in Remark \ref{C1_diff}.  This allows the mapping $\Phi^m$ defined by $\eta^m$ to be a $C^1$ diffeomorphism.

Consider the problem \eqref{l_linear_elliptic} with $\a$ and $\n$ replaced with $\a^m$ and $\n^m$.  Since $\eta^m \in H^{k+5/2}(\Sigma)$, we may apply Lemma \ref{l_stokes_solvable} to deduce the existence of a unique pair $(u^m,p^m)$ that solve \eqref{l_linear_elliptic} (with $\a^m,\n^m$) and that satisfy
\begin{equation}\label{l_s_reg_1}
 \norm{u^m}_{r} + \norm{p^m}_{r-1} \ls C(\norm{\eta^m}_{k+5/2}) \left( \norm{F^1}_{r-2} + \norm{F^2}_{r-1} + \norm{F^3}_{r-3/2}  \right) 
\end{equation}
for $r=2,\dotsc,k+1$, whenever the right hand side is finite.  We rewrite the equations \eqref{l_linear_elliptic} as a perturbation of the usual Stokes equations on $\Omega$:
\begin{equation}\label{l_s_reg_4} 
 \begin{cases}
 - \Delta u^m + \nab p^m =   F^1 + G^{1,m} & \text{in }\Omega \\
\diverge{u^m}   =  F^2  + G^{2,m} & \text{in }\Omega \\
(p^m I - \sg u^m) e_3 =   F^3 + G^{3,m} & \text{on } \Sigma \\
u^m =0 & \text{on }\Sigma_b.
 \end{cases}
\end{equation}
Suppose that $\norm{\eta^m}_{k+1/2} \le 1$, which implies that $\norm{\eta^m}_{k+1/2}^\ell \le \norm{\eta^m}_{k+1/2}$ for any $\ell \ge 1$.  This fact and a straightforward calculation reveal that 
\begin{equation}\label{l_s_reg_2}
\begin{split}
\norm{G^{1,m}}_{r-2} & \le C \norm{\eta^m}_{k-1/2}  \left(  \norm{u^m}_{r} + \norm{p^m}_{r-1} \right), \\
\norm{G^{2,m}}_{r-1} & \le C \norm{\eta^m}_{k-1/2}    \norm{u^m}_{r}, 
\end{split}
\end{equation}
and 
\begin{multline}\label{l_s_reg_3}
 \snormspace{G^{3,m}}{r-3/2}{\Sigma}  \le C \norm{\eta^m}_{k-1/2}  \left(  \snormspace{u^m}{r-1/2}{\Sigma} + \snormspace{p^m}{r-3/2}{\Sigma} \right) \\
\le C \norm{\eta^m}_{k-1/2}  \left(  \norm{u^m}_{r} + \norm{p^m}_{r-1}  \right)
\end{multline}
for $r=2,\dotsc,k$ and a constant $C>0$ independent of $\eta$ and $m$.    In the case $r=k+1$ a minor variant of this argument shows that
\begin{multline}\label{l_s_reg_3_2}
 \norm{G^{1,m}}_{k-1} + \norm{G^{2,m}}_{k} + \snormspace{G^{3,m}}{k-1/2}{\Sigma} \le C \norm{\eta^m}_{k-1/2}  \left(  \norm{u^m}_{k-1} + \norm{p^m}_{k} \right) \\
+  C \norm{\eta^m}_{k+1/2}\norm{u^m}_{7/2} 
\end{multline}
for $C$ independent of $\eta$ and $m$.  The key to this variant is that nowhere in the terms $G^{i,m}$ do there occur products of the highest derivative count of both $\eta^m$ and $u^m$ (or $p^m$).    Note that the right sides of \eqref{l_s_reg_2}, \eqref{l_s_reg_3}, and \eqref{l_s_reg_3_2} are finite by virtue of the estimate \eqref{l_s_reg_1}.

Since the boundaries $\Sigma$ and $\Sigma_b$ are smooth and the problem \eqref{l_s_reg_4} has constant coefficients, we may  argue as in Lemma \ref{l_stokes_solvable}, employing the elliptic estimates of \cite{adn_2} as done in  Lemma 3.3 of \cite{beale_1}, to arrive at the  estimate
\begin{equation}\label{l_s_reg_5}
\norm{u^m}_{r} + \norm{p^m}_{r-1} \le C \left( \norm{F^1 + G^{1,m}}_{r-2} + \norm{F^2 + G^{2,m}}_{r-1} + \norm{F^3+ G^{3,m}}_{r-3/2}  \right)
\end{equation}
for $r=2,\dotsc,k+1$ and for $C>0$ independent of $\eta$ and $m$.  We may then combine \eqref{l_s_reg_2}--\eqref{l_s_reg_3} with \eqref{l_s_reg_5} to find that, if $\norm{\eta^m}_{k-1/2} \le 1$, then
\begin{multline}\label{l_s_reg_6}
\norm{u^m}_{r} + \norm{p^m}_{r-1} \le C \left( \norm{F^1}_{r-2} + \norm{F^2 }_{r-1} + \norm{F^3}_{r-3/2}  \right) \\
 +  C \norm{\eta^m}_{k-1/2}  \left(  \norm{u^m}_{r} + \norm{p^m}_{r-1}  \right)
+ \delta_{r,k+1} C \norm{\eta^m}_{k+1/2}\norm{u^m}_{7/2}.
\end{multline}
On the right side of \eqref{l_s_reg_6} we have written $\delta_{r,k+1}$ for the quantity that vanishes when $r \neq k+1$ and is unity when $r=k+1$.  

We now derive the estimate \eqref{l_s_reg_0}.  Since $\eta^m \to \eta$ in $H^{k-1/2}$  we may assume that $m$ is sufficiently large so that
$\norm{\eta^m}_{k-1/2} \le 2 \norm{\eta}_{k-1/2}$.  Then if 
\begin{equation}\label{l_s_reg_7}
 \norm{\eta}_{k-1/2} \le \min\left\{\frac{1}{4C}, \frac{1}{2} \right\}:=\ep_0
\end{equation}
for $C>0$ the constant appearing on the right side of \eqref{l_s_reg_6}, the bound \eqref{l_s_reg_6} may be rearranged to get
\begin{equation}\label{l_s_reg_8}
\norm{u^m}_{r} + \norm{p^m}_{r-1} \le 2C \left( \norm{F^1}_{r-2} + \norm{F^2 }_{r-1} + \norm{F^3}_{r-3/2}  \right),
\end{equation}
for $r=2,\dotsc,k$ when the right side is finite.

The bound \eqref{l_s_reg_8} implies that the sequence $\{u^m,p^m\}$ is uniformly bounded in $H^r\times H^{r-1}$, so up to the extraction of a subsequence, $u^m \rightharpoonup u^0$ weakly in $H^r(\Omega)$ and  $p^m \rightharpoonup p^0$ weakly in $H^{r-1}(\Omega)$.  Since $\eta^m \to \eta$ in $H^{k-1/2}(\Sigma)$, we also have that $\a^m -\a \to 0$, $J^m -J \to 0$ in $H^{k-1}(\Omega)$, and $\n^m -\n \to 0$ in $H^{k-3/2}(\Sigma)$.  We multiply the equation $\diva u^m =F^2$ by $J^m w$ for $w \in C^\infty_c(\Omega)$ to see that
\begin{multline}
 \int_\Omega F^2 w J^m = \int_\Omega \diverge_{\a^m} (u^m) w J^m \\
= - \int_\Omega u^m \cdot \nab_{\a^m} w J^m \to 
- \int_\Omega u^0 \cdot \naba w J
= \int_\Omega \diva (u^0) w J,
\end{multline}
from which we deduce that $\diva(u^0) = F^2$.  Then we multiply the first equation in \eqref{l_linear_elliptic} (with $u^m$, etc) by $wJ^m$ for $w \in \H1$ and integrate by parts to see that
\begin{equation}
  \int_\Omega \hal \sg_{\a^m} u^m : \sg_{\a^m} w J^m - p^m \diverge_{\a^m}(w) J^m = \int_\Omega F^1 \cdot w J^m - \int_\Sigma F^3 \cdot w.
\end{equation}
Passing to the limit $m \to \infty$, we  deduce that
\begin{equation}
  \int_\Omega \hal \sg_{\a} u^0 : \sg_{\a} w J - p^0 \diva wJ = \int_\Omega F^1 \cdot w J - \int_\Sigma F^3 \cdot w,
\end{equation}
which reveals, upon integrating by parts again, that $u^0,p^0$ satisfy \eqref{l_linear_elliptic}.  Since $u,p$ are the unique solutions to \eqref{l_linear_elliptic}, we have that $u  = u^0$, $p=p^0$.  This, weak lower semi-continuity, and the bound \eqref{l_s_reg_8}  imply \eqref{l_s_reg_0}.

Now we derive the estimate \eqref{l_s_reg_01}, supposing that $F^1 \in H^{k-1},$ $F^2 \in H^{k}$, and $F^3 \in H^{k-1/2}$.  The bound \eqref{l_s_reg_8} with $r=4$ implies that 
\begin{equation}\label{l_s_reg_9}
 \norm{u^m}_{4}  \le 2C \left( \norm{F^1}_{2} + \norm{F^2 }_{3} + \norm{F^3}_{5/2}  \right) < \infty.
\end{equation}
Since $\eta^m \to \eta$ in $H^{k+1/2}$, we are free to assume that $m$ is sufficiently large so that $\norm{\eta^m}_{k+1/2} \le 2 \norm{\eta}_{k+1/2}$.  Then if $\norm{\eta}_{k-1/2}\le \ep_0$ we may use \eqref{l_s_reg_6} and \eqref{l_s_reg_9} to deduce that
\begin{multline}\label{l_s_reg_10}
\norm{u^m}_{k+1} + \norm{p^m}_{k} \le 2C \left( \norm{F^1}_{k-1} + \norm{F^2 }_{k} + \norm{F^3}_{k-1/2}  \right) \\
+ 4C \norm{\eta}_{k+1/2} \left( \norm{F^1}_{2} + \norm{F^2 }_{3} + \norm{F^3}_{5/2}  \right).
\end{multline}
We may then argue as above to extract weak limits, show that the limits equal $u$ and $p$, and then deduce that the bound \eqref{l_s_reg_10} holds with $u^m$ and $p^m$ replaced by $u$ and $p$.  This is \eqref{l_s_reg_01}.

\end{proof}

\subsection{The $\a-$Poisson problem}

Next we consider the scalar elliptic problem 
\begin{equation}\label{l_linear_elliptic_p}
 \begin{cases}
  \da p = f^1 & \text{in } \Omega \\
  p = f^2 & \text{on }\Sigma \\
   \naba p \cdot \nu = f^3 & \text{on } \Sigma_b,
 \end{cases}
\end{equation}
where $\nu$ is the outward-pointing normal on $\Sigma_b$.  We will eventually discuss the strong solvability of this problem, but first we consider the weak formulation of the problem.  We define a scalar $\h^1$ in a natural way through the norm 
\begin{equation}
 \hn{f}{1}^2 = \int_\Omega J \abs{\naba f}^2.
\end{equation}
Note that $\hn{f}{1}^2 = \hn{\sqrt{2} f e_1}{1}$, where the right side is the $\h^1$ norm for vectors.  Then Lemma \ref{l_norm_equivalence} shows that this scalar norm generates the same topology as the usual scalar $H^1$ norm.  

For the weak formulation we suppose $f^1 \in ({^0}H^1(\Omega))^*$, $f^2 \in H^{1/2}(\Sigma)$, and $f^3 \in H^{-1/2}(\Sigma_b)$.  Let $\bar{p}\in H^1(\Omega)$ be an extension of $f^2$ so that $\supp(\bar{p}) \subset \{ -(\inf b)/2 < x_3 \le 0 \}$.  We switch unknowns to $q = p - \bar{p}$.
Then we can define a weak formulation of \eqref{l_linear_elliptic_p} by finding a $q \in {^0}H^1(\Omega)$ so that
\begin{equation}\label{l_a_poisson_weak}
\iph{q}{\varphi}{1} = -\iph{\bar{p}}{\varphi}{1} -\br{f^1,\varphi}_{*} + \br{f^3,\varphi}_{-1/2} \text{ for all } \varphi \in {^0}H^1(\Omega),
\end{equation}
where  $\br{\cdot,\cdot}_{*}$ is the dual pairing with ${^0}H^1(\Omega)$ and $\br{\cdot,\cdot}_{-1/2}$ is the dual pairing with $H^{1/2}(\Sigma_b)$.  The existence and uniqueness of a solution to \eqref{l_a_poisson_weak} follows from standard arguments, and the resulting $p =q + \bar{p} \in H^1(\Omega)$ satisfies
\begin{equation}\label{l_a_poisson_weak_estimate}
 \hn{p}{1}^2 \ls \left( \norm{f^1}_{({^0}H^1(\Omega))^*}^2 + \snormspace{f^2}{1/2}{\Sigma}^2 +  \snormspace{f^3}{-1/2}{\Sigma_b}^2   \right).
\end{equation}

In the event that the action of $f^1$ is given in a more specific fashion, we will rewrite the PDE \eqref{l_linear_elliptic_p} to accommodate the structure of $f^1$.  To make this precise, suppose that the action of $f^1$ on an element $\varphi \in {^0}H^1(\Omega)$ is given by
\begin{equation}
 \br{f^1,\varphi}_* = \iph{g_0}{\varphi}{0} + \iph{G}{\naba \varphi}{0}
\end{equation}
for $(g_0, G) \in H^0(\Omega;\Rn{}) \times H^0(\Omega;\Rn{3})$ with $\norm{g_0}_{0}^2 + \norm{G}_{0}^2 = \norm{f^1}_{({^0}H^1(\Omega))*}^2$ (standard arguments show that it is always possible to uniquely write $f^1$ in this way).  
Then \eqref{l_a_poisson_weak} may be rewritten as
\begin{equation}\label{l_a_poisson_weak_2}
\iph{\naba p+G}{\naba \varphi}{0} = - \iph{g_0}{\varphi}{0} + \br{f^3,\varphi}_{-1/2} \text{ for all } \varphi \in {^0}H^1(\Omega).
\end{equation}
We may take $\varphi \in C_c^\infty(\Omega)$ in this equality and integrate by parts to see that $\diva(\naba p + G) = g_0 \in \h^0$, which allows us to deduce from Lemma \ref{l_boundary_dual_estimate} that $(\naba p + G)\cdot \nu \in H^{-1/2}(\Sigma_b)$.  This serves as motivation for us to say that $p$ is a weak solution to the PDE
\begin{equation}\label{l_a_poisson_div}
 \begin{cases}
  \diva(\naba p +G) = g_0 \in H^0(\Omega)  \\
  p = f^2 \in H^{1/2}(\Sigma)  \\
   (\naba p +G)\cdot \nu = f^3 \in H^{-1/2}(\Sigma_b).
 \end{cases} 
\end{equation}
This way of writing the weak solution will be utilized later in Theorem \ref{l_strong_solution}.  Note that when $f^1 \in H^0(\Omega)$, there is no need to make this distinction since then $G =0$ and $f^1 =g_0$.

Our next result on this problem is the analogue of Lemma \ref{l_stokes_solvable}; it establishes the strong solvability of \eqref{l_linear_elliptic_p} and some regularity.

\begin{lem}\label{l_a_poisson_solvable}
Suppose that $\eta \in H^{k+1/2}(\Sigma)$ for $k\ge 3$ is as small as in Remark \ref{C1_diff} so that the mapping $\Phi$ defined by \eqref{mapping_def} is a $C^1$ diffeomorphism of $\Omega$ to $\Omega'=\Phi(\Omega).$  If $f^1 \in H^0(\Omega)$, $f^2\in H^{3/2}(\Sigma)$, and $f^3 \in H^{1/2}(\Sigma_b)$, then the problem \eqref{l_linear_elliptic_p} admits a unique strong solution $p \in H^2(\Omega)$.  Moreover, for $r = 2,\dotsc,k-1$ we have the estimate
\begin{equation}\label{l_ap_solve_0}
 \norm{p}_{r} \ls C(\eta) \left( \norm{f^1}_{r-2} + \norm{f^2}_{r-1/2} + \norm{f^3}_{r-3/2}  \right),
\end{equation}
whenever the right hand side is finite, where $C(\eta)$ is a constant depending on $\norm{\eta}_{k+1/2}$.
\end{lem}

\begin{proof}
If $f^2 \in H^{r-1/2}(\Sigma)$ for $r=2,\dotsc,k-1$, there exists a $\psi \in H^{r}(\Omega)$ so that $\psi \vert_{\Sigma} = f^2$, $\supp(\psi) \subset \{-(\inf b)/2 < x_3 \le 0\}$, and $\norm{\psi}_{r} \ls \norm{f^2}_{r-1/2}$.  Writing $p = q+ \psi$, the problem \eqref{l_linear_elliptic_p} may be rewritten for the unknown $q$ as
\begin{equation}\label{l_ap_solve_1}
 \begin{cases}
  \da q = f^1 + g^1 & \text{in } \Omega \\
  q = 0 & \text{on }\Sigma \\
  \naba q\cdot \nu = f^3 & \text{on } \Sigma_b,
 \end{cases}
\end{equation}
where $g^1 = -\da \psi \in H^{r-2}$.  

The problem \eqref{l_ap_solve_1} may be solved as in Lemma \ref{l_stokes_solvable} by transforming to the domain $\Omega'$, where the problem for $Q=q \circ \Phi^{-1}$ becomes $\Delta Q = (f^1 + g^1)\circ \Phi^{-1}$ in $\Omega'$ with boundary conditions $Q =0$ on $\Sigma'$ and $\nab Q \cdot \nu=f^3 \circ \Phi^{-1}$ on $\Sigma_b$.  The existence of a unique solution to this problem is established in the non-periodic case in Lemma 2.8 of \cite{beale_1}, and estimates of the form \eqref{l_ap_solve_0} for $Q$ hold by virtue of the elliptic estimates in \cite{adn_1}, adapted to $\Omega'$ as in \cite{beale_1}.  This method may be adapted easily to the periodic case as well.  Then the existence and uniqueness of a solution to \eqref{l_linear_elliptic_p} satisfying \eqref{l_ap_solve_0} follows by transforming to $q = Q\circ \Phi$ on $\Omega$ for a solution to \eqref{l_ap_solve_1} and then applying Lemma \ref{l_sobolev_composition}.
\end{proof}

Our next result is the analogue of Proposition \ref{l_stokes_regularity} for the problem \eqref{l_linear_elliptic_p}.  For our purposes, we only need a  regularity gain up to $k$, and this is less important than the estimate in terms of a constant independent of $\eta$.  Notice again that the smallness assumption is stated in $H^{k-1/2}$ even though we require $\eta \in H^{k+1/2}$.

\begin{prop}\label{l_a_poisson_regularity}
Let $k\ge 4$ be an integer and suppose that $\eta \in H^{k+1/2}$.  There exists $\ep_0 >0$ so that if $\norm{\eta}_{k-1/2} \le \ep_0$, then solutions to \eqref{l_linear_elliptic_p} satisfy
\begin{equation}\label{l_ap_reg_0}
   \norm{p}_{r} \le C \left( \norm{f^1}_{r-2} + \norm{f^2}_{r-1/2} + \norm{f^3}_{r-3/2}  \right)
\end{equation}
for $r=2,\dotsc,k$, whenever the right side is finite.   Here $C$ is a constant that does not depend on $\eta$. 
\end{prop}
\begin{proof}
The proof is similar to that of Proposition \ref{l_stokes_regularity}.  We smooth $\eta$ to get $\eta^m$ and solve \eqref{l_linear_elliptic_p} with $\a$ replaced with $\a^m$.  Then we rewrite the problem as a perturbation of the Poisson problem
\begin{equation}
 \begin{cases}
  \Delta p^m = f^1 + g^{1,m} & \text{in }\Omega \\
  p^m = f^2 & \text{on }\Sigma \\
  \nab p^m \cdot \nu = f^3 + g^{3,m} \text{on }\Sigma_b.
 \end{cases}
\end{equation}
The constants in the elliptic estimates for this problem do not depend on $\eta^m$, and we may estimate $g^{i,m}$ in terms of $p^m$.  Then if $\norm{\eta}_{k-1/2} \le \ep_0$ for some $\ep_0$ sufficiently small, we can absorb the highest Sobolev norms on the right side of the elliptic estimate into the left side, and we deduce \eqref{l_ap_reg_0} for $p^m$.  Then we pass to the limit $m \to \infty$.

\end{proof}

\section{Solving the time-dependent problem \eqref{l_linear_forced}}\label{lwp_3}

\subsection{The weak solution}

In our analysis of problem \eqref{l_linear_forced} we will employ two notions of solution: weak and strong.  The definition of a weak solution to \eqref{l_linear_forced} is motivated by assuming the existence of a smooth solution to \eqref{l_linear_forced}, multiplying by $J v$ for  $v \in \h^{1}_T$,  integrating over $\Omega$ by parts, and then in time from $0$ to $T$ to see that
\begin{equation}\label{l_weak_motivation}
 \ip{\dt u}{v}_{\h^0_T} + \hal \ip{u}{v}_{\h^1_T} - \ip{p}{\diva v}_{\h^0_T} = \ip{F^1}{v}_{\h^0_T} - \ip{F^3}{v}_{0,\Sigma,T}
\end{equation}
for $ \ip{F^3}{v}_{0,\Sigma,T} = \int_0^T \int_\Sigma F^3 \cdot v.$  Suppose that 
\begin{equation}\label{l_weak_data_assumptions}
 F^1 \in (\h^1_T)^*, F^3 \in L^2([0,T];H^{-1/2}(\Sigma)), \text{ and } u_0 \in \y(0),
\end{equation}
where $\y(0)$ is defined by \ref{l_Y_space_def}.  Then our definition of a weak solution of \eqref{l_linear_forced} requires only that a relaxed form of \eqref{l_weak_motivation} holds.  In particular, we say that $(u,p)$ is a weak solution of \eqref{l_linear_forced} if
\begin{equation}\label{l_weak_solution_pressure}
 \begin{cases}
  u \in \x_T, \dt u \in (\h^1_T)^*, p \in \h^0_T, &  \\
  \br{\dt u,v}_{*}  + \hal \ip{u}{v}_{\h^1_T}  - \ip{p}{\diva v}_{\h^0_T} = \br{F^1,v}_{*} - \br{F^3,v}_{-1/2} & \text{for every } v \in \h^1_T, \\
u(0) = u_0, & 
 \end{cases}
\end{equation}
where $\br{\cdot,\cdot}_{*}$ denotes the dual pairing between $(\h^1_T)^*$ and $\h^1_T$, and $\br{\cdot ,\cdot }_{-1/2}$ denotes the dual pairing between $L^2([0,T];H^{-1/2}(\Sigma))$ and $L^2([0,T];H^{1/2}(\Sigma))$.  The third condition in \eqref{l_weak_solution_pressure} only makes sense in light of Lemma \ref{l_x_time_diff}.

If we were to restrict our class of test functions in \eqref{l_weak_solution_pressure} to $v \in \x_T$, then the term $\ip{p}{\diva v}_{\h^0_T}$ would vanish, and we would be left with a ``pressureless'' weak formulation of the problem involving only the velocity field.  This leads us to define a weak formulation without the pressure.  Suppose the data satisfy \eqref{l_weak_data_assumptions}. Then $u$ is a pressureless weak solution of \eqref{l_linear_forced} if 
\begin{equation}\label{l_weak_solution_pressureless}
 \begin{cases}
  u \in \x_T, \dt u \in (\h^1_T)^*,&  \\
  \br{\dt u,\psi}_{*}  + \hal \ip{u}{\psi}_{\h^1_T}  = \br{F^1,\psi}_{*} - \br{F^3,\psi}_{-1/2} & \text{for every }\psi \in \x_T, \\
u(0) = u_0. & 
 \end{cases}
\end{equation}
A more natural assumption for this formulation would be to require $\dt u \in (\x_T)^*.$  However, since $\x_T \subset \h^1_T$, the usual theory of Hilbert spaces provides a unique operator $E: (\x_T)^*\to (\h^1_T)^*$ with the property that $Ef\vert_{\x_T} = f$  and  $\norm{Ef}_{(\h^1_T)^*}= \norm{f}_{(\x_T)^*}$ for all $f \in (\x_T)^*$.  Using this $E$, we  regard $\dt u \in (\x_T)^*$ as an element of $(\h^1_T)^*$ in a natural way, which allows us to require  that $\dt u \in (\h^1_T)^*$.

Since our aim is to construct solutions to \eqref{l_linear_forced} with high regularity, we will not need to directly construct weak solutions to \eqref{l_weak_solution_pressureless} or \eqref{l_weak_solution_pressure}.  Rather, weak solutions to problems of this type will arise as a byproduct of our construction of strong solutions of \eqref{l_linear_forced}.  As such, for our purposes, it will suffice to ignore the issue of existence and only record a couple results on the properties of weak solutions.

We now record a result on some integral equalities and bounds satisfied by solutions of \eqref{l_weak_solution_pressureless}.

\begin{lem}\label{l_weak_solutions_integration}
Suppose that $u$ is a weak solution of \eqref{l_weak_solution_pressureless}.  Then for a.e. $t \in [0,T]$, 
\begin{multline}\label{l_ws_int_01}
\hal \norm{u(t)}_{\h^0(t)}^2 + \hal \int_0^t \norm{u(s)}_{\h^1(s)}^2 ds = \hal \norm{u(0)}_{\h^0(0)}^2 + \int_0^t \br{F^1(s),u(s)}_{(\h^1(s))^*}ds \\
- \int_0^t \br{F^3(s),u(s)}_{H^{-1/2}(\Sigma)}ds + \hal \int_0^t \int_\Omega \abs{u(s)}^2 \dt J(s) ds.
\end{multline}
Also
\begin{equation}\label{l_ws_int_02}
 \sup_{0\le t \le T} \norm{u(t)}_{\h^0(t)}^2 + \norm{u}_{\h^1_T}^2 \ls \exp\left(C_0(\eta) T \right) \left(\norm{u(0)}_{\h^0(0)}^2  + \norm{F^1}_{(\h^1_T)^*}^2 + \norm{F^3}_{L^2H^{-1/2}}^2  \right),
\end{equation}
where $ C_0(\eta) := \sup_{0\le t \le T} \pnorm{\dt J K}{\infty}.$
\end{lem}

\begin{proof}
The identity \eqref{l_ws_int_01} follows directly from \eqref{l_weak_solution_pressureless} and Lemma \ref{l_x_time_diff} by using the test function $\psi = u \chi_{[0,t]} \in \x_T$, where $\chi_{[0,t]}$ is a temporal indicator function equal to unity on the interval $[0,t]$.

From \eqref{l_ws_int_01} it is straightforward to derive the  inequality
\begin{multline}\label{l_ws_int_1}
\hal \norm{u(t)}_{\h^0(t)}^2 + \hal \norm{u}_{\h^1_t}^2 \le \hal \norm{u(0)}_{\h^0(0)}^2 +\norm{F^1}_{(\h^1_t)^*} \norm{u}_{\h^1_t} \\ + \norm{F^3}_{L^2([0,t];H^{-1/2})} \norm{u}_{L^2([0,t];H^{1/2})}  
 + \frac{C_0(\eta)}{2} \norm{u}_{\h^0_t}^2,
\end{multline}
where we have written 
\begin{equation}
 \norm{u}_{\h^k_t}^2 = \int_0^t \norm{u(s)}_{\h^k(s)}^2 ds  \text{ for }k=0,1,
\end{equation}
and similarly defined $\norm{F^1}_{(\h^1_t)^*}$.  Note that, according to Remark \ref{l_A_korn_trace}, we may control $\snormspace{u}{1/2}{\Sigma} \le C \hn{u}{1}$ for a constant $C$ independent of $\eta$.  This, inequality \eqref{l_ws_int_1}, and Cauchy's inequality then imply that
\begin{multline}\label{l_ws_int_3}
\hal \norm{u(t)}_{\h^0(t)}^2 + \frac{1}{8} \norm{u}_{\h^1_t}^2 \le \hal \norm{u(0)}_{\h^0(0)}^2 +2 \norm{F^1}_{(\h^1_t)^*}^2 \\
+ 2 C \norm{F^3}_{L^2([0,t];H^{-1/2})}^2   
 + \frac{C_0(\eta)}{2} \norm{u}_{\h^0_t}^2.
\end{multline}
Then \eqref{l_ws_int_02} follows from the differential inequality \eqref{l_ws_int_3} and Gronwall's lemma.
\end{proof}

We can now parlay the results of Lemma \ref{l_weak_solutions_integration} into uniqueness results for weak solutions to \eqref{l_weak_solution_pressureless} and \eqref{l_weak_solution_pressure}.

\begin{prop}\label{l_weak_unique}
Weak solutions to \eqref{l_weak_solution_pressureless} are unique.  Also, weak solutions $(u,p)$ to \eqref{l_weak_solution_pressure} are unique.
\end{prop}
\begin{proof}
If $u^1$ and $u^2$ are both weak solutions to \eqref{l_weak_solution_pressureless}, then $w = u^1 - u^2$ is a weak solution with $F^1=0$, $F^3=0$, and $w(0) = u^1(0) - u^2(0) = 0$.  Then the bound \eqref{l_ws_int_02} of Lemma \ref{l_weak_solutions_integration} implies that $w=0$; hence solutions to \eqref{l_weak_solution_pressureless} are unique.

Now, if $(u,p)$ are a weak solution to \eqref{l_weak_solution_pressure}, then we can restrict to test functions $\psi \in \x_T$ to find that $u$ is a weak solution to \eqref{l_weak_solution_pressureless}.  As such, $u$ is unique.  To see that $p$ is unique we define $\Lambda \in (\h^1_T)^*$ via
\begin{equation}
 \Lambda(v) =   \br{\dt u,v}_{*}  + \hal \ip{u}{v}_{\h^1_T}  - \br{F^1,v}_{*} + \br{F^3,v}_{-1/2}.
\end{equation}
Since $u$ is a weak solution to \eqref{l_weak_solution_pressureless}, we have that $\Lambda(v) = 0$ for all $v \in \x_T$.  Proposition \ref{l_pressure_decomp} them implies that there exists a unique $q \in \h^0_T$ so that 
$\ip{q}{\diva v}_{\h^0_T} = \Lambda(v)$   for all $v \in \h^1_T.$  It follows that $q=p$ and that $p$ is unique.

\end{proof}

\subsection{The strong solution}

Now we turn to the construction of strong solutions to \eqref{l_linear_forced}.  We will make stronger assumptions on the data $F^1,F^3$, $u_0$ than we made in the weak formulation \eqref{l_weak_data_assumptions}.  In particular, we will assume that the forcing functions satisfy
\begin{equation}\label{l_F_assumptions}
\begin{split}
F^1 &\in L^2([0,T]; H^1(\Omega)),  \dt F^1  \in L^2([0,T]; (\H1)^*),  \\
F^3 &\in L^2([0,T]; H^{3/2}(\Sigma)), \dt F^3 \in L^2([0,T]; H^{-1/2}(\Sigma)), \\
F^1(0)&\in H^0(\Omega), F^3(0) \in H^{1/2}(\Sigma).
\end{split}
\end{equation}
Note that, owing to Lemma \ref{l_x_time_diff}, \eqref{l_F_assumptions} implies that $F^1 \in C^0([0,T]; H^0(\Omega))$ and  $F^3 \in C^0([0,T]; H^{1/2}(\Sigma))$.
The initial data will also be taken to be more regular; we take $u_0 \in H^2(\Omega) \cap \x(0)$.

The solution that we construct will satisfy \eqref{l_linear_forced} in the strong sense, but we will also show that $(D_t u,\dt p)$ satisfy an equation of the form \eqref{l_linear_forced} in the weak sense of \eqref{l_weak_solution_pressure}.  Here we define
\begin{equation}\label{l_Dt_def}
 D_t u := \dt u - R u \text{ for } R:= \dt M M^{-1}
\end{equation}
with $M$ the matrix defined by \eqref{l_M_def}.  We employ the operator $D_t$ because it preserves the $\diva-$free condition.  Before turning to the result, we define the quantity 
\begin{equation}\label{l_K_def}
 \mathcal{K}(\eta) : = \sup_{0 \le t \le T} \left( \norm{\eta}_{9/2}^2 + \norm{\dt \eta}_{7/2}^2 + \norm{\dt^2 \eta}_{5/2}^2 \right).
\end{equation}
We also define an orthogonal projection onto the tangent space of the surface $\{x_3 = \eta_0\}$ according to 
\begin{equation}\label{l_Pi0_def}
 \Pi_0 v = v - (v\cdot \n_0) \n_0 \abs{\n_0}^{-2}
\end{equation}
for $\n_0 = (-\p_1 \eta_0,-\p_2 \eta_0,1)$.  By construction, $\Pi_0 v = 0$ if and only if $v \parallel \n_0$.

\begin{thm}\label{l_strong_solution}
Suppose that $F^1,F^3$ satisfy \eqref{l_F_assumptions}, that $u_0 \in H^2(\Omega) \cap \x(0)$, and that $u_0$, $F^3(0)$ satisfy the compatibility condition
\begin{equation}\label{l_ss_01}
 \Pi_0 \left( F^3(0) + \sg_{\a_0} u_0 \n_0 \right) =0, \text{where } \n_0 = (-\p_1 \eta_0,-\p_2 \eta_0,1),
\end{equation}
and $\Pi_0$ is the projection defined by \eqref{l_Pi0_def}.  Further suppose that $\mathcal{K}(\eta)$ is less than the smaller of $\ep_0$ from Lemma \ref{l_norm_equivalence} and $\ep_0$ from Proposition \ref{l_stokes_regularity} (in particular, this requires $\mathcal{K}(\eta) \le 1$).  Then there exists a unique strong solution $(u,p)$ to \eqref{l_linear_forced} so that 
\begin{equation}\label{l_ss_06}
\begin{split}
 u & \in \x_T \cap C^0([0,T]; H^2(\Omega)  ) \cap L^2([0,T]; H^3(\Omega)  ), \\
 \dt u & \in C^0([0,T]; H^0(\Omega) ) \cap L^2([0,T]; H^1(\Omega)  ),  \dt^2 u \in (\h^1_T)^*, \\
 p & \in C^0([0,T]; H^1(\Omega) ) \cap L^2([0,T]; H^2(\Omega) ), \dt p \in L^2([0,T]; H^0(\Omega) ).
\end{split}
\end{equation}
The solution satisfies the estimate
\begin{multline}\label{l_ss_02}
   \norm{u}_{L^\infty H^2}^2 + \norm{u}_{L^2 H^3}^2 + \norm{\dt u}_{L^\infty H^0}^2 + \norm{\dt u}_{L^2 H^1}^2  + \norm{\dt^2 u}_{(\h^1_T)^*}^2+ \norm{p}_{L^\infty H^1}^2 + \norm{\dt p}_{L^2 H^0}^2 \\
\ls  (1+ \mathcal{K}(\eta)) \exp\left(C(1+ \mathcal{K}(\eta)) T \right) 
\left( \norm{u_0}_{2}^2 + \norm{F^1(0)}_{0}^2 + \norm{F^3(0)}_{ 1/2}^2 \right. \\
+  \left.   \norm{F^1}_{L^2H^1}^2 + \norm{\dt F^1}_{L^2 (\H1)^*}^2 
+ \norm{F^3}_{L^2H^{3/2}}^2 + \norm{ \dt F^3}_{L^2 H^{-1/2}}^2 \right),
\end{multline}
where $C$ is a constant independent of $\eta$.  The initial pressure, $p(0)\in H^1(\Omega)$, is determined in terms of $u_0, F^1(0), F^3(0)$ as the weak solution to
\begin{equation}\label{l_ss_04}
 \begin{cases}
  \diverge_{\a_0}(\nab_{\a_0} p(0) - F^1(0)) = -\diverge_{\a_0} (R(0) u_0 ) \in H^0(\Omega) \\
  p(0) = ( F^3(0)  + \sg_{\a_0} u_0  \n_0 )\cdot \n_0 \abs{\n_0}^{-2} \in H^{1/2}(\Sigma) \\
  (\nab_{\a_0} p(0) - F^1(0)) \cdot \nu = \Delta_{\a_0} u_0 \cdot \nu \in H^{-1/2}(\Sigma_b)
 \end{cases}
\end{equation}
in the sense of \eqref{l_a_poisson_div}. Also, $D_t u(0) = \dt u(0) - R(0) u_0$ satisfies 
\begin{equation}\label{l_ss_05}
 D_t u(0) = \Delta_{\a_0} u_0 - \nab_{\a_0} p(0) + F^1(0) - R(0) u_0 \in \y(0),
\end{equation}
where $\y(0)$ is defined by \eqref{l_Y_space_def}.

Moreover, $(D_t u , \dt p)$ satisfy
\begin{equation}\label{l_ss_03}
  \begin{cases}
\dt (D_t u) - \da (D_t u) + \naba (\dt p) = D_t F^1 + G^1 & \text{in }\Omega \\
\diva(D_t u)=0 & \text{in }\Omega\\
\Sa(\dt p, D_t u) \n = \dt F^3 + G^3 & \text{on }\Sigma \\
D_t u =0 & \text{on }\Sigma_b,
 \end{cases}
\end{equation}
in the weak sense of \eqref{l_weak_solution_pressure}, where $G^1,G^3$ are defined by 
\begin{equation}
 G^1 =   -(R + \dt J K) \da u - \dt R u + (\dt J K + R - R^T) \naba p + 
\diva( \sg_{\a} (Ru) + R \sg_{\a}u + \sg_{\dt \a} u )
\end{equation}
with $R^T$ denoting the matrix transpose of $R$, and 
\begin{equation}
 G^3 = \sg_{\a}(R u) \n - (p I - \sg_{\a} u) \dt \n + \sg_{\dt \a} u \n.
\end{equation}
Here the inclusions \eqref{l_ss_06} guarantee that $G^1$ and $G^3$ satisfy the same inclusions as $F^1, F^3$ listed in \eqref{l_F_assumptions}, whereas \eqref{l_ss_04} guarantees that the initial data $D_t u(0) \in \y(0)$.
\end{thm}

\begin{proof}
The result will be established by first solving a pressureless problem and then introducing the pressure via Proposition \ref{l_pressure_decomp}.  For the pressureless problem we will make use of the Galerkin method.  We divide the proof into several steps.  

Step 1 -- The Galerkin setup

In order to utilize the Galerkin method, we must first construct a countable basis of $H^2(\Omega) \cap \x(t)$ for each $t \in [0,T]$.  Since the requirement $\diva v=0$ is time-dependent, any basis of this space must also be time-dependent.  For each $t \in [0,T]$, the space $H^2(\Omega) \cap \x(t)$ is separable, so the existence of a countable basis is not an issue.  The technical difficulty is that, in order for the basis to be useful in the Galerkin method, we must be able differentiate the basis elements in time, and we must be able to express these time derivatives in terms of finitely many basis elements.  Fortunately, it is possible to overcome this difficulty by employing the matrix $M(t)$, defined by \eqref{l_M_def}.

Since $H^2(\Omega) \cap \Hsig$ is separable, it possesses a countable basis $\{w^j \}_{j=1}^\infty$.  Note that this basis is not time-dependent.  Define $\psi^j=\psi^j(t) := M(t) w^j$.  According to Proposition \ref{l_M_iso}, $\psi^j(t) \in H^2(\Omega) \cap \x(t)$, and $\{\psi^j(t)\}_{j=1}^\infty$ is a basis of $H^2(\Omega) \cap \x(t)$ for each $t \in [0,T]$.  Moreover, 
\begin{equation}\label{l_ss_1}
 \dt \psi^j(t) = \dt M(t) w^j = \dt M(t)M^{-1}(t) M(t) w^j = \dt M(t)M^{-1}(t) \psi^j(t) := R(t) \psi^j(t),
\end{equation}
which allows us to express $\dt \psi^j$ in terms of $\psi^j$.  For any integer $m\ge 1$ we define the finite dimensional space $\x_m(t):=\text{span}\{\psi^1(t),\dotsc,\psi^m(t)\} \subset H^2(\Omega) \cap \x(t)$, and we write $\mathcal{P}^m_t: H^2(\Omega) \to \x_m(t)$ for the $H^2(\Omega)$ orthogonal projection onto $\x_m(t)$.  Clearly, for each $v \in H^2(\Omega) \cap \x(t)$ we have that $\mathcal{P}^m_t v \to v$ as $m \to \infty$.

The next ingredient needed for the Galerkin method is the orthogonal projection onto the tangent space of the surface $\{x_3 = \eta(0)\}$, $\Pi_0$, defined by \eqref{l_Pi0_def}.  This projection will be used to compensate for the fact that our finite-dimensional Galerkin approximation of the initial data $u_0$ may fail to satisfy the compatibility conditions \eqref{l_ss_01}.

Step 2 -- Solving the Galerkin problem

For our Galerkin problem we will first construct a solution to the pressureless problem as follows.  For each $m\ge 1$ we define an approximate solution
\begin{equation}
 u^m(t) = a^m_j(t) \psi^j(t), \text{with } a^m_j :[0,T] \to \Rn{} \text{ for }j=1,\dotsc,m,
\end{equation}
where as usual we use the Einstein convention of summation of the repeated index $j$.  We want to choose the coefficients $a^m_j$ so that 
\begin{equation}\label{l_ss_2}
  \iph{\dt u^m}{\psi}{0}  + \hal \iph{u^m}{\psi}{1} = \iph{F^1}{\psi}{0} - \ip{F^3 - \Pi_0(F^3(0) + \sg_{\a_0} (\mathcal{P}^m_0 u_0)  \n_0  ) }{\psi}_{0,\Sigma}  
\end{equation}
for each $\psi \in \x_m(t)$, where we have written $\ip{\cdot}{\cdot}_{0,\Sigma}$ for the usual $H^0(\Sigma)$ inner-product, and where $\Pi_0$ and $\mathcal{P}^m_0$ are defined in the previous step.  We supplement the equation \eqref{l_ss_2} with the initial condition 
\begin{equation}\label{l_ss_3}
 u^m(0) = \mathcal{P}^m_0 u_0 \in \x_m(0).
\end{equation}
Appealing to \eqref{l_ss_1}, we find that $\dt u^m(t) = \dot{a}^m_j(t) \psi^j(t) + R(t) u^m(t)$, and hence \eqref{l_ss_2} is equivalent to the system of ODEs for $a^m_j$ given by
\begin{multline}\label{l_ss_4}
 \dot{a}_j^m \iph{\psi^j }{\psi^k}{0} +  a^m_j \left( \iph{R(t) \psi^j }{\psi^k}{0} + \hal \iph{\psi^j}{\psi^k}{1} \right)  \\
= \iph{F^1}{\psi^k}{0} - \ip{F^3 - \Pi_0(F^3(0) + \sg_{\a_0} u^m(0) \n_0  ) }{\psi^k}_{0,\Sigma}
\end{multline}
for $j,k=1,\cdots,m$.  The $m \times m$ matrix with $j,k$ entry $\iph{\psi^j }{\psi^k}{0}$ is invertible, the coefficients of the linear system \eqref{l_ss_4} are  $C^1([0,T])$, and the forcing term is $C^0([0,T])$, so the usual well-posedness theory of ODEs guarantees the existence of $a^m_j \in C^1([0,T])$, a unique solution to \eqref{l_ss_4} that satisfies the initial conditions induced by \eqref{l_ss_3}.  This, in turn, provides the desired solution, $u^m$, to \eqref{l_ss_2}--\eqref{l_ss_3}.  Since $F^1,F^3$ satisfy \eqref{l_F_assumptions}, the equation \eqref{l_ss_4} may be differentiated in time to see that actually $a^m_j \in C^{1,1}([0,T])$, with $a^m_j$ twice differentiable a.e. in $[0,T]$.

Note that throughout the rest of the proof, we use constants $C$ and the symbol $\ls$ with the assumption that the constants do not depend on $m$.

Step 3 -- Energy estimates for $u^m$

Since $u^m(t) \in \x_m(t)$, we may use $\psi = u^m$ as a test function in \eqref{l_ss_2}.  Doing so, employing Remark \ref{l_A_korn_trace}, and using the fact that $\Pi_0$ is an orthogonal projection,  we may derive the bound
\begin{multline}\label{l_ss_5}
 \dt \hal \hn{u^m}{0}^2 +  \hal \hn{u^m}{1}^2 \le C\hn{F^1}{0} \hn{u^m}{1} - \hal \int_\Omega \abs{u^m}^2 \dt J
\\  + C \hn{u^m}{1} \left( \snormspace{F^3}{1/2}{\Sigma} + \snormspace{F^3(0) + \sg_{\a_0} u^m(0) \n_0 }{0}{\Sigma}  \right).
\end{multline}
We may then apply Cauchy's inequality to \eqref{l_ss_5} to find that
\begin{multline}\label{l_ss_7}
 \dt \hal \hn{u^m}{0}^2 +  \frac{1}{8} \hn{u^m}{1}^2 \le C \snormspace{F^3(0) + \sg_{\a_0} u^m(0) \n_0 }{0}{\Sigma}^2 \\
+ C \left( \hn{F^1}{0}^2 +  \snormspace{F^3}{1/2}{\Sigma}^2 \right) 
+ C_0(\eta) \hal \hn{u^m}{0}^2
\end{multline}
for $C_0(\eta):= 1+ \sup_{0\le t \le T} \pnorm{\dt J K}{\infty}$.  Note that since $\mathcal{P}^m_0$ is the $H^2(\Omega)$ orthogonal projection, we may use Proposition \ref{l_norm_equivalence} to bound
\begin{equation}\label{l_ss_8}
 \hn{u^m(0)}{0} \le 2 \norm{u^m(0)}_{0} \le 2 \norm{u^m(0)}_{2} =  2 \norm{\mathcal{P}^m_0 u_0}_{2} \le 2 \norm{u_0}_{2}.
\end{equation}
Now we can apply Gronwall's lemma to the differential inequality \eqref{l_ss_7} and utilize \eqref{l_ss_8} to deduce energy estimates for $u^m$:
\begin{multline}\label{l_ss_9}
 \sup_{0\le t \le T} \hn{u^m}{0}^2 + \norm{u^m}_{\h^1_T}^2  
\le \sup_{0\le t \le T} \hn{u^m}{0}^2 + \int_0^T \exp\left( C_0(\eta) (T-s) \right) \hn{u^m(s)}{1}^2 ds
\\ 
\ls \exp\left(C_0(\eta) T \right) \left( \snormspace{F^3(0) + \sg_{\a_0} u^m(0) \n_0 }{0}{\Sigma}^2 +   \norm{u_0}_{2}^2   + \norm{F^1}_{\h^0_T}^2 +  \norm{F^3}_{L^2H^{1/2} }^2 \right).
\end{multline}

Step 4 -- Estimate of $\hn{\dt u^m(0)}{0}$

We will eventually derive energy estimates for $\dt u^m$ similar to those derived in the previous step for $u^m$, but first we must be able to estimate $\hn{\dt u^m(0)}{0}$.  If $u \in H^2(\Omega) \cap \x(t)$, $\psi \in \h^1$, then an integration by parts reveals that
\begin{equation}\label{l_ss_10}
 \hal \iph{u}{\psi}{1} =   \int_\Omega -\da u \cdot \psi J + \int_\Sigma (\sg_{\a} u \n)\cdot \psi = \iph{-\da u}{\psi}{0} + \ip{\sg_{\a} u \n}{\psi}_{0,\Sigma}.
\end{equation}
Evaluating \eqref{l_ss_2} at $t=0$ and employing \eqref{l_ss_10}, we find that
\begin{equation}\label{l_ss_11}
  \iph{\dt u^m(0)}{\psi}{0}  = \iph{\Delta_{\a_0} u^m(0) + F^1(0)}{\psi}{0} - \ip{\Pi_0^\bot (F^3(0) + \sg_{\a_0} u^m(0)  \n_0  ) }{\psi}_{0,\Sigma}
\end{equation}
for all $\psi \in \x_m(0)$, where we have written $\Pi_0^\bot = I - \Pi_0$ for the orthogonal projection onto the line generated by $\n_0$.

For $\psi \in \x_m(0)$, we must estimate the last term in \eqref{l_ss_11} in terms of $\hn{\psi}{0}$.  This is possible due to the appearance of $\Pi_0^\bot$ and Lemma \ref{l_boundary_dual_estimate}.  Indeed, we know that
\begin{equation}
 \Pi_0^\bot (F^3(0) + \sg_{\a_0} u^m(0)  \n_0  ) = ( F^3(0)\cdot \n_0 + \sg_{\a_0} u^m(0)  \n_0 \cdot \n_0 ) \frac{\n_0}{\abs{\n_0}^2},
\end{equation}
which implies, since $\abs{\n_0}^2 \ge 1$ and $\diverge_{\a_0} \psi =0$, that
\begin{multline}\label{l_ss_12}
\abs{\ip{\Pi_0^\bot (F^3(0) + \sg_{\a_0} u^m(0)  \n_0  ) }{\psi}_{0,\Sigma}} \le
\abs{\n_0}^2 \abs{\ip{\Pi_0^\bot (F^3(0) + \sg_{\a_0} u^m(0)  \n_0  ) }{\psi}_{0,\Sigma}} \\
 = \abs{ \ip{ F^3(0)\cdot \n_0 + \sg_{\a_0} u^m(0)  \n_0 \cdot \n_0  }{\psi \cdot \n_0}_{0,\Sigma} }\\
\le \snormspace{\psi \cdot \n_0}{-1/2}{\Sigma} \snormspace{( F^3(0) + \sg_{\a_0} u^m(0)  \n_0) \cdot \n_0 ) }{1/2}{\Sigma}  \\
\ls C_1(\eta) \hn{\psi}{0} \snormspace{ F^3(0) + \sg_{\a_0} u^m(0)  \n_0  }{1/2}{\Sigma}. 
\end{multline}
In the last inequality we have used Lemmas \ref{l_boundary_dual_estimate} and \ref{i_sobolev_product_1},  and we have written $ C_1(\eta):= \norm{\n_0}_{C^1(\Sigma)}.$

By virtue of \eqref{l_ss_1}, we have that 
\begin{equation}\label{l_ss_13}
\dt u^m(t) - R(t) u^m(t) = \dot{a}^m_j(t) \psi^j(t) \in \x_m(t),
\end{equation}
so that $\psi = \dt u^m(0) - R(0) u^m(0) \in \x_m(0)$ is a valid choice of a test function in \eqref{l_ss_11}.  We plug this $\psi$ into \eqref{l_ss_11}, rearrange, and employ the bound \eqref{l_ss_12} to see that
\begin{multline}\label{l_ss_14}
 \hn{\dt u^m(0)}{0}^2 \le \hn{R(0) u^m(0)}{0} \hn{\dt u^m(0)}{0} \\
+ \hn{\dt u^m(0) - R(0) u^m(0)}{0} \hn{\Delta_{\a_0} u^m(0) + F^1(0)}{0}
\\
+ C C_1(\eta) \hn{\dt u^m(0) - R(0) u^m(0)}{0} \snormspace{ F^3(0) + \sg_{\a_0} u^m(0)  \n_0  }{1/2}{\Sigma}.
\end{multline}
A simple computation and \eqref{l_ss_8} imply that $\hn{\Delta_{\a_0} u^m(0)}{0}  \ls \norm{\a_0}_{C^1}^2 \norm{u_0}_{2}.$  This allows us to use Cauchy's inequality and \eqref{l_ss_8} to derive from  \eqref{l_ss_14} the bound
\begin{equation}\label{l_ss_15}
 \hn{\dt u^m(0)}{0}^2  \ls C_2(\eta) \left( \norm{u_0}_{2}^2 + \hn{F^1(0)}{0}^2 + \snormspace{ F^3(0) + \sg_{\a_0} u^m(0)  \n_0  }{1/2}{\Sigma}^2  \right)
\end{equation}
for  $C_2(\eta) := 1+ \pnorm{R(0)}{\infty}^2 + \norm{\a_0}_{C^1}^2 + C_1(\eta)^2$.  This is our desired estimate of $\hn{\dt u^m(0)}{0}$.

Step 5 -- Energy estimates for $\dt u^m$

We now turn to estimates for $\dt u^m$ of a similar form to those we already derived for $u^m$.  Suppose for now that $\psi(t) = b^m_j(t) \psi^j$ for $b^m_j \in C^{0,1}([0,T])$, $j=1,\cdots,m$; it is easily verified, as in \eqref{l_ss_13}, that $\dt \psi - R(t) \psi \in \x_m(t)$ as well.  We now use this $\psi$ in \eqref{l_ss_2}, temporally differentiate the resulting equation, and then subtract from this the equation \eqref{l_ss_2} with test function $\dt \psi - R \psi$; this eliminates the appearance of $\dt \psi$ and leaves us with the equality
\begin{multline}\label{l_ss_16}
 \br{\dt^2 u^m,\psi}_{(\h^1)*} + \hal \iph{\dt u^m}{\psi}{1} = \br{\dt F^1,\psi}_{(\h^1)*} - \ip{\dt F^3}{\psi }_{0,\Sigma} - \ip{F^3}{R \psi}_{0,\Sigma} \\
+ \iph{F^1}{(\dt J K + R) \psi}{0}  - \iph{\dt u^m}{(\dt J K + R) \psi}{0} - \hal \iph{u^m}{R \psi}{1} \\
- \hal \int_\Omega  \left(\dt J K \sg_{\a} u^m: \sg_{\a} \psi + \sg_{\dt \a} u^m: \sg_{\a} \psi + \sg_{\a} u^m:\sg_{\dt \a} \psi    \right)J.
\end{multline}
Note here that the terms involving $\br{\cdot,\cdot}_{(\h^1)*}$ appear when we temporally differentiate because of Lemma \ref{l_x_time_diff}.  

According to \eqref{l_ss_13} and the fact that $a^m_j$ is twice differentiable a.e., we may use $\psi = \dt u^m(t) - R(t) u^m(t) \in \x_m(t)$ as a test function in \eqref{l_ss_16}.  Plugging in this $\psi$ and arguing as in the previous steps by employing Remark \ref{l_A_korn_trace}, Cauchy's inequality, and trace embeddings, we may deduce from \eqref{l_ss_16} that
\begin{multline}\label{l_ss_17}
\dt \left( \hal \hn{\dt u^m}{0}^2 - \iph{\dt u^m}{R u^m}{0}  \right) + \frac{1}{8} \hn{\dt u^m}{1}^2 \le C C_3(\eta) \hn{u^m}{1}^2 \\
+ C_0(\eta) \left( \hal \hn{\dt u^m}{0}^2 - \iph{\dt u^m}{R u^m}{0}  \right)
+ C \left( \hn{F^1}{0}^2 + \norm{\dt F^1}_{(\h^1)^*}^2 \right)\\
+ C \left( \snormspace{F^3}{1/2}{\Sigma}^2 +  \snormspace{\dt F^3}{-1/2}{\Sigma}^2 \right)
\end{multline}
for $C_0(\eta)$ as defined above and 
\begin{multline}
C_3(\eta) :=  \sup_{0 \le t \le T}   \left[1 +
\norm{R }_{C^1}^2 + \pnorm{\dt R }{\infty}^2  + \pnorm{\dt \a}{\infty}^2
+ \left(1+\pnorm{\a }{\infty}^2 \right) \left(1 + \pnorm{\dt J K}{\infty}^2 \right) \right]  \\
\times 
\sup_{0 \le t \le T}  \left[1+ \norm{R}_{C_1}^2 \right].
\end{multline}
Then \eqref{l_ss_17}, Gronwall's lemma, and a further application of Cauchy's inequality imply that
\begin{multline}\label{l_ss_18}
  \sup_{0\le t \le T} \hn{\dt u^m}{0}^2 + \norm{\dt u^m}_{\h^1_T}^2 \ls  \exp\left(C_0(\eta)T\right) \left( \hn{\dt u^m(0)}{0}^2 + C_2(\eta) \hn{u^m(0)}{0}^2  \right) \\
+ C_3(\eta)\left( \sup_{0\le t \le T} \hn{u^m}{0}^2 + \int_0^T \exp\left( C_0(\eta) (T-s) \right) \hn{u^m(s)}{1}^2 ds   \right)\\
+ \exp\left(C_0(\eta)T\right) \left( \norm{F^1}_{\h^0_T}^2 + \norm{\dt F^1}_{(\h^1_T)^*}^2  + \norm{F^3}_{L^2 H^{1/2}}^2  + \norm{\dt F^3}_{L^2 H^{-1/2}}^2\right).
\end{multline}
Now we combine \eqref{l_ss_18} with the estimates \eqref{l_ss_8}, \eqref{l_ss_9}, and \eqref{l_ss_15} to deduce our energy estimates for $\dt u^m$:
\begin{multline}\label{l_ss_19}
  \sup_{0\le t \le T} \hn{\dt u^m}{0}^2 + \norm{\dt u^m}_{\h^1_T}^2 \\
\ls  (C_2(\eta) + C_3(\eta)) \exp\left(C_0(\eta)T\right)  \left( \norm{u_0}_{2}^2 + \hn{F^1(0)}{0}^2 + \snormspace{ F^3(0) + \sg_{\a_0} u^m(0)  \n_0  }{1/2}{\Sigma}^2  \right) \\
+ \exp\left(C_0(\eta)T\right) \left[  C_3(\eta) \left(\norm{F^1}_{\h^0_T}^2 + \norm{F^3}_{L^2 H^{1/2}}^2\right) + \norm{\dt F^1}_{(\h^1_T)^*}^2    + \norm{\dt F^3}_{L^2 H^{-1/2}}^2\right].
\end{multline}

Step 6 -- Improved energy estimate for $u^m$

We can now improve our energy estimates for $u^m$ by using $\psi = \dt u^m(t) - R(t) u^m(t) \in \x_m(t)$ as a test function in \eqref{l_ss_2}.  Plugging this in and rearranging yields the equality
\begin{multline}\label{l_ss_20}
 \dt \frac{1}{4} \hn{u^m}{1}^2 + \hn{\dt u^m}{0}^2 =  \iph{\dt u^m}{R u^m}{0} + \hal \iph{u^m}{R u^m}{1} + \iph{F^1}{\dt u^m - R u^m}{0} \\
- \ip{F^3 - \Pi_0(F^3(0) + \sg_{\a_0} u^m(0)  \n_0  ) }{\dt u^m - R u^m}_{0,\Sigma} \\
+\hal \int_\Omega \left( \sg_{\a} u^m : \sg_{\dt \a} u^m + \dt J K \frac{\abs{\sg_{\a} u^m}^2}{2}  \right) J.
\end{multline}
We may then argue as before to use \eqref{l_ss_20} to derive the inequality
\begin{multline}\label{l_ss_21}
 \dt \frac{1}{4} \hn{u^m}{1}^2 + \hn{\dt u^m}{0}^2 \le   
 C \snormspace{F^3(0) + \sg_{\a_0} u^m(0) \n_0 }{1/2}{\Sigma}^2 \\
+ C \left( \hn{F^1}{0}^2 + \snormspace{F^3}{1/2}{\Sigma}^2  \right) 
+ C \left( \hn{\dt u^m}{1}^2 + C_3(\eta) \hn{u^m}{1}^2 \right).
\end{multline}
We could regard \eqref{l_ss_21} as a differential inequality for $\hn{u^m}{1}^2$ and apply Gronwall's lemma as before, but this is not necessary since we already control $\norm{u^m}_{\h^1_T}^2$ and $\norm{\dt u^m}_{\h^1_T}^2$.  Indeed, we may simply integrate \eqref{l_ss_21} in time to deduce an improved energy estimate for $u^m$:
\begin{multline}\label{l_ss_22}
  \sup_{0\le t \le T} \hn{u^m}{1}^2 + \norm{\dt u^m}_{\h^0_T}^2 \\
\ls  (C_2(\eta) + C_3(\eta)) \exp\left(C_0(\eta)T\right)  \left( \norm{u_0}_{2}^2 + \hn{F^1(0)}{0}^2 + \snormspace{ F^3(0) + \sg_{\a_0} u^m(0)  \n_0  }{1/2}{\Sigma}^2  \right) \\
+ \exp\left(C_0(\eta)T\right) \left[  C_3(\eta) \left(\norm{F^1}_{\h^0_T}^2 + \norm{F^3}_{L^2 H^{1/2}}^2\right) + \norm{\dt F^1}_{(\h^1_T)^*}^2    + \norm{\dt F^3}_{L^2 H^{-1/2}}^2\right].
\end{multline}

Step 7 -- Estimating terms in \eqref{l_ss_19}, \eqref{l_ss_22}

In order to use \eqref{l_ss_19} and \eqref{l_ss_22} as uniform bounds, we must first remove the appearance of $u^m(0)$ on the right side of the estimates.  For this we use Lemma  \ref{i_sobolev_product_2}, the embedding $H^2(\Omega) \hookrightarrow H^{3/2}(\Sigma)$, and the bound $\norm{u^m(0)}_{2} \le \norm{u_0}_{2}$ to find that
\begin{equation}\label{l_ss_23}
\snormspace{ F^3(0)  + \sg_{\a_0} u^m(0)  \n_0   }{1/2}{\Sigma}^2  
\ls C_4(\eta)   \left( \snormspace{F^3(0)}{1/2}{\Sigma}^2 + \norm{u_0}_{2}^2 \right)
\end{equation}
for $C_4(\eta):=1 + \norm{\n_0}_{C^1(\Sigma)}^2 
\norm{\a_0}_{C^1}^2.$

We now seek to estimate the constants $C_i(\eta)$, $i=0,\dotsc,4$ in terms of the quantity $\mathcal{K}(\eta)$.  A simple computation shows that
\begin{equation}\label{l_ss_31}
 C_0(\eta) + (C_2(\eta) + C_3(\eta))(1+ C_4(\eta)) \le \sup_{0\le t \le T} \mathcal{Q}_1( \norm{ \bar{\eta} }_{C^2}^2, \norm{\dt \bar{\eta}}_{C^2}^2, \norm{\dt^2 \bar{\eta}}_{C^1}^2),
\end{equation}
where $\mathcal{Q}_1$ is a polynomial in three variables.  According to Lemma \ref{i_poisson_interp} in the non-periodic case and Lemma \ref{p_poisson_2} in the periodic case, we have the estimate $\norm{\dt^j \bar{\eta}}_{C^k}^2 \ls \norm{\dt^j \eta}_{k+3/2}^2$ for $j,k\ge 0$.  This, \eqref{l_ss_31}, and the fact that $\mathcal{K}(\eta) \le 1$ then imply that
\begin{equation}\label{l_ss_32}
  C_0(\eta) + (C_2(\eta) + C_3(\eta))(1+ C_4(\eta)) \le \mathcal{Q}_1(\mathcal{K}(\eta),\mathcal{K}(\eta), \mathcal{K}(\eta)) \le C(1+ \mathcal{K}(\eta))
\end{equation}
for a constant $C$ independent of $\eta$.

Step 8 -- Passing to the limit

We now utilize the energy estimates  \eqref{l_ss_19} and \eqref{l_ss_22} in conjunction with \eqref{l_ss_23} to pass to the limit $m \to \infty$.   According to these energy estimates  and Lemma \ref{l_norm_equivalence}, we have that the sequence $\{u^m\}$ is uniformly bounded in $L^\infty H^1$ and $\{\dt u^m\}$ is uniformly bounded in $L^\infty H^0 \cap L^2 H^1$.  Up to the extraction of a subsequence, we then know that
\begin{equation}
 u^m \wstar u \text{ weakly-}* \text{ in } L^\infty H^1, \dt u^m \wstar \dt u \text{ in } L^\infty H^0, \text{ and } \dt u^m \rightharpoonup \dt u \text{ weakly in } L^2H^1.
\end{equation}
By lower semi-continuity and \eqref{l_ss_32}, the energy estimates imply that the quantity
\begin{equation}
\norm{u}_{L^\infty H^1}^2 + \norm{\dt u}_{L^\infty H^0}^2 + \norm{\dt u}_{L^2 H^1}^2  
\end{equation}
is bounded above by the right hand side of \eqref{l_ss_02}.

Because of these convergence results, we can integrate \eqref{l_ss_16} in time from $0$ to $T$ and send $m \to \infty$ to deduce that $\dt^2 u^m \rightharpoonup \dt^2 u$ weakly in $L^2 (\H1)^*$, with the action of $\dt^2 u$ on an element  $\psi \in L^2 \H1$ defined by replacing $u^m$ with $u$ everywhere in \eqref{l_ss_16}.  It is more natural to regard $\dt^2 u \in (\x_T)^*$ since the action of $\dt^2 u$ is defined with test functions in $\x_T$, but the reasoning presented after \eqref{l_weak_solution_pressureless} is applicable to $\dt^2 u$, so we may regard $\dt^2 u \in L^2 \H1$ without ambiguity.  From the equation resulting from passing to the limit in \eqref{l_ss_16}, it is straightforward to show that $\norm{\dt^2 u}_{(\h^1_T)^*}^2$ is bounded by the right hand side of \eqref{l_ss_02}.  This bound then shows that $\dt u \in C^0 L^2$.

Step 9 -- The strong solution

Due to the convergence established in the last step, we may pass to the limit in \eqref{l_ss_2} for a.e. $t \in [0,T]$.  Since $u^m(0) \to u_0$ in $H^2$ and $u_0, F^3(0)$ satisfy the compatibility condition \eqref{l_ss_01}, we have that
\begin{equation}
\snormspace{\Pi_0 ( F^3(0)  + \sg_{\a_0} u^m(0)  \n_0  ) }{1/2}{\Sigma}
\to 0
\end{equation}
In the limit, \eqref{l_ss_2} implies that for a.e. $t$, 
\begin{equation}\label{l_ss_24}
  \iph{\dt u}{\psi}{0}  + \hal \iph{u}{\psi}{1} = \iph{F^1}{\psi}{0} - \ip{F^3}{\psi}_{0,\Sigma} \text{ for every }  \psi \in \x(t).
\end{equation}

Now we introduce the pressure.  Define the functional $\Lambda_t \in (\h^1(t))^*$ so that $\Lambda_t(v)$ equals the difference between the left and right sides of \eqref{l_ss_24}, with $\psi$ replaced by $v \in \h^1(t)$.  Then $\Lambda_t(v) = 0$ for all $v \in \x(t)$, so by Proposition \ref{l_pressure_decomp} there exists a unique $p(t) \in \h^0(t)$ so that $\iph{p(t)}{\diva v}{0} = \Lambda_t(v)$ for all $v \in \h^1(t)$.  This is equivalent to 
\begin{equation}\label{l_ss_25}
  \iph{\dt u}{v}{0}  + \hal \iph{u}{v}{1} - \iph{p}{\diva v}{0} = \iph{F^1}{v}{0} - \ip{F^3}{v}_{0,\Sigma} \text{ for every }  v \in \h^1(t),
\end{equation}
which in particular implies that $(u,p)$ is the unique weak solution to \eqref{l_linear_forced} in the sense of \eqref{l_weak_solution_pressure}.  

For a.e. $t\in[0,T]$, $(u(t),p(t))$ is the unique weak solution to the elliptic problem \eqref{l_linear_elliptic} in the sense of \eqref{l_elliptic_weak}, with $F^1$ replaced by $F^1(t) - \dt u(t)$, $F^2=0$, and $F^3$ replaced by $F^3(t)$.  Since $F^1(t) - \dt u(t) \in H^0(\Omega)$ and $F^3(t) \in H^{1/2}(\Sigma)$,  Lemma \ref{l_stokes_solvable} implies that this elliptic problem admits a unique strong solution, which must coincide with the weak solution.  We may then apply Proposition \ref{l_stokes_regularity} and Lemma \ref{l_norm_equivalence} for the bound
\begin{equation}\label{l_ss_26}
\norm{u(t)}_{r}^2 +   \norm{p(t)}_{r-1}^2 \ls \left(\hn{\dt u(t)}{r-2}^2 +  \norm{F^1(t)}_{r-2}^2 + \snormspace{F^3(t)}{r-3/2}{\Sigma}^2 \right)
\end{equation}
when $r=2,3$.  When $r=2$ we take the supremum of \eqref{l_ss_26} over $t\in[0,T]$, and  when $r=3$ we integrate over $[0,T]$; the resulting inequalities imply that $u \in L^\infty H^2 \cap L^2 H^3$ and $p \in L^\infty H^1 \cap L^2 H^2$ with estimates as in \eqref{l_ss_02}.  This, in turn, implies that $(u,p)$ is a strong solution to \eqref{l_linear_forced}.

Since we already know that $u\in L^2 H^3$ and $\dt u \in L^2 H^1$, Lemma \ref{l_sobolev_infinity} implies that  $u \in C^0 H^2$.  Then since $F^1 - \dt u \in C^0 H^0$ and $\sg_{\a} u \n + F^3 \in C^0 H^{1/2}(\Sigma)$, we know that $\naba p \in C^0 H^0$ and $p \in C^0 H^{1/2}(\Sigma)$ as well, from which we see, via Poincar\'e's inequality (Lemma \ref{poincare_b}), that $p \in C^0 H^1$.  With these continuity results established, we can compute $p(0)$ and $\dt u(0)$.  We start with the Dirichlet condition for $p(0)$ on $\Sigma$, the second equation in \eqref{l_ss_04}.   Since $p \in C^0 H^1(\Omega)$, $u \in C^0 H^2(\Omega)$, and $F^3 \in C^0 H^{1/2}(\Sigma)$, the boundary condition $S_{\a}(p,u) \n = F^3$, which holds in $H^{1/2}(\Sigma)$ for each $t>0$, can be evaluated at $t=0$.  Then the Dirichlet condition for $p(0)$ on $\Sigma$ in \eqref{l_ss_04} is easily deduced by solving $S_{\a_0}(p(0),u_0) \n_0 = F^3(0)$ for $p(0)$.

Now we derive the PDE satisfied by $p(0)$ and compute $\dt u(0)$.  Let $\varphi \in H^2(\Omega)$ be a scalar function satisfying $\varphi\vert_\Sigma =0$ and $\nab \varphi \vert_{\Sigma_b}=0$.  Then $\naba \varphi = \a \nab \varphi \in \h^1(t)$, and we may choose $v = \naba \varphi$ as a test function  in \eqref{l_ss_25}.  Since $ \dt u - R u \in \x(t)$, we can integrate by parts to see that
\begin{equation}
\begin{split}
 \iph{\dt u}{\naba \varphi}{0}  &=  \iph{\dt u - R u}{\naba \varphi}{0}  + \iph{R u}{\naba \varphi}{0} =  \iph{R u}{\naba \varphi}{0} \text{ and }\\
 \iph{p}{\diva \naba \varphi}{0} & = \iph{-\naba p}{\naba \varphi}{0} +\ip{p}{\naba \varphi \cdot \n}_{0,\Sigma}.
\end{split}
\end{equation}
This, \eqref{l_ss_10}, \eqref{l_ss_25}, and \eqref{l_linear_forced} imply that
\begin{equation}\label{l_ss_33}
\iph{R u + \naba p - \da u - F^1}{\naba \varphi}{0} = 0 \text{ for all such }\varphi.
\end{equation}
By the established continuity properties, we may set $t=0$ in \eqref{l_ss_33}, again integrate by parts, and employ a density argument to see that
\begin{equation}\label{l_ss_34}
\iph{ \nab_{\a_0} p(0) - F^1(0)}{\nab_{\a_0} \varphi}{0} = -\iph{-\diverge_{\a_0} (R(0)u_0) }{\varphi}{0} + \br{\Delta_{\a_0} u_0 \cdot \nu, \varphi}_{-1/2}
\end{equation}
for all $\varphi \in {^0}H^1(\Omega)$.  This establishes that $p(0)$ is the  weak solution to \eqref{l_ss_04}.  According to \eqref{l_a_poisson_weak_estimate} we then have that $p(0) \in H^1(\Omega)$.  This and \eqref{l_ss_25} allow us to solve for $\dt u(0)$ as in \eqref{l_ss_05}, and then \eqref{l_ss_33}  implies that $\dt u(0) -R(0) u_0 \in \y(0)$ since then $D_t u(0) \bot \nab_{\a(0)} \varphi$ for every $\varphi \in {^0}H^1(\Omega)$.

Step 10 -- The weak solution satisfied by $D_t u= \dt u - R u$

Now we seek to use \eqref{l_ss_16} to determine the PDE satisfied by $D_t  u$.  As mentioned above, we may integrate \eqref{l_ss_16} in time from $0$ to $T$ and pass to the limit $m \to \infty$.  For any $\psi \in \x_T$ we have $R \psi \in \h^1_T$, so that we may replace all of the terms $R \psi$ in the resulting equation by using $v=R\psi$ in \eqref{l_ss_25}; this yields the equality
\begin{multline}\label{l_ss_27}
 \br{\dt^2 u,\psi}_{*} + \hal \ip{\dt u}{\psi}_{\h^1_T} = \br{\dt F^1,\psi}_{*} - \br{\dt F^,\psi }_{-1/2}  \\
+ \ip{\dt J K F^1}{  \psi}_{\h^0_T}  - \ip{\dt J K \dt u}{  \psi}_{\h^0_T} - \ip{p}{\diva (R\psi)}_{\h^0_T} \\
- \hal \int_0^T \int_\Omega  \left(\dt J K \sg_{\a} u: \sg_{\a} \psi + \sg_{\dt \a} u: \sg_{\a} \psi + \sg_{\a} u:\sg_{\dt \a} \psi    \right)J
\end{multline}
for all $\psi \in \x_T$.  In \eqref{l_ss_27} we have employed the same duality notation as in \eqref{l_weak_solution_pressure}.

Now we define $\Lambda \in (\h^1_T)^*$ with $\Lambda(v)$ equal to the difference between the left and right sides of \eqref{l_ss_27} with $\psi$ replaced with $v$.  As above, we may use Proposition \ref{l_pressure_decomp} to find a unique $q \in \h^0_T$ so that $\Lambda(v) = \ip{q}{\diva v}_{\h^0_T}$ for all $v \in \h^1_T$.  Simple computations reveal that $\dt(J \a_{ij}) = -J \a_{kj} R_{ki}$ and  $J \a_{kj} \p_j R_{ki} = \p_j (J \a_{kj}  R_{ki}) = -\dt \p_j(J \a_{ij}) = 0$, which imply that
\begin{multline}
 \ip{p}{\diverge_{\dt \a} v + \dt J K \diva v }_{\h^0_T} = \int_0^T \int_\Omega p \p_j v_i \dt(J \a_{ij}) = \int_0^T \int_\Omega p \p_j v_i J \a_{kj} R_{ki} \\
= \int_0^T \int_\Omega  p J \a_{kj} \p_j(R_{ki} v_i) - pJ \a_{kj} v_i \p_j R_{ki} = \int_0^T \int_\Omega  p J \a_{kj} \p_j(R_{ki} v_i) =  \ip{p}{\diva (R v )}_{\h^0_T}.
\end{multline}
This, in turn, implies that the equation $\Lambda(v) = \ip{q}{\diva v}_{\h^0_T}$ is the same as that which would result from computing the temporal distributional derivative of \eqref{l_ss_25}; we deduce that $q = \dt p$ and  that
\begin{multline}\label{l_ss_28}
 \br{\dt^2 u,v}_{*} + \hal \ip{\dt u}{v}_{\h^1_T} -\ip{\dt p}{\diva v}_{\h^0_T} = \br{\dt F^1,v}_{*} - \br{\dt F^3,v }_{-1/2}   \\
+ \ip{\dt J K F^1}{  v}_{\h^0_T}  - \ip{\dt J K \dt u}{  v }_{\h^0_T} - \ip{p}{\diva (R v)}_{\h^0_T} \\
- \hal \int_0^T \int_\Omega  \left(\dt J K \sg_{\a} u: \sg_{\a} v + \sg_{\dt \a} u: \sg_{\a} v + \sg_{\a} u:\sg_{\dt \a} v    \right)J
\end{multline}
for all $v \in \h^1_T$.  As before, we may deduce from \eqref{l_ss_28} the bound for $\norm{\dt p}_{L^2 L^2}^2$ stated in \eqref{l_ss_02}.  

We now rewrite the terms in \eqref{l_ss_28} to derive the PDE for $D_t u, \dt p$.  A straightforward computation shows that on $\Sigma$,  $R^T \n = \dt \n$, so that we may integrate by parts for the equality
\begin{equation}\label{l_ss_29}
 \ip{p}{\diva (R v)}_{\h^0_T} = -\ip{R^T \naba p}{v}_{\h^0_T} + \br{p R^T \n , v}_{-1/2} = -\ip{R^T \naba p}{v}_{\h^0_T} + \br{p \dt \n , v}_{-1/2},
\end{equation}
where $R^T$ is the matrix transpose of $R$.  Another integration by parts yields
\begin{multline}\label{l_ss_30}
- \hal \int_0^T \int_\Omega  \left(\dt J K \sg_{\a} u: \sg_{\a} v + \sg_{\dt \a} u: \sg_{\a} v + \sg_{\a} u:\sg_{\dt \a} v    \right)J \\
=  \ip{\diva(R \sg_{\a} u + \sg_{\dt \a} u)}{v}_{\h^0_T} - \br{\sg_{\a} u \dt \n  + \sg_{\dt \a} u \n,v}_{-1/2}.
\end{multline}
We replace the appearance of $\dt^2 u$ with $\dt D_t u$ via
\begin{equation}\label{l_ss_35}
 \br{\dt^2 u,v}_{*} = \br{\dt D_t u,v}_{*} + \ip{R \dt u}{v}_{\h^0_T} + \ip{\dt R u}{v}_{\h^0_T}.
\end{equation}
Since $(u,p)$ are a strong solution to \eqref{l_linear_forced}, we may multiply by $(R^T + \dt JK)v$ and integrate to see that
\begin{equation}\label{l_ss_36}
 \ip{(\dt J K + R) (F^1 - \dt u)}{v}_{\h^0_T} =  \ip{(\dt J K + R) (-\da u + \naba p )}{ v}_{\h^0_T}.
\end{equation}
We may then combine \eqref{l_ss_28}--\eqref{l_ss_36} with the fact that $D_t u = \dt u - Ru \in \x_T$ to deduce that $(D_t u,\dt p)$ are weak solutions of \eqref{l_ss_03} with $D_t u(0)\in \y(0)$ given by \eqref{l_ss_05}.  Here, the inclusions $G^1 \in(\h^1_T)^*$ and $G^3 \in L^2([0,T]; H^{-1/2}(\Sigma))$ are easily established from the above bounds on $u,p$.

\end{proof}

\subsection{Higher regularity}\label{l_higher_reg}

Throughout this section we write $L^2 H^{-1} = L^2 (\H1)^*$.  In order to state our higher regularity results for the problem \eqref{l_linear_forced}, we must be able to define the forcing terms and initial data for the problem that results from temporally differentiating \eqref{l_linear_forced} several times.  To this end, we first define some mappings. 
Given $F^1,F^3,v,q$ we define the vector fields $\mathfrak{G}^0,\mathfrak{G}^1$ on $\Omega$ and $\mathfrak{G}^3$ on $\Sigma$ by
\begin{equation}\label{l_G_def}
\begin{split}
 \mathfrak{G}^0(F^1,v,q) &= \Delta_{\a} v - \nab_{\a} q + F^1 - R v, \\
 \mathfrak{G}^1(v,q) & =   -(R + \dt J K) \da v - \dt R v + (\dt J K + R - R^T) \naba q \\ 
&+ \diva( \sg_{\a} (R v) + R \sg_{\a} v + \sg_{\dt \a} v ), \text{ and } \\
 \mathfrak{G}^3(v,q) &=  \sg_{\a}(R v) \n - (q I - \sg_{\a} v) \dt \n + \sg_{\dt \a} v \n,
\end{split}
\end{equation}
and  we define the functions $\mathfrak{f}^1$ on $\Omega$, $\mathfrak{f}^2$ on $\Sigma$, and $\mathfrak{f}^3$ on $\Sigma_b$ according to
\begin{equation}\label{l_f_def}
\begin{split}
 \mathfrak{f}^1(F^1,v) &= \diverge_{\a} (F^1 - R v ), \\
 \mathfrak{f}^2(F^3,v) &= ( F^3  + \sg_{\a} v  \n) \cdot  \n \abs{\n}^{-2}, \text{ and } \\
 \mathfrak{f}^3(F^1,v) &= (F^1+ \Delta_{\a} v) \cdot \nu.
\end{split}
\end{equation}
In the definitions of $\mathfrak{G}^i$ and $\mathfrak{f}^i$ we assume that $\a,\n,R$ (recall that $R$ is defined by \eqref{l_Dt_def}), etc are evaluated at the same $t$ as $F^1,F^3, v,q$.  These mappings allow us to define the forcing terms as follows.  Write $F^{1,0} = F^1$ and $F^{3,0} = F^3$.  When $F^1$, $F^3$, $u$, and $p$ are sufficiently regular for the following to make sense, we then recursively define the vectors
\begin{equation}\label{l_Fj_def}
\begin{split}
 F^{1,j} & := D_t F^{1,j-1} + \mathfrak{G}^1(D_t^{j-1} u,\dt^{j-1} p) = D_t^j F^1 + \sum_{\ell=0}^{j-1} D_t^\ell \mathfrak{G}^1(D_t^{j-\ell-1} u, \dt^{j-\ell-1} p), \\
 F^{3,j} & := \dt F^{3,j-1} + \mathfrak{G}^3(D_t^{j-1} u,\dt^{j-1} p) = \dt^j F^3 + \sum_{\ell=0}^{j-1} \dt^\ell \mathfrak{G}^3(D_t^{j-\ell-1} u, \dt^{j-\ell-1} p) 
\end{split}
\end{equation}
on $\Omega$ and $\Sigma$, respectively, for $j=1,\dotsc,2N$.

Now we define various sums of norms of $F^1,F^3$, and $\eta$ that will appear in our estimates.  
Define the quantities
\begin{equation}\label{l_Ffrak_def}
\begin{split}
\mathfrak{F}(F^1,F^3) &:= \sum_{j=0}^{2N} \ns{\dt^j F^1}_{L^2 H^{4N-2j -1}} + \ns{\dt^j F^3}_{L^2 H^{4N-2j -1/2}}  \\
& + \sum_{j=0}^{2N-1} \ns{\dt^j F^1}_{L^\infty H^{4N-2j -2}} + \ns{\dt^j F^3}_{L^\infty H^{4N-2j -3/2}}, \\
\mathfrak{F}_0(F^1,F^3) &:=  \sum_{j=0}^{2N-1} \ns{\dt^j F^1(0)}_{4N-2j -2} + \ns{\dt^j F^3(0)}_{4N-2j -3/2}.
\end{split}
\end{equation}
For brevity, we will only write $\mathfrak{F}$ for $\mathfrak{F}(F^1,F^3)$ and $\mathfrak{F}_0$ for $\mathfrak{F}_0(F^1,F^3)$ throughout the rest of this section.  Lemmas \ref{l_sobolev_infinity} and \ref{l_x_time_diff} imply that if $\mathfrak{F} < \infty$, then   
\begin{equation}
 \dt^j F^1 \in C^0([0,T];H^{4N-2j-2}(\Omega)) \text{ and } \dt^j F^3 \in C^0([0,T];H^{4N-2j-3/2}(\Sigma))
\end{equation}
for $j=0,\dotsc,2N-1$.  The same lemmas also imply that the sum of the $L^\infty H^{k}$ norms in the definition of $\mathfrak{F}$ can be bounded by a constant that depends on $T$ times the sum of the $L^2 H^{k+1}$ norms.  To avoid the introduction of a constant that depends on $T$, we will retain the $L^\infty$ terms.  For $\eta$ we define
\begin{equation}\label{l_Kfrak_def}
\begin{split}
\mathfrak{D}(\eta) & := \ns{\eta}_{L^2 H^{4N+1/2}} + \ns{\dt \eta}_{L^2 H^{4N-1/2}} +  \sum_{j=2}^{2N+1} \ns{\dt^j \eta}_{L^2 H^{4N-2j+5/2}}, \\ 
\mathfrak{E}(\eta) & := \sum_{j=0}^{2N} \ns{\dt^j \eta}_{L^\infty H^{4N-2j}}, \text{ and } \mathfrak{K}(\eta)  := \mathfrak{E}(\eta) + \mathfrak{D}(\eta)
\end{split}
\end{equation}
as well as
\begin{equation}\label{l_E0frak_def}
 \mathfrak{E}_0(\eta)  := \ns{\eta(0)}_{4N} + \ns{\dt \eta(0)}_{4N-1}  + \sum_{j=2}^{2N} \ns{\dt^j \eta(0)}_{4N-2j+3/2}.
\end{equation}
Again, Lemma \ref{l_sobolev_infinity} implies that  $\eta \in C^0([0,T];H^{4N}(\Sigma))$, $\dt \eta \in C^0([0,T];H^{4N-1}(\Sigma))$, and $\dt^j \eta \in C^0([0,T];H^{4N-2j+3/2}(\Sigma))$  for $j=2,\dotsc,2N$.   Throughout the rest of this section we will assume that $\mathfrak{K}(\eta), \mathfrak{E}_0(\eta) \le 1$, which implies that $\mathcal{Q}( \mathfrak{K}(\eta) ) \ls 1 + \mathfrak{K}(\eta)$ and $\mathcal{Q}( \mathfrak{E}_0(\eta) ) \ls 1 + \mathfrak{E}_0(\eta)$ for any polynomial $\mathcal{Q}$.  Note  that $\mathcal{K}(\eta) \le  \mathfrak{E}(\eta) \le \mathfrak{K}(\eta)$, where $\mathcal{K}(\eta)$ is defined by \eqref{l_K_def};  also, we have that $\ns{\eta_0}_{4N-1/2} \le \mathfrak{E}_0(\eta)$.

We now record an estimate of the $F^{i,j}$ in terms of $\mathfrak{F}, \mathfrak{K}(\eta)$ and certain norms of $u,p$.

\begin{lem}\label{l_iteration_estimates_1}
For $m =1,\dotsc,2N-1$ and $j=1,\dotsc,m$, the following estimates hold whenever the right hand sides are finite:  
\begin{multline}\label{l_ie1_01}
 \ns{F^{1,j}}_{L^2 H^{2m-2j+1}}  +  \ns{F^{3,j}}_{L^2 H^{2m-2j+3/2}} 
\ls (1+ \mathfrak{K}(\eta)) \left(  \mathfrak{F} + \sum_{\ell=0}^{j-1} \ns{\dt^\ell u}_{L^\infty H^{2m-2\ell}}  \right. \\
\left. + \sum_{\ell=0}^{j-1}  \ns{\dt^\ell p}_{L^\infty H^{2m-2\ell-1}}  + \sum_{\ell=0}^{j-1} \ns{\dt^\ell u}_{L^2 H^{2m-2\ell+1}} + \ns{\dt^\ell p}_{L^2 H^{2m-2\ell}}   \right),
\end{multline}
\begin{multline}\label{l_ie1_04}
 \ns{F^{1,j}}_{L^\infty H^{2m-2j}}  +  \ns{F^{3,j}}_{L^\infty H^{2m-2j+1/2}}  \\
\ls (1+ \mathfrak{K}(\eta)) \left(  \mathfrak{F} + \sum_{\ell=0}^{j-1} \ns{\dt^\ell u}_{L^\infty H^{2m-2\ell}} + \ns{\dt^\ell p}_{L^\infty H^{2m-2\ell-1}}     \right),
\end{multline}
and
\begin{multline}\label{l_ie1_02}
 \ns{\dt F^{1,m}}_{L^2 H^{-1}} +  \ns{ \dt F^{3,m}}_{L^2 H^{-1/2}} 
\ls (1+ \mathfrak{K}(\eta)) \left( \mathfrak{F} + \sum_{\ell=0}^{m} \ns{\dt^\ell u}_{L^\infty H^{2m-2\ell}} \right. \\
\left.+ \sum_{\ell=0}^{m-1} \ns{\dt^\ell p}_{L^\infty H^{2m-2\ell-1}} + \sum_{\ell=0}^{m} \ns{\dt^\ell u}_{L^2 H^{2m-2\ell+1}} + \ns{\dt^\ell p}_{L^2 H^{2m-2\ell}}   \right).
\end{multline}

Similarly, for  $j = 1,\dotsc, 2N-1$, 
\begin{multline}\label{l_ie1_05}
\ns{F^{1,j}(0)}_{4N-2j-2} + \ns{F^{3,j}(0)}_{4N-2j-3/2} 
\\  \ls (1+ \mathfrak{E}_0(\eta)) \left( \mathfrak{F}_0 + \sum_{\ell=0}^{j-1} 
\ns{\dt^\ell u(0)}_{4N-2\ell} + \ns{\dt^\ell p(0)}_{4N-2\ell-1} \right).
\end{multline}

\end{lem}

\begin{proof}
The estimates follow from simple but lengthy computations, invoking standard arguments.  As such, we present only a sketch of how to derive the estimates \eqref{l_ie1_01}.  The  estimates \eqref{l_ie1_04}--\eqref{l_ie1_05} follow from similar arguments.

To derive the estimate \eqref{l_ie1_01}, we use the definition of $F^{1,j}, F^{3,j}$ given by \eqref{l_Fj_def} and expand all terms using the Leibniz rule and the definition $D_t$ to rewrite $F^{i,j}$ as a sum of products of two terms: one involving products of various derivatives of $\bar{\eta}$, and one linear in  derivatives of $u$, $p$, $F^1$, or $F^3$.  For a.e. $t \in [0,T]$ we then estimate the the norm ($H^{2m-2j+1}$ and $H^{2m-2j+3/2}$, respectively) of the resulting products  by using the usual algebraic properties of Sobolev spaces (i.e. Lemma \ref{i_sobolev_product_1}) in conjunction with the Sobolev embeddings.  The resulting inequalities may then be integrated in time from $0$ to $T$ to find an inequality of the form
\begin{equation}\label{l_ie1_4}
 \ns{F^{1,j}}_{L^2 H^{2m-2j+1}}  +  \ns{F^{3,j}}_{L^2 H^{2m-2j+3/2}} \ls  \mathcal{Q}(\mathfrak{E}(\eta)) (\mathfrak{D}(\eta)  Y_{\infty} + Y_{2}) ,
\end{equation}
where $\mathcal{Q}(\cdot)$ is a polynomial, 
\begin{multline}
 Y_{\infty} = \sum_{j=0}^{2N-1} \ns{\dt^j F^1}_{L^\infty H^{4N-2j -2}} + \ns{\dt^j F^3}_{L^\infty H^{4N-2j -3/2}} \\
+ \sum_{\ell=0}^{j-1} \ns{\dt^\ell u}_{L^\infty H^{2m-2\ell}}   + \sum_{\ell=0}^{j-1}  \ns{\dt^\ell p}_{L^\infty H^{2m-2\ell-1}},
\end{multline}
and
\begin{multline}
 Y_{2} = \sum_{j=0}^{2N} \ns{\dt^j F^1}_{L^2 H^{4N-2j -1}} + \ns{\dt^j F^3}_{L^2 H^{4N-2j -1/2}} \\
+ \sum_{\ell=0}^{j-1} \ns{\dt^\ell u}_{L^2 H^{2m-2\ell+1}} + \ns{\dt^\ell p}_{L^2 H^{2m-2\ell}}.
\end{multline}
Since $\mathfrak{K}(\eta) \le 1$, we know that $\mathcal{Q}(\mathfrak{E}(\eta)) (1+\mathfrak{D}(\eta))  \ls (1+\mathfrak{K}(\eta))$, and the bound \eqref{l_ie1_01}  follows immediately from \eqref{l_ie1_4}.
\end{proof}

Next we record an estimates for the difference between $\dt v$ and $D_t v$ for a general $v$.  The proof is similar to that of Lemma \ref{l_iteration_estimates_1}, and is thus omitted.

\begin{lem}\label{l_iteration_estimate_3}
If $k = 0,\dotsc,4N-1$ and $v$ is sufficiently regular, then
\begin{equation}\label{l_ie3_01}
\ns{\dt v - D_t v}_{L^2 H^k} \ls (1+ \mathfrak{K}(\eta)) \ns{v}_{L^2 H^k},
\end{equation}
and if $k=0,\dotsc,4N-2$, then
\begin{equation}\label{l_ie3_06}
\ns{\dt v - D_t v}_{L^\infty H^{k}} \ls (1+ \mathfrak{K}(\eta)) \ns{v}_{L^\infty H^{k}}.
\end{equation}

If $m=1,\dotsc,2N-1$, $j=1,\dotsc,m$,  and $v$ is sufficiently regular, then
\begin{equation}\label{l_ie3_02}
 \ns{\dt^j v - D_t^j v}_{L^2 H^{2m-2j+3}}  \ls (1+ \mathfrak{K}(\eta)) \sum_{\ell=0}^{j-1}\left( \ns{\dt^\ell v}_{L^2 H^{2m-2\ell+1}}  + \ns{\dt^\ell v}_{L^\infty H^{2m-2\ell}} \right),
\end{equation}
\begin{equation}\label{l_ie3_05}
 \ns{\dt^j v - D_t^j v}_{L^\infty H^{2m-2j+2}}  \ls (1+ \mathfrak{K}(\eta)) \sum_{\ell=0}^{j-1} \ns{\dt^\ell v}_{L^\infty H^{2m-2\ell}},
\end{equation}
and
\begin{multline}\label{l_ie3_03}
 \ns{\dt D_t^m v - \dt^{m+1} v}_{L^2 H^1} + \ns{\dt^2 D_t^{m} v - \dt^{m+2} v}_{L^2 H^{-1}} \\
\ls (1+ \mathfrak{K}(\eta)) \sum_{\ell=0}^{m+1} \left( \ns{\dt^\ell v}_{L^2 H^{2m-2\ell+1}} + \ns{\dt^\ell v}_{L^\infty H^{2m-2\ell}} \right).
\end{multline}

Also, if $j=0,\dotsc,2N$, and $v$ is sufficiently regular, then
\begin{equation}\label{l_ie3_04}
 \ns{\dt^j v(0) - D_t^j v(0)}_{4N-2j} \ls (1+ \mathfrak{E}_0(\eta))\sum_{\ell=0}^{j-1} \ns{\dt^\ell v(0)}_{4N-2\ell}.
\end{equation}
\end{lem}

Now we record an estimate for the terms $\mathfrak{G}^0$ and $\mathfrak{f}^i$ (defined in \eqref{l_G_def} and \eqref{l_f_def}, respectively) that will be used in computing initial data.

\begin{lem}\label{l_iteration_estimates_2}
Suppose that $v,q,G^1,G^3$ are evaluated at $t=0$ and are sufficiently regular for the right sides of the following estimates to make sense.  For $j=0,\dotsc,2N-1$, we have
\begin{multline}\label{l_ie2_02}
 \ns{\mathfrak{G}^0(G^1,v,q)}_{4N-2j-2} \\
\ls (1+ \ns{\eta(0)}_{4N} + \ns{\dt \eta(0)}_{4N-1}) \left(  \ns{v}_{4N-2j} + \ns{q}_{4N-2j-1} + \ns{G^1}_{4N-2j-2}  \right).
\end{multline}
If $j=0,\dotsc,2N-2$, then 
\begin{multline}\label{l_ie2_01}
 \ns{\mathfrak{f}^1(G^1,v)}_{4N-2j-3} + \ns{\mathfrak{f}^2(G^3,v)}_{4N-2j-3/2} + \ns{\mathfrak{f}^3(G^1,v)}_{4N-2j-5/2} \\
\ls (1+ \ns{\eta(0)}_{4N} )\left( \ns{G^1}_{4N-2j-2} + \ns{G^3}_{4N-2j-3/2}  + \ns{v}_{4N-2j}\right).
\end{multline}
For $j=2N-1$, if  $\diverge_{\a(0)} v=0$ in $\Omega$, then  
\begin{multline}\label{l_ie2_03}
 \ns{\mathfrak{f}^2(G^3,v)}_{1/2} + \ns{\mathfrak{f}^3(G^1,v)}_{-1/2} 
\ls (1+ \ns{\eta(0)}_{4N} )\left( \ns{G^1}_{2} + \ns{ G^3}_{1/2}  + \ns{v}_{2}\right).
\end{multline}
\end{lem}
\begin{proof}
The proof of the estimates \eqref{l_ie2_02} and \eqref{l_ie2_01} as well as the $\mathfrak{f}^2$ estimate in \eqref{l_ie2_03} can be carried out as in the proof Lemma \ref{l_iteration_estimates_1}.  We omit further details.    For the $\mathfrak{f}^3$ estimate of \eqref{l_ie2_03}, we note that $\diverge_{\a(0)} v=0$ implies that $\diverge_{\a(0)} \Delta_{\a(0)} v=0$, so that Lemmas \ref{l_boundary_dual_estimate} and \ref{l_norm_equivalence} provide the bound $\ns{\Delta_{\a(0)} v \cdot \nu }_{H^{-1/2}(\Sigma_b)} \ls \ns{\Delta_{\a(0)} v}_{0}$.  We may then argue as in Lemma \ref{l_iteration_estimates_1} to derive the $\mathfrak{f}^3$ bound.

\end{proof}

Now we assume that $u_0 \in H^{4N}(\Omega)$, $\eta_0 \in H^{4N+1/2}(\Sigma)$, $\mathfrak{F}_0 < \infty$, and that $\ns{\eta_0}_{4N-1/2} \le \mathfrak{E}_0(\eta) \le 1$ is sufficiently small for the hypothesis of Propositions \ref{l_stokes_regularity} and \ref{l_a_poisson_regularity} to hold when $k=4N$.  We will   iteratively construct the initial data $D_t^j u(0)$ for $j=0,\dotsc,2N$ and $\dt^j p(0)$ for $j=0,\dotsc,2N-1$.  To do so, we will first construct all but the highest order data,  and then we will state some compatibility conditions for the data.  These are necessary to construct $D_t^{2N} u(0)$ and $\dt^{2N-1}p(0)$, and to construct high-regularity solutions in Theorem \ref{l_linear_wp}.

We now turn to the construction of  $D_t^j u(0)$ for $j=0,\dotsc,2N-1$ and $\dt^j p(0)$ for $j=0,\dotsc,2N-2$, which will  employ Lemma \ref{l_iteration_estimates_2} in conjunction with estimate \eqref{l_ie1_05} of Lemma \ref{l_iteration_estimates_1} and \eqref{l_ie3_04} of Lemma \ref{l_iteration_estimate_3}.   For $j=0$ we write $F^{1,0}(0) = F^{1}(0) \in H^{4N-2}$, $F^{3,0}(0) = F^3(0) \in H^{4N-3/2}$, and $D_t^0 u(0) = u_0 \in H^{4N}.$  Suppose now that $F^{1,\ell} \in H^{4N - 2\ell-2}$, $F^{3,\ell} \in H^{4N-2\ell-3/2}$, and $D_t^\ell u(0) \in H^{4N-2\ell}$ are given for $0\le \ell \le j \in [0,2N-2]$; we will define $\dt^j p(0) \in H^{4N-2j-1}$ as well as $D_t^{j+1}u(0) \in H^{4N-2j-2}$, $F^{1,j+1}(0) \in H^{4N-2j-4}$, and $F^{3,j+1}(0) \in H^{4N-2j-7/2}$, which allows us to define all of said data via iteration.  By virtue of estimate \eqref{l_ie2_01}, we know that $f^1 = \mathfrak{f}^1(F^{1,j}(0),D_t^j u(0)) \in H^{4N-2j-3},$ $f^2 = \mathfrak{f}^2(F^{3,j}(0),D_t^j u(0)) \in H^{4N-2j-3/2},$ and $f^3 = \mathfrak{f}^3(F^{1,j}(0),D_t^j u(0)) \in H^{4N-2j-5/2}$.  This allows us to define $\dt^j p (0)$ as the solution to \eqref{l_linear_elliptic_p} with this choice of $f^1, f^2, f^3$, and then Proposition \ref{l_a_poisson_regularity} with $k = 4N$ and $r=4N-2j-1 < k$ implies that $\dt^j p(0) \in H^{4N-2j-1}$.  Now the estimates \eqref{l_ie1_05}, \eqref{l_ie3_04}, and \eqref{l_ie2_02} allow us to define 
\begin{equation}
 \begin{split}
D_t^{j+1} u(0) &:= \mathfrak{G}^0(F^{1,j}(0), D_t^j u(0),\dt^j p(0)) \in H^{4N-2j-2},  \\
F^{1,j+1}(0) &:= D_t F^{1,j}(0) + \mathfrak{G}^1(D_t^j u(0) ,\dt^j p(0)) \in H^{4N-2j-4}, \text{ and }\\
F^{3,j+1}(0) &:= \dt F^{3,j}(0) + \mathfrak{G}^3(D_t^j u(0) , \dt^j p(0)) \in H^{4N-2j-7/2}.
 \end{split}
\end{equation}
Using the above analysis, we  iteratively construct all of the desired data except for $D_t^{2N} u(0)$ and $\dt^{2N-1}p(0)$.

By construction, the initial data $D_t^j u(0)$ and $\dt^j p(0)$ are determined in terms of $u_0$ as well as  $\dt^\ell F^1(0)$ and $\dt^\ell F^3(0)$ for $\ell = 0,\dotsc,2N-1$.  In order to use these in Theorem \ref{l_strong_solution} and to construct $D_t^{2N} u(0)$ and $\dt^{2N-1}p(0)$, we must enforce compatibility conditions for $j=0,\dotsc,2N-1$.  For such $j$, we say that the $j^{th}$ compatibility condition is satisfied if 
\begin{equation}\label{l_comp_cond}
 \begin{cases}
  D_t^j u(0) \in \x(0) \cap H^2(\Omega) \\
  \Pi_0( F^{3,j}(0) + \sg_{\a_0} D_t^j u(0) \n_0)=0 .
 \end{cases}
\end{equation}
Note that the construction of $D_t^j u(0)$ and $\dt^j p(0)$ ensures that $D_t^j u(0) \in H^2(\Omega)$ and that $\diverge_{\a_0}(D_t^j u(0))=0$, so the condition $D_t^j u(0) \in \x(0) \cap H^2(\Omega)$ may be reduced to the condition $D_t^j u(0) \vert_{\Sigma} =0$.

It remains only to define $\dt^{2N-1} p(0)\in H^1$ and $D_t^{2N} u(0) \in H^0$.  According to the $j=2N-1$ compatibility condition \eqref{l_comp_cond}, $\diverge_{\a_0} D_t^{2N-1} u(0)=0$, which means that we can use  estimate  \eqref{l_ie2_03} of Lemma \ref{l_iteration_estimates_2} to see that $f^2 = \mathfrak{f}^2(F^{3,2N-1}(0),D_t^{2N-1} u(0)) \in H^{1/2}$ and $f^3 = \mathfrak{f}^3(F^{1,2N-1}(0),D_t^{2N-1} u(0)) \in H^{-1/2}$.  We also see from \eqref{l_comp_cond} that $g_0 = -\diverge_{\a_0} ( R(0) D_t^{2N-1} u(0)) \in H^0$.  Then since $G = -F^{1,2N-1} \in H^0$, we can  define $\dt^{2N-1} p(0) \in H^1$ as a weak solution to \eqref{l_linear_elliptic_p} in the sense of \eqref{l_a_poisson_div} with this choice of $f^2,$ $f^3,$ $g_0,$ and $G$  Then we define 
\begin{equation}
D_t^{2N} u(0) = \mathfrak{G}^0( F^{1,2N-1}(0), D_t^{2N-1} u(0), \dt^{2N-1} p(0)) \in H^0, 
\end{equation}
employing \eqref{l_ie2_02} for the inclusion in $H^0$.  In fact, the construction of $\dt^{2N-1} p(0)$ guarantees that $D_t^{2N} u(0) \in \y(0)$.  In addition to providing the above inclusions, the bounds \eqref{l_ie1_05}, \eqref{l_ie2_01}, \eqref{l_ie2_02} also imply the estimate 
\begin{equation}\label{l_init_data_estimate}
 \sum_{j=0}^{2N} \ns{ D_t^j u(0) }_{4N-2j} + \sum_{j=0}^{2N-1} \ns{ \dt^j p(0) }_{4N-2j-1} 
   \ls (1 + \mathfrak{E}_0(\eta) \left(  \ns{u_0}_{4N} + \mathfrak{F}_0   \right).
\end{equation}
Note that, owing to estimate \eqref{l_ie3_04}, the bound \eqref{l_init_data_estimate} also holds with $\dt^j u(0)$ replacing $D_t^j u(0)$ on the left.

Before stating our result on higher regularity for solutions to problem \eqref{l_linear_forced}, we define two quantities associated to $u,p$.  Write
\begin{equation}\label{l_DEfrak_def}
\begin{split}
 \mathfrak{D}(u,p) & := \sum_{j=0}^{2N+1} \ns{\dt^j u}_{L^2 H^{4N-2j +1}} + \sum_{j=0}^{2N}\ns{\dt^j p}_{ L^2 H^{4N-2j}}, \\
 \mathfrak{E}(u,p) & := \sum_{j=0}^{2N} \ns{\dt^j u}_{L^\infty H^{4N-2j }} + \sum_{j=0}^{2N-1}\ns{\dt^j p}_{L^\infty H^{4N-2j-1}}, \\
 \mathfrak{K}(u,p) & := \mathfrak{E}(u,p) + \mathfrak{D}(u,p) .
\end{split}
\end{equation}
Note that, again, Lemmas \ref{l_x_time_diff} and \ref{l_sobolev_infinity} imply that $\mathfrak{E}(u,p) \le C(T) \mathfrak{D}(u,p)$ for $C(T)$ a constant depending on $T$.  To avoid introducing this constant we will use both $\mathfrak{E}(u,p)$ and $\mathfrak{D}(u,p)$.

\begin{thm}\label{l_linear_wp}
Suppose that $u_0 \in H^{4N}(\Omega)$, $\eta_0 \in H^{4N+1/2}(\Sigma)$,  $\mathfrak{F} < \infty$, and that $\mathfrak{K}(\eta)\le 1$ is sufficiently small so that $\mathcal{K}(\eta)$, defined by \eqref{l_K_def}, satisfies the hypotheses of Theorem \ref{l_strong_solution}  and Proposition  \ref{l_a_poisson_regularity}.  Let $D_t^j u(0) \in H^{4N-2j}(\Omega)$ and $\dt^j p(0) \in H^{4N-2j-1}$ for $j=0,\dotsc,2N-1$ along with $D_t^{2N} u(0) \in \y(0)$ all be determined as above in terms of $u_0$ and $\dt^j F^1(0)$, $\dt^j F^3(0)$ for $j=0,\dotsc,2N-1$.  Suppose that for $j=0,\dotsc,2N-1$, the initial data satisfy the $j^{th}$ compatibility condition \eqref{l_comp_cond}.

Then there exists a unique strong solution $(u,p)$ to \eqref{l_linear_forced} so that 
\begin{equation}\label{l_lwp_00}
\begin{split}
 \dt^j u & \in C^0([0,T]; H^{4N-2j}(\Omega)) \cap L^2([0,T];H^{4N-2j+1}(\Omega)) \text{ for } j=0,\dotsc,2N, \\ 
 \dt^j p & \in C^0([0,T]; H^{4N-2j-1}(\Omega)) \cap L^2([0,T];H^{4N-2j}(\Omega)) \text{ for } j=0,\dotsc,2N-1, \\
\dt^{2N+1} u & \in (\h^1_T)^*, \text{ and } \dt^{2N} p \in L^2([0,T];H^0(\Omega)).
\end{split}
\end{equation}
The pair $(D_t^j u, \dt^j p)$ satisfies the PDE
\begin{equation}\label{l_lwp_01}
\begin{cases}
\dt (D_t^j u) - \da (D_t^j u) + \naba (\dt^j p) = F^{1,j} & \text{in }\Omega \\
\diva(D_t^j u)=0 & \text{in }\Omega\\
\Sa(\dt^j p, D_t^j u) \n =  F^{3,j} & \text{on }\Sigma \\
D_t^j u =0 & \text{on }\Sigma_b
\end{cases}
\end{equation}
in the strong sense with initial data $(D_t^j u(0), \dt^j p(0))$ for $j=0,\dotsc,2N-1$, and in the weak sense of \eqref{l_weak_solution_pressure} with initial data $D_t^{2N} u(0) \in \y(0)$ for $j=2N$.  Here the vectors $F^{1,j}$ and $F^{3,j}$ are as defined by \eqref{l_Fj_def}.  Moreover, the solution satisfies the estimate
\begin{equation}\label{l_lwp_02}
\mathfrak{E}(u,p) + \mathfrak{D}(u,p) \ls  (1+\mathfrak{E}_0(\eta) + \mathfrak{K}(\eta)) \exp\left(C(1+ \mathfrak{E}(\eta)) T \right) \left( \ns{u_0}_{4N} + \mathfrak{F}_0 + \mathfrak{F} \right)
\end{equation}
for a constant $C>0,$ independent of $\eta$.
\end{thm}

\begin{proof}

For notational convenience, throughout the proof we write
\begin{equation}
 \z := (1+ \mathfrak{E}_0(\eta) + \mathfrak{K}(\eta)) \exp\left(C(1+ \mathfrak{E}(\eta)) T \right) \left( \ns{u_0}_{4N} +  \mathfrak{F}_0+ \mathfrak{F} \right).
\end{equation}

Since the $0^{th}$ order compatibility condition \eqref{l_comp_cond} is satisfied and $\mathfrak{K}(\eta)$ is small enough for $\mathcal{K}(\eta)$ to satisfy the hypotheses of Theorem \ref{l_strong_solution}, we may apply Theorem \ref{l_strong_solution}.  It guarantees the existence of $(u,p)$ satisfying the inclusions $\dt^j u \in L^2 H^{3-2j}$ for $j=0,1,2$ and $\dt^j p \in L^2 H^{2-2j}$ for $j=0,1$.  The  $(D_t^j u,\dt^j p)$ are  solutions in that \eqref{l_lwp_01} is satisfied in the strong sense when $j=0$ and in the weak sense when $j=1$.  Finally, the estimate \eqref{l_ss_02} holds, but we may replace its right hand side by $\z$ since $\mathcal{K}(\eta) \le \mathfrak{E}(\eta) \le \mathfrak{K}(\eta)$. 

For an integer $m\ge 0$, let $\mathbb{P}_m$ denote the proposition asserting the following three statements.  First, that $(D_t^j u, \dt^j p)$ are solutions to \eqref{l_lwp_01} in the strong sense for $j=0,\dotsc,m$ and in the weak sense for $j=m+1$.  Second, that $\dt^j u \in L^2 H^{2m-2j+3}$ for $j=0,1,\dots, m+2$, $\dt^j u \in L^\infty H^{2m-2j+2}$ for $j=0,1,\dots, m+1$,  $\dt^j p \in L^2 H^{2m-2j+2}$ for $j=0,1,\dotsc,m+1$, and $\dt^j p \in L^\infty H^{2m-2j+1}$ for $j=0,1,\dotsc,m$.  Third, that the estimate
\begin{multline}\label{l_lwp_1}
\sum_{j=0}^{m+1} \ns{\dt^j u}_{L^\infty H^{2m-2j+2}} + \sum_{j=0}^{m} \ns{\dt^j p}_{L^\infty H^{2m-2j+1}} \\ +
 \sum_{j=0}^{m+2} \ns{\dt^j u}_{L^2 H^{2m-2j+3}} + \sum_{j=0}^{m+1} \ns{\dt^j p}_{L^2 H^{2m-2j+2}}
\ls \z
\end{multline}
holds.

The above analysis implies that $\mathbb{P}_0$ holds.  We claim that if $\mathbb{P}_m$ holds for some $m = 0,\dotsc,2N-2$, then $\mathbb{P}_{m+1}$ also holds.  Once the claim is established, a finite induction implies that $\mathbb{P}_m$ holds for all $m=0,\dotsc,2N-1$, which immediately implies all of the conclusions of the theorem.  The rest of the proof is dedicated to the proof of this claim.

Suppose that $\mathbb{P}_m$ holds for some $m = 0,\dotsc,2N-2$.  We may then combine \eqref{l_lwp_1} with the estimates \eqref{l_ie1_01}, \eqref{l_ie1_04}, and \eqref{l_ie1_02} of Lemma \ref{l_iteration_estimates_1} to see that
\begin{multline}\label{l_lwp_2} 
\sum_{j=1}^{m+1} \left( \ns{F^{1,j}}_{L^2 H^{2m-2j+3}}  +  \ns{F^{3,j}}_{L^2 H^{2m-2j+7/2}} + \ns{F^{1,j}}_{L^\infty H^{2m-2j+2}}  +  \ns{F^{3,j}}_{L^\infty H^{2m-2j+5/2}} \right)  
\\ 
+ \ns{\dt F^{1,m+1}}_{L^2 H^{-1}} +  \ns{ \dt F^{3,m+1}}_{L^2 H^{-1/2}} 
\ls (1+ \mathfrak{K}(\eta)) \left( \mathfrak{F} + \sum_{\ell=0}^{m+1} \ns{\dt^\ell u}_{L^\infty H^{2m-2\ell+2}} \right. \\
\left. + \sum_{\ell=0}^m \ns{\dt^\ell p}_{L^\infty H^{2m-2\ell+1}} 
+ \sum_{\ell=0}^{m+1} \ns{\dt^\ell u}_{L^2 H^{2m-2\ell+3}} 
+  \ns{\dt^\ell p}_{L^2 H^{2m-2\ell+2}} \right) \\
\ls (1+ \mathfrak{K}(\eta)) \left( \mathfrak{F} + \z \right) \ls \z.
\end{multline}
The last inequality in \eqref{l_lwp_2} follows from the fact that $\mathfrak{K}(\eta) \le 1$ and the definition of $\z$.  

We now show that the first assertion of $\mathbb{P}_{m+1}$ holds.  To this end, we note that the estimate \eqref{l_lwp_2}  implies that $F^{1,m+1}\in L^2 H^1$, $\dt F^{1,m+1} \in L^2 H^{-1}$, $F^{3,m+1} \in L^2 H^{3/2}$, and $\dt F^{3,m+1} \in L^2 H^{-1/2}$.  These inclusions, together with the fact that $D_t^{m+1} u(0)$ satisfies the $(m+1)^{st}$ order compatibility condition \eqref{l_comp_cond}, allow us to apply Theorem \ref{l_strong_solution} to solve problem \eqref{l_linear_forced}, with $F^1, F^3$ replaced by $F^{1,m+1}, F^{3,m+1}$ and with initial data $D_t^{m+1} u(0)$.  The resulting strong solution solution must equal $(D_t^{m+1} u, \dt^{m+1} p)$, the weak solution to \eqref{l_lwp_01} provided by $\mathbb{P}_{m}$, since strong solutions are also weak solutions and Proposition \ref{l_weak_unique} guarantees that weak solutions are unique.  Furthermore, the theorem implies that $(D_t^{m+2} u, \dt^{m+2} p)$ are a weak solution to \eqref{l_lwp_01}.  Since $\mathbb{P}_m$ already provided that $(D_t^j u, \dt^j p)$ are solutions to \eqref{l_lwp_01} in the strong sense for $j=0,\dotsc,m$, we deduce that the first assertion of $\mathbb{P}_{m+1}$ holds.

It remains to prove the the second and third assertions of $\mathbb{P}_{m+1}$; they are intertwined and will be derived simultaneously.  To begin, we note that the previous application of Theorem \ref{l_strong_solution} also provides, by way of \eqref{l_ss_02},  the estimate
\begin{multline}\label{l_lwp_3}
  \ns{D_t^{m+1} u}_{L^2 H^3} + \ns{\dt D_t^{m+1} u}_{L^2 H^1} + \ns{\dt^2 D_t^{m+1} u}_{L^2 H^{-1}} + \ns{\dt^{m+1} p}_{L^2 H^2} + \ns{\dt^{m+2} p}_{L^2 H^0} \\
+ \ns{D_t^{m+1} u}_{L^\infty H^2} + \ns{\dt D_t^{m+1} u}_{L^\infty H^0} + \ns{\dt^{m+1} p}_{L^\infty H^1}
\\
\ls 
(1+ \mathfrak{K}(\eta)) \exp\left(C(1+ \mathfrak{E}(\eta)) T \right) \left( \ns{D_t^{m+1}u(0) }_{ 2} + \ns{F^{1,m+1}(0)}_{0} + \ns{F^{3,m+1}(0)}_{1/2} + \mathfrak{F} \right) \\
\ls (1+ \mathfrak{E}_0(\eta) + \mathfrak{K}(\eta)) \exp\left(C(1+ \mathfrak{E}(\eta)) T \right) \left( \ns{u_0 }_{ 4N} + \mathfrak{F}_0  + \mathfrak{F} \right) \ls \z,
\end{multline}
where in the second inequality we have employed estimate \eqref{l_ie1_05} to control the $F^{i,m+1}(0)$ terms and the bound \eqref{l_init_data_estimate} to bound the resulting temporal derivatives or $u$ and $p$ at $t=0$.    The estimates of the $u$ terms in \eqref{l_lwp_3}, together with the estimates  \eqref{l_ie3_02}--\eqref{l_ie3_03} of Lemma \ref{l_iteration_estimate_3} and the estimate \eqref{l_lwp_1}, imply that
\begin{multline}\label{l_lwp_6}
 \ns{\dt^{m+1} u}_{L^2 H^3} + \ns{\dt^{m+2} u}_{L^2 H^1} + \ns{\dt^{m+3} u}_{L^2 H^{-1}} 
+  \ns{\dt^{m+1} u}_{L^\infty H^2} + \ns{\dt^{m+2} u}_{L^\infty H^0} \\
\ls 
(1+ \mathfrak{K}(\eta)) \left( \sum_{\ell=0}^{m+2}  \ns{\dt^\ell u}_{L^2 H^{2m-2\ell+3}} + \sum_{\ell=0}^{m+1} \ns{\dt^\ell u}_{L^\infty H^{2m-2\ell+2}} \right) + \z \\
\ls (1+ \mathfrak{K}(\eta)) \z + \z \ls \z.
\end{multline}
Hence
\begin{multline}\label{l_lwp_4}
\sum_{j=m+1}^{m+2}  \ns{\dt^{j} u}_{L^\infty H^{2(m+1) - 2j +2}}
\sum_{j=m+1}^{m+1} \ns{\dt^{j} p}_{L^\infty H^{2(m+1) -2j +1}} \\
+ \sum_{j=m+1}^{m+3} \ns{\dt^{j} u}_{L^2 H^{2(m+1) - 2j +3}} + \sum_{j=m+1}^{m+2} \ns{\dt^{j} p}_{L^2 H^{2(m+1) -2j +2}} 
\ls \z,
\end{multline}
which means that in order to derive the estimate \eqref{l_lwp_1} with $m$ replaced by $m+1$, it suffices to prove that
\begin{multline}\label{l_lwp_5}
 \sum_{j=0}^{m} \ns{\dt^{j} u}_{L^\infty H^{2(m+1) - 2j +2}} 
+ \ns{\dt^{j} p}_{L^\infty H^{2(m+1) -2j +1}} \\
+ \sum_{j=0}^{m} \ns{\dt^{j} u}_{L^2 H^{2(m+1) - 2j +3}} 
+ \ns{\dt^{j} p}_{L^2 H^{2(m+1) -2j +2}} \ls \z.
\end{multline}
Once \eqref{l_lwp_5} is established, summing \eqref{l_lwp_4} and \eqref{l_lwp_5} implies that \eqref{l_lwp_1} holds with $m$ replaced by $m+1$, which further implies that the second and third assertions of $\mathbb{P}_{m+1}$ hold, so that then all of $\mathbb{P}_{m+1}$ holds. 

In order to prove \eqref{l_lwp_5} we will use the elliptic regularity of Proposition \ref{l_stokes_regularity} (with $k=4N$) and an iteration argument.  The estimates of $D_t^{m+1} u$ in \eqref{l_lwp_3}, together with  \eqref{l_lwp_1} and the estimates \eqref{l_ie3_01} and \eqref{l_ie3_06} of Lemma \ref{l_iteration_estimate_3}, allow us to deduce that
\begin{equation}\label{l_lwp_7}
\ns{\dt D_t^{m} u}_{L^\infty H^2} + \ns{\dt D_t^{m} u}_{L^2 H^3} \ls \z.
\end{equation}
Since \eqref{l_lwp_01}  is satisfied in the strong sense for $j=m$, we may rearrange to find that for a.e. $t \in [0,T]$,  $(D_t^m,\dt^m p)$ solve the elliptic problem \eqref{l_linear_elliptic} with $F^1$ replaced by $F^{1,m} - \dt D_t^m u$, $F^2 =0$, and $F^3$ replaced by $F^{3,m}$.  We may then apply Proposition \ref{l_stokes_regularity} with $r=5$ to deduce that the estimate \eqref{l_s_reg_0} holds for a.e. $t \in [0,T]$; squaring this estimate and integrating over $[0,T]$ then yields the inequality
\begin{multline}\label{l_lwp_8}
 \ns{D_t^m u}_{L^2 H^5} + \ns{\dt^m p}_{L^2 H^4} \ls \ns{F^{1,m} - \dt D_t^m u}_{L^2 H^3} + \ns{F^{3,m}}_{L^2 H^{7/2}} \\
\ls \ns{F^{1,m}}_{L^2 H^3} +\ns{ \dt D_t^m u}_{L^2 H^3}  + \ns{F^{3,m}}_{L^2 H^{7/2}} \ls \z,
\end{multline}
where in the last inequality we have used \eqref{l_lwp_2} and \eqref{l_lwp_7}.  Similarly, we may apply Proposition \ref{l_stokes_regularity} with $r=4$ to deduce
\begin{equation}\label{l_lwp_8_2}
 \ns{D_t^m u}_{L^\infty H^4} + \ns{\dt^m p}_{L^\infty H^3} \ls \ns{F^{1,m} - \dt D_t^m u}_{L^\infty H^2} + \ns{F^{3,m}}_{L^\infty H^{5/2}} \ls \z.
\end{equation}
We may argue as before to deduce from \eqref{l_lwp_8} and \eqref{l_lwp_8_2} that  
\begin{equation}
\ns{\dt^m u}_{L^\infty H^4} + \ns{\dt^m u}_{L^2 H^5} \ls \z
 \end{equation}
as well.  This argument may be iterated to estimate $\dt^j u$, $\dt^j p$ for $j=1,\dotsc,m$; this yields the estimate
\begin{multline}\label{l_lwp_9}
 \sum_{j=1}^{m} \ns{\dt^{j} u}_{L^\infty H^{2(m+1) - 2j +2}} + \ns{\dt^{j} p}_{L^\infty H^{2(m+1) -2j +1}} \\
+  \sum_{j=1}^{m} \ns{\dt^{j} u}_{L^2 H^{2(m+1) - 2j +3}} + \ns{\dt^{j} p}_{L^2 H^{2(m+1) -2j +2}} \ls \z.
\end{multline}
We then apply Proposition \ref{l_stokes_regularity} with $r=2(m+1)+2 \le 4N$ to see that
\begin{multline}\label{l_lwp_10}
 \ns{u}_{L^\infty H^{2(m+1)+1}} + \ns{p}_{L^\infty H^{2(m+1)+1}} \ls \ns{F^{1} - \dt u}_{L^\infty H^{2(m+1)}} + \ns{F^{3}}_{L^\infty H^{2(m+1)+1/2}} \\
\ls \ns{F^{1}}_{L^\infty H^{2(m+1)}} + \ns{\dt u}_{L^\infty H^{2(m+1)}} + \ns{F^{3}}_{L^\infty H^{2(m+1)+1/2}}  \ls \z,
\end{multline}
and then again with  $r=2(m+1)+3 \le 4N+1$ to see that
\begin{multline}\label{l_lwp_10_2}
 \ns{u}_{L^2 H^{2(m+1)+3}} + \ns{p}_{L^2 H^{2(m+1)+2}} \ls \ns{F^{1} - \dt u}_{L^2 H^{2(m+1)+1}} + \ns{F^{3}}_{L^2 H^{2(m+1)+3/2}}  \\
+ \ns{\eta}_{L^2 H^{4N+1/2}} \left( \ns{F^{1} - \dt u}_{L^\infty H^{2}} + \ns{F^{3}}_{L^\infty H^{5/2}}  \right) 
\ls \ns{F^{1}}_{L^2 H^{2(m+1)+1}} + \ns{\dt u}_{L^2 H^{2(m+1)+1}} \\
+ \ns{F^{3}}_{L^2 H^{2(m+1)+3/2}} + \mathfrak{K}(\eta) (\mathfrak{F} +\z)  \ls \z.
\end{multline}
Summing \eqref{l_lwp_9}--\eqref{l_lwp_10_2} then gives \eqref{l_lwp_5}, completing the proof.

\end{proof}

\section{Preliminaries for the nonlinear problem}\label{lwp_4}

\subsection{Forcing estimates}

We want to eventually use our linear theory for the problem \eqref{l_linear_forced} in order to solve the nonlinear problem \eqref{geometric}.  To do so, we define forcing terms $F^1,F^3$ to be used in the linear theory that match the terms in \eqref{geometric}.  That is,  given $u,\eta$, we define 
\begin{equation}\label{l_F_forcing_def}
\begin{split}
F^1(u,\eta) &= \dt \bar{\eta} \tilde{b} K \p_3 u - u \cdot \naba u,  \text{ and } \\
F^3(u,\eta) &= \eta \n = -\eta D \eta + \eta e_3,
\end{split}
\end{equation}
where $\a,\n,K$ are determined as usual by $\eta$.  

We will need to be able to estimate various norms of $F^1(u,\eta)$ and $F^3(u,\eta)$ in terms of the  norms of $u$ and $\eta$ that appear in $\mathfrak{K}(\eta)$, $\mathfrak{E}_0(\eta)$, and $\mathfrak{K}(u,p)$, defined by \eqref{l_Kfrak_def}, \eqref{l_E0frak_def}, and \eqref{l_DEfrak_def}, respectively.  The norms of the $F^i$ terms are contained in $\mathfrak{F}$ and $\mathfrak{F}_0$, as defined by \eqref{l_Ffrak_def}.  We will actually need a slight modification of $\mathfrak{K}(u,p)$, which we define as
\begin{equation}\label{l_Kfrak_u2n_def}
 \mathfrak{K}_{2N}(u) = \sum_{j=0}^{2N} \ns{\dt^j u}_{L^2 H^{4N-2j +1}} 
+ \ns{\dt^j u}_{L^\infty H^{4N-2j }}.
\end{equation}

Our estimates are the content of the following lemma.

\begin{lem}\label{l_Ffrak_bound}
Suppose that $\mathfrak{K}(\eta) \le 1$ and $\mathfrak{K}_{2N}(u) < \infty$.  Then 
\begin{equation}\label{l_ffb_0}
 \mathfrak{F}(F^1(u,\eta),F^3(u,\eta)) \ls \left[1+T + \mathfrak{K}(\eta)\right] \mathfrak{E}(\eta)    +\mathfrak{K}(\eta)\left[  \mathfrak{K}_{2N}(u) + (\mathfrak{K}_{2N}(u))^2\right] + (\mathfrak{K}_{2N}(u))^2.
\end{equation}
\end{lem}
\begin{proof}
All terms in the definition of $F^1(u,\eta)$, $F^3(u,\eta)$ are quadratic or higher-order except the term $\eta e_3$ in $F^3$.  As such, we may argue as in the proof of  Lemma \ref{l_iteration_estimates_1} to deduce the bound
\begin{equation}\label{l_ffb_1}
 \mathfrak{F}(F^1(u,\eta),F^3(u,\eta) - \eta e_3) \ls \mathfrak{E}(\eta) \mathfrak{K}(\eta)  +  \mathfrak{K}(\eta)(\mathfrak{K}(\eta) +  \mathfrak{K}_{2N}(u) + (\mathfrak{K}_{2N}(u))^2) + (\mathfrak{K}_{2N}(u))^2. 
\end{equation}
Here the appearance of the term $\mathfrak{E}(\eta) \mathfrak{K}(\eta)$ is due to the term $\eta D \eta$ in $F^3$, while the appearance of  $\mathfrak{K}_{2N}(u)^2$ is due to the term $u \cdot \nab u$ that appears when we write  $u \cdot \naba u = u \cdot \nab u + u \cdot \nab_{\a - I} u$ in $F^1$.

On the other hand, by definition, we have
\begin{multline}\label{l_ffb_2}
 \mathfrak{F}(0, \eta e_3) = \sum_{j=0}^{2N} \ns{\dt^j \eta}_{L^2 H^{4N-2j-1/2}} + \sum_{j=0}^{2N-1} \ns{\dt^j \eta}_{L^\infty H^{4N- 2j -3/2}} \\
\ls (1+T) \sum_{j=0}^{2N} \ns{\dt^j \eta}_{L^\infty H^{4N-2j}} = (1+T) \mathfrak{E}(\eta).
\end{multline}
Then, since $\mathfrak{F}(X,Y+Z) \ls \mathfrak{F}(X,Y) + \mathfrak{F}(0,Z)$, we may combine \eqref{l_ffb_1} with \eqref{l_ffb_2} to deduce \eqref{l_ffb_0}.

\end{proof}

\subsection{Data estimates}\label{l_data_section}

In the construction of the initial data performed after Lemma \ref{l_iteration_estimates_2} it was assumed that $\dt^j \eta(0)$ for $j=0,\dotsc,2N$ and $\dt^j F^1(0)$, $\dt^j F^3(0)$ for $j=0,\dotsc,2N-1$  were all known.  Knowledge of the former allowed us to compute  $R(0)$, $\a_0$, $\n_0$, etc along with their temporal derivatives; these quantities then served as coefficients in deriving the initial conditions for $u,p$ and their temporal derivatives.  Since for the full nonlinear problem the function $\eta$ is unknown and its evolution is coupled to that of $u$ and $p$, we must revise the construction of the data to include this coupling, assuming only that $u_0$ and $\eta_0$ are given.  This will also reveal the compatibility conditions that must be satisfied by $u_0$ and $\eta_0$ in order to solve the nonlinear problem \eqref{geometric}. To this end we first define the quantities
\begin{equation}\label{l_EF0_def}
 \mathcal{E}_0 := \ns{u_0}_{4N} + \ns{\eta_0}_{4N}, \text{ and }  \mathcal{F}_0 :=  \ns{\eta_0}_{4N+1/2}.
\end{equation}
For our estimates we must also introduce the quantity
\begin{equation}\label{l_E0frak_up_def}
 \mathfrak{E}_0(u,p) = \sum_{j=0}^{2N} \ns{\dt^j u(0)}_{4N-2j} + \sum_{j=0}^{2N-1} \ns{\dt^j p(0)}_{4N-2j-1}.
\end{equation}

We will also need a more exact enumeration of the terms in $\mathfrak{E}_0(u,p)$, $\mathfrak{E}_0(\eta)$, and $\mathfrak{F}_0$ (as defined in \eqref{l_E0frak_up_def}, \eqref{l_E0frak_def}, and \eqref{l_Ffrak_def}, respectively).  For $j=0,\dotsc,2N-1$ we define
\begin{equation}\label{l_F0_prec_def}
 \mathfrak{F}_0^j(F^1(u,\eta),F^3(u,\eta)) := \sum_{\ell=0}^j \ns{\dt^\ell F^1(0)}_{4N-2\ell -2} + \ns{\dt^\ell F^3(0) }_{4N-2\ell -3/2}, \text{ and }
\end{equation}
\begin{equation}\label{l_E0_prec_eta_def}
 \mathfrak{E}_0^j(\eta) := \ns{\eta_0}_{4N} + \ns{\dt \eta(0)}_{4N-1} +  \sum_{\ell=2}^j \ns{\dt^\ell \eta(0)}_{4N-2\ell +3/2},
\end{equation}
with the sum in \eqref{l_E0_prec_eta_def} only including the first term when $j=0$ and only the first two terms when $j=1$.  For $j=0$ we write $\mathfrak{E}_0^0(u,p) := \ns{u_0}_{4N}$, and for $j=1,\dotsc,2N$ we  write
\begin{equation}\label{l_E0_prec_up_def}
 \mathfrak{E}_0^j(u,p) := \sum_{\ell=0}^j \ns{\dt^\ell u(0)}_{4N-2j} + \sum_{\ell=0}^{j-1} \ns{\dt^\ell p(0)}_{4N-2j-1}.
\end{equation}

The following lemma records more refined versions of the estimates \eqref{l_ie1_05} and \eqref{l_ie3_04} as well as some other related estimates that are useful in dealing with the initial data.

\begin{lem}\label{l_full_iteration_estimates}

For  $F^1(u,\eta)$ and $F^3(u,\eta)$ defined by \eqref{l_F_forcing_def} and $j=0,\dotsc,2N-1$, it holds that
\begin{equation}\label{l_fie_01}
 \mathfrak{F}_0^j(F^1(u,\eta),F^3(u,\eta)) \le P_j( \mathfrak{E}_0^{j+1}(\eta),\mathfrak{E}_0^j(u,p) )
\end{equation}
for $P_j(\cdot,\cdot)$ a polynomial so that $P_j(0,0)=0$.

For $j=1,\dotsc,2N-1$ let $F^{1,j}(0)$ and $F^{3,j}(0)$ be determined by \eqref{l_Fj_def} and \eqref{l_F_forcing_def}, using $\dt^\ell \eta(0)$, $\dt^\ell u(0)$, and $\dt^\ell p(0)$ for appropriate values of $\ell$.  Then
\begin{equation}\label{l_fie_02}
 \ns{F^{1,j}(0)}_{4N-2j-2} + \ns{F^{3,j}(0)}_{4N-2j-3/2} \le P_j( \mathfrak{E}_0^{j+1}(\eta),\mathfrak{E}_0^j(u,p) )
\end{equation}
for $P_j(\cdot,\cdot)$ a polynomial so that $P_j(0,0)=0$.

For $j=0,\dotsc,2N$ it holds that
\begin{equation}\label{l_fie_03}
 \ns{\dt^j u(0) - D_t^j u(0) }_{4N-2j} \le P_j( \mathfrak{E}_0^{j}(\eta),\mathfrak{E}_0^j(u,p) )
\end{equation}
for $P_j(\cdot,\cdot)$ a polynomial so that $P_j(0,0)=0$.

For $j=1,\dotsc,2N-1$ it holds that
\begin{equation}\label{l_fie_04}
 \ns{ \sum_{\ell=0}^{j} \binom{j}{\ell} \dt^\ell \n(0) \cdot \dt^{j-\ell} u(0) }_{H^{4N-2j+3/2}(\Sigma)} \le P_j( \mathfrak{E}_0^{j}(\eta),\mathfrak{E}_0^j(u,p) )
\end{equation}
for $P_j(\cdot,\cdot)$ a polynomial so that $P_j(0,0)=0$.  Also, 
\begin{equation}\label{l_fie_05}
\ns{u_0 \cdot \n_0}_{H^{4N-1}(\Sigma)} \ls \ns{u_0}_{4N}+\ns{\eta_0}_{4N}.
\end{equation}
\end{lem}
\begin{proof}
These bounds may be derived by arguing as in the proof of Lemma \ref{l_iteration_estimates_1}.  As such, we again omit the details. 
\end{proof}

This lemma allows us to modify the construction presented after Lemma \ref{l_iteration_estimates_2} to construct all of the initial data $\dt^j u(0)$, $\dt^j \eta(0)$ for $j=0,\dotsc,2N$ and $\dt^j p(0)$ for $j=0,\dotsc,2N-1$.  Along the way we will also derive estimates of $\mathfrak{E}_0(u,p) + \mathfrak{E}_0(\eta)$ in terms of $\mathcal{E}_0$ and determine the compatibility conditions for $u_0,\eta_0$ necessary for existence of solutions to \eqref{geometric}.

We assume that $u_0, \eta_0$ satisfy $\mathcal{F}_0 <\infty$ and that $\ns{\eta_0}_{4N-1/2} \le \mathcal{E}_0\le 1$ is sufficiently small for the hypothesis of Proposition  \ref{l_a_poisson_regularity} to hold when $k=4N$.  As before, we will iteratively construct the initial data, but this time we will use the estimates in Lemma \ref{l_full_iteration_estimates}.  Define $\dt \eta(0) = u_0 \cdot \n_0$, where $u_0 \in H^{4N-1/2}(\Sigma)$ when traced onto $\Sigma$, and $\n_0$ is determined in terms of $\eta_0$.  Estimate \eqref{l_fie_05} implies that
$\ns{\dt \eta(0)}_{4N-1} \ls \mathcal{E}_0$, and hence that $\mathfrak{E}_0^0(u,p) + \mathfrak{E}_0^1(\eta) \ls \mathcal{E}_0$.  We may use this bound in \eqref{l_fie_01} with $j=0$ to find that
\begin{equation}\label{l_dc_0}
\mathfrak{F}_0^0(F^1(u,\eta),F^3(u,\eta)) \le  P_0( \mathfrak{E}_0^{1}(\eta),\mathfrak{E}_0^0(u,p) ) \le P(\mathcal{E}_0)
\end{equation}
for a polynomial $P(\cdot)$ so that $P(0)=0$.  Note that in this estimate and in the estimates below, we employ a convention with polynomials of $\mathcal{E}_0$ similar to the one we employ with constants: they are allowed to change from line to line, but they always satisfy $P(0)=0$. 

Suppose now that $j \in [0,2N-2]$ and that $\dt^\ell u(0)$ are known for $\ell=0,\dotsc,j$, $\dt^\ell \eta(0)$ are known for $\ell = 0,\dotsc,j+1$, and $\dt^\ell p(0)$ are known for $\ell = 0,\dotsc,j-1$ (with the understanding that nothing is known of $p(0)$ when $j=0$), and that
\begin{equation}\label{l_dc_1}
 \mathfrak{E}_0^{j}(u,p) + \mathfrak{E}_0^{j+1}(\eta) + \mathfrak{F}_0^{j}(F^1(u,\eta),F^3(u,\eta))  \le P(\mathcal{E}_0).
\end{equation}
According to the estimates \eqref{l_fie_02} and \eqref{l_fie_03}, we then know that
\begin{equation}\label{l_dc_2}
 \ns{F^{1,j}(0)}_{4N-2j-2} + \ns{F^{3,j}(0)}_{4N-2j-3/2} + \ns{D_t^j u(0)}_{4N-2j} \le  P(\mathcal{E}_0).
\end{equation}
By virtue of estimates \eqref{l_ie2_01} and \eqref{l_dc_1}, we know that 
\begin{multline}\label{l_dc_2_2}
 \ns{\mathfrak{f}^1(F^{1,j}(0),D_t^j u(0))}_{4N-2j-3} + \ns{\mathfrak{f}^2(F^{3,j}(0),D_t^j u(0))}_{4N-2j-3/2} \\
+ \ns{\mathfrak{f}^3(F^{1,j}(0),D_t^j u(0))}_{4N-2j-5/2} \le P(\mathcal{E}_0).
\end{multline}
This allows us to define $\dt^j p (0)$ as the solution to \eqref{l_linear_elliptic_p} with  $f^1, f^2, f^3$ given by $\mathfrak{f}^1$, $\mathfrak{f}^2$, $\mathfrak{f}^3$.  Then Proposition \ref{l_a_poisson_regularity} with $k = 4N$ and $r=4N-2j-1 < k$ implies that 
\begin{equation}\label{l_dc_3}
 \ns{\dt^j p(0)}_{4N-2j-1} \le   P(\mathcal{E}_0).
\end{equation}
Now the estimates \eqref{l_ie2_02}, \eqref{l_dc_1}, and \eqref{l_dc_2} allow us to define 
\begin{equation}
D_t^{j+1} u(0) := \mathfrak{G}^0(F^{1,j}(0), D_t^j u(0),\dt^j p(0)) \in H^{4N-2j-2}, 
\end{equation}
and owing to \eqref{l_fie_03}, we have the estimate 
\begin{equation}\label{l_dc_4}
\ns{\dt^{j+1} u(0)}_{4N-2(j+1)} \le P(\mathcal{E}_0). 
\end{equation}
Now we define $\dt^{j+2} \eta(0) = \sum_{\ell=0}^{j+1} \binom{j}{\ell} \dt^\ell \n(0) \cdot \dt^{j-\ell} u(0)$.  The estimate \eqref{l_fie_05}, together with \eqref{l_dc_1} and \eqref{l_dc_4} then imply that
\begin{equation}\label{l_dc_5}
\ns{\dt^{j+2} \eta(0)}_{4N-2(j+2)+3/2} \le P(\mathcal{E}_0). 
\end{equation}
We may combine \eqref{l_dc_1} with \eqref{l_dc_3}--\eqref{l_dc_5} to deduce that 
\begin{equation}\label{l_dc_6}
  \mathfrak{E}_0^{j+1}(u,p) + \mathfrak{E}_0^{j+2}(\eta)   \le P(\mathcal{E}_0),
\end{equation}
but then \eqref{l_fie_01} implies that $\mathfrak{F}_0^{j+1}(F^1(u,\eta),F^3(u,\eta))  \le P(\mathcal{E}_0)$ as well, and we deduce that the bound \eqref{l_dc_1} also holds with $j$ replaced by $j+1$.

Using the above analysis, we may iterate from $j=0,\dotsc,2N-2$ to deduce that
\begin{equation}\label{l_dc_7}
  \mathfrak{E}_0^{2N-1}(u,p) + \mathfrak{E}_0^{2N}(\eta) + \mathfrak{F}_0^{2N-1}(F^1(u,\eta),F^3(u,\eta))  \le P(\mathcal{E}_0).
\end{equation}
After this iteration, it remains only to define $\dt^{2N-1} p(0)$ and  $D_t^{2N} u(0)$.  In order to do this, we must first impose the compatibility conditions on $u_0$ and $\eta_0$.  These are the same as in \eqref{l_comp_cond}, but because now the temporal derivatives of $\eta$ have been constructed as well, we restate them in a slightly different way.  Let $\dt^j u(0)$, $F^{1,j}(0)$, $F^{3,j}(0)$ for $j=0,\dots,2N-1$, $\dt^j \eta(0)$ for $j=0,\dotsc,2N$, and $\dt^j p(0)$ for $j=0,\dotsc,2N-2$ be constructed in terms of $\eta_0, u_0$ as above.  Let $\Pi_0$ be the projection defined in terms of $\eta_0$ as in \eqref{l_Pi0_def} and $D_t$ be the operator defined by \eqref{l_Dt_def}.  We say that $u_0,\eta_0$ satisfy the $(2N)^{th}$ order compatibility conditions if
\begin{equation}\label{l_comp_cond_2N}
\begin{cases}
 \diverge_{\a_0} (D_t^j u(0)) =0 & \text{in }\Omega \\
 D_t^j u(0) =0 & \text{on } \Sigma_b \\
 \Pi_0 (F^{3,j}(0) + \sg_{\a_0} D_t^j u(0) \n_0 ) = 0 & \text{on } \Sigma
\end{cases}
\end{equation}
for $j=0,\dotsc,2N-1$.  Note that if $u_0,\eta_0$ satisfy \eqref{l_comp_cond_2N}, then the $j^{th}$ order compatibility condition \eqref{l_comp_cond} is satisfied for $j=0,\dotsc,2N-1$.

Now we define $\dt^{2N-1} p(0)$ and  $D_t^{2N} u(0)$. We use the compatibility conditions \eqref{l_comp_cond_2N} and argue as above and in the derivation of \eqref{l_ie2_03} in Lemma \ref{l_iteration_estimates_2} to estimate
\begin{equation}\label{l_dc_10}
\ns{\mathfrak{f}^2(F^{3,2N-1}(0),D_t^{2N-1} u(0))}_{1/2} 
+ \ns{\mathfrak{f}^3(F^{1,2N-1}(0),D_t^{2N-1} u(0))}_{-1/2} \le P(\mathcal{E}_0)
\end{equation}
and
\begin{equation}\label{l_dc_11}
 \ns{F^{1,2N-1}(0)}_{0} + \ns{\diverge_{\a_0}(R(0) D_t^{2N-1} u(0))  }_{0} \le P(\mathcal{E}_0).
\end{equation}
We then  define $\dt^{2N-1} p(0) \in H^1$ as a weak solution to \eqref{l_linear_elliptic_p} in the sense of \eqref{l_a_poisson_div} with this choice of $f^2=\mathfrak{f}^2$,  $f^3= \mathfrak{f}^3,$ $g_0=-\diverge_{\a_0}(R(0) D_t^{2N-1} u(0))$, and $G=-F^{1,2N-1}(0)$.  The estimate \eqref{l_a_poisson_weak_estimate}, when combined with \eqref{l_dc_10}--\eqref{l_dc_11}, allows us to deduce that
\begin{equation}\label{l_dc_8}
 \ns{\dt^{2N-1} p(0)}_{1} \le   P(\mathcal{E}_0). 
\end{equation}
Then we set $D_t^{2N} u(0) = \mathfrak{G}^0(F^{1,2N-1}(0), D_t^{2N-1} u(0),\dt^{2N-1} p(0))$, employing \eqref{l_ie2_02} to see that $D_t^{2N} \in H^0$.  In fact, the construction of $\dt^{2N-1} p(0)$ guarantees that $D_t^{2N} u(0) \in \y(0)$.  Arguing as before, we  also have the estimate 
\begin{equation}\label{l_dc_9}
\ns{\dt^{2N} u(0)}_{0} \ls P(\mathcal{E}_0) 
\end{equation}
This completes the construction of the initial data, but we will record a form of the estimates \eqref{l_dc_7}, \eqref{l_dc_8}--\eqref{l_dc_9} in the following proposition.

\begin{prop}\label{l_data_norm_comparison}
Suppose that $u_0, \eta_0$ satisfy $\mathcal{F}_0 <\infty$ and that $\mathcal{E}_0\le 1$ is sufficiently small for the hypothesis of Proposition  \ref{l_a_poisson_regularity} to hold when $k=4N$.  Let $\dt^j u(0)$, $\dt^j \eta(0)$ for $j=0,\dotsc,2N$ and $\dt^j p(0)$ for $j=0,\dotsc,2N-1$ be given as above.  Then
\begin{equation}\label{l_dnc_0}
\mathcal{E}_0 \le \mathfrak{E}_0(u,p) + \mathfrak{E}_0(\eta) \ls \mathcal{E}_0
\end{equation}
\end{prop}
\begin{proof}
The first inequality in \eqref{l_dnc_0} is trivial.  Summing \eqref{l_dc_7} and \eqref{l_dc_8}--\eqref{l_dc_9} yields the estimate $\mathfrak{E}_0(u,p) + \mathfrak{E}_0(\eta) \le P(\mathcal{E}_0)$ for a polynomial $P$ satisfying $P(0)=0$.  Since $\mathcal{E}_0\le 1$, we have that $P(\mathcal{E}_0) \ls \mathcal{E}_0$, and the last inequality in \eqref{l_dnc_0} follows directly.  
\end{proof}

\subsection{Transport problem}

Thus far we have considered solving for $(u,p)$, given $\eta$.  Now we discuss how to solve for $\eta$, given $u$ (more precisely, its trace on $\Sigma$).  We do so by considering the transport problem
\begin{equation}\label{l_transport_equation}
\begin{cases}
\dt \eta + u_1 \p_1 \eta  + u_2 \p_2 \eta  = u_3 & \text{in } \Rn{2}\\
\eta(0) = \eta_0.
\end{cases}
\end{equation}

We now state a well-posedness theory for \eqref{l_transport_equation} involving the quantities $\mathcal{E}_0$, $\mathcal{F}_0$, $\mathfrak{K}_{2N}(u)$, $\mathfrak{K}(\eta)$ as defined by \eqref{l_EF0_def}, \eqref{l_Kfrak_u2n_def}, \eqref{l_Kfrak_def}, respectively.  We will also need one more quantity, which we write as
\begin{equation}\label{l_Fcal_def}
 \mathcal{F}(\eta) := \ns{\eta}_{L^\infty H^{4N+1/2}}.
\end{equation}

\begin{thm}\label{l_transport_theorem}
Suppose that $u_0,\eta_0$ satisfy $\mathcal{F}_0 < \infty$  and that $\mathfrak{E}_0(\eta) \le 1$ is sufficiently small for the hypothesis of Proposition \ref{l_a_poisson_regularity} to hold when $k=4N$.  Let $\dt^j \eta(0), \dt^j u(0)$ for $j=1,\dotsc,2N$ be defined in terms of $u_0,\eta_0$ as in Section \ref{l_data_section} and suppose that $u$ satisfies  $\mathfrak{K}_{2N}(u) \le 1$  and achieves the initial conditions $\dt^j u(0)$ for $j=0,\dotsc,2N$.  Then  the problem  \eqref{l_transport_equation} admits a unique solution $\eta$ that satisfies $\mathcal{F}(\eta) + \mathfrak{K}(\eta) < \infty$ and achieves the initial data $\dt^j \eta(0)$ for $j=0,\dotsc,2N$.  Moreover, there exists a $0 < \bar{T}\le 1$, depending on $N$, so that if $0 < T \le \bar{T} \min\{1, 1/\mathcal{F}_0\}$, then we have the 
estimates
\begin{equation}\label{l_tt_01}
 \mathcal{F}(\eta) \ls \mathcal{F}_0 + T \mathfrak{K}_{2N}(u), 
\end{equation}
\begin{equation}\label{l_tt_02}
  \mathfrak{E}(\eta) \ls \mathcal{E}_0 +  T \mathfrak{K}_{2N}(u),
\end{equation}
\begin{equation}\label{l_tt_03}
\mathfrak{D}(\eta) \ls  \mathcal{E}_0 +  T \mathcal{F}_0  + \mathfrak{K}_{2N}(u).
\end{equation}
\end{thm}
\begin{proof}

The proof proceeds through three steps.  We first establish the solvability of problem \eqref{l_transport_equation}, then we establish  the $L^\infty H^k$ estimates needed to bound $\mathfrak{E}(\eta)$ as in \eqref{l_tt_02}, and then we handle the $L^2 H^k$ estimates for the terms in $\mathfrak{D}(\eta)$ to derive \eqref{l_tt_03}.  Summing the bounds \eqref{l_tt_02} and \eqref{l_tt_03} shows that $\mathfrak{K}(\eta) < \infty$.

Step 1 -- Solving the transport equation

The assumptions on $u$ imply, via trace theory, that  $u \in L^2([0,T]; H^{4N+1/2}(\Sigma))$, which allows us to employ the a priori estimates for solutions of the transport equation derived in \cite{danchin} (more specifically, Proposition 2.1 with $p=p_2 =r=2$, $\sigma = 4N+1/2$).  Although the well-posedness of \eqref{l_transport_equation} is not proved in \cite{danchin}, it can be deduced from the a priori estimates in a standard way; full details are provided in Theorem 3.3.1 of the unpublished note \cite{danchin_notes}.  The result is that \eqref{l_transport_equation} admits a unique solution $\eta \in C^0([0,T]; H^{4N+1/2}(\Sigma))$ with $\eta(0) = \eta_0$ that satisfies the estimate 
\begin{equation}\label{l_tt_1}
 \norm{\eta}_{L^\infty H^{4N+1/2}} \le \exp\left( C \int_0^T \norm{u(t)}_{H^{4N+1/2}(\Sigma)}dt  \right) \left( \sqrt{\mathcal{F}_0} + \int_0^T \norm{u_3(t)}_{H^{4N+1/2}(\Sigma)}dt \right)
\end{equation}
for $C>0$.   By  trace theory, we have $\norm{u(t)}_{H^{4N+1/2}(\Sigma)} \ls \sqrt{\mathfrak{K}_{2N}(u)}$, so that the Cauchy-Schwarz inequality implies
$C \int_0^T \norm{u(t)}_{H^{4N+1/2}(\Sigma)}dt \ls \sqrt{T} \sqrt{\mathfrak{K}_{2N}(u)} \ls \sqrt{T} $,
and hence that 
\begin{equation}\label{l_tt_2}
\exp\left( C \int_0^T \norm{u(t)}_{H^{4N+1/2}(\Sigma)}dt  \right) \le 2 
\end{equation}
for $T \le \bar{T}$ with $\bar{T} \le 1$ sufficiently small.  We deduce from \eqref{l_tt_1} and \eqref{l_tt_2} that 
\begin{equation}\label{l_tt_3}
 \sqrt{\mathcal{F}(\eta)} \le 2( \sqrt{\mathcal{F}_0} + \sqrt{T \mathfrak{K}_{2N}(u)}), 
\end{equation}
from which \eqref{l_tt_01} easily follows.

Step 2 -- Bounding $\mathfrak{E}(\eta)$

Proposition 2.1 of \cite{danchin} also implies the a priori estimate
\begin{multline}\label{l_tt_4}
 \norm{\eta}_{L^\infty H^{4N}} \le \exp\left( C \int_0^T \norm{u(t)}_{H^{4N+1/2}(\Sigma)}dt  \right) \left( \norm{\eta_0}_{4N} + \int_0^T \norm{u_3(t)}_{H^{4N}(\Sigma)}dt \right) \\
\ls( \sqrt{\mathfrak{E}_0(\eta)} + \sqrt{T \mathfrak{K}_{2N}(u)}),
\end{multline}
where we have used the smallness of $\bar{T}$, trace theory, and Cauchy-Schwarz as above.  Since $\dt \eta$ satisfies  $\dt \eta = u_3 - D \eta \cdot u$  and  $\mathfrak{K}_{2N}(u) < \infty$, we know that  $\dt \eta$ is temporally differentiable and satisfies $\dt (\dt \eta) + u \cdot D (\dt \eta) = \dt u_3 - \dt u \cdot D \eta$ with initial condition $\dt \eta(0) = u_0 \cdot \n_0$, which matches the initial data constructed in terms of $u_0,\eta_0$.  We may again apply Proposition 2.1 of \cite{danchin} and then use \eqref{l_tt_4} to find
\begin{multline}\label{l_tt_5}
 \norm{\dt \eta}_{L^\infty H^{4N-2}} \le 2 \left( \norm{\dt \eta(0)}_{4N-2} + \int_0^T \norm{\dt u_3}_{H^{4N-2}(\Sigma)}  + \norm{\dt u \cdot D \eta}_{H^{4N-2}(\Sigma)}  \right) \\
\ls  \norm{\dt \eta(0)}_{4N-2} + \left( 1+  \norm{ \eta}_{L^\infty H^{4N-1}} \right) \int_0^T \norm{\dt u}_{H^{4N-2}(\Sigma)}     
\ls  \sqrt{\mathfrak{E}_0(\eta)} \\ 
+ \sqrt{T \mathfrak{K}_{2N}(u)}   \left( 1+  \norm{ \eta}_{L^\infty H^{4N-1}} \right) \ls \sqrt{\mathfrak{E}_0(\eta)} + \sqrt{T \mathfrak{K}_{2N}(u)}   \left( 1+ \sqrt{\mathfrak{E}_0(\eta)} + \sqrt{T \mathfrak{K}_{2N}(u)}  \right) \\
\ls P( \sqrt{\mathfrak{E}_0(\eta)}, \sqrt{T \mathfrak{K}_{2N}(u)} )  
\end{multline}
for a polynomial $P(\cdot,\cdot)$ with $P(0,0)=0$.  A straightforward modification of this argument allows us to iterate to obtain,  for $j=1,\dotsc,2N$, the estimate
\begin{equation}\label{l_tt_6}
 \norm{\dt^j \eta}_{L^\infty H^{4N-2j}}  \le P( \sqrt{\mathfrak{E}_0(\eta)}, \sqrt{T \mathfrak{K}_{2N}(u)} )
\end{equation}
for $P(\cdot,\cdot)$ a polynomial with $P(0,0)=0$.  We also find that the initial data $\dt^j \eta(0)$ is achieved for $j=0,\dots,2N$.  Squaring \eqref{l_tt_4} and \eqref{l_tt_6} and summing, we deduce that $\mathfrak{E}(\eta) \le P( \mathfrak{E}_0(\eta), T \mathfrak{K}_{2N}(u) )$ for another polynomial with $P(0,0)=0$.  Since $ \mathfrak{E}_0(\eta)\le 1$ and  $T \mathfrak{K}_{2N}(u) \le \bar{T} \mathfrak{K}_{2N}(u) \le 1$, we then have that 
\begin{equation}\label{l_tt_7}
 \mathfrak{E}(\eta) \ls \mathfrak{E}_0(\eta) +  T \mathfrak{K}_{2N}(u),
\end{equation}
which yields \eqref{l_tt_02} when combined with Proposition \ref{l_data_norm_comparison}.

Step 3 -- Bounding $\mathfrak{D}(\eta)$

Now we control the terms in $\mathfrak{D}(\eta)$.  From \eqref{l_tt_3}, Cauchy-Schwarz, and the fact that $T \le 1$, we see that
\begin{equation}\label{l_tt_8}
 \norm{\eta}_{L^2 H^{4N+1/2}} \le \sqrt{T} \sqrt{\mathcal{F}(\eta)} \le 2(\sqrt{T \mathcal{F}_0} + \sqrt{\mathfrak{K}_{2N}(u)}).  
\end{equation}
We may then use the equation \eqref{l_transport_equation}, trace theory, the fact that $H^{4N-1/2}(\Sigma)$ is an algebra, and estimate \eqref{l_tt_8} to bound
\begin{multline}\label{l_tt_9}
  \norm{\dt \eta}_{L^2 H^{4N-1/2}} \ls \norm{u_3}_{L^2 H^{4N-1/2}} + \norm{u}_{L^\infty H^{4N-1/2}} \norm{\eta}_{L^2 H^{4N+1/2}} \\
 \ls \sqrt{ \mathfrak{K}_{2N}(u) } (1 +   \sqrt{T \mathcal{F}_0} + \sqrt{\mathfrak{K}_{2N}(u)} ) \ls  P(\sqrt{T \mathcal{F}_0} , \sqrt{ \mathfrak{K}_{2N}(u) } )
\end{multline}
for $P$ a polynomial with $P(0,0)=0$.  We argue similarly (employing \eqref{l_tt_9} along the way) to find that
\begin{multline}\label{l_tt_10}
  \norm{\dt^2 \eta}_{L^2 H^{4N-3/2}} \ls \norm{\dt u_3}_{L^2 H^{4N-1/2}} + \norm{\eta}_{L^\infty H^{4N-1/2}} \norm{\dt u}_{L^2 H^{4N-3/2}} \\ 
+ \norm{\dt \eta}_{L^2 H^{4N-1/2}} \norm{u}_{L^\infty H^{4N-3/2}}  
 \ls \sqrt{ \mathfrak{K}_{2N}(u) } (1 +   \norm{\eta}_{L^\infty H^{4N-1/2}} + \norm{\dt \eta}_{L^2 H^{4N-1/2}} ) 
\\ 
\ls \sqrt{ \mathfrak{K}_{2N}(u) } ( 1 + \sqrt{\mathfrak{E}(\eta)} +  P(\sqrt{T \mathcal{F}_0} , \sqrt{ \mathfrak{K}_{2N}(u) } )) \ls P(\sqrt{T \mathcal{F}_0} , \sqrt{ \mathfrak{K}_{2N}(u) } , \sqrt{\mathfrak{E}(\eta)} ) 
\end{multline}
for a polynomial $P$ with $P(0,0,0)=0$.  Iterating this argument for $j=2,\dotsc,2N+1$ then yields the inequalities
\begin{equation}\label{l_tt_11}
  \norm{\dt^j \eta}_{L^2 H^{4N-2j+5/2}} \le P(\sqrt{T \mathcal{F}_0} , \sqrt{ \mathfrak{K}_{2N}(u) } , \sqrt{\mathfrak{E}(\eta)}) 
\end{equation}
for a polynomial with $P(0,0,0)=0$.  We may then square and sum \eqref{l_tt_8}--\eqref{l_tt_11} to find that $\mathfrak{D}(\eta) \le P(T \mathcal{F}_0  , \mathfrak{K}_{2N}(u)  , \mathfrak{E}(\eta))$, but then \eqref{l_tt_7} and the bound $T \le 1$ imply that $\mathfrak{D}(\eta) \le P(T \mathcal{F}_0  , \mathfrak{K}_{2N}(u)  , \mathfrak{E}_0(\eta) )$ for another $P$.  By assumption, $T \mathcal{F}_0   \le \bar{T} \le 1$, and $\mathfrak{K}_{2N}(u), \mathfrak{E}_0(\eta) \le 1$ as well;  hence 
\begin{equation}\label{l_tt_12}
 \mathfrak{D}(\eta) \ls  T \mathcal{F}_0  + \mathfrak{K}_{2N}(u)  + \mathfrak{E}_0(\eta),
\end{equation}
which provides the estimate \eqref{l_tt_03} when combined with Proposition \ref{l_data_norm_comparison}.

\end{proof}

\subsection{An extension result}

In our nonlinear well-posedness argument we will need to be able to take the initial data $\dt^j u(0)$, $j=0,\dotsc,2N$, constructed in Section \ref{l_data_section} and extend it to a function $u$ satisfying $\mathfrak{K}_{2N}(u) \ls \mathfrak{E}_0(u,0)$.  This extension is the content of the following lemma.

\begin{lem}\label{l_sobolev_extension}
Suppose that $\dt^j u(0) \in H^{4N-2j}(\Omega)$ for $j=0,\dotsc,2N$.  Then there exists an extension $u$, achieving the initial data,  so that $\dt^j u\in L^2([0,\infty);H^{4N-2j+1}(\Omega)) \cap L^\infty([0,\infty); H^{4N-2j}(\Omega))$ for $j=0,\dotsc,2N$.  Moreover $\mathfrak{K}_{2N}(u) \ls \mathfrak{E}_0(u,0)$, where in the definition of $\mathfrak{K}_{2N}(u)$ we take $T = \infty$. 
\end{lem}
\begin{proof}
Owing to the usual theory of extensions and restrictions in Sobolev spaces, it suffices to prove the result with $\Omega$ replaced by $\Rn{3}$ in the non-periodic case and $(L_1\mathbb{T}) \times (L_2\mathbb{T}) \times \Rn{}$ in the periodic case.  The proof in the periodic case can be derived from the non-periodic proof by trivially changing some integrals over frequencies to sums.  As such, we present only the proof in $\Rn{3}$.

Let $f_j \in H^{4N-2j}(\Rn{3})$ denote the extension of $\dt^j u(0)\in H^{4N-2j}(\Omega)$.  It suffices to construct $F_j(x,t)$ for $j=0,\dotsc,2N$ so that $\dt^k F_j(x,0) = \delta_{j,k} f_j(x)$ ($\delta_{j,k}$ is the Kronecker delta) and 
\begin{equation}\label{l_sobex_1}
 \ns{\dt^k F_j }_{L^2 H^{4N-2k+1}} + \ns{\dt^k F_j }_{L^\infty H^{4N-2k}} \ls \ns{f_j}_{4n-2j}
\end{equation}
for $k=0,\dotsc,2N$.  Indeed, with such $F_j$ in hand, the sum $F = \sum_{j=0}^{2N} F_j$ is the desired extension. Note that in the norms of \eqref{l_sobex_1} the symbol $L^p H^m$ denotes $L^p([0,\infty);H^m(\Rn{3}))$.  

Let $\varphi_j \in C_c^\infty(\Rn{})$ be such that $\varphi_j^{(k)}(0) = \delta_{j,k}$ for $k=0,\dotsc,2N$ (here $(k)$ is the number of derivatives).  We then define $\hat{F}_j(\xi,t) = \varphi_j(t \br{\xi}^2) \hat{f}_j(\xi) \br{\xi}^{-2j}$, where $\hat{\cdot}$ denotes the Fourier transform and $\br{\xi} = \sqrt{1+\abs{\xi}^2}$.  By construction, $\dt^k \hat{F}_j(\xi,t) = \varphi_j^{(k)}(t \br{\xi}^2) \hat{f}_j(\xi) \br{\xi}^{2(k-j)}$ so that $\dt^k F(\cdot,0) = \delta_{j,k} f_j$.  We estimate
\begin{equation}
\begin{split}
 \ns{\dt^k F_j(\cdot,t)}_{4N-2k} & = \int_{\Rn{3}} \br{\xi}^{2(4N-2k)}  \abs{\varphi_j^{(k)}(t \br{\xi}^2)}^2  \abs{\hat{f}_j(\xi)}^2  \br{\xi}^{2(2k-2j)} d\xi \\
& = \int_{\Rn{3}}   \abs{\varphi_j^{(k)}(t \br{\xi}^2)}^2  \abs{\hat{f}_j(\xi)}^2  \br{\xi}^{2(4N-2j)} d\xi
\le  \ns{\varphi_j^{(k)}}_{L^\infty}  \ns{f_j}_{4N-2j},
\end{split}
\end{equation}
so that $\ns{\dt^k F_j}_{L^\infty H^{4N-2k}} \ls \ns{f_j}_{4N-2j}.$  Similarly, 
\begin{equation}\label{l_sobex_2}
\begin{split}
\ns{\dt^k F_j}_{L^2 H^{4N-2k+1}} &= \int_0^\infty \int_{\Rn{3}} \br{\xi}^{2(4N-2k+1)}  \abs{\varphi_j^{(k)}(t \br{\xi}^2)}^2  \abs{\hat{f}_j(\xi)}^2  \br{\xi}^{2(2k-2j)} d\xi dt\\
& = \int_0^\infty \int_{\Rn{3}}   \abs{\varphi_j^{(k)}(t \br{\xi}^2)}^2  \abs{\hat{f}_j(\xi)}^2  \br{\xi}^{2(4N-2j+1)} d\xi dt  \\
& = \int_{\Rn{3}} \abs{\hat{f}_j(\xi)}^2  \br{\xi}^{2(4N-2j+1)} \left( \int_0^\infty    \abs{\varphi_j^{(k)}(t \br{\xi}^2)}^2 dt\right)  d\xi \\
& = \int_{\Rn{3}} \abs{\hat{f}_j(\xi)}^2  \br{\xi}^{2(4N-2j+1)} \left( \frac{1}{\br{\xi}^2}\int_0^\infty    \abs{\varphi_j^{(k)}(r)}^2 dr\right)  d\xi \\
&= \ns{\varphi_j^{(k)}}_{L^2} \int_{\Rn{3}} \abs{\hat{f}_j(\xi)}^2  \br{\xi}^{2(4N-2j)} d\xi = \ns{\varphi_j^{(k)}}_{L^2}\ns{f_j}_{4N-2j}
\end{split}
\end{equation}
so that $\ns{\dt^k F_j}_{L^2 H^{4N-2k+1}} \ls \ns{f_j}_{4N-2j}.$  Note that in \eqref{l_sobex_2}, we have used Fubini's theorem to switch the order of integration; this is possible since $\varphi$ is compactly supported.  We then have that $F_j$ satisfies the desired properties, completing the proof.

\end{proof}

\section{Local well-posedness of the nonlinear problem}\label{lwp_5}

\subsection{Sequence of approximate solutions}

In order to construct the solution to \eqref{geometric}, we will pass to the limit in a sequence of approximate solutions.  The construction of this sequence is the content of our next result.

\begin{thm}\label{l_iteration}
Assume the initial data are given as in Section \ref{l_data_section} and satisfy the $(2N)^{th}$  compatibility conditions \eqref{l_comp_cond_2N}.   There exist $0 <  \delta<1$ and $0< \bar{T} <1$ so that if $\mathcal{E}_0 \le \delta$, $\mathcal{F}_0 < \infty$, and $0< T \le T_0:= \bar{T} \min\{1,1/\mathcal{F}_0\}$, then there exists an infinite sequence $\{(u^m,p^m,\eta^m)\}_{m=1}^\infty$ with the following three properties.  First, for $m\ge 1$ it holds that
\begin{equation}\label{l_it_01}
 \begin{cases}
  \dt u^{m+1} -\Delta_{\a^m} u^{m+1} + \nab_{\a^m} p^{m+1}  = \dt \bar{\eta}^m \tilde{b} K^m \p_3 u^{m} - u^m \cdot \nab_{\a^m} u^m & \text{in } \Omega \\
 \diverge_{\a^m} u^{m+1} = 0 & \text{in }\Omega \\
 S_{\a^m}(p^{m+1},u^{m+1}) \n^m = \eta^m \n^m & \text{on } \Sigma \\
 u^{m+1} = 0 & \text{on } \Sigma_b
 \end{cases}
\end{equation}
and 
\begin{equation}\label{l_it_02}
  \dt \eta^{m+1} = u^{m+1} \cdot \n^{m+1}  \text{ on } \Sigma,
\end{equation}
where $\a^m, \n^m, K^m$ are given in terms of $\eta^m$.  Second, $(u^m,p^m,\eta^m)$ achieve the initial data for each $m\ge 1$, i.e. $\dt^j u^m(0) = \dt^j u(0)$ and $\dt^j \eta^m(0) = \dt^j \eta(0)$ for $j=0,\dotsc,2N$, while $\dt^j p^m(0) = \dt^j p(0)$ for $j=0,\dotsc,2N-1$.  Third, for each $m\ge 1$ we have the estimates
\begin{equation}\label{l_it_03}
  \mathfrak{K}(\eta^m) + \mathfrak{K}(u^m,p^m)  \le C(\mathcal{E}_0 + T \mathcal{F}_0), \text{ and } 
  \mathcal{F}(\eta^m)  \le C (\mathcal{F}_0 +  \mathcal{E}_0 + T \mathcal{F}_0)
\end{equation}
for a universal constant $C>0$.
\end{thm}

\begin{proof}

We divide the proof into three steps.  First, we construct an initial pair $(u^0,\eta^0)$ that will be used as a starting point for constructing $(u^m,p^m,\eta^m)$ for $m\ge 1$.  Second, we prove that if $(u^m,p^m,\eta^m)$ are known and satisfy certain estimates, then we can construct $(u^{m+1},p^{m+1},\eta^{m+1})$.  Third, we combine the first two steps in an appropriate way to iteratively construct all of the  $(u^m,p^m,\eta^m)$.  Throughout the proof we will need to explicitly enumerate the various constants appearing in estimates where previously we have written $\ls$.  We do so with $C_1,\dotsc,C_{9} > 0$.

Before proceeding to the steps, we define some terms and make some assumptions.  Let $\delta_1>0$ be such that if $\mathfrak{K}(\eta) \le \delta_1$, then the hypotheses of Theorem \ref{l_linear_wp} are satisfied. Similarly, let $\delta_2>0$ be the constant such that if $\mathfrak{E}_0(\eta) \le \delta_2$, then the hypotheses of Theorem \ref{l_transport_theorem} are satisfied.  We assume  that $\delta$ is sufficiently small so that $\mathcal{E}_0 \le \delta$ satisfies the hypotheses of Proposition \ref{l_data_norm_comparison}, and so that (using the estimate \eqref{l_dnc_0}) 
\begin{equation}\label{l_it_1}
 \mathfrak{E}_0(\eta) + \mathfrak{E}_0(u,p) \le C_1 \mathcal{E}_0 \le C_1 \delta \le  \min\{1,\delta_2\}.
\end{equation} 
 This allows us to use \eqref{l_fie_01} of Lemma \ref{l_full_iteration_estimates} with $j=2N-1$ to bound 
\begin{equation}\label{l_it_2}
 \mathfrak{F}_0(F^1(u,\eta),F^2(u,\eta)) \le C_2 \mathcal{E}_0.
\end{equation}

Step 1 -- Seeding the sequence

We begin by extending the initial data $\dt^j u(0) \in H^{4N-2j}(\Omega)$ to a time-dependent function $u^0$ so that $\dt^j u^0(0) = \dt^j u(0)$.  We do so by applying Lemma \ref{l_sobolev_extension}.  Although this produces a $u^0$ defined on the time interval $[0,\infty)$, we may restrict to $[0,T]$ without increasing any of the space-time norms in  $\mathfrak{K}_{2N}(u^0)$.  We may combine the estimate of $\mathfrak{K}_{2N}(u^0)$ provided by Lemma \ref{l_sobolev_extension} with \eqref{l_it_1} to bound
\begin{equation}\label{l_it_3}
\mathfrak{K}_{2N}(u^0) \le C_3 \mathcal{E}_0. 
\end{equation}

With $u^0$ in hand, we define $\eta^0$ as the solution to \eqref{l_transport_equation} with $u^0$ replacing $u$.  To do so, we apply Theorem \ref{l_transport_theorem}, the hypotheses of which are satisfied  by virtue of \eqref{l_it_1} and \eqref{l_it_3} if we further restrict to $C_3 \delta \le 1$.  Restricting $\bar{T}$ as in the theorem, we find our solution $\eta^0$, which satisfies $\dt^j \eta^0(0) = \dt^j \eta(0)$ as well as the estimates
\begin{equation}\label{l_it_4}
 \begin{split}
  \mathcal{F}(\eta^0) &\le C_4 (\mathcal{F}_0 + T \mathfrak{K}_{2N}(u^0)) \\
  \mathfrak{E}(\eta^0) & \le C_5 (\mathcal{E}_0 + T \mathfrak{K}_{2N}(u^0)) \\
  \mathfrak{D}(\eta^0) & \le C_6 (\mathcal{E}_0 + T \mathcal{F}_0 + \mathfrak{K}_{2N}(u^0)).
 \end{split}
\end{equation}

Step 2 -- The iteration argument

We claim that there exist $\gamma_1,\gamma_2,\gamma_3, \gamma_4 >0$ and $0 < \tilde{\delta}$, $\tilde{T} <1$ (both depending on the $\gamma_i$) so that if $\delta \le \tilde{\delta}$ and $\bar{T} \le \tilde{T}$, then the following property is satisfied.  If $(u^m,\eta^m)$ are known and satisfy the estimates 
\begin{equation}\label{l_it_20}
 \begin{split}
  \mathfrak{E}(\eta^m) & \le \gamma_1(\mathcal{E}_0 + T \mathcal{F}_0), \; \mathfrak{D}(\eta^m) \le \gamma_2 (\mathcal{E}_0 + T \mathcal{F}_0), \\
 \mathfrak{K}_{2N}(u^m) & \le \gamma_3 (\mathcal{E}_0 + T \mathcal{F}_0), \; \mathcal{F}(\eta^m) \le  C_4 \mathcal{F}_0+ \gamma_4 (\mathcal{E}_0 + T \mathcal{F}_0),
 \end{split}
\end{equation}
then there exists a unique triple $(u^{m+1},p^{m+1},\eta^{m+1})$ that achieves the initial data, satisfies \eqref{l_it_01} and \eqref{l_it_02}, and obeys the estimates 
\begin{equation}\label{l_it_21}
 \begin{split}
  \mathfrak{E}(\eta^{m+1}) & \le \gamma_1(\mathcal{E}_0 + T \mathcal{F}_0), \; \mathfrak{D}(\eta^{m+1}) \le \gamma_2 (\mathcal{E}_0 + T \mathcal{F}_0), \\
 \mathfrak{K}(u^{m+1},p^{m+1}) & \le \gamma_3 (\mathcal{E}_0 + T \mathcal{F}_0), \; \mathcal{F}(\eta^{m+1}) \le  C_4 \mathcal{F}_0+ \gamma_4 (\mathcal{E}_0 + T \mathcal{F}_0).
 \end{split}
\end{equation}
To prove the claim, we will first use $\eta^m$ to solve for $(u^{m+1},p^{m+1})$, and then we will use the resulting $u^{m+1}$ to solve for $\eta^{m+1}$.  Along the way, we will restrict the size of $\tilde{\delta}$ and $\tilde{T}$ in terms of $\gamma_i$, $i=1,2,3,4$. We will define the $\gamma_i$ in terms of the $C_i$, so the $\tilde{\delta}$ and $\tilde{T}$ can be thought of as universal constants.  Note that the estimates of \eqref{l_it_21} are stronger than those of \eqref{l_it_20} since $\mathfrak{K}_{2N}(u^{m+1}) \le \mathfrak{K}(u^{m+1},p^{m+1})$.   This asymmetry is useful to us since in Step 1 we have not bothered to construct $p^0$, so only $(u^0,\eta^0)$ are available to begin the iterative construction of  $\{(u^m,p^m,\eta^m)\}_{m=1}^\infty$.

We assume initially that 
\begin{equation}
 \tilde{\delta}, \tilde{T} \le \frac{1}{2}  \min\left \{ \frac{\min\{1,\delta_1\}}{(\gamma_1 + \gamma_2)}, \frac{1}{\gamma_3} \right\},
\end{equation}
so that \eqref{l_it_20} implies that $\mathfrak{K}_{2N}(u^m) \le 1$ and 
\begin{equation}\label{l_it_5}
 \mathfrak{K}(\eta^m) = \mathfrak{E}(\eta^m) + \mathfrak{D}(\eta^m)
\le (\gamma_1 + \gamma_2) (\mathcal{E}_0 + T_0 \mathcal{F}_0)  \le \min\{\delta_1,1\},
\end{equation}
the latter of which allows us to use Theorem \ref{l_linear_wp} to produce a unique pair $(u^{m+1},p^{m+1})$ that achieves the desired initial data and satisfies \eqref{l_it_01}.  Moreover, from \eqref{l_lwp_02} and \eqref{l_it_1}--\eqref{l_it_2}, we have the estimate
\begin{multline}\label{l_it_6}
 \mathfrak{K}(u^{m+1},p^{m+1}) \le C_7 (1 + \mathcal{E}_0 + \mathfrak{K}(\eta^m)) \exp\left( C_8(1+ \mathfrak{E}(\eta^m)) T\right) \times \\
\left[ (1+C_2)\mathcal{E}_0 + \mathfrak{F}(F^1(u^m,\eta^m),F^3(u^m,\eta^m)) \right].
\end{multline}
Assume that $2 \tilde{T} C_8 \le \log{2}$; then
\begin{equation}\label{l_it_7}
 C_7 (1 + \mathcal{E}_0 + \mathfrak{K}(\eta^m)) \exp\left( C_8(1+ \mathfrak{E}(\eta^m)) T\right) \le 3 C_7 \exp( 2 C_8 \tilde{T} )  \le 6 C_7. 
\end{equation}
On the other hand, we can use our bounds on $\eta^m, u^m$ in Lemma \ref{l_Ffrak_bound} to see  that
\begin{equation}\label{l_it_8}
 \mathfrak{F}(F^1(u^m,\eta^m),F^3(u^m,\eta^m)) \le C_9 \left[3 \mathfrak{E}(\eta^m) + 2 \mathfrak{K}(\eta^m) \mathfrak{K}_{2N}(u^m)  +(\mathfrak{K}_{2N}(u^m))^2 \right].
\end{equation}
Combining \eqref{l_it_6}--\eqref{l_it_8} with \eqref{l_it_20} then shows that
\begin{multline}\label{l_it_9}
 \mathfrak{K}(u^{m+1},p^{m+1}) \le   6 C_7 \left[ (1+C_2) \mathcal{E}_0 + 3 C_9 \gamma_1 (\mathcal{E}_0 + T \mathcal{F}_0) \right. \\
\left.
+ 2 C_9 \gamma_3(\gamma_1+\gamma_2)(\mathcal{E}_0 + T \mathcal{F}_0)^2 
+ C_9 \gamma_3^2 (\mathcal{E}_0 + T \mathcal{F}_0)^2
\right].
\end{multline}

We have now enumerated all of the constants $C_i$, $i=1,\dotsc,9$, that we need to define the $\gamma_i$, $i=1,\dotsc,4$.  We choose the values of the $\gamma_i$ according to
\begin{equation}\label{l_it_23}
\begin{split}
 \gamma_1 &:= 2C_5, \gamma_3 := 6 C_7(3 + C_2 + 3 C_9 \gamma_1  ) + C_3,\\
 \gamma_4 &:= C_4,  \gamma_2 := C_6(1+\gamma_3).
\end{split}
\end{equation}
Notice that even though we have used $\gamma_1$ to define $\gamma_3$ and $\gamma_3$ to define $\gamma_2$, all of the $\gamma_i$ are determined in terms of the constants $C_i$.  

Now we will use the choice of the $\gamma_i$ in \eqref{l_it_23} to derive the $\mathfrak{K}(u^{m+1},p^{m+1})$ estimate of \eqref{l_it_21} from \eqref{l_it_9}.  To do this, we further restrict 
\begin{equation}
 \tilde{\delta},\tilde{T} \le \frac{1}{2} \min\left\{ \frac{1}{2 C_9 \gamma_3(\gamma_1+\gamma_2)}, \frac{1}{C_9 \gamma_3^2}   \right\}.
\end{equation}
Then since $\mathcal{E}_0 + T \mathcal{F}_0 \le \tilde{\delta} + \tilde{T}$, we may use  \eqref{l_it_9} to bound
\begin{equation}\label{l_it_22}
 \mathfrak{K}(u^{m+1},p^{m+1}) \le 6 C_7(3 + C_2 + 3 C_9 \gamma_1  )(\mathcal{E}_0 + T \mathcal{F}_0) \le  \gamma_3 (\mathcal{E}_0 + T \mathcal{F}_0).
\end{equation}

Now we construct $\eta^{m+1}$.  Recall that  $\tilde{\delta}, \tilde{T} \le 1/(2\gamma_3)$; this and \eqref{l_it_22} then yield the bound $\mathfrak{K}_{2N}(u^{m+1}) \le 1$.   This estimate then allows us to apply Theorem \ref{l_transport_theorem} to find $\eta^{m+1}$ that solves  \eqref{l_it_02} and achieves the initial data.  Estimates \eqref{l_tt_01}--\eqref{l_tt_03}  of the theorem, together with \eqref{l_it_22} and the bound $T_0 \gamma_3 \le \tilde{T} \gamma_3 \le 1$, imply that
\begin{equation}\label{l_it_13}
 \begin{split}
  \mathcal{F}(\eta^{m+1}) &\le C_4 (\mathcal{F}_0 + T_0 \mathfrak{K}_{2N}(u^{m+1}))  \le C_4 \mathcal{F}_0 + C_4 (\mathcal{E}_0 + T \mathcal{F}_0 )   \\
  \mathfrak{E}(\eta^{m+1}) & \le C_5 (\mathcal{E}_0 + T_0 \mathfrak{K}_{2N}(u^{m+1})) \le 2 C_5 (\mathcal{E}_0 + T \mathcal{F}_0   ) \\
  \mathfrak{D}(\eta^{m+1}) & \le C_6 (\mathcal{E}_0 + T \mathcal{F}_0 + \mathfrak{K}_{2N}(u^{m+1})) \le C_6(1+\gamma_3) (\mathcal{E}_0 + T \mathcal{F}_0 ).
 \end{split} 
\end{equation}
Using the definitions of the $\gamma_i$ given in \eqref{l_it_23}, we see from \eqref{l_it_13} that the $\eta^{m+1}$ estimates of \eqref{l_it_21} hold.  Then, owing to \eqref{l_it_22}, all of the estimates in \eqref{l_it_21} hold, which completes the proof of the claim.

Step 3 -- Construction of the full sequence

We assume that $\gamma_1,\gamma_2,\gamma_3,\gamma_4$ are given by \eqref{l_it_23}  and that $\tilde{\delta}$ and $\tilde{T}$ are as small as in Step 2.  We assume that $\delta \le \tilde{\delta}$ and $\bar{T} \le \tilde{T}$ in addition to the other restrictions on their size made in Step 1 and before.   Returning to  \eqref{l_it_3}, note that $C_3 \le \gamma_3$, which means that  $\mathfrak{K}_{2N}(u^0) \le \gamma_3 (\mathcal{E}_0 + T \mathcal{F}_0)$.  We can also combine \eqref{l_it_3} and \eqref{l_it_4} and further restrict  $\bar{T}\le 1/C_3$ to deduce that
\begin{equation}
 \begin{split}
  \mathcal{F}(\eta^0) &  \le C_4 \mathcal{F}_0 + T_0 C_3 C_4 \mathcal{E}_0  \le C_4 \mathcal{F}_0 + \gamma_4 (\mathcal{E}_0 + T \mathcal{F}_0) \\
  \mathfrak{E}(\eta^0) & \le C_5 (1+ T_0 C_3)\mathcal{E}_0  \le 2 C_5 \mathcal{E}_0 \le \gamma_1 (\mathcal{E}_0 + T \mathcal{F}_0) \\
  \mathfrak{D}(\eta^0) & \le C_6 (\mathcal{E}_0 + T \mathcal{F}_0 +  C_3 \mathcal{E}_0) \le C_6(1+C_3) (\mathcal{E}_0 + T \mathcal{F}_0) \le \gamma_2(\mathcal{E}_0 + T \mathcal{F}_0).
 \end{split}
\end{equation}
Note that in the last inequality we have used the fact that $C_3 \le \gamma_3$ to bound $C_6(1+C_3) \le C_6(1+\gamma_3) =  \gamma_2$.  We are then free to use the pair $(u^0,\eta^0)$ as the starting point in Step 2, which allows us to construct $(u^1,p^1,\eta^1)$ satisfying the desired PDE and initial conditions, along with the estimates
\begin{equation}
 \begin{split}
  \mathfrak{E}(\eta^{1}) & \le \gamma_1(\mathcal{E}_0 + T \mathcal{F}_0), \; \mathfrak{D}(\eta^{1}) \le \gamma_2 (\mathcal{E}_0 + T \mathcal{F}_0), \\
 \mathfrak{K}(u^{1},p^{1}) & \le \gamma_3 (\mathcal{E}_0 + T \mathcal{F}_0), \; \mathcal{F}(\eta^{1}) \le  C_4 \mathcal{F}_0+ \gamma_4 (\mathcal{E}_0 + T \mathcal{F}_0).
 \end{split}
\end{equation}
We then iterate from $m=1,\dotsc,\infty$, using $(u^m,\eta^m)$ and Step 2 to produce the next element of the sequence, $(u^{m+1},p^{m+1},\eta^{m+1})$, which satisfies \eqref{l_it_21}.  All of the conclusions of the theorem follow.

\end{proof}

\subsection{Contraction}

While the estimates \eqref{l_it_03} of Theorem \ref{l_iteration} will allow us to extract weak limits from the sequence $\{(u^m,p^m,\eta^m)\}_{m=1}^\infty$, weak convergence of a subsequence is not enough to allow us to pass to the limit in \eqref{l_it_01}--\eqref{l_it_02} in order to  produce the desired solution to \eqref{geometric}.  We are thus led to study the strong convergence of the sequence, and in particular to consider its contraction in some norm.

We now define the norms in which we will show the sequence contracts.  For $T>0$ we define 
\begin{equation}\label{l_MNFrak_def}
\begin{split}
 \mathfrak{N}(v,q;T) & = \ns{v}_{L^\infty H^2} + \ns{v}_{L^2 H^3} + \ns{\dt v}_{L^\infty H^0} + \ns{\dt v}_{L^2 H^1} + \ns{q}_{L^\infty H^1} + \ns{q}_{L^2 H^2}, \\
 \mathfrak{M}(\zeta;T)& = \ns{\zeta}_{L^\infty H^{5/2}} + \ns{\dt \zeta}_{L^\infty H^{3/2}} + \ns{\dt^2 \zeta}_{L^2 H^{1/2}},
\end{split}
\end{equation}
where we write $L^pH^k$ for $L^p([0,T];H^k(\Omega))$ in $\mathfrak{N}$ and $L^p([0,T];H^k(\Sigma))$ in $\mathfrak{M}$.  

The next result provides a comparison of $\mathfrak{N}$ for pairs of solutions to problems of the form \eqref{l_it_01}--\eqref{l_it_02}.  We will use it later in Theorem \ref{l_nwp} to show that the sequence of approximate solutions contracts, but we will also use it to prove the uniqueness of solutions to \eqref{geometric}.  In order to avoid confusion with the sequence $\{(u^m,p^m,\eta^m)\}$, we refer to velocities as $v^j, w^j$, pressures as $q^j$, and surface functions as $\zeta^j$ for $j=1,2$.

\begin{thm}\label{l_contraction}
Let $w^1,w^2,$ $v^1, v^2$, $q^1,q^2,$ and $\zeta^1, \zeta^2$ satisfy 
\begin{equation}\label{l_con_01}
 \sup \left\{ \mathfrak{E}(\zeta^1), \mathfrak{E}(\zeta^2), \mathfrak{E}(v^1,q^1),\mathfrak{E}(v^2,q^2),\mathfrak{E}(w^1,0), \mathfrak{E}(w^2,0) \right\} \le \ep,
\end{equation}
where the temporal $L^\infty$ norms in $\mathfrak{E}$ are computed over the interval $[0,T]$ with $0< T$.  Suppose that for $j=1,2$, 
\begin{equation}\label{l_con_02}
 \begin{cases}
  \dt v^{j} -\Delta_{\a^j} v^{j} + \nab_{\a^j} q^{j}  = \dt \bar{\zeta}^{j} \tilde{b} K^{j} \p_3 w^{j} - w^j \cdot \nab_{\a^j} w^j & \text{in } \Omega \\
 \diverge_{\a^j} v^{j} = 0 & \text{in }\Omega \\
 S_{\a^j}(q^{j},v^{j}) \n^j = \zeta^j \n^j & \text{on } \Sigma \\
 v^{j} = 0 & \text{on } \Sigma_b,\\
  \dt \zeta^{j} = w^{j} \cdot \n^{j}  & \text{on } \Sigma,
 \end{cases}
\end{equation}
where $\a^j, K^j, \n^j$ are determined by $\zeta^j$ as usual.  Further suppose that $\dt^k v^1(0) = \dt^k v^2(0)$  for $k=0,1$,   $\zeta^1(0) =  \zeta^2(0)$, and $q^1(0) =  q^2(0)$. 

Then there exist $\ep_1 >0$, $T_1 >0$ so that if $\ep \le \ep_1$ and $0< T \le T_1$, then 
\begin{equation}\label{l_con_04}
 \mathfrak{N}(v^1 -v^2,q^1-q^2; T) \le \frac{1}{2} \mathfrak{N}(w^1 -w^2,0; T) 
\end{equation}
and
\begin{equation}\label{l_con_05}
 \mathfrak{M}(\zeta^1 -\zeta^2; T) \le 2 \mathfrak{N}(w^1 -w^2,0; T). 
\end{equation}
\end{thm}
\begin{proof}

The proof proceeds through six steps.  First, we define  $v = v^1 - v^2$, $w = w^1 - w^2$, $q = q^1 - q^2$, and derive the PDEs satisfied by $v,q$.  We also identify the energy evolution for some norms of $\dt v, \dt q$.  Second, we bound various forcing terms that appear in the energy evolution and on the right side of the PDEs for $v,q$.  Third, we prove some bounds for $\dt v, \dt q$, using the energy evolution equation.  Fourth, we use elliptic estimates to bound norms of $v, q$.  Fifth, we derive estimates for $\zeta^1 - \zeta^2$ in terms of $w$.  Sixth, we close the estimate to derive the contraction estimates \eqref{l_con_04}, \eqref{l_con_05}.

Step 1 -- PDEs and energy evolution for differences

We now derive the PDE satisfied by $v,q$, which are defined above.  We subtract the equations in \eqref{l_con_02} with $j=2$ from the same equations with $j=1$.  With the help of some simple algebra, we can write the resulting equations in terms of $v, q$:
\begin{equation}\label{l_con_1}
 \begin{cases}
 \dt v -\Delta_{\a^1} v + \nab_{\a^1} q  = \diverge_{\a^1}( \sg_{(\a^1-\a^2)}v^2  ) + H^1 & \text{in } \Omega \\
 \diverge_{\a^1} v = H^2 & \text{in }\Omega \\
 S_{\a^1}(q,v) \n^1 = \sg_{(\a^1 - \a^2)} v^2 \n^1 + H^3 & \text{on } \Sigma \\
 v = 0 & \text{on } \Sigma_b \\
 v(t=0) = 0,
 \end{cases}
\end{equation}
where $H^1,H^2,H^3$ are defined by
\begin{multline}\label{l_con_2}
 H^1 = \diverge_{(\a^1-\a^2)}(\sg_{\a^2} v^2) - (\a^1 - \a^2) \nab q^2 \\
+ \dt \bar{\zeta}^1 \tilde{b} K^1 (\p_3 w^1 - \p_3 w^2) 
+ (\dt \bar{\zeta}^1 - \dt \bar{\zeta}^2)\tilde{b}K^1 \p_3 w^2  
+ \dt \bar{\zeta}^1 \tilde{b} (K^1 - K^2) \p_3 w^2  \\
- (w^1 - w^2) \cdot \nab_{\a^1} w^1 
- w^2\cdot \nab_{\a^1}(w^1-w^2) 
- w^2 \cdot \nab_{(\a^1-\a^2)} w^2,
\end{multline}
\begin{equation}\label{l_con_3}
 H^2 = -\diverge_{(\a^1-\a^2)} v^2,
\end{equation}
\begin{multline}\label{l_con_4}
 H^3 = -q^2 (\n^1 - \n^2) + \sg_{\a^1} v^2(\n^1 - \n^2) - \sg_{(\a^1-\a^2)}v^2 (\n^1 - \n^2) \\
+ (\zeta^1 - \zeta^2)\n^1 + \zeta^2 (\n^1 - \n^2).
\end{multline}

The solutions are sufficiently regular for us to differentiate \eqref{l_con_1} in time, which results in the equations
\begin{equation}\label{l_con_5}
 \begin{cases}
 \dt (\dt v) -\Delta_{\a^1} (\dt v) + \nab_{\a^1} (\dt q)  = \diverge_{\a^1}( \sg_{(\dt \a^1-\dt \a^2)}v^2  ) + \tilde{H}^1 & \text{in } \Omega \\
 \diverge_{\a^1} \dt v = \tilde{H}^2 & \text{in }\Omega \\
 S_{\a^1}(\dt q,\dt v) \n^1 = \sg_{(\dt \a^1 - \dt \a^2)} v^2 \n^1 + \tilde{H}^3 & \text{on } \Sigma \\
 \dt v = 0 & \text{on } \Sigma_b \\
 \dt v(t=0) = 0,
 \end{cases}
\end{equation}
where $\tilde{H}^1$, $\tilde{H}^2$, and $\tilde{H}^3$ are given by
\begin{multline}\label{l_con_6}
\tilde{H}^1 = \dt H^1 + \diverge_{\dt \a^1}( \sg_{(\a^1-\a^2)} v^2) + \diverge_{\a^1}(\sg_{(\a^1-\a^2)} \dt v^2) \\
+ \diverge_{\dt \a^1}(\sg_{\a^1} v) + \diverge_{\a^1}(\sg_{\dt \a^1} v) - \nab_{\dt \a^1} q,
\end{multline}
\begin{equation}\label{l_con_7}
 \tilde{H}^2 = \dt H^2 - \diverge_{\dt \a^1} v,
\end{equation}
\begin{equation}\label{l_con_8}
 \tilde{H}^3 = \dt H^3   + \sg_{(\a^1 - \a^2)} \dt v^2 \n^1 + \sg_{(\a^1 - \a^2)} v^2 \dt \n^1 - S_{\a^1}(q,v) \dt \n^1 + \sg_{\dt \a^1} v \n^1.
\end{equation}

Now we multiply \eqref{l_con_5} by $J^1 \dt v$, integrate over $\Omega$, and integrate by parts as in the proof of Theorem \ref{l_strong_solution} to deduce the evolution equation
\begin{multline}\label{l_con_9}
 \dt \int_\Omega \frac{\abs{\dt v}^2}{2}J^1 + \hal \int_\Omega \abs{\sg_{\a^1} \dt v}^2 J^1 = \int_\Omega \frac{\abs{\dt v}^2}{2} (\dt J^1 K^1) J^1  + \int_\Omega J^1 \dt q \tilde{H}^2   \\
+ \int_\Omega J^1   \left( \diverge_{\a^1}( \sg_{(\dt \a^1-\dt \a^2)}v^2  ) + \tilde{H}^1 \right)\cdot \dt v      \\
- \int_\Sigma \left(\sg_{(\dt \a^1 - \dt \a^2)} v^2 \n^1   + \tilde{H}^3\right)  \cdot \dt v.
\end{multline}
Another integration by parts reveals that
\begin{multline}\label{l_con_10}
 \int_\Omega J^1   \diverge_{\a^1}( \sg_{(\dt \a^1 - \dt \a^2)}v^2  )   \cdot \dt v = -\hal \int_\Omega  J^1 \sg_{(\dt \a^1 - \dt \a^2)}v^2 : \sg_{\a^1} \dt v \\
+ \int_\Sigma \sg_{(\dt \a^1 - \dt \a^2)} v^2 \n^1 \cdot \dt v.
\end{multline}
We then employ \eqref{l_con_10} to rewrite \eqref{l_con_9}, and then we integrate in time from $0$ to $t < T$; since $\dt v(t=0) =0$, we arrive at the equation
\begin{multline}\label{l_con_11}
 \int_\Omega \frac{\abs{\dt v}^2}{2}J^1(t) + \hal \int_0^t \int_\Omega \abs{\sg_{\a^1} \dt v}^2 J^1 = \int_0^t \int_\Omega \frac{\abs{\dt v}^2}{2} (\dt J^1 K^1) J^1  \\
+ \int_0^t \int_\Omega J^1 (  \tilde{H}^1 \cdot \dt v +   \tilde{H}^2 \dt q)   
-\hal \int_0^t \int_\Omega  J^1 \sg_{(\dt \a^1-\dt \a^2)}v^2 : \sg_{\a^1} \dt v 
- \int_0^t \int_\Sigma   \tilde{H}^3 \cdot \dt v.
\end{multline}

Step 2 -- Estimates of the forcing terms

In order for the equation \eqref{l_con_11} to be useful, we must be able to estimate the terms that appear on its right.  To this end, we now derive estimates for $\tilde{H}^1, \tilde{H}^2, \dt \tilde{H}^2$ in $H^0(\Omega)$ and   $\tilde{H}^3$ in $H^{-1/2}(\Sigma)$.  We claim that the following estimates hold; in each we have written $P(\cdot)$ for a polynomial so that $P(0)=0$.
\begin{multline}\label{l_con_12}
 \norm{ \tilde{H}^1 }_{0} \ls P(\sqrt{\ep}) \Big[ \norm{\zeta^1 - \zeta^2}_{3/2} + \norm{\dt \zeta^1 - \dt \zeta^2}_{1/2} + \norm{\dt^2 \zeta^1 - \dt^2 \zeta^2}_{0}    \\
+ \norm{w^1 - w^2}_{1} + \norm{\dt w^1 - \dt w^2}_{1} + \norm{v}_{2} + \norm{q}_{1}  \Big]
\end{multline}
\begin{equation}\label{l_con_13}
 \norm{ \tilde{H}^2 }_{0} \ls  P(\sqrt{\ep}) \left[ \norm{\zeta^1 - \zeta^2}_{1/2} +  \norm{\dt \zeta^1 - \dt \zeta^2}_{1/2}  + \norm{v}_{1} \right]
\end{equation}
\begin{multline}\label{l_con_14}
 \norm{\dt  \tilde{H}^2 }_{0} \ls  P(\sqrt{\ep}) \Big[ \norm{\zeta^1 - \zeta^2}_{1/2} +  \norm{\dt \zeta^1 - \dt \zeta^2}_{1/2} +\norm{\dt^2 \zeta^1 - \dt^2 \zeta^2}_{1/2} \\
 + \norm{v}_{1}  + \norm{\dt v}_{1} \Big]
\end{multline}
\begin{multline}\label{l_con_15}
 \norm{\tilde{H}^3}_{-1/2} \ls P(\sqrt{\ep})  \left[ \norm{\zeta^1 - \zeta^2 }_{1/2} + \norm{\dt \zeta^1 - \dt \zeta^2}_{1/2}   + \norm{v}_{2} + \norm{q}_{1} \right] \\
+ \norm{\dt \zeta^1 - \dt \zeta^2}_{-1/2}
\end{multline}
According to the definitions \eqref{l_con_6}--\eqref{l_con_8}, all of the summands in $\tilde{H}^1,$ $\tilde{H}^2$, $\dt \tilde{H}^2$ are quadratic, of the form $X \times Y$, where $Y$ is one of $v$, $q$, $\dt^j \zeta^1 - \dt^j \zeta^2$ for $j=0,1,2$, or $\dt^j w^1 - \dt^j w^2$ for $j=0,1$.  The bounds \eqref{l_con_12}--\eqref{l_con_14} may be established by estimating the products $X\times Y$  with Lemmas \ref{i_sobolev_product_1}, \ref{i_poisson_grad_bound}, \ref{p_poisson},  \ref{p_poisson_2}, and \ref{i_poisson_interp}   and the usual Sobolev and trace embeddings;  the appearance of the terms $P(\sqrt{\ep})$ is due to the $X$ terms, whose appropriate Sobolev norm may be bounded above by a polynomial in 
\begin{equation}\label{l_con_15_2}
\sqrt{\sup \left\{ \mathfrak{E}(\zeta^1), \mathfrak{E}(\zeta^2), \mathfrak{E}(v^1,q^1),\mathfrak{E}(v^2,q^2),\mathfrak{E}(w^1,0), \mathfrak{E}(w^2,0) \right\} }\le  \sqrt{\ep}.
\end{equation}
The estimate \eqref{l_con_15} follows similarly by using \eqref{i_s_p_03} of Lemma \ref{i_sobolev_product_1}, except that $\tilde{H}^3$ has a single term, namely $(\dt \zeta^1- \dt \zeta^2) e_3$, that is not quadratic and that causes the last term on the right side of \eqref{l_con_15} to not be multiplied by $P(\sqrt{\ep})$.  The same sort of argument also allows us to deduce the bound
\begin{equation}\label{l_con_15_3}
 \norm{ \sg_{(\dt \a^1-\dt \a^2)}v^2    }_{0} \ls  P(\sqrt{\ep})   \left[  \norm{ \zeta^1 -  \zeta^2}_{1/2}+  \norm{\dt \zeta^1 - \dt \zeta^2}_{1/2} \right].
\end{equation}

We will eventually employ an elliptic estimate with \eqref{l_con_1}, so we will also need estimates of $H^1$, $H^2$, $H^3$ and the two other terms appearing on the right side of \eqref{l_con_1}.  The following estimates hold for $r=0,1$ (again $P$ denotes a polynomial with $P(0)=0$):
\begin{equation}\label{l_con_16}
 \norm{H^1}_{r} \ls P(\sqrt{\ep}) \left[ \norm{\zeta^1 - \zeta^2}_{r+1/2} + \norm{\dt \zeta^1 - \dt \zeta^2}_{r-1/2} + \norm{w^1 - w^2}_{r+1} \right]
\end{equation}
\begin{equation}\label{l_con_17}
 \norm{H^2}_{r+1} \ls P(\sqrt{\ep}) \norm{\zeta^1 - \zeta^2}_{r+3/2}
\end{equation}
\begin{equation}\label{l_con_18}
 \norm{H^3}_{r+1/2} \ls P(\sqrt{\ep})  \norm{\zeta^1 - \zeta^2}_{r+3/2} + \norm{ \zeta^1 -  \zeta^2}_{r+1/2} 
\end{equation}
\begin{equation}\label{l_con_19}
 \norm{\diverge_{\a^1}( \sg_{(\a^1-\a^2)}v^2  ) }_{r} \ls P(\sqrt{\ep})   \norm{\zeta^1 - \zeta^2}_{r+1/2} 
\end{equation}
\begin{equation}\label{l_con_20}
 \norm{\sg_{(\a^1 - \a^2)} v^2 \n^1}_{r+1/2} \ls P(\sqrt{\ep})  \norm{\zeta^1 - \zeta^2}_{r+3/2}.
\end{equation}
The proof of \eqref{l_con_16}--\eqref{l_con_20} may be carried out in the same manner we used above to prove \eqref{l_con_12}--\eqref{l_con_15}.

Step 3 -- Estimates of $\dt v$ from \eqref{l_con_11}

Now we employ the estimates of the forcing terms from the previous step in \eqref{l_con_11} in order to deduce estimates for $\dt v$.  First we note that, owing to \eqref{l_con_15_2} and Sobolev embeddings, we can bound 
\begin{equation}\label{l_con_21}
\pnorm{J^1}{\infty} + \pnorm{K^1}{\infty}   \ls 1+ P(\sqrt{\ep}) \text{ and } \pnorm{\dt J^1}{\infty}  \ls P(\sqrt{\ep})
\end{equation}
for $P$ a polynomial with $P(0)=0$.   

Because of the time derivative on $q$, the most delicate term in \eqref{l_con_11} is the product $J^1 \tilde{H}^2 \dt q$.  To handle it we integrate by parts in time and use the fact that $q(0) =0$ to see that
\begin{multline}
\int_0^t \int_\Omega J^1 \tilde{H}^2 \dt q = \int_0^t \left[ \dt \int_\Omega J^1 q \tilde{H}^2 - \int_\Omega \dt J^1 q \tilde{H}^2 + J^1 q \dt \tilde{H}^2   \right] \\
= \int_\Omega J^1 q \tilde{H}^2 (t) - J^1 q \tilde{H}^2 (0) - \int_0^t \int_\Omega \dt J^1 q \tilde{H}^2 + J^1 q \dt \tilde{H}^2 \\
= \int_\Omega J^1 q \tilde{H}^2 (t)  - \int_0^t \int_\Omega \dt J^1 q \tilde{H}^2 + J^1 q \dt \tilde{H}^2.
\end{multline}
This, \eqref{l_con_21}, and the estimates \eqref{l_con_13} and \eqref{l_con_14} then imply that
\begin{multline}\label{l_con_22}
 \int_0^t \int_\Omega J^1 \tilde{H}^2 \dt q \ls P(\sqrt{\ep}) \norm{q}_{L^\infty H^0} \left[ \sum_{j=0}^1 \norm{\dt^j \zeta^1 - \dt^j \zeta^2}_{L^\infty H^{1/2}} + \norm{v}_{L^\infty H^{1}}\right] \\
+ P(\sqrt{\ep}) \int_0^t \norm{q}_{0} \left[ \sum_{j=0}^2 \norm{\dt^j \zeta^1 - \dt^j \zeta^2}_{1/2} + \norm{v}_{1} + \norm{\dt v}_{1}\right],
\end{multline}
where the $L^\infty$ norms are computed over the temporal interval $[0,T]$.

The other terms on the right of \eqref{l_con_11} are not so delicate and may be estimated directly with \eqref{l_con_12}, \eqref{l_con_15}, and \eqref{l_con_15_3}.  Indeed, these estimates together with trace theory and the Poincar\'e inequality imply that
\begin{multline}\label{l_con_23} 
\int_0^t \int_\Omega J^1   \tilde{H}^1 \cdot \dt v -\hal  J^1 \sg_{(\dt \a^1-\dt \a^2)}v^2 : \sg_{\a^1} \dt v 
- \int_0^t \int_\Sigma   \tilde{H}^3 \cdot \dt v
\\ \le 
\int_0^t \pnorm{J^1}{\infty} \norm{\tilde{H}^1}_{0} \norm{\dt v}_{0} + \hal \pnorm{J^1}{\infty} \norm{\sg_{(\dt \a^1-\dt \a^2)}v^2}_{0} \norm{\sg_{\a^1} \dt v  }_{0}  \\
+ \int_0^t \norm{\tilde{H}^3}_{-1/2} \snormspace{\dt v}{1/2}{\Sigma}
 \ls  \int_0^t \norm{\dt v}_{1} \left[ P(\sqrt{\ep}) \sqrt{\mathcal{Z}} + \norm{\dt \zeta^1 - \dt\zeta^2 }_{-1/2}\right],
\end{multline}
where we have written 
\begin{multline}\label{l_con_24}
 \mathcal{Z} := \ns{\zeta^1 - \zeta^2}_{3/2} + \ns{\dt \zeta^1 - \dt \zeta^2}_{1/2} + \ns{\dt^2 \zeta^1 - \dt^2 \zeta^2}_{1/2}    \\
+ \ns{w^1 - w^2}_{1} + \ns{\dt w^1 - \dt w^2}_{1} + \ns{v}_{2} + \ns{q}_{1}.
\end{multline}
Also, we may use \eqref{l_con_15_2} to bound
\begin{equation}\label{l_con_25}
 \int_0^t \int_\Omega \frac{\abs{\dt v}^2}{2} (\dt J^1 K^1) J^1 \le C\sqrt{\ep} \int_0^t \int_\Omega \frac{\abs{\dt v}^2}{2}  J^1 
\end{equation}
for some constant $C>0$.

We now combine the estimates \eqref{l_con_22}, \eqref{l_con_23}, and \eqref{l_con_25} with \eqref{l_con_11}, employ Lemma \ref{l_norm_equivalence} to bound $\norm{\dt v}_{1}/2 \le \norm{ \sqrt{J^1} \sg_{\a^1} \dt v}_{0}$, and utilize Cauchy's inequality to absorb $\int_0^t \ns{\dt v}_{1}$ onto the left side of the resulting inequality; this yields the bound
\begin{multline}\label{l_con_26}
 \hal \int_\Omega \abs{\dt v}^2 J^1(t) + \frac{1}{8} \int_0^t \ns{\dt v}_{1} \le C \sqrt{\ep} \int_0^t \int_\Omega \frac{\abs{\dt v}^2}{2}  J^1 + P(\sqrt{\ep}) \int_0^t \ns{q}_{0}  \\
+ P(\sqrt{\ep}) \norm{q}_{L^\infty H^0} \left[ \sum_{j=0}^1 \norm{\dt^j \zeta^1 - \dt^j \zeta^2}_{L^\infty H^{1/2}} + \norm{v}_{L^\infty H^{1}}\right] \\
+ P(\sqrt{\ep}) \int_0^t \norm{q}_{0} \left[ \sum_{j=0}^2 \norm{\dt^j \zeta^1 - \dt^j \zeta^2}_{1/2} + \norm{v}_{1} \right] \\
+ \int_0^t  \left[ P(\sqrt{\ep}) \mathcal{Z} + C \ns{\dt \zeta^1 - \dt\zeta^2 }_{-1/2}\right].
\end{multline}
This bound can be viewed as a differential inequality of the form 
\begin{equation}
x(t) + y(t) \le C \sqrt{\ep} \int_0^t x(s)ds + F(t), 
\end{equation}
where $x,y,F\ge0$, $x(0)=0$, and $F(t)$ is increasing in $t$.   Gronwall's lemma then implies that 
\begin{equation}\label{l_con_27}
 x(t) + y(t) \le e^{C \sqrt{\ep} t} F(t).
\end{equation}
We assume that $\ep_1$ and $T_1$ are sufficiently small for $e^{C \sqrt{\ep} t} \le e^{C \sqrt{\ep_1} T_1} \le 2$.  Then from \eqref{l_con_26}, \eqref{l_con_27}, and Lemma \ref{l_norm_equivalence} we  deduce  the bound
\begin{multline}\label{l_con_28}
 \ns{\dt v}_{L^\infty H^{0}} + \ns{\dt v}_{L^2 H^{1}} \le P(\sqrt{\ep}) \ns{q}_{L^2 H^{0}} 
+ C  \ns{\dt \zeta^1 - \dt\zeta^2 }_{L^2 H^{-1/2}} + \int_0^T  P(\sqrt{\ep}) \mathcal{Z} 
 \\
+ P(\sqrt{\ep}) \norm{q}_{L^\infty H^0} \left[ \sum_{j=0}^1 \norm{\dt^j \zeta^1 - \dt^j \zeta^2}_{L^\infty H^{1/2}} + \norm{v}_{L^\infty H^{1}}\right] \\
+ P(\sqrt{\ep})  \norm{q}_{L^2 H^{0}} \left[ \sum_{j=0}^2 \norm{\dt^j \zeta^1 - \dt^j \zeta^2}_{L^2 H^{1/2}} + \norm{v}_{L^2 H^{1}} \right],
\end{multline}
where again  the temporal $L^\infty$ and $L^2$ norms are computed over $[0,T]$.

Step 4 -- Elliptic estimates for $v$ and $q$

In order to close our estimates, we must be able to estimate $v$ and $q$.  This will be accomplished with an elliptic estimate.  We  combine Proposition \ref{l_stokes_regularity} with the estimates \eqref{l_con_16}--\eqref{l_con_20} to deduce the bound for $r=0,1$, 
\begin{multline}\label{l_con_29}
 \ns{v}_{r+2} + \ns{q}_{r+1} \ls \ns{\dt v}_{r} + \ns{H^1 }_{r} + \ns{\diverge_{\a^1}( \sg_{(\a^1-\a^2)}v^2  ) }_{r}   + \ns{H^2}_{r+1} \\
+ \ns{H^3}_{r+1/2} +  \ns{\sg_{(\a^1 - \a^2)} v^2 \n^1}_{r+1/2} 
\ls \ns{\dt v}_{r} + \ns{\zeta^1 - \zeta^2}_{r+1/2} \\
+ P(\sqrt{\ep}) \left[ \ns{\zeta^1 - \zeta^2}_{r+3/2}  + \ns{\dt \zeta^1 - \dt \zeta^2}_{r-1/2}  + \ns{w^1 - w^2}_{r+1} \right].
\end{multline}
We set $r=0$ in \eqref{l_con_29} and then take the supremum in time over $[0,T]$ to find
\begin{multline}\label{l_con_30}
 \ns{v}_{L^\infty H^2} + \ns{q}_{L^\infty H^1} \ls  \ns{\dt v}_{L^\infty H^0} + \ns{\zeta^1 - \zeta^2}_{L^\infty H^{1/2}} \\
+ P(\sqrt{\ep}) \left[ \ns{\zeta^1 - \zeta^2}_{L^\infty H^{3/2}}  + \ns{\dt \zeta^1 - \dt \zeta^2}_{L^\infty H^{-1/2}}  + \ns{w^1 - w^2}_{L^\infty H^{1}} \right].
\end{multline}
Then we set $r=1$ in \eqref{l_con_29} and integrate over $[0,T]$ to find
\begin{multline}\label{l_con_31}
 \ns{v}_{L^2 H^3} + \ns{q}_{L^2 H^2} \ls  \ns{\dt v}_{L^2 H^1} + \ns{\zeta^1 - \zeta^2}_{L^2 H^{3/2}} \\
+ P(\sqrt{\ep}) \left[ \ns{\zeta^1 - \zeta^2}_{L^2 H^{5/2}}  + \ns{\dt \zeta^1 - \dt \zeta^2}_{L^2 H^{1/2}}  + \ns{w^1 - w^2}_{L^2 H^{2}} \right].
\end{multline}

Step 5 -- Estimates of $\zeta^1 - \zeta^2$

Now we turn to estimating the difference  $\zeta^1 - \zeta^2$ in terms of $w^1 - w^2$.  We subtract the equations satisfied by $\zeta^2$ from the one for $\zeta^1$ to find that 
\begin{equation}\label{l_con_32}
 \begin{cases}
  \dt(\zeta^1 - \zeta^2) + w^1 \cdot D (\zeta^1 - \zeta^2) = (w^1 - w^2) \cdot \n^2 & \text{in } \Sigma \\
  (\zeta^1 - \zeta^2)(t=0) = 0.
 \end{cases}
\end{equation}
The PDE \eqref{l_con_32} is a transport equation for $\zeta^1 - \zeta^2$, so we can employ Lemma \ref{i_sobolev_transport} to estimate
\begin{multline}
 \norm{\zeta^1-\zeta^2}_{L^\infty H^{5/2}} \le \exp\left( C \int_0^T \snormspace{w^1(r)}{7/2}{\Sigma}dr \right) \int_0^T \snormspace{(w^1 - w^2)\cdot \n^2 (r)}{5/2}{\Sigma} dr \\
\ls  e^{C \sqrt{T} \sqrt{\ep}} (1+P(\sqrt{\ep}))\int_0^T \norm{(w^1 - w^2)(r)}_{3} dr \\
\ls  e^{C \sqrt{T} \sqrt{\ep}} (1+P(\sqrt{\ep}))\sqrt{T}  \norm{w^1 - w^2}_{L^2 H^{3}}.
\end{multline}
We can further restrict $\ep_1$ and $T_1$ so that $e^{C \sqrt{T} \sqrt{\ep}} \le 2$ and $1+ P(\sqrt{\ep}) \le 2$; then 
\begin{equation}\label{l_con_33}
 \norm{\zeta^1-\zeta^2}_{L^\infty H^{5/2}} \ls  \sqrt{T}  \norm{w^1 - w^2}_{L^2 H^{3}}.
\end{equation}
Then we use the first equation in \eqref{l_con_32}, trace theory, and the estimate \eqref{l_con_33} to see that
\begin{multline}\label{l_con_34}
 \norm{\dt \zeta^1 - \dt \zeta^2 }_{L^\infty H^{3/2}} \le \norm{(w^1-w^2)\cdot \n^2}_{L^\infty H^{3/2}} + \norm{w^1 \cdot D(\zeta^1 - \zeta^2)}_{L^\infty H^{3/2}} \\
\ls (1+P(\sqrt{\ep}))\norm{w^1 - w^2}_{L^\infty H^{3/2}(\Sigma)} + P(\sqrt{\ep}) \norm{\zeta^1-\zeta^2}_{L^\infty H^{5/2}} \\
\ls  \norm{w^1 - w^2}_{L^\infty H^{2}} +  P(\sqrt{\ep}) \sqrt{T}  \norm{w^1 - w^2}_{L^2 H^{3}}.
\end{multline}
Similarly, we differentiate \eqref{l_con_32} in time to find that
\begin{multline}\label{l_con_35}
 \norm{\dt^2 \zeta^1 - \dt^2 \zeta^2 }_{L^2 H^{1/2}} \ls (1+P(\sqrt{\ep})) \norm{\dt w^1 - \dt w^2}_{L^2 H^{1}} + P(\sqrt{\ep}) \Big[ \norm{ w^1 -  w^2}_{L^2 H^{1}} \\
+ \norm{ \zeta^1 -  \zeta^2}_{L^2 H^{3/2}} 
+ \norm{ \dt \zeta^1 -  \dt \zeta^2}_{L^2 H^{3/2}} \Big]
 \ls
\norm{\dt w^1 - \dt w^2}_{L^2 H^{1}} \\
+ P(\sqrt{\ep}) \sqrt{T} 
\Big[\norm{ w^1 -  w^2}_{L^\infty H^{1}} + \norm{ \zeta^1 -  \zeta^2}_{L^\infty H^{3/2}} + \norm{ \dt \zeta^1 -  \dt \zeta^2}_{L^\infty H^{3/2}} \Big] \\ \ls
 \norm{\dt w^1 - \dt w^2}_{L^2 H^{1}} + P(\sqrt{\ep}) \sqrt{T} \norm{ w^1 -  w^2}_{L^\infty H^{2}} 
+ P(\sqrt{\ep}) T    \norm{w^1 - w^2}_{L^2 H^{3}}.
\end{multline}

Step 6 -- Synthesis: contraction

We now have all of the ingredients to prove our contraction result.   We write
\begin{equation}\label{l_con_36}
\begin{split}
 \mathfrak{N}^{v}(T) &:= \mathfrak{N}(v^1 - v^2,  q^1-q^2;T) \\ 
 \mathfrak{N}^{w}(T) &:= \mathfrak{N}(w^1 - w^2,  0;T) \\
 \mathfrak{M}(T) &:= \mathfrak{M}(\zeta^1 - \zeta^2;T),
\end{split}
\end{equation}
where $\mathfrak{M}$ and $\mathfrak{N}$ are defined by \eqref{l_MNFrak_def}.  We will first rewrite the bounds \eqref{l_con_28}, \eqref{l_con_30}, and \eqref{l_con_31} in terms of these new quantities.

We begin with the right side of \eqref{l_con_28}.  According to the definition of $\mathcal{Z}$, \eqref{l_con_24}, we may bound
\begin{equation}\label{l_con_37}
\ns{q}_{L^2 H^0} + \int_0^T \mathcal{Z} \ls   (1+T) \left[\mathfrak{M}(T) + \mathfrak{N}^{w}(T)\right] + T \mathfrak{N}^{v}(T) 
\end{equation}
Similarly, 
\begin{multline}\label{l_con_38}
 \norm{q}_{L^2 H^{0}} \left[ \sum_{j=0}^2 \norm{\dt^j \zeta^1 - \dt^j \zeta^2}_{L^2 H^{1/2}} + \norm{v}_{L^2 H^{1}} \right] 
\\ \ls \sqrt{T} \sqrt{ \mathfrak{N}^{v}(T) }  \left[ (1+ \sqrt{T})\sqrt{\mathfrak{M}(T)} + \sqrt{T}  \sqrt{\mathfrak{N}^{v}(T)} \right],
\end{multline}
\begin{equation}\label{l_con_39}
 \ns{\dt \zeta^{m+1} - \dt \zeta^2}_{L^2 H^{-1/2}} \le T \mathfrak{M}(T),
\end{equation}
and 
\begin{multline}\label{l_con_40}
 \norm{q}_{L^\infty H^0} \left[ \sum_{j=0}^1 \norm{\dt^j \zeta^1 - \dt^j \zeta^2}_{L^\infty H^{1/2}} + \norm{v}_{L^\infty H^{1}}\right] 
 \\ \ls 
\sqrt{ \mathfrak{N}^{v}(T) }  \left[ \sqrt{\mathfrak{M}(T)} +   \sqrt{\mathfrak{N}^{v}(T)}\right].
\end{multline}
Then, using \eqref{l_con_37}--\eqref{l_con_40} and Cauchy's inequality, we may rewrite \eqref{l_con_28} as
\begin{multline}\label{l_con_41}
 \ns{\dt v}_{L^\infty H^{0}} + \ns{\dt v}_{L^2 H^{1}}  \ls  \left[T + P(\sqrt{\ep})(1+T)\right] \mathfrak{M}(T) + \left[P(\sqrt{\ep})(1+T)\right] \mathfrak{N}^{w}(T) \\
+ \left[P(\sqrt{\ep})(1+T)\right] \mathfrak{N}^{v}(T).
\end{multline}

Now we turn to the elliptic estimates \eqref{l_con_30}--\eqref{l_con_31}.  The bound \eqref{l_con_30} becomes
\begin{equation}\label{l_con_42}
 \ns{v}_{L^\infty H^2} + \ns{q}_{L^\infty H^1} \ls \ns{\dt v}_{L^\infty H^0} + \ns{\zeta^1 - \zeta^2}_{L^\infty H^{1/2}}   + P(\sqrt{\ep})  \left[ \mathfrak{M}(T) + \mathfrak{N}^{w}(T)   \right].
\end{equation}
Note here that we have kept the term with $\zeta^1 - \zeta^2$ because it does not yet have a small multiplier in front of it.  On the other hand, the bound \eqref{l_con_31} becomes 
\begin{equation}\label{l_con_43}
 \ns{v}_{L^2 H^3} + \ns{q}_{L^2 H^2} \ls  \ns{\dt v}_{L^2 H^1} +  T(1+P(\sqrt{\ep})  )\left[ \mathfrak{M}(T) + \mathfrak{N}^{w}(T)   \right].
\end{equation}
We need not retain the $\zeta^1 - \zeta^2$ term in \eqref{l_con_43} since we can control the square of the temporal $L^2$ norm by the square of the $L^\infty$ norm to pick up a $T$ factor.

Next we reformulate the bounds \eqref{l_con_33}--\eqref{l_con_35} in a similar fashion.  The estimate \eqref{l_con_33} becomes
\begin{equation}\label{l_con_44}
 \ns{\zeta^1 - \zeta^2}_{L^\infty H^{5/2}} \ls T \mathfrak{N}^{w}(T).
\end{equation}
Similarly, we may sum  \eqref{l_con_34} and  \eqref{l_con_35} to bound
\begin{equation}\label{l_con_45}
 \ns{\dt \zeta^1 - \dt \zeta^2}_{L^\infty H^{3/2}}  + \ns{\dt \zeta^1 - \dt \zeta^2}_{L^2 H^{1/2}} \ls \left[ 1+(T+T^2) P(\sqrt{\ep}) \right] \mathfrak{N}^{w}(T).
\end{equation}
Summing \eqref{l_con_44} and \eqref{l_con_45} yields
\begin{equation}\label{l_con_46}
 \mathfrak{M}(T) \ls \left[ 1+(T+T^2) P(\sqrt{\ep}) \right] \mathfrak{N}^{w}(T).
\end{equation}
The estimate \eqref{l_con_05} directly follows from \eqref{l_con_46} and the definitions \eqref{l_con_36} if $\ep_1$ and $T_1$ are small enough.

We now combine the above to get an estimate for $\mathfrak{N}^{v}$ from our estimates for $v,q$.  Note that due to \eqref{l_con_44}, estimate \eqref{l_con_42} also holds with $\ns{\zeta^1 - \zeta^2}_{L^\infty H^{1/2}}$ replaced by $T \mathfrak{N}^{w}(T)$ on the right.  We  then add this modified version of \eqref{l_con_42} to \eqref{l_con_43}, and then add to this a large constant times \eqref{l_con_41}.  If the constant is chosen to be sufficiently large, we can absorb the appearances of $\dt v$ norms on the right side into the left; doing so, we arrive at the bound
\begin{multline}
 \mathfrak{N}^{v}(T) \ls \left[T + P(\sqrt{\ep})(1+T)\right] \mathfrak{M}(T) + \left[T + P(\sqrt{\ep})(1+T)\right] \mathfrak{N}^{w}(T) \\
+ \left[P(\sqrt{\ep})(1+T)\right] \mathfrak{N}^{v}(T).
\end{multline}
This estimate may be combined with \eqref{l_con_46} to see that
\begin{multline}\label{l_con_47}
 \mathfrak{N}^{v}(T) \ls  \left[ 1+(T+T^2) P(\sqrt{\ep}) \right] \left[T + P(\sqrt{\ep})(1+T)\right] \mathfrak{N}^{w}(T) \\
+ \left[P(\sqrt{\ep})(1+T)\right] \mathfrak{N}^{v}(T).
\end{multline}
By further restricting $\ep_1$ and $T_1$, we may replace \eqref{l_con_47} by $\mathfrak{N}^{v}(T) \le \frac{1}{4} \mathfrak{N}^{w}(T) + \frac{1}{2} \mathfrak{N}^{v}(T),$ which may be rearranged to see that $\mathfrak{N}^{v}(T) \le \frac{1}{2} \mathfrak{N}^{w}(T),$ which gives \eqref{l_con_04} after using the definitions of $\mathfrak{N}^{w}(T)$, $\mathfrak{N}^{v}(T)$ given in \eqref{l_con_36}.
\end{proof}

\subsection{Local well-posedness: the proof of Theorem \ref{intro_lwp}}

Now we combine Theorems \ref{l_iteration} and \ref{l_contraction} to produce a solution to problem \eqref{geometric}.  Note that Theorem \ref{intro_lwp} follows directly from the following theorem by changing notation.

\begin{thm}\label{l_nwp}
Assume that $u_0,\eta_0$ satisfy $\mathcal{E}_0, \mathcal{F}_0 < \infty$ and that the initial data $\dt^j u(0)$, etc are as constructed in Section \ref{l_data_section} and satisfy the $(2N)^{th}$  compatibility conditions \eqref{l_comp_cond_2N}.  Then there exist  $0< \delta_0,T_0 <1$   so that if $\mathcal{E}_0 \le \delta_0$ and $0< T \le T_0 \min\{1,1/\mathcal{F}_0\}$, then the following hold.  There exists a solution triple $(u,p,\eta)$ to the  problem \eqref{geometric} on the time interval $[0,T]$ that  achieves the initial data  and  satisfies
\begin{equation}
 \mathfrak{K}(\eta) + \mathfrak{K}(u,p) \le C(\mathcal{E}_0 + T \mathcal{F}_0) \text{ and } \mathcal{F}(\eta) \le C (\mathcal{F}_0 + \mathcal{E}_0 + T \mathcal{F}_0)
\end{equation}
for a universal constant $C>0$.  The solution is unique among functions that achieve the initial data and satisfy $\mathfrak{E}(\eta) + \mathfrak{E}(u,p) < \infty$.  Moroever, $\eta$ is such that the mapping $\Phi(\cdot,t)$, defined by \eqref{mapping_def}, is a $C^1$ diffeomorphism for each $t \in [0,T]$. 
\end{thm}

\begin{proof}

We again divide the proof into several steps.  First, we use Theorem \ref{l_iteration} to construct a sequence of approximate solutions.  Then we use Theorem \ref{l_contraction} to show the sequence contracts in the norm $\sqrt{\mathfrak{M}(\eta;T) + \mathfrak{N}(u,p;T)}$, which yields strong convergence of the sequence.  Next, we use an interpolation argument to improve the convergence results.  These then allow us to pass to the limit in the PDEs to deduce that the limit solves the problem \eqref{geometric}.  Finally, we again use Theorem \ref{l_contraction} to show that our solution is unique.

We assume throughout the proof that $T_0 \le \min\{T_1,\bar{T}\}$, where $\bar{T}$ is given by Theorem \ref{l_iteration}, and $T_1$ is given by Theorem \ref{l_contraction}.  Let $C>0$ denote the universal constant in Theorem \ref{l_iteration}.  We further assume that $T_0 \le \ep_1 /(2C)$, where $\ep_1>0$ is the constant from Theorem \ref{l_contraction}.

Step 1 -- The sequence of approximate solutions

Suppose that $\delta_0 \le \delta$, where $\delta$ is given in Theorem \ref{l_iteration}.  The hypotheses then allow us to apply Theorem \ref{l_iteration} to produce the sequence of triples 
$\{ (u^m,p^m, \eta^m )\}_{m=1}^\infty$, all elements of which achieve the initial data,  satisfy the PDEs \eqref{l_it_01}, \eqref{l_it_02}, and obey the bounds 
\begin{equation}\label{l_nwp_1}
 \sup_{m\ge 1} \left( \mathfrak{K}(\eta^m) + \mathfrak{K}(u^m,p^m) \right) \le C(\mathcal{E}_0 + T \mathcal{F}_0)  \text{ and } \sup_{m\ge 1}\mathcal{F}(\eta^m) \le C (\mathcal{F}_0 + \mathcal{E}_0 + T \mathcal{F}_0).
\end{equation}
We further assume that $\delta_0$ is small enough for $C \delta_0 \le \ep_1/2$ (with $\ep_1$ again from Theorem \ref{l_contraction}) so that \eqref{l_nwp_1}  implies, in particular,  that
\begin{equation}\label{l_nwp_13}
 \sup_{m\ge 1} \max\left\{  \mathfrak{E}(\eta^m) , \mathfrak{E}(u^m,p^m) \right\} \le C(\mathcal{E}_0 + T \mathcal{F}_0) \le C(\delta_0 + T_0) \le \ep_1.
\end{equation}

The uniform bounds \eqref{l_nwp_1} allow us to take weak and weak-$*$ limits, up to the extraction of a subsequence:
\begin{equation}\label{l_nwp_2}
 \begin{cases}
  \dt^j u^m \rightharpoonup \dt^j u  & \text{weakly in }L^2([0,T];H^{4N-2j+1}(\Omega)) \text{ for }j=0,\dotsc,2N+1\\
  \dt^j u^m \wstar \dt^j u  &          \text{weakly-}* \text{ in } L^\infty([0,T];H^{4N-2j}(\Omega)) \text{ for }j=0,\dotsc,2N\\ 
  \dt^j p^m \rightharpoonup \dt^j p  & \text{weakly in } L^2([0,T];H^{4N-2j}(\Omega)) \text{ for }j=0,\dotsc,2N \\
  \dt^j p^m \wstar \dt^j p  &          \text{weakly-}* \text{ in } L^\infty([0,T];H^{4N-2j-1}(\Omega)) \text{ for }j=0,\dotsc,2N-1  
 \end{cases}
\end{equation}
and 
\begin{equation}\label{l_nwp_3}
\begin{cases}
   \eta^m \rightharpoonup \eta  & \text{weakly in }L^2([0,T];H^{4N+1/2}(\Sigma)) \\
  \dt \eta^m \rightharpoonup \dt \eta  & \text{weakly in }L^2([0,T];H^{4N-1/2}(\Sigma)) \\
  \dt^j \eta^m \rightharpoonup \dt^j \eta  & \text{weakly in }L^2([0,T];H^{4N-2j+5/2}(\Sigma)) \text{ for }j=2,\dotsc,2N+1 \\
   \eta^m \wstar  \eta  &          \text{weakly-}* \text{ in } L^\infty([0,T];H^{4N+1/2}(\Sigma))   \\
  \dt^j \eta^m \wstar \dt^j \eta  &          \text{weakly-}* \text{ in } L^\infty([0,T];H^{4N-2j}(\Sigma)) \text{ for }j=1,\dotsc,2N.  \\
 \end{cases}
\end{equation}
Note that in the first convergence result of \eqref{l_nwp_2} we mean $H^{-1}(\Omega) = (\H1)^*$ when $j=2N+1$.  According to the weak and weak-$*$ lower semicontinuity of the norms in $\mathfrak{K}(\eta^m)$, $\mathfrak{K}(u^m,p^m)$, and $\mathcal{F}(\eta^m)$ we find that the limit $(u,p,\eta)$ satisfies
\begin{equation}\label{l_nwp_4}
 \mathfrak{K}(\eta) + \mathfrak{K}(u,p)  \le C(\mathcal{E}_0 + T \mathcal{F}_0)  \text{ and } \mathcal{F}(\eta) \le C (\mathcal{F}_0 + \mathcal{E}_0 + T \mathcal{F}_0). 
\end{equation}

The collection of triples $(v,q,\zeta)$ that achieve the initial data, i.e. $\dt^j v(0) = \dt^j u(0)$, $\dt^j \zeta(0) = \dt^j \eta(0)$, for $j=0,\dotsc,2N$ and $\dt^j q(0) = \dt^j p(0)$ for $j=0,\dotsc,2N-1$, is clearly convex; Lemma \ref{l_sobolev_infinity} implies that it is also  closed with respect to the topology generated by the norm $\sqrt{\mathfrak{D}(\zeta) + \mathfrak{D}(v,q)}$.  As such, the collection is also closed in the corresponding weak topology.  Then, since each $(u^m,p^m,\eta^m)$ is in this collection, we deduce that the limit $(u,p,\eta)$ is as well.  Hence $(u,p,\eta)$ achieves the initial data.

Step 2 -- Contraction

Now we want to improve the weak convergence results of the previous step to strong convergence  in the norm $\sqrt{\mathfrak{M}(\eta;T) + \mathfrak{N}(u,p;T)}$, where $\mathfrak{M}$ and $\mathfrak{N}$ are defined by \eqref{l_MNFrak_def}.  For $m \ge 1$ we set $v^1 = u^{m+2}$, $v^2 = u^{m+1}$, $w^1 = u^{m+1}$, $w^2 = u^m$, $q^1 = p^{m+2}$, $q^2 = p^{m+1}$, $\zeta^1 = \eta^{m+1}$, $\zeta^2 = \eta^m$ in Theorem \ref{l_contraction}.  Because of \eqref{l_it_01}--\eqref{l_it_02} we have that \eqref{l_con_02} holds;  the initial data of $w^j,v^j,q^j,\zeta^j$ match for $j=1,2$ by construction.  Also, \eqref{l_nwp_13} implies that \eqref{l_con_01} holds, so all of the hypotheses of Theorem \ref{l_contraction} are satisfied.  Then \eqref{l_con_04} and \eqref{l_con_05} imply that
\begin{equation}\label{l_nwp_5}
 \mathfrak{N}(u^{m+2} - u^{m+1},p^{m+2}-p^{m+1};T) \le \frac{1}{2}  \mathfrak{N}(u^{m+1} - u^m,p^{m+1}-p^m;T) 
\end{equation}
and 
\begin{equation}\label{l_nwp_6}
 \mathfrak{M}(\eta^{m+1} - \eta^m;T) \le 2 \mathfrak{N}(u^{m+1} - u^m,p^{m+1}-p^m;T).
\end{equation}

The bound \eqref{l_nwp_5} implies that the sequence $\{(u^m,p^m)\}_{m=1}^\infty$  is Cauchy in the norm $\sqrt{\mathfrak{N}(\cdot,\cdot;T)}$, so as $m \to \infty$ 
\begin{equation}\label{l_nwp_7} 
\begin{cases}
 u^m \to u & \text{in } L^\infty([0,T];H^2(\Omega))\cap L^2([0,T],H^3(\Omega))  \\
 \dt u^m \to \dt u & \text{in }L^\infty([0,T];H^0(\Omega))\cap L^2([0,T],H^1(\Omega)) \\ 
 p^m \to  p & \text{in } L^\infty([0,T];H^1(\Omega))\cap L^2([0,T],H^2(\Omega)).
\end{cases}
\end{equation}
Because of  \eqref{l_nwp_6},  we further deduce that  the sequence $\{\eta^m\}_{m=1}^\infty$  is Cauchy in the norm $\sqrt{\mathfrak{M}(\cdot;T)}$, so that as $m \to \infty$
\begin{equation}\label{l_nwp_8} 
\begin{cases}
\eta^m \to \eta & \text{in } L^\infty([0,T];H^{5/2}(\Sigma)) \\
\dt \eta^m \to \dt \eta &\text{in }  L^\infty([0,T];H^{3/2}(\Sigma)) \\
\dt^2 \eta^m \to \dt^2 \eta &\text{in } L^2([0,T];H^{1/2}(\Sigma)).
\end{cases}
\end{equation}

Step 3 -- Interpolation for improved strong convergence

Since $(u^m,p^m,\eta^m)$ obey the bounds \eqref{l_nwp_1}, we can parlay the convergence results \eqref{l_nwp_7}, \eqref{l_nwp_8} into convergence in better norms by use of interpolation theory.  We first interpolate with $L^2 H^0$ norms of temporal derivatives (such estimates take the form
\begin{equation}
\norm{\dt^k f}_{L^2 H^0} \le C(T) \norm{f}_{L^2 H^0}^\theta \norm{\dt^j f}_{L^2 H^0}^{1-\theta} 
\end{equation}
for $j > k\ge 0$ and $\theta = \theta(j,k)\in(0,1)$ and $C(T)$ a constant depending on $T$), which reveals that
\begin{equation}\label{l_nwp_9}
 \begin{cases}
  \dt^j u^m \to \dt^j u &\text{in } L^2([0,T] ;H^0(\Omega))  \text{ for }j=0,\dotsc,2N-1 \\ 
  \dt^j p^m \to \dt^j p &\text{in } L^2([0,T] ;H^0(\Omega))  \text{ for }j=0,\dotsc,2N-1 \\
  \dt^j \eta^m \to \dt^j \eta &\text{in } L^2([0,T] ; H^0(\Sigma))  \text{ for }j=0,\dotsc,2N. \\
 \end{cases}
\end{equation}
Here the range of $j$ is determined by the range of $j$ appearing in $\mathfrak{D}(\eta)$ and $\mathfrak{D}(u,p)$.  Then we use spatial interpolation between $H^0$ and $H^k$ to deduce from  \eqref{l_nwp_9} that
\begin{equation}\label{l_nwp_10}   
 \begin{cases}
  \dt^j u^m \to \dt^j u &\text{in } L^2([0,T] ;H^{4N-2j}(\Omega))  \text{ for }j=0,\dotsc,2N-1 \\ 
  \dt^j p^m \to \dt^j p &\text{in } L^2([0,T] ;H^{4N-2j-1}(\Omega))  \text{ for }j=0,\dotsc,2N-1 \\
  \eta^m \to \eta &\text{in } L^2([0,T] ;H^{4N}(\Sigma))  \\
  \dt \eta^m \to \dt \eta &\text{in } L^2([0,T] ;H^{4N-1}(\Sigma))  \\ 
  \dt^j \eta^m \to \dt^j \eta &\text{in } L^2([0,T] ;H^{4N-2j+2}(\Sigma))  \text{ for }j=2,\dotsc,2N. \\
 \end{cases}
\end{equation}
Here the Sobolev index is determined by the Sobolev index $k$ in the $L^2 H^k$ norms of $\mathfrak{D}(\eta)$ and $\mathfrak{D}(u,p)$.  Finally, we use the temporal $L^2$ convergence of \eqref{l_nwp_10} to get $L^\infty$ and $C^0$ convergence by applying Lemma \ref{l_sobolev_infinity}.  This yields
\begin{equation}\label{l_nwp_11} 
 \begin{cases}
  \dt^j u^m \to \dt^j u &\text{in } C^0([0,T] ;H^{4N-2j-1}(\Omega))  \text{ for }j=0,\dotsc,2N-2 \\ 
  \dt^j p^m \to \dt^j p &\text{in } C^0([0,T] ;H^{4N-2j-2}(\Omega))  \text{ for }j=0,\dotsc,2N-2 \\
  \eta^m \to \eta &\text{in } C^0([0,T] ;H^{4N-1/2}(\Sigma))  \\
  \dt \eta^m \to \dt \eta &\text{in } C^0([0,T] ;H^{4N-3/2}(\Sigma))  \\ 
  \dt^j \eta^m \to \dt^j \eta &\text{in } C^0([0,T] ;H^{4N-2j+1}(\Sigma))  \text{ for }j=2,\dotsc,2N-1. \\
 \end{cases}
\end{equation}

Step 4 -- Passing to the limit in the PDEs

The strong convergence results of \eqref{l_nwp_11} are more than sufficient for us to pass to the limit in the equations \eqref{l_it_01}, \eqref{l_it_02} for each $t \in [0,T]$.  Doing so, we find that the limits $(u,p,\eta)$ are a strong solution to problem \eqref{geometric} on the time interval $t\in[0,T]$.

Step 5 -- Uniqueness 

We now turn to the question of uniqueness of our solution $(u,p,\eta)$.  Suppose that $(v,q,\zeta)$ is another solution to \eqref{geometric} on the time interval $[0,T]$ that achieves the same initial data as $(u,p,\eta)$ and which satisfies $\mathfrak{E}(\zeta) + \mathfrak{E}(v,q) < \infty$.  By continuity we may restrict to a temporal subinterval $[0,T_*]\subset [0,T]$ so that $\mathfrak{E}_0(\eta) + \mathfrak{E}_0(u,p) \le \mathfrak{E}(\zeta) + \mathfrak{E}(v,q) \le \ep_1$, where $\ep_1$ is given in Theorem \ref{l_contraction} and the norms are computed on $[0,T_*]$.  We then set $v^1 =w^1 = u$, $v^2 =w^2= v$, $q^1 = p$, $q^2 = q$, $\zeta^1 = \eta$, and $\zeta^2 = \zeta$ in Theorem \ref{l_contraction} to deduce that
\begin{equation}
 \mathfrak{N}(u - v,p-q;T_*) \le \frac{1}{2}  \mathfrak{N}(u-v,p-q;T_*)  \text{ and } 
 \mathfrak{M}(\eta  - \zeta;T_*) \le 2 \mathfrak{N}(u - v,p -q;T_*),
\end{equation}
which implies that $u=v$, $p=q$, $\eta = \zeta$ on the time interval $[0,T_*]$.  This argument can then be iterated in the usual way, repeatedly increasing $T_*$, to extend the uniqueness to all of the interval $[0,T]$.

Step 6 -- Diffeomorphism 

It is easy to check that the  smallness of $\mathfrak{K}(\eta)$ is sufficient to guarantee that the map $\Phi$, given by \eqref{mapping_def}, is a $C^1$ diffeomorphism for each $t \in[0,T]$.

\end{proof}

\chapter{Preliminaries for the a priori estimates}\label{section_prelims}

In this chapter we present some preliminary results that we will use in both Chapters \ref{section_inf} and \ref{section_per}.  In Section \ref{prelim_1} we present two forms of equations similar to \eqref{geometric} and describe the corresponding energy evolution structure.  In Section \ref{prelim_2} we record some useful lemmas.

\section{Forms of the equations}\label{prelim_1}

\subsection{Geometric}

We now give a linear formulation of the PDE \eqref{geometric} in its geometric form.   Suppose that $\eta, u$ are known and that $\a, \n, J,$ etc are given in terms of $\eta$ as usual (\eqref{A_def}, etc).  We then consider the linear equation for $(v,q,\zeta)$ given by
\begin{equation}\label{linear_geometric}
\begin{cases}
  \dt v - \dt \bar{\eta} \tilde{b} K \p_3 v + u \cdot \naba v + \diva S_{\a}(q,v) = F^1  & \text{in } \Omega \\
 \diva v = F^2  & \text{in } \Omega \\
 S_{\a}(q,v) \n = \zeta \n + F^3 & \text{on } \Sigma \\
 \dt \zeta - \n \cdot v = F^4  & \text{on } \Sigma\\
v =0 & \text{on } \Sigma_b.
\end{cases}
\end{equation}

Now we record the natural energy evolution associated to solutions $v,q,\zeta$ of \eqref{linear_geometric}.

\begin{lem}\label{geometric_evolution}
Suppose that $u$ and $\eta$ are given solutions to \eqref{geometric}.  Suppose $(v,q,\zeta)$ solve \eqref{linear_geometric}.  Then
\begin{equation}\label{i_ge_ev_0}
 \dt \left( \hal \int_\Omega  J \abs{v}^2  + \hal \int_\Sigma \abs{\zeta}^2 \right) 
+ \hal \int_\Omega J \abs{ \sg_{\mathcal{A}} v}^2 = \int_\Omega J (v \cdot F^1 +   q  F^2)  
+ \int_\Sigma -v \cdot F^3 + \zeta F^4 .
\end{equation}
\end{lem}
\begin{proof}
We multiply the $i^{th}$ component of the first equation of \eqref{linear_geometric} by $J v_i$, sum over $i$ and integrate over $\Omega$ to find that 
\begin{equation}\label{i_ge_ev_1}
 I + II = III
\end{equation}
for 
\begin{equation}
 I = \int_\Omega \dt v_i J v_i - \dt \bar{\eta} \tilde{b} \p_3 v_i v_i  + u_j \mathcal{A}_{jk} \p_k v_i J v_i,
\end{equation}
\begin{equation}
 II = \int_\Omega    \mathcal{A}_{jk}\p_k S_{ij}(v,q) J v_i, 
\text{ and }
 III = \int_\Omega F^1 \cdot v J.
\end{equation}
In order to integrate by parts in $I,II$ we will utilize the geometric identity $\p_k(J \mathcal{A}_{ik})=0$ for each $i$.

Then
\begin{equation}
 I = \dt  \int_\Omega \frac{\abs{v}^2 J}{2} + \int_\Omega -\frac{\abs{v}^2 \dt J}{2} - \dt \bar{\eta} \tilde{b} \p_3 \frac{ \abs{v}^2}{2} +  u_j \p_k \left( J \mathcal{A}_{jk}  \frac{\abs{v}^2}{2} \right):= I_1 + I_2.
\end{equation}
Since $\tilde{b} = 1 + x_3/b$, an integration by parts and an application of the boundary condition $v=0$ on $\Sigma_b$ reveals that
\begin{multline}
I_2 = \int_\Omega -\frac{\abs{v}^2 \dt J}{2} - \dt \bar{\eta} \tilde{b} \p_3 \frac{ \abs{v}^2}{2} +  u_j \p_k \left( J \mathcal{A}_{jk}  \frac{\abs{v}^2}{2} \right) = \int_\Omega -\frac{\abs{v}^2 \dt J}{2} + \frac{\abs{v}^2}{2} \left( \frac{\dt \bar{\eta}}{b} + \tilde{b} \dt \p_3 \bar{\eta} \right) \\
- \int_\Omega \p_k u_j J \mathcal{A}_{jk}  \frac{\abs{v}^2}{2} + \hal \int_\Sigma - \dt \eta \abs{v}^2
+   u_j J \mathcal{A}_{jk} e_3\cdot e_k \abs{v}^2.
\end{multline}
It is straightforward to verify that $ \dt J  = \dt \bar{\eta}/b + \tilde{b} \dt \p_3 \bar{\eta}$ in $\Omega$ and that  $J \mathcal{A}_{jk} e_3\cdot e_k = \n_j$ on $\Sigma$.  Then since $u,\eta$ satisfy $\p_k u_j \mathcal{A}_{jk}  =0$ and $\dt \eta = u \cdot \n$, we have $I_2 =  0$.  Hence
\begin{equation}\label{i_ge_ev_2}
 I = \dt  \int_\Omega \frac{\abs{v}^2 J}{2}.
\end{equation}

A similar integration by parts shows that
\begin{multline}
 II =  \int_\Omega -  \mathcal{A}_{jk}  S_{ij}(v,q) J  \p_k v_i  + \int_\Sigma J \mathcal{A}_{j3} S_{ij}(v,q) v_i \\
= \int_\Omega -q \mathcal{A}_{ik} \p_k v_i J + J\frac{\abs{\sg_\mathcal{A} v}^2}{2} + \int_\Sigma  S_{ij}(v,q) \n_j v_i
\end{multline}
so that
\begin{equation}
 II = \int_\Omega - q J F^2 +  J\frac{\abs{\sg_\mathcal{A} v}^2}{2} + \int_\Sigma \zeta \n \cdot v + v \cdot F^3.
\end{equation}
But
\begin{equation} 
\int_\Sigma \zeta \n \cdot v  = \int_\Sigma \zeta (\dt \zeta - F^4)  = \dt \int_\Sigma \frac{\abs{\zeta}^2}{2} + \int_\Sigma - \zeta F^4,
\end{equation}
which means
\begin{equation}\label{i_ge_ev_3}
 II = \int_\Omega - q J F^2 +  J\frac{\abs{\sg_\mathcal{A} v}^2}{2} + \dt \int_\Sigma \frac{\abs{\zeta}^2}{2} + \int_\Sigma - \zeta F^4.
\end{equation}
Now \eqref{i_ge_ev_0} follows from \eqref{i_ge_ev_1}, \eqref{i_ge_ev_2}, and \eqref{i_ge_ev_3}.
\end{proof}

In order to utilize \eqref{linear_geometric} we apply the differential operator $\pa=\dt^{\alpha_0}$  to \eqref{geometric}.  The resulting equations are \eqref{linear_geometric} for $v = \pa u$, $q = \pa p$, and $\zeta = \pa \eta$, where
\begin{equation}\label{F_def_start}
 F^1 = F^{1,1} + F^{1,2} + F^{1,3} + F^{1,4} + F^{1,5} + F^{1,6}
\end{equation}
for
\begin{equation}
 F^{1,1}_i = \sum_{0 < \beta < \alpha } C_{\alpha,\beta}   \p^\beta ( \dt \bar{\eta} \tilde{b} K)      \p^{\alpha - \beta} \p_3 u_i   + \sum_{0 < \beta \le \alpha}  C_{\alpha,\beta}   \p^{\alpha-\beta} \dt \bar{\eta} \p^{\beta} (\tilde{b} K)     \p_3 u_i
\end{equation}
\begin{equation}
 F^{1,2}_i  = - \sum_{ 0 < \beta \le \alpha } C_{\alpha,\beta} \left(  \p^\beta ( u_j  \mathcal{A}_{jk} ) \p^{\alpha - \beta} \p_k u_i  +  \p^\beta \mathcal{A}_{ik} \p^{\alpha-\beta} \p_k p\right)
\end{equation}
\begin{equation}
  F^{1,3}_i = \sum_{0 <  \beta \le \alpha} C_{\alpha,\beta} \p^\beta \mathcal{A}_{j\ell} \p^{\alpha - \beta} \p_\ell (\mathcal{A}_{im} \p_m u_j + \mathcal{A}_{jm}\p_m u_i)
\end{equation}
\begin{equation}
 F^{1,4}_i = \sum_{0 < \beta < \alpha} C_{\alpha,\beta}   \mathcal{A}_{jk} \p_k (\p^\beta \mathcal{A}_{i\ell} \p^{\alpha - \beta} \p_\ell u_j  + \p^\beta \mathcal{A}_{j\ell} \p^{\alpha -\beta} \p_\ell u_i)    
\end{equation}
\begin{equation}
 F^{1,5}_i = \p^{\alpha} \dt \bar{\eta}  \tilde{b} K     \p_3 u_i     
\text{ and } F^{1,6}_i =  \mathcal{A}_{jk} \p_k (\pa \mathcal{A}_{i\ell}  \p_\ell u_j  + \pa \mathcal{A}_{j\ell}  \p_\ell u_i).
\end{equation}
In these equations, the terms $C_{\alpha,\beta}$ are constants that depend on $\alpha$ and $\beta$.  The term $F^2 = F^{2,1} + F^{2,2}$ for
\begin{equation}\label{i_F2_def}
 F^{2,1} = - \sum_{0 < \beta < \alpha} C_{\alpha,\beta}  \p^\beta \mathcal{A}_{ij} \p^{\alpha-\beta} \p_j u_i
\text{ and }
  F^{2,2} = -\pa \mathcal{A}_{ij}  \p_j u_i.
\end{equation}
We write $F^3 = F^{3,1} + F^{3,2}$ for
\begin{equation}
 F^{3,1} = \sum_{0 < \beta \le \alpha} C_{\alpha,\beta}  \p^\beta D \eta ( \p^{\alpha-\beta} \eta -\p^{\alpha-\beta} p )
\end{equation}
\begin{equation}
 F^{3,2}_i = \sum_{0 < \beta \le \alpha} C_{\alpha,\beta} (  \p^\beta ( \n_j \mathcal{A}_{im} ) \p^{\alpha-\beta} \p_m u_j + \p^\beta ( \n_j \mathcal{A}_{jm} ) \p^{\alpha-\beta} \p_m u_i ). 
\end{equation}
Finally, 
\begin{equation}\label{F_def_end}
 F^4 =  \sum_{0 < \beta \le \alpha} C_{\alpha,\beta}  \p^\beta D \eta \cdot \p^{\alpha-\beta} u.
\end{equation}

\subsection{Perturbed linear}

Writing the equations in the form \eqref{geometric} is more faithful to the geometry of the free boundary problem, but it is inconvenient for many of our a priori estimates.  This stems from the fact that if we want to think of the coefficients of the equations for $u,p$ as being frozen for  a fixed free boundary given by $\eta$, then the underlying linear operator has non-constant coefficients.  This makes it unsuitable for applying differential operators.

To get around this problem, in many parts of Chapters \ref{section_inf} and \ref{section_per} we will analyze the PDE in a different formulation, which looks like a perturbation of the linearized problem.  The utility of this form of the equations lies in the fact that the linear operators have constant coefficients.   The equations in this form are
\begin{equation}\label{linear_perturbed}
 \begin{cases}
  \dt u + \nab p - \Delta u = G^1 & \text{in }\Omega \\
  \diverge{u} = G^2 & \text{in }\Omega \\
  (p I - \sg u - \eta I)e_3 = G^3 & \text{on }\Sigma \\
  \dt \eta - u_3 = G^4 & \text{on } \Sigma \\
  u =0 & \text{on }\Sigma_b.
 \end{cases}
\end{equation}
Here we have written  $G^1 = G^{1,1} + G^{1,2} + G^{1,3} + G^{1,4} + G^{1,5}$ for
\begin{equation}\label{Gi_def_start}
 G^{1,1}_i = (\delta_{ij} - \mathcal{A}_{ij} )\p_j p 
\end{equation}
\begin{equation}
 G^{1,2}_i = u_j\mathcal{A}_{jk} \p_k u_i 
\end{equation}
\begin{equation}
 G^{1,3}_i = [ K^2(1+A^2 + B^2) - 1]\p_{33}u_i - 2AK \p_{13}u_i -2BK \p_{23} u_i 
\end{equation}
\begin{equation}
 G^{1,4}_i 
= [  -K^3(1+A^2+B^2) \p_3 J + AK^2 (\p_1 J  + \p_3 A) + BK^2 (\p_2 J + \p_3 B) - K (\p_1 A + \p_2 B)      ]\p_3 u_i
\end{equation}
\begin{equation}
 G^{1,5}_i = \dt \bar{\eta}(1+x_3/b)K \p_3 u_i.
\end{equation}
$G^2$ is the function
\begin{equation}
 G^2= AK \p_3 u_1 + BK \p_3 u_2  + (1-K)\p_3 u_3,
\end{equation}
and $G^3$ is the vector 
\begin{multline}\label{G3_def}
G^3 :=  \p_1 \eta
\begin{pmatrix}
 p-\eta -2(\p_1 u_1 -AK \p_3 u_1  ) \\
 -\p_2 u_1 - \p_1 u_2  + BK \p_3 u_1 + AK \p_3 u_2 \\
 -\p_1 u_3 - K \p_3 u_1 + AK \p_3 u_3
\end{pmatrix}
\\ + 
\p_2 \eta
\begin{pmatrix}
  -\p_2 u_1 - \p_1 u_2  + BK \p_3 u_1 + AK \p_3 u_2  \\
  p-\eta -2(\p_2 u_2 -BK \p_3 u_2  )  \\
 -\p_2 u_3 - K \p_3 u_2 + BK \p_3 u_3
\end{pmatrix}
+ 
\begin{pmatrix}
  (K-1) \p_3 u_1  +AK \p_3 u_3 \\
  (K-1) \p_3 u_2  +BK \p_3 u_3  \\
  2(K-1)\p_3 u_3
\end{pmatrix}.
\end{multline}
Finally,
\begin{equation}\label{Gi_def_end}
 G^4 = -D \eta \cdot u.
\end{equation}

\begin{remark}\label{G3_remark}
The appearance of the term $(p-\eta)$ in the first two rows of the first two vectors in the definition of $G^3$ can cause some technical problems later when we attempt to estimate $G^3$.  Notice though, that according to \eqref{linear_perturbed}, we may write
\begin{multline}\label{G3_alternate}
 (p-\eta) = 2 \p_3 u_3 + G^3 \cdot e_3 = \p_1 \eta (-\p_1 u_3 - K \p_3 u_1 + AK \p_3 u_3) \\+ \p_2 \eta (-\p_2 u_3 - K \p_3 u_2 + BK \p_3 u_3) + 2 K \p_3 u_3
\end{multline}
on $\Sigma$.  We may then replace the appearances of $(p-\eta)$ in \eqref{G3_def} with the right side of \eqref{G3_alternate}.

\end{remark}

At several points in our analysis, we will need to localize \eqref{linear_perturbed} by multiplying by a cutoff function.  This leads us to consider the energy evolution for a minor modification of \eqref{linear_perturbed}.

\begin{lem}\label{general_evolution}
Suppose $(v,q,\zeta)$ solve
\begin{equation}\label{g_e_0}
 \begin{cases}
  \dt v + \nab q - \Delta v = \Phi^1 & \text{in }\Omega \\
  \diverge{v} = \Phi^2 & \text{in }\Omega \\
  (q I - \sg v )e_3 = a \zeta e_3  + \Phi^3 & \text{on }\Sigma \\
  \dt \zeta - v_3 = \Phi^4 & \text{on } \Sigma \\
  v =0 & \text{on }\Sigma_b,
 \end{cases}
\end{equation}
where either $a=0$ or $a=1$.  Then
\begin{equation}\label{g_e_01}
 \dt  \left( \hal \int_\Omega \abs{v}^2  + \hal\int_\Sigma a \abs{\zeta}^2 \right) 
+ \hal \int_\Omega \abs{\sg v}^2 
= \int_\Omega v \cdot \Phi^1 + q \Phi^2 
+ \int_\Sigma -v \cdot \Phi^3 +  a \zeta \Phi^4.
\end{equation} 
\end{lem}
\begin{proof}
 We take the inner-product of the first equation in \eqref{g_e_0} with $v$ and integrate over $\Omega$ to find 
\begin{equation}\label{g_e_1}
 \dt \int_\Omega \frac{\abs{v}^2}{2} - \int_\Omega (qI - \sg v) : \nab u + \int_\Sigma (qI - \sg v) e_3 \cdot u = \int_\Omega v \cdot \Phi^1.
\end{equation}
We then use the second equation in \eqref{g_e_0} to compute
\begin{equation}\label{g_e_2}
 \int_\Omega -(qI - \sg v) : \nab u = \int_\Omega -q \diverge{v} + \frac{\abs{\sg v}^2}{2} = \int_\Omega -q \Phi^2 + \frac{\abs{\sg v}^2}{2}.
\end{equation}
The boundary conditions in \eqref{g_e_0} provide the equality
\begin{equation}\label{g_e_3}
 \int_\Sigma (qI - \sg v) e_3 \cdot v = \int_\Sigma a \zeta v_3 + v \cdot \Phi^3 = \dt \int_\Sigma a \frac{\abs{\zeta}^2}{2} + \int_\Sigma -a \zeta \Phi^4 + v\cdot \Phi^3.
\end{equation}
Combining \eqref{g_e_1}--\eqref{g_e_3} then yields \eqref{g_e_01}.

\end{proof}

\section{Some initial lemmas}\label{prelim_2}

The following result is useful for removing the appearance of $J$ factors.

\begin{lem}\label{infinity_bounds}
There exists a universal $0 < \delta < 1$ so that if $\ns{\eta}_{5/2} \le \delta$, then 
\begin{equation}\label{infb_01}
 \pnorm{J-1}{\infty}^2 +\pnorm{A}{\infty}^2 + \pnorm{B}{\infty}^2 \le \hal, \text{ and } 
  \pnorm{K}{\infty}^2 + \pnorm{\mathcal{A}}{\infty}^2 \ls 1.
\end{equation}
\end{lem}
\begin{proof}
According to the definitions of $A,B,J$ given in \eqref{ABJ_def} and Lemmas \ref{i_poisson_grad_bound} and \ref{p_poisson}, we may bound
\begin{equation}
 \pnorm{J-1}{\infty}^2 +\pnorm{A}{\infty}^2 + \pnorm{B}{\infty}^2\ls \norm{\bar{\eta}}^2_3 \ls  \norm{\eta}_{5/2}^2.
\end{equation}
Then if $\delta$ is sufficiently small, we find that the first inequality in \eqref{infb_01} holds.  As a consequence  $\pnorm{K}{\infty}^2  + \pnorm{\mathcal{A}}{\infty}^2 \ls 1$, which is the second inequality in \eqref{infb_01}.
\end{proof}

We now compute $\dt \eta$ in terms of a pair of auxiliary functions, $U_1, U_2$ defined on $\Sigma$.  Our result holds in both the periodic and non-periodic cases.  Note that in our analysis later, $u$ and $\eta$ will always be sufficiently smooth to justify the calculations in the next Lemma, and it  will always hold that $U_i \in H^1(\Sigma)$.

\begin{lem}\label{dt_eta}
For $i=1,2$, define $U_i : \Sigma \to \Rn{}$ by
\begin{equation}
 U_i(x') = \int_{-b(x')}^0 J(x',x_3) u_i(x', x_3) dx_3,
\end{equation}
where we understand that $-b(x') = -b <0$ in the non-periodic case.  Then $\dt \eta = - \p_1 U_1 - \p_2 U_2$ on $\Sigma$.
\end{lem}
\begin{proof}
If $\Sigma= \Rn{2}$, then let $\varphi \in \mathscr{S}(\Sigma)$.  If $\Sigma = (L_1 \mathbb{T}) \times  (L_2 \mathbb{T})$, then let $\varphi \in C^\infty(\Sigma)$.  In either case, on $\Sigma$ we have that $u \cdot \n = u \cdot (J \mathcal{A} e_3) = J \mathcal{A}^T u \cdot e_3 = J \mathcal{A}^T u \cdot \nu$, where   $\nu = e_3$ is the unit normal to $\Sigma$.   We may use the equation for $\dt \eta$ in \eqref{geometric} and the divergence theorem  to compute
\begin{multline}\label{dt_eta_1}
 \int_\Sigma \dt \eta \varphi = \int_\Sigma (-u_1 \p_1 \eta - u_2 \p_2 \eta + u_3)  \varphi = \int_\Sigma \varphi J \mathcal{A}_{ij} u_i  \nu_j = \int_\Omega \p_j ( \varphi J \mathcal{A}_{ij}  u_i) \\
= \int_\Omega \p_j \varphi J \mathcal{A}_{ij}  u_i + \varphi \p_j(J \mathcal{A}_{ij}) u_i + \varphi J \mathcal{A}_{ij} \p_j u_i = \int_\Omega \p_j \varphi J \mathcal{A}_{ij}  u_i, 
\end{multline}
where the last equality follows from  the geometric identity $\p_j (J \mathcal{A}_{ij})=0$ and the equation $\mathcal{A}_{ij} \p_j u_i=0$, which is the second equation in \eqref{geometric}.  According to the  definition of $\mathcal{A}$ given by  \eqref{A_def}, we may write $\mathcal{A}_{ij} = \delta_{ij} + \delta_{j3} Z_i$
for $\delta_{ij}$ the Kronecker delta and $Z = K(-A e_1 -B e_2 + e_3)$.  Then
\begin{equation}\label{dt_eta_2}
 \int_\Omega \p_j \varphi J \mathcal{A}_{ij}  u_i  = \int_\Omega \p_j \varphi J u_i (\delta_{ij} + \delta_{j3} Z_j) = \int_\Omega \p_i \varphi J u_i + \int_\Omega \p_3 \varphi J u_i  Z_i 
=  \int_\Omega \p_i \varphi J u_i 
\end{equation}
since $\p_3 \varphi = 0$, a consequence of the fact that  $\varphi = \varphi(x_1,x_2)$ is independent of $x_3$.   Again because $\varphi$ depends only on $(x_1,x_2)=x' \in \Sigma$, we may  write
\begin{equation}\label{dt_eta_3}
 \int_\Omega \p_i \varphi J u_i = \int_\Sigma \p_i \varphi(x') \int_{-b(x')}^0 J(x',x_3) u_i(x',x_3)dx_3 dx'= \int_\Sigma \p_i \varphi(x') U_i(x') dx'.
\end{equation}
Now we chain together \eqref{dt_eta_1}, \eqref{dt_eta_2}, and \eqref{dt_eta_3} and integrate by parts to deduce that
\begin{equation}
 \int_\Sigma \dt \eta \varphi = \int_\Sigma -\varphi \p_i U_i.
\end{equation}
Since this holds for any $\varphi \in \mathscr{S}(\Sigma)$ (resp. $C^\infty(\Sigma)$), we then have that $\dt \eta = - \p_i U_i$.  
\end{proof}

Now we parlay this calculation into a computation of the average of $\eta$ over $\Sigma$ in the periodic case.

\begin{lem}\label{avg_const}
Suppose that $\Sigma = (L_1 \mathbb{T}) \times (L_2 \mathbb{T})$.  Then for all $t\ge 0$ where the solution exists, 
\begin{equation}
 \int_\Sigma \eta(x',t) dx' = \int_\Sigma \eta_0(x') dx'.
\end{equation}
\end{lem}
\begin{proof}
 According to Lemma \ref{dt_eta}, we have that $\dt \eta = - \p_1 U_1 - \p_2 U_2$ on $\Sigma$.  Then we may integrate by parts to find that
\begin{equation}
 \frac{d}{dt} \int_\Sigma \eta(x',t) dx' = \int_\Sigma \dt \eta(x',t) dx' = \int_\Sigma (-\p_1 U_1(x',t) - \p_2 U_2(x',t)) dx' =0.
\end{equation}

\end{proof}

\chapter{Global well-posedness and decay in the infinite case}\label{section_inf}

\section{Introduction}\label{inf_1}

In this chapter we prove Theorem \ref{intro_inf_gwp}.  Throughout the chapter we assume that $N \ge 5$  and $\lambda \in (0,1)$ are both fixed.    Notice that Theorem \ref{intro_inf_gwp} is phrased with the choice $N=5$.

In the rest of Section \ref{inf_1} we define the energies and dissipations that are relevant to the non-periodic problem, and we prove some preliminary lemmas.  In Section \ref{inf_2} we perform our bootstrap interpolation argument to control various quantities in terms of $\se{N+2,m}$ and $\sd{N+2,m}$.  In Section \ref{inf_3} we present estimates of the nonlinear forcing terms $G^i$ (as defined in \eqref{Gi_def_start}--\eqref{Gi_def_end}) and some other nonlinearities.  In Section \ref{inf_4} we use the geometric form of the equations to estimate the evolution of the highest-order temporal derivatives.  We also analyze the natural (no derivatives) energy in this context.  Section \ref{inf_5} concerns similar energy evolution estimates for the other horizontal derivatives.  For these we employ the linear perturbed framework with the $G^i$ forcing terms.  In Section \ref{inf_6} we assemble the estimates of Sections \ref{inf_4} and \ref{inf_5} into  unified estimates.  Section \ref{inf_7} concerns the comparison estimates, where we show how to estimate the full energies and dissipations in terms of their horizontal counterparts.  Section \ref{inf_8} combines all of the analysis of Sections \ref{inf_2}--\ref{inf_7} into our a priori estimates for solutions to \eqref{geometric}.  Section \ref{inf_9} concerns a specialized version of the local well-posedness theorem that includes the boundedness of $\il$ terms.  Finally, in Section \ref{inf_10} we record our global well-posedness and decay result, proving Theorem \ref{intro_inf_gwp}.

Below, in \eqref{i_total_energy}, we will define the total energy $\g$ that we use in the global well-posedness analysis.  For the purposes of deriving our a priori estimates, we will assume throughout Sections \ref{inf_2}--\ref{inf_8} that solutions are given on the interval $[0,T]$ and that $\g(T) \le \delta$ for $0 < \delta < 1$ as small as in Lemma \ref{infinity_bounds} so that its conclusions hold.  This also means that $\se{2N}(t) \le 1$ for $t \in [0,T]$.  We should remark that Theorem \ref{intro_lwp} does not produce solutions that necessarily satisfy $\g(T) < \infty$.  All of the terms in  $\g(T)$  are controlled by Theorem \ref{intro_lwp} except those  involving the Riesz operator: $\ns{\il u}_{0}$, $\ns{\il \eta}_{0}$, and $\int_0^T \ns{\il u(t)}_{1}dt$.  To guarantee that these terms are well-defined, we must prove a specialized version of the local well-posedness result, Theorem \ref{i_infinite_lwp}.   In principle, we should record this before the a priori estimates, but the technique we use to control the $\il$ terms is based on one we develop for the a priori estimates, so we present the theorem in  Section \ref{inf_9} after the a priori estimates.  Note that the bounds of Theorem \ref{i_infinite_lwp} control more than just $\g(T)$ (in particular, $\dt^{2N+1} u$ and $\dt^{2N} p$), and the extra control it provides guarantees that all of the calculations used in the a priori estimates are justified.

\subsection{Definitions and notation}

Below we define the energies and dissipations we will use in our analysis.  We state them in general in terms of two integers $n,m\in \mathbb{N}$ with $n\ge m$.  In our actual analysis we will take $n = 2N$ and $n =N+2$ for $N \ge 5$ and $m=1,2$.  Recall that we  employ the derivative conventions described in Section \ref{def_and_term}.  We define the horizontal instantaneous energy with minimal derivative count $m$ (or just horizontal energy, for short) by
\begin{equation}\label{i_horizontal_energy_min}
 \seb{n,m} :=  \ns{\dbm{2n-1} u}_{0} + \ns{D \bar{D}^{2n-1} u}_{0} + \ns{\sqrt{J} \dt^{n} u }_{0} + \ns{\dbm{2n} \eta}_{0}.
\end{equation}
Here the first three terms are split in this manner for the technical convenience of adding the $\sqrt{J}$ term to only the highest temporal derivative.

\begin{remark}\label{i_horizontal_remark}
In light of Lemma \ref{infinity_bounds}, we see that $\seb{n,m}$ satisfies
\begin{equation}
\frac{1}{2} \left(\ns{\dbm{2n} u}_{0} + \ns{\dbm{2n} \eta}_{0} \right) \le  \seb{n,m} \le \frac{3}{2} \left(\ns{\dbm{2n} u}_{0} + \ns{\dbm{2n} \eta}_{0} \right)
\end{equation}
\end{remark}

We define the horizontal dissipation rate with minimal derivative count $m$ (horizontal dissipation) by
\begin{equation}\label{i_horizontal_dissipation_min}
 \sdb{n,m} :=   \ns{\dbm{2n} \sg u}_{0}.
\end{equation}
Let $\il$ be defined by \eqref{il_def_1}--\eqref{il_def_2}.  The horizontal energy without a minimal derivative restriction is
\begin{equation}\label{i_horizontal_energy}
 \seb{n} := \ns{\i_{\lambda} u}_{0} + \ns{\bar{D}_{0}^{2n} u}_{0} + \ns{\i_{\lambda} \eta}_{0} +  \ns{\bar{D}_{0}^{2n} \eta}_{0},
\end{equation}
and the horizontal dissipation without a minimal derivative restriction is 
\begin{equation}\label{i_horizontal_dissipation}
 \sdb{n} := \ns{\sg \i_{\lambda} u}_{0} + \ns{\bar{D}_{0}^{2n} \sg u}_{0}.
\end{equation}

In addition to the horizontal energy and dissipation, we must also define full energies and dissipations, which involve full derivatives.  We write the full energy as
\begin{equation}\label{i_energy_def}
 \se{n} := 
\ns{\i_{\lambda} u}_{0}+  \sum_{j=0}^{n}  \ns{\dt^j  u}_{2n-2j}       
+  \sum_{j=0}^{n-1} \ns{\dt^j  p}_{2n-2j-1}   
+ \ns{\i_{\lambda} \eta}_{0} +  \sum_{j=0}^{n} \ns{\dt^j \eta}_{2n-2j},
\end{equation}
and we define the full dissipation rate  by 
\begin{multline}\label{i_dissipation_def}
 \sd{n} := 
\ns{\i_{\lambda} u}_{1} + \sum_{j=0}^{n}  \ns{\dt^j  u}_{2n-2j+1}   
+ \ns{ \nab  p  }_{2n-1}    
+  \sum_{j=1}^{n-1} \ns{\dt^j  p}_{2n-2j}   \\
+ \ns{D \eta}_{2n-3/2} + \ns{\dt \eta}_{2n-1/2} +\sum_{j=2}^{n+1} \ns{\dt^j \eta}_{2n-2j+ 5/2}.
\end{multline}
We define a similar energy with a minimal derivative count of one by
\begin{multline}\label{i_energy_min_1}
\se{n,1} := \seb{n,1} +
\ns{\nab^2 u}_{2n-2} + \sum_{j=1}^{n} \ns{\dt^j u}_{2n-2j} \\ +
\ns{\nab p}_{2n-2} + \sum_{j=1}^{n-1} \ns{\dt^j p}_{2n-2j-1} +
 \ns{D \eta}_{2n-1} + \sum_{j=1}^{n} \ns{\dt^j \eta}_{2n-2j},
\end{multline} 
and with a minimal derivative count of two by
\begin{multline}\label{i_energy_min_2} 
\se{n,2} := \seb{n,2} +
\ns{\nab^3 u}_{2n-3} + \sum_{j=1}^{n} \ns{\dt^j u}_{2n-2j} \\
+
\ns{\nab^2 p}_{2n-3} + \sum_{j=1}^{n-1} \ns{\dt^j p}_{2n-2j-1} 
 +
\ns{D^2 \eta}_{2n-2} + \sum_{j=1}^{n} \ns{\dt^j \eta}_{2n-2j}.
\end{multline} 
Similarly, the dissipation with a minimal derivative count of one is
\begin{multline}\label{i_dissipation_min_1} 
\sd{n,1}:= \sdb{n,1}  +   \ns{\nab^{3} u }_{2n-2}  + \sum_{j=1}^{n}  \ns{\dt^j  u}_{2n-2j+1}  \\ 
+ \ns{ \nab^{2} p  }_{2n-2}    
+  \sum_{j=1}^{n-1} \ns{\dt^j  p}_{2n-2j}   
+ \ns{D^{2} \eta}_{2n-5/2} + \ns{\dt  \eta}_{2n-1/2} +\sum_{j=2}^{n+1} \ns{\dt^j \eta}_{2n-2j+ 5/2},
\end{multline}
while the dissipation with a minimal derivative count of two is
\begin{multline}\label{i_dissipation_min_2}
\sd{n,2} :=  \sdb{n,2} + \ns{\nab^{4} u  }_{2n-3}   + \sum_{j=1}^{n}  \ns{\dt^j  u}_{2n-2j+1}  
+ \ns{ \nab^{3} p  }_{2n-3}  
+ \ns{\dt  \nab  p  }_{2n-3} \\
+  \sum_{j=2}^{n-1} \ns{\dt^j  p}_{2n-2j} 
+ \ns{D^{3} \eta}_{2n-7/2} 
+ \ns{D \dt \eta}_{2n-3/2} +  \sum_{j=2}^{n+1} \ns{\dt^j \eta}_{2n-2j+ 5/2}.
\end{multline}
Note that by definition $\se{n,m} \ge \seb{n,m}$ and $\sd{n,m} \ge \sdb{n,m}$.  In all of these definitions, the index $n$ counts the highest number of time derivatives used.

Certain norms of $\eta$ and $u$ will play a special role in our analysis; we write
\begin{equation}\label{i_transport_def}
 \f :=   \ns{\eta}_{4N+1/2}
\end{equation}
and
\begin{equation}\label{i_K_def}
 \k := \pns{\nab u}{\infty} + \pns{\nab^2 u}{\infty} + \sum_{i=1}^{2} \snormspace{D u_i}{2}{\Sigma}^2.
\end{equation}
Note that the regularity of $u$ will always be sufficiently high for the $L^\infty$ norms in $\k$ to be considered as $L^\infty(\bar{\Omega})$ norms, where $\bar{\Omega}$ is the closure of $\Omega$.  Finally, we define the total energy  we will use in our analysis:
\begin{equation}\label{i_total_energy}
 \g(t) = \sup_{0 \le r \le t} \se{2N}(r) + \int_0^t \sd{2N}(r) dr + \sum_{m=1}^2 \sup_{0 \le r \le t} (1+r)^{m+\lambda} \se{N+2,m}(r) 
+  \sup_{0\le r \le t} \frac{\f(r)}{(1+r)} .
\end{equation}

\subsection{Some initial estimates}

We have the following Lemma that constrains $N$.

\begin{lem}\label{i_N_constraint}
If $N \ge 4$, then for $m=1,2$ we have that  $\se{N+2,m} \ls \se{2N}$  and $\sd{N+2,m} \ls \se{2N}.$
\end{lem}
\begin{proof}
The proof follows by simply comparing the definitions of these terms.
\end{proof}

Now we present an estimate of $\i_{1} \dt \eta$.

\begin{lem}\label{i_dt_eta_h_dot}
We have the estimate $\ns{\i_{1} \dt \eta}_{0} \ls \ns{u}_{0} \le \se{2N}.$
\end{lem}
\begin{proof}
According to Lemma \ref{dt_eta}, we  have that $\dt \eta = - \p_i U_i$, where $U_i$, $i=1,2$, is defined in the lemma.  It is easy to see that $U_i \in H^1(\Sigma)$.   Taking the Fourier transform, we find that
\begin{equation}
\ns{\i_{1} \dt \eta }_{0} =
 \int_\Sigma  \abs{\xi}^{-2} \abs{ \widehat{\dt \eta}(\xi)}^2  d\xi \ls 
 \int_\Sigma  \abs{\xi}^{-2} \abs{ \xi \cdot \widehat{U}(\xi)}^2  d\xi
\ls \int_\Sigma   \abs{\widehat{U}(\xi)}^2  d\xi = \snormspace{U}{0}{\Sigma}^2.
\end{equation}
However, H\"older's inequality and Lemma \ref{infinity_bounds} imply  that $\snormspace{U}{0}{\Sigma} \ls \pnorm{J}{\infty} \norm{u}_{0} \ls  \norm{u}_{0},$ so the desired estimate follows.
\end{proof}

\section{Interpolation estimates at the $N+2$ level}\label{inf_2}

\subsection{Initial interpolation estimates for $\eta, \bar{\eta}, u$ and $\nab p$} \label{interp_sec_1}

The fact that $\se{N+2,m}$ and $\sd{N+2,m}$, $m=1,2$, have a minimal count of derivatives creates numerous problems when we try to estimate terms with fewer derivatives in terms of $\se{N+2,m}$ and $\sd{N+2,m}$.  Our way around this is to interpolate between $\se{N+2,m}$ (or $\sd{N+2,m}$) and $\se{2N}$.  In Sections \ref{interp_sec_1}--\ref{interp_sec_2} we will prove various interpolation inequalities of the form
\begin{equation}\label{interp_form}
 \ns{X} \ls (\se{N+2,m})^\theta (\se{2N})^{1-\theta} \text{ and } \ns{X} \ls (\sd{N+2,m})^\theta (\se{2N})^{1-\theta},
\end{equation}
where $\theta \in (0,1]$, $X$ is some quantity, and $\norm{\cdot}$ is some norm (usually either $H^0$ or $L^\infty$).  

In the interest of brevity, we will record these estimates in tables that only list the value of $\theta$ in the estimate.   Before each table we will tell which norms are being considered and give a rough summary of the terms $X$ that appear in the table.  For example, we might write ``the following table encodes the power in the $H^0(\Sigma)$ and $H^0(\Omega)$ interpolation estimates for $\eta$ and $\bar{\eta}$ and their derivatives,'' before the following table.
\begin{displaymath}
\begin{array}{| l | c   c c |}
\hline
X & \se{N+2,1} & \sd{N+2,1} \sim \se{N+2,2} & \sd{N+2,2} \\ \hline 
 \eta, \bar{\eta} & 
\theta_1 &   
\theta_2  & \theta_3   \\ \hline
D \eta, \nab \bar{\eta} & 
\theta_4  &   
\theta_5 & \theta_6   \\ \hline
\end{array}
\end{displaymath}
We understand this to mean that
\begin{equation}
 \ns{\eta}_{0} \ls (\se{N+2,1})^{\theta_1} (\se{2N})^{1-\theta_1},\; 
\ns{\eta}_0 \ls (\sd{N+2,1})^{\theta_2} (\se{2N})^{1-\theta_2}, \;  
\ns{\eta}_{0} \ls (\se{N+2,2})^{\theta_2} (\se{2N})^{1-\theta_2},
\end{equation}
\begin{multline}
\ns{\eta}_0 \ls (\sd{N+2,2})^{\theta_3} (\se{2N})^{1-\theta_3},\;  
\snormspace{\nab \bar{\eta}}{0}{\Omega}^2 \ls (\se{N+2,1})^{\theta_4} (\se{2N})^{1-\theta_4}, \\ \snormspace{\nab \bar{\eta}}{0}{\Omega}^2 \ls (\sd{N+2,1})^{\theta_5} (\se{2N})^{1-\theta_5}, 
\end{multline}
etc.  When we write $\sd{N+2,1} \sim \se{N+2,2}$ in a table, it means that $\theta$ is the same when interpolating between $\sd{N+2,1}$ and $\se{2N}$ and between $\se{N+2,2}$ and $\se{2N}$.  When we write multiple entries for $X$, we mean that the same interpolation estimates hold for each item listed.  Often, we will have a $\theta$ appearing in a table of the form $\theta = 1/(1+r)$.  When we write this, we mean that the desired interpolation inequality holds with this $\theta$ for any fixed $r \in (0,1)$, and the constant in the inequality then depends on $r$.

We must record estimates for too many choices of $X$ to allow us to write the full details of each estimate.  However, most of the estimates are straightforward, so in our proofs we will frequently present only a sketch of how to obtain them, providing details only for the most delicate estimates.  The terms we estimate are often linear combinations of several terms, each of which would get a different interpolation power.  When this occurs, we will record the lowest power achieved by a term in the sum.  According to  Lemma \ref{i_N_constraint}, this is justified by the estimate
\begin{multline}
 \se{2N}^{1-\theta} \se{N+2,m}^\theta + \se{2N}^{1-\kappa} \se{N+2,m}^\kappa =  \se{2N}^{1-\theta} \se{N+2,m}^\theta + \se{2N}^{1-\kappa} \se{N+2,m}^{\kappa - \theta} \se{N+2,m}^\theta \\
\ls  \se{2N}^{1-\theta} \se{N+2,m}^\theta + \se{2N}^{1-\kappa} \se{2N}^{\kappa - \theta} \se{N+2,m}^\theta \ls \se{2N}^{1-\theta} \se{N+2,m}^\theta
\end{multline}
for $0\le \theta \le  \kappa\le 1$.  A similar estimate holds with $\se{N+2,m}$ replaced by $\sd{N+2,m}$.   It may happen that in estimating a product of two or more terms, we end up with  estimates of the form
\begin{equation}\label{interp_form_prod}
 \ns{X} \ls (\se{N+2,m})^{\theta_1} (\se{2N})^{1-\theta_1}  (\se{N+2,m})^{\theta_2} (\se{2N})^{1-\theta_2}
\end{equation}
with $\theta_1+ \theta_2 > 1$.  In this case, Lemma \ref{i_N_constraint} again allows us to bound
\begin{equation}
 \ns{X} \ls (\se{N+2,m})^{1} (\se{N+2,m})^{\theta_1 + \theta_2-1} (\se{2N})^{2-\theta_1-\theta_2} \ls \se{N+2,m} \se{2N} \le \se{N+2,m},  
\end{equation}
where we have used the bound $\se{2N} \le 1$.  It might also happen that \eqref{interp_form_prod} occurs with $\theta_1 < 1$ and $\theta_2 = 1/(1+r)$, in which case we always understand that $r$ is chosen so that $\theta_1 + \theta_2 = 1$.

Now that our notation is explained, we turn to the estimates themselves   We begin with estimates of $\eta$.

\begin{lem}\label{i_interp_eta}
The following table encodes the power in the $L^\infty(\Sigma)$ and $L^\infty(\Omega)$ interpolation estimates for $\eta$ and $\bar{\eta}$ and their derivatives.

\begin{displaymath}
\begin{array}{| l | c   c c |}
\hline
X & \se{N+2,1} & \sd{N+2,1} \sim \se{N+2,2} & \sd{N+2,2} \\ \hline 
 \eta, \bar{\eta} & 
(\lambda+1)/(\lambda+1+r) &   
(\lambda+1)/(\lambda+2)  & (\lambda+1)/(\lambda+3)   \\ \hline
D \eta, \nab \bar{\eta}& 
1  &   
(\lambda+2)/(\lambda+2+r) & (\lambda+2)/(\lambda+3)   \\ \hline
D^2 \eta, \nab^2 \bar{\eta} & 
 1 &    
1 & (\lambda+3)/(\lambda+3+r)   \\ \hline
D^3 \eta, \nab^3 \bar{\eta} & 
1 &   
1 & 1   \\  \hline
\dt \eta, \dt \bar{\eta} & 
 1 &  
1 & 2/(2+r)   \\ \hline
D \dt \eta, \nab \dt \bar{\eta} & 
 1 &  
1 & 1   \\ \hline
\end{array}
\end{displaymath}

The following table encodes the power in the $H^0(\Sigma)$ and $H^0(\Omega)$ interpolation estimates for $\eta$ and $\bar{\eta}$ and their derivatives.

\begin{displaymath}
\begin{array}{| l | c   c c |}
\hline
X & \se{N+2,1} & \sd{N+2,1} \sim \se{N+2,2} & \sd{N+2,2} \\ \hline 
 \eta, \bar{\eta} & 
\lambda/(\lambda+1) &   
\lambda/(\lambda+2)  & \lambda/(\lambda+3)   \\ \hline
D \eta, \nab \bar{\eta} & 
1  &   
(\lambda+1)/(\lambda+2) & (\lambda+1)/(\lambda+3)   \\ \hline
D^2 \eta, \nab^2 \bar{\eta} & 
 1 &    
1 & (\lambda+2)/(\lambda+3)   \\ \hline
D^3 \eta, \nab^3 \bar{\eta} & 
1 &   
1 & 1   \\  \hline
\dt \eta, \dt \bar{\eta} & 
 1 &  
1 & 1/2   \\ \hline
D \dt \eta, \nab \dt \bar{\eta} & 
 1 &  
1 & 1   \\ \hline
\end{array}
\end{displaymath}
\end{lem}

\begin{proof}
The estimates  follow directly from the Sobolev embeddings and  Lemmas \ref{i_poisson_interp} and \ref{i_sigma_interp}, using the bounds $\ns{\il \eta}_{0} \le \se{2N}$ and $\ns{\i_{1} \dt \eta}_{0} \ls \se{2N}$, the latter of which is a consequence of Lemma \ref{i_dt_eta_h_dot}.  
\end{proof}

Now we record some estimates involving $u$.

\begin{lem}\label{i_interp_u}
The following table encodes the power in the $L^\infty(\Omega)$ and $L^\infty(\Sigma)$ interpolation estimates for $u$  and its derivatives.

\begin{displaymath}
\begin{array}{| l | c   c c |}
\hline
X & \se{N+2,1} & \sd{N+2,1} \sim \se{N+2,2} & \sd{N+2,2} \\ \hline 
 u & 
1/(1+r) &    
1/2 & 
1/3   \\ \hline
D u & 
 1 &  
2/(2+r) & 
2/3   \\ \hline
\nab u & 
1/(1+r) &    
1/2 & 
1/3   \\ \hline
D^2 u & 
1 &   
1 & 
1/(1+r)   \\  \hline
D \nab u & 
 1 &  
2/(2+r) & 
2/3   \\ \hline
 \nab^2 u  & 
 1 &  
1/(1+r) & 
1/2   \\ \hline
\nab^3 u  & 
 1 &  
1 & 
1/(1+r)   \\ \hline
\nab^4 u  & 
 1 &  
1 & 
1   \\ \hline
\dt u  & 
 1 &  
1 & 
1   \\ \hline
\end{array}
\end{displaymath}

The following table encodes the power in the $H^0(\Omega)$ interpolation estimates for $u$  and its derivatives.

\begin{displaymath}
\begin{array}{| l | c c  c c |}
\hline
X & \se{N+2,1} & \sd{N+2,1} & \se{N+2,2} & \sd{N+2,2} \\ \hline 
 u & 
\lambda/(\lambda+1) &   
\lambda/(\lambda+1)  & 
\lambda/(\lambda+2) &
\lambda/(\lambda+2)   \\ \hline
D u& 
1  &   
1 & 
(\lambda+1)/(\lambda+2) &
(\lambda+1)/(\lambda+2)  \\ \hline
D^2 u & 
1 &   
1 &
1 & 
1   \\  \hline
\nab D^2 u & 
1 &   
1 &
1 & 
1   \\  \hline
\dt u & 
1 &   
1 &
1 & 
1   \\  \hline
\end{array}
\end{displaymath}

The following table encodes the power in some improved $L^\infty(\Sigma)$ interpolation estimates for $u$  and its tangential derivatives on $\Sigma$.  Here we restrict to $r \in(0,1/2)$.

\begin{displaymath}
\begin{array}{| l | c c  c c |}
\hline
X & \se{N+2,1} & \sd{N+2,1} & \se{N+2,2} & \sd{N+2,2} \\ \hline 
 u & 
1/(1+r) &   
1/(1+r)  & 
1/2 &
1/2   \\ \hline
D u& 
1  &   
2/(2+r) & 
2/(2+r) &
2/(2+r)  \\ \hline
\end{array}
\end{displaymath}

\end{lem}

\begin{proof}

The estimates of the first two tables follow directly from Sobolev embeddings and Lemmas \ref{i_slice_interp} and \ref{poincare_usual}.  For the $L^\infty(\Sigma)$ estimates of the last table, we use $r \in [0,1/2)$ in  \eqref{i_sig_i_2} of Lemma \ref{i_sigma_interp} along with trace estimates and Lemma \ref{poincare_usual}  to bound
\begin{multline}
 \pnormspace{u}{\infty}{\Sigma}^2 \ls (\snormspace{u}{0}{\Sigma}^2)^{(s+r-1)/(s+r)} (\snormspace{D^s u}{r}{\Sigma})^{1/(s+r)} \\ \ls
(\ns{u}_{1/2})^{(s+r-1)/(s+r)} (\ns{D^s u}_{1})^{1/(s+r)} \\ \ls
(\ns{u}_{1/2})^{(s+r-1)/(s+r)} (\ns{D^s \nab u}_{0})^{1/(s+r)}.
\end{multline}
For $\se{N+2,1}$ and $\sd{N+2,1}$ we choose $s=1$ and $r \in (0,1/2)$, while for $\se{N+2,2}$ and $\sd{N+2,m}$ we choose $s=2$ and $r=0$.  In both cases, $\ns{u}_{1/2} \le \se{2N}$ and $\ns{D^s \nab u}_{0} \le \se{N+2,m}$.  A similar argument works for the $Du$ estimates in $L^\infty(\Sigma)$.
\end{proof}

Now we estimate $\nab p$ in $L^\infty$.

\begin{lem}\label{i_interp_p}

The following table encodes the power in the $L^\infty(\Omega)$ interpolation estimates for $\nab p$  and its derivatives.
\begin{displaymath}
\begin{array}{| l | c   c c |}
\hline
X & \se{N+2,1} & \sd{N+2,1} \sim \se{N+2,2} & \sd{N+2,2} \\ \hline 
 \nab p & 
1 &   
1/(1+r)  & 
1/2   \\ \hline
\nab^2 p & 
1  &   
1 & 
1/(1+r)   \\ \hline
\nab^3 p & 
 1 &    
1 & 
1   \\ \hline
\dt \nab p & 
1 &   
1 & 
1   \\  \hline
\end{array}
\end{displaymath}
\end{lem}
\begin{proof}
 The estimates  follow directly from the Sobolev embeddings and Lemma \ref{i_slice_interp}.
\end{proof}

\subsection{Interpolation estimates for $G^i$, $i=1,2,3,4$}

Now that we have some preliminary estimates for $u, \eta, \bar{\eta}$, and $\nab p$ (plus some of their derivatives), we can estimate the $G^i$ forcing terms defined in \eqref{Gi_def_start}--\eqref{Gi_def_end}.

\begin{lem}\label{i_interp_G1}
The following table encodes the power in the $L^\infty(\Omega)$ interpolation estimates for $G^{1,i}$, $i=1,\dotsc,5$ and $G^1$ and their spatial derivatives.
\begin{displaymath}
\begin{array}{| l | c   c c |}
\hline
X & \se{N+2,1} & \sd{N+2,1} \sim \se{N+2,2} & \sd{N+2,2} \\ \hline 
 G^{1,1} & 
1 &   
1  & 
(3\lambda+5)/(2\lambda+6)   \\ \hline
\nab G^{1,1}& 
1  &   
1 &
1  \\ \hline
G^{1,2} & 
1 &   
1 & 
2/3   \\  \hline
D G^{1,2} & 
1 &   
1 & 
1   \\  \hline
\nab G^{1,2} & 
1 &   
1 &
2/3   \\  \hline
G^{1,3} & 
1 &   
1 & 
(3\lambda +5)/(2\lambda+6)   \\  \hline
\nab G^{1,3} & 
1 &   
1 &
1   \\  \hline
G^{1,4} & 
1 &   
1 &
1   \\  \hline
\nab G^{1,4} & 
1 &   
1 &
1   \\  \hline
G^{1,5} & 
1 &   
1 &
1   \\  \hline
\nab G^{1,5} & 
1 &   
1 &
1   \\  \hline
G^{1} & 
1 &   
1 &
2/3   \\  \hline
D G^{1} & 
1 &   
1 &
1   \\  \hline
\nab G^{1} & 
1 &   
1 &
2/3   \\  \hline
\end{array}
\end{displaymath}

The following table encodes the power in the $H^0(\Omega)$ interpolation estimates for $G^{1,i}$, $i=1,\dotsc,5$ and $G^1$ and their spatial derivatives.
\begin{displaymath}
\begin{array}{| l | c c  c c |}
\hline
X & \se{N+2,1} & \sd{N+2,1} & \se{N+2,2} & \sd{N+2,2} \\ \hline 
 G^{1,1} & 
1 &   
1  & 
1 &
(3\lambda+3)/(2\lambda+6)   \\ \hline
\nab G^{1,1}& 
1  &   
1 & 
1 &
(3\lambda+5)/(2\lambda+6)  \\ \hline
G^{1,2} & 
1 &   
(3\lambda+1)/(2\lambda+2) & 
(3\lambda+2)/(2\lambda+4) &
(5\lambda+2)/(4\lambda+8)   \\  \hline
D G^{1,2} & 
1 &   
1 &
1 & 
(5\lambda + 4)/(3 \lambda+6)   \\  \hline
G^{1,3} & 
1 &   
1 &
1 & 
(3\lambda+3)/(2\lambda+6)   \\  \hline
\nab G^{1,3} & 
1 &   
1 &
1 & 
(3\lambda+5)/(2\lambda+6)   \\  \hline
G^{1,4} & 
1 &   
1 &
1 & 
(4\lambda+6)/(3\lambda+9)   \\  \hline
D G^{1,4} & 
1 &   
1 &
1 & 
1   \\  \hline
G^{1,5} & 
1 &   
1 &
1 & 
5/6   \\  \hline
\nab G^{1,5} & 
1 &   
1 &
1 & 
1   \\  \hline
G^{1} & 
1 &   
(3\lambda+1)/(2\lambda+2) &
(3\lambda+2)/(2\lambda+4) & 
(5\lambda+2)/(4\lambda+8)   \\  \hline
D G^{1} & 
1 &   
1 &
1 & 
(5\lambda + 4)/(3\lambda+6)   \\  \hline
\end{array}
\end{displaymath}
\end{lem}
\begin{proof}

The definitions of $G^{1,i}$ show that these terms are linear combinations of products of one or more terms that can be estimated in either $L^\infty$ or $H^0$ by using Sobolev embeddings and Lemmas \ref{i_interp_eta}, \ref{i_interp_u}, and \ref{i_interp_p}.  For the $L^\infty$ table we estimate products using the usual algebra of $L^\infty$: $\pnorm{X Y}{\infty} \le \pnorm{X}{\infty} \pnorm{Y}{\infty}$.  For the $H^0$ table, we estimate products with both 
\begin{equation}
 \ns{ XY}_{0} \le \ns{X}_{0} \pnorm{Y}{\infty} \text{ and } \ns{ XY}_{0} \le \ns{Y}_{0} \pnorm{X}{\infty},
\end{equation}
and then take the larger value of $\theta$ produced by these two bounds.

\end{proof}

Now we estimate $G^2$.  The proof works as in Lemma \ref{i_interp_G1}, so we omit it.

\begin{lem}\label{i_interp_G2}
The following table encodes the power in the $L^\infty(\Omega)$ and $L^\infty(\Sigma)$ interpolation estimates for $G^{2}$ and its spatial derivatives.
\begin{displaymath}
\begin{array}{| l | c   c c |}
\hline
X & \se{N+2,1} & \sd{N+2,1} \sim \se{N+2,2} & \sd{N+2,2} \\ \hline 
 G^{2} & 
1 &   
1  & 
(4\lambda+6)/(3\lambda+9)   \\ \hline
D G^{2}& 
1  &   
1 &
1  \\ \hline
\nab G^{2}& 
1  &   
1 &
(3\lambda+5)/(2\lambda+6)  \\ \hline
\nab^2 G^{2} & 
1 &   
1 & 
1   \\  \hline
\end{array}
\end{displaymath}

The following table encodes the power in the $H^0(\Omega)$ interpolation estimates for $G^{2}$ and its spatial derivatives.
\begin{displaymath}
\begin{array}{| l | c   c c |}
\hline
X & \se{N+2,1} & \sd{N+2,1} \sim \se{N+2,2} & \sd{N+2,2} \\ \hline 
 G^{2} & 
1 &   
(3\lambda+2)/(2\lambda+4)  & 
(4\lambda+3)/(3\lambda+9)   \\ \hline
D G^{2}& 
1  &   
1 &
(4\lambda+6)/(3\lambda+9)  \\ \hline
\nab G^{2}& 
1  &   
1 &
(3\lambda+3)/(2\lambda+6)  \\ \hline
\nab^2 G^{2} & 
1 &   
1 & 
(3\lambda+5)/(2\lambda+6)   \\  \hline
\end{array}
\end{displaymath}
\end{lem}

Now we record $G^3$ estimates.  Recall that by Remark \ref{G3_remark},  we may remove the appearance of $(p-\eta)$ in $G^3$.  This allows us to perform the estimates of $G^3$ terms as in  Lemmas \ref{i_interp_G1} and \ref{i_interp_G2}, so we again omit  the proof.

\begin{lem}\label{i_interp_G3}
 The following table encodes the power in the $L^\infty(\Sigma)$ interpolation estimates for $G^{3}$ and its spatial derivatives.
\begin{displaymath}
\begin{array}{| l | c   c c |}
\hline
X & \se{N+2,1} & \sd{N+2,1} \sim \se{N+2,2} & \sd{N+2,2} \\ \hline 
 G^{3} & 
1 &  
1  & 
(4\lambda+6)/(3\lambda+9)   \\ \hline
D G^{3}& 
1  &   
1 &
1  \\ \hline
D^2 G^{3} & 
1 &   
1 & 
1   \\  \hline
\end{array}
\end{displaymath}

The following table encodes the power in the $H^0(\Sigma)$ interpolation estimates for $G^{3}$ and its spatial derivatives.
\begin{displaymath}
\begin{array}{| l | c   c c |}
\hline
X & \se{N+2,1} & \sd{N+2,1} \sim \se{N+2,2} & \sd{N+2,2} \\ \hline 
 G^{3} & 
1 &   
(3\lambda+2)/(2\lambda+4)  & 
(4\lambda+3)/(3\lambda+9)   \\ \hline
D G^{3}& 
1  &   
1 &
(4\lambda+6)/(3\lambda+9)  \\ \hline
D^2 G^{3} & 
1 &   
1 & 
1   \\  \hline
\end{array}
\end{displaymath}
\end{lem}

Now for $G^4$ estimates.  We again omit the proof.

\begin{lem}\label{i_interp_G4}
 The following table encodes the power in the $L^\infty(\Sigma)$ interpolation estimates for $G^{4}$ and its spatial derivatives.
\begin{displaymath}
\begin{array}{| l | c   c c |}
\hline
X & \se{N+2,1} & \sd{N+2,1} \sim \se{N+2,2} & \sd{N+2,2} \\ \hline 
 G^{4} & 
1 &  
1  & 
1   \\ \hline
D G^{4}& 
1  &   
1 &
1  \\ \hline
D^2 G^{4}& 
1  &   
1 &
1  \\ \hline
\end{array}
\end{displaymath}

The following table encodes the power in the $H^0(\Sigma)$ interpolation estimates for $G^{4}$ and its spatial derivatives.
\begin{displaymath}
\begin{array}{| l | c  c  c |}
\hline
X & \se{N+2,1} & \sd{N+2,1} \sim \se{N+2,2} & \sd{N+2,2} \\ \hline 
 G^{4} & 
1 &   
1  & 
(3\lambda+5)/(2\lambda+6)   \\ \hline
D G^{4}& 
1  &   
1 &
1  \\ \hline
D^2 G^{4}& 
1  &   
1 &
1  \\ \hline
\end{array}
\end{displaymath}
\end{lem}

\subsection{Improved estimates for $u, \nab p$}

Now we will use the structure of the equations \eqref{linear_perturbed} to improve our estimates for $u, \nab p$, etc.  Our first estimate is for $D p$.  It constitutes an improvement of our existing $L^\infty$ estimate, Lemma \ref{i_interp_p}, as well as a first $H^0$ estimate.

\begin{lem}\label{i_improved_p}

The following table encodes the power in an $L^{\infty}(\Omega)$ interpolation estimate.
\begin{displaymath}
\begin{array}{| l | c   c c |}
\hline
 & \se{N+2,1} & \sd{N+2,1} \sim \se{N+2,2} & \sd{N+2,2} \\ \hline 
 Dp  & 
1 &  
1/(1+r)  & 
(\lambda+2)/(\lambda+3)   \\ \hline
\end{array}
\end{displaymath}

The following table encodes the power in an $H^0(\Omega)$ interpolation estimate.
\begin{displaymath}
\begin{array}{| l | c   c c |}
\hline
 & \se{N+2,1} & \sd{N+2,1} \sim \se{N+2,2} & \sd{N+2,2} \\ \hline 
Dp & 
1  &   
(\lambda+1)/(\lambda+2) &
(\lambda+1)/(\lambda+3)  \\ \hline
\end{array}
\end{displaymath}

\end{lem}

\begin{proof}
In order to record the proof of both the $H^0$ and $L^\infty$ estimates at the same time, we will generically write $\norm{\cdot}$ to refer to either the $H^0(\Omega)$ or $L^\infty(\Omega)$ norm. Similarly, we will write $\norm{\cdot}_{\Sigma}$ to refer to the $H^0(\Sigma)$ or $L^\infty(\Sigma)$ norm.  The starting point is an application of  Lemma \ref{poincare_b} to bound
\begin{equation}\label{i_imp_p_1}
 \ns{D p} \ls \ns{D p}_{\Sigma} + \ns{ \p_3 Dp}.
\end{equation}
We will estimate both of the terms on the right hand side in order to prove the lemma.

In order to estimate $Dp$ on $\Sigma$ we utilize the boundary conditions in \eqref{linear_perturbed} to write
\begin{equation}
 \p_i p = \p_i \eta + 2 \p_i \p_3 u_3 + \p_i(G^3 \cdot e_3)
\end{equation}
for $i=1,2$.  From this we easily see that 
\begin{equation}\label{i_imp_p_2}
 \ns{D p}_{\Sigma} \ls \ns{ D \eta }_{\Sigma} +  \ns{D G^3}_{\Sigma} + \ns{D \p_3 u_3}_{\Sigma}.
\end{equation}
The first two terms may be estimated with Lemmas \ref{i_interp_eta} and \ref{i_interp_G3}, but we must further exploit the structure of the equations in order to control the last term.  For the $H^0$ estimate we use trace theory and the relation 
\begin{equation}\label{i_imp_p_3}
\p_3 u_3 = G^2 - \p_1 u_1 - \p_2 u_2 
\end{equation}
to find
\begin{equation}
 \ns{D \p_3 u_3}_{H^0(\Sigma)} \ls \ns{D \p_3 u_3}_{1} \ls \ns{D G^2}_{1} + \ns{D^2 u}_{1}.
\end{equation}
Since $D^2 u =0$ on $\Sigma_b$ we may use Poincar\'e,  Lemma \ref{poincare_usual}, to bound $\ns{D^2 u}_{1} \ls \ns{\nab D^2 u}_{0}$, so that upon replacing in the previous inequality we find 
\begin{equation}\label{i_imp_p_4}
 \ns{D \p_3 u_3}_{H^0(\Sigma)} \ls  \ns{D G^2}_{0}+ \ns{D\nab G^2}_{0} + \ns{D^2 \nab  u}_{0}.
\end{equation}
For the corresponding $L^\infty$ estimate we again use \eqref{i_imp_p_3}  to bound
\begin{equation}
\ns{ D \p_3 u_3}_{L^\infty(\Sigma)} \ls  \ns{ D G^2}_{L^\infty(\Sigma)} + \ns{ D^2 u}_{L^\infty(\Sigma)}.
\end{equation}
By Lemma \ref{poincare_usual} we know that
$ \ns{ D^2 u}_{L^\infty(\Sigma)} \ls \ns{ \nab D^2 u}_{L^\infty(\Omega)}$, and also  Lemma \ref{i_interp_G2} guarantees that $ \ns{ D G^2}_{L^\infty(\Sigma)}  \ls \ns{ D G^2}_{L^\infty(\Omega)} $, so we may replace these to arrive at the bound
\begin{equation}\label{i_imp_p_5}
\ns{ D \p_3 u_3}_{L^\infty(\Sigma)} \ls  \ns{ D G^2}_{L^\infty(\Omega)} + \ns{ \nab D^2 u}_{L^\infty(\Omega)}.
\end{equation}
Then from \eqref{i_imp_p_4} and \eqref{i_imp_p_5} we know that
\begin{equation}\label{i_imp_p_6}
  \ns{D \p_3 u_3}_{\Sigma} \ls  \ns{D G^2} + \ns{D\nab G^2} + \ns{D^2 \nab  u}.
\end{equation}
Combining \eqref{i_imp_p_2} with \eqref{i_imp_p_6} yields
\begin{equation}
  \ns{D p}_{\Sigma} \ls \ns{ D \eta }_{\Sigma} +  \ns{D G^3}_{\Sigma} + \ns{D G^2} + \ns{D\nab G^2} + \ns{D^2 \nab  u}.
\end{equation}
We may then employ Lemmas \ref{i_interp_eta}, \ref{i_interp_u}, \ref{i_interp_p}, \ref{i_interp_G2}, \ref{i_interp_G3} to derive the interpolation power for $ \ns{D p}_{\Sigma}$; we record this power in the following table. Both the $L^\infty$ and $H^0$ powers are determined by $D \eta$, but the $L^\infty$ estimate only improves the result of Lemma \ref{i_interp_p} for $\sd{N+2,2}$.
\begin{displaymath}
\begin{array}{| l | c   c c |}
\hline
 & \se{N+2,1} & \sd{N+2,1} \sim \se{N+2,2} & \sd{N+2,2} \\ \hline 
 \ns{D p}_{L^\infty(\Sigma)}  & 
1 &  
1/(1+r)  & 
(\lambda+2)/(\lambda+3)   \\ \hline
\ns{D p}_{H^0(\Sigma)} & 
1  &   
(\lambda+1)/(\lambda+2) &
(\lambda+1)/(\lambda+3)  \\ \hline
\end{array}
\end{displaymath}

Now we will estimate the term $\ns{\p_3 D p}$.  For this we use \eqref{linear_perturbed} to write
\begin{equation}
 \p_i \p_3 p = \p_i [ (\p_1^2 + \p_2^2 - \dt)u_3 +  \p_3^2 u_3 + G^1 \cdot e_3].
\end{equation}
for $i=1,2$.  Again using \eqref{i_imp_p_3}, we may write
\begin{equation}
 \p_i \p_3^2 u_3 = \p_i \p_3 ( G^2 - \p_1 u_1 - \p_2 u_2).
\end{equation}
Combining these two equations then shows that
\begin{equation}
 \ns{D \p_3 p} \ls \ns{D^3 u} + \ns{D^2 \nab u} + \ns{D \dt u} + \ns{D G^1}  + \ns{D \nab G^2}.
\end{equation}
We may then employ Lemmas  \ref{i_interp_u}, \ref{i_interp_p}, \ref{i_interp_G1},  and \ref{i_interp_G2} to derive the interpolation power for $ \ns{D \p_3 p}$; we record this power in the following table. The $H^0$ powers are determined by $DG^1$, but note that the $L^\infty$ estimate  does not improve the result of Lemma \ref{i_interp_p}.
\begin{displaymath}
\begin{array}{| l | c   c c |}
\hline
 & \se{N+2,1} & \sd{N+2,1} \sim \se{N+2,2} & \sd{N+2,2} \\ \hline 
 \ns{D \p_3 p}_{L^\infty}  & 
1 &  
1  & 
1/(1+r)   \\ \hline
\ns{D \p_3 p}_{0} & 
1  &   
1 &
(5\lambda+4)/(3\lambda+6)  \\ \hline
\end{array}
\end{displaymath}

Now we return to \eqref{i_imp_p_1} and employ our estimates of $\ns{D p}_{\Sigma}$ and $\ns{D \p_3 p}$ to deduce the desired interpolation powers for $\ns{Dp}$.

\end{proof}

With this lemma in hand, we can now derive improved estimates for $u$.

\begin{prop}\label{i_improved_u}
Let
\begin{equation}
 \theta_1(\lambda) =  \min\left\{\frac{ 5 \lambda + 2}{4\lambda + 8}, \frac{\lambda+1}{\lambda + 3} \right\} \text{ and }
\theta_2(\lambda) = \min\left\{\frac{9 \lambda + 10}{8\lambda + 16}, \frac{\lambda+2}{\lambda + 3} \right\}.
\end{equation}

The following table encodes the improved power in the $L^\infty(\Omega)$ interpolation estimate for $u$ and its derivatives.
\begin{displaymath}
\begin{array}{| l | c   c c |}
\hline
 & \se{N+2,1} & \sd{N+2,1} \sim \se{N+2,2} & \sd{N+2,2} \\ \hline 
 u  & 
1 &  
2/(2+r)  & 
2/3   \\ \hline
\p_3 u_i, i=1,2 & 
1 &  
1  & 
2/3   \\ \hline
\p_3 u_3 & 
1 &  
2/(2+r)  & 
2/3   \\ \hline
\nab u& 
1 &  
2/(2+r)  & 
2/3   \\ \hline
\nab^2 u& 
1  &   
2/(2+r) &
2/3  \\ \hline
\end{array}
\end{displaymath} 

The following table encodes the power in the $H^0(\Omega)$ interpolation estimate for $u$ and its derivatives.  
\begin{displaymath}
\begin{array}{| l | c c  c c |}
\hline
 & \se{N+2,1} & \sd{N+2,1} & \se{N+2,2} & \sd{N+2,2} \\ \hline 
 u  & 
1 &  
(\lambda+1)/(\lambda+2) &
(\lambda+1)/(\lambda+2)  & 
\theta_1(\lambda)    \\ \hline
\p_3 u_i, i=1,2 & 
1 &  
(\lambda+1)/(\lambda+2)  & 
(\lambda+1)/(\lambda+2) &
\theta_1(\lambda)   \\ \hline
\p_3 u_3 & 
1 &  
(3\lambda+2)/(2\lambda+4)  &
(3\lambda+2)/(2\lambda+4) & 
(4\lambda+3)/(3\lambda+9)   \\ \hline
D u& 
1 &  
1 &
(2\lambda+3)/(2\lambda+4)  & 
 \theta_2(\lambda)   \\ \hline
\nab u& 
1 &  
(\lambda+1)/(\lambda+2) &
(\lambda+1)/(\lambda+2)  & 
\theta_1(\lambda)   \\ \hline
 D \nab u  & 
1 &  
1 &
(2\lambda+3)/(2\lambda+4)  & 
 \theta_2(\lambda)   \\ \hline
D \p_3 u_3& 
1 &  
1 &
1  & 
(4\lambda+6)/(3\lambda+9)    \\ \hline
\nab \p_3 u_3& 
1 &  
1 &
(2\lambda+3)/(2\lambda+4)  & 
(3\lambda+3)/(2\lambda+6)    \\ \hline
\nab^2 u& 
1 &  
(\lambda+1)/(\lambda+2) &
(\lambda+1)/(\lambda+2)  & 
\theta_1(\lambda)    \\ \hline
\end{array}
\end{displaymath}

The following table encodes the improved power in the $L^\infty(\Omega)$ interpolation estimate for $\nab p$.  
\begin{displaymath}
\begin{array}{| l | c   c c |}
\hline
 & \se{N+2,1} & \sd{N+2,1} \sim \se{N+2,2} & \sd{N+2,2} \\ \hline 
\nab p& 
1 &  
2/(2+r) &
2/3   \\ \hline
\end{array}
\end{displaymath}

The following table encodes the power in the $H^0(\Omega)$ interpolation estimate for derivatives of $p$.  
\begin{displaymath}
\begin{array}{| l | c c  c c |}
\hline
 & \se{N+2,1} & \sd{N+2,1} & \se{N+2,2} & \sd{N+2,2} \\ \hline 
 \p_3 p  & 
1 &  
(3\lambda+1)/(2\lambda+2) &
(3\lambda+2)/(2\lambda+4)  & 
(5\lambda+2)/(4\lambda+8)    \\ \hline
\nab p& 
1 &  
(\lambda+1)/(\lambda+2) &
(\lambda+1)/(\lambda+2)  & 
\theta_1(\lambda)   \\ \hline
\end{array}
\end{displaymath} 

\end{prop}
\begin{proof}

As in Lemma \ref{i_improved_p} we will write $\norm{\cdot}$ and $\norm{\cdot}_{\Sigma}$ to refer to both the $H^0$ and $L^\infty$ norms on $\Omega$ and $\Sigma$ respectively.  We divide the proof into several steps, beginning with estimates of $\nab u$.  With these established, we can extend to estimates of $u$, $D \nab u$,  $D u$, $D \p_3 u_3$,  and $\nab \p_3 u_3$ by employing Poincar\'e's inequality and interpolation.  This in turn leads to estimates for $\p_3 p$ and $\nab^2 u$.

Step 1 --  Estimates of $\nab u$

To begin the $\nab u$ estimates, we split the components of $\nab u$ into those involving $x_1,x_2$ derivatives and those involving $x_3$ derivatives.  Indeed, we have
\begin{equation}\label{i_imp_u_1}
 \ns{\nab u} \ls \ns{D u} + \ns{\p_3 u_3} + \sum_{i=1}^2 \ns{\p_3 u_i}.
\end{equation}
Lemma \ref{i_interp_u} provides an estimate of $Du$  but not of $\p_3 u$, so we must use the structure of the equations \eqref{linear_perturbed} to estimate the latter two terms.  

To estimate $\p_3 u_3$ we use equation \eqref{linear_perturbed} to bound
\begin{equation}
 \ns{ \p_3 u_3} \ls \ns{ G^2} + \ns{D u}.
\end{equation}
Then Lemmas \ref{i_interp_u} and \ref{i_interp_G2} provide  interpolation estimates of $G^2$ and $D u$ and hence the estimates of $\p_3 u_3$  listed in the tables.  The $D u$ term determines the power for $L^\infty$, while  the power is determined by $G^2$ for $H^0$.

To estimate $\p_3 u_i$ for $i=1,2$ we first apply Lemma \ref{poincare_b} to get
\begin{equation}\label{i_imp_u_2}
 \ns{\p_3 u_i} \ls \ns{\p_3 u_i}_{\Sigma} + \ns{\p_{3}^2 u_i}.
\end{equation}
For the first term on the right we use equation \eqref{linear_perturbed} to bound
\begin{equation}
 \ns{\p_3 u_i}_{\Sigma} \ls \ns{D u_3}_{\Sigma} + \ns{ G^3}_{\Sigma}.
\end{equation}
Since $Du=0$ on $\Sigma_b$ we can use trace theory, Lemma \ref{poincare_usual},  and the equation $\diverge u = G^2$ for
\begin{equation}
 \ns{D u_3}_{\Sigma} \ls \ns{\nab D u_3} \ls  \ns{D^2 u} + \ns{D G^2} 
\end{equation}
For the second term on the right side of \eqref{i_imp_u_2} we use \eqref{linear_perturbed} to bound
\begin{equation}\label{i_imp_u_3}
 \ns{ \p_3^2 u_i} \ls \ns{\dt u} + \ns{D^2 u} + \ns{D p} + \ns{G^1}.
\end{equation}
We may then combine estimates \eqref{i_imp_u_2}--\eqref{i_imp_u_3} to deduce that
\begin{equation}
\ns{\p_3 u_i} \ls \ns{\dt u} + \ns{D^2 u} + \ns{D p} +  \ns{G^1} + \ns{D G^2} + \ns{ G^3}_{\Sigma}.
\end{equation}
Now we use Lemma \ref{i_interp_u}, \ref{i_interp_G1}--\ref{i_interp_G3}, and \ref{i_improved_p} to find the interpolation powers for $\p_3 u_i, i=1,2$ listed in the tables.  For $L^\infty$ the power is determined by $G^1$, while for $H^0$ the power is determined by $Dp$ for $\se{N+2,1}, \se{N+2,2},$ and $\sd{N+2,1}$ but by the smaller of the powers of $Dp$ and $G^1$ for $\sd{N+2,2}$.

With estimates for $Du$, $\p_3 u_3$, and $\p_3 u_i$ for $i=1,2$ in hand, we return to \eqref{i_imp_u_1} to derive the estimates for $\nab u$ listed in the tables.  For the $L^\infty$ estimate the power is determined by $Du$, while for $H^0$ it is determined by $\p_3 u_i$, $i=1,2$.

Step 2 -- Extensions to estimates of $u$, $D \nab u$, $D \p_3 u_3$, and $\nab \p_3 u_3$

Now we apply Lemma \ref{poincare_usual} to control $u$ in terms of $\nab u$:
\begin{equation}
 \ns{u} \ls \ns{\nab u}.
\end{equation}
Our estimates for $\nab u$ then provide the estimates for $u$ listed in the tables.  

We now turn to $D \nab u$.  Clearly $\ns{D \nab u}_{0}$ is conrolled by both $\se{N+2,1}$ and $\sd{N+2,1}$, which yields the powers of $1$ in the tables.  An application of \eqref{i_sl_i_2} from Lemma \ref{i_slice_interp} with $\lambda =0$, $q=1,$ and $s=1$ shows that
\begin{equation}
 \ns{D \nab u}_{0} \ls \left( \ns{ \nab u}_{0} \right)^{1/2} \left( \ns{D^2 \nab u}_{0} \right)^{1/2}.
\end{equation}
We  employ this in conjunction with our estimate for $\nab u$ and the estimate of $D^2 \nab u$ from Lemma \ref{i_interp_u} to get the interpolation powers for $D\nab u$ listed in the tables for $\se{N+2,2}$ and $\sd{N+2,2}$.  The estimates for $Du$ listed in the tables follow immediately from the estimates for $D \nab u$ via Poincar\'e:
\begin{equation}
 \ns{D u } \ls \ns{D \nab u}.
\end{equation}

In order to estimate $D \p_3 u_3$ and $\nab \p_3 u_3$ in $H^0$ we use that $\diverge{u} = G^2$ for
\begin{equation}\label{i_imp_u_4}
\ns{\nab \p_3 u_3}_0 \ls \ns{\nab G^2}_0 + \ns{D \nab u}_0.
\end{equation}
and
\begin{equation}
\ns{D \p_3 u_3}_0 \ls \ns{D G^2}_0 + \ns{D^2 u}_0.
\end{equation}
Then our estimate for $D \nab u$ and Lemmas \ref{i_interp_u} and \ref{i_interp_G2} yield the estimates listed in the tables.  For $\nab \p_3 u_3$  the power is determined by $D \nab u$ for $\se{N+2,1},\sd{N+2,1}, \se{N+2,2}$ and by $\nab G^2$ for $\sd{N+2,2}$.  For $D \p_3 u_3$ the power is determined by $DG^2$. 

Step 3 -- Estimates of $\p_3 p$ and $\nab p$

Lemma \ref{i_improved_p} provides estimates for $Dp$, so to complete an estimate for $\nab p$ we only need to consider $\p_3 p$.  For this we again use \eqref{linear_perturbed} to bound
\begin{equation}
 \ns{\p_3 p} \ls  \ns{\p_3^2 u_3} + \ns{D^2 u} +  \ns{\dt u} + \ns{G^1}.
\end{equation}
This and \eqref{i_imp_u_4} then imply that
\begin{equation}
 \ns{\p_3 p} \ls   \ns{D \nab u} + \ns{D^2 u} +  \ns{\dt u} + \ns{G^1} + \ns{\nab G^2},
\end{equation}
and we may use Lemmas \ref{i_interp_u}, \ref{i_interp_G1}, and \ref{i_interp_G2} along with our new $D \nab u$ estimate to determine the powers in the tables for $\p_3 p$.  In  the $L^\infty$ estimate the power is determined by $D\nab u$,  and in the $H^0$ estimate the power is determined by $G^1$.  Then the estimates for $\nab p$ follow by comparing the $D p$ estimates of Lemma \ref{i_improved_p} to the $\p_3 p$ estimates.

Step 4 -- Estimates of $\nab^2 u$

Finally we consider $\nab^2 u$, which we decompose according to $x_1,x_2$ and $x_3$ derivatives:
\begin{equation}
 \ns{\nab^2 u} \ls \ns{D^2 u} + \ns{D \nab u} + \ns{\p_3^2 u_3} + \sum_{i=1}^2 \ns{\p_3^2 u_i}.
\end{equation}
According to our bounds \eqref{i_imp_u_3} and \eqref{i_imp_u_4} we may replace this with
\begin{equation}
 \ns{\nab^2 u} \ls \ns{\dt u} + \ns{D^2 u} + \ns{D \nab u}  + \ns{D p} + \ns{G^1} + \ns{\nab G^2}.
\end{equation}
Then Lemmas \ref{i_interp_u}, \ref{i_interp_G1}, \ref{i_interp_G2}, and \ref{i_improved_p} with our new estimate of $D \nab u$ provide the estimates in the table for $\nab^2 u$.  The power in the $L^\infty$ estimate is determined by $D \nab u$, while for $H^0$ it is determined by $Dp$ for $\se{N+2,1}, \se{N+2,2},$ and $\sd{N+2,1}$ but by the smaller of the powers of $Dp$ and $G^1$ for $\sd{N+2,2}$.  
 
\end{proof}

\subsection{Bootstrapping: first iteration}

We now use the improved estimates of Proposition \ref{i_improved_u} to improve the estimates of
$G^i$, $i=1,\dots,4$ recorded in Lemmas \ref{i_interp_G1}--\ref{i_interp_G4}.  We will only record the improvements for the $H^0(\Omega)$ estimates.  

\begin{lem}\label{i_bs_G1_1}
The following table encodes the power in the $H^0(\Omega)$ interpolation estimates for $G^{1,i}$, $i=1,\dotsc,5$ and $G^1$ and their spatial derivatives.

\begin{displaymath}
\begin{array}{| l | c c  c c |}
\hline
X & \se{N+2,1} & \sd{N+2,1} & \se{N+2,2} & \sd{N+2,2} \\ \hline 
 G^{1,1} & 
1 &   
1  & 
1 &
(5\lambda+6)/(3\lambda+9)   \\ \hline
\nab G^{1,1}& 
1  &   
1 & 
1 &
1  \\ \hline
G^{1,2} & 
1 &   
1 & 
1 &
(23 \lambda + 22)/(12 \lambda + 24)   \\  \hline
\nab G^{1,2} & 
1 &   
1 &
1 & 
(23 \lambda + 22)/(12 \lambda + 24)   \\  \hline
G^{1,3} & 
1 &   
1  & 
1 &
(5\lambda+6)/(3\lambda+9)   \\ \hline
\nab G^{1,3} & 
1 &   
1 &
1 & 
1   \\  \hline
G^{1,4} & 
1 &   
1 &
1 & 
1   \\  \hline
\nab G^{1,4} & 
1 &   
1 &
1 & 
1   \\  \hline
G^{1,5} & 
1 &   
1 &
1 & 
1   \\  \hline
\nab G^{1,5} & 
1 &   
1 &
1 & 
1   \\  \hline
G^{1} & 
1 &   
1 &
1 & 
(5\lambda+6)/(3\lambda+9)   \\  \hline
\nab G^{1} & 
1 &   
1 &
1 & 
(23\lambda + 22)/(12\lambda+24)   \\  \hline
\end{array}
\end{displaymath}

\end{lem}
\begin{proof}

We perform the estimates as in Lemma \ref{i_interp_G1}, except that now we use the improved interpolation estimates of Lemma \ref{i_improved_p} and Proposition \ref{i_improved_u}.

\end{proof}

Now for $G^2$ estimates.  We omit the proof.

\begin{lem}\label{i_bs_G2_1}
The following table encodes the power in the $H^0(\Omega)$ interpolation estimates for $G^{2}$ and its spatial derivatives.
\begin{displaymath}
\begin{array}{| l | c  c c c |}
\hline
X & \se{N+2,1} & \sd{N+2,1} & \se{N+2,2} & \sd{N+2,2} \\ \hline 
 G^{2} & 
1 &   
1  &
1  & 
(7\lambda+6)/(3\lambda+9)   \\ \hline
D G^{2}& 
1  &   
1  &
1 &
1  \\ \hline
\nab G^{2}& 
1  &   
1  &
1 &
(5\lambda+5)/(2\lambda+6)  \\ \hline
\nab^2 G^{2} & 
1 &   
1  &
1 & 
1  \\  \hline
\end{array}
\end{displaymath}
\end{lem}

Now for $G^3$ estimates.  Again we omit the proof.

\begin{lem}\label{i_bs_G3_1}

The following table encodes the power in the $H^0(\Sigma)$ interpolation estimates for $G^{3}$ and its spatial derivatives.
\begin{displaymath}
\begin{array}{| l | c c  c c |}
\hline
X & \se{N+2,1} & \sd{N+2,1} & \se{N+2,2} & \sd{N+2,2} \\ \hline 
 G^{3} & 
1 &   
1  &
1 & 
(5\lambda+6)/(3\lambda+9)   \\ \hline
D G^{3}& 
1  &   
1  &
1 &
(5\lambda+6)/(3\lambda+9)  \\ \hline
D^2 G^{3} & 
1 &   
1  &
1 & 
1   \\  \hline
\end{array}
\end{displaymath}
\end{lem}

Now for $G^4$ estimates.  The proof is omitted.

\begin{lem}\label{i_bs_G4_1}
 The following table encodes the power in the $H^0(\Sigma)$ interpolation estimates for $G^{4}$ and its spatial derivatives.
\begin{displaymath}
\begin{array}{| l | c  c c c |}
\hline
X & \se{N+2,1} & \sd{N+2,1} & \se{N+2,2} & \sd{N+2,2} \\ \hline 
 G^{4} & 
1 &   
1  &
1  & 
1   \\ \hline
D G^{4}& 
1  &   
1 &
1 &
1  \\ \hline
D^2 G^{4}& 
1  &   
1 &
1 &
1  \\ \hline
\end{array}
\end{displaymath}
\end{lem}

The improved estimates for $G^i$, $i=1,\dotsc,4$ now allow us to improve the $H^0$ estimates for $u$ and its derivatives in Proposition \ref{i_improved_u}.

\begin{thm}\label{i_bs_u}
 The following table encodes the power in the $H^0(\Omega)$ interpolation estimate for $u$ and its derivatives.  
\begin{displaymath}
\begin{array}{| l | c c  c c |}
\hline
 & \se{N+2,1} & \sd{N+2,1} & \se{N+2,2} & \sd{N+2,2} \\ \hline 
 u  & 
1 &  
(\lambda+1)/(\lambda+2) &
(\lambda+1)/(\lambda+2)  & 
(\lambda+1)/(\lambda+3)    \\ \hline
\p_3 u_3 & 
1 &  
1  &
(2\lambda+3)/(2\lambda+4) & 
(\lambda+2)/(\lambda+3)   \\ \hline
D u& 
1 &  
1 &
(2\lambda+3)/(2\lambda+4)  & 
(\lambda+2)/(\lambda+3)  \\ \hline
\nab u& 
1 &  
(\lambda+1)/(\lambda+2) &
(\lambda+1)/(\lambda+2)  & 
(\lambda+1)/(\lambda+3)    \\ \hline
 D \nab u  & 
1 &  
1 &
(2\lambda+3)/(2\lambda+4)  & 
(\lambda+2)/(\lambda+3)   \\ \hline
\nab \p_3 u_3& 
1 &  
1  &
(2\lambda+3)/(2\lambda+4) & 
(\lambda+2)/(\lambda+3)   \\ \hline
\nab^2 u& 
1 &  
(\lambda+1)/(\lambda+2) &
(\lambda+1)/(\lambda+2)  & 
(\lambda+1)/(\lambda+3)    \\ \hline
\end{array}
\end{displaymath} 

The following table encodes the power in the $H^0(\Omega)$ interpolation estimate for derivatives of $p$.  
\begin{displaymath}
\begin{array}{| l | c c  c c |}
\hline
 & \se{N+2,1} & \sd{N+2,1} & \se{N+2,2} & \sd{N+2,2} \\ \hline 
 \p_3 p  & 
1 &  
1  &
(2\lambda+3)/(2\lambda+4) & 
(\lambda+2)/(\lambda+3)   \\ \hline
\nab p& 
1 &  
(\lambda+1)/(\lambda+2) &
(\lambda+1)/(\lambda+2)  & 
(\lambda+1)/(\lambda+3) \\ \hline
\end{array}
\end{displaymath} 
\end{thm}
\begin{proof}
The argument is essentially identical to that employed in Proposition \ref{i_improved_u}, except that now we use Lemmas \ref{i_bs_G1_1}--\ref{i_bs_G4_1} for estimates of $G^i$ and Proposition \ref{i_improved_u} for estimates of $Du, D^2 u$.  As such, we will only mention which terms determine the power for each estimate.

For $\nab u$ the power is determined by $Dp$, and then Poincar\'e and interpolation give the estimates for $u$, $D \nab u$, and $Du$.  In the $\p_3 p$ estimate the power is determined by $D \nab u$, and in the $\nab p$ estimate the power is determined by $Dp$.  The power in the $\nab^2 u$ estimate is determined by $Dp$.  

The only estimate not modeled on one in Proposition \ref{i_improved_u} is the one for $\p_3 u_3$.  We employ the equation $\diverge{u} = G^2$ to bound
\begin{equation}
 \ns{ \p_3 u_3} \ls \ns{G^2} + \ns{D u} \text{ and } \ns{\nab  \p_3 u_3} \ls \ns{\nab G^2} + \ns{D\nab u}.
\end{equation}
The estimates of $\p_3 u_3$ and $\nab \p_3 u_3$ in the table follow from these, with the power of the former determined by $Du$ and the latter determined by $D \nab u$.
\end{proof}

\subsection{Bootstrapping: second iteration}\label{interp_sec_2}

We now use the improved estimates of Theorem \ref{i_bs_u} to improve the estimates of
$G^i$, $i=1,2$ recorded in Lemmas \ref{i_bs_G1_1}--\ref{i_bs_G2_1}.   We once again omit the proof.

\begin{thm}\label{i_bs_G}
The following table encodes the power in the $H^0(\Omega)$ interpolation estimates for $G^{1,i}$, $i=1,\dotsc,5$ and $G^1$ and their spatial derivatives.

\begin{displaymath}
\begin{array}{| l | c c  c c |}
\hline
X & \se{N+2,1} & \sd{N+2,1} & \se{N+2,2} & \sd{N+2,2} \\ \hline 
 G^{1,1} & 
1 &   
1  & 
1 &
(2\lambda+2)/(\lambda+3)   \\ \hline
\nab G^{1,1}, \nab^2 G^{1,1}& 
1  &   
1 & 
1 &
1  \\ \hline
G^{1,2}, \nab G^{1,2}, \nab^2 G^{1,2} & 
1 &   
1 & 
1 &
1   \\  \hline
G^{1,3} & 
1 &   
1  & 
1 &
(2\lambda+2)/(\lambda+3)   \\ \hline
\nab G^{1,3}, \nab^2 G^3 & 
1 &   
1 &
1 & 
1   \\  \hline
G^{1,4}, \nab G^{1,4}, \nab^2 G^{1,4} & 
1 &   
1 &
1 & 
1   \\  \hline
G^{1,5}, \nab G^{1,5}, \nab^2 G^{1,5} & 
1 &   
1 &
1 & 
1   \\  \hline
G^{1} & 
1 &   
1  & 
1 &
(2\lambda+2)/(\lambda+3)   \\ \hline
\nab G^{1}, \nab^2 G^1 & 
1 &   
1 &
1 & 
1   \\  \hline
\end{array}
\end{displaymath}

The following table encodes the power in the $H^0(\Omega)$ interpolation estimates for $G^{2}$ and its spatial derivatives.
\begin{displaymath}
\begin{array}{| l | c  c c c |}
\hline
X & \se{N+2,1} & \sd{N+2,1} & \se{N+2,2} & \sd{N+2,2} \\ \hline 
 G^{2}, \nab G^2, \nab^2 G^2 & 
1 &   
1  &
1  & 
1   \\ \hline
\end{array}
\end{displaymath}
 \end{thm}

Now we make final improvements to our estimates.

\begin{prop}\label{i_bs_u_2}
The following table encodes the power in the $H^0(\Omega)$ interpolation estimates for $D \p_3 u_i$ for $i=1,2$.
\begin{displaymath}
\begin{array}{| l | c  c c c |}
\hline
X & \se{N+2,1} & \sd{N+2,1} & \se{N+2,2} & \sd{N+2,2} \\ \hline 
D \p_3 u_i, i=1,2 &   
1  &
1 & 
1  & 
(\lambda + 2) / (\lambda+3)   \\ \hline
\end{array}
\end{displaymath}

The following table encodes the power in an $H^2(\Sigma)$ estimates for $D u_i$ for $i=1,2$.
\begin{displaymath}
\begin{array}{| l | c  c c c |}
\hline
X & \se{N+2,1} & \sd{N+2,1} & \se{N+2,2} & \sd{N+2,2} \\ \hline 
D u_i, i=1,2 &   
1  &
1 & 
1  & 
(\lambda + 2) / (\lambda+3)   \\ \hline
\end{array}
\end{displaymath}

The following table encodes the power in the improved $H^0(\Sigma)$ interpolation estimates for $\dt \eta$.
\begin{displaymath}
\begin{array}{| l | c  c c c |}
\hline
X & \se{N+2,1} & \sd{N+2,1} & \se{N+2,2} & \sd{N+2,2} \\ \hline 
\dt \eta &   
1  &
1 & 
1  & 
(\lambda + 2) / (\lambda+3)   \\ \hline
\end{array}
\end{displaymath}

\begin{proof}
We may argue exactly as in Lemma \ref{i_improved_p} to bound
\begin{multline}\label{i_bs_u_2_1}
 \ns{D^2 p} \ls \ns{D^2 \eta} + \ns{D^2 \dt u} + \ns{D^4 u} + \ns{D^3 \nab u} \\
+ \ns{D^2 G^1} + \ns{D^2 G^2} + \ns{D^2 \nab G^2}  + \ns{D^2 G^3}_\Sigma.
\end{multline}
We may also argue as in Proposition \ref{i_improved_u} to bound
\begin{equation}\label{i_bs_u_2_2}
 \ns{D \p_3 u_i} \ls  \ns{D \dt u} + \ns{D^3 u} + \ns{D^2 p}  
+ \ns{D G^1} + \ns{D G^2}  + \ns{D G^3}_\Sigma
\end{equation}
for $i=1,2$.  Combining \eqref{i_bs_u_2_1} and \eqref{i_bs_u_2_2} and employing Theorems \ref{i_bs_u} and \ref{i_bs_G} and Lemmas \ref{i_bs_G3_1} and \ref{i_bs_G4_1}, we then find the $H^0(\Omega)$ estimates for $D \p_3 u_i$, $i=1,2$ listed in the table.  The power is determined by $D^2 \eta$.

We now turn to the $\ns{D u_i}_{H^2(\Sigma)}$ estimate for $i=1,2$.  We employ trace theory and the  Poincar\'e inequality to bound
\begin{equation}
 \ns{D u_i}_{H^0(\Sigma)} \ls \ns{D \p_3 u_i}_{0} \text{ and } \ns{D^3 u_i}_{H^0(\Sigma)} \ls \ns{D^3 \p_3 u_i}_{0},
\end{equation}
and then we utilize our new estimate for $D \p_3 u_i$ to deduce the $H^2(\Sigma)$ estimates listed in the table.  The power is determined by $D \p_3 u_i$ since $D^3 \p_3 u_i$ has four derivatives and hence has a power of $1$.

Finally, for the $\dt \eta$ estimate we use \eqref{linear_perturbed}, trace theory, and Lemma \ref{poincare_usual} to bound
\begin{equation}
 \ns{\dt \eta}_{H^0(\Sigma)} \ls  \ns{u_3}_{H^0(\Sigma)} +  \ns{G^4}_{H^0(\Sigma)} \ls \ns{\nab u_3}_{0} +  \ns{G^4}_{H^0(\Sigma)}.
\end{equation}
Then Theorem \ref{i_bs_u} and Lemma \ref{i_bs_G4_1} provide the $\dt \eta$ estimate for $\sd{N+2,2}$ listed in the table, with the power determined by $\nab u_3$; the estimates for $\se{N+2,1}, \se{N+2,2}, \sd{N+2,1}$ come from Lemma \ref{i_interp_eta}.
\end{proof}

\end{prop}

Now we record an interpolation estimate for $\k$, as defined by \eqref{i_K_def}.

\begin{lem}\label{i_K_estimate}
We have that $\k \ls \se{N+2,2}^{(8+2\lambda)/(8+4\lambda)}.$
\end{lem}
\begin{proof}
By definition, $\k = \pns{\nab u}{\infty} + \pns{\nab^2 u}{\infty} + \sum_{i=1}^2 \snormspace{D u_i}{2}{\Sigma}^2$.  We may then use the $H^2(\Sigma)$ interpolation estimate of Proposition \ref{i_bs_u_2} and the $L^\infty$ interpolation estimate of Proposition \ref{i_improved_u} with $r=2\lambda/(4+\lambda)$ to bound $\k \ls \se{N+2,2}^{2/(2+r)}$.  The choice of $r$ implies that $2/(2+r) = (8+2\lambda)/(8+4\lambda)$, and the result follows.
\end{proof}

\subsection{Estimates at the high end }

Our analysis in Sections \ref{interp_sec_1}--\ref{interp_sec_2} dealt with the problems associated with estimating terms involving fewer derivatives than appear in $\se{N+2,m}, \sd{N+2,m}$.  We now turn to the problem of estimating terms involving more derivatives than are controlled by $\sd{N+2,m}$.  We accomplish such an estimate  by interpolating between $\sd{N+2,m}$ and $\se{2N}$, which controls more derivatives since $N \ge 5$.  Fortunately, the only term we must concern ourselves with is $\nab^{2N+3} \bar{\eta}$, and to simplify things we will only estimate it in terms of $\sd{N+2,2}$.  This suffices since $\sd{N+2,2} \ls \sd{N+2,1}$.

\begin{lem}\label{i_eta_half_over}
We have the estimate
\begin{equation}\label{i_e_h_0}
 \ns{D^{2N+4} \eta}_{1/2} + \ns{\nab^{2N+5} \bar{\eta}}_{0} \ls (\se{2N})^{2/(4N-7)} (\sd{N+2,2})^{(4N-9)/(4N-7)}.
\end{equation}
\end{lem}
\begin{proof}
According to Lemma \ref{i_poisson_grad_bound}, with $q=2N+5$, we may bound 
\begin{equation}
\ns{\nab^{2N+5} \bar{\eta}}_{0} 
\ls \norm{ \eta}_{\dot{H}^{2N+9/2}(\Sigma)}^2 
\ls \ns{D^{2N+4} \eta}_{1/2},
\end{equation}
so it suffices to prove \eqref{i_e_h_0} with only the $D^{2N+4} \eta$ term on the left side.  To prove this, we will use a standard Sobolev interpolation inequality:
 \begin{equation}\label{i_e_h_1}
 \norm{f}_{s} \ls \norm{f}_{s-r}^{q/(r+q)} \norm{f}_{s+q}^{r/(r+q)}
\end{equation}
for $s,q> 0$ and $0 \le r \le s$.  Applying this to $f = D^{3} \eta$ with $s = 2N + 3/2$, $r=1$, and $q = 2N - 9/2$, we find that
\begin{equation}
 \norm{D^{2N+4} \eta}_{1/2} \le \norm{ D^{3} \eta }_{2N + 3/2} \ls \norm{D^{3} \eta}_{2N+1/2}^{(4N-9)/(4N-7) } \norm{D^{3} \eta}_{4N-3}^{2/(4N-7)}.
\end{equation}
The desired inequality then follows by squaring and using the definitions of $\se{2N}$ and $\sd{N+2,2}$.
\end{proof}

Our next result utilizes Lemma \ref{i_eta_half_over} to estimate  products such as   $u  D^{2N+4} \eta$.

\begin{lem}\label{i_eta_half_product}
   Let $P = P(K,D\eta)$ be a polynomial in $K, D\eta$.  Then there exists a $\theta >0$ so that
\begin{equation}\label{i_e_h_p_0}
 \snormspace{(D^{2N+4} \eta)  u }{1/2}{\Sigma}^2  +\snormspace{(D^{2N+4} \eta ) P \nab u }{1/2}{\Sigma}^2  \ls  \se{2N}^\theta \sd{N+2,2}.
\end{equation}
Let $Q = Q(K,\tilde{b},\nab \bar{\eta})$ be a polynomial.  Then exists a $\theta >0$ so that
\begin{equation}\label{i_e_h_p_00}
 \ns{(\nab^{2N+5} \bar{\eta} ) Q \nab u }_{0}  \ls  \se{2N}^\theta \sd{N+2,2}.
\end{equation}
\end{lem}

\begin{proof}
According to  the bound \eqref{i_s_p_02} of Lemma \ref{i_sobolev_product_1}, we may bound
\begin{multline}\label{i_e_h_p_1}
  \snormspace{(D^{2N+4} \eta)  u }{1/2}{\Sigma}^2  +\snormspace{(D^{2N+4} \eta ) P \nab u }{1/2}{\Sigma}^2 \\ \ls 
\snormspace{D^{2N+4} \eta }{1/2}{\Sigma}^2 \snormspace{u}{2}{\Sigma}^2    +\snormspace{D^{2N+4} \eta }{1/2}{\Sigma}^2   \snormspace{ P\nab u}{2}{\Sigma}^2.
\end{multline}
Trace theory implies that
\begin{multline}\label{i_e_h_p_2}
 \snormspace{u}{2}{\Sigma}^2  + \snormspace{\nab u}{2}{\Sigma}^2 \le   \snormspace{u}{0}{\Sigma}^2 +  \snormspace{D^2 u}{0}{\Sigma}^2  + \snormspace{\nab u}{0}{\Sigma}^2  + \snormspace{D^2 \nab u}{0}{\Sigma}^2 \\
\ls \ns{\nab u}_{0} + \ns{D^2 \nab u}_{0} +  \ns{\nab^2 u}_{0} + \ns{\nab^2 D^2 u}_{0},
\end{multline}
but then an application of Theorem \ref{i_bs_u} to all the terms on the right side shows that 
\begin{equation}\label{i_e_h_p_3}
 \snormspace{u}{2}{\Sigma}^2  + \snormspace{\nab u}{2}{\Sigma}^2 \ls 
\left( \sd{N+2,2} \right)^{(1+\lambda)/(3+\lambda)}.
\end{equation}
It is easy to see, based on the terms controlled by $\se{2N}$, that 
$  \snormspace{ P }{2}{\Sigma}^2 \ls \se{2N} \le 1$.  We may then combine this with \eqref{i_e_h_p_3} and \eqref{i_s_p_01} of Lemma \ref{i_sobolev_product_1} to deduce that
\begin{equation}
 \snormspace{u}{2}{\Sigma}^2  + \snormspace{P \nab u}{2}{\Sigma}^2  \ls 
\left( \sd{N+2,2} \right)^{(1+\lambda)/(3+\lambda)}.
\end{equation}
Then this bound, \eqref{i_e_h_p_1}, and Lemma \ref{i_eta_half_over} imply that
\begin{equation}\label{i_e_h_p_4}
 \snormspace{(D^{2N+4} \eta)  u }{1/2}{\Sigma}^2  +\snormspace{(D^{2N+4} \eta ) P \nab u }{1/2}{\Sigma}^2  \ls  \se{2N}^\theta \sd{N+2,2}^{\kappa}
\end{equation}
for some $\theta >0$ and for 
\begin{equation}
 \kappa = \frac{4N-9}{4N-7} + \frac{\lambda +1}{\lambda+3} \ge \frac{4N-9}{4N-7} + \frac{1}{3} = \frac{16N -34}{12N-21} \ge 1
\end{equation}
since $N \ge 4$.  Since $\sd{N+2,2} \le \se{2N} \le 1$, we may bound $\sd{N+2,2}^{\kappa} \le \sd{N+2,2}$ in \eqref{i_e_h_p_4}, which then yields \eqref{i_e_h_p_0}.

To derive  \eqref{i_e_h_p_00} we first bound
\begin{equation}
  \ns{(\nab^{2N+5} \bar{\eta} ) Q \nab u }_{0} \le  \ns{\nab^{2N+5} \bar{\eta} }_{0} \pnorm{ \nab u }{\infty}^2 \pnorm{Q   }{\infty}^2.
\end{equation}
The first term on the right is controlled with Lemma \ref{i_eta_half_over}.  The second term satisfies
\begin{equation}
 \pnorm{ \nab u }{\infty}^2 \ls (\sd{N+2,2})^{2/3}
\end{equation}
by virtue of the $L^\infty$ estimates of Proposition \ref{i_improved_u}.  The third term satisfies $\pnorm{Q   }{\infty}^2 \ls \se{2N} \le 1$ by Sobolev embeddings and the definition of $\se{2N}$.  The estimate \eqref{i_e_h_p_00} follows by combining these bounds as above.

\end{proof}

\section{Nonlinear estimates}\label{inf_3}

\subsection{Estimates of $G^i$ at the $N+2$ level}

We now provide estimates of $G^i$ in terms of $\se{N+2,m}$ and $\sd{N+2,m}$.  Notice that our estimates are somewhat stronger than those stated in, say Theorem \ref{i_bs_G}, since we include some power of $\se{2N}$ multiplied by $\se{N+2,m}$ or $\sd{N+2,m}$.

\begin{thm}\label{i_G_estimates_half}
Let  $m\in\{1,2\}$.  Then there exists a $\theta >0$ so that
\begin{multline}\label{i_G_e_h_0}
\ns{ \bar{\nab}_m^{2(N+2)-2} G^1}_{0} +  \ns{ \bar{\nab}_0^{2(N+2)-2}  G^2}_{1} +
 \ns{ \dbm{2(N+2)-2} G^3}_{1/2}  + \ns{\bar{D}_0^{2(N+2)-2} G^4}_{1/2}
\\  \ls 
\se{2N}^{\theta}\se{N+2,m}
\end{multline}
and
\begin{multline}\label{i_G_e_h_00}
\ns{ \bar{\nab}_m^{2(N+2)-1} G^1}_{0} +  \ns{ \bar{\nab}_0^{2(N+2)-1}  G^2}_{1} +
 \ns{ \dbm{2(N+2)-1} G^3}_{1/2} + \ns{\bar{D}_0^{2(N+2)-1} G^4}_{1/2}
\\ 
+\ns{\bar{D}^{2(N+2)-2} \dt G^4}_{1/2}
\ls
\se{2N}^\theta \sd{N+2,m}.
\end{multline}
\end{thm}
\begin{proof}

The estimates of these nonlinearities are fairly routine to derive: we note that all terms are quadratic or of higher order; then we apply the differential operator and expand using the Leibniz rule; each term in the resulting sum is also at least quadratic, and we estimate one term in $H^k$ ($k=0,1/2$, or $1$ depending on $G^i$) and the other term in $L^\infty$ or $H^{m}$ for $m$ depending on $k$, using Sobolev embeddings, trace theory, and Lemmas \ref{i_sobolev_product_1}, \ref{i_poisson_grad_bound}, and \ref{i_poisson_interp}--\ref{i_slice_interp}.  The derivative count in the differential operators is chosen in order to allow estimation by $\se{N+2,m}$ in \eqref{i_G_e_h_0} and by $\sd{N+2,m}$ in \eqref{i_G_e_h_00}.  There is only one difficulty that arises.  Because $\se{N+2,m}$ and $\sd{N+2,m}$ involve minimal derivative counts, there may be terms in the sum $\pa G^i$ that cannot be directly estimated.  To handle these terms, we invoke the interpolation results of Theorems \ref{i_bs_u} and  \ref{i_bs_u_2} and Proposition \ref{i_improved_u}, as well as the specialized interpolation results of Lemma \ref{i_eta_half_product}.  A detailed proof of the estimates is quite lengthy, so for the sake of brevity we present only a sketch.

Let $\alpha \in \mathbb{N}^{1+3}$ with $m \le \abs{\alpha} \le 2(N+2)-2$ and consider $\pa G^1$.  Since $G^1$ involves $\nab p$ and  $\p^\beta u$, $\p^\beta \bar{\eta}$ with $\abs{\beta} \le 2$, we find that $\pa G^1$ involves at most (with parabolic counting) $2(N+2)-1$ derivatives of $p$, and at most $2(N+2)$ derivatives of $u$ and $\bar{\eta}$.  We have that $G^1$ is a linear combination of at least quadratic terms, and as such, so is $\pa G^1$.  Let us consider a generic term in the sum  $\pa G^1$, which we write as $X Y$ with $X$ of the form $\p^\beta u$  or $\p^\beta \bar{\eta}$ with $\abs{\beta} \le 2(N+2)$ or else $\p^\beta p$ with $\abs{\beta} \le 2(N+2)-1$, and $Y$ a polynomial in lower-order derivatives.  If $\abs{\beta}$ is sufficiently large with respect to $m$, then the minimal derivative count is exceeded and  we may estimate $\ns{X}_{0} \le \se{N+2,m}$.   It is easy to verify, using Sobolev embeddings and Lemmas \ref{i_sobolev_product_1}, \ref{i_poisson_grad_bound}, and \ref{i_poisson_interp}--\ref{i_slice_interp}, that we always have $\pns{Y}{\infty} \ls \se{2N}^\theta$ for some $\theta >0$.  Then 
\begin{equation}
 \ns{XY}_{0} \ls \ns{X}_{0} \pns{Y}{\infty} \ls \se{N+2,m} \se{2N}^\theta.
\end{equation}
On the other hand, if $\abs{\beta}$ is not large, then we must resort to interpolation, using Theorems \ref{i_bs_u} and  \ref{i_bs_u_2} and Proposition \ref{i_improved_u}.  In this case, it can be verified that we always get estimates of the form $\ns{X}_{0} \ls (\se{2N})^{1-\theta_1} (\se{N+2,m})^{\theta_1}$ and $\pns{Y}{\infty} \ls (\se{2N})^{\theta_2} (\se{N+2,m})^{\theta_3}$ with $\theta_1 \in (0,1]$, $\theta_2, \theta_3 \ge 0$,  and $\theta_1 + \theta_3 \ge 1$ so that
\begin{equation}
 \ns{XY}_{0} \ls \ns{X}_{0} \pns{Y}{\infty} \ls \se{N+2,m} \se{2N}^\theta
\end{equation}
for some $\theta>0$.  This analysis works for every $XY$ appearing in $\pa G^1$, so 
\begin{equation}
 \ns{\bar{\nab}_m^{2(N+2)-2} G^2}_{0} \ls \se{N+2,m} \se{2N}^\theta
\end{equation}
for some $\theta >0$.  It can then be verified, through a straightforward but lengthy analysis like that used above, that all of the estimates in \eqref{i_G_e_h_0} hold.  We note though, that in order to estimate the $G^3$ terms, we must use Remark \ref{G3_remark} to remove the appearance of $(p-\eta)$ in $G^3$.

Now we sketch the proof of the estimates in \eqref{i_G_e_h_00}.  We may argue as above to estimate all terms that arise in $\pa G^i$ with two exceptions: terms involving $\nab^{2N+5} \bar{\eta}$ on $\Omega$ or $D^{2N+4} \eta$ on $\Sigma$.  These always have the form of the terms estimated in Lemma \ref{i_eta_half_product}, so we may use it for estimates in terms of $\se{2N}^\theta \sd{N+2,2}$, which suffice for \eqref{i_G_e_h_00} since $\sd{N+2,2} \ls \sd{N+2,1}$.  Then \eqref{i_G_e_h_00} follows by combining the estimates of the exceptional terms with the estimates of the terms as above.

\end{proof}

\subsection{Estimates of $G^i$ at the $2N$ level}

Now we derive estimates for the nonlinear $G^i$ terms at the $2N$ level.

\begin{thm}\label{i_G_estimates}
Let $m\in\{1,2\}$.  Then there exists a $\theta >0$ so that
\begin{equation}\label{i_G_e_0}
\ns{ \bar{\nab}_0^{4N-2} G^1}_{0} +  \ns{ \bar{\nab}_0^{4N-2}  G^2}_{1} +
 \ns{ \bar{D}_{0}^{4N-2} G^3}_{1/2} + \ns{\bar{D}_0^{4N-2} G^4}_{1/2}
 \ls 
\se{2N}^{1+\theta},
\end{equation}
\begin{multline}\label{i_G_e_00}
\ns{ \bar{\nab}_{0}^{4N-2} G^1}_{0} +  \ns{ \bar{\nab}_0^{4N-2}  G^2}_{1} +
 \ns{ \bar{D}_{0}^{4N-2} G^3}_{1/2} + \ns{\bar{D}_0^{4N-2} G^4}_{1/2} \\
+ 
\ns{ \bar{\nab}^{4N-3} \dt G^1}_{0}  +  \ns{ \bar{\nab}^{4N-3} \dt G^2}_{1} +
 \ns{ \bar{D}^{4N-3} \dt G^3}_{1/2}  + \ns{\bar{D}^{4N-2} \dt G^4}_{1/2}
 \\ 
\ls \se{2N}^\theta \sd{2N},
\end{multline}
and
\begin{multline}\label{i_G_e_000}
 \ns{ \nab^{4N-1} G^1}_{0} +  \ns{  \nab^{4N-1}  G^2}_{1} +
 \ns{ D^{4N-1} G^3}_{1/2} + \ns{D^{4N-1} G^4}_{1/2}
\\ \ls
\se{2N}^\theta \sd{2N} + \k \f.
\end{multline}
\end{thm}
\begin{proof}

As explained in Theorem \ref{i_G_estimates_half}, the estimates are routine and lengthy, so we present only a sketch.  The estimates in \eqref{i_G_e_0} are straightforward since $\se{2N}$ has no minimal derivative restrictions.  They may be derived using Sobolev embeddings, trace theory, and Lemmas \ref{i_sobolev_product_1}, \ref{i_poisson_grad_bound}, and the $L^\infty$ estimates of \ref{i_poisson_interp}. 

The only terms with minimal derivatives in $\sd{2N}$ are $D\eta$ and $\nab p$.  The latter presents no problem since, owing to Remark \ref{G3_remark}, $p$ itself never appears in any of the $G^i$ terms.  The former may be dealt with by using Lemmas \ref{i_poisson_interp} and \ref{i_sigma_interp}  to produce interpolations estimates of $\bar{\eta}$  and $\eta$ in terms of $D \eta$.   Whenever  interpolation is needed to estimate these terms, there are always other terms multiplying them that allow for the recovery of a power of $1$ on $\sd{2N}$.  Using these estimates with  Sobolev embeddings, trace theory, and Lemmas \ref{i_sobolev_product_1},  \ref{i_poisson_grad_bound}, and \ref{i_poisson_interp} then yields \eqref{i_G_e_00}.

We now turn to the derivation of \eqref{i_G_e_000}.  Consider  $\pa G^i$ with $\abs{\alpha} = 4N-1$ and $\alpha_0 =0$, i.e. purely spatial derivatives, and expand $\pa G^i$ using the Leibniz rule.  With two exceptions, we may argue as in the derivation of \eqref{i_G_e_00} to estimate the desired norms of all of the resulting terms  by $\se{2N}^\theta \sd{2N}$ for $\theta >0$.  The exceptional terms are ones involving either $\nab^{4N+1} \bar{\eta}$ in $\Omega$ or $D^{4N} \eta$ on $\Sigma$.  We will now show how to estimate the exceptional terms with $\k \f$, as defined by \eqref{i_K_def} and \eqref{i_transport_def}.

In $\nab^{4N-1} G^1$, there are terms of the form $\p^{\beta} \bar{\eta} Q \p^\gamma u$, with 
\begin{equation}\label{iGe_1}
Q = Q(A,B,J,K,\nab A, \nab B, \nab J) 
\end{equation}
a polynomial and $\beta,\gamma \in \mathbb{N}^3$ with $\abs{\beta} = 4N+1$ and $\abs{\gamma}=1$.   To estimate such a term, we use  Lemma \ref{i_poisson_grad_bound}  to bound
\begin{equation}\label{iGe_2}
 \ns{\nab^{4N+1} \bar{\eta}}_{0} \ls \ns{D^{4N+1/2} \eta}_{0} \ls \f.
\end{equation}
Sobolev embeddings imply that $\pns{Q}{\infty} \ls \se{2N}^\theta \ls 1$ for some $\theta >0$, so
\begin{equation}\label{iGe_3}
 \ns{\p^{\beta} \bar{\eta} Q \p^\gamma u}_{0} \ls \ns{\nab^{4N+1} \bar{\eta}}_{0} \pns{\nab u}{\infty} \pns{Q}{\infty} \ls  \ns{D^{4N+1/2} \eta}_{0} \pns{\nab u}{\infty} \ls \f \k.
\end{equation}
This estimate then yields the $G^1$ estimate in \eqref{i_G_e_000}.

In $\nab^{4N-1} G^2$ there are terms of the form $\p^\beta \bar{\eta} Q \p^\gamma u$ with $Q = Q(A,B,K)$ a polynomial and $\beta,\gamma \in \mathbb{N}^3$ with $\abs{\beta} = 4N$, $\abs{\gamma}=1$.  Again, Sobolev embeddings imply that $\ns{Q}_{C^1(\Omega)} \ls \se{2N}^\theta \ls 1$, so
\begin{multline}
 \ns{\p^\beta \bar{\eta} Q \p^\gamma u}_{1} \ls  \ns{Q}_{C^1(\Omega)} \ns{\p^\beta \bar{\eta}  \p^\gamma u}_{1} \ls \ns{\p^\beta \bar{\eta}  \p^\gamma u}_{0}  + \ns{\p^\beta \bar{\eta}  \nab \p^\gamma u}_{0} + \ns{\nab \p^\beta \bar{\eta}  \p^\gamma u}_{0} \\
\ls \ns{\nab^{4N} \bar{\eta}}_{0} \ns{\nab u}_{C^1(\Omega)} + \ns{\nab^{4N+1} \bar{\eta}}_{0} \pns{\nab u}{\infty}
\ls  \ns{ \eta }_{4N-1/2} \ns{\nab u}_{3}  + \k \f \\
\ls \se{2N} \sd{2N} + \k \f,
\end{multline}
where again we have used Lemma \ref{i_poisson_grad_bound} and Sobolev embeddings.  This estimate yields the  $G^2$ estimate in \eqref{i_G_e_000}.

In $D^{4N-1} G^3$  there are terms of the form $\p^\beta \eta Q \p^\gamma u$, where $\beta \in \mathbb{N}^2$ with $\abs{\beta}=4N$, $\gamma \in \mathbb{N}^{3}$ with $\abs{\gamma}=1$, and $Q$ is a term for which we can estimate $\ns{Q}_{C^1(\Sigma)} \ls \se{2N}^\theta \ls 1$.  Then Lemma \ref{i_sobolev_product_2}  implies that
\begin{equation}
 \ns{\p^\beta \eta Q \p^\gamma u}_{1/2} \ls \ns{\p^\beta \eta }_{1/2} \ns{ Q \p^\gamma u}_{C^1} \ls
\ns{ \eta }_{4N+1/2} \ns{ Q}_{C^1} \ns{\nab u}_{C^1(\Sigma)} \ls \f \k,
\end{equation}
where in the last inequality we have used $\ns{\nab u}_{C^1(\Sigma)} \ls \k$, which follows since $\nab u$ and $\nab^2 u$ are continuous on the closure of $\Omega$.  This estimate yields the  $G^3$ estimate in \eqref{i_G_e_000}.

In $D^{4N-1} G^4$  the exceptional terms are of the form $\p^\beta u_i$, where $\beta \in \mathbb{N}^2$ with $\abs{\beta}=4N$ and $i=1,2$.  Then Lemma \ref{i_sobolev_product_1}  implies that
\begin{equation}
 \ns{\p^\beta \eta  u_1}_{1/2} \ls \ns{\p^\beta \eta }_{1/2} \snormspace{u_i}{2}{\Sigma}^2 \ls \f \k.
\end{equation}
This estimate yields the  $G^4$ estimate in \eqref{i_G_e_000}.

\end{proof}

\subsection{Estimates of  other nonlinearities}\label{i_lambda_nonlinearities}

The next result provides estimates for $\i_{\lambda} G^i$ and its derivatives.  

\begin{prop}\label{i_riesz_G}
We have that  
\begin{equation}\label{i_r_G_0}
 \ns{\i_\lambda G^1}_{1} + \ns{  \i_\lambda G^2}_{2} + \ns{  \i_\lambda \dt G^2}_{0}  \ls \se{2N} \min\{ \se{2N}, \sd{2N}   \}
\end{equation}
and
\begin{equation}\label{i_r_G_00}
 \ns{\i_\lambda G^3}_{1} + \ns{\i_\lambda G^4}_{1} \ls \se{2N} \min\{ \se{2N}, \sd{2N}   \}.
\end{equation}
Also, 
\begin{equation}\label{i_r_G_000}
 \ns{\i_\lambda G^4}_{0}  \ls  \sd{2N}^2.
\end{equation}
\end{prop}

\begin{proof}
For each $i=1,2$ and for $\alpha \in \mathbb{N}^{1+3}$ such that $\abs{\alpha} \le 2$ we can write $\pa G^i = P^i_\alpha Q^i_\alpha$, where $P^i_\alpha$ is polynomial in the terms  $\p^\beta \tilde{b},$ $\p^\beta K,$ $\p^\beta \bar{\eta},$ and $\p^\beta u$ for $\beta \in \mathbb{N}^{1+3}$ with $\abs{\beta} \le 4$, and $Q^i_\alpha$ is linear in the terms $\p^\beta  \nab u$, $\p^\beta  \nab^2 u$, and $\p^\beta \nab p$ for $\abs{\beta}\le 2$.  Then we may employ the bound \eqref{i_r_p_0} of Lemma \ref{i_riesz_prod} to see that
\begin{equation}\label{i_r_G_1}
 \ns{\pa \i_{\lambda} G^i}_{0} \ls \ns{P^i_\alpha}_{0} \left(\ns{ Q^i_\alpha  }_{1} \right)^{\lambda} 
\left(\ns{D Q^i_\alpha}_{1} \right)^{1-\lambda}.
\end{equation}
It is then easily verified, using the Sobolev embedding, Lemmas \ref{i_sobolev_product_1}, \ref{i_poisson_grad_bound}, and \ref{i_poisson_interp} and the fact that $\se{2N} \le 1$, that
\begin{equation}
 \ns{P^i_\alpha}_{0} \ls \se{2N} \text{ and } \ns{Q^i_\alpha}_{2} \ls  \min\{ \se{2N}, \sd{2N}   \},
\end{equation}
which, together with \eqref{i_r_G_1}, implies \eqref{i_r_G_0}.  

For $i=3,4$ and $\alpha \in \mathbb{N}^2$ so that $\abs{\alpha}\le 1$, we may similarly decompose $\pa G^i = P^i_\alpha Q^i_\alpha$.  We then argue as above, employing the bound \eqref{i_r_p_01} of Lemma \ref{i_riesz_prod} as well as trace estimates, to deduce \eqref{i_r_G_00}.  The bound \eqref{i_r_G_000}  also follows from Lemma \ref{i_riesz_prod} and trace estimate since
\begin{equation}
 \ns{\i_\lambda G^4}_{0} \ls  
\snormspace{u}{0}{\Sigma}^2 \left(\ns{ D \eta }_{0} \right)^{\lambda} \left(\ns{D^2 \eta}_{0} \right)^{1-\lambda}
\le \sd{2N} \sd{2N}^\lambda \sd{2N}^{1-\lambda} =\sd{2N}^2.
\end{equation}
\end{proof}

Now we provide some further estimates of product terms that will be useful later when we analyze the energy evolution for $\i_{\lambda} u$ and $\i_{\lambda} \eta.$

\begin{lem}\label{i_riesz_u}
It holds that
\begin{equation}\label{i_r_u_0}
\ns{\i_{\lambda} [ (AK) \p_3 u_1 + (BK) \p_3 u_2 ]  }_{0} +  \sum_{i=1}^2 \ns{\i_{\lambda} [ u \p_i K ]}_{0} \ls  \sd{2N}^2
\end{equation}
and
\begin{equation}\label{i_r_u_00}
 \ns{\i_{\lambda} [(1-K) u]}_{0} \ls \left( \se{2N}\right)^{1/(1+\lambda)} \left( \sd{2N}\right)^{(1+2 \lambda)/(1+\lambda)} .
\end{equation}
Also,  
\begin{equation}\label{i_r_u_000}
 \ns{\i_{\lambda} [(1-K) G^2 ]}_{0} \ls \se{2N} \sd{2N}^2.
\end{equation}
\end{lem}

\begin{proof}
We apply Lemma \ref{i_riesz_prod}, treating the $AK,BK,\p_i K$ terms as $f$ and the $u,\nab u$ terms as $g$,  to bound
\begin{multline}
 \ns{\i_{\lambda} [ (AK) \p_3 u_1 + (BK) \p_3 u_2 ]  }_{0} +  \sum_{i=1}^2 \ns{\i_{\lambda} [ u \p_i K ]}_{0} 
\\  \ls 
( \ns{AK}_{0} + \ns{BK}_{0} + \ns{D K}_{0}   )\ns{u}_{3}.
\end{multline}
From Lemma \ref{infinity_bounds}, the fact that $\p_i K = -K^2 \p_i J$, and Lemma \ref{i_poisson_grad_bound},  we know that
\begin{equation}
 \ns{AK}_{0} + \ns{BK}_{0} + \ns{D K}_{0} \ls \ns{\nab \bar{\eta}}_{1} \ls \ns{ D \eta}_{1} \le \sd{2N}.
\end{equation}
Then, since $\ns{u}_{3}\le \sd{2N}$, we know that \eqref{i_r_u_0} holds.

Now, since $1-K = K(1-J)$, we can again use Lemmas \ref{i_riesz_prod} and \ref{infinity_bounds} to see that
\begin{equation}
 \ns{\i_{\lambda} [(1-K) u]}_{0} \ls  \ns{K(1-J)}_{0} \ns{u}_{2} \ls \ns{\bar{\eta}}_{1} \ns{u}_{2}.
\end{equation}
To control $\bar{\eta}$ we use Lemmas \ref{i_poisson_grad_bound} and \ref{i_sigma_interp} to bound
\begin{multline}\label{i_r_u_1}
 \ns{\bar{\eta}}_{1} \ls \ns{\eta}_{0} + \ns{D \eta}_{0} \\
\ls 
\left(\ns{\i_{\lambda} \eta}_{0} \right)^{1/(1+\lambda)} 
\left(\ns{D \eta}_{0} \right)^{\lambda/(1+\lambda)} 
+ \left(\ns{D \eta}_{0}\right)^{1/(1+\lambda)}  
 \left( \ns{D \eta}_{0} \right)^{\lambda/(1+\lambda)}
\\
\ls \left( \se{2N}\right)^{1/(1+\lambda)} \left( \sd{2N}\right)^{ \lambda/(1+\lambda)}.
\end{multline}
Then \eqref{i_r_u_00} follows from these two estimates and the fact that $\ns{u}_{2} \le \sd{2N}$.

For the estimate of the $(1-K)G^2$ term, we once more use Lemma \ref{i_riesz_prod} to see that
\begin{equation}\label{i_r_u_2}
\ns{\i_{\lambda} [(1-K) G^2 ]}_{0}  \ls \ns{G^2}_0 \ns{1-K}_{2}.
\end{equation}
By differentiating the equation $JK =1$, we may compute the derivatives of $K$ in terms of the derivatives of $J$; this allows us to bound, by virtue of Lemmas \ref{infinity_bounds} and \ref{i_poisson_grad_bound},
\begin{equation} 
 \ns{1-K}_{2} \ls \ns{\bar{\eta}}_{2} \ls \ns{\eta}_{2} \le \ns{\eta}_{0} + \ns{D \eta}_{1}.
\end{equation}
Then we may argue as in \eqref{i_r_u_1} to estimate the right side of this inequality, and we deduce that
\begin{equation}\label{i_r_u_3}
 \ns{1-K}_{2} \ls \left( \se{2N}\right)^{1/(1+\lambda)} \left( \sd{2N}\right)^{ \lambda/(1+\lambda)}.
\end{equation}
On the other hand,
\begin{equation}
 \ns{G^2}_{0} \ls \ns{\nab u}_{0} ( \pns{\bar{\eta}}{\infty} + \pns{\nab \bar{\eta}}{\infty} ).
\end{equation}
We estimate the $L^\infty$ norms by using \eqref{i_p_i_4} of Lemma \ref{i_poisson_interp} first with $q=0$, $s=1$, $r= \lambda^2 + \lambda$ and then with $q=1$, $s=1$, $r = \lambda^2 + 2\lambda$ to see that
\begin{multline}
 \pns{\bar{\eta}}{\infty} + \pns{\nab \bar{\eta}}{\infty} \\
\ls 
\left( \ns{\i_{\lambda} \eta}_{0} \right)^{\lambda/(\lambda+1)} 
\left( \ns{D \eta}_{0} \right)^{1/(\lambda+1)} + 
\left( \ns{\i_{\lambda} \eta}_{0} \right)^{\lambda/(\lambda+1)} 
\left( \ns{D^2 \eta}_{0} \right)^{1/(\lambda+1)} 
\\ \le 
\left( \se{2N} \right)^{\lambda/(\lambda+1)} \left( \sd{2N} \right)^{1/(\lambda+1)}.
\end{multline}
Then, since $\ns{\nab u}_{0} \le \sd{2N}$, we have that
\begin{equation}
 \ns{G^2}_{0} \ls \left( \se{2N} \right)^{\lambda/(\lambda+1)} \left( \sd{2N} \right)^{1+ 1/(\lambda+1)},
\end{equation}
which yields \eqref{i_r_u_000} when combined with \eqref{i_r_u_2} and \eqref{i_r_u_3}.

\end{proof}

Now we provide an estimate of for $\dt^j \mathcal{A}$ when $j= 2N+1$ and when $j=N+3$.

\begin{lem}\label{i_A_time_derivs}
We have that
\begin{equation}\label{i_atd_0}
\ns{\dt^{2N+1} \mathcal{A}}_{0} \ls \sd{2N},
\end{equation}
while for $m=1,2$, 
\begin{equation}\label{i_atd_00}
\ns{\dt^{N+3} \mathcal{A}}_{0} \ls \sd{N+2,m}.
\end{equation}
\end{lem}

\begin{proof}
We will only prove \eqref{i_atd_0}; the bound \eqref{i_atd_00} follows from similar analysis.  Since
$\ns{\dt^{2N+1} \eta}_{1/2} \le \sd{2N}$ and temporal derivatives commute with the Poisson integral, we may employ Lemma \ref{i_poisson_grad_bound} to bound
\begin{equation}
\ns{\dt^{2N+1} \bar{\eta}}_{1}  =  \ns{\dt^{2N+1} \bar{\eta}}_{0} +  \ns{\nab \dt^{2N+1} \bar{\eta}}_{0}\ls  \ns{\dt^{2N+1}  \eta}_{1/2} \le \sd{2N}.
\end{equation}
From this we easily deduce that
\begin{equation}
 \ns{\dt^{2N+1} J}_{0} +  \ns{\dt^{2N+1} K}_{0} \ls \sd{2N}. 
\end{equation}
This, the previous bound, and the Sobolev embeddings then imply \eqref{i_atd_0} since the components of $\mathcal{A}$ are either unity, $K$, $\p_1 \bar{\eta} \tilde{b} K$, or $\p_2 \bar{\eta} \tilde{b} K$.
\end{proof}

\section{Energy evolution using the geometric form}\label{inf_4}

\subsection{Estimates of the perturbations when $\pa = \dt^{\alpha_0}$ is applied to \eqref{geometric} }

We now present estimates of the perturbations $F^i$, defined by \eqref{F_def_start}--\eqref{F_def_end} when $\pa = \dt^{2N}$.

\begin{thm}\label{i_F_estimates}
Let $\pa = \dt^{2N}$ and let $F^1$, $F^2$, $F^3$, $F^4$ be defined by \eqref{F_def_start}--\eqref{F_def_end}.  Then
\begin{equation}\label{i_F_e_01}
 \ns{F^{1} }_{0} + \ns{\dt (J F^{2} ) }_{0} + \ns{F^{3}}_{0} + \norm{F^{4}}_{0} \ls \se{2N} \sd{2N}.
\end{equation}
\end{thm}

\begin{proof}

We first consider the $F^1$ estimate.  Each term in the sums that define $F^{1}$ is at least quadratic.  It is straightforward to see that each such term can be written in the form $X Y$, where we  $X$ involves fewer temporal derivatives than $Y$, and we may use the usual Sobolev embeddings and Lemmas \ref{i_sobolev_product_1} and \ref{i_poisson_grad_bound} along with the definitions of $\se{2N}$ and $\sd{2N}$ to estimate
\begin{equation}
 \pnorm{X}{\infty}^2 \ls \se{2N} \text{ and } \ns{Y}_{0} \ls \sd{2N}.
\end{equation}
Then $\ns{XY}_{0} \le \pnorm{X}{\infty}^2 \ns{Y}_{0} \ls \se{2N}\sd{2N}$, and the $F^1$ estimate in \eqref{i_F_e_01} follows by summing.  A similar argument, also employing trace estimates, yields the $F^3$ and $F^4$ estimates in \eqref{i_F_e_01}.  Note though, that to estimate the $\beta = \alpha$ term in  $F^{3,1}$, we use Remark \ref{G3_remark} to replace $(p-\eta)$.

The same analysis also works for $\dt (J F^{2,1})$ and shows that $\ns{\dt (J F^{2,1})}_{0} \ls \se{2N}\sd{2N}$.  To handle $\dt (J F^{2,2})$ we must also be able to estimate $\ns{\dt^{2N+1}\mathcal{A}}_{0}\ls \sd{2N}$, but this is possible due to Lemma \ref{i_A_time_derivs}.  Then a similar splitting into $L^\infty$ and $H^0$ estimates shows that $\ns{\dt (J F^{2,2})}_{0} \ls \se{2N}\sd{2N}$, and then the $\dt (J F^2)$ estimate in \eqref{i_F_e_01} follows since $F^2 = F^{2,1}+F^{2,2}$. 
\end{proof}

We now present estimates for these perturbations when $\pa = \dt^{N+2}$.

\begin{thm}\label{i_F_estimates_half}
Let $\pa = \dt^{N+2}$ and let $F^1$, $F^2$, $F^3$, $F^4$ be defined by \eqref{F_def_start}--\eqref{F_def_end}.  Then for $m=1,2$ we have
\begin{equation}\label{i_F_e_h_01}
 \ns{F^{1} }_{0} + \ns{\dt (J F^{2} ) }_{0} + \ns{F^{3}}_{0} + \norm{F^{4}}_{0} \ls \se{2N} \sd{N+2,m}.
\end{equation}\
Also, if $N\ge 3$, then there exists a $\theta > 0$ so that
\begin{equation}\label{i_F_e_h_02}
 \ns{F^2}_{0} \ls \se{2N}^\theta \se{N+2,m}
\end{equation}
for $m=1,2$.

\end{thm}
\begin{proof}

The proof of \eqref{i_F_e_h_01} is essentially the same as that of Theorem \ref{i_F_estimates}.  For the $F^1$, $F^3$, and $F^4$ estimates we note that each term in their definition is of the form $X Y$ where $X$ involves fewer temporal derivatives than $Y$, which involves at least two temporal derivatives.  We estimate $\pnorm{X}{\infty}^2 \ls \se{2N}$ and $\ns{Y}_{0} \ls \sd{N+2,m}$ and then sum to get \eqref{i_F_e_h_01}.  Note that since $Y$ involves at least two temporal derivatives, there is no problem estimating it in terms of $\sd{N+2,m}$.  The $\dt (J F^2)$ estimate works similarly, except we must also use the bound \eqref{i_atd_00} from Lemma \ref{i_A_time_derivs}.  Note also that in estimating the $\beta=\alpha$ term in $F^{3,1}$, we must employ Remark \ref{G3_remark} to remove $(p-\eta)$.

We now turn to the proof of \eqref{i_F_e_h_02}. Recall that $F^2 = F^{2,1}+ F^{2,2}$.  Since the sum in $F^{2,1}$ runs over $1\le \beta \le N+1$, we may bound
\begin{multline}\label{i_F_e_h_1}
 \ns{F^{2,1}}_{0} \ls \sum_{1\le \beta \le N+1} \pnorm{\dt^\beta \mathcal{A}}{\infty}^2 \ns{\dt^{N+2-\beta} u}_{1} \\
\ls \sum_{1\le \beta \le N+1} \se{2N} \ns{\dt^{N+2-\beta} u}_{2(N+2)-2(N+2-\beta)} \ls \se{2N} \se{N+2,m}.
\end{multline}
For $F^{2,2}$, a  calculation reveals that
\begin{multline}\label{i_F_e_h_2}
 F^{2,2} = - \dt^{N+2} \mathcal{A}_{ij} \p_j u_i = - \dt^{N+2} \mathcal{A}_{i3} \p_3 u_i \\
= \dt^{N+2} (\p_1 \bar{\eta} \tilde{b} K) \p_3 u_1 + \dt^{N+2} (\p_2 \bar{\eta} \tilde{b} K) \p_3 u_2 - \dt^{N+2} K \p_3 u_3.
\end{multline}
We may use the $L^\infty$ interpolation estimate of Proposition \ref{i_improved_u} to bound $\pnorm{\p_3 u_i}{\infty}^2 \ls \se{N+2,m}$ for $i=1,2$ and $m=1,2$, which then implies that
\begin{equation}\label{i_F_e_h_3}
 \ns{ \dt^{N+2} (\p_1 \bar{\eta} \tilde{b} K) \p_3 u_1 + \dt^{N+2} (\p_2 \bar{\eta} \tilde{b} K) \p_3 u_2}_{0}    \ls    \se{2N} \se{N+2,m}
\end{equation}
if we estimate $\p_3 u_i$ in $L^\infty$ and the $\dt^{N+1}$ terms in $H^0$.  On the other hand, the relation $JK=1$, the Leibniz rule, and Lemma \ref{i_poisson_grad_bound} imply that
\begin{multline}
 \ns{\dt^{N+2} K}_{0} \ls \sum_{1 \le \gamma \le N+2} \ns{\dt^\gamma J}_{0} \ls \sum_{1 \le \gamma \le N+2} \ns{\dt^\gamma \bar{\eta}}_{1}  \ls \sum_{1 \le \gamma \le N+2} \ns{\dt^\gamma \eta}_{1/2} \\
= \sum_{1 \le \gamma \le N+1} \ns{\dt^\gamma \eta}_{1/2} + \ns{\dt^{N+2} \eta }_{1/2} \ls \se{N+2,m} + \ns{\dt^{N+2} \eta }_{1/2}.
\end{multline}
To handle the last term we must use the standard Sobolev interpolation \eqref{i_e_h_1} with $s=r=1/2$ and $q=2N-9/2$:
\begin{equation}
 \ns{\dt^{N+2} \eta }_{1/2} \ls (\ns{\dt^{N+2} \eta }_{0})^{\kappa} (\ns{\dt^{N+2} \eta }_{2N-4})^{1-\kappa} \ls (\se{N+2,m})^{\kappa} (\se{2N})^{1-\kappa}
\end{equation}
for $\kappa = (4N-8)/(4N-9)$.  Then
\begin{multline}\label{i_F_e_h_4}
 \ns{\dt^{N+2} K \p_3 u_3 }_{0} \le \ns{\dt^{N+2} K}_{0} \pnorm{\p_3 u_3}{\infty}^2 \\
\ls \se{N+2,m} \pnorm{\p_3 u_3}{\infty}^2 + (\se{N+2,m})^{\kappa} (\se{2N})^{1-\kappa} \pnorm{\p_3 u_3}{\infty}^2.
\end{multline}
For the first term on the right we bound $\pnorm{\p_3 u_3}{\infty}^2 \ls \se{2N}$, and for the second we use the $L^\infty$ interpolation bound of Proposition \ref{i_improved_u} with $r =1/2$ so that $2/(2+r) = 5/4 \ge 1-\kappa $ and $\pnorm{\p_3 u_3}{\infty}^2 \ls \se{N+2,m}^{2/(2+r)} \ls \se{N+2,m}^{1-\kappa}$.  Then these estimates and \eqref{i_F_e_h_4} imply that
\begin{equation}\label{i_F_e_h_5}
 \ns{\dt^{N+2} K \p_3 u_3 }_{0} \ls \se{N+2,m} (\se{2N})^{1-\kappa}.
\end{equation}
We then combine \eqref{i_F_e_h_2}, \eqref{i_F_e_h_3}, and \eqref{i_F_e_h_5} to see that
\begin{equation}\label{i_F_e_h_6}
 \ns{F^{2,2} }_{0} \ls \se{N+2,m} (\se{2N})^{1-\kappa}. 
\end{equation}
Then the estimate \eqref{i_F_e_h_02} follows from \eqref{i_F_e_h_1} and \eqref{i_F_e_h_6}.

\end{proof}

\subsection{Energy evolution with the highest and lowest count of temporal derivatives}

We now show the time-integrated evolution estimate for $2N$ temporal derivatives.

\begin{prop}\label{i_temporal_evolution}
There exists a $\theta>0$ so that
\begin{equation}\label{i_te_0}
 \ns{ \dt^{2N} u(t)}_{0} + \ns{\dt^{2N} \eta(t)}_{0} + \int_0^t \ns{\sg \dt^{2N} u}_{0}  \ls \se{2N}(0) +  (\se{2N}(t))^{3/2}   
+ \int_0^t \se{2N}^{\theta}  \sd{2N}.
\end{equation}
\end{prop}

\begin{proof}
We apply $\pa = \dt^{2N}$ to \eqref{geometric}.  Then $v = \dt^{2N} u$, $q= \dt^{2N} p$, and $\zeta = \dt^{2N} \eta$ solve \eqref{linear_geometric} with $F^i$, $i=1,2,3,4$ given by \eqref{F_def_start}--\eqref{F_def_end}.  Applying Lemma \ref{geometric_evolution} to these functions and then integrating in time from $0$ to $t$ gives
\begin{multline}\label{i_te_1}
 \hal \int_\Omega  J \abs{\dt^{2N} u(t)}^2  + \hal \int_\Sigma \abs{\dt^{2N} \eta(t)}^2  
+ \hal\int_0^t \int_\Omega J \abs{ \sg_{\mathcal{A}} \dt^{2N} u}^2 
= \hal \int_\Omega  J \abs{\dt^{2N} u(0)}^2 \\ + \hal \int_\Sigma \abs{\dt^{2N} \eta(0)}^2
+ \int_0^t \int_\Omega J (\dt^{2N} u \cdot F^1 + \dt^{2N} p F^2) 
+ \int_0^t \int_\Sigma - \dt^{2N} u \cdot F^3 + \dt^{2N} \eta F^4.
\end{multline}
We will estimate all of the terms involving $F^i$ on the right side of this equation.

We begin with the $F^1$ term.  According to Theorem \ref{i_F_estimates} and Lemma \ref{infinity_bounds}, we may bound
\begin{multline}\label{i_te_2}
 \int_0^t \int_\Omega J \dt^{2N} u \cdot F^{1} \le \int_0^t  \norm{\dt^{2N} u}_{0}   \pnorm{J}{\infty} \norm{F^{1}}_0 \ls \int_0^t  \sqrt{\sd{2N} } \sqrt{\se{2N} \sd{2N}} \\
= \int_0^t \sqrt{\se{2N}} \sd{2N}.
\end{multline}
Similarly, we use Theorem \ref{i_F_estimates} and trace theory to  handle the $F^3$ and $F^4$ terms:
\begin{multline}\label{i_te_3}
 \int_0^t \int_\Sigma - \dt^{2N} u \cdot F^3 + \dt^{2N} \eta F^4 \le \int_0^t \snormspace{\dt^{2N} u}{0}{\Sigma} \norm{F^3}_{0} + \norm{\dt^{2N} \eta}_{0} \norm{F^4}_{0} \\
\ls \int_0^t \left( \norm{\dt^{2N} u}_{1} + \norm{\dt^{2N} \eta}_{0} \right)\sqrt{\se{2N}\sd{2N}} \ls 
\int_0^t \sqrt{\se{2N}} \sd{2N}.
\end{multline}

For the term $\dt^{2N} p F^2$, there is one more time derivative on $p$ than can be controlled by $\sd{2N}$.  We are then forced to integrate by parts in time:
\begin{equation}
 \int_0^t \int_\Omega  \dt^{2N} p  J F^{2} = - \int_0^t \int_\Omega  \dt^{2N-1} p  \dt(J F^{2} ) + \int_\Omega (\dt^{2N-1} p  J F^{2} )(t) - \int_\Omega (\dt^{2N-1} p  J F^{2} )(0).
\end{equation}
Then according to Theorem \ref{i_F_estimates} we may estimate
\begin{multline}
 - \int_0^t \int_\Omega  \dt^{2N-1} p  \dt(J F^{2} ) \ls \int_0^t \norm{\dt^{2N-1} p }_{0} \norm{\dt(J F^{2} )}_{0} \ls \int_0^t \sqrt{\sd{2N}} \sqrt{\se{2N} \sd{2N}} \\
= \int_0^t \sqrt{\se{2N}} \sd{2N}.
\end{multline}
On the other hand, it is easy to verify using the Sobolev embeddings that
\begin{equation}
 \int_\Omega (\dt^{2N-1} p  J F^{2} )(t) - \int_\Omega (\dt^{2N-1} p  J F^{2} )(0) \ls \se{2N}(0) + (\se{2N}(t))^{3/2}.
\end{equation}
Hence
\begin{equation}\label{i_te_4}
 \int_0^t \int_\Omega  \dt^{2N} p  J F^{2} \ls \se{2N}(0) + (\se{2N}(t))^{3/2} + \int_0^t \sqrt{\se{2N}} \sd{2N}.
\end{equation}
Now we combine \eqref{i_te_2}, \eqref{i_te_3}, and \eqref{i_te_4} to deduce that
\begin{multline}\label{i_te_5}
 \hal \int_\Omega  J \abs{\dt^{2N} u(t)}^2  + \hal \int_\Sigma \abs{\dt^{2N} \eta(t)}^2  
+ \hal\int_0^t \int_\Omega J \abs{ \sg_{\mathcal{A}} \dt^{2N} u}^2  \\
\ls \se{2N}(0) + (\se{2N}(t))^{3/2} + \int_0^t \sqrt{\se{2N}} \sd{2N}.
\end{multline}

We now seek to replace $J \abs{ \sg_{\mathcal{A}} \dt^{2N} u}^2$ with $\abs{\sg \dt^{2N} u}^2$ and $J \abs{\dt^{2N} u(t)}^2$  with $\abs{\dt^{2N} u(t)}^2$ 
in \eqref{i_te_5}.  To this end we write 
\begin{equation}\label{i_te_8}
 J\abs{ \sg_{\mathcal{A}} \dt^{2N} u}^2 = \abs{\sg \dt^{2N} u}^2 + (J-1) \abs{\sg \dt^{2N} u}^2 + J \left(\sg_{\mathcal{A}} \dt^{2N} u + \sg \dt^{2N} u\right): \left(\sg_{\mathcal{A}} \dt^{2N} u - \sg \dt^{2N} u\right)
\end{equation}
and estimate the last three terms on the right side.  For the last term we  note that
\begin{equation}
 \sg_{\mathcal{A}} \dt^{2N} u \pm \sg \dt^{2N} u  = (\mathcal{A}_{ik} \pm \delta_{ik})\p_k \dt^{2N} u_j + (\mathcal{A}_{jk} \pm \delta_{jk})\p_k \dt^{2N} u_i
\end{equation}
so that Sobolev embeddings and Lemma \ref{i_poisson_grad_bound} provide the bounds
\begin{equation}
 \abs{ \sg_{\mathcal{A}} \dt^{2N} u - \sg \dt^{2N} u } \ls \sqrt{\se{2N}} \abs{\nab \dt^{2N} u} 
\text{ and }
 \abs{ \sg_{\mathcal{A}} \dt^{2N} u + \sg \dt^{2N} u } \ls (1+ \sqrt{\se{2N}}) \abs{\nab \dt^{2N} u}.
\end{equation}
We then get
\begin{multline}\label{i_te_9}
\int_0^t \int_\Omega \abs{J \left(\sg_{\mathcal{A}} \dt^{2N} u + \sg \dt^{2N} u\right): \left(\sg_{\mathcal{A}} \dt^{2N} u - \sg \dt^{2N} u\right)} \\
\ls \int_0^t (\sqrt{\se{2N}} + \se{2N}) \int_\Omega \abs{\nab \dt^{2N} u}^2  \ls \int_0^t \sqrt{\se{2N}} \sd{2N}.
\end{multline}
Similarly, 
\begin{equation}\label{i_te_10}
 \int_0^t \int_\Omega \abs{J-1} \abs{\sg \dt^{2N} u}^2 \ls \int_0^t \sqrt{\se{2N}} \sd{2N} 
\text{ and } \int_\Omega \abs{J-1} \abs{\dt^{2N} u(t)}^2 \ls  (\se{2N}(t))^{3/2}.
\end{equation}
We may then use \eqref{i_te_8} and \eqref{i_te_9}--\eqref{i_te_10} to replace in \eqref{i_te_5} and derive the bound  \eqref{i_te_0}.
\end{proof}

Now we prove a similar result for when $\dt^{N+2}$  is applied.  This time, however, we do not want an inequality that is integrated in time, so we are forced to introduce an error term involving $\dt^{N+1} p$.

\begin{prop}\label{i_temporal_evolution_half}
Let $F^2$ be given by \eqref{i_F2_def} with $\pa = \dt^{N+2}$.  Then it holds that 
\begin{equation}\label{i_teh_0}
 \dt \left( \ns{\sqrt{J} \dt^{N+2} u}_{0} + \ns{\dt^{N+2} \eta}_{0} - 2\int_\Omega J \dt^{N+1} p   F^{2}\right) + \ns{\sg \dt^{N+2} u}_{0}  \ls   \sqrt{\se{2N}} \sd{N+2,m}.
\end{equation}
\end{prop}
\begin{proof}
We apply $\pa = \dt^{N+2}$ to \eqref{geometric}.  Then $v = \dt^{N+2} u$, $q= \dt^{N+2} p$, and $\zeta = \dt^{N+2} \eta$ solve \eqref{linear_geometric} with $F^i$, $i=1,2,3,4$ given by \eqref{F_def_start}--\eqref{F_def_end}.  Applying Lemma \ref{geometric_evolution} to these functions  gives
\begin{multline}\label{i_teh_1}
\dt \left( \hal \int_\Omega  J \abs{\dt^{N+2} u}^2  + \hal \int_\Sigma \abs{\dt^{N+2} \eta}^2  \right)
+ \hal  \int_\Omega J \abs{ \sg_{\mathcal{A}} \dt^{N+2} u}^2 \\
=   \int_\Omega J (\dt^{N+2} u \cdot F^1 + \dt^{N+2} p F^2) 
+  \int_\Sigma - \dt^{N+2} u \cdot F^3 + \dt^{N+2} \eta F^4.
\end{multline}
We will estimate all of the terms involving $F^i$ on the right side of this equation as in Proposition \ref{i_temporal_evolution}.

We begin with the $F^1$ term.  According to Theorem \ref{i_F_estimates_half} and Lemma \ref{infinity_bounds}, we may bound
\begin{multline}\label{i_teh_2}
  \int_\Omega J \dt^{N+2} u \cdot F^{1} \le   \norm{\dt^{N+2} u}_{0}   \pnorm{J}{\infty} \norm{F^{1}}_0 \ls   \sqrt{\sd{N+2,m} } \sqrt{\se{2N} \sd{N+2,m}} \\
=  \sqrt{\se{2N}} \sd{N+2,m}.
\end{multline}
Similarly, we use Theorem \ref{i_F_estimates_half} and trace theory to  handle the $F^3$ and $F^4$ terms:
\begin{multline}\label{i_teh_3}
 \int_\Sigma - \dt^{N+2} u \cdot F^3 + \dt^{N+2} \eta F^4 \le  \snormspace{\dt^{N+2} u}{0}{\Sigma} \norm{F^3}_{0} + \norm{\dt^{N+2} \eta}_{0} \norm{F^4}_{0} \\
\ls  \left( \norm{\dt^{N+2} u}_{1} + \norm{\dt^{N+2} \eta}_{0} \right)\sqrt{\se{2N}\sd{N+2,m}} \ls  \sqrt{\se{2N}} \sd{N+2,m}.
\end{multline}

For the term $\dt^{N+2} p F^2$, there is one more time derivative on $p$ than can be controlled by $\sd{N+2,m}$.  We are then forced to pull out a time derivative:
\begin{equation}
 \int_\Omega  \dt^{N+2} p  J F^{2} = \dt \int_\Omega \dt^{N+1} p  J F^{2}    - \int_\Omega  \dt^{N+1} p  \dt(J F^{2} ).
\end{equation}
Then according to Theorem \ref{i_F_estimates_half} we may estimate
\begin{multline}
 - \int_\Omega  \dt^{N+1} p  \dt(J F^{2} ) \ls  \norm{\dt^{N+1} p }_{0} \norm{\dt(J F^{2} )}_{0} \ls  \sqrt{\sd{N+2,m}} \sqrt{\se{2N} \sd{N+2,m}} \\
=  \sqrt{\se{2N}} \sd{N+2,m}.
\end{multline}
Hence
\begin{equation}\label{i_teh_4}
 \int_0^t \int_\Omega  \dt^{2N} p  J F^{2} \ls \dt \int_\Omega \dt^{N+1} p  J F^{2}  +  \sqrt{\se{2N}} \sd{N+2,m}.
\end{equation}

Now we combine \eqref{i_teh_1}--\eqref{i_teh_3} and \eqref{i_teh_4} to deduce that
\begin{multline}\label{i_teh_5}
\dt \left( \hal \int_\Omega  J \abs{\dt^{N+2} u}^2  + \hal \int_\Sigma \abs{\dt^{N+2} \eta}^2 - \int_\Omega \dt^{N+1} p  J F^{2} \right)
+ \hal  \int_\Omega J \abs{ \sg_{\mathcal{A}} \dt^{N+2} u}^2 \\
\ls \sqrt{\se{2N}} \sd{N+2,m}.
\end{multline}
We may argue as in \eqref{i_te_8}--\eqref{i_te_10} of Theorem \ref{i_temporal_evolution} to show that
\begin{equation}\label{i_teh_6}
\hal  \int_\Omega  \abs{ \sg \dt^{N+2} u}^2 \ls  \hal  \int_\Omega J \abs{ \sg_{\mathcal{A}} \dt^{N+2} u}^2 + \sqrt{\se{2N}}\sd{N+2,m}.
\end{equation}
Then \eqref{i_teh_0} follows from \eqref{i_teh_5} and \eqref{i_teh_6}.

\end{proof}

Finally, we record the basic energy estimate when no derivatives are applied.

\begin{prop}\label{i_basic_energy_evolution}
It holds that 
\begin{equation}\label{i_bee_0}
 \dt \left( \hal \int_\Omega  J \abs{u}^2  + \hal \int_\Sigma \abs{\eta}^2 \right) 
+ \hal \int_\Omega J \abs{ \sg_{\mathcal{A}} u}^2 = 0.
\end{equation}
In particular
\begin{equation}\label{i_bee_00}
\ns{u(t)}_{0} + \ns{\eta(t)}_{0}  +  \int_0^t \ns{\sg u}_{0} \ls \se{2N}(0) + \int_0^t \sqrt{\se{2N}}\sd{2N}.
\end{equation}
\end{prop}
\begin{proof}
Setting $F^i=0$ in Lemma \ref{geometric_evolution} for $i=1,2,3,4$ yields \eqref{i_bee_0}.  
We may argue as in  \eqref{i_te_8}--\eqref{i_te_10} of Theorem \ref{i_temporal_evolution} to estimate
\begin{equation}
\hal  \int_\Omega  \abs{ \sg u}^2 \ls  \hal  \int_\Omega J \abs{ \sg_{\mathcal{A}}  u}^2 + \sqrt{\se{2N}}\sd{2N}.
\end{equation}
Similarly, Lemma \ref{infinity_bounds} allows us to estimate
\begin{equation}
 \frac{1}{4}\int_\Omega  \abs{u}^2 \le  \hal \int_\Omega  J\abs{u}^2.
\end{equation}
Now we may integrate \eqref{i_bee_0} in time from $0$ to $t$ and use these two estimates to derive \eqref{i_bee_00}.
\end{proof}

\section{Energy evolution in the perturbed linear form}\label{inf_5}

\subsection{Energy evolution for horizontal derivatives}

We now estimate how the evolution of the horizontal energy is coupled to the horizontal dissipation and the full energy and dissipation.

\begin{lem}\label{i_eta_evolution}
Let $\alpha \in \mathbb{N}^2$ be such that $\abs{\alpha} = 4N$, i.e. let $\pa$ be $4N$ spatial derivatives in the $x_1,x_2$ directions.  Then
\begin{equation}\label{i_ee_0} 
 \abs{ \int_\Sigma     \pa  \eta    \pa G^4 } \ls \sqrt{\se{2N}} \sd{2N} + \sqrt{\sd{2N} \k \f}.
\end{equation} 
\end{lem}

\begin{proof}

Throughout the proof $\beta$ will always denote an element of $\mathbb{N}^2$, and we will write $D f \cdot \p^\beta u = \p_1 f \p^\beta u_1 + \p_2 f \p^\beta u_2$ for a function $f$ defined on $\Sigma$.  Then by the Leibniz rule, we have that
\begin{equation}
 \pa G^4 = \pa (D \eta \cdot u) = D \pa \eta \cdot u + \sum_{\substack{0 < \beta \le \alpha \\ \abs{\beta} = 1}} C_{\alpha,\beta} D \p^{\alpha-\beta} \eta \cdot \p^\beta u + \sum_{\substack{0 < \beta \le \alpha \\ \abs{\beta} \ge 2}} C_{\alpha,\beta} D \p^{\alpha-\beta} \eta \cdot \p^\beta u
\end{equation}
for constants $C_{\alpha,\beta}$ depending on $\alpha$ and $\beta$.  We will analyze each of the three terms on the right separately.

For the first term, we integrate by parts to see that.
\begin{equation}
 \int_\Sigma \pa \eta D \pa \eta \cdot u = \hal \int_\Sigma D \abs{\pa \eta}^2 \cdot u = 
-\hal \int_{\Sigma} \pa \eta \pa \eta (\p_1 u_1 + \p_2 u_2).
\end{equation}
This then allows us to use \eqref{i_s_p_03} of Lemma \ref{i_sobolev_product_1} to bound
\begin{multline}\label{i_ee_1}
 \abs{\int_\Sigma \pa \eta D \pa \eta \cdot u} \ls \norm{\pa \eta}_{1/2} \snormspace{\pa \eta(\p_1 u_1 + \p_2 u_2)}{-1/2}{\Sigma} \\
\le \norm{\eta}_{4N + 1/2} \norm{\pa \eta}_{-1/2} \snormspace{\p_1 u_1 + \p_2 u_2}{2}{\Sigma} \\
\le \norm{\eta}_{4N + 1/2} \norm{D \eta}_{4N-3/2} \snormspace{\p_1 u_1 + \p_2 u_2}{2}{\Sigma}
\le \sqrt{\f \sd{2N} \k}.
\end{multline}
Similarly, for the second term we estimate
\begin{multline}\label{i_ee_2}
 \abs{\int_\Sigma \pa \eta \sum_{\substack{0 < \beta \le \alpha \\ \abs{\beta} = 1}} C_{\alpha,\beta} D \p^{\alpha-\beta} \eta \cdot \p^\beta u } \ls  
\norm{D^{4N} \eta}_{1/2}\norm{ D^{4N}  \eta}_{-1/2} \sum_{i=1}^2 \snormspace{D u_i}{2}{\Sigma} \\
\le \norm{  \eta}_{4N+1/2} \norm{ D   \eta}_{4N-3/2} \sum_{i=1}^2 \snormspace{D u_i}{2}{\Sigma}
\le \sqrt{\f \sd{2N} \k}.
\end{multline}

For the third term we first note that $\norm{\pa \eta}_{-1/2} \le  \norm{D \eta}_{4N-3/2} \le \sqrt{\sd{2N}}$, which allows us to bound
\begin{multline}
\abs{\int_\Sigma \pa \eta  D \p^{\alpha-\beta} \eta \cdot \p^\beta u } \le \norm{\pa \eta}_{-1/2} \snormspace{   D \p^{\alpha-\beta} \eta \cdot \p^\beta u   }{1/2}{\Sigma} \\
\le  \sqrt{\sd{2N}} \snormspace{   D \p^{\alpha-\beta} \eta \cdot \p^\beta u   }{1/2}{\Sigma}.
\end{multline}
We estimate the last term on the right using Lemma \ref{i_sobolev_product_1}, but in different ways depending on $\abs{\beta}$:
\begin{multline}
 \snormspace{   D \p^{\alpha-\beta} \eta \cdot \p^\beta u   }{1/2}{\Sigma} \ls
\begin{cases}
 \norm{D \p^{\alpha-\beta} \eta}_{1/2} \snormspace{\p^\beta u}{2}{\Sigma} & \text{for }2 \le \abs{\beta} \le 2N \\
 \norm{D \p^{\alpha-\beta} \eta}_{2} \snormspace{\p^\beta u}{1/2}{\Sigma} & \text{for }2N+1 \le \abs{\beta} \le 4N
\end{cases}
\\ \ls
\begin{cases}
 \norm{D \eta}_{4N-3/2} \norm{u}_{2N+3}  & \text{for }2 \le \abs{\beta} \le 2N \\
 \norm{D \eta}_{2N+1} \norm{u}_{4N+1}     & \text{for }2N+1 \le \abs{\beta} \le 4N
\end{cases},
\end{multline}
so that $\snormspace{   D \p^{\alpha-\beta} \eta \cdot \p^\beta u   }{1/2}{\Sigma} \ls \sqrt{\se{2N} \sd{2N}}$ for all $0 < \beta \le \alpha$ with $\abs{\beta}\ge 2$.  Hence
\begin{equation}\label{i_ee_3}
 \abs{\int_\Sigma \pa \eta \sum_{\substack{0 < \beta \le \alpha \\ \abs{\beta} \ge 2}} C_{\alpha,\beta} D \p^{\alpha-\beta} \eta \cdot \p^\beta u }
\ls \sqrt{\sd{2N}} \sqrt{\se{2N} \sd{2N}} = \sqrt{\se{2N}} \sd{2N}.
\end{equation}

The estimate \eqref{i_ee_0} then follows from \eqref{i_ee_1}, \eqref{i_ee_2}, and \eqref{i_ee_3}.
\end{proof}

Now we prove an estimate for horizontal derivatives up to order $2N$, excluding $\pa = \dt^{2N}$ and no derivatives.

\begin{prop}\label{i_derivative_evolution}
Suppose that  $\alpha \in \mathbb{N}^{1+2}$ is such that $\alpha_0 \le 2N-1$ and $1\le \abs{\alpha} \le 4N$.  Then there exists a $\theta >0$ so that
\begin{equation}\label{i_de_01}
 \dt  \left( \hal \int_\Omega \abs{\pa  u}^2  + \hal\int_\Sigma \abs{\pa \eta}^2 \right) 
+ \hal \int_\Omega \abs{\sg \pa  u}^2 
\ls  \se{2N}^{\theta} \sd{2N} + \sqrt{\sd{2N} \k \f},
\end{equation}
and in particular,
\begin{multline}\label{i_de_02}
\ns{\bar{D}_{1}^{4N-1} u}_{0} +  \ns{D \bar{D}^{4N-1} u}_{0} + \ns{\bar{D}_{1}^{4N-1} \eta}_{0}  + \ns{D \bar{D}^{4N-1} \eta}_{0} \\
+  \int_0^t \ns{\bar{D}_{1}^{4N-1}  \sg u }_{0} + \ns{D \bar{D}^{4N-1} \sg u}_{0} 
 \ls \seb{2N}(0) + \int_0^t \se{2N}^{\theta} \sd{2N} + \sqrt{\sd{2N} \k \f}.
\end{multline}

\end{prop}

\begin{proof}

Let  $\alpha \in \mathbb{N}^{1+2}$ satisfy $\alpha_0 \le 2N-1$ and $1\le \abs{\alpha} \le 4N$. Note that the constraint on $\alpha_0$ implies that we do not exceed the number of temporal derivatives of $p$ that we can control.  An application of Lemma \ref{general_evolution} to $v=  \pa u $, $q =  \pa p$, $\zeta =\pa \eta$ with $\Phi^1 =   \pa G^1$, $\Phi^2 =  \pa G^2$, $\Phi^3 =\pa G^3$, $\Phi^4 = \pa G^4$, and $a=1$ reveals that
\begin{multline}\label{i_de_1}
 \dt  \left( \hal \int_\Omega \abs{\pa u}^2  + \hal\int_\Sigma \abs{\pa \eta}^2 \right) 
+ \hal \int_\Omega \abs{\sg \pa u}^2 
= \int_\Omega  \pa  u \cdot \pa G^1  +  \pa  p \pa G^2  \\
+ \int_\Sigma -\pa u \cdot \pa G^3 + \pa \eta \pa G^4.
\end{multline}

Assume initially that   $1\le \abs{\alpha}\le 4N-1$.  Then according to the estimates \eqref{i_G_e_00}--\eqref{i_G_e_000} of Theorem \ref{i_G_estimates} and the definition of $\sd{2N}$, we have
\begin{multline}\label{i_de_2}
\abs{ \int_\Omega     \pa  u \cdot   \pa G^1 +  \pa p \pa G^2 } \le \norm{\pa  u}_{0}  \norm{ \pa G^1 }_{0} + 
\norm{\pa  p}_{0}  \norm{ \pa G^2 }_{0} \\
\ls \sqrt{\sd{2N}} \sqrt{ \se{2N}^{\theta} \sd{2N} + \k \f} \ls \se{2N}^{\kappa} \sd{2N} + \sqrt{\sd{2N} \k \f},
\end{multline}
where in the last equality we have written $\kappa = \theta/2$ for $\theta>0$ the number provided by Theorem \ref{i_G_estimates}.  Similarly, we may use Theorem \ref{i_G_estimates} along with the trace estimate $\snormspace{\pa  u}{0}{\Sigma} \ls \norm{\pa u}_{1} \le \sqrt{\sd{2N}}$ to find that
\begin{multline}\label{i_de_3}
\abs{ \int_\Sigma -\pa u \cdot \pa G^3 + \pa \eta \pa G^4 } \le 
\snormspace{\pa  u}{0}{\Sigma}  \norm{ \pa G^3 }_{0} + 
\norm{\pa  \eta}_{0}  \norm{ \pa G^4 }_{0} \\
\ls \sqrt{\sd{2N}} \sqrt{ \se{2N}^{\theta} \sd{2N} + \k \f} \ls \se{2N}^{\kappa} \sd{2N} + \sqrt{\sd{2N} \k \f}.
\end{multline}

Now assume that $\abs{\alpha}=4N$.  Since $\alpha_0 \le 2N-1$, we may write $\alpha = \beta +(\alpha-\beta)$ for some $\beta \in \mathbb{N}^2$ with $\abs{\beta}=1$, i.e. $\pa$ involves at least one spatial derivative.  Since $\abs{\alpha-\beta} = 4N-1$, we can then integrate by parts and use \eqref{i_G_e_000} of Theorem \ref{i_G_estimates} to see that
\begin{multline}\label{i_de_4}
\abs{ \int_\Omega     \pa  u \cdot   \pa G^1 } = \abs{ \int_\Omega     \p^{\alpha+\beta}  u \cdot   \p^{\alpha-\beta} G^1 } 
 \le \norm{\p^{\alpha+\beta}  u}_{0}  \norm{ \p^{\alpha-\beta} G^1 }_{0} \\
\le \norm{\p^{\alpha}  u}_{1}  \norm{ \bar{\nab}^{4N-1} G^1 }_{0} 
\ls \sqrt{\sd{2N}} \sqrt{ \se{2N}^{\theta} \sd{2N} + \k \f} \ls \se{2N}^{\kappa} \sd{2N} + \sqrt{\sd{2N} \k \f}.
\end{multline}
For the pressure term we do not need to integrate by parts:
\begin{multline}\label{i_de_5}
\abs{ \int_\Omega   \pa p \pa G^2 } 
\le \norm{\pa p}_{0}  \norm{ \p^{\alpha-\beta}\p^\beta G^1 }_{0}
\le \norm{\pa p}_{0}  \norm{ \bar{\nab}^{4N-1}   G^1 }_{1} 
  \\
\ls \sqrt{\sd{2N}} \sqrt{ \se{2N}^{\theta} \sd{2N} + \k \f} \ls \se{2N}^{\kappa} \sd{2N} + \sqrt{\sd{2N} \k \f}.
\end{multline}
We integrate by parts and use the trace estimate $H^1(\Omega) \hookrightarrow H^{1/2}(\Sigma)$ to see that
\begin{multline}\label{i_de_6}
\abs{ \int_\Sigma     \pa  u \cdot   \pa G^3 } = \abs{ \int_\Sigma  \p^{\alpha+\beta}  u \cdot   \p^{\alpha-\beta} G^3 } 
 \le \snormspace{\p^{\alpha+\beta}  u}{-1/2}{\Sigma}  \norm{ \p^{\alpha-\beta} G^3 }_{1/2} \\
\le \snormspace{\p^{\alpha}  u}{1/2}{\Sigma}  \norm{ \db^{4N-1} G^3 }_{1/2} 
 \le \norm{\p^{\alpha}  u}_{1}  \norm{ \db^{4N-1} G^3 }_{1/2} \\
\ls \sqrt{\sd{2N}} \sqrt{ \se{2N}^{\theta} \sd{2N} + \k \f} \ls \se{2N}^{\kappa} \sd{2N} + \sqrt{\sd{2N} \k \f}. 
\end{multline}
For the term $\pa \eta \pa G^4$ we must split to two cases: $\alpha_0\ge 1$ and $\alpha_0 =0$.  In the former case, there is at least one temporal derivative in $\pa$, so $ \norm{\p^{\alpha}  \eta}_{1/2}  \le \sqrt{\sd{2N}}$, and hence
\begin{multline}\label{i_de_7}
\abs{ \int_\Sigma     \pa  \eta    \pa G^4 } = \abs{ \int_\Sigma  \p^{\alpha+\beta}  \eta     \p^{\alpha-\beta} G^4 } 
 \le \norm{\p^{\alpha+\beta}  \eta}_{-1/2}   \norm{ \p^{\alpha-\beta} G^4 }_{1/2} \\
\le \norm{\p^{\alpha}  \eta}_{1/2}   \norm{ \db^{4N-1} G^3 }_{1/2} 
\ls \sqrt{\sd{2N}} \sqrt{ \se{2N}^{\theta} \sd{2N} + \k \f} \ls \se{2N}^{\kappa} \sd{2N} + \sqrt{\sd{2N} \k \f}. 
\end{multline}
In the latter case, $\alpha_0=0$, so that $\pa$ involves only spatial derivatives; in this case we use Lemma \ref{i_eta_evolution} to bound
\begin{equation}\label{i_de_8}
 \abs{ \int_\Sigma     \pa  \eta    \pa G^4 } \ls \sqrt{\se{2N}} \sd{2N} + \sqrt{\sd{2N} \k \f}.
\end{equation}

Now, in light of \eqref{i_de_1}--\eqref{i_de_8} we know that \eqref{i_de_01} holds.  The bound \eqref{i_de_02} follows by applying \eqref{i_de_01} to all $1 \le \abs{\alpha} \le 4N$ with $\alpha_0 \le 2N-1$, summing, and integrating in time from $0$ to $t$.
\end{proof}

Our next result provides some preliminary interpolation estimates for $G^2$ and $G^4$ in terms of the dissipation at the $N+2$ level, but with a power greater than $1$.

\begin{lem}\label{i_eta_evolution_half}
We have the estimate
\begin{equation}\label{i_eeh_0}
 \norm{D^{2N+3} G^4}_{1/2} \ls\left( \sd{N+2,2} \right)^{1 + 2/(4N-7)}.
\end{equation}
Also, there exists a $\theta>0$ so that
\begin{equation}\label{i_eeh_00}
\ns{D G^4}_{0} \ls \se{2N}^\theta \left(\sd{N+2,1}\right)^{1+ 1/(\lambda+2)}, \text{ and }
\ns{\bar{D}^2 G^4}_{0} \ls \se{2N}^\theta \left(\sd{N+2,2}\right)^{1+ 1/(\lambda+3)}.
\end{equation}
Finally,
\begin{equation}\label{i_eeh_000}
\pns{D G^2}{1} \ls \se{2N}^\theta \left(\sd{N+2,1}\right)^{1+ \lambda/(\lambda+2)}, \text{ and }
\pns{\bar{D}^2 G^2}{1} \ls \se{2N}^\theta \left(\sd{N+2,2}\right)^{1+ \lambda/(\lambda+3)}.
\end{equation}

\end{lem}
\begin{proof}

Let $\alpha \in \mathbb{N}^2$ be such that $\abs{\alpha}= 2(N+2)-1$.  The Leibniz rule, Lemma \ref{i_sobolev_product_1}, and trace theory imply  that
\begin{multline}\label{i_eeh_1}
 \norm{\pa G^4}_{1/2} \ls \sum_{\substack{\beta \le \alpha \\ \abs{\beta} \le N+2}} \norm{D \p^{\beta} \eta}_{2} \snormspace{\p^{\alpha-\beta} u}{1/2}{\Sigma}  + \sum_{\substack{\beta \le \alpha \\ N+3 \le \abs{\beta} \le 2N+3}} \norm{D \p^{\beta} \eta}_{1/2} \snormspace{\p^{\alpha-\beta} u}{2}{\Sigma} \\ \ls
\norm{D \eta}_{N+4} \norm{D_{N+1}^{2N+3} u}_{1} + \norm{D^3  \eta}_{2(N+2)-5/2} \snormspace{ u}{N+2}{\Sigma}.
\end{multline}

Trace theory, Poincar\'e's inequality,  and the $H^0(\Omega)$ interpolation result for $\nab u$ of Lemma \ref{i_bs_u} imply that
\begin{multline}\label{i_eeh_2}
 \snormspace{u}{N+2}{\Sigma}^2  \ls \snormspace{u}{0}{\Sigma}^2 + \snormspace{D^{N+2} u}{0}{\Sigma}^2 
\ls \ns{\nab u}_{0} + \ns{D^{N+2} u}_{1}  
\\ \le \sd{N+2,2}^{(\lambda+1)/(\lambda+3)} + 
\left( \se{2N} \right)^{(\lambda+2)/(\lambda+3)}
\left( \sd{N+2,2}    \right)^{(\lambda+1)/(\lambda+3)} \ls 
\sd{N+2,2}^{(\lambda+1)/(\lambda+3)}.
\end{multline}
Since $N\ge 5$ and $\lambda \in (0,1)$, we may define  
\begin{equation}
 q = \frac{8N + 2\lambda -8}{4N(1+ \lambda)-9\lambda -13} \in \left[ \frac{8N-6}{8N-22},\frac{8N-8}{4N-13} \right] \subset [1, 2N-9/2].
\end{equation}
Using this $q$, $r=1$ and $s = 2(N+2)-5/2$ in the standard Sobolev interpolation inequality \eqref{i_e_h_1}, we find that
\begin{multline}\label{i_eeh_3}
 \ns{D^3 \eta}_{2(N+2)-5/2} \ls 
\left( \ns{D^3 \eta }_{2(N+2)-7/2} \right)^{q/(1+q)} 
\left( \ns{D^3 \eta }_{2(N+2)-5/2+q} \right)^{1/(1+q)}
\\ \ls 
\left( \sd{N+2,2} \right)^{q/(1+q)} 
\left( \se{2N} \right)^{1/(1+q)}
\ls \left( \sd{N+2,2} \right)^{q/(1+q)}.
\end{multline}
Our choice of $q$ implies that
\begin{equation}
 \frac{\lambda +1}{\lambda +3} + \frac{q}{q+1} = 1 + \frac{2}{4N-7},
\end{equation}
so that  \eqref{i_eeh_2} and \eqref{i_eeh_3} then imply that
\begin{equation}\label{i_eeh_4}
\ns{D^3 \eta}_{2(N+2)-5/2} \snormspace{u}{N+2}{\Sigma}^2 \ls \left( \sd{N+2,2} \right)^{1 + 2/(4N-7)}.
\end{equation}

The $H^0(\Sigma)$ interpolation result for $D \eta$ of Lemma \ref{i_interp_eta} implies that
\begin{multline}\label{i_eeh_5}
 \ns{D \eta}_{N+4} \ls \ns{D \eta}_{0} + \ns{D^{3} \eta}_{N+2} \\ 
\ls \sd{N+2,2}^{(\lambda+1)/(\lambda+3)} + 
\left( \ns{D^{3} \eta}_{N+2} \right)^{(\lambda+2)/(\lambda+3)}
\left( \ns{D^{3} \eta}_{N+2} \right)^{(\lambda+1)/(\lambda+3)}
\\ \le \sd{N+2,2}^{(\lambda+1)/(\lambda+3)} + 
\left( \se{2N} \right)^{(\lambda+2)/(\lambda+3)}
\left( \sd{N+2,2}    \right)^{(\lambda+1)/(\lambda+3)} \ls 
\sd{N+2,2}^{(\lambda+1)/(\lambda+3)}.
\end{multline}
On the other hand, using the same $q$ as above and Lemma \ref{i_korn}, we have 
\begin{multline}\label{i_eeh_6}
  \norm{D_{N+1}^{2N+3} u}_{1} = 
\left( \norm{D_{N+1}^{2N+3} u}_{1} \right)^{q/(q+1)} 
\left( \norm{D_{N+1}^{2N+3} u}_{1} \right)^{1/(q+1)}
\\ \ls
\left( \sd{N+2,2} \right)^{q/(1+q)}
\left( \se{2N} \right)^{1/(1+q)}
\le \left( \sd{N+2,2} \right)^{q/(1+q)}.
\end{multline}
Then  \eqref{i_eeh_5} and \eqref{i_eeh_6} imply that
\begin{equation}\label{i_eeh_7}
 \ns{D \eta}_{N+4} \norm{D_{N+1}^{2N+3} u}_{1} \ls \left( \sd{N+2,2} \right)^{1 + 2/(4N-7)}.
\end{equation}
We then combine \eqref{i_eeh_1}, \eqref{i_eeh_4}, and \eqref{i_eeh_7} to deduce \eqref{i_eeh_0}.

We now turn to the proof of the bounds \eqref{i_eeh_00} and \eqref{i_eeh_000}.  The bounds \eqref{i_eeh_00} may be deduced by applying an operator $\pa$ with $\alpha \in \mathbb{N}^{1+2}$ satisfying either $\abs{\alpha}=1$ or $\abs{\alpha}=2$ to $G^4$, and then estimating the resulting products with one norm taken in $H^0$ and the others in $L^\infty$,  employing the $H^0$ and $L^\infty$ interpolation estimates for $\eta, u$ and their derivatives  recorded in Lemma \ref{i_interp_eta}, Proposition \ref{i_improved_u}, and Theorem \ref{i_bs_u}.  The bounds \eqref{i_eeh_000} may be deduced similarly except that at least two terms in the resulting products must be estimated in $H^0$ to deduce the resulting $L^1$ bounds.  This presents no problem since $G^2$ is a linear combination of products of two or more terms.

\end{proof}

With this lemma in place, we may record the estimates for the evolution of the energy at the $N+2$ level.

\begin{prop}\label{i_derivative_evolution_half}
Suppose that $m\in \{1,2\}$ and $\alpha \in \mathbb{N}^{1+2}$ is such that $\alpha_0 \le N+1$ and $m \le \abs{\alpha} \le 2(N+2)$.  Then there exists a $\theta >0$ so that
\begin{equation}\label{i_deh_01}
 \dt  \left(  \ns{\pa u}_{0} + \ns{\pa \eta}_{0} \right) 
+  \ns{\sg \pa  u}_{0} \\
\ls \se{2N}^{\theta} \sd{N+2,m}.
\end{equation}
In particular,
\begin{multline}\label{i_deh_02}
\dt \left( \ns{\bar{D}_{m}^{2N+3} u}_{0} +  \ns{D \bar{D}^{2N+3} u}_{0}  
+ \ns{\bar{D}_{m}^{2N+3} \eta}_{0}  + \ns{D \bar{D}^{2N+3} \eta}_{0}
\right) 
 \\
+   \ns{\bar{D}_{m}^{2N+3} \sg u }_{0} + \ns{D \bar{D}^{2N+3} \sg u}_{0} 
 \ls  \se{2N}^{\theta} \sd{N+2,m}.
\end{multline}

\end{prop}
\begin{proof}
For $m\in \{1,2\}$ and $\alpha \in \mathbb{N}^{1+2}$ such that $\alpha_0 \le N+1$ and $m \le \abs{\alpha} \le 2(N+2)$, we argue as in Proposition \ref{i_derivative_evolution} to deduce that \eqref{i_de_1} holds.  Let $X_\alpha$ denote the right hand side of \eqref{i_de_1} for our range of $\alpha$.  To bound $X_\alpha$, we break to three cases.

If $m+1 \le\abs{\alpha} \le  2(N+2)-1$ or $\abs{\alpha} = 2(N+2)$ with $1\le \alpha_0 \le N+1$, then we know from trace theory and the definitions of $\sd{N+2,m}$   that 
\begin{equation}
 \ns{\pa u}_{0} + \ns{\pa p}_{0} + \snormspace{\pa u}{1/2}{\Sigma}^2 + \ns{\pa \eta}_{1/2} \ls \sd{N+2,m}.
\end{equation}
This allows us to argue as in Proposition \ref{i_derivative_evolution}, employing Theorem \ref{i_G_estimates_half} in place of Theorem \ref{i_G_estimates}, to bound
\begin{equation}\label{i_deh_1}
\abs{X_\alpha}  \ls \se{2N}^\theta \sd{N+2,m}
\end{equation}
for some $\theta>0$.

Now consider $\abs{\alpha}= 2(N+2)$ with $\alpha_0=0$.   In this case we still know that
\begin{equation}
\ns{\pa u}_{1} + \ns{\pa p}_{0} + \snormspace{\pa u}{1/2}{\Sigma}^2 \ls \sd{N+2,m},
\end{equation}
so we may argue as in Proposition \ref{i_derivative_evolution}, integrating by parts and using these bounds as well as those from Theorem \ref{i_G_estimates_half} to show that the first, second, and third integrals in the definition of $X_\alpha$ are bounded by $\se{2N}^\theta \sd{N+2,m}$.  For the fourth integral, we  control $\ns{\pa \eta}_{1/2}$ through the interpolation estimate of Lemma \ref{i_eta_half_over}:
\begin{equation}
 \ns{\pa \eta}_{1/2} \le \ns{D^{2N+4} \eta}_{1/2} \ls 
\left( \se{2N} \right)^{2/(4N-7)}
\left( \sd{N+2,2} \right)^{(4N-9)/(4N-7)}.
\end{equation}
Then we may integrate by parts with $\alpha = \beta+ (\alpha-\beta)$, $\abs{\beta}=1$ and employ this estimate along with \eqref{i_eeh_0}  of Lemma \ref{i_eta_evolution_half} to see that 
\begin{multline}
 \abs{\int_\Sigma \pa \eta \pa G^4} =
\abs{\int_\Sigma \p^{\alpha+\beta} \eta \p^{\alpha-\beta} G^4} \le 
\norm{\p^{\alpha+\beta} \eta}_{-1/2} \norm{\p^{\alpha-\beta} G^4}_{1/2}
\\ \le 
\norm{\p^{\alpha} \eta}_{1/2} \norm{D^{2N+3} G^4}_{1/2}
\ls 
\sqrt{\left( \se{2N} \right)^{2/(4N-7)}
\left( \sd{N+2,2} \right)^{(4N-9)/(4N-7)} }
\sqrt{\left( \sd{N+2,2} \right)^{1+ 2/(4N-7)}}
\\ =
\left( \se{2N} \right)^{1/(4N-7)}
 \sd{N+2,2} \le 
\left( \se{2N} \right)^{1/(4N-7)}
 \sd{N+2,m}.
\end{multline}
Hence, when $\abs{\alpha} = 2(N+2)$ with $\alpha_0=0$ we also have that there is a $\theta>0$ so that
\begin{equation}\label{i_deh_2}
\abs{X_\alpha}  \ls \se{2N}^\theta \sd{N+2,m}.
\end{equation}

Finally, we consider the case of $\abs{\alpha}=m$ for $m=1,2$.  In this case we only know that
\begin{equation}
\ns{\pa u}_{1} +  \snormspace{\pa u}{1/2}{\Sigma}^2 \ls \sd{N+2,m},
\end{equation}
so only the first and third integrals of $X_\alpha$ may be handled directly as above to be bounded by $\se{2N}^\theta \sd{N+2,m}$.  For the fourth term we first use the $H^0(\Sigma)$ interpolation results of Lemma \ref{i_interp_eta} and Theorem \ref{i_bs_u_2} to bound
\begin{equation}
 \ns{D \eta}_{0} \ls \left( \sd{N+2,1} \right)^{(\lambda+1)/(\lambda+2)}
\text{ and }
 \ns{D^2 \eta}_{0} + \ns{\dt \eta}_{0} \ls \left( \sd{N+2,2} \right)^{(\lambda+2)/(\lambda+3)}.
\end{equation}
Then by \eqref{i_eeh_00} of Lemma \ref{i_eta_evolution_half}, we know that
\begin{multline}
 \abs{\int_\Sigma \pa \eta \pa G^4} \le \norm{\pa \eta}_{0} \norm{\pa G^4}_{0} \\
\ls 
\begin{cases}
\sqrt{ \left( \sd{N+2,1} \right)^{(\lambda+1)/(\lambda+2)}} \sqrt{  \se{2N}^\theta \left(\sd{N+2,1}\right)^{1+ 1/(\lambda+2)} } & \text{for }m=1\\
\sqrt{\left( \sd{N+2,2} \right)^{(\lambda+2)/(\lambda+3)} }
\sqrt{ \se{2N}^\theta \left(\sd{N+2,2}\right)^{1+ 1/(\lambda+3)}} & \text{for } m=2
\end{cases}
\\
\le \se{2N}^{\theta/2} \sd{N+2,m}.
\end{multline}
For the third term we first use Lemma \ref{i_slice_interp}  to bound
\begin{equation}
 \pns{D p}{\infty} \ls \left( \sd{N+2,1} \right)^{2/(\lambda+2)}
\text{ and }
 \ns{D^2 \eta}_{0} + \ns{\dt \eta}_{0} \ls \left( \sd{N+2,2} \right)^{3/(\lambda+3)}.
\end{equation}
Then by \eqref{i_eeh_000} of Lemma \ref{i_eta_evolution_half}, we know that
\begin{multline}
 \abs{\int_\Omega \pa p \pa G^2} \le \pnorm{\pa p}{\infty} \pnorm{\pa G^2}{1} \\
\ls 
\begin{cases}
\sqrt{ \left( \sd{N+2,1} \right)^{2/(\lambda+2)}} \sqrt{  \se{2N}^\theta \left(\sd{N+2,1}\right)^{1+ \lambda/(\lambda+2)} } & \text{for }m=1\\
\sqrt{\left( \sd{N+2,2} \right)^{3/(\lambda+3)} }
\sqrt{ \se{2N}^\theta \left(\sd{N+2,2}\right)^{1+ \lambda/(\lambda+3)}} & \text{for } m=2
\end{cases}
\\
\le \se{2N}^{\theta/2} \sd{N+2,m}.
\end{multline}
Hence when $\abs{\alpha}=m$ for $m=1,2$ it also holds that
\begin{equation}\label{i_deh_3}
\abs{X_\alpha}  \ls \se{2N}^\theta \sd{N+2,m}.
\end{equation}

Now, by \eqref{i_deh_1}, \eqref{i_deh_2}, and \eqref{i_deh_3} we know that \eqref{i_deh_01} holds.  The bound \eqref{i_deh_02} follows by summing \eqref{i_deh_01} over the specified range of $\alpha$.

\end{proof}

\subsection{Energy evolution for $\i_\lambda u$ and $\i_\lambda \eta$}\label{i_lambda_energy_evolution}

Before we can analyze the energy evolution for $\i_\lambda u$ and $\i_\lambda \eta$ we must first prove a lemma that provides control of $\i_{\lambda} p$.

\begin{lem}\label{i_pressure_riesz}
It holds that
\begin{equation}\label{i_p_r_0}
 \ns{\i_{\lambda} p}_{0} \ls \se{2N}, \text{ and }
\end{equation}
\begin{equation}\label{i_p_r_00}
 \ns{\i_{\lambda} D p}_{0} \ls \left(\se{2N}\right)^{\lambda/(1+\lambda)} \left(\sd{2N} \right)^{1/(1+\lambda)}.
\end{equation}
\end{lem}

\begin{proof}
Let $\alpha \in \mathbb{N}^2$ be such that $\abs{\alpha} \in \{0,1\}$.  We may apply Lemma \ref{poincare_b} to see that
\begin{equation}\label{i_p_r_1}
 \ns{\pa \i_{\lambda} p}_{0} \ls \snormspace{\pa \i_{\lambda} p}{0}{\Sigma}^2 + \ns{\p_3 \pa \i_{\lambda} p}_{0}.
\end{equation}
In order to estimate each term on the right we will use the structure of the equation \eqref{linear_perturbed}.  Indeed, using the boundary condition, we find that
\begin{equation}\label{i_p_r_2}
 \snormspace{\pa \i_{\lambda} p}{0}{\Sigma}^2 \ls \ns{\pa \i_{\lambda} \eta}_{0}  + \snormspace{\pa \i_{\lambda} \p_3 u_3}{0}{\Sigma}^2 + \ns{\pa \i_{\lambda} G^3}_{0}.
\end{equation}
Trace theory and the divergence equation in \eqref{linear_perturbed} allow us to bound
\begin{multline}\label{i_p_r_3}
\snormspace{\pa \i_{\lambda} \p_3 u_3}{0}{\Sigma}^2 \ls \ns{\pa \i_{\lambda} \p_3 u_3}_{1} 
\ls \ns{\pa \i_{\lambda} G^2}_{1} + \ns{\pa \i_{\lambda} D u}_{1} 
\\ \ls
\ns{ \i_{\lambda} D u}_{2} + \ns{\i_{\lambda} G^2}_{2}, 
\end{multline}
regardless of whether $\abs{\alpha}=0$ or $1$.  To estimate this $\i_{\lambda} Du$ term we apply Lemmas \ref{i_riesz_derivative} and \ref{poincare_usual} to see that
\begin{equation}\label{i_p_r_4}
\ns{ \i_{\lambda} D u}_{2} \ls \sum_{k=1}^2 \ns{ \i_{\lambda} D \nab^{k} u}_{0} \ls  \sum_{k=1}^2 \left(\ns{\nab^{k} u}_{0}\right)^\lambda \left(\ns{D \nab^{k} u}_{0}\right)^{1-\lambda}  \ls \ns{u}_{3}.
\end{equation}
By chaining together the bounds \eqref{i_p_r_2}--\eqref{i_p_r_4} and employing the $G^i$ estimates of Proposition \ref{i_riesz_G}, we deduce that
\begin{equation}\label{i_p_r_5}
 \snormspace{\pa \i_{\lambda} p}{0}{\Sigma}^2 \ls  \ns{\pa \i_{\lambda} \eta}_{0} + \ns{u}_{3} + \se{2N} \min\{ \se{2N},\sd{2N} \}. 
\end{equation}

Now we estimate  $\p_3 \pa \i_{\lambda} p$  by using the first equation in \eqref{linear_perturbed} to bound
\begin{equation}\label{i_p_r_6}
 \ns{ \pa \i_{\lambda} \p_3 p}_{0} \ls \ns{\pa \i_{\lambda} \dt u_3}_{0} + \ns{\pa \i_{\lambda} D^2 u}_{0} + \ns{ \pa \i_{\lambda} \p_3^2 u_3 }_{0} + \ns{ \pa \i_{\lambda} G^1}_{0}.
\end{equation}
When $\abs{\alpha}=1$ we can use Lemma \ref{i_riesz_derivative} to see that
\begin{equation}\label{i_p_r_7}
 \ns{\pa \i_{\lambda} \dt u_3}_{0} \ls \ns{ \i_{\lambda} D \dt u_3}_{0} \ls 
\left( \ns{\dt u_3}_{0} \right)^{\lambda} \left( \ns{D \dt u_3}_{0} \right)^{1-\lambda} \le \ns{\dt u}_{1}.
\end{equation}
When $\abs{\alpha}=0$ we cannot use Lemma \ref{i_riesz_derivative} directly, so we first use Poincar\'e's inequality and the divergence equation in \eqref{linear_perturbed}, and then use Lemma \ref{i_riesz_derivative}: 
\begin{multline}\label{i_p_r_8}
 \ns{ \i_{\lambda} \dt u_3}_{0} \ls \ns{ \p_3 \i_{\lambda} \dt u_3}_{0} = \ns{ \i_{\lambda} \dt \p_3 u_3}_{0} \ls 
\ns{  \i_{\lambda}  \dt G^2}_{0} + \ns{ \i_{\lambda} D \dt u}_{0} \\
\ls  \ns{  \i_{\lambda}  \dt G^2}_{0} + \ns{ \dt u}_{1}.
\end{multline}
Then \eqref{i_p_r_7} and \eqref{i_p_r_8} imply that, regardless of whether $\abs{\alpha}=0$ or $1$, we may bound
\begin{equation}\label{i_p_r_9}
 \ns{ \pa \i_{\lambda} \dt u_3 }_{0} \ls \ns{  \i_{\lambda}  \dt G^2}_{0} + \ns{ \dt u}_{1}.
\end{equation}
The term $\pa \i_{\lambda} D^2 u$ may be estimated as in \eqref{i_p_r_4}:
\begin{equation}\label{i_p_r_10}
 \ns{\pa \i_{\lambda} D^2 u}_{0} \ls \ns{u}_{3}.
\end{equation}
To estimate the term $\pa \i_{\lambda} \p_3^2 u_3$, we again use the divergence equation to bound
\begin{equation}\label{i_p_r_11}
\ns{\pa \i_{\lambda} \p_3^2 u_3}_{0} \ls \ns{\pa \i_{\lambda} \p_3 G^2}_{0} +  \ns{\pa \i_{\lambda} \p_3 D u}_{0} \ls  \ns{\pa \i_{\lambda} \p_3 G^2}_{0} + \ns{u}_{3},
\end{equation}
where in the second inequality we have again argued as in \eqref{i_p_r_4}.  Then \eqref{i_p_r_6} and \eqref{i_p_r_9}--\eqref{i_p_r_11}, together with Proposition \ref{i_riesz_G}, imply that
\begin{equation}\label{i_p_r_12}
 \ns{ \pa \i_{\lambda} \p_3 p}_{0} \ls \ns{u}_{3} + \ns{\dt u}_{1}  + \se{2N}\min\{ \se{2N},\sd{2N} \}.
\end{equation}

The estimates \eqref{i_p_r_5} and \eqref{i_p_r_12} may be combined with \eqref{i_p_r_1} to show that
\begin{equation}\label{i_p_r_13}
 \ns{\pa \i_{\lambda} p}_{0} \ls \ns{\pa \i_{\lambda} \eta}_{0} + \ns{u}_{3} + \ns{\dt u}_{1}+ \se{2N} \min\{ \se{2N},\sd{2N} \}. 
\end{equation}
When $\abs{\alpha}=0$ we bound the first three terms on the right side of \eqref{i_p_r_13} by $\se{2N}$ and use the fact that $\se{2N}^2 \le \se{2N} \le 1$ to deduce \eqref{i_p_r_0}.  When $\abs{\alpha}=1$, we first use Lemma \ref{i_sigma_interp} to bound
\begin{equation}\label{i_p_r_14}
 \ns{\pa \i_{\lambda} \eta}_{0} \le \ns{D \i_{\lambda} \eta}_{0} \ls \left( \ns{\i_{\lambda} \eta}_{0} \right)^{\lambda/(1+\lambda)} \left( \ns{D\eta}_{0} \right)^{1/(1+\lambda)} 
\ls \left( \se{2N} \right)^{\lambda/(1+\lambda)} \left( \sd{2N} \right)^{1/(1+\lambda)}.
\end{equation}
Then we  use the fact that $\se{2N} \le 1$ to bound 
\begin{multline}
 \se{2N} \min\{ \se{2N},\sd{2N} \} \le  \left( \min\{ \se{2N},\sd{2N} \}  \right)^{\lambda/(1+\lambda)} \left( \min\{ \se{2N},\sd{2N} \}  \right)^{1/(1+\lambda)}
\\ \le 
\left(  \se{2N}   \right)^{\lambda/(1+\lambda)} 
\left(  \sd{2N} \right)^{1/(1+\lambda)}.
\end{multline}
Similarly, since $\ns{u}_{3} + \ns{\dt u}_{1}  \le \min\{\se{2N},\sd{2N}\}$, we have
\begin{equation}\label{i_p_r_16}
 \ns{u}_{3} + \ns{\dt u}_{1}  \le \left(  \se{2N}   \right)^{\lambda/(1+\lambda)} 
\left(  \sd{2N} \right)^{1/(1+\lambda)}.
\end{equation}
We then combine \eqref{i_p_r_13} with \eqref{i_p_r_14}--\eqref{i_p_r_16} to deduce \eqref{i_p_r_00}.
\end{proof}

Our next lemma provides a bound for the integral of the product $\i_{\lambda} p \i_{\lambda} G^2$.  The estimate is essential to analyzing the energy evolution of $\i_{\lambda} u$ and $\i_{\lambda} \eta$. 

\begin{lem}\label{i_pressure_riesz_interaction}
It holds that
\begin{equation}\label{i_pri_0}
\abs{\int_\Omega \i_{\lambda} p \i_{\lambda} G^2 }\ls \sqrt{\se{2N}} \sd{2N}.
\end{equation}
\end{lem}

\begin{proof}
We begin by writing
\begin{equation}
 \int_\Omega \i_{\lambda} p \i_{\lambda} G^2 = I + II
\end{equation}
for 
\begin{equation}
 I := \int_\Omega \i_{\lambda} p \i_{\lambda} [(AK) \p_3 u_1 + (BK)\p_3 u_2],
\text{ and }
 II := \int_\Omega \i_{\lambda} p \i_{\lambda}  (1-K) \p_3 u_3.
\end{equation}
The term $I$ is straightforward to estimate because of the bounds \eqref{i_r_u_0} of Lemma \ref{i_riesz_u} and \eqref{i_p_r_0} of Lemma \ref{i_pressure_riesz}:
\begin{equation}\label{i_pri_1}
\abs{ I } \le \norm{\i_{\lambda} p}_{0} \norm{\i_{\lambda} [(AK) \p_3 u_1 + (BK)\p_3 u_2]} \ls \sqrt{\se{2N}}\sd{2N}.
\end{equation}

To estimate the term $II$, we must first use the divergence equation in \eqref{linear_perturbed} to rewrite
\begin{equation}
 (1-K) \p_3 u_3 = (1-K)[G^2 - \p_1 u_1 - \p_2 u_2] 
\end{equation}
so that
\begin{equation}
II =  \int_\Omega \i_{\lambda} p \i_{\lambda} [(1-K) G^2] - \int_\Omega \i_{\lambda} p \i_{\lambda} [(1-K) (\p_1 u_1 + \p_2 u_2)] := II_1 + II_2.
\end{equation}
For the term $II_1$ we use the estimates \eqref{i_p_r_0} of Lemma \ref{i_pressure_riesz} and \eqref{i_r_u_000} of Lemma \ref{i_riesz_u} to bound
\begin{equation}\label{i_pri_2}
 \abs{II_1} \le \norm{\i_{\lambda} p }_{0} \norm{\i_{\lambda} [(1-K) G^2]}_{0} \ls \sqrt{\se{2N}} \sqrt{\se{2N} \sd{2N}^2} = \se{2N} \sd{2N}.
\end{equation}
In order to control the term $II_2$ we first integrate by parts: 
\begin{equation}
II_2 =  \int_\Omega \i_{\lambda} \p_1 p \i_{\lambda} [(1-K) u_1] + \i_{\lambda} \p_2 p \i_{\lambda} [(1-K) u_2] 
-   \i_{\lambda}  p \i_{\lambda} [ u_1  \p_1 K   + u_2  \p_2 K ]. 
\end{equation}
Then we  use Lemmas \ref{i_pressure_riesz} and \ref{i_riesz_u} to estimate
\begin{multline}\label{i_pri_3}
\abs{ II_2} \le \norm{\i_{\lambda} D p}_{0} \norm{\i_{\lambda} [(1-K) u]}_{0} 
+ \norm{\i_{\lambda}  p}_{0}  \sum_{i=1}^2 \ns{\i_{\lambda} [ u \p_i K ]}_{0}
\\ \ls
\sqrt{ \left(\se{2N}\right)^{\lambda/(1+\lambda)} \left(\sd{2N} \right)^{1/(1+\lambda)}  }  \sqrt{\left( \se{2N}\right)^{1/(1+\lambda)} \left( \sd{2N}\right)^{(1+2 \lambda)/(1+\lambda)} }
+
\sqrt{\se{2N}} \sqrt{\sd{2N}^2}  \\
=\sqrt{\se{2N}}  \sd{2N}.
\end{multline}
Since $\se{2N} \le 1$, we can combine \eqref{i_pri_2} and \eqref{i_pri_3} to find that $\abs{II} \ls  \sqrt{\se{2N}}  \sd{2N}$, which yields \eqref{i_pri_0} when combined with \eqref{i_pri_1}.

\end{proof}

With these two lemmas in hand, we can now estimate how the energies of $\i_{\lambda} u$ and $\i_{\lambda} \eta$ evolve.

\begin{prop}\label{i_riesz_en_evolve}
It holds that
\begin{equation}\label{i_ree_0}
 \dt  \left( \hal \int_\Omega \abs{ \i_{\lambda} u }^2  + \hal\int_\Sigma \abs{ \i_{\lambda} \eta }^2 \right) 
+ \hal \int_\Omega \abs{\sg \i_{\lambda} u }^2 
\ls  \sqrt{\se{2N}} \sd{2N} .
\end{equation}
In particular,
\begin{equation}\label{i_ree_00}
  \hal \int_\Omega \abs{ \i_{\lambda} u }^2  +  \hal \int_\Sigma \abs{ \i_{\lambda} \eta }^2  
+ \hal \int_0^t \int_\Omega \abs{\sg \i_{\lambda} u }^2 \\ 
\ls \se{2N}(0) + \int_0^t \sqrt{\se{2N}} \sd{2N} .
\end{equation}
\end{prop}

\begin{proof}
We apply $\i_{\lambda}$ to the equations \eqref{linear_perturbed} and then use Lemma \ref{general_evolution} to see that
\begin{multline}\label{i_ree_1}
  \dt  \left( \hal \int_\Omega \abs{\i_{\lambda} u}^2  + \hal\int_\Sigma  \abs{\i_{\lambda} \eta}^2 \right) 
+ \hal \int_\Omega \abs{\sg \i_{\lambda} u}^2 
= \int_\Omega \i_{\lambda} u \cdot \i_{\lambda} G^1 + \i_{\lambda} p \i_{\lambda} G^2 
\\
+ \int_\Sigma - \i_{\lambda} u \cdot \i_{\lambda}  G^3 +  \i_{\lambda} \eta \i_{\lambda}  G^4.
\end{multline}
We will estimate each term on the right side of the equation.  First we use trace theory and \eqref{i_r_G_0} and \eqref{i_r_G_00} of Lemma \ref{i_riesz_G} to  bound the first and third terms:
\begin{multline}
\abs{\int_\Omega \i_{\lambda} u \cdot \i_{\lambda} G^1} + \abs{ \int_\Sigma  \i_{\lambda} u \cdot \i_{\lambda}  G^3}
 \le \norm{\i_{\lambda} u}_{0} \norm{\i_{\lambda} G^1}_{0} + \norm{\i_{\lambda} u}_{1} \norm{\i_{\lambda} G^3}_{0} \\
\ls \sqrt{\sd{2N}} \sqrt{\se{2N} \sd{2N}} = \sqrt{\se{2N}} \sd{2N}.
\end{multline}
For the third term we use Lemma \ref{i_pressure_riesz_interaction} for
\begin{equation}
 \abs{\int_\Omega \i_{\lambda} p \i_{\lambda} G^2 }\ls \sqrt{\se{2N}} \sd{2N}.
\end{equation}
Finally, for the fourth term we use \eqref{i_r_G_000} of Lemma \ref{i_riesz_G}:
\begin{equation}\label{i_ree_2}
 \int_\Sigma \i_{\lambda} \eta \i_{\lambda}  G^4 \le \norm{\i_{\lambda} \eta}_{0} \norm{\i_{\lambda}  G^4}_{0} \ls \sqrt{\se{2N}} \sqrt{\sd{2N}^2} = \sqrt{\se{2N}} \sd{2N}.
\end{equation}
The bound \eqref{i_ree_0} follows by combining \eqref{i_ree_1}--\eqref{i_ree_2}, and then \eqref{i_ree_00} follows from \eqref{i_ree_0} by integrating in time from $0$ to $t$.
 
\end{proof}

\section{Energy evolution estimates}\label{inf_6}

We now assemble the estimates of the previous two sections into  an estimate for the evolution of $\seb{2N}$ and $\sdb{2N}$.

\begin{thm}\label{i_evolution_estimate}
There exists a $\theta >  0$ so that
\begin{multline}
 \seb{2N}(t) + \int_0^t \sdb{2N}(r) dr \ls \se{2N}(0) + (\se{2N}(t))^{3/2} + \int_0^t (\se{2N}(r))^\theta \sd{2N}(r) dr \\
+ \int_0^t  \sqrt{\sd{2N}(r) \k(r) \f(r)} dr.
\end{multline}
\end{thm}
\begin{proof}
The result follows by summing the estimates of Propositions  \ref{i_temporal_evolution}, \ref{i_basic_energy_evolution},  \ref{i_derivative_evolution}, and \ref{i_riesz_en_evolve} and recalling the definition of $\seb{2N}$ and $\sdb{2N}$ given by \eqref{i_horizontal_energy} and \eqref{i_horizontal_dissipation}, respectively.

\end{proof}

We can also assemble the estimates of the previous two sections into a similar estimate for the evolution of $\seb{N+2,m}$ and $\sdb{N+2,m}$.

\begin{thm}\label{i_evolution_estimate_half}
Let $F^2$ be given by \eqref{i_F2_def} with $\pa = \dt^{N+2}$.  There exists a $\theta >  0$ so that
\begin{equation}
\dt \left( \seb{N+2,m} - 2 \int_\Omega J \dt^{N+1} p F^2 \right) + \sdb{N+2,m} 
\ls  \se{2N}^{\theta} \sd{N+2,m}.
\end{equation}
\end{thm}
\begin{proof}
The result follows by summing the estimates of Propositions  \ref{i_temporal_evolution_half} and \ref{i_derivative_evolution_half} and recalling the definition of $\seb{N+2,m}$ and $\sdb{N+2,m}$ given by \eqref{i_horizontal_energy_min} and \eqref{i_horizontal_dissipation_min}, respectively.

\end{proof}

\section{Comparison results}\label{inf_7}

We now prove a pair of estimates that compare the full dissipation and energy to the horizontal dissipation and energy.  We will show that, up to some error terms, the instantaneous energy $\se{2N}$ is comparable to the horizontal energy $\seb{2N}$ and that the dissipation rate $\sd{2N}$ is comparable to the horizontal dissipation rate $\sdb{2N}$.  We will also prove similar results for $\seb{N+2,m}$ and $\sdb{N+2,m}$.  To prove results for both $2N$ and $N+2$, we will first prove  general estimates involving $\sd{n}$ and $\se{n}$, and then we will specialize to the cases $n=N+2$ and $n=2N$.  The dissipation estimates are more involved, so we begin with them.

\subsection{Dissipation}

We first consider the dissipation rate.

\begin{thm}\label{i_dissipation_bound_general}
Let $m\in \{1,2\}$ and 
\begin{multline}
 \mathcal{Y}_{n,m} :=  \norm{ \bar{\nab}_m^{2n-1} G^1}_{0}^2 +  \norm{ \bar{\nab}_0^{2n-1}  G^2}_{1}^2 \\ +
 \ns{ \dbm{2n-1} G^3}_{1/2} + \ns{\bar{D}_0^{2n-1} G^4}_{1/2}
+ \ns{\bar{D}_0^{2n-2} \dt G^4}_{1/2}.
\end{multline}
If $m=1$, then
\begin{multline}\label{i_D_b_0}
 \ns{\nab^{3} u }_{2n-2}  + \sum_{j=1}^n  \ns{\dt^j  u}_{2n-2j+1}   
+ \ns{ \nab^{2} p  }_{2n-2}    
+  \sum_{j=1}^{n-1} \ns{\dt^j  p}_{2n-2j}   \\
+ \ns{D^{2} \eta}_{2n-5/2} +  \ns{\dt \eta}_{2n-1/2}
 +  \sum_{j=2}^{n+1} \ns{\dt^j \eta}_{2n-2j+ 5/2}
\ls \sdb{n,m} + \mathcal{Y}_{n,m}.
\end{multline}
If $m=2$, then  
\begin{multline}\label{i_D_b_00}
 \ns{\nab^{4} u  }_{2n-3}   + \sum_{j=1}^n  \ns{\dt^j  u}_{2n-2j+1}  
+ \ns{ \nab^{3} p  }_{2n-3}  
+ \ns{\dt  \nab  p  }_{2n-3} 
+  \sum_{j=2}^{n-1} \ns{\dt^j  p}_{2n-2j} \\
+ \ns{D^{3} \eta}_{2n-7/2} 
+ \ns{D \dt \eta}_{2n-3/2} +  \sum_{j=2}^{n+1} \ns{\dt^j \eta}_{2n-2j+ 5/2}
\ls \sdb{n,m} + \mathcal{Y}_{n,m}.
\end{multline}

\end{thm}

\begin{proof}

In this proof we must use a separate counting for spatial and temporal derivatives, so unlike elsewhere in the paper, we now only use $\alpha \in \mathbb{N}^2$ to refer to spatial derivatives.  In order to compactly write our estimates, throughout the proof we write
\begin{equation}
 \z := \sdb{n,m} + \mathcal{Y}_{n,m}.
\end{equation}
The proof is divided into several steps.

Step 1 -- Application of Korn's inequality

Since any horizontal or temporal derivative of $u$ vanishes on the lower boundary $\Sigma_b$, we may apply  Lemma \ref{i_korn} to derive the bound
\begin{equation}\label{i_D_b_1}
    \ns{ \dbm{2n} u }_{1} \ls \ns{ \dbm{2n} \sg u }_{0} = \sdb{n,m}.
\end{equation}
This $H^1(\Omega)$ bound will be more useful in what follows than an $H^0(\Omega)$ estimate of the symmetric gradient.

Step 2 -- Initial estimates of the pressure and improvement of $u$ estimates

Let $0\le j \le n-1$ and  $\alpha\in \mathbb{N}^2$ be  such that  
\begin{equation}\label{i_D_b_2}
 m \le 2j + \abs{\alpha} \le 2n -1.
\end{equation}
Note that if $2j + \abs{\alpha} = 2n-1$, then the condition $j \le n-1$ implies that $\abs{\alpha} \ge 1$.  This means that we are free to use \eqref{i_D_b_1} to bound
\begin{equation}\label{i_D_b_3}
 \ns{\partial^\alpha \dt^{j+1} u}_{0} \le \ns{\dbm{2n} u}_{1} \ls \z.
\end{equation}
In order to extract further information, we apply the operator $\dt^j \pa $ to the first two equations in \eqref{linear_perturbed} to find that
\begin{equation}\label{i_D_b_4}
  \pa \dt^{j+1}  u - \Delta \pa \dt^j  u + \nab \pa \dt^j  p = \pa \dt^j G^1
\end{equation}
\begin{equation}\label{i_D_b_5}
   \diverge{\pa \dt^j  u} = \pa \dt^j  G^2 .
\end{equation}
Because of the constraints on $j, \alpha$ given by \eqref{i_D_b_2} we may control
\begin{equation}\label{i_D_b_6}
\ns{ \pa \dt^j  G^1}_{0} +  \ns{ \pa \dt^j  G^2}_{1} \le \ns{\dbm{2n-1} G^1 }_{0} + \ns{\dbm{2n-1} G^2}_{1} \le \z.
\end{equation}
We will utilize the structure of  \eqref{i_D_b_4}--\eqref{i_D_b_5} in conjunction with \eqref{i_D_b_3} and \eqref{i_D_b_6} in order to improve our estimates.

We begin by utilizing \eqref{i_D_b_5} to control one of the terms in the third component of \eqref{i_D_b_4}.  We have 
\begin{equation}\label{i_D_b_7}
 \partial^\alpha \dt^j  (\partial_3 u_3)
=   \partial^\alpha \dt^j (-\partial_1 u_1 - \partial_2 u_2 + G^2)
\end{equation}
so that \eqref{i_D_b_1} and \eqref{i_D_b_6} imply
\begin{equation}\label{i_D_b_8}
 \ns{\p_3^2 \partial^\alpha \dt^j  u_3}_{0} \ls \ns{ \dbm{2n} u}_{1}   + \ns{ \dbm{2n-1}  G^2}_{1}  \ls \z.
\end{equation}
A further application of \eqref{i_D_b_1} to control  $(\p_1^2+\p_2^2) \pa \dt^j u_3$ then provides the estimate
\begin{equation}\label{i_D_b_9}
 \ns{\Delta \pa \dt^j  u_3}_{0}  \ls \z.
\end{equation}
  Applying the bounds  \eqref{i_D_b_3}, \eqref{i_D_b_6}, and \eqref{i_D_b_9}  to  the third component of \eqref{i_D_b_4}, we arrive at a partial bound for   the  pressure:
\begin{equation}\label{i_D_b_10} 
  \ns{\partial_3 \pa  \dt^j p }_{0} \ls \z.
\end{equation}

It remains to control the terms $\partial_i \partial^\alpha \dt^j p$ and $\partial_3^2 \partial^\alpha \dt^j u_i$ for $i=1,2$.  To accomplish this,  we employ an elliptic estimate of $\curl{u}:=\omega$.  Taking the curl of \eqref{i_D_b_4} eliminates the pressure gradient and yields
\begin{equation}
  \partial^\alpha \dt^{j+1} \omega = \Delta \partial^\alpha \dt^j \omega + \curl(\pa \dt^j G^1  ).
\end{equation}
We only need the first two components $\omega_1 = \partial_2 u_3 - \partial_3 u_2$, $\omega_2 = \partial_3 u_1 - \partial_1 u_3$, for which we use the $\Sigma$ boundary condition \eqref{linear_perturbed}
\begin{equation}
\partial_i u_3 + \partial_3 u_i  = \sg u e_3 \cdot e_i = -G^3 \cdot e_i \text{ for }i=1,2
\end{equation}
to derive the boundary conditions
\begin{equation}
\begin{cases}
 \omega_1 = 2 \partial_2 u_3 + G^3 \cdot e_2 &\text{on }\Sigma \\
 \omega_2 = -2 \partial_1 u_3 -G^3 \cdot e_1 &\text{on }\Sigma.
\end{cases}
\end{equation}
No similar boundary condition is available on $\Sigma_b$, so we must resort to a localization using a cutoff function $\chi = \chi(x_3)$ given by $\chi \in C^\infty_c(\Rn{})$ with $\chi(x_3) = 1$ for $x_3 \in \Omega_1:= [-2b/3,0]$ and $\chi(x_3)=0$ for $x_3 \notin (-3b/4,1/2)$.

The functions $\chi \omega_i$, $i=1,2$, satisfy
\begin{equation}\label{i_D_b_11}
 \Delta \partial^\alpha \dt^j (\chi \omega_i) =  \chi (\partial^\alpha \dt^{j+1}  \omega_i) + 2 (\partial_3 \chi )(\partial_3 \partial^\alpha \dt^j \omega_i)  + (\partial_3^2 \chi)( \partial^\alpha \dt^j \omega_i) - \chi \curl( \pa \dt^j G^1)
\end{equation}
in $\Omega$ as well as the boundary conditions
\begin{equation}\label{i_D_b_12}
\begin{cases}
 \partial^\alpha \dt^j (\chi \omega_1) = 2  \partial_2 \partial^\alpha \dt^j u_3 + \pa \dt^j G^3 \cdot e_2 &\text{on }\Sigma \\
  \partial^\alpha \dt^j (\chi \omega_2) = -2 \partial_1 \partial^\alpha \dt^j u_3 - \pa \dt^j G^3 \cdot e_1 &\text{on }\Sigma \\
 \partial^\alpha \dt^j(\chi \omega_1) = \partial^\alpha \dt^j(\chi \omega_2) = 0 & \text{on }\Sigma_b.
\end{cases}
\end{equation}
In order to employ an elliptic estimate of $\pa \dt^j (\chi \omega_i)$ we must first prove two auxiliary estimates.

First we derive an estimate of the $H^{-1}(\Omega)=(H^1_0(\Omega))^*$ norm of each term on the right side of equation \eqref{i_D_b_11}.  Let $\varphi \in H^1_0(\Omega)$.  When $\alpha \neq 0$ we may write $ \alpha = \beta + (\alpha-\beta)$ with $\abs{\beta} =1$ and integrate by parts to bound
\begin{equation}
 \abs{\int_\Omega \varphi \chi \partial^\alpha \dt^{j+1} \omega_i} = 
\abs{\int_\Omega \p^\beta \varphi \chi \partial^{\alpha-\beta} \dt^{j+1} \omega_i} \le 
\norm{\varphi}_{1} \norm{\chi  \dbm{2n} \omega_i}_{0}
\end{equation}
since $2(j+1) + \abs{\alpha-\beta} = 2j + \abs{\alpha} +1 \in [m+1, 2n]$.  We may use \eqref{i_D_b_1} for
\begin{equation}
 \norm{\chi  \dbm{2n} \omega_i}_{0}^2 \ls  \ns{ \dbm{2n} u}_{1} \ls \z.
\end{equation}
Chaining these inequalities together when $\alpha \neq 0$ and taking the supremum over all $\varphi$ such that $\norm{\varphi}_{1}\le 1$, we get
\begin{equation}\label{i_D_b_13}
 \snorm{\partial^\alpha \dt^{j+1} \omega_i}{-1}^2 \ls \z.
\end{equation}
A similar argument without an integration by parts shows that \eqref{i_D_b_13} is also true when $\alpha =0$ since in this case the condition $j \le n-1$ implies that $m+ 2 \le 2(j+1) \le 2n$.
Similarly integrating by parts with $\p_3$ in the dual-pairing, we may estimate the second term on the right side of \eqref{i_D_b_11}: 
\begin{equation}
 \snorm{2 (\partial_3 \chi )(\partial_3 \partial^\alpha \dt^j \omega_i) }{-1}^2 \ls (\pnorm{\partial_3 \chi}{\infty}^2 + \pnorm{\partial_3^2 \chi}{\infty}^2) \ns{ \dbm{2n} \omega_i}_{0} \ls \ns{ \dbm{2n} u}_{1} \ls \z.
\end{equation}
The third term may be estimated without integration by parts in the dual-pairing:
\begin{equation}
 \snorm{(\partial_3^2 \chi)( \partial^\alpha \dt^j \omega_i)}{-1}^2 \ls \pnorm{\partial_3^2 \chi}{\infty}^2 \ns{ \dbm{2n} \omega_i}_{0} \ls \ns{ \dbm{2n} u}_{1} \ls \z.
\end{equation}
The fourth term is estimated by integrating by parts with the $\curl$ operator and using \eqref{i_D_b_6}:
\begin{equation}
 \snorm{ \chi \curl( \pa \dt^j G^1)  }{-1}^2 \ls (\pnorm{ \chi}{\infty}^2 + \pnorm{\partial_3 \chi}{\infty}^2) \ns{ \dbm{2n-1} G^1}_{0} \ls \z.
\end{equation}
Combining these four estimates of the right hand side of \eqref{i_D_b_11} yields
\begin{equation}\label{i_D_b_14} 
 \snorm{\Delta \partial^\alpha \dt^j (\chi \omega_i)}{-1}^2 \ls \z \text{ for }i=1,2.
\end{equation}

Next, to complete the elliptic estimate of $\pa \dt^j (\chi \omega_i)$, we also need $H^{1/2}(\Sigma)$ estimates for the boundary terms on the right side of the first two equations in \eqref{i_D_b_12}.
We may estimate the $\p_i u_3$, $i=1,2$, terms with the embedding $H^1(\Omega) \hookrightarrow H^{1/2}(\Sigma)$:
\begin{equation} 
  \snormspace{\partial^\alpha \dt^j \p_1 u_3}{1/2}{\Sigma}^2 + \snormspace{\partial^\alpha \dt^j \p_2 u_3 }{1/2}{\Sigma}^2  \ls   \ns{ \dbm{2n} u   }_{1} \ls \z.
\end{equation}
On the other hand,  estimates of $G^3$ are already built into $\z$:
\begin{equation}
 \ns{\partial^\alpha \dt^j G^3 }_{1/2} \le \ns{\dbm{2n-1} G^3 }_{1/2} \le \mathcal{Y}_{n,m} \le  \z.
\end{equation}
Since $\chi \omega_i=0$ on $\Sigma_b$ for $i=1,2$ we then deduce that 
\begin{equation}\label{i_D_b_15}
 \snormspace{\partial^\alpha \dt^j (\chi \omega_i) }{1/2}{\partial \Omega}^2 \ls \z \text{ for }i=1,2.
\end{equation}

Now according to \eqref{i_D_b_14}, \eqref{i_D_b_15},  standard elliptic estimates, and the fact that $\chi = 1$ on $\Omega_1=[-2b/3,0]$  we have
\begin{equation} \label{i_D_b_16}
\snormspace{\partial^\alpha \dt^j  \omega_i}{1}{\Omega_1}^2  \ls \ns{\partial^\alpha \dt^j (\chi \omega_i)}_{1}  \ls \z \text{ for } i=1,2.
\end{equation}
We may then rewrite
\begin{equation}
 \partial_3^2 \partial^\alpha \dt^j u_1 =   \p_3 \partial^\alpha \dt^j ( \omega_2 + \partial_1 u_3 ) \text{ and }  \partial_3^2 \partial^\alpha \dt^j u_2 =   \p_3 \partial^\alpha \dt^j ( \partial_2 u_3 -\omega_1)
\end{equation}
and deduce from \eqref{i_D_b_16} and \eqref{i_D_b_1} that for $i=1,2$ we have
\begin{equation}\label{i_D_b_17}
 \snormspace{\partial_3^2 \partial^\alpha \dt^j u_i}{0}{\Omega_1}^2 \ls \ns{\dbm{2n} u_3}_{1} + \sum_{k=1}^2 \snormspace{\partial^\alpha \dt^j  \omega_k}{1}{\Omega_1}^2 \ls
\z.
\end{equation}
We then apply this estimate  along with \eqref{i_D_b_1}  and \eqref{i_D_b_6}  to the first two components of equation \eqref{i_D_b_4} to find that
\begin{equation}\label{i_D_b_18} 
 \snormspace{\partial_i \partial^\alpha \dt^j p}{0}{\Omega_1}^2 \ls \z \text{ for }i=1,2.
\end{equation}
Now we sum the estimates \eqref{i_D_b_1}, \eqref{i_D_b_8}, \eqref{i_D_b_10}, \eqref{i_D_b_17}, and \eqref{i_D_b_18} over all $j \le n-1$ and $\alpha \in \mathbb{N}^2$ with $m \le 2j + \abs{\alpha} \le 2n-1$  to deduce that
\begin{equation}\label{i_D_b_19} 
 \snormspace{\dbm{2n-1}  u}{2}{\Omega_1}^2 
+  \snormspace{\dbm{2n-1} \nab p}{0}{\Omega_1}^2 \ls \z.
\end{equation}

Step 3 -- Bootstrapping, $\eta$ estimates, and improved pressure estimates

Now we make use of Lemma \ref{i_bootstrap_estimate} to bootstrap from \eqref{i_D_b_19} to 
\begin{multline}\label{i_D_b_20}
 \snormspace{\nab^{2+m} u  }{2n-m-1}{\Omega_1}^2 +\snormspace{D^m u  }{2n-m+1}{\Omega_1}^2 + \sum_{j=1}^n  \snormspace{\dt^j  u}{2n-2j+1}{\Omega_1}^2  \\
+ \snormspace{ \nab^{1+m} p  }{2n-m-1}{\Omega_1}^2  
+  \sum_{j=1}^{n-1} \snormspace{\dt^j \nab p}{2n-2j-1}{\Omega_1}^2 \ls \z.
\end{multline}

With this estimate in hand, we may derive some estimates for $\eta$ on $\Sigma$ by employing the boundary conditions of \eqref{linear_perturbed}:
\begin{equation}\label{i_D_b_21}
 \eta = p - 2 \p_3 u_3 - G^3_3,
\end{equation}
\begin{equation}\label{i_D_b_22}
 \dt \eta = u_3 + G^4.
\end{equation}
Then \eqref{i_D_b_20} allows us to differentiate \eqref{i_D_b_21} to find that
\begin{multline}\label{i_D_b_23}
 \ns{D^{1+m} \eta}_{2n-m-3/2} \ls   \snormspace{D^{1+m} p}{2n-m-3/2}{\Sigma}^2  
+ \snormspace{ D^{1+m} \p_3 u_3}{2n-m-3/2}{\Sigma}^2 \\
+ \ns{D^{1+m} G^3}_{2n-m-3/2} 
\ls \snormspace{\nab^{1+m} p}{2n-m-1}{\Omega_1}^2  
+ \snormspace{ \nab^{2+m}  u}{2n-m-1}{\Omega_1}^2 \\
+ \ns{ G^3}_{2n - 1/2}
\ls \z.
\end{multline}
Similarly, for $j=2,\dotsc,n+1$ we may apply $\dt^{j-1}$ to \eqref{i_D_b_22} and estimate
\begin{multline}\label{i_D_b_24}
 \ns{\dt^j \eta}_{2n-2j+ 5/2} \ls \snormspace{\dt^{j-1} u_3}{2n-2j+ 5/2}{\Sigma}^2  
+ \ns{ \dt^{j-1} G^4}_{2n-2j+ 5/2} \\
\ls \snormspace{\dt^{j-1} u}{2n-2(j-1) +1}{\Omega_1}^2  
+ \ns{ \dt^{j-1} G^4}_{2n-2(j-1) + 1/2} \ls \z.
\end{multline}
It remains only to consider $\dt \eta$; in this case we must consider $m=1$ and $m=2$ separately.  For $m=1$, we again use  \eqref{i_D_b_22} to see that
\begin{equation}\label{i_D_b_24_1}
 \ns{\dt \eta}_{2n-1/2} \ls \snormspace{ u_3}{2n-1/2}{\Sigma}^2  
+ \ns{ G^4}_{2n-1/2} \ls \snormspace{ u_3}{2n-1/2}{\Sigma}^2 + \z,
\end{equation}
but now we use Lemma \ref{poincare_trace},  trace theory, and the second equation in \eqref{linear_perturbed} for the estimate
\begin{multline}\label{i_D_b_24_2}
 \snormspace{ u_3}{2n-1/2}{\Sigma}^2 \ls \snormspace{ u_3}{0}{\Sigma}^2 + \snormspace{ D u_3}{2n-3/2}{\Sigma}^2 \ls  
\snormspace{ \p_3 u_3}{0}{\Omega}^2 + \snormspace{ D u_3}{2n-1}{\Omega_1}^2 \\
\ls  \ns{ G^2}_{0} + \ns{ D u}_{0} + \snormspace{ D u}{2n-1}{\Omega_1}^2 \ls \z
\end{multline}
by \eqref{i_D_b_6} and \eqref{i_D_b_20}. Chaining \eqref{i_D_b_24_1}--\eqref{i_D_b_24_2} together implies that
\begin{equation}\label{i_D_b_24_3}
 \ns{\dt \eta}_{2n-1/2} \ls \z \text{ when }m=1.
\end{equation}
For $m=2$, we differentiate \eqref{i_D_b_22} for the bound
\begin{equation}\label{i_D_b_24_4}
 \ns{D \dt \eta}_{2n-3/2} \ls \snormspace{D u_3}{2n-3/2}{\Sigma}^2  
+ \ns{D G^4}_{2n-3/2} \ls \snormspace{D u_3}{2n-3/2}{\Sigma}^2  + \z,
\end{equation}
but then the analog of \eqref{i_D_b_24_2} is 
\begin{equation}\label{i_D_b_24_5}
 \snormspace{D u_3}{2n-3/2}{\Sigma}^2 \ls  \ns{ D G^2}_{0} +\ns{ D^2 u}_{0}+ \snormspace{ D^2 u}{2n-2}{\Omega_1}^2 \ls \z.
\end{equation}
Hence \begin{equation}\label{i_D_b_24_6}
 \ns{D \dt \eta}_{2n-3/2} \ls \z \text{ when }m=2.
\end{equation}
Summing  estimates \eqref{i_D_b_23}, \eqref{i_D_b_24}, \eqref{i_D_b_24_3}, and \eqref{i_D_b_24_6}  over $j=0,\dotsc,n+1$ yields
\begin{equation}\label{i_D_b_25}
\ns{D^{2} \eta}_{2n-5/2} + \ns{\dt \eta}_{2n-1/2}+ \sum_{j=2}^{n+1} \ns{\dt^j \eta}_{2n-2j+ 5/2} \ls \z \text{ for }m=1, \text{ and }
\end{equation}
\begin{equation}\label{i_D_b_25_2}
\ns{D^{3} \eta}_{2n-7/2} + \ns{D \dt \eta}_{2n-3/2} +  \sum_{j=2}^{n+1} \ns{\dt^j \eta}_{n-2j+ 5/2} \ls \z \text{ for }m=2.
\end{equation}

The $\eta$ estimates \eqref{i_D_b_25}--\eqref{i_D_b_25_2} now allow us to further improve the estimates for the pressure.  Indeed, for $j=2,\dotsc,n-1$ we may use Lemma \ref{poincare_b} and \eqref{i_D_b_21} to bound
\begin{multline}
 \snormspace{\dt^j p}{0}{\Omega_1}^2 
\ls \snormspace{\dt^j \eta}{0}{\Sigma}^2 +\snormspace{\partial_3 \dt^j u_3}{0}{\Sigma}^2 + \ns{ \dt^j G^3}_{0} + \snormspace{\dt^j \nab p}{0}{\Omega_1}^2  \\
\ls \snormspace{ \dt^j u_3}{2}{\Omega_1}^2 + \z \ls \z.
\end{multline}
This, \eqref{i_D_b_20}, and \eqref{i_D_b_25}--\eqref{i_D_b_25_2} allow us to improve \eqref{i_D_b_20}; when $m=1$ we find that 
\begin{multline}\label{i_D_b_26}
 \snormspace{\nab^{3} u }{2n-2}{\Omega_1}^2 +\snormspace{D u  }{2n}{\Omega_1}^2 + \sum_{j=1}^n  \snormspace{\dt^j  u}{2n-2j+1}{\Omega_1}^2  
+ \snormspace{ \nab^{2} p  }{2n-2}{\Omega_1}^2 \\  
+  \sum_{j=1}^{n-1} \snormspace{\dt^j  p}{2n-2j}{\Omega_1}^2 + 
\ns{D^{2} \eta}_{2n-5/2} + \ns{\dt \eta}_{2n-1/2} +\sum_{j=2}^{n+1} \ns{\dt^j \eta}_{2n-2j+ 5/2} \ls \z,
\end{multline}
and when $m=2$ we get the estimate
\begin{multline}\label{i_D_b_26_2}
 \snormspace{\nab^{4} u  }{2n-3}{\Omega_1}^2 +\snormspace{D^2 u  }{2n-1}{\Omega_1}^2 + \sum_{j=1}^n  \snormspace{\dt^j  u}{2n-2j+1}{\Omega_1}^2  \\
+ \snormspace{ \nab^{3} p  }{2n-3}{\Omega_1}^2  
+ \snormspace{\dt  \nab  p  }{2n-3}{\Omega_1}^2 
+  \sum_{j=2}^{n-1} \snormspace{\dt^j  p}{2n-2j}{\Omega_1}^2 \\
+ \ns{D^{3} \eta}_{2n-7/2} 
+ \ns{D \dt \eta}_{2n-3/2} +  \sum_{j=2}^{n+1} \ns{\dt^j \eta}_{n-2j+ 5/2} \ls \z.
\end{multline}

Step 4 -- Estimates in $\Omega_2$

We now extend our estimates to the lower part of the domain, i.e. $\Omega_2 := [-b,-b/3]$, by  applying Lemma \ref{i_dissipation_lower_domain} to deduce that \eqref{i_d_l_d_0} holds when $m=1$ and \eqref{i_d_l_d_00} holds when $m=2$.  We will now show that $\x_{n,m}$, defined by \eqref{i_d_l_d_000}, can be controlled by $\z$.  The key to this is that, by construction, $\supp(\nab \chi_2) \subset \Omega_1$, which implies that the $H^1$ and $H^2$ defined in the lemma satisfy $\supp(H^1)\cup \supp(H^2) \subset \Omega_1$.  This allows us to use the estimates \eqref{i_D_b_26} in the case $m=1$ and \eqref{i_D_b_26_2}  in the case $m=2$ to bound
\begin{equation}
 \ns{\bar{D}_{m+1}^{2n-1} H^1 }_0 + \ns{\bar{D}_{m+1}^{2n-1} H^2 }_0  \ls \z.
\end{equation}
In order to estimate $\dt H^1 \cdot e_i$ for $i=1,2$, we note that it does not involve the pressure:
\begin{equation}
 \dt H^1 \cdot e_i = - (\p_3 \chi_2) \p_3 \dt u_i - (\p_3^2 \chi_2) \dt u_i.
\end{equation}
Then we may again use \eqref{i_D_b_26}--\eqref{i_D_b_26_2} to see that 
\begin{equation}
\sum_{i=1}^2 \ns{\dt H^1\cdot e_i}_{2n-3} \ls \z,
\end{equation}
so that $ \x_{n,m} \ls \z.$   Replacing in \eqref{i_d_l_d_0} and \eqref{i_d_l_d_00}, we then find that
\begin{multline}\label{i_D_b_27}
 \snormspace{\nab^{3} u }{2n-2}{\Omega_2}^2  + \sum_{j=1}^n  \snormspace{\dt^j  u}{2n-2j+1}{\Omega_2}^2  \\
+ \snormspace{ \nab^{2} p  }{2n-2}{\Omega_2}^2   
+  \sum_{j=1}^{n-1} \snormspace{\dt^j  p}{2n-2j}{\Omega_2}^2  
\ls \z
\end{multline}
for $m=1$, while for $m=2$  
\begin{multline}\label{i_D_b_27_2}
 \snormspace{\nab^{4} u  }{2n-3}{\Omega_2}^2  + \sum_{j=1}^n  \snormspace{\dt^j  u}{2n-2j+1}{\Omega_2}^2  
+ \snormspace{ \nab^{3} p  }{2n-3}{\Omega_2}^2  \\
+ \snormspace{\dt  \nab  p  }{2n-3}{\Omega_2}^2 
+  \sum_{j=2}^{n-1} \snormspace{\dt^j  p}{2n-2j}{\Omega_2}^2 
\ls \z.
\end{multline}

Step 5 -- Synthesis and conclusion 

To conclude, we note that $\Omega = \Omega_1 \cup \Omega_2$, which allows us to add the localized estimates \eqref{i_D_b_26} and \eqref{i_D_b_27} to deduce \eqref{i_D_b_0}, and to add 
\eqref{i_D_b_26_2} to \eqref{i_D_b_27_2} to deduce \eqref{i_D_b_00}.

\end{proof}

We now present the key bootstrap estimate used in the proof of Theorem \ref{i_dissipation_bound_general}.

\begin{lem}\label{i_bootstrap_estimate}
Let $\mathcal{Y}_{n,m}$ and $\Omega_1$ be as defined in Theorem \ref{i_dissipation_bound_general}.   Suppose that 
\begin{equation}\label{i_b_e_0}
 \snormspace{\dbm{2n-2r+1} u}{2r}{\Omega_1}^2 +
\snormspace{\dbm{2n-2r+1} \nab p}{2r-2}{\Omega_1}^2 \ls \sdb{n,m}+ \mathcal{Y}_{n,m}
\end{equation}
for an integer $r \in [1, \dotsc,n-(m+1)/2]$.  Then
\begin{multline}\label{i_b_e_00}
 \snormspace{ \dt^{n-r} u}{2r+1}{\Omega_1}^2 
+ \snormspace{\dt^{n-r} \nab p}{2r-1}{\Omega_1}^2 \\
+\snormspace{\dbm{2n-2(r+1)+1} u}{2r+2}{\Omega_1}^2    
 + \snormspace{\dbm{2n-2(r+1)+1} \nab p}{2r}{\Omega_1}^2 \ls \sdb{n,m}+ \mathcal{Y}_{n,m}.
\end{multline}
Moreover, if \eqref{i_b_e_0} holds with $r=1$, then for $m=1,2$ we have that
\begin{multline}\label{i_b_e_000}
 \snormspace{\nab^{2+m} u  }{2n-m-1}{\Omega_1}^2 +\snormspace{D^m u  }{2n-m+1}{\Omega_1}^2 + \sum_{j=1}^n  \snormspace{\dt^j  u}{2n-2j+1}{\Omega_1}^2  \\
+ \snormspace{ \nab^{1+m} p  }{2n-m-1}{\Omega_1}^2  
+  \sum_{j=1}^{n-1} \snormspace{\dt^j \nab p}{2n-2j-1}{\Omega_1}^2
 \ls \sdb{n,m}+ \mathcal{Y}_{n,m}.
\end{multline}
\end{lem}
\begin{proof}
Throughout the proof we will write $\z := \sdb{n,m}+ \mathcal{Y}_{n,m}$.  Let $\ell \in \{1,2\}$ and take $0\le j \le n-1-r$ and  $\alpha\in \mathbb{N}^2$ so that  
$m \le 2j + \abs{\alpha} \le 2n -2r +1 -\ell$.  We apply the  differential operator $\partial_3^{2r-2+\ell} \partial^\alpha \dt^j$ to the first  equation in \eqref{linear_perturbed} and  split into separate equations for its third and first two components;  after some rearrangement, these read
\begin{equation}\label{i_b_e_1}
\partial_3^{2r-1+\ell}\pa \dt^j p =  - \partial_3^{2r-2+\ell} \pa \dt^{j+1} u_3 + \Delta  \partial_3^{2r-2+\ell} \pa \dt^j u_3 +  \partial_3^{2r-2+\ell} \pa \dt^j G^1_3
\end{equation}
and
\begin{equation}\label{i_b_e_2}
\Delta  \partial_3^{2r-2+\ell} \pa \dt^j u_i  =    \partial_3^{2r-2+\ell} \pa \dt^{j+1} u_i  + \p_i  \partial_3^{2r-2+\ell}\pa \dt^j p -  \partial_3^{2r-2+\ell} \pa \dt^j G^1_i
\end{equation}
for $i=1,2$.  Notice that the constraints on  $r, j, \abs{\alpha}$ imply that
$ m \le \abs{\alpha} + (2r -2 + \ell ) + 2j \le 2n -1,$ 
so we may  estimate 
\begin{equation}\label{i_b_e_4}
\ns{\partial_3^{2r-2+\ell}  \pa \dt^j  G^1}_{0} +  \ns{\partial_3^{2r-2+\ell} \pa \dt^j  G^2}_{1} \le \mathcal{Y}_{n,m}  \le \z.
\end{equation}
Since  $2r-2+\ell \ge 0$, we know that
\begin{equation}
 \snormspace{\partial_3^{2r-2+\ell}  \partial^\alpha \dt^{j+1} u}{0}{\Omega_1}^2 \le 
\snormspace{\partial^\alpha \dt^{j+1} u}{2r-2+\ell}{\Omega_1}^2.
\end{equation}

If $\ell=2$ then $\abs{\alpha} + 2(j+1) \le 2n -2r + 1$ so that
\begin{equation}
 \snormspace{\partial^\alpha \dt^{j+1} u}{2r-2+\ell}{\Omega_1}^2 = \snormspace{\partial^\alpha \dt^{j+1} u}{2r}{\Omega_1}^2 \le \snormspace{\dbm{2n-2r+1} u}{2r}{\Omega_1}^2 \le \z. 
\end{equation}
On the other hand, if $\ell =1$, then either $\alpha =0$, in which case the bound on $j$ implies that $2(j+1) \le 2n - 2r$, and hence
\begin{equation}
 \snormspace{\partial^\alpha \dt^{j+1} u}{2r-2+\ell}{\Omega_1}^2 = \snormspace{ \dt^{j+1} u}{2r-1}{\Omega_1}^2 \le \snormspace{\dbm{2n-2r+1} u}{2r}{\Omega_1}^2 \le \z,  
\end{equation}
or else $\abs{\alpha} \ge 1$, and so $\alpha = \beta + (\alpha-\beta)$ for $\abs{\beta}=1$, which implies that 
\begin{multline}
 \snormspace{\partial^\alpha \dt^{j+1} u}{2r-2+\ell}{\Omega_1}^2 = \snormspace{\pa \dt^{j+1} u}{2r-1}{\Omega_1}^2
\le  \snormspace{\p^{\alpha-\beta} \dt^{j+1} u}{2r}{\Omega_1}^2 \\
\le \snormspace{\dbm{2n-2r+1} u}{2r}{\Omega_1}^2 \le \z.  
\end{multline}
Then in either case, 
\begin{equation}\label{i_b_e_5}
 \snormspace{\partial_3^{2r-2+\ell}  \partial^\alpha \dt^{j+1} u}{0}{\Omega_1}^2 \le  \z.
\end{equation}

We have written the equations \eqref{i_b_e_1}--\eqref{i_b_e_2} in this form so as to be able to employ the estimates \eqref{i_b_e_0}, \eqref{i_b_e_4}, \eqref{i_b_e_5} to derive \eqref{i_b_e_00}. We must consider the case of $\ell=1$ and $\ell=2$ separately, starting with $\ell=1$.

Let $\ell=1$.  According to the equation $\diverge{u} = G^2$ (the second of \eqref{linear_perturbed}) and the bounds \eqref{i_b_e_0} and \eqref{i_b_e_4} we may estimate
\begin{multline}\label{i_b_e_6}
 \snormspace{\partial_3^{2r+1} \partial^\alpha \dt^j  u_3}{0}{\Omega_1}^2   = \snormspace{\partial_3^{2r} \partial^\alpha \dt^j  (G^2 - \p_1 u_1 - \p_2 u_2)}{0}{\Omega_1}^2   \\
\ls \norm{ \partial_3^{2r-1}  \partial^\alpha \dt^j  G^2 }_{1}^2 + \snormspace{ \partial^\alpha \dt^j  (\p_1 u_1 +\p_2 u_2)}{2r}{\Omega_1}^2
 \ls   \z,
\end{multline}
and hence
\begin{equation}\label{i_b_e_7}
 \snormspace{\Delta (\partial_3^{2r-1} \partial^\alpha \dt^j  u_3)}{0}{\Omega_1}^2 \ls \snormspace{\partial_3^{2r+1} \partial^\alpha \dt^j  u_3}{0}{\Omega_1}^2 + \snormspace{\partial_3^{2r-1}(\p_1^2+\p_2^2) \partial^\alpha \dt^j  u_3}{0}{\Omega_1}^2 \ls \z.
\end{equation}
We may then use \eqref{i_b_e_4}, \eqref{i_b_e_5}, and \eqref{i_b_e_7}  in \eqref{i_b_e_1} for the pressure estimate
\begin{equation}\label{i_b_e_8}
 \snormspace{\partial_3^{2r} \partial^\alpha \dt^j  p}{0}{\Omega_1}^2 \ls   \z.
\end{equation}
Turning now to the $i=1,2$ components, we note that by \eqref{i_b_e_0}
\begin{multline}\label{i_b_e_9}
 \snormspace{\partial_i \partial_3^{2r-1}  \partial^\alpha \dt^j p}{0}{\Omega_1}^2 + 
 \snormspace{(\partial_1^2 + \partial_2^2 ) \partial_3^{2r-1}  \partial^\alpha \dt^j  u_i}{0}{\Omega_1}^2 \\
\ls 
\snormspace{ \dbm{2n-2r+1} \nab p}{2r-2}{\Omega_1}^2 + 
 \snormspace{\dbm{2n-2r+1}  u}{2r}{\Omega_1}^2 
\ls \z
\end{multline}
for $i=1,2$.  Plugging this, \eqref{i_b_e_4}, and \eqref{i_b_e_5} into \eqref{i_b_e_2}  then shows that
\begin{equation}\label{i_b_e_10}
\snormspace{\partial_3^{2r+1} \partial^\alpha \dt^j  u_i}{0}{\Omega_1}^2 \ls   \z \text{ for }i=1,2.
\end{equation}
Upon summing   \eqref{i_b_e_6}, \eqref{i_b_e_8}, and \eqref{i_b_e_10} over $0\le j \le 2n-r-1$  and $\alpha$ satisfying  $m \le 2j + \abs{\alpha} \le 2n-2r$, we deduce, in light of \eqref{i_b_e_0}, that
\begin{equation}\label{i_b_e_11}
 \snormspace{ \dbm{2n-2r} u}{2r+1}{\Omega_1}^2 
+   \snormspace{\dbm{2n-2r} \nab p}{2r-1}{\Omega_1}^2
 \ls \z.
\end{equation}

In the case $\ell=2$ we may argue as in the case $\ell=1$, utilizing both \eqref{i_b_e_0} and \eqref{i_b_e_11}  to derive the bound 
\begin{equation}\label{i_b_e_12}
 \snormspace{ \dbm{2n-2r-1} u}{2r+2}{\Omega_1}^2 
+   \snormspace{\dbm{2n-2r-1} \nab p}{2r}{\Omega_1}^2
 \ls \z.
\end{equation}
Then we may add \eqref{i_b_e_11} to \eqref{i_b_e_12} to deduce \eqref{i_b_e_00}.

Now we turn to the proof of \eqref{i_b_e_000}, assuming that \eqref{i_b_e_0} holds with $r=1$.  By \eqref{i_b_e_00} we may iterate with $r=2$, $r=3$, etc, until
\begin{equation}
 r = 
\begin{cases}
  n -1 & \text{if } m=1 \\
  n-2  & \text{if } m=2
\end{cases}
\text{ so that } 
 2n -2(r+2) +1 =
\begin{cases}
  1 & \text{if } m=1 \\
  3  & \text{if } m=2.
\end{cases}
\end{equation}
Summing the resulting bounds yields the estimates
\begin{equation}\label{i_b_e_13}
 \snormspace{D^1_1 u  }{2n}{\Omega_1}^2 + \sum_{j=1}^n  \snormspace{\dt^j  u}{2n-2j+1}{\Omega_1}^2
+\snormspace{D^1_1 \nab p  }{2n-2}{\Omega_1}^2 +  \sum_{j=1}^{n-1} \snormspace{\dt^j \nab p}{2n-2j-1}{\Omega_1}^2
 \ls \z
\end{equation}
in the case $m=1$ and 
\begin{multline}\label{i_b_e_14}
 \snormspace{D^3_2 u  }{2n-2}{\Omega_1}^2 + \snormspace{D^1_0 \dt u  }{2n-2}{\Omega_1}^2 + \sum_{j=2}^n  \snormspace{\dt^j  u}{2n-2j+1}{\Omega_1}^2 \\
+ \snormspace{D^3_2 \nab p  }{2n-4}{\Omega_1}^2 + \snormspace{D^1_0 \dt \nab p  }{2n-4}{\Omega_1}^2 +   \sum_{j=2}^{n-1} \snormspace{\dt^j \nab p}{2n-2j-1}{\Omega_1}^2
 \ls \z
\end{multline}
in the case $m=2$. 

As a first step, we  improve the estimate \eqref{i_b_e_14}.  Let $0\le j$ and $\alpha \in \mathbb{N}^2$ be such that $2j + \abs{\alpha}=2$ and apply the operator $\p_3^{2n-3} \pa \dt^j$ to the first equation of \eqref{linear_perturbed} and split into components as above to get
\begin{equation} 
\partial_3^{2n-2} \pa \dt^j p =  - \partial_3^{2n-3} \pa \dt^{j+1} u_3 + \Delta  \partial_3^{2n-3} \pa \dt^j u_3 +  \partial_3^{2n-3} \pa \dt^j G^1_3
\end{equation}
and
\begin{equation} 
\Delta  \partial_3^{2n-3} \pa \dt^j u_i  =    \partial_3^{2n-3} \pa \dt^{j+1} u_i  + \p_i  \partial_3^{2n-3}\pa \dt^j p -  \partial_3^{2n-3} \pa \dt^j G^1_i
\end{equation}
for $i=1,2$.  We may then argue as above, utilizing \eqref{i_b_e_14} and \eqref{i_b_e_0}, to deduce the bounds
\begin{equation}\label{i_b_e_15}
  \snormspace{\partial_3^{2n-1} \pa  \dt^j  u_3}{0}{\Omega_1}^2 +  \snormspace{\partial_3^{2n-3} \pa \dt^j u}{0}{\Omega_1}^2 \ls \z,
\end{equation}
which, in turn, implies that 
\begin{equation}\label{i_b_e_16}
  \snormspace{\partial_3^{2n-2} \pa  \dt^j  p}{0}{\Omega_1}^2 +  \snormspace{\partial_3^{2n-1} \pa \dt^j u_i}{0}{\Omega_1}^2 \ls \z
\end{equation}
for $i=1,2$.  We may then use \eqref{i_b_e_15}--\eqref{i_b_e_16} with \eqref{i_b_e_14} to deduce that
\begin{equation}\label{i_b_e_17}
 \snormspace{D^2_2 u  }{2n-1}{\Omega_1}^2 + \sum_{j=1}^n  \snormspace{\dt^j  u}{2n-2j+1}{\Omega_1}^2
+\snormspace{D^2_2 \nab p  }{2n-3}{\Omega_1}^2 +  \sum_{j=1}^{n-1} \snormspace{\dt^j \nab p}{2n-2j-1}{\Omega_1}^2
 \ls \z
\end{equation}
in the case $m=2$.

Now we claim that if for $m=1,2$ we have the inequality
\begin{multline}\label{i_b_e_18}
 \snormspace{D^m_m u  }{2n-m+1}{\Omega_1}^2 + \sum_{j=1}^n  \snormspace{\dt^j  u}{2n-2j+1}{\Omega_1}^2  \\
+\snormspace{D^m_m \nab p  }{2n-m-1}{\Omega_1}^2 +  \sum_{j=1}^{n-1} \snormspace{\dt^j \nab p}{2n-2j-1}{\Omega_1}^2
 \ls \z,
\end{multline}
then the inequality 
\begin{equation}\label{i_b_e_19}
 \snormspace{\nab^{2+m} u  }{2n-m-1}{\Omega_1}^2 
+\snormspace{ \nab^{1+m} p  }{2n-m-1}{\Omega_1}^2 
 \ls \z
\end{equation}
also holds, which establishes the desired bound, \eqref{i_b_e_000}, because of our inequalities \eqref{i_b_e_13} in the case $m=1$ and \eqref{i_b_e_17} in the case $m=2$.  We begin the proof of the claim by noting that since $2 \ge m$  we may use \eqref{i_b_e_18} to bound
\begin{equation}\label{i_b_e_20}
 \snormspace{\p_3^m D^2 u  }{2n-m-1}{\Omega_1}^2 
+ \snormspace{ \p_3^{m-1} D^2 p  }{2n-m-1}{\Omega_1}^2 
 \ls \z.
\end{equation}

Now we let $\abs{\alpha}=1$ and apply $\p_3^m \pa$ to the  second equation of \eqref{linear_perturbed} to find that
\begin{equation}\label{i_b_e_21}
 \snormspace{\p_3^{m+1} \pa u_3  }{2n-m-1}{\Omega_1}^2 \ls \snormspace{\p_3^m D G^2  }{2n-m-1}{\Omega_1}^2 + \snormspace{\p_3^m D^2 u  }{2n-m-1}{\Omega_1}^2   \ls \z.
\end{equation}
Then we apply $\p_3^{m-1} \pa$ to the first equation of \eqref{linear_perturbed} to bound
\begin{multline}\label{i_b_e_22}
 \snormspace{\p_3^{m} \pa p  }{2n-m-1}{\Omega_1}^2 
\ls \snormspace{\p_3^{m+1} \pa u_3  }{2n-m-1}{\Omega_1}^2 \\
+ \snormspace{\p_3^{m-1} \pa \bar{D}_2^2 u_3  }{2n-m-1}{\Omega_1}^2
+ \snormspace{\p_3^{m-1} \pa G^2  }{2n-m-1}{\Omega_1}^2
\ls \z
\end{multline}
and
\begin{multline}\label{i_b_e_23}
 \snormspace{\p_3^{m+1} \pa u_i  }{2n-m-1}{\Omega_1}^2 
\ls \snormspace{\p_3^{m-1} \pa \bar{D}_2^2 u }{2n-m-1}{\Omega_1}^2 \\
+ \snormspace{\p_3^{m-1} \pa D p  }{2n-m-1}{\Omega_1}^2
+ \snormspace{\p_3^{m-1} \pa G^2  }{2n-m-1}{\Omega_1}^2
\ls \z
\end{multline}
for $i=1,2$.  Summing \eqref{i_b_e_21}--\eqref{i_b_e_23} over all $\abs{\alpha}=1$ then yields the inequality
\begin{equation}\label{i_b_e_24}
 \snormspace{\p_3^{m+1} D u  }{2n-m-1}{\Omega_1}^2 
+ \snormspace{ \p_3^{m} D p  }{2n-m-1}{\Omega_1}^2 
 \ls \z.
\end{equation}

Now we use \eqref{i_b_e_24} to improve to one more $\p_3$ and one fewer horizontal derivative.   We  apply $\p_3^{m+1}$ to the second equation of \eqref{linear_perturbed} to find that
\begin{equation}\label{i_b_e_25}
 \snormspace{\p_3^{m+2}  u_3  }{2n-m-1}{\Omega_1}^2 \ls \snormspace{\p_3^{m+1}  G^2  }{2n-m-1}{\Omega_1}^2 + \snormspace{\p_3^{m+1} D u  }{2n-m-1}{\Omega_1}^2   \ls \z.
\end{equation}
Then we apply $\p_3^{m} $ to the first equation of \eqref{linear_perturbed} to bound
\begin{multline}\label{i_b_e_26}
 \snormspace{\p_3^{m+1}  p  }{2n-m-1}{\Omega_1}^2 
\ls \snormspace{\p_3^{m+2}  u_3  }{2n-m-1}{\Omega_1}^2 \\
+ \snormspace{\p_3^{m}  \bar{D}_2^2 u_3  }{2n-m-1}{\Omega_1}^2
+ \snormspace{\p_3^{m}  G^2  }{2n-m-1}{\Omega_1}^2
\ls \z
\end{multline}
and
\begin{multline}\label{i_b_e_27}
 \snormspace{\p_3^{m+2}  u_i  }{2n-m-1}{\Omega_1}^2 
\ls \snormspace{\p_3^{m}  \bar{D}_2^2 u }{2n-m-1}{\Omega_1}^2 \\
+ \snormspace{\p_3^{m}  D p  }{2n-m-1}{\Omega_1}^2
+ \snormspace{\p_3^{m}  G^2  }{2n-m-1}{\Omega_1}^2
\ls \z
\end{multline}
for $i=1,2$.  Summing \eqref{i_b_e_25}--\eqref{i_b_e_27} then yields the inequality
\begin{equation}\label{i_b_e_28}
 \snormspace{\p_3^{m+2}  u  }{2n-m-1}{\Omega_1}^2 
+ \snormspace{ \p_3^{m+1}  p  }{2n-m-1}{\Omega_1}^2 
 \ls \z.
\end{equation}

Finally, to complete the proof of the claim, we note that
\begin{multline}
 \snormspace{\nab^{2+m} u  }{2n-m-1}{\Omega_1}^2 
+\snormspace{ \nab^{1+m} p  }{2n-m-1}{\Omega_1}^2  
\ls  \snormspace{D^m_m u  }{2n-m+1}{\Omega_1}^2 +\snormspace{D^m_m \nab p  }{2n-m-1}{\Omega_1}^2 \\
+ \sum_{l=1}^2 \snormspace{\p_3^{m+2-\ell} D^\ell u  }{2n-m-1}{\Omega_1}^2 
+ \snormspace{ \p_3^{m+1-\ell} D^\ell p  }{2n-m-1}{\Omega_1}^2.
\end{multline}
This and the bounds \eqref{i_b_e_18}, \eqref{i_b_e_20}, \eqref{i_b_e_24}, and \eqref{i_b_e_28} prove the claim.

\end{proof}

The following result allows for control of the dissipation rate in the lower domain.

\begin{lem}\label{i_dissipation_lower_domain}
Let $\chi_2 \in C_c^\infty(\Rn{})$ be such that $\chi_2(x_3) =1$ for $x_3 \in \Omega_2 := [-b,-b/3]$ and $\chi_2(x_3) =0$ for  $x_3 \notin (-2b,-b/6)$.  Let 
\begin{equation}
 H^1 = \p_3 \chi_2 (p e_3 -2 \p_3 u) -(\p_3^2 \chi_2) u \text{ and } H^2 = \p_3 \chi_2 u_3.
\end{equation}
Define 
\begin{equation}\label{i_d_l_d_000}
\x_{n,m} = \ns{\bar{D}_{m+1}^{2n-1} H^1 }_0 + \ns{\bar{D}_{m+1}^{2n-1} H^2 }_0 + \sum_{i=1}^2 \ns{\dt H^1\cdot e_i}_{2n-3},
\end{equation}
and let $\mathcal{Y}_{n,m}$ be as defined in Theorem \ref{i_dissipation_bound_general}.  If $m=1$, then
\begin{multline}\label{i_d_l_d_0}
 \snormspace{\nab^{3} u }{2n-2}{\Omega_2}^2  + \sum_{j=1}^n  \snormspace{\dt^j  u}{2n-2j+1}{\Omega_2}^2  \\
+ \snormspace{ \nab^{2} p  }{2n-2}{\Omega_2}^2   
+  \sum_{j=1}^{n-1} \snormspace{\dt^j  p}{2n-2j}{\Omega_2}^2  
\ls \sdb{n,m} +\mathcal{Y}_{n,m} + \x_{n,m}.
\end{multline}
If $m=2$, then
\begin{multline}\label{i_d_l_d_00}
 \snormspace{\nab^{4} u  }{2n-3}{\Omega_2}^2  + \sum_{j=1}^n  \snormspace{\dt^j  u}{2n-2j+1}{\Omega_2}^2  
+ \snormspace{ \nab^{3} p  }{2n-3}{\Omega_2}^2  \\
+ \snormspace{\dt  \nab  p  }{2n-3}{\Omega_2}^2 
+  \sum_{j=2}^{n-1} \snormspace{\dt^j  p}{2n-2j}{\Omega_2}^2 
\ls \sdb{n,m} +\mathcal{Y}_{n,m} + \x_{n,m} .
\end{multline}
\end{lem}

\begin{proof}

When we localize with $\chi_2$ we find that $\chi_2 u$ and $\chi_2 p$ solve
\begin{equation}\label{i_d_l_d_05}
 \begin{cases}
   - \Delta (\chi_2 u) + \nab (\chi_2 p) = - \dt (\chi_2 u) + \chi_2 G^1  + H^1 &\text{in }\Omega \\
  \diverge(\chi_2 u) = \chi_2 G^2 + H^2 &\text{in }\Omega\\
  ((\chi_2 p) I - \sg (\chi_2 u)) e_3 = 0  &\text{on }\Sigma\\
 \chi_2 u = 0 & \text{on } \Sigma_b.
 \end{cases}
 \end{equation}
Let $0\le j\le n-1$ and $\alpha \in \mathbb{N}^2$ be so that $m+1 \le \abs{\alpha}+2j \le  2n-1.$  Then we may apply Lemma \ref{i_linear_elliptic} to see that
\begin{multline}\label{i_d_l_d_1}
 \ns{\pa \dt^j (\chi_2 u ) }_{2n-\abs{\alpha} - 2j+1} + \ns{\pa \dt^j (\chi_2 p)  }_{2n-\abs{\alpha}-2j}   \ls 
\ns{\pa \dt^{j+1} (\chi_2 u)  }_{2n-\abs{\alpha}-2(j+1) +1} \\
+ \ns{\pa \dt^j ( \chi_2  G^1 + H^1)  }_{2n-\abs{\alpha}-2j-1} 
+ \ns{\pa \dt^j (\chi_2 G^2 + H^2) }_{2n-\abs{\alpha}-2j} \\
\ls \ns{\pa \dt^{j+1} (\chi_2 u)  }_{2n-\abs{\alpha}-2(j+1) +1} + \mathcal{Y}_{n,m} + \x_{n,m}.
\end{multline}
We first use estimate \eqref{i_d_l_d_1} and a finite induction to arrive at initial estimates for $\chi_2 u$ and $\chi_2 p$; we will then use the structure of the equations \eqref{linear_perturbed} to improve these estimates. 

Our finite induction will be performed on $\ell \in [1,2n-m-1]$,  starting with the first two initial values, $\ell=1$ and $\ell=2$.  We use the definition of $\sdb{n,m}$ and Lemma \ref{i_korn} in conjunction with the bounds on $j, \abs{\alpha}$ to  see that
\begin{equation}
 \ns{\pa \dt^{j+1} (\chi_2 u)  }_{0} \ls \ns{\pa \dt^{j+1}  u  }_{0} \ls \sdb{n,m}.
\end{equation}
Then \eqref{i_d_l_d_1} with $\abs{\alpha} + 2j =2n-1 = 2n-\ell$ implies that
\begin{equation}
 \ns{\pa \dt^{j} (\chi_2 u ) }_{2} + \ns{\pa \dt^{j} (\chi_2 p)  }_{1}  \ls \ns{\pa \dt^{j+1} (\chi_2 u)  }_{0} + \mathcal{Y}_{n,m} + \x_{n,m} \ls  \sdb{n,m} +\mathcal{Y}_{n,m} + \x_{n,m}.
\end{equation}
Applying this bound for  all  $\alpha$ and $j$ satisfying $\abs{\alpha}+2j =2n-1$ and summing, we find
\begin{equation}\label{i_d_l_d_2}
 \ns{\bar{D}_{2n-1}^{2n-1} (\chi_2 u ) }_{2} + \ns{\bar{D}_{2n-1}^{2n-1} (\chi_2 p)  }_{1}  \ls  \sdb{n,m} +\mathcal{Y}_{n,m} + \x_{n,m}. 
\end{equation}
When $\ell=2$ and $\abs{\alpha}+2j = 2n-\ell = 2n-2$, a similar application of Lemma \ref{i_korn} implies 
\begin{equation}
 \ns{\pa \dt^{j+1} (\chi_2 u)  }_{1}  \ls \sdb{n,m}
\end{equation}
so that
\begin{equation}
 \ns{\pa \dt^{j} (\chi_2 u ) }_{3} + \ns{\pa \dt^{j} (\chi_2 p)  }_{3}  \ls \ns{\pa \dt^{j+1} (\chi_2 u)  }_{1} + \mathcal{Y}_{n,m} + \x_{n,m} \ls  \sdb{n,m} +\mathcal{Y}_{n,m} + \x_{n,m}.
\end{equation}
This may be summed over $2j + \abs{\alpha} = 2n -2$ for the estimate 
\begin{equation}\label{i_d_l_d_3}
 \ns{\bar{D}_{2n-2}^{2n-2} (\chi_2 u ) }_{3} + \ns{\bar{D}_{2n-2}^{2n-2} (\chi_2 p)  }_{2}  \ls  \sdb{n,m} +\mathcal{Y}_{n,m} + \x_{n,m}. 
\end{equation}
Then \eqref{i_d_l_d_2} and \eqref{i_d_l_d_3} imply that
\begin{equation}\label{i_d_l_d_4}
 \ns{\bar{D}_{2n-1}^{2n-1} (\chi_2 u ) }_{2}  +  \ns{\bar{D}_{2n-2}^{2n-2} (\chi_2 u ) }_{3}  + \ns{\bar{D}_{2n-1}^{2n-1} (\chi_2 p)  }_{1}  +\ns{\bar{D}_{2n-2}^{2n-2} (\chi_2 p)  }_{2}   \ls  \sdb{n,m} +\mathcal{Y}_{n,m} + \x_{n,m}. 
\end{equation}

Now suppose that the inequality
\begin{equation}\label{i_d_l_d_5}
 \sum_{\ell=1}^{\ell_0} \ns{\bar{D}_{2n-\ell}^{2n-\ell} (\chi_2 u ) }_{\ell+1}  + \ns{\bar{D}_{2n-\ell}^{2n-\ell} (\chi_2 p)  }_{\ell}    
 \ls  \sdb{n,m} +\mathcal{Y}_{n,m} + \x_{n,m} 
\end{equation}
holds for $2 \le \ell_0 < 2n-m-1$.  We claim that \eqref{i_d_l_d_5} holds with $\ell_0$ replaced by $\ell_0+1$.  Suppose $\abs{\alpha} +2j = 2n - (\ell_0+1)$  and apply \eqref{i_d_l_d_1}  to see that
\begin{multline}
 \ns{\pa \dt^{j} (\chi_2 u ) }_{\ell_0+2} + \ns{\pa \dt^{j} (\chi_2 p)  }_{\ell_0+1}  \ls \ns{\pa \dt^{j+1} (\chi_2 u)  }_{\ell_0} + \mathcal{Y}_{n,m} + \x_{n,m} \\
 \ls  \sdb{n,m} +\mathcal{Y}_{n,m} + \x_{n,m},
\end{multline}
where in the last inequality we have invoked \eqref{i_d_l_d_5} with $\abs{\alpha}+2(j+1) = 2n-(\ell_0+1)+2=2n - (\ell_0-1)$.  This proves the claim, so by  finite induction the bound \eqref{i_d_l_d_5} holds for all $\ell_0 = 2,\dotsc,2n-m-1$. 
Choosing $\ell_0 = 2n-m-1$ yields the estimate
\begin{equation}\label{i_d_l_d_6}
 \sum_{\ell=1}^{2n-m-1} \ns{\bar{D}_{2n-\ell}^{2n-\ell} (\chi_2 u ) }_{\ell+1}  + \ns{\bar{D}_{2n-\ell}^{2n-\ell} (\chi_2 p)  }_{\ell}    
 \ls  \sdb{n,m} +\mathcal{Y}_{n,m} + \x_{n,m},
\end{equation}
which implies, by virtue of the fact that $\chi_2=1$ on $\Omega_2$, that
\begin{equation}\label{i_d_l_d_7}
 \sum_{\ell=1}^{2n-m-1} \snormspace{\bar{D}_{2n-\ell}^{2n-\ell}   u  }{\ell+1}{\Omega_2}^2  + \snormspace{\bar{D}_{2n-\ell}^{2n-\ell}  p  }{\ell}{\Omega_2}^2    
 \ls  \sdb{n,m} +\mathcal{Y}_{n,m} + \x_{n,m}.
\end{equation}

Now we will improve the estimate \eqref{i_d_l_d_7} by using the equations \eqref{linear_perturbed}, considering the cases $m=1,2$ separately.  Let $m=1$.  Since $m+1=2$, the bound \eqref{i_d_l_d_7} already covers all temporal derivatives, so we must only improve spatial derivatives.
First note that \eqref{i_d_l_d_7} implies that
\begin{equation}\label{i_d_l_d_8}
\snormspace{ \p_3 D^{2}   u  }{2n-2}{\Omega_2}^2  + \snormspace{D^{2}  p  }{2n-2}{\Omega_2}^2 \ls \sdb{n,m} +\mathcal{Y}_{n,m} + \x_{n,m}.
\end{equation}
Then we may apply the operator $\p_3 D$ to the divergence equation in \eqref{linear_perturbed} to bound
\begin{equation}
 \snormspace{ \p_3^2 D   u_3  }{2n-2}{\Omega_2}^2 \ls \snormspace{ \p_3 D   G^2  }{2n-2}{\Omega_2}^2 + \snormspace{\p_3 D^{2}   u  }{2n-2}{\Omega_2}^2 \ls \sdb{n,m} +\mathcal{Y}_{n,m} + \x_{n,m}.
\end{equation}
Then applying the operator $D$ to the first equation in \eqref{linear_perturbed} implies that
\begin{multline}
 \snormspace{ \p_3 D  p  }{2n-2}{\Omega_2}^2 + \snormspace{ \p_3^2 D   u_i  }{2n-2}{\Omega_2}^2
\ls \snormspace{  D   G^1  }{2n-2}{\Omega_2}^2 + \snormspace{ D^2  p  }{2n-2}{\Omega_2}^2 \\
 + \snormspace{  D \bar{D}_2^2  u  }{2n-2}{\Omega_2}^2 + \snormspace{ \p_3^2 D   u_3  }{2n-2}{\Omega_2}^2 \ls \sdb{n,m} +\mathcal{Y}_{n,m} + \x_{n,m}
\end{multline}
for $i=1,2$.  We can then iterate this process, applying $\p_3^2$ to the divergence equation, then $\p_3$ to the first equation in \eqref{linear_perturbed}, and using all of the bounds derived from the previous step, to deduce that
\begin{equation}\label{i_d_l_d_9}
 \snormspace{ \p_3^2   p  }{2n-2}{\Omega_2}^2 + \snormspace{ \p_3^3   u  }{2n-2}{\Omega_2}^2
 \ls \sdb{n,m} +\mathcal{Y}_{n,m} + \x_{n,m}.
\end{equation}
Combining \eqref{i_d_l_d_8}--\eqref{i_d_l_d_9} yields the estimate
\begin{equation}
 \snormspace{ \nab^3 u  }{2n-2}{\Omega_2}^2 + \snormspace{ \nab^2   p  }{2n-2}{\Omega_2}^2
 \ls \sdb{n,m} +\mathcal{Y}_{n,m} + \x_{n,m},
\end{equation}
which together with \eqref{i_d_l_d_7} implies \eqref{i_d_l_d_0}.

In the case $m=2$, we can argue as in the case $m=1$ to control the spatial derivatives.  That is, we first control $\p_3 D^3 u, D^3 p$, then iteratively apply operators with an increasing number of $\p_3$ powers to arrive at the bound
\begin{equation}\label{i_d_l_d_10}
 \snormspace{ \nab^4 u  }{2n-3}{\Omega_2}^2 + \snormspace{ \nab^3   p  }{2n-3}{\Omega_2}^2
 \ls \sdb{n,m} +\mathcal{Y}_{n,m} + \x_{n,m}.
\end{equation}
It remains to control $\dt u$ and $\dt \nab p$.  For the latter we apply $\p_3 \dt$ to the divergence equation to bound
\begin{equation}
 \snormspace{ \p_3^2 \dt   u_3  }{2n-3}{\Omega_2}^2 \ls \snormspace{ \p_3 \dt   G^2  }{2n-3}{\Omega_2}^2 + \snormspace{\p_3 \dt D   u  }{2n-3}{\Omega_2}^2 \ls \sdb{n,m} +\mathcal{Y}_{n,m} + \x_{n,m}.
\end{equation}
Then applying $\dt$ to the third component of the first equation in \eqref{linear_perturbed} shows that
\begin{multline}
 \snormspace{ \p_3 \dt  p  }{2n-3}{\Omega_2}^2 
\ls \snormspace{  \dt   G^1  }{2n-3}{\Omega_2}^2  
 + \snormspace{  \dt \bar{D}_2^2  u  }{2n-3}{\Omega_2}^2 + \snormspace{ \p_3^2 \dt   u_3  }{2n-3}{\Omega_2}^2 \\
 \ls \sdb{n,m} +\mathcal{Y}_{n,m} + \x_{n,m},
\end{multline}
which in turn implies that 
\begin{equation}\label{i_d_l_d_11}
 \snormspace{ \nab \dt  p  }{2n-3}{\Omega_2}^2 \ls \snormspace{ \p_3 \dt  p  }{2n-3}{\Omega_2}^2 + \snormspace{ D \dt  p  }{2n-3}{\Omega_2}^2 \ls \sdb{n,m} +\mathcal{Y}_{n,m} + \x_{n,m}.
\end{equation}
We may control $\dt u_3$ by applying $\dt$ to the divergence equation in \eqref{linear_perturbed} to find that
\begin{equation}
 \snormspace{ \p_3 \dt  u_3  }{2n-2}{\Omega_2}^2 \ls \snormspace{ \dt  G^2  }{2n-2}{\Omega_2}^2 + \snormspace{ \bar{D}_3^3  u  }{2n-2}{\Omega_2}^2 \ls \sdb{n,m} +\mathcal{Y}_{n,m} + \x_{n,m},
\end{equation}
but then since $\dt u_3=0$ on $\Sigma$ we can use Poincar\'e's inequality (Lemma \ref{poincare_usual}) to bound 
\begin{multline}\label{i_d_l_d_15}
\snormspace{ \dt  u_3  }{2n-1}{\Omega_2}^2 \ls  \snormspace{ \p_3 \dt  u_3  }{2n-2}{\Omega_2}^2 + \snormspace{ \dt u_3}{0}{\Omega_2}^2 + \snormspace{D_1^{2n-1} \dt u_3}{0}{\Omega_2}^2 \\
\ls \snormspace{ \p_3 \dt  u_3  }{2n-2}{\Omega_2}^2 + \snormspace{D_0^{2n-1} \dt u_3}{1}{\Omega_2}^2
\ls \sdb{n,m} +\mathcal{Y}_{n,m} + \x_{n,m}.
\end{multline}
Control of the terms $\dt u_i$, $i=1,2$ is slightly more delicate; for it we appeal to the first of the localized equations \eqref{i_d_l_d_05} rather than \eqref{linear_perturbed}.  The reason for this is that using \eqref{i_d_l_d_05} will allow us to control $\p_3^2\dt (\chi_2 u_i)$ in all of $\Omega$, which  will give us control of $\dt (\chi_2 u_i)$ in all of $\Omega$ via Poincar\'e and hence control of $\dt u_i$ in $\Omega_2$.  If instead we used \eqref{linear_perturbed}, then control of $\p_3^2\dt u_i$ in $\Omega_2$ would not yield the desired control of $\dt u_i$ in $\Omega_2$ because we could not apply Poincar\'e's inequality.  We apply $\dt$ to the $i=1,2$ components of the first localized equation in \eqref{i_d_l_d_05} and use \eqref{i_d_l_d_6} to see that
\begin{multline}\label{i_d_l_d_12}
 \snormspace{ \p_3^2 \dt (\chi_2 u_i) }{2n-3}{\Omega}^2 
\ls 
\snormspace{  \dt H^1\cdot e_i  }{2n-3}{\Omega}^2 +
\snormspace{  \chi_2 \dt G^1  }{2n-3}{\Omega}^2 \\
+  \snormspace{ \dt D (\chi_2 p)  }{2n-3}{\Omega}^2 
 + \snormspace{  \dt \bar{D}_2^2  (\chi_2 u) }{2n-3}{\Omega}^2 
 \ls 
\sdb{n,m} +\mathcal{Y}_{n,m} + \x_{n,m}.
\end{multline}
Now, since $\dt(\chi_2 u_i)$ and $\p_3 \dt (\chi_2 u_i)$ both vanish in an open set near $\Sigma$, we may apply Poincar\'e's inequality twice and use \eqref{i_d_l_d_12} to find that
\begin{multline}\label{i_d_l_d_14}
\snormspace{  \dt u_i }{2n-1}{\Omega_2}^2  \ls  \snormspace{  \dt (\chi_2 u_i) }{2n-1}{\Omega}^2  \ls \snormspace{ \p_3^2 \dt (\chi_2 u_i) }{2n-3}{\Omega}^2 \\
\ls \sdb{n,m} +\mathcal{Y}_{n,m} + \x_{n,m} +   \snormspace{  \dt u  }{2n-2}{\Omega_1}^2.
\end{multline}
To conclude the analysis for $m=2$ we sum \eqref{i_d_l_d_10}, \eqref{i_d_l_d_11}, \eqref{i_d_l_d_15}, and \eqref{i_d_l_d_14} to derive \eqref{i_d_l_d_00}.

\end{proof}

\subsection{Instantaneous energy}

Now we estimate the instantaneous energy.  The proof is based on an argument very similar to the one used in the proof of Lemma \ref{i_dissipation_lower_domain}.

\begin{thm}\label{i_energy_bound_general}
Define
\begin{equation}
 \mathcal{W}_{n,m} =  \norm{ \bar{\nab}_m^{2n-2} G^1}_{0}^2 +  \norm{ \bar{\nab}_0^{2n-2}  G^2}_{1}^2 +
 \ns{ \dbm{2n-2} G^3}_{1/2} + \ns{\bar{D}_0^{2n-2} G^4}_{1/2}.
\end{equation}
If $m=1$, then
\begin{multline}\label{i_E_b_0}
\ns{\nab^2 u}_{2n-2} + \sum_{j=1}^{n} \ns{\dt^j u}_{2n-2j} +
\ns{\nab p}_{2n-2} + \sum_{j=1}^{n-1} \ns{\dt^j p}_{2n-2j-1} 
\\ +
\ns{D \eta}_{2n-1} + \sum_{j=1}^{n} \ns{\dt^j \eta}_{2n-2j}
\ls \seb{n,m} +\mathcal{W}_{n,m}.
\end{multline} 
If $m=2$, then
\begin{multline}\label{i_E_b_00}
\ns{\nab^3 u}_{2n-3} + \sum_{j=1}^{n} \ns{\dt^j u}_{2n-2j} +
\ns{\nab^2 p}_{2n-3} + \sum_{j=1}^{n-1} \ns{\dt^j p}_{2n-2j-1} 
\\ +
\ns{D^2 \eta}_{2n-2} + \sum_{j=1}^{n} \ns{\dt^j \eta}_{2n-2j}
\ls \seb{n,m} +\mathcal{W}_{n,m}.
\end{multline}

\end{thm}
\begin{proof}
The proof is quite similar to that of Lemma \ref{i_dissipation_lower_domain}, so we will not fill in all of the details.  Throughout the proof we will employ the notation $\z:= \seb{n,m} + \mathcal{W}_{n,m}$.

Let $0 \le j \le n-1$ and $\alpha \in \mathbb{N}^2$ satisfy $m \le \abs{\alpha} + 2j \le 2n-2$.  To begin, we utilize the equations \eqref{linear_perturbed} with the elliptic estimate Lemma \ref{i_linear_elliptic} to bound
\begin{multline}\label{i_E_b_1}
 \ns{\pa \dt^j  u  }_{2n-\abs{\alpha} - 2j} + \ns{\pa \dt^j  p }_{2n-\abs{\alpha}-2j-1}   \ls 
\ns{\pa \dt^{j+1}   u  }_{2n-\abs{\alpha}-2j-2} 
+ \ns{\pa \dt^j  G^1 }_{2n-\abs{\alpha}-2j-2} \\
+ \ns{\pa \dt^j  G^2 }_{2n-\abs{\alpha}-2j-1} 
+ \ns{\pa \dt^j  \eta }_{2n-\abs{\alpha}-2j-3/2}
+ \ns{\pa \dt^j  G^3 }_{2n-\abs{\alpha}-2j-3/2}.
\end{multline}
The constraints on $j, \alpha$ allow us to bound
\begin{equation}\label{i_E_b_2}
 \ns{\pa \dt^j  G^1 }_{2n-\abs{\alpha}-2j-2} +  \ns{\pa \dt^j  G^2 }_{2n-\abs{\alpha}-2j-1} + \ns{\pa \dt^j  G^3 }_{2n-\abs{\alpha}-2j-3/2} \ls \mathcal{W}_{n,m},
\end{equation}
and similarly 
\begin{equation}\label{i_E_b_3}
 \ns{\pa \dt^j  \eta }_{2n-\abs{\alpha}-2j-3/2} \ls \seb{n,m},
\end{equation}
so that \eqref{i_E_b_1}--\eqref{i_E_b_3} imply that
\begin{equation}\label{i_E_b_4}
 \ns{\pa \dt^j  u  }_{2n-\abs{\alpha} - 2j} + \ns{\pa \dt^j  p }_{2n-\abs{\alpha}-2j-1}   \ls \z + \ns{\pa \dt^{j+1}   u  }_{2n-\abs{\alpha}-2j-2}.
\end{equation}
As in Lemma \ref{i_dissipation_lower_domain}, we argue with a finite induction on $\ell \in [2,2n-m]$, beginning with $\ell =2,3$.  When $\ell =2$ and $\abs{\alpha}+2j = 2n-2 = 2n-\ell$, the definition of $\seb{n,m}$ implies that
\begin{equation}
 \ns{\pa \dt^{j+1}   u  }_{0} \ls \seb{n,m},
\end{equation}
which may be inserted into \eqref{i_E_b_4} for
\begin{equation}
 \ns{\pa \dt^j  u  }_{2} + \ns{\pa \dt^j  p }_{1}   \ls \z.
\end{equation}
Summing over all $\alpha$ and $j$ satisfying $\abs{\alpha}+2j = 2n -2$ shows that
\begin{equation}\label{i_E_b_5}
 \ns{\bar{D}_{2n-2}^{2n-2}  u  }_{2} + \ns{\bar{D}_{2n-2}^{2n-2}  p }_{1}   \ls \z.
\end{equation}
For $\ell =3$ we note that $\abs{\alpha}+2j = 2n-3$ implies that $j \le n-2$, so that $\abs{\alpha}\ge 1$.  This allows us to write $\alpha = (\alpha-\beta) + \beta$ for $\abs{\beta}=1$ and to use \eqref{i_E_b_5} to see that
\begin{equation}
\ns{\pa \dt^{j+1}  u }_{1} \le \ns{\p^{\alpha-\beta} \dt^{j+1}  u }_{2} \le  \ns{\bar{D}_{2n-2}^{2n-2}  u  }_{2}   \ls \seb{n,m}.
\end{equation}
Then we can plug this into \eqref{i_E_b_4} for each $\abs{\alpha}+2j=2n-3$ and sum to arrive at the bound
\begin{equation}\label{i_E_b_6}
 \ns{\bar{D}_{2n-3}^{2n-3}  u  }_{3} + \ns{\bar{D}_{2n-3}^{2n-3}  p }_{2}   \ls \z.
\end{equation}
Now we may use finite induction as in \eqref{i_d_l_d_5}--\eqref{i_d_l_d_7} of Lemma \ref{i_dissipation_lower_domain} to ultimately deduce the estimate
\begin{equation}\label{i_E_b_7}
 \sum_{\ell=2}^{2n-m} \ns{\bar{D}_{2n-\ell}^{2n-\ell}  u  }_{\ell} + \ns{\bar{D}_{2n-\ell}^{2n-\ell}  p }_{\ell-1}   \ls \z.
\end{equation}

Now we improve the estimate \eqref{i_E_b_7} by utilizing the structure of the equations \eqref{linear_perturbed}, again arguing as in Lemma \ref{i_dissipation_lower_domain}.  The energy bound \eqref{i_E_b_7} in the case $m=2$ is structurally similar to the bound \eqref{i_d_l_d_7} for the dissipation in the case $m=1$, so we may argue as in \eqref{i_d_l_d_8}--\eqref{i_d_l_d_9}, differentiating the equations \eqref{linear_perturbed} (with obvious modifications to the Sobolev indices and number of derivatives applied) and bootstrapping until we arrive at the bound
\begin{equation}\label{i_E_b_8}
 \ns{\nab^3 u}_{2n-3} + \ns{\nab^2 p}_{2n-3} \ls \z.
\end{equation}
Then \eqref{i_E_b_7} and \eqref{i_E_b_8} imply the bound \eqref{i_E_b_0}.

In the case $m=1$ we apply $\p_3$ to the divergence equation in \eqref{linear_perturbed} to see that
\begin{equation}\label{i_E_b_9}
 \ns{\p_3^2 u_3}_{2n-2} \ls  \ns{\p_3 G^2}_{2n-2} + \ns{\p_3 D u}_{2n-2} \ls \z.
\end{equation}
We then use the first equation in \eqref{linear_perturbed} to bound
\begin{equation}\label{i_E_b_10}
 \ns{\p_3 p}_{2n-2}+  \sum_{i=1}^2 \ns{\p_3^2 u_i}_{2n-2}  \ls \ns{  G^1}_{2n-2} + \ns{\p_3 D u}_{2n-2} + \ns{Dp}_{2n-2} \ls \z.
\end{equation}
Then \eqref{i_E_b_7}, \eqref{i_E_b_9}, and \eqref{i_E_b_10} imply that
\begin{equation}
 \ns{\nab^2 u}_{2n-2} + \ns{\nab p}_{2n-2} \ls \z,
\end{equation}
and hence that \eqref{i_E_b_00} holds.
\end{proof}

\subsection{Specialization: estimates at the $2N$ and $N+2$ levels }

We now specialize the general results contained in Theorems \ref{i_dissipation_bound_general} and \ref{i_energy_bound_general} to the specific cases of $n=2N$ with no minimal derivative restriction, and to the case $n=N+2$ with minimal derivative count $m=1,2$.

\begin{thm}\label{i_dissipation_bound_2n}
There exists a $\theta >0$ so that
\begin{equation}\label{i_db2n_0}
 \sd{2N} \ls \sdb{2N} + \se{2N}^\theta \sd{2N}   + \k \f.
\end{equation}
\end{thm}
\begin{proof}
We apply Theorem \ref{i_dissipation_bound_general} with $n=2N$ and $m=1$ to see that \eqref{i_D_b_0} holds.  Theorem \ref{i_G_estimates} provides the estimate
\begin{equation}
\mathcal{Y}_{2N,1} \ls  \se{2N}^\theta \sd{2N}   + \k \f
\end{equation}
for some $\theta >0$.  We may then use this in \eqref{i_D_b_0} to find that
\begin{multline}\label{i_db2n_1}
 \ns{\nab^{3} u }_{4N-2}  + \sum_{j=1}^{2N}  \ns{\dt^j  u}_{4N-2j+1}   
+ \ns{ \nab^{2} p  }_{4N-2}    
+  \sum_{j=1}^{2N-1} \ns{\dt^j  p}_{4N-2j}   \\
+ \ns{D^{2} \eta}_{4N-5/2} + \ns{\dt  \eta}_{4N-1/2} +\sum_{j=2}^{2N+1} \ns{\dt^j \eta}_{4N-2j+ 5/2}
\ls \sdb{2N} + \se{2N}^\theta \sd{2N}   + \k \f.
\end{multline}

We can improve the estimate for $u$ in \eqref{i_db2n_1} by using the fact that $\sdb{2N}$ does not have a minimal derivative count.  Indeed, by definition, we know  that
\begin{equation}\label{i_db2n_2}
 \ns{\i_{\lambda} u}_{1} + \ns{u}_{1} \ls \sdb{2N}.
\end{equation}
Now, since $\Omega$ satisfies the uniform cone property, we can apply Corollary 4.16 of \cite{adams} to bound
\begin{equation}\label{i_db2n_3}
 \ns{u}_{4N+1} \ls \ns{u}_{0} + \ns{\nab^{4N+1} u}_{0} \ls \ns{u}_{1} + \ns{\nab^3 u}_{4N-2}.
\end{equation}
Then \eqref{i_db2n_1}--\eqref{i_db2n_3} imply that
\begin{equation}\label{i_db2n_4}
\ns{\i_{\lambda} u}_{1}+  \ns{u}_{4N+1} \ls \sdb{2N} + \se{2N}^\theta \sd{2N}   + \k \f.
\end{equation}

We can use this improved estimate of $u$ to improve the estimate of $p$ by employing the first equation of \eqref{linear_perturbed} to bound
\begin{equation}\label{i_db2n_9}
 \ns{\nab p}_{4N-1} \ls \ns{\dt u}_{4N-1} + \ns{\Delta u}_{4N-1} + \ns{G^1}_{4N-1}.
\end{equation}
The bounds \eqref{i_db2n_1} and \eqref{i_db2n_4} imply that
\begin{equation}\label{i_db2n_5}
  \ns{\dt u}_{4N-1} + \ns{\Delta u}_{4N-1} \ls \sdb{2N} + \se{2N}^\theta \sd{2N}   + \k \f,
\end{equation}
while \eqref{i_G_e_00}--\eqref{i_G_e_000} of Theorem \ref{i_G_estimates} imply that
\begin{equation}\label{i_db2n_6}
 \ns{G^1}_{4N-1} \ls \se{2N}^\theta \sd{2N}   + \k \f.
\end{equation}
Hence \eqref{i_db2n_4}--\eqref{i_db2n_6} combine to show that 
\begin{equation}\label{i_db2n_7}
 \ns{\nab p}_{4N-1} \ls \sdb{2N} + \se{2N}^\theta \sd{2N}   + \k \f.
\end{equation}

Finally, we improve the estimate for $\eta$.  We use the boundary condition on $\Sigma$ of \eqref{linear_perturbed} to bound
\begin{multline}\label{i_db2n_8}
 \ns{D \eta}_{4N-3/2}  \ls \snormspace{D p}{4N-3/2}{\Sigma}^2 +  \snormspace{D \p_3 u_3 }{4N-3/2}{\Sigma}^2 +  \ns{D G^3}_{4N-3/2} \\
\ls \ns{D p}_{4N-1}  +  \ns{D \p_3 u_3 }_{4N-1}  +  \ns{D G^3}_{4N-3/2} 
\le \sdb{2N} + \se{2N}^\theta \sd{2N}   + \k \f.
\end{multline}
In the last inequality we have used \eqref{i_db2n_4}, \eqref{i_db2n_7}, and Theorem \ref{i_G_estimates}.  Now \eqref{i_db2n_0} follows from \eqref{i_db2n_1}, \eqref{i_db2n_4}, \eqref{i_db2n_7}, and \eqref{i_db2n_8}.
\end{proof}

Now we perform a similar analysis for the energy at the $2N$ level.

\begin{thm}\label{i_energy_bound_2n}
There exists a $\theta >0$ so that
\begin{equation}\label{i_eb2n_0}
 \se{2N} \ls \seb{2N} + \se{2N}^{1+\theta}. 
\end{equation} 
\end{thm}
\begin{proof}
 We apply Theorem \ref{i_energy_bound_general} with $n=2N$ and $m=1$ to see that \eqref{i_E_b_0} holds.  Theorem \ref{i_G_estimates} provides the estimate
\begin{equation}
\mathcal{W}_{2N,1} \ls  \se{2N}^{1+\theta}   
\end{equation}
for some $\theta >0$.  Replacing in \eqref{i_E_b_0} shows that
\begin{multline}\label{i_eb2n_1}
\ns{\nab^2 u}_{4N-2} + \sum_{j=1}^{2N} \ns{\dt^j u}_{4N-2j} +
\ns{\nab p}_{4N-2} + \sum_{j=1}^{2N-1} \ns{\dt^j p}_{4N-2j-1} 
\\ +
\ns{D \eta}_{4N-1} + \sum_{j=1}^{2N} \ns{\dt^j \eta}_{4N-2j}
\ls \seb{2N} +\se{2N}^{1+\theta}.
\end{multline} 
The  definition of $\seb{2N}$ implies that
\begin{equation}
 \ns{\il u}_{0} + \ns{u}_{0} + \ns{\il \eta}_{0} + \ns{\eta }_{0} \le \seb{2N}.
\end{equation}
We may then sum the previous two bounds and employ Corollary 4.16 of \cite{adams} as in the proof of Theorem \ref{i_dissipation_bound_2n} to find that
\begin{multline}\label{i_eb2n_2}
\ns{\il u}_{0} + \sum_{j=0}^{2N} \ns{\dt^j u}_{4N-2j} +
\ns{\nab p}_{4N-2} + \sum_{j=1}^{2N-1} \ns{\dt^j p}_{4N-2j-1} 
 \\+
\ns{\il \eta}_{0}+  \sum_{j=0}^{2N} \ns{\dt^j \eta}_{4N-2j}
\ls \seb{2N} +\se{2N}^{1+\theta}.
\end{multline} 
It remains only to estimate $\ns{p}_{4N-1}$; since Lemma \ref{poincare_b} implies that 
\begin{equation}\label{i_eb2n_3}
\ns{p}_{4N-1} \ls \ns{p}_{0} + \ns{\nab p}_{4N-2} \ls \snormspace{p}{0}{\Sigma}^2 + \ns{\nab p}_{4N-2}, 
\end{equation}
it suffices to estimate $\snormspace{p}{0}{\Sigma}^2$.  We do this by using the boundary condition in \eqref{linear_perturbed}, trace theory, and estimate \eqref{i_G_e_0} of Theorem \ref{i_G_estimates}:
\begin{equation}\label{i_eb2n_4}
\snormspace{p}{0}{\Sigma}^2 \ls \ns{\eta}_{0} + \ns{G^3}_{0} + \snormspace{\p_3 u_3}{0}{\Sigma}^2
\ls \ns{\eta}_{0} + \ns{u}_{4N} + \se{2N}^{1+\theta}.
\end{equation}
Then the estimate \eqref{i_eb2n_0} easily follows from \eqref{i_eb2n_2}--\eqref{i_eb2n_4}.

\end{proof}

We now consider the dissipation at the $N+2$ level.

\begin{thm}\label{i_dissipation_bound_n}
For $m=1,2$ there exists a $\theta >0$  so that
\begin{equation}\label{i_dbn_0}
 \sd{N+2,m} \ls \sdb{N+2,m} + \se{2N}^{\theta} \sd{N+2,m}.
\end{equation}  
\end{thm}
\begin{proof}
 We apply Theorem \ref{i_dissipation_bound_general} with $n=N+2$ to see that \eqref{i_D_b_0} holds for $m=1$ and \eqref{i_D_b_00} holds for $m=2$.  Theorem \ref{i_G_estimates_half} provides the estimate
\begin{equation}
\mathcal{Y}_{N+2,m} \ls  \se{2N}^\theta \sd{N+2,m}
\end{equation}
for some $\theta >0$.  The bound \eqref{i_dbn_0} follows from using this in \eqref{i_D_b_0}--\eqref{i_D_b_00}.
\end{proof}

We now consider the energy at the $N+2$ level.

\begin{thm}\label{i_energy_bound_n}
For $m=1,2$ there exists a $\theta >0$  so that
\begin{equation}\label{i_ebn_0}
 \se{N+2,m} \ls \seb{N+2,m} + \se{2N}^{\theta} \se{N+2,m}.
\end{equation}   
\end{thm}
\begin{proof}
We apply Theorem \ref{i_energy_bound_general} with $n=N+2$ to see that \eqref{i_E_b_0} holds when $m=1$ and \eqref{i_E_b_00} holds when $m=2$.  Theorem \ref{i_G_estimates_half} provides the estimate
\begin{equation}
\mathcal{W}_{N+2,m} \ls  \se{2N}^{\theta} \se{N+2,m}  
\end{equation}
for some $\theta >0$.  The bound \eqref{i_ebn_0} follows from using this in \eqref{i_E_b_0}--\eqref{i_E_b_00}.
\end{proof}

\section{A priori estimates}\label{inf_8}

In this section we will combine the energy evolution estimates and the comparison estimates to derive a priori estimates for the total energy, $\g$, defined by \ref{i_total_energy}.

\subsection{Estimates involving  $\f$ and $\k$}

We begin with an estimate for $\f$.

\begin{lem}\label{i_specialized_transport_estimate}
There exists a $C>0$ so that
\begin{multline}\label{i_ste_0}
\sup_{0\le r \le t}  \f(r) \ls \exp\left(C \int_0^t \sqrt{\k(r)} dr \right) \\
\times \left[ \f(0)   +   t \int_0^t (1+\se{2N}(r)) \sd{2N}(r)dr  + \left( \int_0^t \sqrt{\k(r) \f(r)} dr\right)^2 \right].
\end{multline}
\end{lem}
\begin{proof}
Throughout this proof we will write $u = \tilde{u} + u_3 e_3$, i.e. we write $\tilde{u}$ for the part of $u$ parallel to $\Sigma$.  Then $\eta$ solves the transport equation $\dt \eta + \tilde{u} \cdot D \eta = u_3$ on $\Sigma$.  We may then use Lemma \ref{i_sobolev_transport} with $s = 1/2$ to estimate
\begin{equation}\label{i_ste_9}
 \sup_{0 \le r \le t} \norm{\eta(r)}_{1/2} \le \exp \left(C \int_0^t \snormspace{D \tilde{u}(r)}{3/2}{\Sigma} dr \right) \left[ \norm{\eta_0}_{1/2} + \int_0^t \snormspace{u_3(r)}{1/2}{\Sigma}dr  \right].
\end{equation}
By the definition of $\k$, \eqref{i_K_def}, we may bound $\snormspace{D \tilde{u}(r)}{3/2}{\Sigma}  \le \sqrt{\k(r)},$ but we may also use trace theory to bound  $\snormspace{u_3(r)}{1/2}{\Sigma} \ls \sd{2N}(r).$  This allows us to square both sides of \eqref{i_ste_9} and utilize Cauchy-Schwarz to deduce that
\begin{equation}\label{i_ste_1}
 \sup_{0 \le r \le t} \ns{\eta(r)}_{1/2} \ls \exp \left(2C \int_0^t \sqrt{K(r)} dr \right) \left[ \ns{\eta_0}_{1/2} + t \int_0^t \sd{2N}(r) dr  \right]. 
\end{equation}

To go to higher regularity we let  $\alpha \in \mathbb{N}^{2}$ with $\abs{\alpha} = 4N$.   Then we apply the operator $\pa$ to the equation $\dt \eta + \tilde{u} \cdot D \eta = u_3$ to see that $\pa \eta$ solves the transport equation
\begin{equation} 
 \dt (\partial^\alpha  \eta) + \tilde{u} \cdot D (\partial^\alpha  \eta) = \partial^\alpha u_3  -
\sum_{0 < \beta \le \alpha } C_{\alpha,\beta} \partial^\beta  \tilde{u} \cdot D \partial^{\alpha-\beta}  \eta
:=G^\alpha
\end{equation}
with the initial condition $\pa \eta_0$.  We may then apply Lemma \ref{i_sobolev_transport} with $s=1/2$ to find that
\begin{equation}\label{i_ste_2}
 \sup_{0\le r \le t} \norm{\partial^\alpha \eta(r)}_{1/2} \le \exp\left(C \int_0^t \snormspace{D \tilde{u}(r)}{3/2}{\Sigma} dr \right) \left[ \norm{\pa \eta_0}_{1/2} + \int_0^t \norm{G^\alpha (r)}_{1/2}dr  \right].
\end{equation}
We will now estimate $\norm{G^\alpha }_{H^{1/2}}$.

For $\beta \in \mathbb{N}^2$ satisfying $2N+1 \le \abs{\beta} \le 4N$ we may apply Lemma \ref{i_sobolev_product_1} with $s_1 = r = 1/2$ and $s_2=2$ to bound
\begin{equation}
\norm{\p^\beta \tilde{u} D \p^{\alpha-\beta} \eta}_{1/2} \ls   \snormspace{\p^\beta \tilde{u}}{1/2}{\Sigma} \norm{D \p^{\alpha-\beta} \eta}_{2}.
\end{equation}
This and trace theory then imply that
\begin{equation}\label{i_ste_3}
\sum_{\substack{0 < \beta \le \alpha \\ 2N+1 \le \abs{\beta} \le 4N }}  \norm{ C_{\alpha,\beta} \partial^\beta  \tilde{u} \cdot D \partial^{\alpha-\beta}  \eta}_{1/2} \ls \norm{D_{2N+1}^{4N} u}_{1} \norm{D_{1}^{2N} \eta}_{2}  \ls \sqrt{\sd{2N} \se{2N}}.
\end{equation}
On the other hand, if $\beta$ satisfies $1 \le \abs{\beta} \le 2N$ then we use Lemma \ref{i_sobolev_product_1} to bound
\begin{equation}
\norm{\p^\beta \tilde{u} D \p^{\alpha-\beta} \eta}_{1/2} \ls   \snormspace{\p^\beta \tilde{u}}{2}{\Sigma} \norm{D \p^{\alpha-\beta} \eta}_{1/2}
\end{equation}
so that
\begin{multline}\label{i_ste_4}
\sum_{\substack{0 < \beta \le \alpha \\ 1 \le \abs{\beta} \le 2N }}  \norm{ C_{\alpha,\beta} \partial^\beta  \tilde{u} \cdot D \partial^{\alpha-\beta}  \eta}_{1/2} \ls \norm{D_{1}^{2N} u}_{3} \norm{D_{2N+1}^{4N-1} \eta}_{2} + \snormspace{D \tilde{u}}{2}{\Sigma} \norm{D^{4N} \eta}_{1/2} \\
\ls \sqrt{\se{2N} \sd{2N}} + \sqrt{\k \f}.
\end{multline}
The only remaining term in $G^\alpha$ is $\pa u_3$, which we estimate with trace theory:
\begin{equation}\label{i_ste_5}
 \snormspace{\pa u_3}{1/2}{\Sigma} \ls \norm{D^{4N} u_3}_{1} \ls \sqrt{\sd{2N}}.
\end{equation}
We may then combine \eqref{i_ste_3}, \eqref{i_ste_4}, and \eqref{i_ste_5} for
\begin{equation}\label{i_ste_6}
 \norm{G^\alpha}_{1/2} \ls (1+ \sqrt{\se{2N}}) \sqrt{ \sd{2N}} + \sqrt{\k \f}.
\end{equation}

Returning now to \eqref{i_ste_2}, we square both sides and employ \eqref{i_ste_6} and our previous estimate of the term in the exponential to find that
\begin{multline}\label{i_ste_7}
 \sup_{0\le r \le t} \ns{\partial^\alpha \eta(r)}_{1/2} \le \exp\left(2 C \int_0^t \sqrt{\k(r)} dr \right) \\
\times \left[ \ns{\pa \eta_0}_{1/2} + t \int_0^t (1+ \se{2N}(r)) \sd{2N}(r)dr  +  \left(\int_0^t \sqrt{\k(r) \f(r)}   dr\right)^2 \right]. 
\end{multline}
Then the estimate \eqref{i_ste_0} follows by summing \eqref{i_ste_7} over all $\abs{\alpha}=4N$, adding the resulting inequality to \eqref{i_ste_1}, and using the fact that $\ns{\eta}_{4N+1/2} \ls \ns{\eta}_{1/2} + \ns{D^{4N} \eta}_{1/2}$. 
\end{proof}

Now we use this result and the $\k$ estimate of Lemma \ref{i_K_estimate} to derive a stronger result.

\begin{prop}\label{i_f_bound}
There exists a universal constant $0< \delta < 1$ so that if $\g(T) \le \delta$, then 
\begin{equation}\label{i_f_b_0}
\sup_{0\le r \le t}  \f(r) \ls  
\f(0) +   t \int_0^t \sd{2N}
\end{equation}
for all $0 \le t \le T$.
\end{prop}
\begin{proof}
Suppose $\g(T) \le \delta \le 1$, for $\delta$ to be chosen later.  Fix $0 \le t \le T$.  Then according to Lemma \ref{i_K_estimate}, we have that $\k \ls \se{N+2,2}^{(8+2\lambda)/(8+4\lambda)}$, which means that
\begin{multline}\label{i_f_b_1}
 \int_0^t \sqrt{\k(r)} dr \ls \int_0^t (\se{N+2,2}(r))^{(8+2\lambda)/(16+8\lambda)} dr \le  \delta^{(8+2\lambda)/(16+8\lambda)} \int_0^t  \frac{1}{(1+r)^{1+\lambda/4}}dr  \\
\le \delta^{(8+2\lambda)/(16+8\lambda)} \int_0^\infty  \frac{1}{(1+r)^{1+\lambda/4}}dr = \frac{4}{\lambda} \delta^{(8+2\lambda)/(16+8\lambda)}.
\end{multline}
Since $\delta \le 1$, this implies that for any constant $C>0$, 
\begin{equation}\label{i_f_b_2}
 \exp\left(C \int_0^t \sqrt{\k(r)}dr  \right) \ls 1.
\end{equation}
Similarly, 
\begin{multline}\label{i_f_b_3}
 \left(\int_0^t \sqrt{\k(r) \f(r)} dr \right)^2 \ls \left( \sup_{0\le r \le t}  \f(r) \right) \left(\int_0^t \sqrt{\k(r)}dr  \right)^2 \\
\ls \left( \sup_{0\le r \le t}  \f(r) \right)\delta^{(8+2\lambda)/(8+4\lambda)}
\end{multline}
Then \eqref{i_f_b_1}--\eqref{i_f_b_3} and  Lemma \ref{infinity_bounds} imply that
\begin{equation}
\sup_{0\le r \le t}  \f(r) \le  
C \left( \f(0) +   t \int_0^t \sd{2N}\right)   + C \delta^{(8+2\lambda)/(8+4\lambda)} \left( \sup_{0\le r \le t}  \f(r) \right),
\end{equation}
for some $C>0$.  Then if $\delta$ is small enough so that $ C \delta^{(8+2\lambda)/(8+4\lambda)} \le 1/2$, we may absorb the right-hand $\f$ term onto the left and deduce \eqref{i_f_b_0}.
\end{proof}

This bound on $\f$ allows us to estimate to estimate the integral of $\k \f$ and $\sqrt{\sd{2N} \k \f}$.

\begin{cor}\label{i_kf_integral}
There exists a universal constant $0< \delta < 1$ so that if $\g(T) \le \delta$, then 
\begin{equation}\label{i_kfi_0}
 \int_0^t \k(r) \f(r) dr \ls \delta^{(8+2\lambda)/(8+4\lambda)} \f(0) + \delta^{(8+2\lambda)/(8+4\lambda)} \int_0^t \sd{2N}(r) dr
\end{equation}
and
\begin{equation}\label{i_kfi_01}
\int_0^t \sqrt{\sd{2N}(r) \k(r) \f(r)} dr \ls \f(0) + \delta^{(8+2\lambda)/(16+8\lambda)} \int_0^t \sd{2N}(r) dr 
\end{equation}
for $0 \le t \le T$.
\end{cor}
\begin{proof}
Let $\g(T) \le \delta$ with $\delta$ as small as in Proposition \ref{i_f_bound} so that estimate \eqref{i_f_b_0} holds.  Lemma \ref{i_K_estimate} implies that
\begin{equation}
 \k(r) \ls (\se{N+2,2}(r))^{(8+2\lambda)/(8+4\lambda)} \ls \delta^{(8+2\lambda)/(8+4\lambda)}\frac{1}{(1+r)^{2+\lambda/2}}.
\end{equation}
This and \eqref{i_f_b_0} then imply that
\begin{multline}
\frac{1}{\delta^{(8+2\lambda)/(8+4\lambda)} } \int_0^t \k(r) \f(r) dr \ls  \f(0) \int_0^t \frac{dr}{(1+r)^{2+\lambda/2}}  \\ 
+  \int_0^t \frac{r}{(1+r)^{2+\lambda/2}} \left( \int_0^r \sd{2N}(s) ds\right) dr \ls 
\f(0) \int_0^\infty \frac{dr}{(1+r)^{2+\lambda/2}} \\
+ \left( \int_0^t \sd{2N}(r) dr\right) 
\left( \int_0^\infty \frac{dr}{(1+r)^{1+\lambda/2}} \right)
\ls \f(0) + \int_0^t \sd{2N}(r) dr,
\end{multline}
which is estimate \eqref{i_kfi_0}.  The estimate \eqref{i_kfi_01} follows from \eqref{i_kfi_0}, Cauchy-Schwarz, and the fact that $\delta \le 1$:
\begin{multline}
 \int_0^t \sqrt{\sd{2N}(r) \k(r) \f(r)} dr \le \left(\int_0^t \sd{2N}(r) dr \right)^{1/2} 
\left( \int_0^t \k(r) \f(r) dr \right)^{1/2} \\ \ls
\left(\int_0^t \sd{2N}(r) dr \right)^{1/2} \left( \delta^{(8+2\lambda)/(8+4\lambda)} \f(0) \right)^{1/2} + \delta^{(8+2\lambda)/(16+8\lambda)} \int_0^t \sd{2N}(r) dr \\
\ls \f(0) + \left( \delta^{(8+2\lambda)/(16+8\lambda)} + \delta^{(8+2\lambda)/(8+4\lambda)} \right)\int_0^t \sd{2N}(r) dr \\ \ls 
\f(0) +  \delta^{(8+2\lambda)/(16+8\lambda)} \int_0^t \sd{2N}(r) dr.
\end{multline}

\end{proof}

\subsection{Boundedness at the $2N$ level}

We now show bounds at the $2N$ level in terms of the initial data.

\begin{thm}\label{i_e2n_bound}
There exists a universal constant $0 < \delta < 1$ so that if $\g(T) \le \delta$, then
\begin{equation}\label{i_e2n_0}
\sup_{0 \le r \le t} \se{2N}(r) + \int_0^t  \sd{2N} + \sup_{0 \le r \le t} \frac{\f(r)}{(1+r)} \ls \se{2N}(0) + \f(0)
\end{equation}
for all $0 \le t \le T$.
\end{thm}

\begin{proof}
Combining the evolution equation estimate of Theorem \ref{i_evolution_estimate} with the comparison estimates of Theorems \ref{i_dissipation_bound_2n} and \ref{i_energy_bound_2n}, we find that
\begin{multline}\label{i_e2n_1}
  \se{2N}(t) + \int_0^t \sd{2N}(r) dr \ls \se{2N}(0) + (\se{2N}(t))^{1+\theta} + \int_0^t (\se{2N}(r))^\theta \sd{2N}(r) dr \\
+ \int_0^t  \sqrt{\sd{2N}(r) \k(r) \f(r)} dr +  \int_0^t   \k(r) \f(r) dr
\end{multline}
for some $\theta >0$.  Let us assume initially that $\delta \le 1$ is as small as in Proposition \ref{i_f_bound} and Corollary \ref{i_kf_integral} so that their conclusions hold.  We may estimate the last two integrals in \eqref{i_e2n_1} with Corollary \ref{i_kf_integral}, using the fact that $\delta \le 1$:
\begin{equation}\label{i_e2n_2}
\int_0^t  \sqrt{\sd{2N}(r) \k(r) \f(r)} dr +  \int_0^t   \k(r) \f(r) dr \ls 
\f(0) + \delta^{(8+2\lambda)/(16+8\lambda)} \int_0^t \sd{2N}(r) dr.
\end{equation}
On the other hand, $\sup_{0\le r \le t} \se{2N}(r) \le \g(T) \le \delta$, so
\begin{equation}\label{i_e2n_3}
 (\se{2N}(t))^{1+\theta} + \int_0^t (\se{2N}(r))^\theta \sd{2N}(r) dr 
\le \delta^\theta \se{2N}(t) + \delta^\theta \int_0^t \sd{2N}(r) dr.
\end{equation}
We may then combine \eqref{i_e2n_1}--\eqref{i_e2n_3} and write $\psi = \min\{\theta, (8+2\lambda)/(16+8\lambda)\} >0$ to deduce the bound
\begin{equation}\label{i_e2n_4}
  \se{2N}(t) + \int_0^t \sd{2N}(r) dr \le C\left( \se{2N}(0) + \f(0)\right) +  C\delta^\theta \se{2N}(t) + C\delta^\psi \int_0^t \sd{2N}(r) dr
\end{equation}
for a constant $C>0$.  Then if $\delta$ is sufficiently small so that $C \delta^\theta \le 1/2$ and $C \delta^\psi \le 1/2$, we may absorb the last two terms on the right side of \eqref{i_e2n_4} into the left, which then yields the estimate 
\begin{equation}\label{i_e2n_5}
  \sup_{0\le r \le t}\se{2N}(r) + \int_0^t \sd{2N}(r) dr \ls \se{2N}(0) + \f(0).
\end{equation}
We then use this and Proposition \ref{i_f_bound} to estimate
\begin{multline}\label{i_e2n_6}
\sup_{0\le r \le t} \frac{\f(r)}{(1+r)} \ls \sup_{0\le r \le t} \frac{\f(0)}{(1+r)} + \sup_{0\le r \le t} \frac{r}{(1+r)} \int_0^r \sd{2N}(s)ds \\
 \ls \f(0) + \int_0^t \sd{2N}(r) dr \ls \se{2N}(0) + \f(0).
\end{multline}
Then \eqref{i_e2n_0} follows by summing \eqref{i_e2n_5} and \eqref{i_e2n_6}.

\end{proof}

\subsection{Decay at the $N+2$ level}

Before showing  the decay estimates, we first need an interpolation result.

\begin{prop}\label{i_n_interp}
There exists a universal $0 < \delta < 1$ so that if $\g(T) \le \delta$, then
\begin{equation}\label{i_ni_00}
\sd{N+2,m}(t) \ls \sdb{N+2,m}(t), \; \se{N+2,m}(t) \ls \seb{N+2,m}(t),
\end{equation}
and  
\begin{equation}\label{i_ni_01}
 \seb{N+2,m}(t) \ls (\se{2N}(t))^{1/(m+\lambda +1)} (\sdb{N+2,m}(t))^{(m+\lambda)/(m+\lambda +1)}
\end{equation}
for $m=1,2$ and $0 \le t \le T$. 

\end{prop}
\begin{proof}
The bound $\g(T) \le \delta$ and Theorems \ref{i_dbn_0} and \ref{i_ebn_0} imply that
\begin{equation}\label{i_ni_1}
 \sd{N+2,m} \le C \sdb{N+2,m} + C \se{2N}^\theta \sd{N+2,m} \le C \sdb{N+2,m} + C \delta^\theta \sd{N+2,m} 
\end{equation}
and
\begin{equation}\label{i_ni_2}
 \se{N+2,m} \le C \seb{N+2,m} + C \se{2N}^\theta \se{N+2,m} \le C \seb{N+2,m} + C \delta^\theta \se{N+2,m} 
\end{equation}
for constants $C>0$ and $\theta>0$.  Then if $\delta$ is small enough so that $C \delta^\theta \le 1/2$, we may absorb the second term on the right side of \eqref{i_ni_1} and \eqref{i_ni_2} into the left to deduce the bounds in \eqref{i_ni_00}.

We now turn to the proof of \eqref{i_ni_01}.  According to Remark \ref{i_horizontal_remark}, we have that
\begin{equation}\label{i_ni_3}
 \seb{N+2,m} \ls \ns{\dbm{2N+4} u}_{0} + \ns{\dbm{2N+4} \eta}_{0},
\end{equation}
and by Lemma \ref{i_korn}, we also know that
\begin{equation}\label{i_ni_4}
 \ns{\dbm{2N+4} u}_{0} \ls \ns{\dbm{2N+4} \sg u}_{0} = \sdb{N+2,m}.
\end{equation}
On the other hand, the definition of $\sd{N+2,m}$, given by \eqref{i_dissipation_min_1} when $m=1$ and \eqref{i_dissipation_min_2} when $m=2$, together with \eqref{i_ni_00} implies that
\begin{equation}\label{i_ni_5}
 \ns{ \bar{D}_{m+1}^{2N+4} \eta}_{0} \le \sd{N+2,m} \ls \sdb{N+2,m}.
\end{equation}
We may then combine \eqref{i_ni_3}--\eqref{i_ni_5} to see that
\begin{equation}\label{i_ni_6}
 \seb{N+2,m} \ls \sdb{N+2,m} + \ns{\bar{D}^m \eta}_{0}.
\end{equation}

In the case $m=1$ we use the $H^0$ interpolation estimates of Lemma \ref{i_interp_eta} to bound
\begin{equation}\label{i_ni_7}
 \ns{\bar{D}^m \eta}_{0} = \ns{D \eta}_{0} \ls (\se{2N})^{1/(2+\lambda)} (\sd{N+2,1})^{(1+\lambda)/(2+\lambda)}.
\end{equation}
In the case $m=2$ we use the $H^0$ interpolation estimates of $D^2 \eta$ from Lemma \ref{i_interp_eta} and the $H^0$ estimate of $\dt \eta$ from Proposition \ref{i_bs_u_2} to bound
\begin{equation}\label{i_ni_8}
\ns{\bar{D}^m \eta}_{0} = \ns{D^2 \eta}_{0} + \ns{\dt \eta}_{0} \ls (\se{2N})^{1/(3+\lambda)} (\sd{N+2,1})^{(2+\lambda)/(3+\lambda)}.
\end{equation}
Together, \eqref{i_ni_7} and \eqref{i_ni_8} may be written as
\begin{equation}\label{i_ni_9}
\ns{\bar{D}^m \eta}_{0} \ls  (\se{2N})^{1/(m+\lambda+1)} (\sd{N+2,1})^{(m+\lambda)/(m+\lambda+1)}.
\end{equation}
Now, according to Lemma \ref{i_N_constraint}, we can bound
\begin{equation}\label{i_ni_10}
\sdb{N+2,m} \le \sd{N+2,m} \ls (\se{2N})^{1/(m+\lambda+1)} (\sd{N+2,m})^{(m+\lambda)/(m+\lambda+1)}.
\end{equation}
Then we use the estimates \eqref{i_ni_9} and \eqref{i_ni_10} to bound the right side of \eqref{i_ni_6}; the bound \eqref{i_ni_01} follows from the resulting inequality and \eqref{i_ni_00}.

\end{proof}

Now we show that the extra integral term appearing in Theorem \ref{i_evolution_estimate_half} can essentially be absorbed into $\seb{N+2,m}$.

\begin{lem}\label{i_stray_control}
Let $F^2$ be defined by \eqref{i_F2_def} with $\pa = \dt^{N+2}$.  There exists a universal $0< \delta <1$ so that if $\g(T) \le \delta$, then
\begin{equation}\label{i_stc_0}
 \frac{2}{3} \seb{N+2,m}(t) \le \seb{N+2,m}(t) - 2 \int_\Omega J(t) \dt^{N+1} p(t) F^2(t) \le \frac{4}{3} \seb{N+2,m}(t)
\end{equation}
for all $0 \le t \le T$.
\end{lem}
\begin{proof}
Suppose that $\delta$ is as small as in Proposition \ref{i_n_interp}.  Then we combine estimate \eqref{i_F_e_h_02} of Theorem \ref{i_F_estimates_half},  Lemma \ref{infinity_bounds}, and estimate \eqref{i_ni_00} of Proposition \ref{i_n_interp} to see that
\begin{multline}
\pnorm{J}{\infty} \norm{\dt^{N+1} p}_{0} \norm{F^2}_{0} \ls \sqrt{\se{N+2,m}} \sqrt{\se{2N}^\theta \se{N+2,m}} \\
=  \se{2N}^{\theta/2} \se{N+2,m} \ls \se{2N}^{\theta/2} \seb{N+2,m} \ls \delta^{\theta/2} \seb{N+2,m}
\end{multline}
for some  $\theta >0$.  This estimate and Cauchy-Schwarz then imply that
\begin{equation}\label{i_stc_1}
 \abs{2 \int_\Omega J \dt^{N+1} p F^2 } \le 2 \pnorm{J}{\infty} \norm{\dt^{N+1} p}_{0} \norm{F^2}_{0} \le C \delta^{\theta/2} \seb{N+2,m} \le \frac{1}{3} \seb{N+2,m}
\end{equation}
if $\delta$ is small enough.  The bound \eqref{i_stc_0} then follows easily from \eqref{i_stc_1}.

\end{proof}

Now we prove decay at the $N+2$ level.

\begin{thm}\label{i_n_decay}
There exists a universal constant $0<\delta<1$ so that if $\g(T) \le \delta$, then
\begin{equation}\label{i_nd_0}
\sup_{0\le r \le t} ( 1 + r)^{m+\lambda} \se{N+2,m}(r) \ls  \se{2N}(0) + \f(0)
\end{equation}
for all $0 \le t \le T$ and for $m\in \{1,2\}$.
\end{thm}
\begin{proof}
Let $\delta$ be as small as in Theorem \ref{i_e2n_bound}, Proposition \ref{i_n_interp}, and Lemma \ref{i_stray_control}.   Theorem \ref{i_evolution_estimate_half} and the estimate \eqref{i_ni_00} of Proposition \ref{i_n_interp} imply that
\begin{equation}\label{i_nd_1}
\dt \left( \seb{N+2,m} - 2 \int_\Omega J \dt^{N+1} p F^2 \right) + \sdb{N+2,m} 
\le  C \se{2N}^{\theta} \sd{N+2,m} \le C \delta^\theta \sdb{N+2,m} \le \hal \sdb{N+2,m}  
\end{equation}
if $\delta$ is small enough (here $\theta>0$).  On the other hand, Theorem \ref{i_e2n_bound}, \eqref{i_ni_01} of Proposition \ref{i_n_interp}, and  \eqref{i_stc_0} of Lemma \ref{i_stray_control}  imply that
\begin{multline}\label{i_nd_2}
0 \le \frac{2}{3} \seb{N+2,m} \le \seb{N+2,m} - 2 \int_\Omega J \dt^{N+1} p F^2 \le \frac{4}{3} \seb{N+2,m}  \\ 
\le C (\se{2N})^{1/(m+\lambda +1)} (\sdb{N+2,m})^{(m+\lambda)/(m+\lambda +1)} 
\le C_0 \z_0^{1/(m+\lambda +1)} (\sdb{N+2,m})^{(m+\lambda)/(m+\lambda +1)}
\end{multline}
for all $0\le t \le T$, where we have written $\z_0:= \se{2N}(0) + \f(0)$, and $C_0$ is a universal constant which we may assume satisfies $C_0 \ge 1$.  Let us write
\begin{equation}
 h(t) =  \seb{N+2,m}(t) - 2 \int_\Omega J(t) \dt^{N+1} p(t) F^2(t) \ge 0, 
\end{equation}
as well as
\begin{equation}
 s = \frac{1}{m+\lambda} \text{ and } C_1 =  \frac{1}{2 C_0^{1+s} \z_0^s  }.
\end{equation}
In these three terms we should distinguish between the cases $m=1$ and $m=2$, but to avoid notational clutter we will abuse notation and only write  $h(t)$, $s$, and $C_1$.   We may then combine \eqref{i_nd_1} with \eqref{i_nd_2} and use our new notation to derive the differential inequality
\begin{equation}\label{i_nd_3}
 \dt h(t) + C_1 (h(t))^{1+s} \le 0
\end{equation}
for $0\le t \le T$.  

Since $h(t)\ge 0$, we may integrate \eqref{i_nd_3} to find that for any $0 \le r \le T$,
\begin{equation}\label{i_nd_4}
 h(r) \le \frac{h(0)}{[1 + s C_1 (h(0))^s r ]^{1/s}}.
\end{equation}
Notice that Remark  \ref{i_horizontal_remark} implies that $\seb{N+2,m} \le (3/2) \se{2N}$.  Then
\eqref{i_nd_2} implies that $h(0) \le (4/3) \seb{N+2,m}(0) \le 2 \se{2N}(0) \le 2 \z_0$, which in turn implies that
\begin{equation}\label{i_nd_5}
 s C_1 (h(0))^s = \frac{s}{2  C_0^{1+s}  } \left(\frac{h(0)}{\z_0}\right)^s \le \frac{s}{2  C_0^{1+s}  } 2^s = \frac{s}{ C_0^{1+s}  }  2^{s-1} \le 1
\end{equation}
since $0< s < 1$ and $C_0 \ge 1$.  A simple computation shows that
\begin{equation}
 \sup_{r \ge 0} \frac{(1+r)^{1/s}}{(1+ M r)^{1/s} } = \frac{1}{M^{1/s}}
\end{equation}
when $0 \le M \le 1$ and $s >0$.  This, \eqref{i_nd_4}, and  \eqref{i_nd_5} then imply that
\begin{equation}\label{i_nd_6}
 (1+r)^{1/s} h(r) \le h(0) \frac{(1+r)^{1/s}}{[1 + s C_1 (h(0))^s r ]^{1/s}} \le 
h(0) \left(\frac{2 C_0^{1+s}}{s}\right)^{1/s} \frac{\z_0}{h(0)} = \left(\frac{2 C_0^{1+s}}{s}\right)^{1/s} \z_0. 
\end{equation}
Now we use \eqref{i_ni_00} of Proposition \ref{i_n_interp} together with \eqref{i_nd_2} to bound
\begin{equation}\label{i_nd_7}
 \se{N+2,m}(r) \ls \seb{N+2,m}(r) \ls h(r) \text{ for } 0\le r \le T.
\end{equation}
The estimate \eqref{i_nd_0} then follows from \eqref{i_nd_6}, \eqref{i_nd_7}, and the fact that $s=1/(m+\lambda)$ and $\z_0 = \se{2N}(0) +\f(0)$.

\end{proof}

\subsection{A priori estimates for $\g$}

We now collect the results of Theorems \ref{i_e2n_bound} and \ref{i_n_decay} into a single bound on $\g$, as defined by \eqref{i_total_energy}.  The estimate recorded specifically names the constant in the inequality with $C_1>0$ so that it can be referenced later.

\begin{thm}\label{i_a_priori}
There exists a universal $0 < \delta < 1$ so that if $\g(T) \le \delta$, then
\begin{equation}\label{i_apr_0}
 \g(t) \le C_1( \se{2N}(0) + \f(0))
\end{equation}
for all $0 \le t \le T$, where $C_1 >0$ is a universal constant.
\end{thm}
\begin{proof}
Let $\delta$ be as small as in Theorems \ref{i_e2n_bound} and \ref{i_n_decay}.  Then the conclusions of the theorems hold, and we may sum them to deduce \eqref{i_apr_0}.
 
\end{proof}

\section{Specialized local well-posedness}\label{inf_9}

\subsection{Propagation of $\i_{\lambda}$ bounds}

To prove Theorem \ref{intro_inf_gwp}, we will combine our a priori estimates, Theorem \ref{i_a_priori}, with a local well-posedness result.  Theorem \ref{intro_lwp} is not quite enough since it does not address the boundedness of $\ns{\il u(t)}_{0}$, $\ns{\il \eta(t)}_{0}$, and $\ns{\il p(t)}_{0}$ for $t >0$.  In order to prove these bounds, we will first study the cutoff operators $\il^m$, which we define now.  Let $m \ge 1$ be an integer.  For a function $f$ defined on $\Omega$, we define the cutoff Riesz potential $\il^m f$ by
\begin{equation}
 \il^m f(x',x_3) = \int_{-b}^0 \int_{\{\abs{\xi} \ge 1/m \}} \hat{f}(\xi,x_3) \abs{\xi}^{-\lambda} e^{2\pi i x'\cdot \xi} d\xi dx_3.
\end{equation}
Similarly, for $f$ defined on $\Sigma$, we set
\begin{equation}
 \il^m f(x') = \int_{ \{ \abs{\xi}\ge 1/m \}} \hat{f}(\xi) \abs{\xi}^{-\lambda} e^{2\pi i x'\cdot \xi} d\xi.
\end{equation}
The operator $\il^m$ is clearly bounded on $H^0(\Omega)$ and $H^0(\Sigma)$, which allows us to apply it to our solutions and then study the evolution of $\il^m u$ and $\il^m \eta$.

Before doing so, we will record some estimates for terms involving $\il^m$ that are analogous to the $\il$ estimates in Sections \ref{i_lambda_nonlinearities} and \ref{i_lambda_energy_evolution} and Appendix \ref{i_riesz_potential}.  We begin with the analog of Lemmas \ref{i_riesz_prod} and \ref{i_riesz_derivative}, which were the starting point for our $\il$ estimates.

\begin{lem}\label{i_cutoff_estimates}
If $\i_\lambda h \in H^0(\Omega)$, then $\ns{\il^m h}_{0} \le  \ns{\il h}_{0}$.  A similar estimate holds if $\i_\lambda h \in H^0(\Sigma)$.  As a consequence, the results of Lemmas \ref{i_riesz_prod} and \ref{i_riesz_derivative} hold with $\i_{\lambda}$ replaced by $\i_{\lambda}^m$ and with the constants in the inequalities independent of $m$.
\end{lem}
\begin{proof}
Suppose that $\i_\lambda h \in H^0(\Omega)$ for some $h$.   Then, writing $\hat{\dot}$ for the horizontal Fourier transform, we easily see that
\begin{equation}
\ns{\il^m h}_{0} = \int_{-b}^0 \int_{ \{ \abs{\xi} \ge 1/m \}} \abs{\hat{h}(\xi,x_3)}^2 \abs{\xi}^{-2\lambda} d\xi dx_3  \le    \ns{\il h}_{0}.
\end{equation}
The corresponding estimate in case $\i_\lambda h \in H^0(\Sigma)$ follows similarly.  Then the estimates of Lemmas \ref{i_riesz_prod} and \ref{i_riesz_derivative} may be combined with these inequalities to replace $\il$ with $\il^m$.
\end{proof}

We do not want our estimates for $\il^m$ to be given in terms of $\se{2N}$ since this energy contains $\il$ terms.  Instead, we desire estimates in terms of a modified energy, which we write as
\begin{equation}\label{i_modified_energy}
 \mathfrak{E}_{2N} = \se{2N} - \ns{\il u}_{0} - \ns{\il \eta}_{0}.
\end{equation}
Lemma \ref{i_cutoff_estimates} allows us prove the following modification of  Proposition \ref{i_riesz_G}.  The proof is a simple adaptation of the one for Proposition \ref{i_riesz_G}, and is thus omitted.

\begin{prop}\label{i_cutoff_G}
 We have that
\begin{equation}
 \ns{\il^m G^1}_{1} + \ns{\il^m G^2}_{2} + \ns{\il^m \dt G^2}_{0} + \ns{\il^m G^3}_{1} + \ns{\il^m G^4}_{1} \ls \mathfrak{E}_{2N}^2.
\end{equation}
Here the constant in the inequality does not depend on $m$.
\end{prop}

We may similarly modify the proof of Lemma \ref{i_riesz_u}.

\begin{lem}\label{i_cutoff_u}
 We have that
\begin{equation}
\ns{\il^m [ (AK) \p_3 u_1 + (BK) \p_3 u_2 ]  }_{0} +  \sum_{i=1}^2 \ns{\il^m [ u \p_i K ]}_{0} \ls  \mathfrak{E}_{2N}^2
\end{equation}
and
\begin{equation}
 \ns{\il^m [(1-K) u]}_{0} + \ns{\il^m [(1-K) G^2 ]}_{0} \ls \mathfrak{E}_{2N}^2.
\end{equation}
Here the constants in the inequalities do not depend on $m$.
\end{lem}

Then Lemma \ref{i_cutoff_u} leads to a modification of  Lemma \ref{i_pressure_riesz}.

\begin{lem}\label{i_cutoff_pressure}
 It holds that
\begin{equation}
 \ns{\il^m p}_{0} \ls \ns{\il^m \eta}_{0}  + \mathfrak{E}_{2N}
\text{ and } \ns{\il^m Dp}_{0} \ls \mathfrak{E}_{2N}.
\end{equation}
Here the constants in the inequalities do not depend on $m$.
\end{lem}

In turn, Lemma \ref{i_cutoff_pressure} gives a variant of  Lemma \ref{i_pressure_riesz_interaction}.

\begin{lem}\label{i_cutoff_interaction}
 It holds that
\begin{equation}
\abs{\int_\Omega \il^m p \il^m G^2} \ls \mathfrak{E}_{2N} \norm{\il^m \eta}_{0} + \mathfrak{E}_{2N}.
\end{equation}
Here the constant in the inequality does not depend on $m$.
\end{lem}

These results now allow us to study the boundedness of $\il u$, etc.  We first apply the operator $\il^m$ to the equations \eqref{linear_perturbed}, which is possible since $\il^m$ is bounded on $H^0(\Omega)$ and $H^0(\Sigma)$.  Then the energy evolution for $\il^m u$ and $\il^m \eta$ allows us to derive bounds for these quantities, which yield bounds for $\il u$ and $\il \eta$ after passing to the limit $m \to \infty$.

\begin{prop}\label{i_lambda_propagate}
Suppose $(u,p,\eta)$ are solutions on the time interval $[0,T]$ and that $\ns{ \il u_0}_{0} + \ns{\il \eta_0}_{0} <\infty$ and $\sup_{0 \le t \le T}   \mathfrak{E}_{2N}(t) \le 1.$  Then
\begin{multline}\label{i_lamp_0}
 \sup_{0\le t \le T} \left( \ns{ \il u(t)}_{0} + \ns{ \il p(t)}_{0} + \ns{\il \eta(t)}_{0} \right)  + \int_0^T \ns{\il u(t)}_{1}dt \\
\ls e^T \left( \ns{ \il u_0}_{0} + \ns{\il \eta_0}_{0} \right)  + e^T  \sup_{0 \le t \le T}   \mathfrak{E}_{2N}(t).
\end{multline}
\end{prop}

\begin{proof}
Since $\il^m$ is a bounded operator on $H^0(\Omega)$ and $H^0(\Sigma)$, we are free to apply it to the equations \eqref{linear_perturbed}.  After doing so we then use Lemma \ref{general_evolution}  to see that
\begin{multline}\label{i_lamp_1}
  \dt  \left( \hal \int_\Omega \abs{\il^m u}^2  + \hal\int_\Sigma  \abs{\il^m \eta}^2 \right) 
+ \hal \int_\Omega \abs{\sg \il^m u}^2 
= \int_\Omega \il^m u \cdot \il^m G^1 + \il^m p \il^m G^2 
\\
+ \int_\Sigma - \il^m u \cdot \il^m  G^3 +  \il^m \eta \il^m  G^4.
\end{multline}
We will estimate each term on the right side of this equation.  First, we use Cauchy-Schwarz and Lemma \ref{i_cutoff_G} to estimate the first and fourth terms:
\begin{multline}\label{i_lamp_2}
\abs{\int_\Omega \il^m u \cdot \il^m G^1} +  \abs{\int_\Sigma \il^m \eta \il^m  G^4 }  \le
\norm{\il^m u}_{0} \norm{\il^m G^1}_{0} +  \norm{\il^m \eta}_{0} \norm{\il^m  G^4}_{0} \\
\le   \hal \ns{\il^m u}_{0} + \frac{1}{4} \ns{\il^m \eta}_{0}  
+ \hal \ns{\il^m G^1}_{0} +   \ns{\il^m G^4}_{0} \le \hal \ns{\il^m u}_{0} + \frac{1}{4} \ns{\il^m \eta}_{0}  + C \mathfrak{E}_{2N}^2
\end{multline}
for $C>0$ independent of $m$.  For the second term  we use Lemma \ref{i_cutoff_interaction} and Cauchy's inequality for
\begin{equation}\label{i_lamp_3}
 \abs{\int_\Omega \il^m p \il^m G^2 }\le C \norm{\il^m \eta}_{0} \mathfrak{E}_{2N}  + C\mathfrak{E}_{2N} \le \frac{1}{4} \ns{\il^m \eta}_{0}  + C (\mathfrak{E}_{2N} + \mathfrak{E}_{2N}^2),
\end{equation}
where again $C>0$ is independent of $m$.  Finally, for the third term we use trace theory,  Lemma \ref{i_cutoff_G}, and Lemma \ref{i_korn}  to bound
\begin{multline}\label{i_lamp_4}
 \abs{ \int_\Sigma  \il^m u \cdot \il^m  G^3}
\le \snormspace{\il^m u}{0}{\Sigma} \norm{\il^m  G^3}_{0}  
\le C \norm{\il^m u}_{1} \norm{\il^m G^3}_{0} \\
\le C \norm{\sg \il^m u}_{0} \mathfrak{E}_{2N} 
\le \frac{1}{4} \ns{\sg \il^m u}_{0} +  C \mathfrak{E}_{2N}^2,
\end{multline}
with $C>0$ independent of $m$.  Now we use \eqref{i_lamp_2}--\eqref{i_lamp_4} to estimate the right side of \eqref{i_lamp_1}; after rearranging the resulting bound, we find that
\begin{equation}\label{i_lamp_5}
\dt \left( \ns{ \il^m u}_{0} +  \ns{\il^m \eta}_{0}  \right) + \frac{1}{2} \ns{\sg \il^m u}_{0} \le \ns{ \il^m u}_{0} + \ns{\il^m \eta}_{0} + C (\mathfrak{E}_{2N} + \mathfrak{E}_{2N}^2)
\end{equation}
for a constant $C>0$ that does not depend on $m$.

The inequality \eqref{i_lamp_5} may be viewed as the differential inequality
\begin{equation}\label{i_lamp_6}
 \dt \se{\lambda,m} + \frac{1}{2} \sd{\lambda,m} \le \se{\lambda,m} + C (\mathfrak{E}_{2N} + \mathfrak{E}_{2N}^2),
\end{equation}
where we have written $\se{\lambda,m} =  \ns{ \il^m u}_{0} + \ns{\il^m \eta}_{0}$ and $\sd{\lambda,m} = \ns{\sg \il^m u}_{0}$.  Applying Gronwall's lemma to \eqref{i_lamp_6} and using the fact that $\mathfrak{E}_{2N}(t) \le 1$  then shows that
\begin{multline}\label{i_lamp_7}
  \se{\lambda,m}(t) +  \hal \int_0^t \sd{\lambda,m}(s) ds \le  \se{\lambda,m}(0) e^t + C \int_0^t e^{t-s} \mathfrak{E}_{2N}(s) ds \\
\le 
 \se{\lambda,m}(0) e^t + C(e^t -1) \sup_{0 \le s \le t}   \mathfrak{E}_{2N}(s)
\end{multline}
where again $C>0$ is independent of $m$.  It is a simple matter to verify, using the definitions of $\il^m$ and $\il$, the Fourier transform in $(x_1,x_2)$, and the monotone convergence theorem, that as $m \to \infty$, 
\begin{equation}\label{i_lamp_8}
   \se{\lambda,m}(s)  =  \ns{ \il^m u(s)}_{0} + \ns{\il^m \eta(s)}_{0} \to  \ns{ \il u(s)}_{0} + \ns{\il \eta(s)}_{0} 
\end{equation}
for both $s=0$ and $s=t$, and
\begin{equation}\label{i_lamp_9}
 \int_0^t \sd{\lambda,m}(s) ds \to \int_0^t \ns{\sg \il u(s)}_{0} ds.
\end{equation}
Now, according to these two convergence results, we may pass to the limit $m \to \infty$ in \eqref{i_lamp_7}; the resulting estimate and Lemma \ref{i_korn} then imply that
\begin{multline}\label{i_lamp_10}
 \sup_{0\le t \le T} \left( \ns{ \il u(t)}_{0} + \ns{\il \eta(t)}_{0} \right)  + \int_0^T \ns{\il u(t)}_{1}dt \\
\ls \left( \ns{ \il u_0}_{0} + \ns{\il \eta_0}_{0} \right) e^T + (e^T -1) \sup_{0 \le t \le T}   \mathfrak{E}_{2N}(t).
\end{multline}

On the other hand, from Lemma \ref{i_cutoff_pressure}, we know that
\begin{equation}
 \ns{\il^m p(t)}_{0} \ls \ns{ \il^m \eta(t)}_{0} + \mathfrak{E}_{2N}(t).
\end{equation}
We may then argue as above, employing the monotone convergence theorem, to pass to the limit $m\to \infty$ in this estimate.  We then find that
\begin{equation}\label{i_lamp_11}
 \sup_{0\le t \le T} \ns{\il p(t)}_{0} \ls \sup_{0 \le t\le T} \ns{ \il \eta(t)}_{0} + \sup_{0\le t \le T} \mathfrak{E}_{2N}(t).
\end{equation}
The estimate \eqref{i_lamp_0} then follows by combining \eqref{i_lamp_10} and \eqref{i_lamp_11}.

\end{proof}

\subsection{Local well-posedness}

We now record the specialized version of the local well-posedness theorem.  We include estimates for $\il u$, $\il \eta$, and $\il p$.  We also separate estimates for $\se{2N}$ and $\sd{2N}$ from estimates for $\f$ and $\mathfrak{E}_{2N}$, the latter of which is defined by \eqref{i_modified_energy}.

\begin{thm}\label{i_infinite_lwp}
Suppose that initial data are given satisfying the compatibility conditions of Theorem \ref{intro_lwp} and  $\ns{u(0)}_{4N} + \ns{\eta(0)}_{4N+1/2} + \ns{\il u(0)}_{0} + \ns{\il \eta(0)}_{0} < \infty$.
Let $\ep >0$.  There exists a $\delta_0 = \delta_0(\ep) >0$ and a  
\begin{equation}\label{i_ilwp_00}
T_0 = C(\ep) \min\left\{1, \frac{1}{\ns{\eta(0)}_{4N+1/2}}  \right\} > 0,
\end{equation}
where $C(\ep)>0$ is a constant depending on $\ep$, so that if $0 < T \le T_0$ and $\ns{u(0)}_{4N} +  \ns{\eta(0)}_{4N} \le \delta_0,$ then there exists a unique solution $(u,p,\eta)$ to \eqref{geometric} on the interval $[0,T]$ that achieves the initial data.  The solution obeys the estimates
\begin{multline}\label{i_ilwp_01}
 \sup_{0 \le t \le T} \se{2N}(t) + \sup_{0 \le t \le T} \ns{\il p(t)}_{0} + \int_0^T \sd{2N}(t) dt
\\
+ \int_0^T \left( \ns{\dt^{2N+1} u(t)}_{({_0}H^1)^*}  + \ns{ \dt^{2N} p(t) }_{0} \right) dt
\le C_2 \left( \ep +  \ns{\il u(0)}_{0} + \ns{\il \eta(0)}_{0}  \right),
\end{multline}
\begin{equation}\label{i_ilwp_02}
 \sup_{0 \le t \le T} \mathfrak{E}_{2N}(t) \le \ep, \text{ and } 
 \sup_{0 \le t \le T} \f(t) \le C_2 \f(0) + \ep
\end{equation}
for $C_2>0$ a universal constant.  Here $\mathfrak{E}_{2N}$ is as defined by \eqref{i_modified_energy}.
\end{thm}
\begin{proof}  
The result follows directly from  Proposition \ref{i_lambda_propagate} and Theorem \ref{intro_lwp}.
\end{proof}

\begin{remark}\label{i_justification_remark}
The finiteness of the terms in \eqref{i_ilwp_01}--\eqref{i_ilwp_02} justifies all of the computations leading to Theorem \ref{i_a_priori}.  In particular, it shows that $\dt^{2N+1} u$ and $\dt^{2N} p$ are well-defined.
\end{remark}

\section{Global well-posedness and decay: proof of Theorem \ref{intro_inf_gwp}}\label{inf_10}

In order to combine the local existence result, Theorem \ref{i_infinite_lwp}, with the a priori estimates of Theorem \ref{i_a_priori}, we must be able to estimate $\g$ in terms of the estimates given in \eqref{i_ilwp_01}--\eqref{i_ilwp_02}.  We record this estimate now.

\begin{prop}\label{i_total_norm_estimate}
Let $\mathfrak{E}_{2N}$ be as defined by \eqref{i_modified_energy}.  There exists a universal constant $C_3>0$ with the following properties.  If $0 \le T$, then we have the estimate
\begin{multline}\label{i_tne_00}
 \g(T) \le   \sup_{0 \le t \le T} \se{2N}(t)  + \int_{0}^{T_2} \sd{2N}(t) dt \\
+  \sup_{0 \le t \le T} \f(t)  + C_3  (1+T)^{2+\lambda} \sup_{0 \le t \le T} \mathfrak{E}_{2N}(t).
\end{multline}
If $0 < T_1 \le T_2$, then we have the estimate
\begin{multline}\label{i_tne_01}
 \g(T_2) \le C_3 \g(T_1) +  \sup_{T_1 \le t \le T_2} \se{2N}(t)  + \int_{T_1}^{T_2} \sd{2N}(t) dt \\
+ \frac{1}{(1+T_1)} \sup_{T_1 \le t \le T_2} \f(t)  + C_3 (T_2-T_1)^2 (1+T_2)^{2+\lambda} \sup_{T_1 \le t \le T_2} \mathfrak{E}_{2N}(t).
\end{multline}
\end{prop}
\begin{proof}
We will only prove the estimate \eqref{i_tne_01}; the bound \eqref{i_tne_00} follows from a similar, but easier argument.  The definition of $\g(T_2)$ allows us to estimate
\begin{multline}\label{i_tne_1}
 \g(T_2) \le \g(T_1) +  \sup_{T_1 \le t \le T_2} \se{2N}(t)  + \int_{T_1}^{T_2} \sd{2N}(t) dt \\
+  \sup_{T_1 \le t \le T_2} \frac{\f(t)}{(1+t)}  + \sum_{m=1}^2 \sup_{T_1 \le t \le T_2} \left( (1+t)^{m+\lambda}  \se{N+2,m}(t) \right).
\end{multline}

Since $N\ge 3$ it is easy to verify that
\begin{equation}
 \sum_{j=0}^{N+2} \ns{\dt^{j+1} u}_{2(N+2) -2j} + \ns{\dt^{j} u}_{2(N+2) -2j} + \ns{\dt^{j+1} \eta}_{2(N+2) -2j} + \ns{\dt^{j} \eta}_{2(N+2) -2j} 
\ls \mathfrak{E}_{2N}
\end{equation}
and
\begin{equation}
\sum_{j=0}^{N+1} \ns{\dt^{j+1} p}_{2(N+2) -2j -1}   + \ns{\dt^{j} p}_{2(N+2) -2j -1} 
\ls \mathfrak{E}_{2N}. 
\end{equation}
For  $j=0,\dotsc,2N$, we may then integrate $\dt \left[ (1+t)^{(m+\lambda)/2} \dt^j u(t)  \right]$ in time from $T_1$ to $T_1 \le t \le T_2$ to deduce the bound
\begin{multline}
 \norm{(1+t)^{(m+\lambda)/2} \dt^j u(t)}_{2N+4-2j} \le  \norm{(1+T_1)^{(m+\lambda)/2} \dt^j u(T_1)}_{2N+4-2j} \\
+ \int_{T_1}^{T_2} (1+s)^{(m+\lambda)/2} \norm{\dt^{j+1} u(s)}_{2N+4-2j}   + \frac{(m+\lambda)}{2}  (1+s)^{(m+\lambda-2)/2} \norm{\dt^{j} u(s)}_{2N+4-2j} \\
\ls \sqrt{\g(T_1)} + (T_2-T_1) (1+T_2)^{1+\lambda/2} \sqrt{ \sup_{T_1 \le t \le T_2}  \mathfrak{E}_{2N}(t)}.
\end{multline}
Squaring both sides of this then yields, for $j=0,\dotsc,N+2$, 
\begin{equation}
 \sup_{T_1 \le t \le T_2} \left( (1+t)^{m+\lambda} \ns{\dt^j u(t)}_{2(N+2)-2j}\right) \ls \g(T_1) + (T_2-T_1)^2 (1+T_2)^{2+\lambda} \sup_{T_1 \le t \le T_2} \mathfrak{E}_{2N}(t). 
\end{equation}
Similar estimates hold for $j=0,\dots,N+2$ with $\dt^j u$ replaced by $\dt^j \eta$ and for $j=0,\dotsc,N+1$ with  $\ns{\dt^j u(t)}_{2(N+2)-2j}$ replaced by $\ns{\dt^j p(t)}_{2(N+2)-2j-1}$.  From these we may then estimate
\begin{equation}\label{i_tne_2}
 \sum_{m=1}^2 \sup_{T_1 \le t \le T_2} \left( (1+t)^{m+\lambda}  \se{N+2,m}(t) \right) \ls \g(T_1) + (T_2-T_1)^2 (1+T_2)^{2+\lambda} \sup_{T_1 \le t \le T_2} \mathfrak{E}_{2N}(t).
\end{equation}
Then \eqref{i_tne_01} follows from \eqref{i_tne_1}, \eqref{i_tne_2}, and the trivial bound
\begin{equation}
 \sup_{T_1 \le t \le T_2} \frac{\f(t)}{(1+t)} \le \frac{1}{(1+T_1)} \sup_{T_1 \le t \le T_2} \f(t).
\end{equation}

\end{proof}

We now turn to our main result.

\begin{thm}\label{i_gwp}
Suppose the initial data $(u_0,\eta_0)$ satisfy the compatibility conditions of Theorem \ref{intro_lwp}.  There exists a $\kappa >0$ so that if $\se{2N}(0) + \f(0) < \kappa$,  then there exists a unique solution $(u,p,\eta)$ on the interval $[0,\infty)$ that achieves the initial data.  The solution obeys the estimate
\begin{equation}\label{i_gwp_0}
 \g(\infty) \le C_1 \left( \se{2N}(0) + \f(0) \right) < C_1 \kappa,
\end{equation}
where $C_1>0$ is given by Theorem \ref{i_a_priori}.
\end{thm}
\begin{proof}
Let $0 < \delta <1$ and $C_1>0$ be the constants from Theorem \ref{i_a_priori}, $C_2>0$ be the constant from Theorem \ref{i_infinite_lwp}, and $C_3>0$ be the constant from Proposition \ref{i_total_norm_estimate}.  According to  \eqref{i_tne_00} of Proposition \ref{i_total_norm_estimate}, if a solution exists on the interval $[0,T]$ with $T < 1$ and obeys the estimates \eqref{i_ilwp_01}--\eqref{i_ilwp_02}, then 
\begin{equation}\label{i_gwp_1}
 \g(T) \le C_2 \kappa +  \ep \left[ C_2 +1 +  C_3 (2)^{2+\lambda} \right].
\end{equation}
If $\ep$ is chosen  so that the latter term in \eqref{i_gwp_1} equals $\delta /2$, then we may choose $\kappa$ sufficiently small so that $C_2 \kappa < \delta/2$ and $\kappa < \delta_0(\ep)$ (with $\delta_0(\ep)$ given by  Theorem \ref{i_infinite_lwp}); then Theorem \ref{i_infinite_lwp} provides a unique solution on $[0,T]$ obeying the estimates \eqref{i_ilwp_01}--\eqref{i_ilwp_02}, and hence $\g(T) \le \delta$.  According to Remark \ref{i_justification_remark}, all of the computations leading to Theorem \ref{i_a_priori} are justified by the estimates \eqref{i_ilwp_01}--\eqref{i_ilwp_02}.

Let us now define
\begin{multline}
 T_*(\kappa) = \sup \{  T>0 \;\vert\; \text{for every choice of initial data satisfying the compatibility} \\
\text{conditions and } \se{2N}(0) + \f(0) < \kappa \text{ there exists a unique solution on } [0,T] \\
\text{that achieves the data and satisfies } \g(T) \le \delta  \}.
\end{multline}
By the above analysis, $T_*(\kappa)$ is well-defined and satisfies $T_*(\kappa)>0$ if $\kappa$ is small enough, i.e. there is a $\kappa_1>0$ so that $T_* : (0,\kappa_1] \to (0,\infty]$.  It is easily verified that $T_*$ is non-increasing on $(0,\kappa_1]$.  Let us now set 
\begin{equation}\label{i_gwp_15}
\ep = \frac{\delta}{3}  \min\left\{ \frac{1}{1+C_2},  \frac{1}{C_3}  \right\} 
\end{equation}
and then define $\kappa_0 \in (0,\kappa_1]$ by
\begin{equation}\label{i_gwp_2}
 \kappa_0 = \min\left\{  \frac{\delta}{3 C_1(C_3 + 2 C_2)}, \frac{\delta_0( \ep)}{C_1}, \kappa_1    \right\},
\end{equation}
where $\delta_0(\ep)$ is given by Theorem \ref{i_infinite_lwp} with $\ep$ given by \eqref{i_gwp_15}.  We claim that $T_*(\kappa_0) = \infty$.  Once the claim is established, the proof of the theorem is complete since then $T_*(\kappa) = \infty$ for all $0 < \kappa \le \kappa_0$.

Suppose, by way of contradiction, that  $T_*(\kappa_0) < \infty$.  We will show that solutions can actually be extended past $T_*(\kappa_0)$ and that these solutions  satisfy $\g(T_2) \le \delta$ for $T_2 > T_*(\kappa_0)$, contradicting the definition of $T_*(\kappa_0)$.  We begin by extending the solutions.   By the definition of $T_*(\kappa_0)$, we know that for every $0 < T_1 < T_*(\kappa_0)$ and  for any choice of data satisfying the compatibility conditions and the bound $\se{2N}(0) + \f(0) < \kappa_0$, there exists a unique solution on $[0,T_1]$ that achieves the initial data and satisfies $\g(T_1) \le \delta$.  Then by Theorem \ref{i_a_priori}, we know that actually
\begin{equation}\label{i_gwp_3}
 \g(T_1) \le C_1 (\se{2N}(0) + \f(0)) < C_1 \kappa_0.
\end{equation}
In particular, this and \eqref{i_gwp_2} imply that 
\begin{equation}\label{i_gwp_4}
 \se{2N}(T_1) + \frac{\f(T_1)}{(1+T_1)} < C_1 \kappa_0 \le \delta_0(\ep) \text{ for all } 0 < T_1 < T_*(\kappa_0),
\end{equation}
where $\ep$ is given by \eqref{i_gwp_15}.  We view $(u(T_1),p(T_1),\eta(T_1))$ as initial data for a new problem; since $(u,p,\eta)$ are already solutions, they  satisfy the compatibility conditions needed to use them as data.  Then since $\se{2N}(T_1) < \delta_0(\ep)$, we can use Theorem \ref{i_infinite_lwp} with $\ep$ given by \eqref{i_gwp_15} to extend  solutions to $[T_1, T_2]$ for any $T_2$ satisfying
\begin{equation}
 0 < T_2 - T_1 \le T_0 = C(\ep) \min\{ 1, \f(T_1)^{-1} \}. 
\end{equation}
In light of \eqref{i_gwp_4}, we may bound
\begin{equation}\label{i_gwp_5}
 \bar{T} :=  C(\ep) \min\left\{1, \frac{1}{\delta_0(\ep) (1+T_*(\kappa_0))} \right\} \le T_0.
\end{equation}
Notice that  $\bar{T}$ depends on $\ep$ (given by \eqref{i_gwp_15}) and $T_*(\kappa_0)$, but is independent of $T_1$.   Let 
\begin{equation}\label{i_gwp_6}
 \gamma = \min\left\{  \bar{T}, T_*(\kappa_0),  \frac{1}{(1+2 T_*(\kappa_0))^{1+\lambda/2}}        \right\},
\end{equation}
and then let us choose $T_1 = T_*(\kappa_0)-\gamma/2$ and $T_2 = T_*(\kappa_0) + \gamma/2$.  The choice of $\gamma$ implies that
\begin{equation}\label{i_gwp_65}
 0 < T_1  < T_*(\kappa_0) < T_2  <  2 T_*(\kappa_0) \text{ and } 0< \gamma= T_2 - T_1 \le \bar{T} \le T_0.
\end{equation}
Then Theorem \ref{i_infinite_lwp} allows us to extend solutions to the interval $[0,T_2]$, and it provides  estimates on the extended interval $[T_1,T_2]$: 
\begin{multline}\label{i_gwp_7}
 \sup_{T_1 \le t \le T_2} \se{2N}(t) + \sup_{T_1 \le t \le T_2} \ns{\il p(t)}_{0} + \int_{T_1}^{T_2} \sd{2N}(t) dt
\\
+ \int_{T_1}^{T_2} \left( \ns{\dt^{2N+1} u(t)}_{({_0}H^1)^*}  + \ns{ \dt^{2N} p(t) }_{0} \right) dt
\le C_2 \left( \ep +  \ns{\il u(T_1)}_{0} + \ns{\il \eta(T_1)}_{0}  \right),
\end{multline}
\begin{equation}\label{i_gwp_8}
 \sup_{T_1 \le t \le T_2} \mathfrak{E}_{2N}(t) \le \ep, \text{ and }
 \sup_{T_1 \le t \le T_2} \f(t) \le C_2 \f(T_1) + \ep.
\end{equation}

Having extended the existence interval, we will now show that $\g(T_2) \le \delta$.  We combine the estimates \eqref{i_gwp_7}--\eqref{i_gwp_8} with \eqref{i_gwp_3}--\eqref{i_gwp_4} and the bound \eqref{i_tne_01} of Proposition \ref{i_total_norm_estimate} to see that
\begin{multline}
 \g(T_2) < C_1 C_3 \kappa_0  + C_2(  \ep + C_1  \kappa_0   ) + \frac{C_1 C_2 \kappa_0 (1+T_1) + \ep}{(1+T_1)} 
+ \ep C_3 (T_2-T_1)^2 (1+T_2)^{2+\lambda} \\
\le \kappa_0 C_1 (C_3 + 2 C_2) + \ep ( 1+ C_2) + \ep C_3 \gamma^2 (1+2 T_*(\kappa_0))^{2+\lambda}  \\
\le  \frac{\delta}{3} + \frac{\delta}{3} + \frac{\delta}{3} = \delta,
\end{multline}
where the second inequality follows from \eqref{i_gwp_65} and the third follows from the choice of $\ep,$ $\kappa_0$, and $\gamma$ given in \eqref{i_gwp_15}, \eqref{i_gwp_2}, and \eqref{i_gwp_6}, respectively.  Hence $\g(T_2) \le \delta$, contradicting the definition of $T_*(\kappa_0)$.  We deduce then that $T_*(\kappa_0) = \infty$, which completes the proof of the claim and the theorem.

\end{proof}

With this result in hand, it is a simple matter to prove Theorem \ref{intro_inf_gwp}.

\begin{proof}[Proof of Theorem \ref{intro_inf_gwp}]

We set $N=5$ in Theorem \ref{i_gwp} to deduce all of the conclusions of Theorem \ref{intro_inf_gwp} except the estimates \eqref{intro_inf_gwp_02}--\eqref{intro_inf_gwp_03}.  Proposition \ref{i_improved_u} implies that
\begin{equation}
 \ns{u}_{C^2(\Omega)} \le C(r) (\se{8})^{r/(2+r)} (\se{6,2})^{2/(2+r)}
\end{equation}
for any $r \in (0,1)$, where $C(r)>0$ is a constant depending on $r$.  Let $0 \le \rho <\lambda$ and then choose $r \in (0,1)$ so that
\begin{equation}
 0< r \le 2 \left( \frac{2+\lambda}{2+\rho} \right) -2 \Rightarrow (2+\rho) \le (2+\lambda)\left( \frac{2}{2+r} \right).
\end{equation}
Then $C(r) = C(\rho)$ and the bound $\mathcal{G}_8(\infty) \le C_1 \kappa$ implies that
\begin{equation}
 \sup_{t\ge 0} (1+t)^{2+\rho} \ns{u(t)}_{C^2(\Omega)} \le C(\rho) C_1 \kappa \sup_{t \ge 0}(1+t)^{2+\rho}  \left(  \frac{1}{(1+t)^{2+\lambda}} \right)^{2/(2+r)} \le C(\rho) C_1 \kappa,
\end{equation}
which is \eqref{intro_inf_gwp_02}.  The estimate \eqref{intro_inf_gwp_03} follows similarly by using the interpolation estimates of Lemma \ref{i_interp_eta} for the $\eta$ terms and the interpolation estimates of Theorem \ref{i_bs_u} for $\ns{u}_{2}$.  In this case, though, no use of $r \in (0,1)$ is necessary because it does not appear in the interpolations.

\end{proof}

\chapter{Global well-posedness and decay in the periodic case}\label{section_per}

\section{Introduction}\label{per_1}

In this chapter we prove Theorem \ref{intro_per_gwp}.  Throughout the chapter we assume that $N \ge 3$  is fixed.  Note that this is an improvement on the constraint $N \ge 5$ used in Chapter \ref{section_inf}; this is possible since we do not need the same interpolation arguments used there.

In the rest of Section \ref{per_1} we define the energies and dissipations that are relevant to the periodic problem, and we describe how we localize to handle the curved boundary $b \in C^\infty(\Sigma)$.  In Section \ref{per_2} we present estimates of the nonlinear forcing terms $G^i$ (as defined in \eqref{Gi_def_start}--\eqref{Gi_def_end}) and some other nonlinearities. In Section \ref{per_3} we use the geometric form of the equations to estimate the evolution of temporal derivatives.    Section \ref{per_4} concerns similar energy evolution estimates for the localized energies.  For these, we employ the linear perturbed framework with the $G^i$ forcing terms.  In the upper localization we apply horizontal spatial derivatives as well as temporal derivatives, but in the lower localization we only apply temporal derivatives.  Section \ref{per_5} concerns the comparison estimates, where we show how to estimate the full energies and dissipations in terms of their horizontal counterparts. Section \ref{per_6} combines all of the analysis of Sections \ref{per_2}--\ref{per_5} into our a priori estimates for solutions to \eqref{geometric} in the periodic setting. Section \ref{per_7} concerns a specialized version of the local well-posedness theorem that guarantees that $\eta$ has zero average for all time.  Finally, in Section \ref{per_8} we record our global well-posedness and decay result, proving Theorem \ref{intro_per_gwp}.

Below, in \eqref{p_total_energy}, we will define the total energy $\g$ that we use in the global well-posedness analysis.  For the purposes of deriving our a priori estimates, we will assume throughout Sections \ref{inf_2}--\ref{inf_8} that solutions are given on the interval $[0,T]$ and that $\g(T) \le \delta$ for $0 < \delta < 1$ as small as in Lemma \ref{infinity_bounds} so that its conclusions hold.  This also means that $\se{2N}(t) \le 1$ for $t \in [0,T]$.  We will also assume throughout that the solutions satisfy the zero average condition 
\begin{equation}
 \int_\Sigma \eta(t) = 0 \text{ for all } t\in [0,T].
\end{equation}

We should remark that Theorem \ref{intro_lwp} does not produce solutions that necessarily satisfy the zero average condition.  To guarantee that this holds, we must record a specialized version of the local well-posedness result, Theorem \ref{p_periodic_lwp}.  We could record this result before the a priori estimates, but to keep the structure of this chapter similar to the structure of Chapter \ref{section_inf}, we postpone until after the a priori estimates.  Note that the bounds of Theorem \ref{p_periodic_lwp} control more than just $\g(T)$, and the extra control it provides guarantees that all of the calculations used in the a priori estimates are justified.

Finally, we note that most of the results of this chapter are variants of ones in Chapter \ref{section_inf}.  We have chosen to focus on highlighting the differences in the periodic case.  As such, for many results we have only sketched how to modify the proof of the corresponding result in Chapter \ref{section_inf}.  Whenever proofs require significantly different arguments, we have written full details.

\subsection{Localization}
Let $0<b_- := \inf_{x'} b(x')$ and $\sup_{x'} b(x') = b_+ < \infty$.  Let $\chi_i \in C_c^\infty(\Rn{})$ for $i=1,2,3$ with the property that 
\begin{equation}\label{p_chi_def} 
\begin{cases}
 \chi_1 =1 \text{ on }   [-3 b_-/4,1] &\text{and } \chi_1 = 0 \text{ on }   (-\infty,-7 b_- /8) \\
 \chi_2 =1 \text{ on }  [-(b_+ +1),  -b_-/2] &\text{and } \chi_2 = 0 \text{ on }  (-3 b_-/8,\infty) \\
 \chi_3 =1 \text{ on }  [-b_-/2,1] &\text{and } \chi_3 = 0 \text{ on } (-\infty,-5b_-/8).
\end{cases}
\end{equation}
We then define the subsets $\Omega_i \subset \Omega$  by
\begin{equation}\label{p_Omega_def}
\begin{split}
 \Omega_1 &= \{ -3 b_-/4 \le x_3 \le 0\} \cap \Omega, \\ 
 \Omega_2 &= \{ -b_+ \le x_3 \le -b_-/2\} \cap \Omega, \\
 \Omega_3 &= \{ -b_-/2 \le x_3 \le 0\} \cap \Omega.
\end{split}
\end{equation}
We will view the functions $\chi_i(x) = \chi_i(x_3)$ as cutoff functions in the vertical direction.  They are constructed so that $\chi_1 =1$ on $\Omega_i$ and so that $\Omega = \Omega_1 \cup \Omega_2= \Omega_3 \cup \Omega_2$, $\Omega_3 \subset \Omega_1$, and $\supp(\nab \chi_2) \subset \Omega_3$.

When we multiply the equations in \eqref{linear_perturbed} by $\chi_i$, $i=1,2$, we find that $( \chi_i u, \chi_i p, \eta)$ solve
\begin{equation}\label{p_localized_equations}
 \begin{cases}
  \dt (\chi_i u) + \nab (\chi_i p) - \Delta (\chi_i u) = \chi_i G^1 + H^{1,i} & \text{in }\Omega \\
  \diverge(\chi_i u) = \chi_i G^2 + H^{2,i} & \text{in }\Omega \\
  ((\chi_i p) I - \sg (\chi_i u) )e_3 = \delta_{i,1} \left( \eta e_3  + G^3 \right) & \text{on }\Sigma \\
   \dt \eta - (\chi_1 u_3) = G^4  &\text{on } \Sigma \\
  \chi_i u =0 & \text{on }\Sigma_b,
 \end{cases}
\end{equation}
where $\delta_{i,1}$ is the Kronecker delta and 
\begin{equation}\label{p_H_def}
H^{1,i} = \p_3 \chi_i (  p e_3 - 2  \p_3 u) - \p_3^2 \chi_i  u 
\text{ and } 
 H^{2,i} = \p_3 \chi_i  u_3.
\end{equation}
The $H$ functions have this form since $\chi_i$ is only a function of $x_3$.

\subsection{Definitions and notation}

We will consider energies and dissipates at both the $N+2$ and $2N$ levels.  To define both at once we consider a generic integer $n\ge 3$.  Recall that we use the derivative conventions described in Section \ref{def_and_term}.  We define the  energy as
\begin{equation}\label{p_energy_def}
 \se{n} = \sum_{j=0}^{n} \left( \ns{\dt^j u}_{2n-2j} + \ns{\dt^j \eta}_{2n-2j} \right) + \sum_{j=0}^{n-1} \ns{\dt^j p}_{2n-2j-1}.
\end{equation}
The corresponding dissipation is
\begin{multline}\label{p_dissipation_def}
 \sd{n} = \sum_{j=0}^{n} \ns{\dt^j u}_{2n-2j+1} + \sum_{j=0}^{n-1} \ns{\dt^j p}_{2n-2j} \\
+  \ns{ \eta}_{2n-1/2} + \ns{\dt \eta}_{2n-1/2} + \sum_{j=2}^{n+1} \ns{\dt^j \eta}_{2n-2j+5/2}.
\end{multline}

For our ``horizontal'' energies and dissipations, we must use different types of derivatives depending on the localization.  In the whole domain we only consider temporal derivatives, writing 
\begin{equation}\label{p_temporal_energies_def}
 \seb{n}^0 = \sum_{j=0}^{n-1} \ns{\dt^j u}_{0} + \ns{\sqrt{J} \dt^n u}_{0} + \sum_{j=0}^{n} \ns{\dt^j \eta}_{0} 
\text{ and }
 \sdb{n}^0 = \sum_{j=0}^{n} \ns{ \sg \dt^j u}_{0}. 
\end{equation}
In the upper localization we allow both horizontal spatial derivatives and temporal derivatives, but we do not allow the highest order temporal derivatives:
\begin{equation}\label{p_upper_energy_def}
 \seb{n}^{+} =  \ns{ \db_0^{2n-1}  (\chi_1 u)}_{0} + \ns{ D \db^{2n-1}  (\chi_1 u)}_{0} + \ns{ \db_0^{2n-1} \eta}_{0} +  \ns{ D \db^{2n-1} \eta}_{0},
\end{equation}
\begin{equation}\label{p_upper_dissipation_def}
  \sdb{n}^{+} =   \ns{ \db_0^{2n-1} \sg (\chi_1 u)}_{0} + \ns{ D \db^{2n-1} \sg (\chi_1 u)}_{0}.
\end{equation}
In the lower localization we only take temporal derivatives, but not all the way to the highest order:
\begin{equation}\label{p_lower_energies_def}
 \seb{n}^{-} = \sum_{j=0}^{n-1} \norm{\dt^j (\chi_2 u)}_{0}^2
\text{ and }
 \sdb{n}^{-} = \sum_{j=0}^{n-1} \norm{ \sg \dt^j (\chi_2 u)}_{0}^2.
\end{equation}

Our specialized energy terms are
\begin{equation}\label{p_transport_def}
 \f =  \ns{ \eta}_{4N+1/2}
\end{equation}
and
\begin{equation}
 \k := \pnorm{\nab u}{\infty}^2 + \pnorm{\nab^2 u}{\infty}^2+ \sum_{i=1}^2 \snormspace{D u_i}{2}{\Sigma}^2.
\end{equation}
The total energy we will use in our global well-posedness result is 
\begin{equation}\label{p_total_energy}
 \g(t) = \sup_{0 \le r \le t} \se{2N}(r) + \int_0^t \sd{2N}(r) dr + \sup_{0 \le r \le t} (1+r)^{4N-8} \se{N+2}(r) + \sup_{0 \le r \le t} \frac{\f(r)}{(1+r)}.
\end{equation}

\section{Nonlinear estimates}\label{per_2}

\subsection{Estimates of $G^i$ at the $N+2$ level}

We now estimate the $G^i$ terms at the $N+2$ level.  The result is similar to ones we have already proved, so we only sketch the proof.

\begin{thm}\label{p_G_estimates_half}
Then there exists a $\theta >0$ so that
\begin{multline}\label{p_G_e_h_0}
\ns{ \bar{\nab}_0^{2(N+2)-2} G^1}_{0} +  \ns{ \bar{\nab}_0^{2(N+2)-2}  G^2}_{1} +
 \ns{ \db_0^{2(N+2)-2} G^3}_{1/2} + \ns{\db_0^{2(N+2)-2} G^4}_{1/2}
\\ 
\ls \se{2N}^{\theta}\se{N+2}
\end{multline}
and
\begin{multline}\label{p_G_e_h_00}
\ns{ \bar{\nab}_0^{2(N+2)-1} G^1}_{0} +  \ns{ \bar{\nab}_0^{2(N+2)-1}  G^2}_{1}  +
 \ns{ \db_0^{2(N+2)-1} G^3}_{1/2}  + \ns{\db_0^{2(N+2)-1} G^4}_{1/2}
\\ 
+\ns{\bar{D}^{2(N+2)-2} \dt G^4}_{1/2}
\ls
\se{2N}^\theta \sd{N+2}.
\end{multline}
\end{thm}
\begin{proof}

Since there are no minimal derivatives in $\se{N+2}$ and $\sd{N+2}$, we need not resort to interpolation arguments like in Theorem \ref{i_G_estimates_half}.  Instead, we may estimate directly, as in the proof of \eqref{i_G_e_0} in Theorem \ref{i_G_estimates}, using Sobolev embeddings, trace theory, and Lemmas \ref{i_sobolev_product_1}, \ref{p_poisson}, and \ref{p_poisson_2}.  Note that for \eqref{p_G_e_h_00}, we do not need specialized estimates like those of Lemma \ref{i_eta_half_product} since in products of the form  $D^{2N+4} \eta Q  u$ we may estimate $D^{2N+4} \eta$ with  $\se{2N}$ and $u$ with $\sd{N+2}$.

\end{proof}

\subsection{Estimates of $G^i$ at the $2N$ level}

Now we estimate $G^i$ at the $2N$ level.  The result is similar to ones we have already proved, so we only sketch the proof.

\begin{thm}\label{p_G_estimates}
Then there exists a $\theta >0$ so that
\begin{equation}\label{p_G_e_0}
\ns{ \bar{\nab}_0^{4N-2} G^1}_{0}  +  \ns{ \bar{\nab}_0^{4N-2}  G^2}_{1} +
 \ns{ \db_{0}^{4N-2} G^3}_{1/2} + \ns{\bar{D}_0^{4N-2} G^4}_{1/2}
 \ls 
\se{2N}^{1+\theta},
\end{equation}
\begin{multline}\label{p_G_e_00}
\ns{ \bar{\nab}_{0}^{4N-2} G^1}_{0}  +  \ns{ \bar{\nab}_0^{4N-2}  G^2}_{1}  +
 \ns{ \db_{0}^{4N-2} G^3}_{1/2}  + \ns{\db_0^{4N-2} G^4}_{1/2} \\
+ 
\ns{ \bar{\nab}^{4N-3} \dt G^1}_{0}  +  \ns{ \bar{\nab}^{4N-3} \dt G^2}_{1}  +
 \ns{ \db^{4N-3} \dt G^3}_{1/2}  + \ns{\db^{4N-2} \dt G^4}_{1/2}
 \\ 
\ls \se{2N}^\theta \sd{2N},
\end{multline}
and
\begin{multline}\label{p_G_e_000}
 \ns{ \nab^{4N-1} G^1}_{0}  +  \ns{ \nab^{4N-1}  G^2}_{1} +
 \ns{ D^{4N-1} G^3}_{1/2}  + \ns{D^{4N-1} G^4}_{1/2}
\\ \ls
\se{2N}^\theta \sd{2N} + \k \f.
\end{multline}
\end{thm}
\begin{proof}

The proof of \eqref{p_G_e_0} and \eqref{p_G_e_00} proceeds as in Theorem \ref{p_G_estimates_half}, using Sobolev embeddings, trace theory, and Lemmas \ref{i_sobolev_product_1}, \ref{p_poisson}, and \ref{p_poisson_2} to estimate $\pa G^i$.  To derive the estimates \eqref{p_G_e_000} we argue in the same way, except for the exceptional terms involving $\nab^{4N+1} \bar{\eta}$ and $D^{4N} \eta$, which can be estimated as in Theorem \ref{i_G_estimates}, using Lemma \ref{p_poisson} in place of Lemma \ref{i_poisson_grad_bound}.

\end{proof}

\subsection{Other nonlinearities}

Now we provide an estimate of for $\dt^j \mathcal{A}$ when $j= 2N+1$ and when $j=N+3$.

\begin{lem}\label{p_A_time_derivs}
We have that
\begin{equation}\label{p_atd_0}
\ns{\dt^{2N+1} \mathcal{A}}_{0} \ls \sd{2N},
\text{ and }
\ns{\dt^{N+3} \mathcal{A}}_{0} \ls \sd{N+2}.
\end{equation}
\end{lem}
\begin{proof}
 The proof is the same as that of Lemma \ref{i_A_time_derivs} except that we use Lemma \ref{p_poisson} in place of Lemma \ref{i_poisson_grad_bound}.
\end{proof}

\section{Global energy evolution in the geometric form}\label{per_3}

\subsection{Estimates of the perturbations when $\pa=\dt^{\alpha_0}$ is applied to \eqref{geometric}}

We now present estimates for the perturbations \eqref{F_def_start}--\eqref{F_def_end} when $\pa = \dt^{\alpha_0}$ for $\alpha_0 \le 2N$.

\begin{thm}\label{p_F_estimates}
Let $\pa = \dt^{\alpha_0}$ with $\alpha_0 \le 2N$ and let $F^1$, $F^2$, $F^3$, $F^4$ be defined by \eqref{F_def_start}--\eqref{F_def_end}.  Then
\begin{equation}\label{p_F_e_01}
 \ns{F^{1} }_{0} + \ns{\dt (J F^{2} ) }_{0} + \ns{F^{3}}_{0} + \norm{F^{4}}_{0} \ls \se{2N} \sd{2N}.
\end{equation}
\end{thm}
\begin{proof}

The proof is the same as that of Theorem \ref{i_F_estimates}, except that we use Lemmas \ref{p_poisson} and \ref{p_poisson_2} to estimate $\bar{\eta}$ terms, we apply Lemma \ref{p_A_time_derivs} in place of Lemma \ref{i_A_time_derivs}, and we allow any $0 \le \alpha_0 \le 2N$.

\end{proof}

We now present estimates for these perturbations when $\pa = \dt^{\alpha_0}$ with $\alpha_0 \le N+2$.

\begin{thm}\label{p_F_estimates_half}
Let $\pa = \dt^{\alpha_0}$ with $\alpha_0 \le N+2$ and let $F^1$, $F^2$, $F^3$, $F^4$ be defined by \eqref{F_def_start}--\eqref{F_def_end}.  Then 
\begin{equation}\label{p_F_e_h_01}
 \ns{F^{1} }_{0} + \ns{\dt (J F^{2} ) }_{0} + \ns{F^{3}}_{0} + \norm{F^{4}}_{0} \ls \se{2N} \sd{N+2}.
\end{equation}\
Also, if $N\ge 3$, then there exists a $\theta> 0$ so that
\begin{equation}\label{p_F_e_h_02}
 \ns{F^2}_{0} \ls \se{2N}^\theta \se{N+2.}
\end{equation}
\end{thm}
\begin{proof}
The proof is the same as that of Theorem \ref{i_F_estimates_half}, except that we use Lemmas \ref{p_poisson} and \ref{p_poisson_2} to estimate $\bar{\eta}$ terms, we apply Lemma \ref{p_A_time_derivs} in place of Lemma \ref{i_A_time_derivs},  we allow any $0 \le \alpha_0 \le N+2$, and we bound $\pns{\p_3 u_3}{\infty} \ls \se{N+2}$ by using Sobolev embeddings rather than  interpolation.

\end{proof}

\subsection{Global energy evolution with only temporal derivatives}

Now we present the applications of Theorems \ref{p_F_estimates} and \ref{p_F_estimates_half}.

\begin{prop}\label{p_global_evolution}
There exists a $\theta>0$ so that 
\begin{equation}\label{p_gl_e_0}
 \seb{2N}^0(t) + \int_0^t \sdb{2N}^0  \ls \se{2N}(0) + (\se{2N}(t))^{3/2} 
+ \int_0^t  (\se{2N})^\theta  \sd{2N}.
\end{equation}
\end{prop}
\begin{proof}
 The proof of Proposition \ref{i_temporal_evolution} works here, using Theorem \ref{p_F_estimates} in place of Theorem \ref{i_F_estimates}, for each $\dt^{\alpha_0}$ with $0 \le \alpha_0 \le 2N$.  The desired estimate follows by summing over $0\le \alpha_0 \le 2N$.
\end{proof}

Now we present the corresponding estimate at the $N+2$ level.

\begin{prop}\label{p_global_evolution_half}
Let $F^2$ be given by \eqref{i_F2_def} with $\pa = \dt^{N+2}$.  Then
\begin{equation}
 \dt \left( \seb{N+2}^0  -2 \int_\Omega J \dt^{N+1} p F^2  \right)+ \sdb{N+2}^0  \ls   \sqrt{\se{2N} }  \sd{N+2}.
\end{equation}
\end{prop}
\begin{proof}
 The proof of Proposition \ref{i_temporal_evolution_half} works here, using Theorem \ref{p_F_estimates_half} in place of Theorem \ref{i_F_estimates_half}, for each $\dt^{\alpha_0}$ with $0 \le \alpha_0 \le N+2$.  The desired estimate follows by summing over $0\le \alpha_0 \le N+2$.
\end{proof}

\section{Localized energy evolution using the perturbed linear form}\label{per_4}

\subsection{Upper localization}

We now estimate how the upper-localization energies evolve.  In order to analyze the upper localization, we will use the equation \eqref{p_localized_equations}  with $i=1$.

\begin{prop}\label{p_upper_evolution}
Let $\alpha \in \mathbb{N}^{1+2}$ so that $\alpha_0 \le 2N-1$ and $\abs{\alpha}\le 4N$.  Then for any $\ep \in (0,1)$ it holds that
\begin{multline}\label{p_u_e_0}
     \ns{\pa (\chi_1 u)}_{0}  + \ns{\pa \eta}_{0}  + \int_0^t \ns{\sg \pa (\chi_1 u)}_{0} \\
\ls \seb{2n}^+(0) +  \int_0^t \se{2N}^\theta \sd{2N} + \sqrt{ \sd{2n} \k \f } + \ep \sd{2N} + \ep^{-8N-1} \sdb{2N}^0.
\end{multline}
In particular,
\begin{equation}\label{p_u_e_00}
 \seb{2N}^+(t) + \int_0^t \sdb{2N}^+ \ls \seb{2N}^+(0) + \int_0^t  \se{2N}^\theta \sd{2N} + \sqrt{ \sd{2n} \k \f } + \ep \sd{2N} + \ep^{-8N-1} \sdb{2N}^0.
\end{equation}

\end{prop}
\begin{proof}
We apply Lemma \ref{general_evolution} to $v=  \chi_1 \pa u $, $q = \chi_1 \pa p$, $\zeta =\pa \eta$ with $a=1$, $\Phi^1 = \chi_1 \pa G^1 + \pa H^{1,1}$, $\Phi^2 = \chi_1 \pa G^2 + \pa H^{2,1}$, $\Phi^3 =\pa G^3$, and $\Phi^4 = \pa G^4$ to find
\begin{multline}\label{p_u_e_1}
 \dt  \left( \hal \int_\Omega \abs{\pa (\chi_1 u)}^2  + \hal\int_\Sigma \abs{\pa \eta}^2 \right) 
+ \hal \int_\Omega \abs{\sg \pa (\chi_1 u)}^2 
= \int_\Omega \chi_1 \pa  u \cdot ( \chi_1 \pa G^1 + \pa H^{1,1}) \\
+ \int_\Omega \chi_1 \pa  p (\chi_1 \pa G^2 + \pa H^{2,1} ) 
+ \int_\Sigma -\pa u \cdot \pa G^3 + \pa \eta \pa G^4.
\end{multline}
Here $H^{1,1}$ and $H^{2,1}$ are given by \eqref{p_H_def}.  We will estimate the terms on the right side of \eqref{p_u_e_1}, beginning with the terms involving $H^{1,1}$ and $H^{2,1}$.

Since $\chi_1$ is only a function of $x_3$, we have that
\begin{equation}
 \pa H^{1,1} = \p_3 \chi_1 ( \pa p e_3 - 2 \pa \p_3 u) - \p_3^2 \chi_1 \pa u \text{ and } \pa H^{2,1} = \p_3 \chi_1 \pa u_3.
\end{equation}
This and the constraints on $\alpha$  allow us to estimate
\begin{multline}\label{p_u_e_2}
 \int_\Omega \chi_1 \pa u  \cdot \pa H^{1,1}  + \chi_1 \pa p \pa H^{2,1} \ls 
\norm{\pa u}_{0}\left(\norm{\pa p}_{0} + \norm{\pa u}_{1} \right) 
+ \norm{\pa p}_{0} \norm{\pa u}_{0} \\
\ls \norm{\pa u}_{0}\left(\norm{\pa p}_{0} + \norm{\pa u}_{1} \right) \ls
\norm{D_0^{4N-2\alpha_0} \dt^{\alpha_0} u}_{0} \sqrt{\sd{2N}} 
  \ls \sqrt{\sd{2N}} \norm{\dt^{\alpha_0} u}_{4N-2\alpha_0} 
\end{multline}
We estimate the $4N-2\alpha_0$ norm with standard Sobolev interpolation:
\begin{equation}
\norm{\dt^{\alpha_0} u}_{4N-2\alpha_0}  \ls \norm{\dt^{\alpha_0} u}_{0}^\theta  \norm{\dt^{\alpha_0} u}_{4N-2\alpha_0+1}^{1-\theta} \le (\sdb{2N}^0)^{\theta/2} (\sd{2N})^{(1-\theta)/2},
\end{equation}
where $\theta = (4N-2\alpha_0 +1)^{-1} \in (0,1)$.  Then Young's inequality allows us to further bound 
\begin{multline}\label{p_u_e_3}
\sqrt{\sd{2N}} \norm{\dt^{\alpha_0} u}_{4N-2\alpha_0} \ls 
\sqrt{\sd{2N}} (\sdb{2N}^0)^{\theta/2} (\sd{2N})^{(1-\theta)/2} 
 = (\sdb{2N}^0)^{\theta/2} (\sd{2N})^{1-\theta/2}  \\
\le \ep\left( 1- \frac{\theta}{2}\right) \sd{2N} + \frac{\theta}{2} \ep^{(\theta-2)/\theta} \sdb{2N}^0 
\le \ep \sd{2N} + \ep^{-8N-1} \sdb{2N}^0,
\end{multline}
where in the last inequality we have used the fact that $(2-\theta)/\theta = 8N -4 \alpha_0 +1$ to find the largest power of $1/\ep$ when  $0 \le \alpha_0 \le 2N$.  Chaining together \eqref{p_u_e_2} and \eqref{p_u_e_3} then yields the bound
\begin{equation}\label{p_u_e_4}
 \int_\Omega \chi_1 \pa u  \cdot \pa H^{1,1}  + \chi_1 \pa p \pa H^{2,1} \ls \ep \sd{2N} + \ep^{-8N-1} \sdb{2N}^0.
\end{equation}

We now turn to estimates of the terms involving $G^i$, $1\le i \le 4$.   We claim that
\begin{equation}\label{p_u_e_5}
 \int_\Omega \chi_1^2 \left( \pa  u \cdot   \pa G^1 +  \pa p \pa G^2 \right) \ls 
(\se{2N})^\theta \sd{2N} + \sqrt{\sd{2N} \k \f}
\end{equation}
and
\begin{equation}\label{p_u_e_6}
 \int_\Sigma -\pa u \cdot \pa G^3 + \pa \eta \pa G^4 \ls 
(\se{2N})^\theta \sd{2N} + \sqrt{\sd{2N} \k \f}
\end{equation}
for some $\theta >0$.  The estimate \eqref{p_u_e_6} may be derived exactly as in Proposition \ref{i_derivative_evolution}, using Theorem \ref{p_G_estimates} to estimate the $G^3$ and $G^4$ terms.  The estimate \eqref{p_u_e_5} may also be derived as in  Proposition \ref{i_derivative_evolution}, using Theorem \ref{p_G_estimates} to estimate the $G^1$ and $G^2$ terms, except that the $\chi_1^2$ terms must be trivially bounded in $L^\infty$.  Note that since $\chi_1$ depends only on $x_3$, integration by parts with $\p_1$ and $\p_2$ does not introduce any new terms through the product rule.

Now, in light of \eqref{p_u_e_1} and  \eqref{p_u_e_4}--\eqref{p_u_e_6}, we have
\begin{multline}\label{p_u_e_9}
\dt  \left(  \int_\Omega \abs{\pa (\chi_1 u)}^2  + \int_\Sigma \abs{\pa \eta}^2 \right) 
+  \int_\Omega \abs{\sg \pa (\chi_1 u)}^2  \\
\ls (\se{2N})^\theta \sd{2N} + \sqrt{\sd{2N} \k \f}
+ \ep \sd{2N} + \ep^{-8N-1} \sdb{2N}^0
\end{multline}
for all $\abs{\alpha} \le 4N$ with $\alpha_0 \le 2N-1$.   The estimate \eqref{p_u_e_0} then follows from \eqref{p_u_e_9} by integrating in time from $0$ to $t$, and then \eqref{p_u_e_00} follows from \eqref{p_u_e_0} by summing over $\alpha$.

\end{proof}

Now we prove a similar estimate at the $N+2$ level.

\begin{prop}\label{p_upper_evolution_half}
Let $\alpha \in \mathbb{N}^{1+2}$ so that $\alpha_0 \le N+1$ and $\abs{\alpha}\le 2(N+2)$.  Then for any $\ep \in (0,1)$ it holds that
\begin{multline}\label{p_ueh_0}
   \dt \left(  \ns{\pa (\chi_1 u)}_{0}  + \ns{\pa \eta}_{0} \right) +  \ns{\sg \pa (\chi_1 u)}_{0} \\
\ls   \se{2N}^\theta \sd{N+2}  + \ep \sd{N+2} + \ep^{-4N-9} \sdb{N+2}^0.
\end{multline}
In particular,
\begin{equation}\label{p_ueh_00}
 \dt \seb{N+2}^+ +   \sdb{N+2}^+ \ls  \se{2N}^\theta \sd{N+2}  + \ep \sd{N+2} + \ep^{-4N-9} \sdb{N+2}^0.
\end{equation}
\end{prop}
\begin{proof}
The argument used in Proposition \ref{p_upper_evolution} may be employed in this case as well, except that we use Theorem \ref{p_G_estimates_half} to estimate the $G^i$ terms, and when we interpolate we have
\begin{equation}
\norm{\dt^{\alpha_0} u}_{2N+4-2\alpha_0}  \ls \norm{\dt^{\alpha_0} u}_{0}^\theta  \norm{\dt^{\alpha_0} u}_{2N+5-2\alpha_0}^{1-\theta}
\end{equation}
for $\theta = (2N+5-2\alpha_0)^{-1} \in (0,1)$ so that $(2-\theta)/\theta = 4N+9 - 2 \alpha_0 \le 4N+9$, which gives the power of $1/\ep$ in the estimates.

\end{proof}

\subsection{Lower localization}

We now consider the evolution of the lower-localization energies at the $2N$ level.

\begin{prop}\label{p_lower_evolution}
Let $j$ be an integer satisfying $0 \le j \le 2N-1$.  Then for any $\ep \in (0,1)$ it holds that
\begin{equation}\label{p_l_e_0}
     \ns{\dt^j (\chi_2 u)}_{0}  + \int_0^t \ns{\sg \dt^j (\chi_1 u)}_{0} 
\ls \seb{2N}^-(0) + \int_0^t \se{2N}^\theta \sd{2N} + \ep \sd{2N} + \ep^{-8N-1} \sdb{2N}^0.
\end{equation}
In particular, 
\begin{equation}\label{p_l_e_00}
 \seb{2N}^-(t) + \int_0^t \sdb{2N}^-  \ls \seb{2N}^-(0) + \int_0^t  (\se{2N})^\theta \sd{2N}  + \ep \sd{2N} + \ep^{-8N-1} \sdb{2N}^0.
\end{equation}
\end{prop}

\begin{proof}

We apply Lemma \ref{general_evolution} to $v=  \chi_2 \dt^j u $, $q = \chi_2 \dt^j p$, $\zeta = \dt^j \eta$ with $a=0$, $\Phi^1 = \chi_2 \dt^j G^1 + \dt^j H^{1,2}$, $\Phi^2 = \chi_2 \dt^j G^2 + \dt^j H^{2,2}$, $\Phi^3 = 0$, and $\Phi^4 = 0$ to find
\begin{multline}\label{p_l_e_1}
 \dt  \left( \hal \int_\Omega \abs{\dt^j (\chi_2 u)}^2  \right) 
+ \hal \int_\Omega \abs{\sg \dt^j (\chi_2 u)}^2 
= \int_\Omega \chi_2 \dt^j  u \cdot ( \chi_2 \dt^j G^1 + \dt^j H^{1,2}) \\
+ \int_\Omega \chi_2 \dt^j  p (\chi_2 \dt^j G^2 + \dt^j H^{2,2} ).
\end{multline}
Here $H^{1,2}$ and $H^{2,2}$ are given by \eqref{p_H_def}.  The right hand side may then be estimated as in Proposition \ref{p_upper_evolution}, using only the temporal derivative estimates of Theorem \ref{p_G_estimates}.  In particular, we have the estimates 
\begin{equation}
 \int_\Omega \chi_2 \dt^j  u \cdot  \dt^j H^{1,2} +  \chi_2 \dt^j  p  \dt^j H^{2,2}  \ls 
\ep \sd{2N} + \ep^{-8N-1} \sdb{2N}^0
\end{equation}
and 
\begin{equation}
 \int_\Omega \chi_2^2 ( \dt^j  u \cdot  \dt^j G^1 + 
\dt^j  p  \dt^j G^2 ) \ls  (\se{2N})^\theta \sd{2N} ,
\end{equation}
which  yield \eqref{p_l_e_0} when combined with \eqref{p_l_e_1} and integrated in time from $0$ to $t$.  Then \eqref{p_l_e_00} follows from \eqref{p_l_e_0} by summing over $0\le j \le 2N-1$.
\end{proof}

Now we prove the corresponding result at the  $N+2$ level.

\begin{prop}\label{p_lower_evolution_half}
Let $j$ be an integer satisfying $0 \le j \le N+1$.  Then for any $\ep \in (0,1)$ it holds that
\begin{equation}\label{p_leh_0}
  \dt \left( \ns{\dt^j (\chi_2 u)}_{0} \right) +   \ns{\sg \dt^j (\chi_1 u)}_{0} 
\ls   \se{2N}^\theta \sd{N+2} + \ep \sd{N+2} + \ep^{-4N-9} \sdb{N+2}^0.
\end{equation}
In particular, 
\begin{equation}\label{p_leh_00}
 \dt \seb{N+2}^- +   \sdb{N+2}^-  \ls   (\se{2N})^\theta \sd{N+2}  + \ep \sd{N+2} + \ep^{-4N-9} \sdb{N+2}^0.
\end{equation}
\end{prop}
\begin{proof}
The proof proceeds as in Proposition \ref{p_lower_evolution}, following Proposition \ref{p_upper_evolution_half} rather than Proposition \ref{p_upper_evolution}, and using the $\dt^j G^i$ estimates of Theorem \ref{p_G_estimates_half} rather than of Theorem \ref{p_G_estimates}.

\end{proof}

\section{Comparison results}\label{per_5}

We now show that, up to some error terms, the instantaneous energy $\se{2N}$ is comparable to the sum $\seb{2N}^0 +  \seb{2N}^+$ and that the dissipation rate $\sd{2N}$ is comparable to the sum $\sdb{2N}^0 + \sdb{2N}^- + \sdb{2N}^+$.  We also prove similar results with $2N$ replaced by $N+2$.

\subsection{Instantaneous energy}

We begin with the result for the instantaneous energy.

\begin{thm}\label{p_energy_bound}
There exists a $\theta >0$ so that
\begin{equation}\label{p_E_b_0}
 \se{2N} \ls \seb{2N}^{+} +  \seb{2N}^0  + (\se{2N})^{1+\theta} 
\end{equation} 
and
\begin{equation}\label{p_E_b_00}
 \se{N+2} \ls \seb{N+2}^{+} +  \seb{N+2}^0  + (\se{2N})^{\theta} \se{N+2}.
\end{equation} 

\end{thm}

\begin{proof}
In order to prove the result at both the $2N$ and $N+2$ levels at the same time, we will generically write $n$ to refer to either quantity.  In the proof we will write
\begin{equation}
 \w_n =  \sum_{j=0}^{n-1} \ns{ \dt^j G^1 }_{2n-2j-2} + \ns{\dt^j  G^2  }_{2n-2j-1}  + \ns{\dt^j  G^3 }_{2n-2j-3/2}.
\end{equation}
Note that  the definitions of $\seb{n}^{+}$ and $\seb{n}^{0}$ guarantee that 
\begin{equation}\label{p_E_b_4}
 \sum_{j=0}^{n} \ns{\dt^j \eta}_{2n-2j} \ls   \seb{n}^{+} + \seb{n}^{0}.
\end{equation}

The key to proving the result is the following elliptic estimate.  Let $j=0,\dotsc,n-1$. Then we may apply $\dt^j$ to the equations of \eqref{linear_perturbed} and  use Lemma \ref{i_linear_elliptic} to see that
\begin{multline}\label{p_E_b_2}
 \norm{\dt^j  u  }_{2n-2j}^2 + \norm{\dt^j  p  }_{2n-2j-1}^2   \ls 
\norm{\dt^{j+1} u   }_{2n-2(j+1) }^2 + \norm{ \dt^j G^1   }_{2n-2j-2}^2 
+ \norm{\dt^j  G^2  }_{2n-2j-1}^2 \\
 + \norm{\dt^j  \eta }_{2n-2j-3/2}^2 + \norm{\dt^j  G^3  }_{2n-2j-3/2}^2 
\ls   \norm{\dt^{j+1} u   }_{2n-2(j+1) }^2 + \seb{n}^{+} + \seb{n}^{0} +  \w_n.
\end{multline}
In the last inequality of \eqref{p_E_b_2} we have used \eqref{p_E_b_4} and the definition of $\w_n$.

We claim that 
\begin{equation}\label{p_E_b_5}
 \se{n} \ls \seb{n}^{+} +  \seb{n}^0  + \w_n.
\end{equation} 
To prove this claim, we will use estimate \eqref{p_E_b_2} and a finite induction.  For  $j=n-1$ we employ the definition of $\seb{n}^0$ in \eqref{p_E_b_2} to  get
\begin{equation}
 \norm{\dt^{n-1}  u  }_{2}^2 + \norm{\dt^{n-1}  p  }_{1}^2   \ls \norm{\dt^{n} u   }_{0 }^2 +\seb{n}^{+} +  \seb{n}^0  +  \w_n \ls \seb{n}^{+} +  \seb{n}^0 + \w_n.
\end{equation}
Now suppose that the inequality
\begin{equation}\label{p_E_b_3}
 \norm{\dt^{n-\ell}  u  }_{2\ell }^2 + \norm{\dt^{n-\ell} p  }_{2\ell-1}^2  \ls 
\seb{n}^{+} + \seb{n}^0 +\w_n
\end{equation}
holds for $1 \le \ell < n$.   We apply \eqref{p_E_b_2} with $j=n-\ell-1$ and use the induction hypothesis \eqref{p_E_b_3} to find
\begin{multline}
 \norm{\dt^{n-\ell-1} u  }_{2(\ell+1)}^2 + \norm{\dt^{n-\ell-1}  p }_{2(\ell+1) -1}^2  \ls
 \norm{\dt^{n-\ell} u   }_{2\ell }^2 +\seb{n}^{+} + \seb{n}^{0} +  \w_n
\\ \ls \seb{n}^{+} +  \seb{n}^0  + \w_n.
\end{multline}
Hence  \eqref{p_E_b_3} holds with $\ell$ replaced by $\ell+1$, and by finite induction, 
\begin{equation}\label{p_E_b_6}
\sum_{j=0}^{n-1} \norm{\dt^{j}  u  }_{2n-2j }^2 + \norm{\dt^{j} p  }_{2n-2j-1}^2  \ls 
\seb{n}^{+} + \seb{n}^0 +\w_n. 
\end{equation}
We then sum \eqref{p_E_b_4}, \eqref{p_E_b_6}, and the trivial inequality $\ns{\dt^n u}_{0} \le \seb{n}^0$ to deduce that \eqref{p_E_b_5} holds.

To conclude, we must estimate $\w_n$ for $n=2N$ and $n=N+2$.  When $n=N+2$, we use \eqref{p_G_e_h_0} of Theorem \ref{p_G_estimates_half} to bound
$ \w_{N+2} \ls (\se{2N})^\theta \se{N+2},$ and when $n=2N$ we use \eqref{p_G_e_0} of Theorem \ref{p_G_estimates} to bound $\w_{2N} \ls (\se{2N})^{1+\theta}.$  These two estimates and \eqref{p_E_b_5} then imply \eqref{p_E_b_0} and \eqref{p_E_b_00}.

\end{proof}

\subsection{Dissipation}

Now we consider the dissipation rate.

\begin{thm}\label{p_dissipation_bound}
For $n=N+2$ or $n=2N$, write
\begin{multline}
 \y_{n} =  \ns{ \bar{\nab}_0^{2n-1} G^1}_{0} +  \ns{ \bar{\nab}_0^{2n-1}  G^2}_{1} \\ +
 \ns{ \db_0^{2n-1} G^3}_{1/2} + \ns{\db_0^{2n-1} G^4}_{1/2} 
+ \ns{\db_0^{2n-2} \dt G^4}_{1/2}.
\end{multline}
Then
\begin{equation}\label{p_D_b_00}
 \sd{n} \ls \sdb{n}^0 + \sdb{n}^- + \sdb{n}^+ + \y_n.
\end{equation}
In particular, there is a $\theta >0$ so that
\begin{equation}\label{p_D_b_01}
 \sd{2N} \ls \sdb{2N}^0 + \sdb{2N}^{-} + \sdb{2N}^{+} +  (\se{2N})^\theta \sd{2N} + \k \f 
\end{equation}
 and
\begin{equation}\label{p_D_b_02}
 \sd{N+2} \ls \sdb{N+2}^0 + \sdb{N+2}^{-} + \sdb{N+2}^{+} + (\se{2N})^\theta \sd{N+2}
\end{equation}

\end{thm}

\begin{proof}
 
In this proof we use a separate counting for spatial and temporal derivatives, so unlike elsewhere, we now use $\alpha \in \mathbb{N}^2$ to refer  only to  spatial derivatives.  In order to compactly write our estimates, throughout the proof we will write
\begin{equation}
 \z :=\sdb{n}^0 + \sdb{n}^{+} +  \y_n.
\end{equation}
The proof is divided into several steps, following those of Theorem \ref{i_dissipation_bound_general}.

Step 1 -- Application of Korn's inequality

First note that according to Lemma \ref{i_korn} we have
\begin{equation}
   \snormspace{ \db_0^{2n-1}    u}{1}{\Omega_1}^2 + \snormspace{ D \db^{2n-1}   u}{1}{\Omega_1}^2 \ls
\ns{ \db_0^{2n-1}  (\chi_1 u)}_{1} + \ns{ D \db^{2n-1} (\chi_1 u)}_{1}
\ls \sdb{n}^{+} 
\end{equation}
and
\begin{equation}\label{p_D_b_1}
  \sum_{j=0}^{n} \norm{  \dt^j u}_{1}^2  \ls \sdb{n}^0.
\end{equation}
Here, we recall that $\Omega_1 \subset \Omega$ is defined in \eqref{p_Omega_def}.  Summing these yields the bound
\begin{equation}\label{p_D_b_2}
    \snormspace{ \db_0^{2n}  u }{1}{\Omega_1}^2  \ls \sdb{n}^{+}  + \sdb{n}^0.
\end{equation}

Step 2 -- Initial estimates of the pressure and improvement of $u$ estimates

Recall that $\chi_3$ is given by \eqref{p_chi_def}, $\Omega_3\subset \Omega_1 $ is given by \eqref{p_Omega_def}, and $\chi_3 = 1$ on $\Omega_3$.  We claim that we have the estimate
\begin{equation}\label{p_D_b_19} 
  \snormspace{\db_0^{2n-1} u}{2}{\Omega_3}^2 
+ \snormspace{\db_0^{2n-1} \nab p}{0}{\Omega_3}^2 \ls \z.
\end{equation}
To prove this, we may argue as in Step 2 of Theorem \ref{i_dissipation_bound_general}, first using the structure of the equations \eqref{linear_perturbed} to derive various estimates of terms involving $\p_3$, and then localizing and using elliptic estimates for $\omega = \curl{u}$ to recover other terms with $\p_3$.  The only difference in the present case is that we localize with $\chi_3$ in place of the generic $\chi$ used in Theorem \ref{i_dissipation_bound_general}.

Step 3 -- Bootstrapping, $\eta$ estimates, and improved pressure estimates

Now we make use of Lemma \ref{p_bootstrap_estimate} to bootstrap from \eqref{p_D_b_19} to 
\begin{equation}\label{p_D_b_20}
  \sum_{j=0}^{n-1} \snormspace{ \dt^j u}{2n-2j+1}{\Omega_3}^2 
+  \sum_{j=0}^{n-1} \snormspace{ \dt^j \nab p}{2n-2j -1}{\Omega_3}^2 \ls \z.
\end{equation}
With this estimate in hand, we may derive some estimates for $\eta$ on $\Sigma$ by employing the boundary conditions of \eqref{linear_perturbed}:
\begin{equation}\label{p_D_b_21}
 \eta = p - 2 \p_3 u_3 - G^3_3,
\end{equation}
\begin{equation}\label{p_D_b_22}
 \dt \eta = u_3 + G^4.
\end{equation}
We differentiate  \eqref{p_D_b_21} and employ \eqref{p_D_b_20} to find that
\begin{multline}
 \ns{D \eta}_{2n-3/2} \ls   \snormspace{D p}{2n-3/2}{\Sigma}^2  
+ \snormspace{ D \p_3 u_3}{2n-3/2}{\Sigma}^2 + \ns{ D G^3}_{2n-3/2} \\
\ls \snormspace{D p}{2n-1}{\Omega_3}^2  
+ \snormspace{ D \p_3 u_3}{2n-1}{\Omega_3}^2 + \ns{ G^3}_{2n - 1/2}
\ls \z,
\end{multline}
so that by the usual Poincar\'{e} inequality on $\Sigma$ (we have that $\eta$ has zero average)  we know
\begin{equation}\label{p_D_b_23}
 \ns{ \eta}_{2n-1/2} \ls \ns{\eta}_{0} + \ns{D\eta}_{2n-3/2} \ls \ns{D\eta}_{2n-3/2} \ls \z.
\end{equation}
Similarly, for $j=2,\dotsc,n+1$ we may apply $\dt^{j-1}$ to \eqref{p_D_b_22} and estimate
\begin{multline}\label{p_D_b_24}
 \norm{\dt^j \eta}_{2n-2j+ 5/2}^2 \ls \snormspace{\dt^{j-1} u_3}{2n-2j+ 5/2}{\Sigma}^2  
+ \ns{ \dt^{j-1} G^4}_{2n-2j+ 5/2} \\
\ls \snormspace{\dt^{j-1} u}{2n-2(j-1) + 1 }{\Omega_3}^2  
+ \ns{ \dt^{j-1} G^4}_{2n-2(j-1) + 1/2} \ls \z.
\end{multline}
It remains only to control $\dt \eta$, which we do again using \eqref{p_D_b_22}:
\begin{equation}\label{p_D_b_24_1}
 \ns{\dt \eta}_{2n-1/2} \ls \snormspace{ u_3}{2n-1/2}{\Sigma}^2  
+ \ns{ G^4}_{2n-1/2} \ls \snormspace{ u_3}{2n}{\Omega_3}^2 + \z \ls \z.
\end{equation}
Summing  estimates \eqref{p_D_b_23}--\eqref{p_D_b_24_1}  then yields
\begin{equation}\label{p_D_b_25}
\ns{ \eta}_{2n-1/2} + \ns{ \dt \eta}_{2n-1/2} +  \sum_{j=2}^{n+1} \norm{\dt^j \eta}_{2n-2j+ 5/2}^2 \ls \z.
\end{equation}

The $\eta$ estimates \eqref{p_D_b_25} now allow us to further improve the estimates for the pressure.  Indeed, for $j=0,\dotsc,n-1$ we may use Lemma \ref{poincare_b} and \eqref{p_D_b_21} to bound
\begin{multline}
 \snormspace{\dt^j p}{0}{\Omega_3}^2 
\ls \ns{\dt^j \eta}_{0} +\snormspace{\partial_3 \dt^j u_3}{0}{\Sigma}^2 + \ns{ \dt^j G^3}_{0} + \snormspace{\dt^j \nab p}{0}{\Omega_3}^2  \\
\ls \snormspace{ \dt^j u_3}{2}{\Omega_3}^2 + \z \ls \z.
\end{multline}
This, \eqref{p_D_b_2}, and \eqref{p_D_b_25} allow us to improve \eqref{p_D_b_20} to 
\begin{multline}\label{p_D_b_26}
  \sum_{j=0}^{n} \snormspace{ \dt^j u}{2n-2j+1}{\Omega_3}^2 
+  \sum_{j=0}^{n-1} \snormspace{ \dt^j  p}{2n-2j }{\Omega_3}^2 \\
+ \ns{ \eta}_{2n-1/2} + \ns{\dt \eta}_{2n-1/2} + \sum_{j=2}^{n+1} \norm{\dt^j \eta}_{2n-2j+ 5/2}^2 \ls \z.
\end{multline}

Step 4 -- Estimates in $\Omega_2$

We now extend our estimates to the lower domain, $\Omega_2$, by initially applying Lemma \ref{p_dissipation_lower_domain} for
\begin{equation}\label{p_D_b_27}
 \sum_{j=0}^{n} \ns{\dt^j (\chi_2 u)}_{2n-2j+1} + \sum_{j=0}^{n-1} \ns{\dt^j (\chi_2 p)}_{2n-2j}
\ls \sdb{n}^{-} + \sdb{n}^{0} + \x_n + \y_n,
\end{equation}
where $\x_n$ is defined by
\begin{equation}
 \x_n = \sum_{j=0}^{n-1}  \ns{\dt^j H^{1,2}}_{2n-2j-1}  
+  \ns{\dt^j H^{2,2}}_{2n-2j}
\end{equation}
for $H^{1,2}$ and $H^{2,2}$ given by \eqref{p_H_def}.   We must now estimate $\x_n$.   For this we note that by construction $\supp(\nab \chi_2) \subset \Omega_3$, which implies that $\supp(H^{1,2})\cup \supp(H^{2,2}) \subset \Omega_3$.  This allows us to use the estimate \eqref{p_D_b_26} to bound
\begin{equation}\label{p_D_b_29}
 \x_n \ls  \sum_{j=0}^{n-1}\left( \snormspace{ \dt^j u}{2n-2j+1}{\Omega_3}^2 
+ \snormspace{ \dt^j  p}{2n-2j }{\Omega_3}^2\right) \ls  \z.
\end{equation}
Then estimates \eqref{p_D_b_27} and \eqref{p_D_b_29} may be combined to get
\begin{multline}\label{p_D_b_30}
\sum_{j=0}^{n} \snormspace{\dt^j u}{2n-2j+1}{\Omega_2}^2 + \sum_{j=0}^{n-1} \snormspace{\dt^j  p }{2n-2j}{\Omega_2}^2 \\
\ls    \sum_{j=0}^{n} \norm{\dt^j (\chi_2 u)}_{2n-2j+1}^2 + \sum_{j=0}^{n-1} \norm{\dt^j (\chi_2 p)}_{2n-2j}^2 
\ls \sdb{n}^{-} + \z.
\end{multline}

Step 5 --Estimates on all of $\Omega$ and conclusion

We recall that $\Omega = \Omega_3 \cup \Omega_2$.  This allows us to  add the localized estimates \eqref{p_D_b_26} and \eqref{p_D_b_30} to deduce \eqref{p_D_b_00}.  In order to deduce \eqref{p_D_b_01} and \eqref{p_D_b_02} from \eqref{p_D_b_00}, we must only estimate $\y_n$ for $n=2N$ and $n=N+2$.  In the case $n=2N$, Theorem \ref{p_G_estimates} provides the estimate $\y_{2N} \ls (\se{2N})^\theta \sd{2N} + \k \f$, and \eqref{p_D_b_01} follows.  In the case $n=N+2$ we use Theorem \ref{p_G_estimates_half} for $\y_{N+2} \ls  (\se{2N})^\theta \sd{N+2}$, and \eqref{p_D_b_02} follows.

\end{proof}

The next result is a key bootstrap estimate used in the proof of Theorem \ref{p_dissipation_bound}.

\begin{lem}\label{p_bootstrap_estimate}
Let $\y_n$ be as defined in Theorem \ref{p_dissipation_bound}.  Suppose that 
\begin{equation}\label{p_b_e_0}
  \snormspace{\db_0^{2n-2r+1}  u}{2r}{\Omega_3}^2 
+  \snormspace{\db_0^{2n-2r+1}  \nab p}{2r-2}{\Omega_3}^2 \ls \sdb{n}^0 + \sdb{n}^+ + \y_n
\end{equation}
for an integer $r \in \{1\dotsc,n-1\}$.  Then
\begin{multline}\label{p_b_e_00}
 \snormspace{ \dt^{n-r} u}{2r+1}{\Omega_3}^2 
+ \snormspace{\dt^{n-r} \nab p}{2r-1}{\Omega_3}^2 \\
+\snormspace{\db_0^{2n-2(r+1)+1} u}{2r+2}{\Omega_3}^2    
 + \snormspace{\db_0^{2n-2(r+1)+1} \nab p}{2r}{\Omega_3}^2 \ls \sdb{n}^0 + \sdb{n}^+ + \y_n.
\end{multline}
Moreover, if \eqref{p_b_e_0} holds with $r=1$, then
\begin{equation}\label{p_b_e_000}
   \sum_{j=0}^{n-1} \snormspace{\dt^j u}{2n-2j+1}{\Omega_3}^2 
+  \sum_{j=0}^{n-1} \snormspace{ \dt^j \nab p}{2n-2j -1}{\Omega_3}^2 \ls \sdb{n}^0 + \sdb{n}^+ + \y_n.
\end{equation}
\end{lem}

\begin{proof}

The estimate  \eqref{p_b_e_00} may be derived as in Lemma  \ref{i_bootstrap_estimate}  by setting $m=0$ in its proof,  using $\Omega_3$ in place of its $\Omega_1$, and $\sdb{n}^0 + \sdb{n}^+$ in place of its $\sdb{n,m}$.

Now we turn to the proof of \eqref{p_b_e_000}, assuming that \eqref{p_b_e_0} holds with $r=1$.  By \eqref{p_b_e_00} we may iterate with $r=2,\dotsc,n-1$ to deduce  that
\begin{multline}\label{p_b_e_12}
 \snormspace{D_0^1 u  }{2n}{\Omega_3}^2 + \sum_{j=1}^n  \snormspace{\dt^j  u}{2n-2j+1}{\Omega_3}^2
\\
+ \snormspace{D_0^1 \nab p  }{2n-2}{\Omega_3}^2 +  \sum_{j=1}^{n-1} \snormspace{\dt^j \nab p}{2n-2j-1}{\Omega_3}^2
 \ls \sdb{n}^0 + \sdb{n}^+ + \y_n.
\end{multline}
Let  $1 \le \ell \le  2n  - 1$.  We apply the operator $\p_3^\ell$ to the first equation of \eqref{linear_perturbed} and split into components  to get
\begin{equation}\label{p_b_e_13}
\partial_3^{\ell+1}  p =  -\dt \partial_3^{\ell} u_3 + \Delta  \partial_3^{\ell}   u_3 +  \partial_3^{\ell}  G^1_3, \text{ and }
\end{equation}
\begin{equation}\label{p_b_e_14}
\partial_3^{\ell+2}  u_i  = -(\p_1^2 + \p_2^2)  \partial_3^{\ell}  u_i  + \dt \partial_3^{\ell}  u_i  + \p_i  \partial_3^{\ell}  p -  \partial_3^{\ell}  G^1_i \text{ for }i=1,2.
\end{equation}
Then \eqref{p_b_e_12}, together with \eqref{p_b_e_13}--\eqref{p_b_e_14} and the equation $\p_3 u_3 = G^4 - \p_1 u_1 - \p_2 u_2$,  allows us to derive the estimates
\begin{equation}\label{p_b_e_15}
\snormspace{\partial_3^{\ell+2}   u}{0}{\Omega_3}^2 +  \snormspace{\partial_3^{\ell+1}    p}{0}{\Omega_3}^2    \ls \sdb{n}^0 + \sdb{n}^+ + \y_n.
\end{equation}
This and \eqref{p_b_e_12} yield \eqref{p_b_e_000}.

\end{proof}

The following result is based on an argument similar to the one used in Theorem \ref{p_energy_bound}.

\begin{lem}\label{p_dissipation_lower_domain}
Let $\y_n$ be as defined in Theorem \ref{p_dissipation_bound}.  Let $H^{1,2}$ and $H^{2,2}$ be given by \eqref{p_H_def}, and write
\begin{equation}
 \x_n = \sum_{j=0}^{n-1}  \ns{\dt^j H^{1,2}}_{2n-2j-1}  
+  \ns{\dt^j H^{2,2}}_{2n-2j}.
\end{equation}
Then
\begin{equation}\label{p_d_l_d_0}
 \sum_{j=0}^{n} \ns{\dt^j (\chi_2 u)}_{2n-2j+1} + \sum_{j=0}^{n-1} \ns{\dt^j (\chi_2 p)}_{2n-2j}
\ls \sdb{n}^{-} + \sdb{n}^{0} + \x_n + \y_n.
\end{equation}
\end{lem}

\begin{proof}
First note that by Lemma \ref{i_korn} we may bound
\begin{equation}\label{p_d_l_d_3}
  \sum_{j=0}^{n} \norm{\dt^j (\chi_2 u)}_{1}^2 \ls \sdb{n}^{-} + \sdb{n}^{0}.
\end{equation}
When we localize with $\chi_2$ we find that $\chi_2 u$ and $\chi_2 p$ solve
\begin{equation}
 \begin{cases}
  \dt (\chi_2 u) - \Delta (\chi_2 u) + \nab (\chi_2 p) = \chi_2 G^1  + H^{1,2} &\text{in }\Omega \\
  \diverge(\chi_2 u) = \chi_2 G^2 + H^{2,2} &\text{in }\Omega\\
  ((\chi_2 p) I - \sg (\chi_2 u)) e_3 = 0  &\text{on }\Sigma\\
 \chi_2 u = 0 & \text{on } \Sigma_b.
 \end{cases}
 \end{equation}
Then for any $j=0,\dotsc,n-1$ we may apply Lemma \ref{i_linear_elliptic} to see that
\begin{multline}\label{p_d_l_d_1}
 \norm{\dt^j (\chi_2 u ) }_{2n-2j+1}^2 + \norm{\dt^j (\chi_2 p)  }_{2n-2j}^2  \\ \ls 
\norm{\dt^{j+1} (\chi_2 u)  }_{2n-2(j+1) +1}^2 + \norm{\dt^j ( \chi_2  G^1 + H^{1,2})  }_{2n-2j-1}^2 
+ \norm{\dt^j (\chi_2 G^2 + H^{2,2}) }_{2n-2j}^2 \\
\ls \norm{\dt^{j+1} (\chi_2 u)  }_{2n-2(j+1) +1}^2 + \y_n + \x_n.
\end{multline}

We will use estimate \eqref{p_d_l_d_1} and a finite induction to prove \eqref{p_d_l_d_0}.  For  $j=n-1$ we use \eqref{p_d_l_d_3} to  get
\begin{equation}
 \norm{\dt^{n-1} (\chi_2 u ) }_{3}^2 + \norm{\dt^{n-1} (\chi_2 p)  }_{2}^2  \ls \norm{\dt^{n} (\chi_2 u)  }_{1}^2 + \y_n  \ls  \sdb{n}^{-} + \sdb{n}^{0} + \y_n + \x_n.
\end{equation}
Now suppose that the inequality
\begin{equation}\label{p_d_l_d_2}
 \norm{\dt^{n-\ell} (\chi_2 u ) }_{2\ell +1}^2 + \norm{\dt^{n-\ell} (\chi_2 p)  }_{2\ell}^2  \ls \sdb{n}^{-} +\sdb{n}^{0} + \y_n + \x_n
\end{equation}
holds for $1 \le \ell < n$.  We claim that \eqref{p_d_l_d_2} holds with $\ell$ replaced by $\ell+1$.  We apply \eqref{p_d_l_d_1} with $j=n-\ell-1$ to get
\begin{multline}
 \norm{\dt^{n-\ell-1} (\chi_2 u ) }_{2(\ell+1)+1}^2 + \norm{\dt^{n-\ell-1} (\chi_2 p)  }_{2(\ell+1)}^2  \ls \norm{\dt^{n-\ell} (\chi_2 u)  }_{2\ell + 1}^2 + \y_n + \x_n  \\
\ls  \sdb{n}^{-} + \sdb{n}^{0} + \y_n + \x_n,
\end{multline}
where in the last inequality we have employed the induction hypothesis \eqref{p_d_l_d_2}.  This proves the claim, so by finite induction the bound \eqref{p_d_l_d_2} holds for all $\ell = 1,\dotsc,n$.  Summing this bound over $\ell = 1,\dotsc,n$ and adding \eqref{p_d_l_d_3} then yields \eqref{p_d_l_d_0}.
\end{proof}

\section{A priori estimates}\label{per_6}
In this section we will combine our energy evolution estimates with the comparison estimates to derive a priori estimates for the full energy, $\g$, defined by \eqref{p_total_energy}.

\subsection{Estimates involving $\f$ and $\k$}

Our first result is an estimate of $\f$.  

\begin{lem}\label{p_specialized_transport_estimate}
It holds that
\begin{multline}\label{p_ste_0}
\sup_{0\le r \le t}  \f(r) \ls \exp\left(C \int_0^t \sqrt{\k(r)} dr \right) \\
\times \left[ \f(0)   +   t \int_0^t (1+\se{2N}(r)) \sd{2N}(r)dr  + \left( \int_0^t \sqrt{\k(r) \f(r)} dr\right)^2 \right].
\end{multline}
\end{lem}
\begin{proof}
The proof of Lemma \ref{i_specialized_transport_estimate} also works in the periodic case since the estimates for solutions to the transport equations given in Lemma \ref{i_sobolev_transport} also hold when $\Sigma = (L_1 \mathbb{T})\times (L_2 \mathbb{T})$.

\end{proof}

Now we use this result to derive a stronger result.

\begin{prop}\label{p_f_bound}
There exists a universal constant $0< \delta < 1$ so that if $\g(T) \le \delta$, then 
\begin{equation}\label{p_f_b_0}
\sup_{0\le r \le t}  \f(r) \ls  
\f(0) +   t \int_0^t \sd{2N}
\end{equation}
for all $0 \le t \le T$.
\end{prop}
\begin{proof}
The Sobolev and trace embeddings allow us to estimate $\k \ls \se{N+2}$.  Then 
\begin{equation}
 \int_0^t \sqrt{\k(r)} dr \ls \int_0^t \sqrt{\se{N+2}(r)} dr \le  \sqrt{\delta} \int_0^\infty  \frac{1}{(1+r)^{2N-4}}dr  \ls \sqrt{\delta}.
\end{equation}
With this estimate in hand, we may argue as in the proof of Proposition \ref{i_f_bound} to see that
\begin{equation}
\sup_{0\le r \le t}  \f(r) \le  
C \left( \f(0) +   t \int_0^t \sd{2N}\right)   + C \delta \left( \sup_{0\le r \le t}  \f(r) \right),
\end{equation}
for some $C>0$.  Then if $\delta$ is small enough so that $ C \delta \le 1/2$, we may absorb the right-hand $\f$ term onto the left and deduce \eqref{p_f_b_0}.

\end{proof}

This bound on $\f$ allows us to estimate  the integral of $\k \f$ and $\sqrt{\sd{2N} \k \f}$.

\begin{cor}\label{p_kf_integral}
There exists a universal constant $0< \delta < 1$ so that if $\g(T) \le \delta$, then 
\begin{equation}\label{p_kfi_0}
 \int_0^t \k(r) \f(r) dr \ls \delta \f(0) + \delta \int_0^t \sd{2N}(r) dr
\end{equation}
and
\begin{equation}\label{p_kfi_01}
\int_0^t \sqrt{\sd{2N}(r) \k(r) \f(r)} dr \ls \f(0) + \sqrt{\delta} \int_0^t \sd{2N}(r) dr 
\end{equation}
for $0 \le t \le T$.
\end{cor}
\begin{proof}
Let $\g(T) \le \delta$ with $\delta$ as small as in Proposition \ref{p_f_bound} so that estimate \eqref{p_f_b_0} holds.  As in Proposition \ref{p_f_bound}, we have that $\k(r) \ls \se{N+2}(r) \le \delta (1+r)^{-4N+8}$.  Using this, we may argue as in Corollary \ref{i_kf_integral} to deduce the desired bounds.

\end{proof}

\subsection{Boundedness at the $2N$ level}

\begin{thm}\label{p_e2n_bound}
There exists a universal constant $\delta>0$ so that if $\g(T) \le \delta$, then
\begin{equation}\label{p_e2n_0}
\sup_{0 \le r \le t} \se{2N}(r) + \int_0^t  \sd{2N}   + \sup_{0 \le r \le t} \frac{\f(r)}{(1+r)} \ls \se{2n}(0) + \f(0)
\end{equation}
for all $0 \le t \le T$.
\end{thm}

\begin{proof}
Fix $0 \le t \le T$.  For any $\ep \in (0,1)$ we may sum the bounds of Propositions \ref{p_upper_evolution} and \ref{p_lower_evolution} to find
\begin{multline}\label{p_e2n_1}
 \seb{2N}^+(t) + \seb{2N}^-(t) + \int_0^t \sdb{2N}^+ + \sdb{2N}^-  \\
\le C_1 \left( \se{2N}(0) + \int_0^t (\se{2N})^\theta \sd{2N} + \sqrt{\sd{2N} \k \f}  + \ep \sd{2N} + \ep^{-8N-1} \sdb{2N}^0 \right).
\end{multline}
for a constant $C_1 >0$ independent of $\ep$. On the other hand, Proposition \ref{p_global_evolution} provides the estimate
\begin{equation}\label{p_e2n_2}
 \seb{2N}^0(t) + \int_0^t \sdb{2N}^0  \le C_2 \left( \se{2N}(0) + (\se{2N}(t))^{3/2} 
+ \int_0^t (\se{2N})^\theta \sd{2N}  \right)
\end{equation}
for a constant $C_2>0$.  We multiply \eqref{p_e2n_2} by $1 + C_*$ for a constant $C_*>0$ (with precise value to be chosen later) and add the resulting inequality to \eqref{p_e2n_1} for
\begin{multline}\label{p_e2n_3}
 \seb{2N}^+(t) + \seb{2N}^-(t) + (1+C_*) \seb{2N}^0(t) + \int_0^t \sdb{2N}^+ + \sdb{2N}^- +(1+ C_*)  \sdb{2N}^0 \\
\le  C_2 (1+C_*) (\se{2N}(t))^{3/2}   
 + (C_1 + C_2 (1+C_*))\left(  \se{2N}(0)  + \int_0^t (\se{2N})^\theta \sd{2N}  \right) 
\\
+ C_1 \int_0^t \sqrt{\sd{2N} \k \f} +  \ep \sd{2N} + \ep^{-8N-1} \sdb{2N}^0.
\end{multline}

From Theorem \ref{p_energy_bound} we know that
\begin{equation}\label{p_e2n_4}
 \se{2N}(t) \le C_3 \left(  \seb{2N}^{+}(t) +  \seb{2N}^0(t)  + (\se{2N}(t))^{1+\theta} \right),
\end{equation}
and from Theorem \ref{p_dissipation_bound} we know that
\begin{equation}\label{p_e2n_5}
\int_0^t  \sd{2N} \le C_3 \int_0^t  \left( \sdb{2N}^0 + \sdb{2N}^{-} + \sdb{2N}^{+} +  (\se{2N})^\theta \sd{2N} + \k \f \right)  
\end{equation}
for a constant $C_3>0$.  We may then combine \eqref{p_e2n_3}--\eqref{p_e2n_5} to see that
\begin{multline}\label{p_e2n_6}
 \frac{1}{C_3} \left(  \se{2N}(t) + \int_0^t  \sd{2N}\right) + C_* \left(  \seb{2N}^0(t) +  \int_0^t \sdb{2N}^0 \right) \le 
  C_2 (1+C_*) (\se{2N}(t))^{3/2}   \\
+ (\se{2N}(t))^{1+\theta} 
 + (1+ C_1 + C_2 (1+C_*))\left(  \se{2N}(0)  + \int_0^t (\se{2N})^\theta \sd{2N}  \right) 
\\
+  \int_0^t C_1 \sqrt{\sd{2N} \k \f} + \k \f + C_1 \int_0^t  \ep \sd{2N} + \ep^{-8N-1} \sdb{2N}^0.
\end{multline}
Now we choose  
\begin{equation}
 \ep = \min\left\{ \frac{1}{2}, \frac{1}{2C_1 C_3} \right\}   \Rightarrow  \ep \in(0,1) \text{ and } \frac{1}{2C_3} \le \frac{1}{C_3} - C_1 \ep,
\end{equation}
and then we choose $C_* = C_1 \ep^{-8N-1}$.  With this choice of $\ep$ and $C_*$, \eqref{p_e2n_6} reduces to 
\begin{multline}\label{p_e2n_9}
    \frac{1}{ 2C_3} \left( \se{2N}(t) +  \int_0^t  \sd{2N}\right)   
\le    C_2 (1+C_*) (\se{2N}(t))^{3/2}   
+ (\se{2N}(t))^{1+\theta} \\
 + (1+ C_1 + C_2 (1+C_*))\left(  \se{2N}(0)  + \int_0^t (\se{2N})^\theta \sd{2N}  \right) 
+  \int_0^t C_1 \sqrt{\sd{2N} \k \f} + \k \f.
\end{multline}

Let us assume that $\delta \in (0,1) $ is as small as in Corollary \ref{p_kf_integral}; this allows us to estimate the integrals involving $\k \f$ and $\sqrt{\sd{2N} \k \f}$ in \eqref{p_e2n_9} to bound
\begin{equation}\label{p_e2n_7}
 \se{2N}(t) + \int_0^t  \sd{2N} \le C_4 \left( \se{2N}(0) +  \f(0) + (\se{2N}(t))^{1+\psi}   
  + \int_0^t (\se{2N})^\theta \sd{2N}  +  \sqrt{\delta} \int_0^t \sd{2N}
\right)
\end{equation}
for $C_4>0$ and $\psi = \min\{1/2,\theta\}>0$.  Now we further assume that $\delta$ is small enough so  that 
\begin{equation}
 C_4 \sqrt{\delta} \le \frac{1}{4}, \; C_4  \delta^\theta \le \frac{1}{4}, \text{ and } C_4 \delta^\psi \le \hal.
\end{equation}
Then since $\sup_{0\le r \le t} \se{N+2}(r) \le \g(T) \le \delta$, \eqref{p_e2n_7} implies that
\begin{equation}\label{p_e2n_8}
\hal \left( \se{2N}(t) + \int_0^t  \sd{2N}\right)  \le C_4 \left( \se{2N}(0)   +\f(0) \right).
\end{equation}
If $\delta$ is further restricted to be as small as in Proposition \ref{p_f_bound}, then we also have that
\begin{multline}\label{p_e2n_10}
\sup_{0\le r \le t} \frac{\f(r)}{(1+r)} \ls \sup_{0\le r \le t} \frac{\f(0)}{(1+r)} + \sup_{0\le r \le t} \frac{r}{(1+r)} \int_0^r \sd{2N}(s)ds \\
 \ls \f(0) + \int_0^t \sd{2N}(r) dr \ls \se{2N}(0) + \f(0).
\end{multline}
Then \eqref{p_e2n_0} follows by summing \eqref{p_e2n_8} and \eqref{p_e2n_10}.

\end{proof}

\subsection{Decay at the $N+2$ level}

Before showing  the decay estimates, we first need an interpolation result.

\begin{prop}\label{p_n_interp}
There exists a universal $0 < \delta < 1$ so that if $\g(T) \le \delta$, then
\begin{equation}\label{p_ni_00}
\sd{N+2}(t) \ls \sdb{N+2}^0(t) + \sdb{N+2}^+(t) + \sdb{N+2}^-(t), \; \se{N+2}(t) \ls \seb{N+2}^0(t) + \seb{N+2}^+(t),
\end{equation}
and  
\begin{equation}\label{p_ni_01}
 \se{N+2} \ls (\sd{N+2})^{(4N-8)/(4N-7)} (\se{2N})^{1/(4N-7)}.
\end{equation}

\end{prop}
\begin{proof}
The bound $\g(T) \le \delta$ and Theorems \ref{p_energy_bound} and \ref{p_dissipation_bound}  imply that
\begin{multline}\label{p_ni_1}
 \sd{N+2} \le C (\sdb{N+2}^0(t) + \sdb{N+2}^+(t) + \sdb{N+2}^-(t)) + C \se{2N}^\theta \sd{N+2} \\
\le C (\sdb{N+2}^0(t) + \sdb{N+2}^+(t) + \sdb{N+2}^-(t)) + C \delta^\theta \sd{N+2} 
\end{multline}
and
\begin{equation}\label{p_ni_2}
 \se{N+2} \le C (\seb{N+2}^0(t) + \seb{N+2}^+(t)) + C \se{2N}^\theta \se{N+2} \le C (\seb{N+2}^0(t) + \seb{N+2}^+(t)) + C \delta^\theta \se{N+2} 
\end{equation}
for constants $C>0$ and $\theta>0$.  Then if $\delta$ is small enough so that $C \delta^\theta \le 1/2$, we may absorb the second term on the right side of \eqref{p_ni_1} and \eqref{p_ni_2} into the left to deduce the bounds in \eqref{p_ni_00}.

We now turn to the proof of \eqref{p_ni_01}, which  is based on the standard Sobolev interpolation inequality:
 \begin{equation}
 \norm{f}_{s} \ls \norm{f}_{s-r}^{q/(r+q)} \norm{f}_{s+q}^{r/(r+q)}
\end{equation}
for $s,q> 0$ and $0 \le r \le s$.  Applying this for $0 \le j \le N+2$ with $s = 2(N+2) - 2 j$, $r=1/2$, and $q=2N-4$ shows that
\begin{equation}
 \norm{\dt^j \eta }_{2(N+2) - 2 j} \ls \norm{\dt^j \eta}_{2(N+2) - 2 j-1/2}^{\theta} \norm{\dt^j \eta }_{4N-2j}^{1-\theta} \ls (\sqrt{\sd{N+2}})^\theta (\sqrt{\se{2N}})^{1-\theta},
\end{equation}
\begin{equation}
 \norm{\dt^j u }_{2(N+2) - 2 j} \ls \norm{\dt^j u}_{2(N+2) - 2 j-1/2}^{\theta} \norm{\dt^j u }_{4N-2j}^{1-\theta} \ls (\sqrt{\sd{N+2}})^\theta (\sqrt{\se{2N}})^{1-\theta},
\end{equation}
where
\begin{equation}
 \theta = \frac{4N-8}{4N-7} \text{ and } 1- \theta = \frac{1}{4N-7}.
\end{equation}
Similarly, we may use $0 \le j \le N+1$ with $s = 2(N+2) - 2 j -1$, $r=1/2$, and $q=2N-4$ 
\begin{equation}
 \norm{\dt^j p }_{2(N+2) - 2 j-1} \ls \norm{\dt^j p}_{2(N+2) - 2 j-3/2}^{\theta} \norm{\dt^j p }_{4N-2j-1}^{1-\theta} \ls (\sqrt{\sd{N+2}})^\theta (\sqrt{\se{2N}})^{1-\theta}.
\end{equation}
We may then sum the squares of these interpolation inequalities to deduce \eqref{p_ni_01}.

\end{proof}

Now we show that the extra integral term appearing in Proposition \ref{p_global_evolution_half} can essentially be absorbed into $\seb{N+2}^0 + \seb{N+2}^+$.

\begin{lem}\label{p_stray_control}
Let $F^2$ be defined by \ref{i_F2_def} with $\pa = \dt^{N+2}$.    There exists a universal $0< \delta <1$ so that if $\g(T) \le \delta$, then
\begin{equation}\label{p_stc_0}
 \frac{2}{3} (\seb{N+2}^0(t) + \seb{N+2}^+(t)) \le \seb{N+2}^0(t) + \seb{N+2}^+(t) - 2 \int_\Omega J(t) \dt^{N+1} p(t) F^2(t) \le \frac{4}{3} (\seb{N+2}^0(t) + \seb{N+2}^+(t))
\end{equation}
for all $0 \le t \le T$.
\end{lem}
\begin{proof}
Suppose that $\delta$ is as small as in Proposition \ref{p_n_interp}.  Then we combine estimate \eqref{p_F_e_h_02} of Theorem \ref{p_F_estimates_half},  Lemma \ref{infinity_bounds}, and estimate \eqref{p_ni_00} of Proposition \ref{p_n_interp} to see that
\begin{multline}
\pnorm{J}{\infty} \norm{\dt^{N+1} p}_{0} \norm{F^2}_{0} \ls \sqrt{\se{N+2}} \sqrt{\se{2N}^\theta \se{N+2}} \\
=  \se{2N}^{\theta/2} \se{N+2} \ls \se{2N}^{\theta/2} (\seb{N+2}^0 + \seb{N+2}^+)\ls \delta^{\theta/2} (\seb{N+2}^0 + \seb{N+2}^+)
\end{multline}
for some  $\theta >0$.  This estimate and Cauchy-Schwarz then imply that
\begin{equation}\label{p_stc_1}
 \abs{2 \int_\Omega J \dt^{N+1} p F^2 } \le 2 \pnorm{J}{\infty} \norm{\dt^{N+1} p}_{0} \norm{F^2}_{0} \le C \delta^{\theta/2} (\seb{N+2}^0 + \seb{N+2}^+) \le \frac{1}{3} (\seb{N+2}^0 + \seb{N+2}^+)
\end{equation}
if $\delta$ is small enough.  The bound \eqref{p_stc_0} follows easily from \eqref{p_stc_1}.

\end{proof}

Now we prove decay at the $N+2$ level.

\begin{thm}\label{p_n_decay}
There exists a universal constant $0<\delta<1$ so that if $\g(T) \le \delta$, then
\begin{equation}\label{p_nd_0}
\sup_{0\le r \le t} ( 1 + r)^{4N-8} \se{N+2}(r) \ls  \se{2N}(0) + \f(0)
\end{equation}
for all $0 \le t \le T$.
\end{thm}
\begin{proof}

Fix $0 \le t \le T$.  According to Propositions \ref{p_upper_evolution_half}, \ref{p_lower_evolution_half}, there exist constants $C_1>0$ so that for any $\ep \in(0,1)$, 
\begin{equation}\label{p_n_d_1}
 \dt ( \seb{N+2}^+ +  \seb{N+2}^- )+  \sdb{N+2}^+  + \sdb{N+2}^- \le C_1 ( \se{2N}^\theta \sd{N+2}  + \ep \sd{N+2} + \ep^{-4N-9} \sdb{N+2}^0 ).
\end{equation}
On the other hand, Proposition \ref{p_global_evolution_half} provides a constant $C_2>0$ so that
\begin{equation}\label{p_n_d_2}
\dt  \left( \seb{N+2}^0 - \int_\Omega 2 J \dt^{N+1}p F^2 \right) +  \sdb{N+2}^0  \le C_2 \sqrt{\se{2N}} \sd{N+2}.
\end{equation}
We multiply inequality \eqref{p_n_d_2} by $1 + C_*$ for $C_*>0$ a constant to be chosen later and add the resulting inequality to \eqref{p_n_d_1} to find
\begin{multline}\label{p_n_d_3}
 \dt \left( \seb{N+2}^+ +  \seb{N+2}^- + (1+ C_*)\seb{N+2}^0 - \int_\Omega 2 J \dt^{N+1}p F^2 \right)+  ( \sdb{N+2}^+  + \sdb{N+2}^- + \sdb{N+2}^0) + C_* \sdb{N+2}^0 \\
\le  (C_1 + C_2 C_*)   (\se{2N})^\psi \sd{N+2}  +  C_1 (\ep \sd{N+2} + \ep^{-4N-9} \sdb{N+2}^0 ),
\end{multline}
where $\psi = \min\{1/2,\theta \}$.

Let $\delta \in (0,1)$ be as small as in both Proposition \ref{p_n_interp} and Lemma \ref{p_stray_control}.  Then 
\begin{equation}\label{p_n_d_4}
 \sd{N+2} \le C_3 ( \sdb{n+2}^0 + \sdb{n+2}^{-} + \sdb{n+2}^{+})  
\end{equation}
for $C_3 \ge 1$ (we know that $C_3 >0$, but we may further assume this).  Then \eqref{p_n_d_4} and \eqref{p_n_d_3} imply that
\begin{multline}\label{p_n_d_5}
 \dt \left( \seb{N+2}^+ +  \seb{N+2}^- + (1+ C_*)\seb{N+2}^0 - \int_\Omega 2 J \dt^{N+1}p F^2 \right) +  \frac{1}{C_3} \sd{n+2}  + C_* \sdb{2n}^0 \\
\le  (C_1 + C_2 C_*)   (\se{2N})^\psi \sd{N+2}  +  C_1 (\ep \sd{N+2} + \ep^{-4N-9} \sdb{N+2}^0 ). 
\end{multline}
Now we choose 
\begin{equation}
 \ep = \min\left\{ \frac{1}{2}, \frac{1}{4 C_1 C_3} \right\} \Rightarrow \frac{3}{4 C_3} \le \frac{1}{C_3} - C_1 \ep
\end{equation}
and $C_* = C_1 \ep^{-4N-9}$.  Further, we assume $\delta$ is sufficiently small so that
\begin{equation}
 ( C_1 + C_2 C_*)  \delta^{\psi} \le \frac{1}{4C_3}.
\end{equation}
With this choice of $\ep, C_*$ and the bound $\se{2N} \le \g(T) \le \delta$, the inequality \eqref{p_n_d_5} implies
\begin{equation}\label{p_n_d_6}
\dt \left( \seb{N+2}^+ +  \seb{N+2}^- + (1+ C_*)\seb{N+2}^0 - \int_\Omega 2 J \dt^{N+1}p F^2 \right) +  \frac{1}{2 C_3} \sd{N+2}   \le   0. 
\end{equation}

Let $\delta$ be as small as in Theorem \ref{p_e2n_bound}, Proposition \ref{p_n_interp}, and Lemma \ref{p_stray_control}.   Then  Theorem \ref{p_e2n_bound}, \eqref{p_ni_01} of Proposition \ref{p_n_interp}, and  \eqref{p_stc_0} of Lemma \ref{p_stray_control}  imply that
\begin{multline}\label{p_nd_2}
0 \le \frac{2}{3} (\seb{N+2}^+ +  \seb{N+2}^-) + (1+C_*) \seb{N+2}^- \le \seb{N+2}^+ +  \seb{N+2}^- + (1+ C_*)\seb{N+2}^0 - \int_\Omega 2 J \dt^{N+1}p F^2  \\ 
\le \frac{4}{3} (\seb{N+2}^+ +  \seb{N+2}^-) + (1+C_*) \seb{N+2}^- 
\le C_4 \se{N+2} \le C C_4 (\se{2N})^{1/(4N-7)} (\sd{N+2})^{(4N-8)/(4N-7)} 
\\
\le C_5 \z_0^{1/(4N-7)} (\sd{N+2})^{(4N-8)/(4N-7)}
\end{multline}
for all $0\le t \le T$, where we have written $\z_0:= \se{2N}(0) + \f(0)$, and $C_4,C_5$ are universal constants which we may assume satisfy $C_5 \ge C_4 \ge 1$.  Let us write
\begin{equation}
 h(t) =  \seb{N+2}^+(t) +  \seb{N+2}^-(t) + (1+ C_*)\seb{N+2}^0(t) - \int_\Omega 2 J(t) \dt^{N+1}p(t) F^2(t)   \ge 0, 
\end{equation}
as well as
\begin{equation}
 s = \frac{1}{4N-8} \text{ and } C_6 =  \frac{1}{2 C_3 C_5^{1+s} \z_0^s  }.
\end{equation}
We may then combine \eqref{p_n_d_6} with \eqref{p_nd_2} and use our new notation to derive the differential inequality
\begin{equation}\label{p_nd_3}
 \dt h(t) + C_6 (h(t))^{1+s} \le 0
\end{equation}
for $0\le t \le T$.  

Since \eqref{p_nd_2} says that $h(t)\ge 0$, we may integrate \eqref{p_nd_3} to find that for any $0 \le r \le T$,
\begin{equation}\label{p_nd_4}
 h(r) \le \frac{h(0)}{[1 + s C_6 (h(0))^s r ]^{1/s}}.
\end{equation}
Then \eqref{p_nd_2} implies that $h(0) \le C_4 \se{N+2}(0) \le C_4 \se{2N}(0) \le C_4 \z_0$, which in turn implies that
\begin{equation}\label{p_nd_5}
 s C_6 (h(0))^s = \frac{s}{2 C_3 C_5^{1+s}  } \left(\frac{h(0)}{\z_0}\right)^s \le \frac{s}{2  C_3 C_5^{1+s}  } C_4^s = \frac{s}{ 2 C_3 C_5  }  \left(\frac{C_4}{C_5}\right)^s \le 1
\end{equation}
since $0< s < 1$, $C_5 \ge C_4 \ge 1$, and $C_3 \ge 1$.  A simple computation shows that
\begin{equation}
 \sup_{r \ge 0} \frac{(1+r)^{1/s}}{(1+ M r)^{1/s} } = \frac{1}{M^{1/s}}
\end{equation}
when $0 \le M \le 1$ and $s >0$.  This, \eqref{p_nd_4}, and  \eqref{p_nd_5} then imply that
\begin{multline}\label{p_nd_6}
 (1+r)^{1/s} h(r) \le h(0) \frac{(1+r)^{1/s}}{[1 + s C_6 (h(0))^s r ]^{1/s}} \\
\le 
h(0) \left(\frac{2 C_3 C_5^{1+s}}{s}\right)^{1/s} \frac{\z_0}{h(0)} = \left(\frac{2 C_3 C_5^{1+s}}{s}\right)^{1/s} \z_0. 
\end{multline}
Now we use \eqref{p_ni_00} of Proposition \ref{p_n_interp} together with \eqref{p_nd_2} to bound
\begin{equation}\label{p_nd_7}
 \se{N+2}(r) \ls \seb{N+2}^0(r) + \seb{N+2}^+(r) \ls h(r) \text{ for } 0\le r \le T.
\end{equation}
The estimate \eqref{p_nd_0} then follows from \eqref{p_nd_6}, \eqref{p_nd_7}, and the fact that $s=1/(4N-8)$ and $\z_0 = \se{2N}(0) +\f(0)$.

\end{proof}

\subsection{A priori estimates for $\g$}

We now collect the results of Theorems \ref{p_e2n_bound} and \ref{p_n_decay} into a single bound on $\g$.  The estimate recorded specifically names the constant in the inequality with $C_1>0$ so that it can be referenced later.

\begin{thm}\label{p_a_priori}
There exists a universal $0 < \delta < 1$ so that if $\g(T) \le \delta$, then
\begin{equation}\label{p_apr_0}
 \g(t) \le C_1( \se{2N}(0) + \f(0))
\end{equation}
for all $0 \le t \le T$, where $C_1 >0$ is a universal constant.
\end{thm}
\begin{proof}
Let $\delta$ be as small as in Theorems \ref{p_e2n_bound} and \ref{p_n_decay}.  Then the conclusions of the theorems hold, and we may sum them to deduce \eqref{p_apr_0}.
 
\end{proof}

\section{Specialized local well-posedness}\label{per_7}

We now record a specialized version of the local well-posedness theorem.  

\begin{thm}\label{p_periodic_lwp}
Suppose the initial data satisfy the compatibility conditions of Theorem \ref{intro_lwp} and  $\ns{u(0)}_{4N} + \ns{\eta(0)}_{4N+1/2}  < \infty$.
Let $\ep >0$.  There exists a $\delta_0 = \delta_0(\ep) >0$ and a  
\begin{equation}\label{p_ilwp_00}
T_0 = C(\ep) \min\left\{1, \frac{1}{\ns{\eta(0)}_{4N+1/2}}  \right\} > 0,
\end{equation}
where $C(\ep)>0$ is a constant depending on $\ep$, so that if $0 < T \le T_0$ and $\ns{u(0)}_{4N} +  \ns{\eta(0)}_{4N} \le \delta_0,$ then there exists a unique solution $(u,p,\eta)$ to \eqref{geometric} on the interval $[0,T]$ that achieves the initial data.  The solution obeys the estimates
\begin{equation}\label{p_plwp_01}
 \sup_{0 \le t \le T} \se{2N}(t)  + \int_0^T \sd{2N}(t) dt
+ \int_0^T \left( \ns{\dt^{2N+1} u(t)}_{({_0}H^1)^*}  + \ns{ \dt^{2N} p(t) }_{0} \right) dt
\le \ep
\end{equation}
and
\begin{equation}\label{p_plwp_02}
 \sup_{0 \le t \le T} \f(t) \le C_2 \f(0) + \ep
\end{equation}
for $C_2>0$ a universal constant.  If $\eta_0$ satisfies the zero average condition
\begin{equation}\label{p_plwp_03}
 \int_{\Sigma} \eta_0 =0, \text{ then } \int_{\Sigma} \eta(t) = 0
\end{equation}
for all $t \in [0,T]$.
\end{thm}
\begin{proof}  
The existence, uniqueness, and estimates follow directly from Theorem \ref{intro_lwp}.  Then \eqref{p_plwp_03} follows from Lemma \ref{avg_const}.

\end{proof}

\begin{remark}\label{p_justification_remark}
The finiteness of the terms on the left of \eqref{p_plwp_01}--\eqref{p_plwp_02} justify all of the computations leading to Theorem \ref{p_a_priori}.
\end{remark}

\section{Global well-posedness and decay: proof of Theorem \ref{intro_per_gwp}}\label{per_8}

In order to combine the local existence result, Theorem \ref{p_periodic_lwp}, with the a priori estimates of Theorem \ref{p_a_priori}, we must be able to estimate $\g$ in terms of the estimates given in \eqref{p_plwp_01}--\eqref{p_plwp_02}.  We record this estimate now.

\begin{prop}\label{p_total_norm_estimate}
Suppose that $N \ge 3$.  Then there exists a universal constant $C_3>0$ with the following properties.  If $0 \le T$, then we have the estimate
\begin{multline}\label{p_tne_00}
 \g(T) \le   \sup_{0 \le t \le T} \se{2N}(t)  + \int_{0}^{T_2} \sd{2N}(t) dt \\
+  \sup_{0 \le t \le T} \f(t)  + C_3  (1+T)^{4N-8} \sup_{0 \le t \le T} \se{2N}(t).
\end{multline}
If $0 < T_1 \le T_2$, then we have the estimate
\begin{multline}\label{p_tne_01}
 \g(T_2) \le C_3 \g(T_1) +  \sup_{T_1 \le t \le T_2} \se{2N}(t)  + \int_{T_1}^{T_2} \sd{2N}(t) dt \\
+ \frac{1}{(1+T_1)} \sup_{T_1 \le t \le T_2} \f(t)  + C_3 (T_2-T_1)^2 (1+T_2)^{4N-8} \sup_{T_1 \le t \le T_2} \se{2N}(t).
\end{multline}
\end{prop}
\begin{proof}
The proof follows from a straightforward modification of the proof of Proposition \ref{i_total_norm_estimate}.
\end{proof}

We now turn to our main result in the periodic case.

\begin{proof}[Proof of Theorem \ref{intro_per_gwp}]
The proof follows from obvious modifications of the proof of Theorem \ref{i_gwp}, using Theorems \ref{p_a_priori} and \ref{p_periodic_lwp} and Proposition \ref{p_total_norm_estimate} in place of Theorems \ref{i_a_priori} and \ref{i_infinite_lwp} and Proposition \ref{i_total_norm_estimate}.

\end{proof}

\appendix

\chapter{Analytic tools}\label{section_appendix}

\section{Products in Sobolev spaces}
We will need some estimates of the product of functions in Sobolev spaces.

\begin{lem}\label{i_sobolev_product_1}
The following hold for sufficiently smooth subsets of $\Rn{n}$.
\begin{enumerate}
 \item Let $0\le r \le s_1 \le s_2$ be such that  $s_1 > n/2$.  Let $f\in H^{s_1}$, $g\in H^{s_2}$.  Then $fg \in H^r$ and
\begin{equation}\label{i_s_p_01}
 \norm{fg}_{H^r} \lesssim \norm{f}_{H^{s_1}} \norm{g}_{H^{s_2}}.
\end{equation}

\item Let $0\le r \le s_1 \le s_2$ be such that  $s_2 >r+ n/2$.  Let $f\in H^{s_1}$, $g\in H^{s_2}$.  Then $fg \in H^r$ and
\begin{equation}\label{i_s_p_02}
 \norm{fg}_{H^r} \lesssim \norm{f}_{H^{s_1}} \norm{g}_{H^{s_2}}.
\end{equation}

\item Let $0\le r \le s_1 \le s_2$ be such that  $s_2 >r+ n/2$. Let $f \in H^{-r}(\Sigma),$ $g \in H^{s_2}(\Sigma)$.  Then $fg \in H^{-s_1}(\Sigma)$ and
\begin{equation}\label{i_s_p_03}
 \norm{fg}_{-s_1} \ls \norm{f}_{-r} \norm{g}_{s_2}.
\end{equation}
\end{enumerate}
\end{lem}

\begin{proof}
The proofs of \eqref{i_s_p_01} and \eqref{i_s_p_02} are standard; the bounds are first proved in $\Rn{n}$ with the Fourier transform, and then the bounds in sufficiently nice subsets of $\Rn{n}$ are deduced by use of an extension operator.  To prove \eqref{i_s_p_03} we argue by duality.  For $\varphi \in H^{s_1}$ we use \eqref{i_s_p_02}bound
\begin{equation}
 \int_\Sigma \varphi f g \ls \norm{\varphi g}_{r} \norm{f}_{-r} \ls  \norm{\varphi}_{s_1}  \norm{g}_{s_2} \norm{f}_{-r},
\end{equation}
so that taking the supremum over $\varphi$ with $\norm{\varphi}_{s_1} \le 1$ we get \eqref{i_s_p_03}.

\end{proof}

We will also need the following variant.
\begin{lem}\label{i_sobolev_product_2}
Suppose that $f \in C^1(\Sigma)$ and $g \in H^{1/2}(\Sigma)$.  Then
\begin{equation}
 \norm{ fg}_{1/2} \ls \norm{f}_{C^1} \norm{g}_{1/2}.
\end{equation}
\end{lem}
\begin{proof}
Consider the operator $F: H^k \to H^k$ given by $F(g) = fg$ for $k=0,1$.  It is a bounded operator for $k=0,1$ since
\begin{equation}
 \norm{fg}_{0} \le \norm{f}_{C^1} \norm{g}_{0} \text{ and } \norm{fg}_{1} \ls \norm{f}_{C^1} \norm{g}_{1}.
\end{equation}
Then the theory of interpolation of operators implies that $F$ is bounded from $H^{1/2}$ to itself, with operator norm less than a constant times $\sqrt{\norm{f}_{C^1}} \sqrt{\norm{f}_{C^1}} = \norm{f}_{C^1}$, which is the desired result.
\end{proof}

\section{Continuity and temporal derivatives}

We will  need the following interpolation result, which affords us control of the $L^\infty H^k$ norm of a function $f$, given that  we control  $f$ in  $L^2 H^{k+m}$ and $\dt f$ in $L^2 H^{k-m}$.

\begin{lem}\label{l_sobolev_infinity}
Let $\Gamma$ denote either $\Sigma$ or $\Omega$.  Suppose that $\zeta \in L^2([0,T]; H^{s_1}(\Gamma))$ and $\dt \zeta \in L^2([0,T]; H^{s_2}(\Gamma))$ for $s_1 \ge s_2 \ge 0$.  Let $s = (s_1+s_2)/2$.  Then  $\zeta \in C^0([0,T]; H^{s}(\Gamma))$ (after possibly being redefined on a set of measure $0$), and 
\begin{equation}\label{l_sobi_01}
 \norm{\zeta}_{L^\infty H^{s}} \ls \left(1 + \frac{1}{T}\right)\left( \ns{ \zeta}_{L^2 H^{s_1}} + \ns{\dt \zeta}_{L^2 H^{s_2}} \right).
\end{equation}
\end{lem}

\begin{proof}

According to the usual theory of extensions and restrictions in Sobolev spaces, it suffices to prove the result with $\Gamma = \Rn{n}$ or $\Gamma = (L_1\mathbb{T})\times(L_2\mathbb{T}) \times \Rn{m}$ for $n = 2,3$, $m= 0,1$.  We will prove the result assuming that $\Gamma = \Rn{n}$;  the proof in the other case may be derived similarly, replacing integrals in Fourier space with sums, etc. Assume for the moment that $\zeta$ is smooth.  Writing  $\hat{\dot}$ for the Fourier transform, we compute
\begin{multline}
 \dt \ns{\zeta(t)}_{s} = 2 \Re\left( \int_{\Rn{n}}  \br{\xi}^{2s} \hat{\zeta}(\xi,t) \overline{\dt \hat{\zeta}(\xi,t)} d\xi  \right) \le 2 \int_{\Rn{n}}  \br{\xi}^{2s} \abs{\hat{\zeta}(\xi,t)}  \abs{\dt \hat{\zeta}(\xi,t)} d\xi  \\
= 2 \int_{\Rn{n}}  \br{\xi}^{s_1} \abs{\hat{\zeta}(\xi,t)}  \br{\xi}^{s_2} \abs{\dt \hat{\zeta}(\xi,t)} d\xi  \le  \int_{\Rn{n}}  \br{\xi}^{2s_1} \abs{\hat{\zeta}(\xi,t)}^2 d\xi + \int_{\Rn{n}} \br{\xi}^{2s_2} \abs{\dt \hat{ \zeta}(\xi,t)}^2 d \xi \\
= \ns{ \zeta(t)}_{s_1} + \ns{\dt \zeta(t)}_{s_2}.
\end{multline}
Hence for $r, t \in [0,T]$, we have that $\ns{\zeta(t)}_{s} \le  \ns{\zeta(r)}_{s} +  \ns{ \zeta}_{L^2 H^{s_1}} + \ns{\dt \zeta}_{L^2 H^{s_2}}$.  We can then integrate both sides of this inequality with respect to $r \in [0,T]$ to deduce the bound
\begin{equation}\label{l_sobi_1}
 \sup_{0 \le t \le T} \ns{\zeta(t)}_{s} \le  \frac{1}{T} \ns{\zeta}_{L^2 H^{s}} +  \ns{ \zeta}_{L^2 H^{s_1}} + \ns{\dt \zeta}_{L^2 H^{s_2}} \ls \left(1 + \frac{1}{T}\right) \left(\ns{ \zeta}_{L^2 H^{s_1}} + \ns{\dt \zeta}_{L^2 H^{s_2}} \right).
\end{equation}
If $\zeta$ is not smooth, we may employ a standard mollification argument (cf. Section 5.9 of \cite{evans}) in conjunction with \eqref{l_sobi_1} to deduce that $\zeta \in C^0([0,T]; H^s({\Rn{n}}))$ and that \eqref{l_sobi_01} holds.

\end{proof}

\section{Estimates of the Riesz potential $\i_\lambda$}\label{i_riesz_potential}

Consider the non-periodic case so that $\Omega = \Rn{2} \times (-b,0)$.  For a function $f$, defined on $\Omega$, we define the Riesz potential $\i_{\lambda} f$ by
\begin{equation}\label{il_def_1}
 \il f(x',x_3) = \int_{-b}^0 \int_{\Rn{2}} \hat{f}(\xi,x_3) \abs{\xi}^{-\lambda} e^{2\pi i x'\cdot \xi} d\xi dx_3.
\end{equation}
Similarly, for $f$ defined on $\Sigma$, we set
\begin{equation}\label{il_def_2}
 \il f(x') = \int_{\Rn{2}} \hat{f}(\xi) \abs{\xi}^{-\lambda} e^{2\pi i x'\cdot \xi} d\xi.
\end{equation}

We have a product estimate that is a fractional analog of the Leibniz rule.

\begin{lem}\label{i_riesz_prod}
Let $\lambda \in (0,1)$.  If $f \in H^0(\Omega)$ and $g, Dg \in H^1(\Omega)$, then 
\begin{equation}\label{i_r_p_0}
 \norm{\i_{\lambda} (fg)}_{0} \ls \norm{f}_0 \norm{g}_{1}^\lambda \norm{Dg}_{1}^{1-\lambda}.
\end{equation}
If $f \in H^0(\Sigma)$ and $g\in H^1(\Sigma)$, then
\begin{equation}\label{i_r_p_01}
 \snormspace{\i_{\lambda} (fg)}{0}{\Sigma} \ls \snormspace{f}{0}{\Sigma} \snormspace{g}{0}{\Sigma}^\lambda \snormspace{Dg}{0}{\Sigma}^{1-\lambda}.
\end{equation}

\end{lem}
\begin{proof}

The Hardy-Littlewood-Sobolev inequality (Theorem 4.3 of \cite{lieb_loss}) implies that $\i_\lambda: L^{2/(1+\lambda)}(\Rn{2}) \to L^{2}(\Rn{2})$ is a bounded linear operator for $\lambda \in (0,1)$.  We may then employ Fubini and apply this result on each slice $\{x_3 = z\}$ for $z\in(-b,0)$ to estimate
\begin{multline}\label{i_r_p_1}
 \int_\Omega \abs{\i_\lambda (fg)}^2 = \int_{-b}^0 \int_{\Rn{2}} \abs{\i_\lambda (fg)}^2 dx' dx_3 \ls \int_{-b}^0 \left(\int_{\Rn{2}} \abs{fg}^{2/(1+\lambda)} dx' \right)^{1+\lambda} dx_3 \\
\le \int_{-b}^0 \left( \int_{\Rn{2}} \abs{f}^2 dx' \right) \left( \int_{\Rn{2}} \abs{g}^{2/\lambda}dx' \right)^\lambda dx_3  \le \sup_{-b \le x_3 \le 0} \pnormspace{g(\cdot,x_3)}{2/\lambda}{\Rn{2}}^2 \int_\Omega \abs{f}^2 ,
\end{multline}
where in the second inequality we have applied H\"older's inequality.  By the Gagliardo-Nirenberg interpolation inequality on $\Rn{2}$ we may bound
\begin{equation}
 \pnormspace{g(\cdot,x_3)}{2/\lambda}{\Rn{2}} \ls \pnormspace{g(\cdot,x_3)}{2}{\Rn{2}}^{\lambda} \pnormspace{D g(\cdot,x_3)}{2}{\Rn{2}}^{1-\lambda},
\end{equation}
but by trace theory we also have
\begin{equation}
 \pnormspace{g(\cdot,x_3)}{2}{\Rn{2}} \ls \norm{g}_{1} \text{ and } \pnormspace{D g(\cdot,x_3)}{2}{\Rn{2}} \ls \norm{D g}_{1},
\end{equation}
so that 
\begin{equation}\label{i_r_p_2}
 \sup_{-b \le x_3 \le 0} \pnormspace{g(\cdot,x_3)}{2/\lambda}{\Rn{2}}^2 \ls 
\norm{g}_{1}^{\lambda} \norm{D g}_{1}^{1-\lambda}.
\end{equation}

Chaining together \eqref{i_r_p_1} and \eqref{i_r_p_2} then yields the estimate \eqref{i_r_p_0}.  A similar argument, not employing Fubini or trace theory, provides the estimate  \eqref{i_r_p_01}.
\end{proof}

Our next result shows how  $\i_\lambda$ interacts with horizontal derivatives in $\Omega$.

\begin{lem}\label{i_riesz_derivative}
Let $\lambda \in (0,1)$.  If $f \in H^k(\Omega)$ for $k \ge 1$ an integer, then 
\begin{equation}\label{i_r_d_0}
 \norm{\i_{\lambda} D^k f}_{0} \ls \norm{D^{k-1} f}_0^\lambda \norm{D^k f}_0^{1-\lambda}.
\end{equation}
\end{lem}
\begin{proof}
 On a fixed horizontal slice $\{x_3 = z\}$ for $z\in (-b,0)$, Parseval's theorem implies that
\begin{multline}
\int_{\Rn{2}} \abs{\i_\lambda D^k f(x',x_3) }^2 dx' \ls \int_{\Rn{2}} \abs{\xi}^{2(k-\lambda)} \abs{ \hat{f}(\xi,x_3) }^2 d\xi \\
= \int_{\Rn{2}} \left(\abs{\xi}^{2(k-1)} \abs{ \hat{f}(\xi,x_3) }^2\right)^{\lambda} \left(\abs{\xi}^{2k} \abs{ \hat{f}(\xi,x_3) }^2\right)^{1-\lambda}  d\xi \\
\ls \left(\int_{\Rn{2}} \abs{D^{k-1} f(x',x_3) }^2 dx' \right)^{\lambda}  \left(\int_{\Rn{2}} \abs{D^{k} f(x',x_3) }^2 dx' \right)^{1-\lambda}.
\end{multline}
Here in the second inequality we have used H\"older and Parseval.  Integrating both sides of this inequality with respect to $x_3 \in (-b,0)$ and again applying H\"older's inequality yields the estimate \eqref{i_r_d_0}.

\end{proof}

\section{Poisson integral: non-periodic case}
For a function $f$, defined on $\Sigma = \Rn{2}$, the Poisson integral in $\Rn{2} \times (-\infty,0)$  is defined by
\begin{equation}\label{poisson_def_inf}
 \mathcal{P}f(x',x_3) = \int_{\Rn{2}} \hat{f}(\xi) e^{2\pi \abs{\xi}x_3} e^{2\pi i x' \cdot \xi} d\xi.
\end{equation}
Although $\mathcal{P} f$ is defined in all of $\Rn{2} \times (-\infty,0)$, we will only need bounds on its norm in the restricted domain $\Omega = \Rn{2} \times (-b,0)$.  This yields a couple improvements of the usual estimates of $\mathcal{P} f$ on the set $\Rn{2} \times (-\infty,0)$.

\begin{lem}\label{i_poisson_grad_bound}
Let $\mathcal{P} f$ be the Poisson integral of a function $f$ that is either in $\dot{H}^{q}(\Sigma)$ or $\dot{H}^{q-1/2}(\Sigma)$ for $q \in \mathbb{N}$ (here $\dot{H}^s$ is the usual homogeneous Sobolev space of order $s$).  Then
\begin{equation}\label{i_p_g_b_0}
 \ns{\nab^q \mathcal{P}f }_{0} \ls  \int_{\Rn{2}} \abs{\xi}^{2q} \abs{\hat{f}(\xi)}^2 \left( \frac{1-e^{-4\pi b\abs{\xi} } }{\abs{\xi}} \right)  d\xi,
\end{equation}
and in particular
\begin{equation}\label{i_p_g_b_00} 
 \ns{\nab^q \mathcal{P}f }_{0} \ls \norm{f}_{\dot{H}^{q-1/2}(\Sigma)}^2 \text{ and }  \ns{\nab^q \mathcal{P}f }_{0} \ls \norm{f}_{\dot{H}^{q}(\Sigma)}^2.
\end{equation}
\end{lem}
\begin{proof}
Employing Fubini, the horizontal Fourier transform, and Parseval, we may bound
\begin{multline}\label{i_p_g_b_1}
 \ns{\nab^q \mathcal{P}f }_{0} \ls \int_{\Rn{2}} \int_{-b}^0 \abs{\xi}^{2q} \abs{\hat{f}(\xi)}^2  e^{4 \pi \abs{\xi} x_3} dx_3 d\xi 
\le \int_{\Rn{2}} \abs{\xi}^{2q} \abs{\hat{f}(\xi)}^2 \left( \int_{-b}^0   e^{4 \pi \abs{\xi} x_3} dx_3\right) d\xi  \\
\ls \int_{\Rn{2}} \abs{\xi}^{2q} \abs{\hat{f}(\xi)}^2 \left( \frac{1-e^{-4\pi b\abs{\xi} } }{\abs{\xi}} \right)  d\xi.
\end{multline}
This is \eqref{i_p_g_b_0}.  To deduce \eqref{i_p_g_b_00} from \eqref{i_p_g_b_0}, we simply note that
\begin{equation}
  \frac{1-e^{-4\pi b\abs{\xi} } }{\abs{\xi}} \le \min\left\{ 4 \pi b,\frac{1}{\abs{\xi}} \right\}, 
\end{equation}
which means we are free to bound the right hand side of \eqref{i_p_g_b_1} by either $\norm{f}_{\dot{H}^{q-1/2}(\Sigma)}^2$ or $\norm{f}_{\dot{H}^{q}(\Sigma)}^2$.
\end{proof}

\section{Poisson integral: periodic case}

Suppose that $\Sigma = (L_1 \mathbb{T}) \times (L_2 \mathbb{T})$.  We define the Poisson integral in $\Omega_- = \Sigma \times (-\infty,0)$ by
\begin{equation}\label{poisson_def_per}
\mathcal{P} f(x) = \sum_{n \in   (L_1^{-1} \mathbb{Z}) \times (L_2^{-1} \mathbb{Z}) }  e^{2\pi i n \cdot x'} e^{2\pi \abs{n}x_3} \hat{f}(n),
\end{equation}
where for $n \in   (L_1^{-1} \mathbb{Z}) \times (L_2^{-1} \mathbb{Z})$ we have written
\begin{equation}
 \hat{f}(n) = \int_\Sigma f(x')  \frac{e^{-2\pi i n \cdot x'}}{L_1 L_2} dx'.
\end{equation}
It is well known that $\mathcal{P}: H^{s}(\Sigma) \rightarrow H^{s+1/2}(\Omega_-)$ is a bounded linear operator for $s>0$.   We now show that how derivatives of $\mathcal{P} f$ can be estimated in the smaller domain $\Omega$.

\begin{lem}\label{p_poisson}
Let $\mathcal{P} f$ be the Poisson integral of a function $f$ that is either in $\dot{H}^{q}(\Sigma)$ or $\dot{H}^{q-1/2}(\Sigma)$ for $q \in \mathbb{N}$.  Then
\begin{equation} 
 \ns{\nab^q \mathcal{P}f }_{0} \ls \norm{f}_{\dot{H}^{q-1/2}(\Sigma)}^2 \text{ and }  \ns{\nab^q \mathcal{P}f }_{0} \ls \norm{f}_{\dot{H}^{q}(\Sigma)}^2.
\end{equation}
\end{lem}

\begin{proof}
Since $\mathcal{P} f$ is defined on $\Sigma \times (-\infty,0)$, it suffices to prove the estimates on $\tilde{\Omega} := \Sigma \times (-b_+,0)$ since $\Omega \subset \tilde{\Omega}$.  By Fubini and Parseval, 
\begin{multline}\label{pp0_1}
 \snormspace{\nab^q \mathcal{P}f }{0}{\tilde{\Omega}}^2   \ls \sum_{n \in   (L_1^{-1} \mathbb{Z}) \times (L_2^{-1} \mathbb{Z})  } \int_{-b_+}^0 \abs{n}^{2q} \abs{\hat{f}(n)}^2  e^{4 \pi \abs{n} x_3} dx_3  \\ 
\ls \sum_{n \in   (L_1^{-1} \mathbb{Z}) \times (L_2^{-1} \mathbb{Z})  }  \abs{n}^{2q} \abs{\hat{f}(n)}^2 \left( \frac{1-e^{-4\pi b_+ \abs{n}}}{\abs{n}} \right).
\end{multline}
However, 
\begin{equation}
  \frac{1-e^{-4\pi b_+ \abs{n} } }{\abs{n}} \le \min\left\{ 4 \pi b_+,\frac{1}{\abs{n}} \right\}, 
\end{equation}
which means we are free to bound the right hand side of \eqref{pp0_1} by either $\norm{f}_{\dot{H}^{q-1/2}(\Sigma)}^2$ or $\norm{f}_{\dot{H}^{q}(\Sigma)}^2$.
\end{proof}

We will also need $L^\infty$ estimates.
\begin{lem}\label{p_poisson_2}
Let $\mathcal{P} f$ be the Poisson integral of a function $f$ that is in $\dot{H}^{q+s}(\Sigma)$ for $q\ge 1$ an integer and $s> 1$.  Then
\begin{equation} 
 \pns{\nab^q \mathcal{P}f }{\infty} \ls \ns{f}_{\dot{H}^{q+s}}.
\end{equation}
The same estimate holds for $q=0$ if $f$ satisfies $\int_{\Sigma} f =0$.
\end{lem}
\begin{proof}
 We estimate
\begin{multline}
\pnorm{\nab^q \mathcal{P}f}{\infty} \ls  \sum_{n \in   (L_1^{-1} \mathbb{Z}) \times (L_2^{-1} \mathbb{Z})  }   \abs{\hat{f}(n)}  \abs{n}^q 
\\ \ls \norm{f}_{\dot{H}^{q+s}} 
\left( \sum_{n \in   (L_1^{-1} \mathbb{Z}) \times (L_2^{-1} \mathbb{Z})  \backslash \{0\}}   \abs{n}^{-2s}\right)^{1/2} 
\ls \norm{f}_{\dot{H}^{q+s}} 
\end{multline}
if $s > 1$.  The same estimate works with $q=0$ if $\hat{f}(0) =0$.
\end{proof}

\section{Interpolation estimates in the infinite case}

Assume that $\Sigma = \Rn{2}$ and $\Omega = \Sigma \times (-b,0)$. We begin with an interpolation result for Poisson integrals, as defined by \ref{poisson_def_inf}.

\begin{lem}\label{i_poisson_interp}
Let $\mathcal{P} f$ be the Poisson integral of $f$, defined on $\Sigma$. Let $\lambda \ge 0$, $q,s \in \mathbb{N}$, and $r\ge 0$.   Then the following estimates hold.

\begin{enumerate}
 \item Let 
\begin{equation}\label{i_p_i_1}
 \theta = \frac{s}{q+ s+\lambda} \text{ and } 1-\theta = \frac{q+\lambda}{q+s+\lambda}.
\end{equation}
Then
\begin{equation}\label{i_p_i_2}
 \ns{\nab^q \mathcal{P}f }_{0} \ls \left( \ns{\i_{\lambda} f }_{0}  \right)^\theta  \left( \ns{D^{q+s} f}_{0}  \right)^{1-\theta}.
\end{equation}

\item
Let $r+s >1$,
\begin{equation}\label{i_p_i_3}
\theta = \frac{r+s-1}{q+s+r+\lambda}, \text{ and } 1-\theta = \frac{q+\lambda+1}{q+s+r+\lambda}.
\end{equation}
Then
\begin{equation}\label{i_p_i_4}
 \pns{\nab^q \mathcal{P}f }{\infty} \ls \left( \ns{\i_{\lambda} f}_{0}  \right)^\theta  \left( \ns{D^{q+s} f }_{r}  \right)^{1-\theta}.
\end{equation}

\item Let $s >1$.  Then 
\begin{equation}\label{i_p_i_5}
 \pns{\nab^q \mathcal{P}f }{\infty} \ls \ns{D^q f}_{s}.
\end{equation}

\end{enumerate}
\end{lem}
\begin{proof}
Employing Fubini, the horizontal Fourier transform, and Parseval, we may bound
\begin{multline}
 \ns{\nab^q \mathcal{P}f }_{0} \ls \int_{\Rn{2}} \int_{-b}^0 \abs{\xi}^{2q} \abs{\hat{f}(\xi)}^2  e^{4 \pi \abs{\xi} x_3} dx_3 d\xi \ls \int_{\Rn{2}} \abs{\xi}^{2q} \abs{\hat{f}(\xi)}^2 d\xi.
\\
= \int_{\Rn{2}} \left( \abs{\xi}^{2(q+s)} \abs{\hat{f}(\xi)}^2 \right)^\theta 
\left(\abs{\xi}^{-2\lambda} \abs{\hat{f}(\xi)}^2  \right)^{1-\theta}  d\xi
\end{multline}
for $\theta$ and $1-\theta$ defined by \eqref{i_p_i_1}.  An application of H\"older's inequality and a second application of Parseval's theorem then provides the estimate \eqref{i_p_i_2}.

For the $L^\infty$ estimate \eqref{i_p_i_4}, we use the definition of $\mathcal{P} f$ and the trivial estimate  $\exp(2 \pi \abs{\xi} x_3) \le 1$ in $\Omega$ to bound
\begin{equation}
 \pnorm{\nab^q \mathcal{P}f }{\infty} \ls \int_{\Rn{2}}  \abs{\xi}^{q} \abs{\hat{f}(\xi)}  d\xi.
\end{equation}
For $R>0$ we split into high and low frequencies to see that
\begin{multline}
 \int_{\Rn{2}}  \abs{\xi}^{q} \abs{\hat{f}(\xi)}  d\xi = \int_{B_R}  \abs{\xi}^{q+ \lambda}  \abs{\xi}^{-\lambda} \abs{\hat{f}(\xi)}  d\xi 
+ \int_{B_R^c}  \abs{\xi}^{q+s} \br{\xi}^r \br{\xi}^{-r}  \abs{\xi}^{-s} \abs{\hat{f}(\xi)}  d\xi
\\
\le \left( \int_{B_R} \abs{\xi}^{2(q+ \lambda)} d\xi  \right)^{1/2} \norm{\i_{\lambda} f}_0 + \left( \int_{B_R^c} \abs{\xi}^{-2s} \br{\xi}^{-2r} d\xi \right)^{1/2} \norm{D^{q+s} f}_{r} 
\\
\ls R^{q+\lambda +1} \norm{\i_{\lambda} f}_0 + R^{-(r+s-1)}\norm{D^{q+s} f}_{r}.
\end{multline}
The condition $r+s>1$ guarantees that integral over $B_R^c$ is finite.   Minimizing the right side with respect to $R \in (0,\infty)$ then yields \eqref{i_p_i_4}.

The estimate \eqref{i_p_i_5} follows from the easy bound
\begin{equation}
 \int_{\Rn{2}}  \abs{\xi}^{q} \abs{\hat{f}(\xi)}  d\xi \ls \norm{D^q f}_{s} \left( \int_{\Rn{2}} \br{\xi}^{-2s} d\xi \right)^{1/2} \ls \norm{D^q f}_{s},
\end{equation}
which holds when $s>1$.

\end{proof}

The next result is a similar interpolation result for functions defined only on $\Sigma$.

\begin{lem}\label{i_sigma_interp}
Let $f$ be  defined on $\Sigma$. Let $\lambda \ge 0$.  Then the following estimates hold.

\begin{enumerate}
 \item Let $q,s \in (0,\infty)$ and
\begin{equation}
 \theta = \frac{s}{q+ s+\lambda} \text{ and } 1-\theta = \frac{q+\lambda}{q+s+\lambda}.
\end{equation}
Then
\begin{equation}\label{i_sig_i_1}
 \ns{D^q  f }_{0} \ls \left( \ns{\i_{\lambda} f }_{0}  \right)^\theta  \left( \ns{D^{q+s} f}_{0}  \right)^{1-\theta}.
\end{equation}

\item
Let $q,s \in \mathbb{N}$,  $r\ge 0$,  $r+s >1$,
\begin{equation}
\theta = \frac{r+s-1}{q+s+r+\lambda}, \text{ and } 1-\theta = \frac{q+\lambda+1}{q+s+r+\lambda}.
\end{equation}
Then
\begin{equation}\label{i_sig_i_2}
 \pns{D^q  f }{\infty} \ls \left( \ns{\i_{\lambda} f}_{0}  \right)^\theta  \left( \ns{D^{q+s} f }_{r}  \right)^{1-\theta}.
\end{equation}

\end{enumerate}
\end{lem}

\begin{proof}
For the $H^0$ estimate we use
\begin{equation}
 \ns{D^q f }_{0} \ls \int_{\Rn{2}} \abs{\xi}^{2q} \abs{\hat{f}(\xi)}^2 d\xi 
\end{equation}
and argue as in Lemma \ref{i_poisson_interp}. For the $L^\infty$ estimate we  bound
\begin{equation}
 \pnorm{D^q f }{\infty} \ls \int_{\Rn{2}}  \abs{\xi}^{q} \abs{\hat{f}(\xi)}  d\xi.
\end{equation}
and again argue as in Lemma \ref{i_poisson_interp}.

\end{proof}

Now we record a similar result for functions defined on $\Omega$ that are not Poisson integrals.  The result follows from estimates on fixed horizontal slices.

\begin{lem}\label{i_slice_interp}
Let $f$ be a function on $\Omega$.   Let $\lambda \ge 0$, $q,s \in \mathbb{N}$, and $r\ge 0$.   Then the following estimates hold.
\begin{enumerate}
 \item Let 
\begin{equation}\label{i_sl_i_1}
 \theta = \frac{s}{q+ s+\lambda} \text{ and } 1-\theta = \frac{q+\lambda}{q+s+\lambda}.
\end{equation}
Then
\begin{equation}\label{i_sl_i_2}
 \ns{D^q  f }_{0} \ls \left( \ns{\i_{\lambda} f }_{0}  \right)^\theta  \left( \ns{D^{q+s} f}_{0}  \right)^{1-\theta}.
\end{equation}

\item
Let $r+s >1$,
\begin{equation}\label{i_sl_i_3}
\theta = \frac{r+s-1}{q+s+r+\lambda}, \text{ and } 1-\theta = \frac{q+\lambda+1}{q+s+r+\lambda}.
\end{equation}
Then
\begin{equation}\label{i_sl_i_4}
 \pns{D^q  f }{\infty} \ls \left( \ns{\i_{\lambda} f}_{1}  \right)^\theta  \left( \ns{D^{q+s} f }_{r+1}  \right)^{1-\theta}
\end{equation}
and
\begin{equation}\label{i_sl_i_05}
 \ns{D^q  f }_{L^\infty(\Sigma)} \ls \left( \ns{\i_{\lambda} f}_{1}  \right)^\theta  \left( \ns{D^{q+s} f }_{r+1}  \right)^{1-\theta}
\end{equation}

\end{enumerate}
\end{lem}

\begin{proof}
We employ the horizontal Fourier transform and Parseval in conjunction with Fubini to bound
\begin{equation}
 \ns{D^q f}_0 \ls \int_{-b}^0 \int_{\Rn{2}} \abs{\xi}^{2q} \abs{\hat{f}(\xi,x_3)}^2 d\xi dx_3.
\end{equation}
For a fixed $x_3$ we may argue as in Lemma \ref{i_poisson_interp} to show that
\begin{equation}
\int_{\Rn{2}} \abs{\xi}^{2q} \abs{\hat{f}(\xi,x_3)}^2 d\xi \le 
\left( \ns{\i_\lambda f(\cdot,x_3)}_0 \right)^\theta \left(  \ns{D^{q+s} f(\cdot,x_3)}_0 \right)^{1-\theta}
\end{equation}
for $\theta$ and $1-\theta$ given by \eqref{i_sl_i_1}.  Combining these two inequalities with H\"older's inequality then shows that
\begin{multline}
 \ns{D^q f}_0 \ls \int_{-b}^0 \left( \ns{\i_\lambda f(\cdot,x_3)}_0 \right)^\theta \left(  \ns{D^{q+s} f(\cdot,x_3)}_0 \right)^{1-\theta} dx_3 \\
\le \left( \ns{\i_\lambda f}_0 \right)^\theta \left(  \ns{D^{q+s} f }_0 \right)^{1-\theta},
\end{multline}
which is \eqref{i_sl_i_2}.

Now for the $L^\infty$ estimate we first work on a horizontal slice $\{x_3 = z\}$ for some $z \in [-b,0]$.  Indeed, using the horizontal Fourier transform on the slice, we have
\begin{equation}\label{i_sl_i_5}
 \pnorm{D^q f(\cdot, x_3)}{\infty} \ls \int_{\Rn{2}} \abs{\xi}^q \abs{\hat{f}(\xi,x_3)}d\xi.
\end{equation}
We may then argue as in Lemma \ref{i_poisson_interp} to show that
\begin{equation}
 \int_{\Rn{2}} \abs{\xi}^q \abs{\hat{f}(\xi,x_3)}d\xi \ls  \left( \norm{\i_{\lambda} f (\cdot, x_3)}_{0}  \right)^\theta  \left( \norm{D^{q+s} f(\cdot,x_3) }_{r}  \right)^{1-\theta} 
\end{equation}
for $\theta$ and $1-\theta$ given by \eqref{i_sl_i_3}.  By the usual trace theory
\begin{equation}\label{i_sl_i_6}
 \norm{\i_{\lambda} f (\cdot, x_3)}_{0} \ls \norm{\i_{\lambda} f }_{1} \text{ and } \norm{D^{q+s} f(\cdot,x_3) }_{r} \ls \norm{D^{q+s} f }_{r+1}.
\end{equation}
Combining \eqref{i_sl_i_5}--\eqref{i_sl_i_6} and taking the supremum over $x_3\in [-b,0]$ then gives \eqref{i_sl_i_4}.  A similar argument yields \eqref{i_sl_i_05}.
\end{proof}

\section{Transport estimate}

Let $\Sigma$ be either periodic or non-periodic.  Consider the equation
\begin{equation}\label{i_transport_eqn}
\begin{cases}
 \dt \eta + u \cdot D \eta = g & \text{in } \Sigma \times (0,T)  \\
 \eta(t=0) = \eta_0
\end{cases}
\end{equation}
with $T\in(0,\infty]$. We have the following estimate of the transport of regularity for solutions to \eqref{i_transport_eqn}, which is a particular case of a more general result proved in \cite{danchin}.  Note that the result in \cite{danchin} is stated for $\Sigma = \Rn{2}$, but the same result holds in the periodic setting $\Sigma = (L_1 \mathbb{T}) \times (L_2 \mathbb{T})$, as described in \cite{danchin_notes}.

\begin{lem}[Proposition 2.1 of \cite{danchin}]\label{i_sobolev_transport}
 Let $\eta$ be a solution to \eqref{i_transport_eqn}.  Then there is a universal constant $C>0$ so that for any $0 \le s <2$
\begin{equation}
 \sup_{0\le r \le t} \snorm{\eta(r)}{s} \le \exp\left(C \int_0^t \snorm{D u(r)}{3/2} dr \right) \left( \snorm{\eta_0}{s} + \int_0^t \snorm{g(r)}{s}dr  \right).
\end{equation}
\end{lem}
\begin{proof}
Use $p=p_2 = 2$, $N=2$, and $\sigma=s$ in Proposition 2.1 of \cite{danchin} along with the embedding    $H^{3/2}  \hookrightarrow B^1_{2,\infty}\cap L^\infty  .$
\end{proof}

\section{Poincar\'{e}-type  inequalities }

Let $\Sigma$ and $\Omega$ be either periodic or non-periodic.

\begin{lem}\label{poincare_b}
 It holds that 
\begin{equation}\label{ipn_01}
 \pnormspace{f}{2}{\Omega}^2 \ls \pnormspace{f}{2}{\Sigma}^2 + \pnormspace{\partial_3 f}{2}{\Omega}^2
\end{equation}
for all $f \in H^1(\Omega)$.   Also, if $f \in W^{1,\infty}(\Omega)$, then
\begin{equation}\label{ipn_02}
 \pnormspace{f}{\infty}{\Omega}^2 \ls \pnormspace{f}{\infty}{\Sigma}^2 + \pnormspace{\partial_3 f}{\infty}{\Omega}^2.
\end{equation}
\end{lem}
\begin{proof}
 By density we may assume that $f$ is smooth.  Writing $x = (x',x_3)$ for $x'\in \Sigma$ and $x_3 \in (-b(x'),0),$ we have
\begin{multline}
 \abs{f(x',x_3)}^2 = \abs{f(x',0)}^2 - 2\int_{x_3}^0 f(x',z)\partial_3 f(x',z) dz \\
\le \abs{f(x',0)}^2 + 2\int_{-b(x')}^0 \abs{f(x',z)}\abs{\partial_3 f(x',z)} dz.
\end{multline}
We may integrate this with respect to $x_3 \in  (-b(x'),0)$ to get 
\begin{equation}
 \int_{-b(x')}^0  \abs{f(x',x_3)}^2 dx_3 \ls \abs{f(x',0)}^2 + 2\int_{-b(x')}^0 \abs{f(x',z)}\abs{\partial_3 f(x',z)} dz.
\end{equation}
Now we integrate over $x' \in \Sigma$ to find
\begin{multline}
 \int_{\Omega} \abs{f(x)}^2 dx \ls \pnormspace{f}{2}{\Sigma}^2 +2 \int_\Omega \abs{f(x)}\abs{\partial_3 f(x)} dx \\
\le 
\pnormspace{f}{2}{\Sigma}^2 +  \ep \pnormspace{f}{2}{\Omega}^2  + \frac{1}{\ep} \pnormspace{\partial_3 f}{2}{\Omega}^2
\end{multline}
for any $\ep>0$.  Choosing $\ep>0$ sufficiently small then  yields \eqref{ipn_01}.  The estimate \eqref{ipn_02} follows similarly, taking suprema rather than integrating.
\end{proof}

A simple modification of the proof of Lemma \ref{poincare_b} yields the following estimates.

\begin{lem}\label{poincare_trace}
It holds that $\snormspace{f}{0}{\Sigma} \ls \snormspace{\p_3 f}{0}{\Omega}$ for $f \in H^1(\Omega)$ so that $f =0$ on $\Sigma_b$. It also holds that $\pnormspace{f}{\infty}{\Sigma} \ls \pnormspace{\p_3 f}{\infty}{\Omega}$ for $f \in W^{1,\infty}(\Omega)$ so that $f =0$ on $\Sigma_b$.  
\end{lem}

We will need a version of Korn's inequality, which is proved, for instance, in Lemma 2.7 \cite{beale_1}.
\begin{lem}\label{i_korn}
It holds that $\norm{u}_{1} \ls \norm{\sg u}_{0}$ for all $u \in H^1(\Omega;\Rn{3})$ so that $u=0$ on $\Sigma_b$.
\end{lem}

We also record the standard Poincar\'e inequality, which applies for functions taking either vector or scalar values.

\begin{lem}\label{poincare_usual}
 It holds that $\norm{f}_{0} \ls \norm{f}_{1} \ls \norm{\nab f}_{0}$ for all $f \in H^1(\Omega)$ so that $f=0$ on $\Sigma_b$.  Also, $\pnormspace{f}{\infty}{\Omega} \ls \norm{f}_{W^{1,\infty}(\Omega)}\ls \pnormspace{\nab f}{\infty}{\Omega}$ for all $f \in W^{1,\infty}(\Omega)$ so that $f=0$ on $\Sigma_b$.
\end{lem}

\section{An elliptic estimate}

The proof of the following estimate may be found in  \cite{beale_1} in the non-periodic case.  The same proof holds in the periodic case with obvious modification.

\begin{lem}\label{i_linear_elliptic}
 Suppose $(u,p)$ solve
\begin{equation}
 \begin{cases}
  -\Delta u + \nab p =\phi \in H^{r-2}(\Omega) \\
  \diverge{u} = \psi \in H^{r-1}(\Omega) \\
  (pI - \sg (u) )e_3 = \alpha \in H^{r-3/2}(\Sigma) \\
  u\vert_{\Sigma_b}=0.
 \end{cases}
\end{equation}
Then for $r \ge 2$,
\begin{equation}
 \snorm{u}{r}^2 + \snorm{p}{r-1}^2 \ls \snorm{\phi}{r-2}^2 + \snorm{\psi}{r-1}^2 + \snorm{\alpha}{r-3/2}^2.
\end{equation}

\end{lem}


\end{document}